\documentclass[11pt, francais]{smfbook}
\usepackage{amsfonts, amssymb, amsmath, latexsym, enumerate, epsfig, color, mathrsfs, fancyhdr, supertabular, stmaryrd, smfthm, yhmath, mathtools, pifont, chngcntr, pinlabel, ifthen}
\usepackage[all]{xy}
\usepackage{tikz-cd}
\usepackage[french,refpage]{nomencl}

\usepackage[frenchb,english]{babel}

\usepackage[T1]{fontenc}
\usepackage[utf8]{inputenc}

\usepackage{hyperref}
\makenomenclature

\makeatletter \@removefromreset{figure}{chapter}\makeatother

\DeclareFontFamily{U}{mathc}{}
\DeclareFontShape{U}{mathc}{m}{it}%
{<->s*[1.03] mathc10}{}
\DeclareMathAlphabet{\mathcal}{U}{mathc}{m}{it}

\DeclareMathOperator{\sHom}{\mathcal{H\mkern-3mu om}}
\DeclareMathOperator{\sKer}{\mathcal{K\mkern-3mu er}}

\newcommand{\N}{\ensuremath{\mathbf{N}}}
\newcommand{\Z}{\ensuremath{\mathbf{Z}}}
\newcommand{\Q}{\ensuremath{\mathbf{Q}}}
\newcommand{\R}{\ensuremath{\mathbf{R}}}
\newcommand{\C}{\ensuremath{\mathbf{C}}}

\newcommand{\A}{\ensuremath{\mathbf{A}}}
\newcommand{\bA}{\ensuremath{\mathbf{A}}}
\renewcommand{\P}{\ensuremath{\mathbf{P}}}

\newcommand{\E}[2]{\ensuremath{\mathbf{A}^{#1,\mathrm{an}}_{#2}}}
\newcommand{\EP}[2]{\ensuremath{\mathbf{P}^{#1,\mathrm{an}}_{#2}}}
\newcommand{\Aunk}{\ensuremath{\mathbf{A}^{1,\mathrm{an}}_{k}}}

\newcommand{\AunA}{\ensuremath{\mathbf{A}^{1,\mathrm{an}}_{\mathcal{A}}}}
\newcommand{\Aunb}{\ensuremath{\mathbf{A}^{1,\mathrm{an}}_{\mathcal{H}(b)}}}

\newcommand\ti[1]{\ensuremath{\tilde{#1}}}
\newcommand{\wti}[1]{\ensuremath{\widetilde{#1}}}
\newcommand{\oD}{\overline{D}}
\newcommand{\oC}{\overline{C}}
\newcommand{\too}{\longrightarrow}
\newcommand{\simtoo}{\overset{\sim}{\longrightarrow}}
\newcommand{\mapstoo}{\longmapsto}

\newcommand{\ho}[1]{\mathbin{\hat{\otimes}_{#1}}}
\newcommand{\hosp}[1]{\mathbin{\hat{\otimes}^\mathrm{sp}_{#1}}}

\DeclareMathOperator{\Frac}{Frac}
\DeclareMathOperator{\Spec}{Spec}
\DeclareMathOperator{\Ob}{Ob}
\DeclareMathOperator{\Gal}{Gal}
\newcommand{\Ker}{\textrm{Ker}}
\renewcommand{\Im}{\textrm{Im}}
\DeclareMathOperator{\Hom}{Hom}

\newcommand{\Id}{\mathrm{Id}}
\newcommand{\id}{\mathrm{id}}

\newcommand{\res}{\mathrm{r\acute{e}s}}
\newcommand{\Res}{\mathrm{R\acute{e}s}}
\newcommand{\um}{\mathrm{um}}
\newcommand{\arc}{\mathrm{arc}}
\newcommand{\triv}{\mathrm{triv}}
\renewcommand{\sp}{\mathrm{sp}}
\renewcommand{\div}{\mathrm{div}}
\DeclareMathOperator*{\colim}{colim}

\newcommand{\hyb}{\mathrm{hyb}}

\newcommand{\tors}{\mathrm{tors}}
\newcommand{\ev}{\mathrm{ev}}
\newcommand{\simto}{\xrightarrow[]{\sim}}

\newcommand{\Ens}{\mathrm{Ens}}

\newcommand{\loc}{\mathrm{loc}}
\DeclareMathOperator{\dimr}{dim_{r}}

\newcommand\h{\textrm{h}}
\newcommand\urc{\!^\urcorner}

\newcommand{\m}{\ensuremath{\mathfrak{m}}}

\newcommand{\p}{\ensuremath{\mathfrak{p}}}
\newcommand{\q}{\ensuremath{\mathfrak{q}}}

\newcommand{\Es}{\mathscr{E}}

\newcommand{\Hs}{\mathscr{H}}

\newcommand{\sH}{\mathscr{H}}

\newcommand{\Ac}{\mathcal{A}}
\newcommand{\Bc}{\mathcal{B}}
\newcommand{\Mc}{\mathcal{M}}
\newcommand{\Oc}{\mathcal{O}}
\newcommand{\Pc}{\mathcal{P}}

\newcommand{\cA}{\mathcal{A}}
\newcommand{\cB}{\mathcal{B}}
\newcommand{\cC}{\mathcal{C}}

\newcommand{\cE}{\mathcal{E}}
\newcommand{\cF}{\mathcal{F}}
\newcommand{\cG}{\mathcal{G}}
\newcommand{\cH}{\mathcal{H}}
\newcommand{\cI}{\mathcal{I}}
\newcommand{\cJ}{\mathcal{J}}
\newcommand{\cK}{\mathcal{K}}
\newcommand{\cM}{\mathcal{M}}
\newcommand{\cN}{\mathcal{N}}
\newcommand{\cO}{\mathcal{O}}
\newcommand{\cR}{\mathcal{R}}
\newcommand{\cT}{\mathcal{T}}
\newcommand{\cU}{\mathcal{U}}
\newcommand{\cV}{\mathcal{V}}
\newcommand{\cW}{\mathcal{W}}
\newcommand{\cX}{\mathcal{X}}
\newcommand{\cY}{\mathcal{Y}}
\newcommand{\cZ}{\mathcal{Z}}

\newcommand{\hatcA}{\mathcal{A}\kern-0.6em\hat{\phantom{\mathcal{A}}}}
\newcommand{\cAAn}{\cA\!-\!\An}

\newcommand{\cn}[2]{\ensuremath{\llbracket{#1},{#2}\rrbracket}}

\DeclarePairedDelimiter{\intoo}{]}{[}
\DeclarePairedDelimiter{\intof}{]}{]}
\DeclarePairedDelimiter{\intfo}{[}{[}
\DeclarePairedDelimiter{\intff}{[}{]}

\newcommand{\la}{\ensuremath{\langle}}
\newcommand{\ra}{\ensuremath{\rangle}}

\newcommand{\eps}{\ensuremath{\varepsilon}}

\newcommand{\ba}{\ensuremath{{\boldsymbol{a}}}}

\newcommand{\bi}{\ensuremath{{\boldsymbol{i}}}}

\newcommand{\br}{\ensuremath{{\boldsymbol{r}}}}
\newcommand{\bs}{\ensuremath{{\boldsymbol{s}}}}
\newcommand{\bt}{\ensuremath{{\boldsymbol{t}}}}
\newcommand{\bu}{\ensuremath{{\boldsymbol{u}}}}

\newcommand{\bT}{\ensuremath{{\boldsymbol{T}}}}

\newcommand{\disp}{\displaystyle}

\newcommand\wc{{\mkern 2mu\cdot\mkern 2mu}}
\newcommand\va{|\wc|}
\newcommand\nm{\|\wc\|}
\newcommand\cf{\textit{cf}.}
\newcommand\an{\textrm{an}}
\newcommand\An{\textrm{An}}

\newcommand{\fonction}[5]{\begin{array}{cccc}
#1 \colon &#2&\too&#3\\
&#4&\mapstoo&#5
\end{array}}
\newcommand{\fonctionsp}[4]{\begin{array}{ccc}
#1&\too&#2\\
#3&\mapstoo&#4
\end{array}}

\newcommand{\mor}[1]{(#1,#1^\sharp)}

\def \spectriv(#1,#2,#3,#4,#5,#6){
\foreach \x in {0,...,#6}
\draw[line width=.001pt] ({#1},{#2}) -- ({#1+#5*cos((\x)/(#6)*(360/#4)+#3)},{#2+#5*sin((\x)/(#6)*(360/#4)+#3)}) ;}

\def \spectrivcentre(#1,#2,#3,#4,#5,#6){
\foreach \x in {-#6,...,#6}
\draw[line width=.001pt] ({#1},{#2}) -- ({#1+#5*cos((\x)/(#6)*(360/#4)+#3)},{#2+#5*sin((\x)/(#6)*(360/#4)+#3)}) ;}

\def \spectrivcentreenplus(#1,#2,#3,#4,#5,#6,#7,#8,#9){
\foreach \y in {-#5,...,#5}
\spectrivcentre(#1+#6*cos((\y)/(#5)*(360/#4)+#3),#2+#6*sin((\y)/(#5)*(360/#4)+#3),(\y)/(#5)*(360/#4)+#3,#7,#8,#9) 
;}

\makeindex
\makenomenclature

\begin{document}
\setlength{\baselineskip}{0.55cm}	

\theoremstyle{definition} 
\newenvironment{nota}{\begin{enonce}[definition]{Notation}}{\end{enonce}}

\title{Espaces de Berkovich globaux\\ cat\'egorie, topologie, cohomologie}
\author{Thibaud Lemanissier}
\address{Lyc\'ee Pierre de Coubertin, Meaux}
\email{\href{mailto:thibaud.lemanissier@ac-creteil.fr}{thibaud.lemanissier@ac-creteil.fr}}

\author{J\'er\^ome Poineau}
\address{Normandie Univ., UNICAEN, CNRS, Laboratoire de math\'ematiques Nicolas Oresme, 14000 Caen, France}
\email{\href{mailto:jerome.poineau@unicaen.fr}{jerome.poineau@unicaen.fr}}
\urladdr{\url{https://poineau.users.lmno.cnrs.fr/}}

\date{\today}

\maketitle

\pagestyle{empty}

\tableofcontents

\pagestyle{headings}

%

\pagenumbering{arabic}


\chapter*{Introduction}

\`A la fin des ann\'ees 1980, V.~Berkovich a propos\'e, dans l'ouvrage~\cite{Ber1}, une nouvelle approche de la g\'eom\'etrie analytique sur un corps $p$-adique, ou plus g\'en\'eralement ultram\'etrique. Elle a rapidement connu un grand succ\`es et trouv\'e des applications dans des domaines vari\'es des math\'ematiques tels que le programme de Langlands, la dynamique complexe, ou encore la g\'eom\'etrie diophantienne.

L'une des sp\'ecificit\'es des espaces de Berkovich r\'eside dans les excellentes propri\'et\'es topologiques dont ils jouissent, et qu'ils partagent avec les espaces analytiques complexes~: compacit\'e locale, connexit\'e par arcs locale, contractibilit\'e locale, etc. Ces propri\'et\'es sont d'autant plus remarquables que les corps ultram\'etriques sur lesquels ces espaces sont d\'efinis, ne sont, en g\'en\'eral, ni localement compacts, ni localement connexes. L'explication r\'eside dans le fait que les espaces de Berkovich ne contiennent pas uniquement les points \og classiques \fg{} (ceux du corps de base), mais de nombreux autres, que l'on peut, par exemple, utiliser pour tracer des chemins entre les premiers. 

\begin{figure}[!h]
\centering
\begin{tikzpicture}
\draw[line width=.001pt] (0,0) -- (0,5.9) ;

\spectrivcentre(0,5.9,295,8,1.5,7) ;

\spectrivcentre(0,4,290,4,.8,12) ;

\spectrivcentre(0,3,270,11,1.3,10) ;

\spectrivcentre(0,1.5,270,6,1.3,14) ;

\def \z{350} ;
\def \l{13} ;
\spectrivcentre(0,4,\z,\l,6,2) ;
\spectrivcentreenplus(0,4,\z,\l,2,1,40,1.2,2) ;
\spectrivcentreenplus(0,4,\z,\l,2,2.5,25,1.2,3) ;
\spectrivcentreenplus(0,4,\z,\l,2,4,18,1.2,4) ;
\spectrivcentreenplus(0,4,\z,\l,2,5.5,14,1.2,5) ;

\def \a{220} ;
\def \b{18} ;
\spectrivcentre(0,4,\a,\b,6,2) ;
\spectrivcentreenplus(0,4,\a,\b,2,1,40,1.2,2) ;
\spectrivcentreenplus(0,4,\a,\b,2,2.5,25,1.2,3) ;
\spectrivcentreenplus(0,4,\a,\b,2,4,18,1.2,4) ;
\spectrivcentreenplus(0,4,\a,\b,2,5.5,14,1.2,5) ;

\spectrivcentre(0,0,270,8,5,3) ;
\spectrivcentre(0,0,270,5,1.2,16) ;
\spectrivcentreenplus(0,0,270,8,3,1.5,25,1.2,3) ;
\spectrivcentreenplus(0,0,270,8,3,3,18,1.2,4) ;
\spectrivcentreenplus(0,0,270,8,3,4.5,14,1.2,5) ;
\end{tikzpicture}
\caption{Une droite de Berkovich traditionnelle.}\label{fig:A1Cp}
\end{figure}

Par d\'efinition, les points des espaces de Berkovich sont associ\'es \`a des semi-normes multiplicatives sur des anneaux bien choisis. Par exemple, tout point classique donne lieu \`a une semi-norme sur un anneau de polyn\^omes, obtenue en composant l'application d'\'evaluation en ce point avec la valeur absolue sur le corps de base, mais bien d'autres types de semi-normes existent. 


Un autre aspect particuli\`erement int\'eressant des espaces de Berkovich est la possibilit\'e de les d\'efinir en prenant comme base, non seulement un corps ultram\'etrique complet, mais, de fa\c{c}on bien plus g\'en\'erale, un anneau de Banach quelconque. On peut, par exemple, choisir comme base le corps des nombres complexes~$\C$ muni de la valeur absolue usuelle. Les espaces de Berkovich qui en r\'esultent sont alors les espaces analytiques complexes classiques. On peut \'egalement partir de l'anneau des entiers relatifs~$\Z$ muni de la valeur absolue usuelle. Comme la description en termes de semi-normes le laisse deviner, les espaces que l'on d\'efinit ainsi se composent \`a la fois de parties $p$-adiques, pour tout nombre premier~$p$, et de parties archim\'ediennes, proches des espaces analytiques complexes. Plus pr\'ecis\'ement, tout espace de Berkovich sur~$\Z$ se projette naturellement sur un espace not\'e~$\cM(\Z)$, le spectre de~$\Z$ au sens de Berkovich, repr\'esent\'e sur la figure~\ref{fig:MZintro}. La fibre au-dessus d'un point $p$-adique de ce spectre est un espace de Berkovich $p$-adique. La fibre au-dessus d'un point archim\'edien est presque un espace analytique complexe (en r\'ealit\'e, un espace analytique complexe quotient\'e par l'action de la conjugaison). De tels espaces pr\'esentent un int\'er\^et du point de vue arithm\'etique.

\begin{figure}[!h]
\centering
\begin{tikzpicture}
\foreach \x [count=\xi] in {-2,-1,...,17}
\draw (0,0) -- ({10*cos(\x*pi/10 r)/\xi},{10*sin(\x*pi/10 r)/\xi}) ;
\foreach \x [count=\xi] in {-2,-1,...,17}
\fill ({10*cos(\x*pi/10 r)/\xi},{10*sin(\x*pi/10 r)/\xi}) circle ({0.07/(sqrt(\xi)}) ;
\draw ({2.5*cos(-pi/5 r)},{2.5*sin(-pi/5 r)}) node[above right, rotate = -pi/5 r]{branche archim\'edienne} ;
\draw ({2.5*cos(-pi/10 r)},{2.5*sin(-pi/10 r)-.05}) node[above right, rotate = - pi/10 r]{branche 2-adique} ;
\draw (2.5,0) node[above right]{branche 3-adique} ;
\end{tikzpicture}
\caption{Le spectre de~$\Z$ au sens de Berkovich.}\label{fig:MZintro}
\end{figure}

Mentionnons un autre exemple d'anneau de Banach int\'eressant~: le corps~$\C^\hyb$, d\'efini comme le corps des nombres complexes~$\C$ muni de la norme dite hybride, \'egale au maximum entre la valeur absolue usuelle et la valeur absolue triviale. Les espaces de Berkovich sur~$\C^\hyb$ se pr\'esentent comme des familles d'espaces analytiques complexes (au-dessus du corps~$\C$ muni de la valeur absolue usuelle~$\va_{\infty}$ ou d'une de ses puissances $\va_{\infty}^\eps$ avec $\eps \in \intof{0,1}$) qui semblent d\'eg\'en\'erer sur un espace de Berkovich ultram\'etrique (au-dessus du corps~$\C$ muni de la valeur absolue triviale~$\va_{0}$), \cf~figure~\ref{fig:A1hybintro}. Ces espaces ont \'et\'e introduits par V.~Berkovich pour \'etudier la structure de Hodge mixte limite d'une famille d'espaces complexes dans~\cite{BerW0}. Ils ont, depuis, \'et\'e utilis\'es pour d'autres types de probl\`emes~: comportement asymptotique de formes volumes (\cf~\cite{BJ}), de mesures d'\'equillibre d'endomorphismes (\cf~\cite{FavreEndomorphisms}), de produits de matrices al\'eatoires dans $\textrm{SL}_{2}(\C)$ (\cf~\cite{DujardinFavreSL2C}), ou encore bornes uniformes de type Manin-Mumford en genre~2 (\cf~\cite{DKY}).

\begin{figure}[!h]
\centering
\begin{tikzpicture}
    \draw[thick] (-7,-3) -- (3,-3);
    \node at (-7,-3) {$|$};
    \node at (3,-3) {$|$};
    \node at (3,-3.5) {$\va_{\infty}$};
    \node at (-7,-3.5) {$\va_{0}$};
    \node at (-2,-3.5) {$\va_{\infty}^\eps$};
    \node at (-2,-3) {$|$};
    
    \draw (-2,0) circle (1.5cm);
    \draw[fill=white] (-2,1.5) circle (0.08cm);
    \node at (-2,1.75) {$\infty$};
    \draw[fill=black] (-2,-1.5) circle (0.08cm);
    \node at (-2,-1.75) {$0$};
    \draw (-.5,0) arc (315:225:2.13cm and 1.3cm);
    \draw[densely dotted] (-.5,0) arc (45:135:2.13cm and 1.3cm);
    
    \draw (3,0) circle (1.5cm);
    \draw[fill=white] (3,1.5) circle (0.08cm);
    \node at (3,1.75) {$\infty$};
    \draw[fill=black] (3,-1.5) circle (0.08cm);
    \node at (3,-1.75) {$0$};
    \draw (4.5,0) arc (315:225:2.13cm and 1.3cm);
    \draw[densely dotted] (4.5,0) arc (45:135:2.13cm and 1.3cm);
    
    \draw[thick] (-7,-1.5) -- (-7,1.5);
    \draw[fill=white] (-7,1.5) circle (0.08cm);
    \node at (-7,1.75) {$\infty$};
    \draw[fill=black] (-7,-1.5) circle (0.08cm);
    \node at (-7,-1.75) {$0$};
    \draw[thick] (-8.5,0) -- (-5.5,0);
    \draw[fill=black] (-8.5,0) circle (0.04cm);
    \draw[fill=black] (-5.5,0) circle (0.04cm);
    \draw[thick] (-8.06066,0.43934-1.5) -- (-5.93934,2.56066-1.5);
    \draw[fill=black] (-8.06066,0.43934-1.5) circle (0.04cm);
    \draw[fill=black] (-5.93934,2.56066-1.5) circle (0.04cm);
    \draw[thick] (-8.06066,1.06066) -- (-5.93934,-1.06066);
    \draw[fill=black] (-8.06066,1.06066) circle (0.04cm);
    \draw[fill=black] (-5.93934,-1.06066) circle (0.04cm);
    \draw (-5.61418,0.574025) -- (-8.38582,-0.574025);
    \draw[fill=black] (-5.61418,0.574025) circle (0.03cm);
    \draw[fill=black] (-8.38582,-0.574025) circle (0.03cm);
    \draw (-6.42597,1.38582) -- (-7.54403,-1.38582);
    \draw[fill=black] (-6.42597,1.38582) circle (0.03cm);
    \draw[fill=black] (-7.54403,-1.38582) circle (0.03cm);
    \draw (-7.57403,1.38582) -- (-6.42597,-1.38582);
    \draw[fill=black] (-7.57403,1.38582) circle (0.03cm);
    \draw[fill=black] (-6.42597,-1.38582) circle (0.03cm);
    \draw (-8.38582,0.574025) -- (-5.61418,-0.574025);
    \draw[fill=black] (-8.38582,0.574025) circle (0.03cm);
    \draw[fill=black] (-5.61418,-0.574025) circle (0.03cm);    
\end{tikzpicture}
\caption{La droite projective analytique sur $\C^\hyb$.}\label{fig:A1hybintro}
\end{figure}

\section{Objectifs du manuscrit}

Bien qu'elle ait \'et\'e formellement introduite d\`es le premier chapitre de~\cite{Ber1}, la th\'eorie des espaces de Berkovich sur un anneau de Banach n'a gu\`ere \'et\'e developp\'ee. Le second auteur leur a consacr\'e deux textes. Le premier, \cite{A1Z}, propose une \'etude d\'etaill\'ee de la droite de Berkovich sur~$\Z$ selon diff\'erents aspects (alg\'ebriques, topologiques, cohomologiques, etc.). Le second, \cite{EtudeLocale}, traite d'espaces arbitraires sur une certaine classe d'anneaux de Banach (contenant~$\Z$ et~$\C^\hyb$), mais uniquement d'un point de vue alg\'ebrique local (noeth\'erianit\'e des anneaux locaux, coh\'erence du faisceau structural, etc.).

L'objectif du pr\'esent manuscrit est de combler quelques lacunes de la th\'eorie des espaces de Berkovich sur un anneau de Banach~$\cA$ 
et d'\'etudier en profondeur certains aspects qui ont, jusqu'ici, \'et\'e laiss\'es de c\^ot\'e, en d\'epit de leur importance. Au cours du texte, nous devrons imposer \`a l'anneau de Banach~$\cA$ diverses propri\'et\'es techniques, variant de chapitre en chapitre, en fonction des besoins. 
Pour simplifier, nous d\'esignons par l'expression \og espace de Berkovich global \fg{} un espace de Berkovich sur un anneau de Banach satisfaisant les propri\'et\'es requises. Le lecteur sourcilleux peut choisir de la remplacer par \og espace de Berkovich sur~$\Z$ \fg{} ou \og espace de Berkovich sur~$\C^\hyb$\fg, la th\'eorie \'etant con\c cue de fa\c con \`a toujours s'appliquer au moins dans ces deux contextes.

Notre \'etude s'organise selon trois axes principaux.

\medbreak

\noindent\textbf{Axe \ding{172}~: D\'efinition de la cat\'egorie des espaces de Berkovich sur un anneau de Banach}\nopagebreak

V.~Berkovich a d\'efini les espaces de Berkovich sur un anneau de Banach d\`es les premi\`eres pages de son ouvrage fondateur~\cite{Ber1}, mais n'a pas introduit de notion de morphisme dans ce contexte. Le premier objectif de ce texte est de combler ce manque. Une fois cette t\^ache effectu\'ee, on dispose de la cat\'egorie des espaces de Berkovich sur un anneau de Banach et il devient possible de consid\'erer et d\'efinir dans un langage ad\'equat diff\'erentes op\'erations d'usage courant~: produits fibr\'es, extension des scalaires (de~$\Z$ \`a un anneau d'entiers de corps de nombres, par exemple), analytification de sch\'emas, etc.

Cette \'etude fait l'objet des chapitres~\ref{def_cat} (d\'efinitions et propri\'et\'es g\'en\'erales dans le cadre des espaces de Berkovich sur un anneau de Banach), \ref{analyse}~(compl\'ements techniques sur des normes d'anneaux de Banach) et~\ref{catan} (propri\'et\'es plus fines dans le cadre des espaces de Berkovich sur une bonne classe d'anneaux de Banach, contenant notamment~$\Z$ et~$\C^{\hyb}$).

Nous attirons l'attention du lecteur sur le fait que, lorsque l'anneau de Banach considéré est un corps valu\'e ultram\'etrique complet~$k$, la d\'efinition propos\'ee ici ne permet pas de retrouver tous les espaces $k$-analytiques d\'efinis par V.~Berkovich dans~\cite{Ber1,Ber2}, mais seulement les espaces sans bords.

\medbreak

\noindent\textbf{Axe \ding{173}~: \'Etude de la topologie des espaces de Berkovich globaux}\nopagebreak

Les propri\'et\'es topologiques des espaces de Berkovich forment l'un des points saillants de la th\'eorie et rec\`elent nombre d'informations subtiles. Dans~\cite{FirstSteps}, V.~Berkovich explique s'\^etre rendu compte d\`es les premiers instants que la topologie de l'analytifi\'ee d'une courbe elliptique sur~$\C_{p}$ (qui est contractile ou homotope \`a un cercle) permettait de retrouver son type de r\'eduction (bonne ou mauvaise, respectivement). Citons \'egalement le travail~\cite{beth} d'A.~Thuillier, qui identifie le type d'homotopie du complexe d'intersection d'un diviseur \`a croisements normaux compactifiant une vari\'et\'e sur un corps parfait \`a celui d'un espace de Berkovich canoniquement associ\'e \`a la vari\'et\'e (montrant ainsi que ce type d'homotopie ne d\'epend pas du diviseur choisi). 

En d\'epit de cet int\'er\^et \'evident, les consid\'erations topologiques sont presque totalement absentes de l'\'etude des espaces de Berkovich sur un anneau de Banach. Seuls quelques r\'esultats existent, limit\'es, pour l'essentiel, au cas de la droite affine analytique sur~$\Z$ (connexit\'e par arcs locale dans~\cite{A1Z}, contractibilit\'e globale dans~\cite{LS}).

Dans ce texte, nous d\'eveloppons les outils n\'ecessaires \`a une \'etude topologique syst\'ematique des espaces de Berkovich globaux 
et obtenons de premiers r\'esultats g\'en\'eraux concrets, telle la connexit\'e par arcs locale. 

Cette \'etude fait l'objet des chapitres~\ref{chap:fini} (pr\'eliminaires sur les morphismes finis), \ref{chap:structurelocale} (structure locale des espaces) et~\ref{chap:topo} (connexit\'e par arcs locale et dimension topologique).

\medbreak

\noindent\textbf{Axe~\ding{174}~: \'Etude de la cohomologie des espaces de Berkovich globaux}\nopagebreak

Comme dans le domaine de la topologie, les aspects cohomologiques des espaces de Berkovich sur un anneau de Banach n'ont \'et\'e abord\'es que dans le cadre de la droite affine analytique sur~$\Z$, dans~\cite{A1Z}. Pourtant, les applications de la g\'eom\'etrie analytique, tant complexe que $p$-adique, interviennent souvent par le biais de calculs de dimension d'espaces de cohomologie ou de r\'esultats d'annulation. 

Nous initions ici l'\'etude de la cohomologie coh\'erente des espaces de Berkovich globaux en dimension sup\'erieure. La premi\`ere \'etape, indispensable \`a tout calcul, est de disposer d'une vaste classe d'espaces dont les groupes de cohomologie coh\'erente sup\'erieurs soient nuls. On peut penser \`a ces espaces comme \`a des analogues 
des espaces affino\"ides en g\'eom\'etrie analytique rigide ou des espaces de Stein en g\'eom\'etrie analytique complexe.

Cette \'etude fait l'objet du chapitre~\ref{chap:Stein} (et utilise de fa\c{c}on essentielle les r\'esultats techniques du chapitre~\ref{analyse}).

\section{Organisation du texte}

Exposons maintenant plus en d\'etails, chapitre par chapitre, le contenu de ce texte. Nous mettons en valeur des r\'esultats qui nous semblent importants, tout en en laissant quelques autres de c\^ot\'e. L'introduction de chacun des chapitres contient une description exhaustive de son contenu.

Le but principal de ce texte est de d\'evelopper la th\'eorie des espaces de Berkovich d\'efinis sur~$\Z$ ou d'autres anneaux de Banach aux propri\'et\'es similaires~: anneaux d'entiers de corps de nombres, corps hybrides comme~$\C^\hyb$, anneaux de valuation discr\`ete, etc. 
Nous ne sommes pas parvenus \`a isoler un ensemble de conditions naturelles \`a imposer \`a un anneau de Banach pour qu'il entre dans le cadre de notre th\'eorie. Les diff\'erents chapitres de ce texte contiennent donc diff\'erentes d\'efinitions (anneau de base, anneau de base g\'eom\'etrique, anneau de Dedekind analytique, etc.), adapt\'es \`a nos besoins sp\'ecifiques. Ce choix pr\'esente l'avantage de mettre en lumi\`ere les propri\'et\'es utilis\'ees pour d\'emontrer chacun des r\'esultats, et nous esp\'erons \'egalement qu'il permettra de prendre en compte plus facilement les d\'eveloppements ult\'erieurs de la th\'eorie (qui pourraient faire intervenir d'autres anneaux de Banach, dont l'id\'ee nous \'echappe \`a l'heure o\`u nous \'ecrivons). 

Insistons sur le fait que tous les r\'esultats que nous d\'emontrons sont valables pour des espaces de Berkovich d\'efinis sur un anneau de Banach qui est l'un de ceux cit\'es en exemple plus haut. Seule l'ultime section~\ref{sec:noetherianite}, sp\'ecifique aux anneaux d'entiers de corps de nombres, fait exception \`a cette r\`egle. Comme pr\'ec\'edemment, nous utiliserons l'expression \og espace de Berkovich global \fg{} pour d\'esigner un espace de Berkovich sur un anneau de Banach satisfaisant les propri\'et\'es requises.

\medbreak

\noindent\textbf{Chapitre~\ref{chap:rappels}~: Pr\'eliminaires et rappels}\nopagebreak

Soit $\cA$ un anneau de Banach. Dans ce chapitre, nous commen\c{c}ons par rappeler en d\'etail la construction des espaces $\cA$-analytiques, en suivant~\cite{Ber1}. Nous rappelons ensuite les principaux r\'esultats obtenus par le second auteur dans ses travaux~\cite{A1Z, EtudeLocale}. Nous en profitons pour d\'emontrer quelques r\'esultats techniques (par exemple sur les bases de voisinages des points, \cf~propositions~\ref{prop:basevoisdim1rigide} et~\ref{prop:basevoisdim1}) et introduire quelques outils dont nous nous servirons dans la suite du texte.

Un r\'esultat majeur est le th\'eor\`eme de division de Weierstra\ss{} (\cf~\cite[th\'eor\`eme~8.3]{EtudeLocale}),
dont nous rappelons maintenant l'\'enonc\'e. Consid\'erons la droite affine analytique $X := \E{1}{\cA}$ (avec coordonn\'ee~$T$) et la projection $\pi \colon X \to \cM(\cA)$.  
Soient~$b$ un point de~$\cM(\cA)$ et~$x$ un point de~$\pi^{-1}(b)$. Nous supposerons que~$x$ est rigide dans la fibre $\pi^{-1}(b) \simeq \E{1}{\cH(b)}$, o\`u $\cH(b)$ d\'esigne le corps r\'esiduel compl\'et\'e du point~$b$. Cela signifie qu'il existe un polyn\^ome irr\'eductible $\mu_{x} \in \cH(b)[T]$ qui s'annule exactement en ce point.

\begin{enonce*}{Th\'eor\`eme \ref{weierstrass}}\index{Theoreme de Weierstrass@Th\'eor\`eme de Weierstra\ss!division}
Pla\c{c}ons-nous dans le cadre d\'ecrit ci-dessus. Soit~$G$ un \'el\'ement de l'anneau local~$\cO_{X,x}$. Supposons que son image dans l'anneau de valuation discr\`ete~$\cO_{\pi^{-1}(b),x}$ n'est pas nulle et notons~$n$ sa valuation.

Alors, pour tout~$F\in\cO_{X,x}$, il existe un unique couple~$(Q,R)\in \cO_{X,x}^2$ tel que 
\begin{enumerate}[i)]
\item $F=QG+R$ ;
\item $R\in\cO_{\cM(\cA),b}[T]$ est un polyn\^ome de degr\'e strictement inf\'erieur \`a~$n\deg(\mu_{x})$.
\end{enumerate}
\end{enonce*}

Dans un souci de pr\'ecision, signalons que l'\'enonc\'e du th\'eor\`eme requiert des hypoth\`eses sur le point~$b$ (par exemple de supposer que $b$ est d\'ecent, \cf~d\'efinition~\ref{def:decent}), qui sont toujours satisfaites si l'anneau~$\cA$ est~$\Z$ ou~$\C^\hyb$, par exemple. 

Dans~\cite{EtudeLocale}, on tire plusieurs cons\'equences de ce th\'eor\`eme dans le cadre des espaces de Berkovich globaux~: noeth\'erianit\'e et excellence des anneaux locaux~$\cO_{x}$, coh\'erence du faisceau strucural, etc. Elles sont rappel\'ees \`a la fin du chapitre.

\medbreak

\noindent\textbf{Chapitre~\ref{def_cat}~: Cat\'egorie des espaces analytiques~: d\'efinitions}\nopagebreak

Soit $\cA$ un anneau de Banach. L'objet principal de ce chapitre est de d\'efinir la notion de morphisme entre deux espaces $\cA$-analytiques. Rappelons que la d\'efinition d'espace analytique sur~$\cA$, comme la d\'efinition d'espace analytique complexe, s'effectue en plusieurs \'etapes~: on commence par les espaces affines analytiques, on consid\`ere ensuite les mod\`eles locaux, qui sont des ferm\'es analytiques d'ouverts des pr\'ec\'edents, puis on recolle ces derniers.\footnote{Si $\cA=k$ est un corps valu\'e ultram\'etrique complet, cette d\'efinition ne permet pas de retrouver tous les espaces $k$-analytiques d\'efinis par V.~Berkovich dans~\cite{Ber1, Ber2}, mais seulement les espaces sans bords.} La d\'efinition de morphisme proc\`ede de la m\^eme logique.

Soient $U$ et $V$ des ouverts d'espaces affines analytiques sur~$\cA$ et $\varphi \colon U \to V$ un morphisme d'espaces annel\'es. Pour tout compact~$U'$ de~$U$ et tout compact~$V'$ de~$V$ contenant~$\varphi(U')$, $\varphi^\sharp$ induit un morphisme 
\[\varphi_{U',V'}^\sharp \colon \cO_{V}(V') \too \cO_{U}(U')\] 
entre espaces norm\'es (munis des normes uniformes sur~$V'$ et~$U'$ respectivement). Nous dirons que $\varphi\colon U \to V$ est un \emph{morphisme analytique} de~$U$ dans~$V$ lorsque tous les morphismes $\varphi_{U',V'}^\sharp$ pr\'ec\'edents sont contractants (\cf~d\'efinition~\ref{def:morphismefermeouvert}).

Soient~$X$ et~$Y$ des ferm\'es analytiques d'ouverts d'espaces affines analytiques sur~$\cA$. Un \emph{morphisme analytique} $\varphi \colon X \to Y$ est un morphisme d'espaces localement annel\'es qui se rel\`eve localement en un morphisme analytique d'ouverts d'espaces affines analytiques au sens pr\'ec\'edent (\cf~d\'efinition~\ref{def:morphismefermeouvert}).

La d\'efinition g\'en\'erale s'obtient par recollement, ce que nous formulerons pr\'ecis\'ement en termes d'atlas (\cf~d\'efinition~\ref{def:morphismegeneral}). Nous obtenons ainsi la \emph{cat\'egorie des espaces $\cA$-analytiques}. Nous la notons~$\cAAn$.

Nous d\'efinissons \'egalement une seconde cat\'egorie, la \emph{cat\'egorie des espaces analytiques au-dessus de~$\cA$}, not\'ee~$\An_{\cA}$. Ses objets sont des couples form\'es d'un morphisme born\'e d'anneaux de Banach $\cA \to \cB$ et d'un espace $\cB$-analytique (\cf~d\'efinition~\ref{def:espaceaudessusdeA}). Ses morphismes sont d\'efinis de fa\c{c}on analogue \`a ceux de~$\cAAn$ (\cf~d\'efinitions~\ref{defi:morphismeaudessusdefouvertaffine}, \ref{defi:morphismeaudessusdeffermeouvert} et~\ref{defi:morphismeaudessusdef}). Dans la suite du texte, cette cat\'egorie nous sera utile pour consid\'erer des changements d'anneau de base.

Insistons sur le fait que, dans ce chapitre, il n'est besoin d'aucune hypoth\`ese sur l'anneau de Banach~$\cA$. Nous n'en imposerons que plus tard, lorsque nous \'etudierons la cat\'egorie~$\cAAn$.

\medbreak

\noindent\textbf{Chapitre~\ref{analyse}~: Quelques r\'esultats topologiques sur les anneaux de fonctions analytiques}\nopagebreak

Ce chapitre est probablement le plus technique du manuscrit. Il d\'ebute par des consid\'erations sur certains anneaux de Banach qui se pr\'esentent comme des quotients et des r\'esultats de comparaison de normes (norme r\'esiduelle, norme uniforme, etc.) sur ceux-ci. La cons\'equence principale en est une \emph{version norm\'ee du th\'eor\`eme de division de Weierstra\ss{}} rappel\'e plus haut. Le diviseur~$G$ \'etant fix\'e, elle permet d'obtenir un contr\^ole de la norme du quotient~$Q$ et du reste~$R$ en fonction de la norme du dividende~$F$ (\cf~th\'eor\`eme~\ref{weierstrassam})

Ce th\'eor\`eme est utilis\'e, \`a la fin du chapitre, pour d\'emontrer un r\'esultat de fermeture des id\'eaux du faisceau structural. Au chapitre~\ref{chap:Stein}, nous le g\'en\'eraliserons en combinant la premi\`ere version avec la th\'eorie des espaces de Stein. Nous pr\'esentons ici directement la forme forte, dont l'\'enonc\'e est plus accessible. Elle est valable sous certaines hypoth\`eses sur l'anneau de Banach~$\cA$. Nous ne les pr\'ecisons pas ici, et nous contentons d'indiquer qu'elles sont satisfaites dans le cas de~$\Z$ ou~$\C^\hyb$. 



\begin{enonce*}{Corollaire~\ref{cor:limiteOU}}\index{Ideal@Id\'eal!ferme@ferm\'e}
Soient~$x$ un point de~$\E{n}{\cA}$ et $I$~un id\'eal de~$\cO_{x}$. Soient~$U$ un voisinage de~$x$ dans~$\E{n}{\cA}$ et $(g_n)_{n\in\N}$ une suite d'\'el\'ements de~$\cO(U)$ qui converge vers un \'el\'ement~$g$ de~$\cO(U)$. Si, pour tout $n\in \N$, l'image de~$g_{n}$ dans~$\cO_{x}$ appartient \`a~$I$, alors l'image de~$g$ dans~$\cO_{x}$ appartient \`a~$I$. 
\end{enonce*}



\medbreak

\noindent\textbf{Chapitre~\ref{catan}~: Cat\'egorie des espaces analytiques~: propri\'et\'es}\nopagebreak

Ce chapitre prolonge l'\'etude de la cat\'egorie des espaces $\cA$-analytiques initi\'ee au chapitre~\ref{def_cat}. En imposant des conditions suppl\'ementaires sur l'anneau~$\cA$, autrement dit en se restreignant \`a une certaine classe d'espaces analytiques globaux, on peut d\'emontrer plusieurs propri\'et\'es attendues.

Notre premier r\'esultat est l'analogue d'un r\'esultat classique, d'apparence anodine, mais fondamental. Il affirme l'existence d'une bijection naturelle entre les fonctions globales sur un espace et les morphismes de cet espace vers la droite affine analytique. 

\begin{enonce*}{Proposition~\ref{morphsec}}\index{Morphisme analytique!vers un espace affine}
Soit~$X$ un espace $\cA$-analytique. L'application 
\[\fonctionsp{\Hom_{\cAAn}(X,\E{1}{\cA})}{\Gamma(X,\cO_{X})}{\varphi}{\varphi^\sharp(T)},\]
o\`u $T$ est la coordonn\'ee sur~$\E{1}{\cA}$, est bijective.
\end{enonce*}

En termes vagues, la surjectivit\'e signifie qu'une fonction~$F$ sur~$X$ \'etant donn\'ee (correspondant \`a $\varphi^\sharp(T)$), les s\'eries convergentes en~$F$ font sens sur~$X$ (la s\'erie $\sum a_{i} F^i$ correspondant \`a $\varphi^\sharp(\sum a_{i} T^i)$). L'injectivit\'e signifie que ces s\'eries sont des fonctions sur~$X$ bien d\'efinies. L'espace~$X$ \'etant localement un ferm\'e analytique (d\'efini par un faisceau d'id\'eaux coh\'erent) d'un ouvert d'un espace affine analytique, on con\c{c}oit que le th\'eor\`eme de fermeture des id\'eaux \'enonc\'e plus haut (\cf~corollaire~\ref{limite}) joue un r\^ole capital dans sa d\'emonstration.

Ce r\'esultat, et sa g\'en\'eralisation \`a un nombre fini de fonctions globales, permet de montrer que plusieurs foncteurs naturels sont repr\'esentables et ainsi de d\'efinir, de fa\c{c}on ad\'equate, l'analytifi\'e d'un sch\'ema localement de type fini sur~$\cA$ (\cf~th\'eor\`eme~\ref{thm:analytification}), l'extension des scalaires d'un espace analytique sur~$\cA$ \`a un anneau d'entiers de corps de nombres (\cf~th\'eor\`eme~\ref{thm:extensionB}), ou encore les produits fibr\'es d'espaces analytiques sur~$\cA$ (\cf~th\'eor\`eme~\ref{produit_fibr\'e}).

\medbreak

\noindent\textbf{Chapitre~\ref{chap:fini}~: \'Etude des morphismes finis}\nopagebreak

Dans ce chapitre, on d\'efinit et \'etudie les morphismes finis d'espaces analytiques. Il s'agit d'une notion essentielle, puisque c'est par le biais des morphismes finis que nous pourrons passer des propri\'et\'es des espaces affines analytiques \`a celles des espaces analytiques g\'en\'eraux.

Nous d\'emontrons les analogues de plusieurs r\'esultats classiques sur les morphismes finis. En voici un embl\'ematique.

\begin{enonce*}{Th\'eor\`eme \ref{thm:fini}}
Soient~$\varphi \colon X\to Y$ un morphisme fini d'espaces analytiques globaux et~$\cF$ un faisceau coh\'erent sur~$X$. Alors, le faisceau~$\varphi_*\cF$ est coh\'erent.
\end{enonce*}

Ce th\'eor\`eme est l'un des ingr\'edients essentiels de la preuve du Nullstellensatz, que nous d\'emontrons sous la forme suivante.

\begin{enonce*}{Corollaire \ref{cor:IVJ}}
Soit $X$ un espace analytique global. Soit~$\cJ$ un faisceau d'id\'eaux coh\'erent sur~$X$. Notons $V(\cJ)$ le lieu des z\'eros de~$\cJ$ et $\cI(V(\cJ))$ le faisceau d'id\'eaux des fonctions s'annulant en tout point de~$V(\cJ)$. Alors, on a
\[\cI(V(\cJ)) = \sqrt{\cJ}.\]
\end{enonce*}

\medbreak

\noindent\textbf{Chapitre~\ref{chap:structurelocale}~: Structure locale des espaces analytiques}\nopagebreak

Dans ce chapitre, nous \'etudions la structure locale des espaces analytiques. Une version du lemme de normalisation de Noether assure que tout espace int\`egre se pr\'esente localement comme un rev\^etement fini (ramifi\'e) d'un espace affine.

\begin{enonce*}{Th\'eor\`eme~\ref{proj}}
Soit~$X$ un espace analytique global. Notons $\pi \colon X \to \cM(\cA)$ le morphisme structural. Soit~$x$ un point de~$X$ en lequel~$\cO_{X,x}$ est int\`egre. Alors, il existe un voisinage ouvert~$U$ de~$x$ dans~$X$, un entier $n\in \N$, un ouvert~$V$ de l'espace affine analytique~$\E{n}{\cA}$ ou~$\E{n}{\cH(\pi(x))}$ et un morphisme fini ouvert $\varphi \colon U\to V$.
\end{enonce*}

D'autre part, la proposition~\ref{decomp} assure que tout espace analytique global peut s'\'ecrire, localement au voisinage d'un point, comme union de ferm\'es analytiques int\`egres en ce point. La combinaison de ces deux r\'esultats fournit donc un r\'esultat de structure locale pr\'ecis pour n'importe quel espace.

\smallbreak

Dans la suite du chapitre, nous utilisons cette description locale afin de comparer les propri\'et\'es d'un sch\'ema \`a celle de son analytifi\'e. Le r\'esultat le plus important est le suivant.

\begin{enonce*}{Th\'eor\`eme \ref{platitude_analytification}}
Soit~$\cX$ un sch\'ema localement de type fini sur~$\cA$. Notons~$\cX^\an$ son analytifi\'e. Alors, le morphisme canonique $\rho_{\cX} \colon \cX^\an \to \cX$  est plat.
\end{enonce*}

L'\'enonc\'e analogue en g\'eom\'etrie analytique complexe ou ultram\'etrique est fondamental. Il intervient notamment dans la preuve des th\'eor\`emes GAGA (\cf~\cite{GAGA}). 

\medbreak

\noindent\textbf{Chapitre~\ref{chap:topo}~: Propri\'et\'es topologiques des espaces analytiques}\nopagebreak

Dans ce chapitre, nous \'etudions finalement la topologie des espaces analytiques. Le r\'esultat principal est le suivant.

\begin{enonce*}{Th\'eor\`eme \ref{th:cpageneral}}
Tout espace analytique global est localement connexe par arcs.
\end{enonce*}

Comme on peut s'y attendre, nous d\'emontrons tout d'abord le r\'esultat pour les espaces affines analytiques. Nous en d\'eduisons le cas g\'en\'eral gr\^ace aux r\'esultats de structure locale du chapitre~\ref{chap:structurelocale}. Cette seconde \'etape n'est pas imm\'ediate car un rev\^etement fini d'un espace connexe par arcs ne le reste pas n\'ecessairement (\cf~remarque~\ref{rem:revetementdegre2}). Cette difficult\'e nous conduit \`a introduire une nouvelle notion de topologie g\'en\'erale, appel\'ee \emph{\'elasticit\'e} (\cf~d\'efinition~\ref{def:elastique}), raffinant la connexit\'e par arcs.

\medbreak

\noindent\textbf{Chapitre~\ref{chap:Stein}~: Espaces de Stein}\nopagebreak

Dans le dernier chapitre de cet ouvrage, nous initions la th\'eorie des espaces de Stein dans le cadre des espaces de Berkovich globaux. 
Notre r\'esultat principal concerne les polydisques ferm\'es sur un anneau de Banach~$\cA$ ad\'equat, par exemple~$\Z$ ou~$\C^\hyb$.

\begin{enonce*}{Corollaires \ref{cor:thA} et~\ref{cor:thB}}
Soient $r_{1},\dotsc,r_{n} \in \R_{>0}$ et~$\cF$ un faisceau coh\'erent sur le polydisque relatif $\oD := \overline{D}(r_{1},\dotsc,r_{n})$ sur~$\cA$. Alors,
\begin{enumerate}[i)]
\item pour tout entier $q\ge 1$, on a $H^q(\oD,\cF) = 0$~;
\item le faisceau~$\cF$ est engendr\'e par l'ensemble de ses sections globales~$\cF(\oD)$.
\end{enumerate}
\end{enonce*}

Nous pouvons ensuite jeter les bases d'une th\'eorie des espaces affino\"ides surconvergents (\cf~d\'efinition~\ref{def:affinoide}) poss\'edant des propri\'et\'es similaires \`a celle de la th\'eorie ultram\'etrique classique. On peut y d\'efinir des domaines rationnels, de Laurent, de Weierstra\ss{} (\cf~d\'efinition~\ref{def:domaines}) et du r\'esultat pr\'ec\'edent d\'ecoule un analogue du th\'eor\`eme d'acyclicit\'e de Tate et du th\'eor\`eme de Kiehl (\cf~th\'eor\`eme~\ref{th:affinoideAB}). En particulier, tout point d'un espace de Berkovich global poss\`ede une base de voisinages form\'e d'affino\"ides surconvergents, en particulier d'espaces dont la cohomologie coh\'erente sup\'erieure s'annule. 

Dans la section~\ref{sec:Bouvert}, nous adaptons ces r\'esultats (sans la propri\'et\'e de g\'en\'eration globale du faisceau) dans un cadre ouvert.

\smallbreak

Nous terminons avec une application des r\'esultats d'annulation cohomologique \`a la noeth\'erianit\'e d'anneaux de \emph{s\'eries arithm\'etiques convergentes}. Soit $n\in \N$. Pour $r_{1},\dotsc,r_{n} \in \intoo{0,1}$, notons 
\[\Z\llbracket T_{1},\dotsc,T_{n} \rrbracket_{> (r_{1},\dotsc,r_{n})}\] 
le sous-anneau de $\Z\llbracket T_{1},\dotsc,T_{n} \rrbracket$ form\'e des s\'eries qui convergent au voisinage du polydisque complexe $\oD(r_{1},\dotsc,r_{n})$. D.~Harbater a montr\'e dans~\cite{HarbaterConvergent} que ces anneaux sont noeth\'eriens lorsque $n=1$. Nous tirons parti des r\'esultats cohomologiques obtenus dans ce chapitre pour g\'en\'eraliser ce r\'esultat.

\begin{enonce*}{Corollaire \ref{cor:noetherienconcret}}
Pour tout $n\in \N$ et tous $r_{1},\dotsc,r_{n} \in \intoo{0,1}$, l'anneau $\Z\llbracket T_{1},\dotsc,T_{n} \rrbracket_{> (r_{1},\dotsc,r_{n})}$ 
est noeth\'erien.
\end{enonce*}

Nous obtenons en r\'ealit\'e une version raffin\'ee de ce r\'esultat qui autorise des coefficients dans $\Z[1/N]$, pour $N\in \N_{\ge 1}$, et fait intervenir des rayons de convergence aux places $p$-adiques, pour $p$ divisant~$N$. En outre, l'anneau~$\Z$ peut \^etre remplac\'e par n'importe quel anneau d'entiers de corps de nombres.

\section{Applications}

Que le lecteur ne s'y trompe pas, nous proposons ici un texte de fondement sur la th\'eorie des espaces de Berkovich globaux, au contenu parfois aride et technique. En d\'emontrant les r\'esultats annonc\'es, nous franchissons une \'etape ingrate, mais n\'ecessaire, pour \'etablir la th\'eorie sur des bases solides et la rendre utilisable, en l'\'etat, dans divers domaines des math\'ematiques.


Comme indiqu\'e plus haut, nous proposons, \`a la fin de ce texte, une premi\`ere application de la th\'eorie \`a des questions de noeth\'erianit\'e d'anneaux de s\'eries arithm\'etiques convergentes en plusieurs variables (\cf~corollaire~ \ref{cor:noetherienconcret}). Il s'agit d'un r\'esultat original, dont l'\'enonc\'e est purement alg\'ebrique, et que nous d\'emontrons ici par des m\'ethodes \'elabor\'ees, reposant de fa\c con essentielle sur la th\'eorie des espaces de Stein analytiques globaux d\'evelopp\'ee au chapitre~\ref{chap:Stein}.

La th\'eorie des espaces de Berkovich globaux se pr\^ete \`a bien d'autres applications dans des domaines vari\'es. Nous en d\'etaillons deux, que nous jugeons particuli\`erement embl\'ematiques.

\medbreak

\noindent\textbf{Dynamique analytique sur~$\Z$}\nopagebreak

La th\'eorie des syst\`emes dynamiques est bien \'etablie dans le cas complexe et, depuis une vingtaine d'ann\'ees, on assiste \`a l'apparition d'un analogue $p$-adique, sous la plume de nombreux auteurs. Nous renvoyons aux textes de survol \cite{CantatChambertLoirDynamique} et \cite{DucrosBourbaki} pour des pr\'ecisions et des r\'ef\'erences.

La question d'une th\'eorie globale unifiant ces th\'eories locales se pose d\'esormais et, dans l'article~\cite{DynamiqueI}, le second auteur en d\'eveloppe les premiers \'el\'ements. Les r\'esultats pr\'esent\'es dans ce manuscrit y jouent un r\^ole fondamental, en particulier ceux concernant les morphismes finis et plats (\cf~chapitre~\ref{chap:fini}), n\'ecessaires pour d\'efinir une notion de degr\'e, et la connexit\'e par arcs locale des espaces (\cf~chapitre~\ref{chap:topo}). Ils permettent de d\'emontrer que l'on dispose, dans le cadre des espaces de Berkovich globaux, d'une bonne th\'eorie des mesures de Radon, o\`u l'on peut, par exemple, d\'efinir les images directe et r\'eciproque par un morphisme fini et plat. 

Le r\'esultat principal de~\cite{DynamiqueI} \'enonce une propri\'et\'e de continuit\'e pour certaines mesures de Radon associ\'ees \`a des syst\`emes dynamiques. Son originalit\'e r\'eside dans le fait que, dans le cadre des espaces de Berkovich globaux, les mesures en question peuvent \^etre support\'ees par des espaces de nature vari\'ee, complexe ou $p$-adique. 

Rappelons que, pour tout corps valu\'e complet~$k$ et toute fraction rationnelle non constante $\psi \in k(T)$, on peut d\'efinir une mesure de Radon~$\mu_{\psi}$ sur l'espace de Berkovich~$\EP{1}{k}$ (qui n'est autre que~$\P^1(\C)$ si $k=\C$), dite mesure d'\'equilibre, essentiellement caract\'eris\'ee par la propri\'et\'e $\psi^* \mu_{\psi} = \deg(\psi) \mu_{\psi}$. 

\begin{enonce*}{Th\'eor\`eme}[\protect{\cite[th\'eor\`eme~A]{DynamiqueI}}]
Soit $Y$ un espace $\cA$-analytique global. Notons $X := Y \times_{\cA} \EP{1}{\cA}$ la droite projective relative au-dessus de~$Y$. Soit~$\varphi$ un endomorphisme polaris\'e de~$X$ au-dessus de~$Y$. 

Pour tout point~$y$ de~$Y$, l'endomorphisme~$\varphi$ induit un endomorphisme~$\varphi_{y}$ de la fibre $X_{y} \simeq \EP{1}{\cH(y)}$. Notons $\mu_{\varphi_{y}}$ la mesure d'\'equilibre associ\'ee, identifi\'ee \`a son image sur~$X$. 

Alors la famille de mesures $(\mu_{\varphi_{y}})_{y\in Y}$ sur~$X$ est continue.
\end{enonce*}

Ce r\'esultat g\'en\'eralise un th\'eor\`eme d\'emontr\'e par C.~Favre dans~\cite{FavreEndomorphisms}, qui traite du cas particulier o\`u la base~$Y$ est un disque hybride (mais autorise une dimension relative arbitraire).


%
%
%

\medbreak

Dans~\cite{DynamiqueII}, le second auteur explore des cons\'equences arithm\'etiques du th\'eor\`eme de continuit\'e pr\'ec\'edent, en \'etudiant des points de torsion de courbes elliptiques, selon une strat\'egie emprunt\'ee \`a~\cite{DKY}. Le syst\`eme dynamique permettant d'\'etablir le lien est celui associ\'e \`a un morphisme de Latt\`es provenant de la multiplication par~2 sur une courbe elliptique, c'est-\`a-dire un morphisme~$L$ faisant commuter le diagramme 
\[\begin{tikzcd}
E \ar[r, "{[}2{]}"] \ar[d, "\pi"]& E \ar[d, "\pi"]\\
\P^1 \ar[r, "L"] & \P^1
\end{tikzcd},\]
o\`u $E$ est une courbe elliptique et $\pi$ le quotient par $\{\pm1\}$. Le r\'esultat principal de l'article~\cite{DynamiqueII} est le suivant. Il r\'esout par l'affirmative une conjecture due \`a F.~Bogomolov, H.~Fu et Yu.~Tschinkel. Pour toute courbe elliptique~$E$, on note $E[\infty]$ (resp. $E[2]$) l'ensemble de ses points de torsion (resp. de 2-torsion).

\begin{enonce*}{Th\'eor\`eme}[\protect{\cite[th\'eor\`eme~A]{DynamiqueII}}]
Il existe $M \in \R_{>0}$ telle que, pour toutes courbes elliptiques~$E_{a}$ et~$E_{b}$ sur~$\bar\Q$ et tous rev\^etements doubles $\pi_{a} \colon E_{a}\to \P^1_{\bar\Q}$ et $\pi_{b} \colon E_{b}\to \P^1_{\bar\Q}$ tels que $\pi_{a}(E_{a}[2]) \ne \pi_{b}(E_{b}[2])$, on ait
\[ \sharp \big( \pi_{a}(E_{a}[\infty]) \cap \pi_{b}(E_{b}[\infty])\big) \le M.  \]
\end{enonce*}

L'apport des espaces de Berkovich sur~$\Z$ dans la preuve du th\'eor\`eme est crucial en ce qu'il permet, par des r\'esultats de continuit\'e tel que celui de~\cite{DynamiqueI}, de ramener l'\'etude du probl\`eme global \`a une \'etude, plus simple, au-dessus du point central de~$\cM(\Z)$. 

\medbreak

\noindent\textbf{Compactifications ultram\'etriques d'espaces complexes}\nopagebreak

Pour terminer cette introduction, pr\'esentons une autre application de la th\'eorie des espaces de Berkovich globaux, cette fois-ci, dans le cas de l'anneau de base~$\C^\hyb$. Elle concerne des questions de compactification de vari\'et\'es analytiques complexes. Il s'agit d'un travail du second auteur en cours de r\'edaction (\cf~\cite{CompactificationsOberwolfach} pour une note introductive).

Soit~$X$ une vari\'et\'e alg\'ebrique complexe, c'est-\`a-dire un sch\'ema s\'epar\'e de type fini sur~$\C$, et soit~$X^\h$ son analytification au sens de la g\'eom\'etrie analytique complexe. Il existe souvent de nombreuses fa\c cons de compactifier~$X^\h$ et certaines propri\'et\'es int\'eressantes de~$X$, telles que la lissit\'e, peuvent dispara\^itre au cours du processus. Les r\'esultats que nous exposons ici permettent de rem\'edier \`a ces d\'efauts, pour peu que l'on s'autorise \`a quitter le monde des vari\'et\'es complexes pour celui des espaces hybrides. 

L'ingr\'edient principal de la construction est l'espace de Berkovich~$X^{\infty}$ introduit par A.~Thuillier dans~\cite{beth}. 
C'est une partie ferm\'ee de l'analytification \`a la Berkovich~$X^{\an}_{\C^0}$ de~$X$ sur~$\C^{0}$, le corps~$\C$ muni de la valeur absolue triviale. A.~Thuillier l'a utilis\'ee de fa\c con spectaculaire pour \'etudier le type d'homotopie du complexe dual associ\'e \`a une compactification \`a croisements normaux de~$X$ (et notamment d\'emontrer qu'il ne d\'epend pas de la compactification choisie).

En utilisant l'analytification~$X^\an_{\C^\hyb}$ de~$X$ sur~$\C^\hyb$, l'espace~$X^\infty$, puis une proc\'edure de normalisation, inspir\'ee par le travail~\cite{TXZ} de L.~Fantini sur les liens ultram\'etriques de singularit\'es complexes, on construit un espace localement annel\'e~$X\urc$ dans lequel s'envoie naturellement~$X^\h$. L'espace~$X\urc$ poss\`ede localement une structure (non canonique) d'espace de Berkovich sur~$\C^\hyb$. 

\begin{enonce*}{Th\'eor\`eme}
Pour toute vari\'et\'e alg\'ebrique complexe~$X$, l'espace $X\urc$ est compact et l'application naturelle
\[X^\h \longrightarrow X\urc\] 
est une immersion ouverte. Le compl\'ementaire de son image s'identifie \`a un quotient de~$X^\infty$.

L'application $X \mapstoo X\urc$ est fonctorielle vis-\`a-vis des morphismes propres.
\end{enonce*}

Remarquons que, $X\urc$ \'etant construit \`a partir d'une analytification de~$X$, il en h\'erite des propri\'et\'es~: lissit\'e\footnote{Nous renvoyons au travail \cite{BergerEtale} de D.~Berger pour la d\'efinition de morphisme lisse dans le cadre des espaces de Berkovich globaux.}, interpr\'etation modulaire s'il en existe, etc.

\medbreak

Les fonctions et faisceaux sur~$X$ et sur $X\urc$ sont tr\`es proches, une propri\'et\'e qui peut \^etre rendue pr\'ecise sous la forme d'un th\'eor\`eme GAGA. Pour l'\'enoncer, indiquons qu'il suit de la construction qu'\`a tout faisceau coh\'erent~$\cF$ sur~$X$, on peut fonctoriellement associer un faisceau coh\'erent sur~$X\urc$, que nous noterons~$\cF\urc$. 

\begin{enonce*}{Th\'eor\`eme}
Soit~$X$ une vari\'et\'e alg\'ebrique complexe. Le foncteur 
\[\cF \in \mathrm{Coh}(X) \longmapsto \cF\urc \in \mathrm{Coh}(X\urc)\]
est une \'equivalence de cat\'egories.

Pour tout faisceau coh\'erent~$\cF$ sur~$X$ et tout $q\in \N$, on a un isomorphisme naturel
\[H^q(X,\cF) \stackrel{\sim}{\longrightarrow} H^q(X\urc,\cF\urc).\]
\end{enonce*}

La th\'eorie des espaces de Berkovich globaux d\'evelopp\'ee dans ce manuscrit constitue un pr\'eliminaire indispensable \`a cette th\'eorie de la compactification. Deux ingr\'edients y jouent un r\^ole pr\'epond\'erant~: le foncteur d'analytification (\cf~chapitre~\ref{catan}), qui intervient dans la construction m\^eme des espaces~$X\urc$, et les r\'esultats d'annulation cohomologiques (\cf~chapitre~\ref{chap:Stein}), sur lesquels repose la preuve du th\'eor\`eme GAGA.

\section{Pr\'erequis}\label{sec:prerequis}

Ce texte est r\'edig\'e de fa\c con \`a pouvoir servir d'introduction \`a la th\'eorie des espaces de Berkovich globaux. Celle-ci trouve sa source dans le livre fondateur de Berkovich~\cite{Ber1} et les textes du second auteur~\cite{A1Z,EtudeLocale}. Nous nous appuyons de fa\c con essentielle sur ces r\'ef\'erences, dont nous utilisons parfois certains r\'esultats sans les red\'emontrer, mais prenons soin de red\'efinir toutes les notions inh\'erentes \`a la th\'eorie des espaces de Berkovich globaux, afin que le manuscrit puisse \^etre lu sans revenir aux textes sus-cit\'es. Le chapitre~\ref{chap:rappels} propose \'egalement un rappel des r\'esultats principaux de~\cite{A1Z,EtudeLocale}.

Les espaces de Berkovich globaux m\^elant espaces de Berkovich classiques (sur un corps valu\'e ultram\'etrique) et espaces analytiques complexes, il est utile que le lecteur soit familier avec les rudiments de ces deux th\'eories. Aucune connaissance pouss\'ee n'est n\'ecessaire, le manuscrit traitant, pour l'essentiel, des propri\'et\'es de base des espaces globaux.



Pour les aspects analytiques complexes, nous renvoyons aux ouvrages classiques~\cite{Gu-Ro,Gr-Re2}, ainsi qu'\`a~\cite{Gr-Re} pour la th\'eorie des espaces de Stein. Les m\'ethodes utilis\'ees dans ce texte empruntent largement \`a ces r\'ef\'erences, et le lecteur consultera avec profit les preuves qui y figurent, dans un cadre techniquement plus ais\'e. 

Pour les aspects analytiques ultram\'etriques, nous renvoyons \`a l'ouvrage de Berkovich~\cite{Ber1}. Pr\'ecisons que n'intervient dans ce manuscrit qu'une version simplifi\'ee des espaces de Berkovich, d\'efinie de fa\c con analogue aux espaces analytiques complexes, et non la version plus sophistiqu\'ee, h\'erit\'ee de la g\'eom\'etrie analytique rigide de John Tate (\cf~\cite{TateRAS}), bas\'ee sur la th\'eorie des alg\`ebres affino\"ides. Une connaissance approfondie de la th\'eorie g\'en\'erale des espaces de Berkovich sur un corps valu\'e n'est donc pas indispensable \`a la lecture de ce texte. En revanche, il est important de comprendre pr\'ecis\'ement la droite de Berkovich sur un corps valu\'e, et particuli\`erement sa topologie, y compris lorsque le corps de base n'est pas alg\'ebriquement clos. Les textes introductifs \cite{TemkinIntroduction,PoineauTurchettiVIASMI} peuvent ici apporter leur aide.


Attirons, une nouvelle fois, l'attention du lecteur sur le fait que, lorsqu'on choisit pour un anneau de Banach de base un corps valu\'e ultram\'etrique, les espaces de Berkovich globaux consid\'er\'es dans ce texte ne permettent pas de retrouver tous les espaces de Berkovich classiques, mais seulement ceux sans bord.

\section{Remerciements}

Le premier auteur tient tout particuli\`erement \`a remercier Antoine Ducros et Fr\'ed\'eric Paugam pour de nombreuses discussions stimulantes. Plusieurs r\'esultats de ce texte ont pu voir le jour durant sa th\`ese gr\^ace \`a leurs conseils et leurs remarques. Le second auteur remercie Dorian Berger pour ses nombreux commentaires. Les deux auteurs souhaitent \'egalement exprimer leur profonde gratitude \`a un rapporteur anonyme, qui a formul\'e de nombreuses suggestions pertinentes et contribu\'e de fa\c con significative \`a l'am\'elioration du manuscrit.

Une partie de ce texte a \'et\'e r\'edig\'ee pendant un s\'ejour du second auteur \`a Francfort. Il remercie l'universit\'e Goethe et ses membres pour les excellentes conditions d'accueil dont il a b\'en\'efici\'e, ainsi que la Deutsche Forschungsgemeinschaft (TRR 326 Geometry and Arithmetic of Uniformized Structures, project number 444845124) pour son soutien financier. 

Les auteurs ont b\'en\'efici\'e du soutien de l'ANR (projet ANR JCJC \og GLOBES\fg{} ANR-12-JS01-0007) et de l'ERC (projet ERC Starting Grant \og TOSSIBERG \fg{} 637027). 

\chapter[Pr\'eliminaires et rappels]{Pr\'eliminaires et rappels}\label{chap:rappels}

Dans ce chapitre, nous pr\'esentons les bases de la g\'eom\'etrie analytique sur un anneau de Banach telle que d\'evelopp\'ee par V.~Berkovich dans le premier chapitre de son ouvrage fondateur~\cite{Ber1}. Nous rappelons \'egalement quelques r\'esultats locaux d\'emontr\'es par le second auteur dans~\cite{A1Z} et~\cite{EtudeLocale}. 

Dans la section~\ref{sec:spectreanalytique}, nous introduisons la d\'efinition de spectre analytique d'un anneau de Banach. Nous en donnons plusieurs exemples fondamentaux, qui nous guideront dans toute la suite de ce texte~: cas des corps valu\'es, de l'anneau des entiers relatifs~$\Z$, des anneaux d'entiers de corps de nombres, des corps hybrides, des anneaux de valuation discr\`ete et des anneaux de Dedekind.  Nous introduisons ensuite, dans la section~\ref{sec:espaceaffineanalytique}, la d\'efinition d'espace affine analytique sur un anneau de Banach. Il s'agit d'un espace localement annel\'e. Les sections de son faisceau structural sont des fonctions que nous qualifierons d'analytiques.

Le reste du chapitre est constitu\'e, pour la majeure partie, de rappels sur les travaux men\'es par le second auteur dans~\cite{A1Z} et~\cite{EtudeLocale}. Nous commen\c{c}ons par des d\'efinitions et r\'esultats techniques utiles. Dans la section~\ref{sec:spconvexe}, nous introduisons les parties spectralement convexes. Il s'agit de parties qui peuvent \^etre naturellement identifi\'ees \`a des spectres analytiques d'anneaux de Banach est nous les utiliserons constamment dans la suite. La section~\ref{sec:dcl} est consacr\'ee aux disques, couronnes et domaines polynomiaux. Nous y rappelons des r\'esultats sur les fonctions globales sur les disques et les couronnes qui, comme on s'y attend, peuvent s'exprimer \`a l'aide de s\'eries convergentes. La section~\ref{description_locale} contient une description pr\'ecise de bases de voisinages des points de la droite affine analytique sur un anneau de Banach. 

La section finale~\ref{sec:resultatslocaux} est plus pouss\'ee. Nous y pr\'esentons quelques outils techniques essentiels pour la suite de notre \'etude, tels les th\'eor\`emes de division et de pr\'eparation de Weierstra\ss. Ils sont valables pour des anneaux de Banach qualifi\'es de basiques, dont nous rappelons la d\'efinition. Indiquons que tous les exemples mentionn\'es plus haut sont basiques. Afin d'obtenir des r\'esultats locaux pr\'ecis, nous distinguons diff\'erents types de points~: points rigides, rigides \'epais ou localement transcendants. Nous concluons en rappelant les r\'esultats principaux de~\cite{EtudeLocale}~:  noeth\'erianit\'e et r\'egularit\'e des anneaux locaux, prolongement analytique, coh\'erence du faisceau structural.

\section{Spectre analytique d'un anneau de Banach}\label{sec:spectreanalytique}

\begin{defi}\index{Norme|textbf}
Une \emph{norme} sur un anneau~$\cA$ est une application 
\[ \nm_{\cA} \colon \cA \too \R_{\ge 0}\]
v\'erifiant les propri\'et\'es suivantes~:
\begin{enumerate}[i)]
\item $\forall a,b\in \cA$, $\|a+b\|_{\cA} \le \|a\|_{\cA} + \|b\|_{\cA}$ ;
\item $\forall a\in \cA$, $\|-a\|_{\cA} = \|a\|_{\cA}$ ;
\item $\forall a \in \cA$, $\|a\|_{\cA} = 0 \Longleftrightarrow a=0$.
\end{enumerate}
\end{defi}

\begin{defi}\index{Norme!sous-multiplicative|textbf}
Une norme~$\nm_{\cA}$ sur un anneau~$\cA$ est dite \emph{sous-multiplicative} si, pour tout ensemble fini~$I$ et toute famille $(a_{i})_{i\in I}$ d'\'el\'ements de~$\cA$, on a
\[ \big\| \prod_{i\in I} a_{i} \big\|_{\cA} \le \prod_{i\in I} \|a_{i}\|_{\cA}.\]
\end{defi}

\begin{rema}
Soit~$\cA$ un anneau. Si $\cA = 0$, l'unique norme d\'efinie sur~$\cA$ est la norme nulle et elle est sous-multiplicative.

Supposons que~$\cA \ne 0$. Soit~$\nm_{\cA}$ une norme sous-multiplicative sur~$\cA$. En appliquant la formule de la d\'efinition avec $I = \emptyset$, on obtient $\|1\|_{\cA} \le 1$. En l'appliquant avec $I = \{0,1\}$ et $a_{0} = a_{1} = 1$, on obtient $\|1\|_{\cA} \le \|1\|_{\cA}^2$, d'o\`u l'on tire $\|1\|_{\cA} \ge 1$, puis $\|1\|_{\cA} = 1$. 
\end{rema}

\begin{defi}\index{Anneau!de Banach|textbf}
Un \emph{anneau de Banach} est un couple $(\cA,\nm_{\cA})$ o\`u $\cA$ est un anneau et $\nm_{\cA}$ une norme sous-multiplicative sur~$\cA$ pour laquelle il est complet.

Un \emph{morphisme born\'e} entre anneaux de Banach $(\cA,\nm_{\cA})$ et $(\cA',\nm_{\cA'})$ est une application $\tau \colon \cA \to \cA'$ telle qu'il existe une constante $C \in \R_{\ge 0}$ v\'erifiant la propri\'et\'e suivante~:
\[ \forall a \in \cA,\ \|\tau(a)\|_{\cA'} \le C\, \|a\|_{\cA}.\]
Un \emph{morphisme contractant} est un morphisme born\'e pour lequel on peut choisir la constante $C=1$.
\end{defi}

Soit~$(\cA,\|.\|_\cA)$ un anneau de Banach. 

\begin{defi}[\protect{\cite[\S 1.2]{Ber1}}]\label{def:MA}%
\index{Spectre analytique|textbf}%
\nomenclature[Fa]{$\cM(\cA)$}{spectre analytique de~$\cA$} 
Le \emph{spectre analytique} de~$(\cA,\|.\|_\cA)$ est l'ensemble~$\cM(\cA,\nm_{\cA})$ des semi-normes multiplicatives born\'ees sur~$\cA$, c'est-\`a-dire l'ensemble des applications
\[\va \colon \cA \too \R_{\ge 0}\]
v\'erifiant les propri\'et\'es suivantes~:
\begin{enumerate}[i)]
\item $|0| = 0$ et $|1| = 1$;
\item $\forall a,b \in \cA,\ |a+b| \le |a| + |b|$ ;
\item $\forall a,b \in \cA,\ |ab| = |a| \, |b|$ ;
\item $\forall a\in \cA,\ |a| \le \|a\|_{\cA}$.
\end{enumerate}
On \'ecrira souvent abusivement $\cM(\cA)$ au lieu de $\cM(\cA,\nm_{\cA})$. 

On munit~$\cM(\cA,\nm_{\cA})$ de la topologie de la convergence simple, c'est-\`a-dire la topologie la plus grossi\`ere telle que, pour tout $a\in \cA$, l'application
\[\fonction{a}{\cM(\cA)}{\R}{|.|}{|a|}\]
soit continue.
\end{defi}

\begin{rema}
On peut remplacer la condition iv) de la d\'efinition~\ref{def:MA} par la condition apparemment plus faible 
\[\textrm{iv')}\ \exists M \in \R_{\ge 0}, \, \forall a\in \cA,\ |a| \le M\, \|a\|_{\cA}.\]
En effet, pour $a\in \cA$ et $n\in \N^\ast$, on a alors $|a|^n = |a^n| \le M\, \|a^n\|_{\cA} \le \|a\|_{\cA}^n$, d'o\`u $|a| \le M^{1/n} \, \|a\|$, puis $|a| \le \|a\|_{\cA}$ en passant \`a la limite quand $n$ tend vers l'infini.
\end{rema}

\begin{lemm}\label{lem:M(tau)topo}
Soit $\tau \colon \cA \to \cA'$ un morphisme born\'e d'anneaux de Banach. Toute semi-norme multiplicative born\'ee sur~$\cA'$ induit, par composition avec~$\tau$, une semi-norme multiplicative born\'ee sur~$\cA$. 

L'application 
\[\cM(\tau) \colon \cM(\cA') \too \cM(\cA)\] 
ainsi d\'efinie est continue.

On obtient ainsi un foncteur contravariant $\cM$ de la cat\'egorie des anneaux de Banach muni des morphismes born\'es vers celle des espaces topologiques.
\qed
\end{lemm}


\begin{theo}[\protect{\cite[theorem~1.2.1]{Ber1}}]\label{th:MAnonvide}
Si $\cA \ne 0$, alors $\cM(\cA)$ est compact et non vide.
\qed
\end{theo}

On pense g\'en\'eralement aux \'el\'ements de~$\cM(\cA)$ comme \`a des points d'un espace et on les note~$x$, $y$, etc. On note alors~$\va_{x}$, $\va_{y}$, etc. les semi-normes sur~$\cA$ associ\'ees.%
\nomenclature[Fb]{$\va_{x}$}{semi-norme sur~$\cA$ associ\'ee \`a point~$x$ de~$\cM(\cA)$}

\begin{defi}\index{Corps!residuel complete@r\'esiduel compl\'et\'e|textbf}
Soit $x\in \cM(\cA)$. L'ensemble
\[\Ker(x) := \{a \in \cA \,\colon |a|_{x} =0\}\]
est un id\'eal premier de~$\cA$.%
\nomenclature[Fc]{$\Ker(x)$}{noyau de~$\va_{x}$}

La semi-norme~$\va_{x}$ induit une valeur absolue sur l'anneau int\`egre $\cA/\Ker(x)$ et son corps de fractions $\Frac(\cA/\Ker(x))$. On appelle \emph{corps r\'esiduel compl\'et\'e} de~$x$ le compl\'et\'e 
\[ \cH(x) := \widehat{\Frac(\cA/\Ker(x))}.\]%
\nomenclature[Fd]{$\cH(x)$}{corps r\'esiduel compl\'et\'e de~$x$}
On notera simplement~$\va$ la valeur absolute induite sur~$\cH(x)$, cela ne pr\^etant pas \`a confusion.
\end{defi}

Pour tout $x\in \cM(\cA)$, on a, par construction, un morphisme contractant naturel \index{Evaluation@\'Evaluation|textbf}
\[ \ev_{x} \colon \cA \too \cH(x).\]%
\nomenclature[Fe]{$\ev_{x}$}{morphisme d'\'evaluation en~$x$}
Pour $a \in \cA$, on pose
\[ a(x) := \ev_{x}(a) \in \cH(x).\]%
\nomenclature[Ff]{$a(x)$}{\'evaluation en~$x$ d'un \'el\'ement~$a$ de~$\cA$}
On a alors
\[ |a(x)| = |a|_{x}.\]

\begin{lemm}\label{lem:M(tau)xx'}\nomenclature[Faa]{$\cM(\tau)$}{application $\cM(\cA') \to \cM(\cA)$ induite par un morphisme d'anneaux de Banach $\tau \colon \cA \to \cA'$}\nomenclature[Fda]{$\tau_{x,x'}$}{application $\cH(x) \to \cH(x')$ induite par un morphisme~$\tau$ d'anneaux de Banach}
Soit $\tau \colon \cA \to \cA'$ un morphisme born\'e d'anneaux de Banach. Soit $x' \in \cM(\cA')$ et posons $x := \cM(\tau)(x')$. Pour tout $a\in \cA$, on a
\[ |\tau(a)(x')| = |a(x)|.\]

Le  morphisme $\cA \to \cH(x')$, obtenu en composant par~$\tau$ le morphisme d'\'evaluation $\ev_{x'} \colon \cA' \to \cH(x')$, se factorise par un morphisme de corps isom\'etrique
\[\tau_{x',x} \colon \cH(x) \too \cH(x').\]
\qed
\end{lemm}

\begin{defi}\label{def:seminormesp}\index{Norme!spectrale|textbf}\index{Norme!uniforme|textbf}\index{Semi-norme|see{Norme}}\index{Anneau!de Banach!uniforme|textbf}
On d\'efinit la \emph{semi-norme spectrale}~$\nm_{\cA,\sp}$ sur~$\cA$ associ\'ee \`a~$\nm_{\cA}$ par 
\[ \forall a\in \cA,\ \|a\|_{\cA,\sp} := \inf_{n \ge 1} (\|a^n\|_{\cA}^{1/n}).\]
On dit que~$\nm_{\cA}$ ou $(\cA,\nm_{\cA})$, voire $\cA$ par abus, est \emph{uniforme} si $\nm_{\cA} = \nm_{\cA,\sp}$.
\end{defi}

\begin{lemm}[\protect{\cite[theorem~1.3.1]{Ber1}}]\label{lem:spuni}
Pour tout $a\in \cA$, on a
\[ \|a\|_{\cA,\sp} = \max_{x\in \cM(\cA)} (|a(x)|).\]
\qed
\end{lemm}

\begin{rema}\label{rem:completionuniforme}
Notons $\hatcA$ le s\'epar\'e compl\'et\'e de~$\cA$ pour la semi-norme spectrale~$\nm_{\cA,\sp}$. Le morphisme de compl\'etion $\cA \to \hatcA$ est contractant. L'application associ\'ee entre les spectres
\[\cM(\hatcA) \too \cM(\cA)\] 
est un hom\'eomorphisme qui induit des isomorphismes isom\'etriques entre les corps r\'esiduels compl\'et\'es.
\end{rema}

Nous allons maintenant dresser une liste d'exemples \`a laquelle nous nous r\'ef\'ererons constamment par la suite. Ce sont eux qui motivent le d\'eveloppement de la th\'eorie. Introduisons, au pr\'ealable, quelques notations.

\begin{nota}\index{Valeur absolue triviale|textbf}%
\nomenclature[Ba]{$\va_{\infty}$}{Valeur absolue usuelle sur $\Z$, $\Q$, $\R$, $\C$, etc.}%
\nomenclature[Bb]{$\va_{p}$}{Valeur absolue $p$-adique normalis\'ee sur $\Z$, $\Q$, $\Q_{p}$, $\C_{p}$, etc.}%
\nomenclature[Bc]{$\va_{0}$}{Valeur absolue triviale sur un anneau}
On note~$\va_{\infty}$ la valeur absolue usuelle sur~$\C$. On note identiquement sa restriction aux sous-anneaux de~$\C$~: $\R$, $\Q$, $\Z$, etc.

Pour tout nombre premier~$p$, on note~$\va_{p}$ la valeur absolue $p$-adique sur~$\Q_{p}$ normalis\'ee par la condition $|p|_{p} = p^{-1}$.  On note identiquement sa restriction aux sous-anneaux de~$\Q_{p}$.

Pour tout anneau~$A$, on note~$\va_{0}$ la \emph{valeur absolue triviale} sur~$A$~: 
\[\fonction{\va_{0}}{\Z}{\R_{\ge 0}}{n}{\begin{cases} 0 \textrm{ si } n=0~;\\ 1 \textrm{ sinon.}\end{cases}}\]
\end{nota}

Ajoutons un lemme technique utile.

\begin{lemm}\label{lem:critcomvalabs}
Soient $k$ un corps et $\va$ et $\va'$ deux valeurs absolues sur~$k$. Supposons que~$\va'$ n'est pas triviale et que
\[\forall a\in k,\ |a|' > 1 \implies |a| \geq |a|'\]
Alors il existe $\eps\in\intof{0,1}$ tel que $\va'=\va^\eps$.
\end{lemm}
\begin{proof}
Puisque~$\va'$ n'est pas triviale, il existe $a \in k$ tel que $|a|'>1$. Par hypoth\`ese, on a $|a| \ge |a|'$, donc il existe $\eps\in\intof{0,1}$ tel que $|a'| = |a|^\eps$.

Soit $b\in k^*$. Soient $n\in\Z$ et $m\in\N^\ast$ tels que $|b|' > (|a|')^{n/m}$. On a alors $|b^m/a^n|' >1$, donc $|b^m/a^n| >1$, puis $|b| > |a|^{n/m}$ et finalement $|b|^\eps > (|a|')^{n/m}$. En utilisant la densit\'e de~$\Q$ dans~$\R$, on en d\'eduit que $|b|^\eps \ge |b|'$. En appliquant ce r\'esultat \`a l'inverse de~$b$, on obtient $|b|^\eps \le |b|'$, d'o\`u $|b|^\eps = |b|'$.
\end{proof}

\index{Spectre analytique!exemples|(}\index{Anneau!de Banach!uniforme|(}

\begin{exem}\label{ex:corpsvalue}\index{Corps!valu\'e|textbf}
Soit~$k$ un corps et~$\va$ une valeur absolue (archim\'edienne ou non) pour laquelle il est complet. Alors $(k,\va)$ est un anneau de Banach uniforme et son spectre analytique~$\cM(k)$ est r\'eduit \`a un point, correspondant \`a~$\va$.
\end{exem}

\begin{exem}\label{ex:Z}\index{Anneau!des entiers relatifs $\Z$|textbf}
L'anneau des entiers relatifs~$\Z$ muni de la valeur absolue usuelle~$\va_{\infty}$ est un anneau de Banach uniforme. La description de son spectre analytique d\'ecoule du th\'eor\`eme d'Ostrowski. Nous l'explicitons ci-dessous.

Le spectre~$\cM(\Z)$ rev\^et la forme d'une \'etoile~: il poss\'ede un point central~$a_{0}$, qui correspond \`a la valeur absolue triviale~$\va_{0}$, et une infinit\'e de branches, index\'ees par les nombres premiers et la place infinie.%
\nomenclature[Ga]{$a_{0}$}{point de~$\cM(\Z)$ associ\'e \`a la valeur absolue triviale~$\va_{0}$}

Soit~$p$ un nombre premier. D\'efinissons la semi-norme
\[\fonction{\va_{p}^{+\infty}}{\Z}{\R_{\ge 0}}{n}{\begin{cases} 0 \textrm{ si } n=0 \mod p~;\\ 1 \textrm{ sinon.}\end{cases}}\]
Pour tout $\eps\in\intof{0,+\infty}$, notons~$a_{p}^\eps$ le point de~$\cM(\Z)$ associ\'e \`a~$\va_{p}^\eps$. Posons $a_{p}^0:= a_{0}$. L'application $\eps\in\intof{0,+\infty}\mapsto a_p^\eps$ r\'ealise alors un hom\'eomorphisme sur la branche associ\'ee \`a~$p$.%
\nomenclature[Gb]{$a_{p}^\eps$}{point de~$\cM(\Z)$ associ\'e \`a la valeur absolue~$\va_{p}^\eps$}

Pour tout $\eps\in\intof{0,1}$, notons~$a_{\infty}^\eps$ le point de~$\cM(\Z)$ associ\'e \`a~$\va_{\infty}^\eps$. Posons $a_{\infty}^0:= a_{0}$. L'application $\eps\in\intff{0,1}\mapsto a_{\infty}^\eps$ r\'ealise alors un hom\'eomorphisme sur la branche associ\'ee \`a la place infinie.%
\nomenclature[Gc]{$a_{\infty}^\eps$}{point de~$\cM(\Z)$ associ\'e \`a la valeur absolue~$\va_{\infty}^\eps$}

Indiquons \'egalement que tout voisinage de~$a_{0}$ contient enti\`erement toutes les branches \`a l'exception d'un nombre fini.

La figure~\ref{fig:MZ} contient une repr\'esentation d'un plongement (non canonique) de~$\cM(\Z)$ dans~$\R^2$ respectant la topologie.

\begin{figure}[!h]
\centering
\begin{tikzpicture}
\foreach \x [count=\xi] in {-2,-1,...,17}
\draw (0,0) -- ({10*cos(\x*pi/10 r)/\xi},{10*sin(\x*pi/10 r)/\xi}) ;
\foreach \x [count=\xi] in {-2,-1,...,17}
\fill ({10*cos(\x*pi/10 r)/\xi},{10*sin(\x*pi/10 r)/\xi}) circle ({0.07/(sqrt(\xi)}) ;

\draw ({10.5*cos(-pi/5 r)},{10.5*sin(-pi/5 r)}) node{$a_{\infty}$} ;
\fill ({5.5*cos(-pi/5 r)},{5.5*sin(-pi/5 r)}) circle (0.07) ;
\draw ({5.5*cos(-pi/5 r)},{5.5*sin(-pi/5 r)-.1}) node[below]{$a_{\infty}^\eps$} ;

\draw ({11*cos(-pi/10 r)/2+.1},{11*sin(-pi/10 r)/2}) node{$a_2^{+\infty}$} ;
\fill ({2.75*cos(-pi/10 r)},{2.75*sin(-pi/10 r)}) circle ({0.07/(sqrt(2)}) ;
\draw ({2.75*cos(-pi/10 r)},{2.75*sin(-pi/10 r)-.05}) node[below]{$a_2^\eps$} ;

\draw ({12*cos(pi/5 r)/5+.1},{12*sin(pi/5 r)/5+.1}) node{$a_p^{+\infty}$} ;
\end{tikzpicture}
\caption{Le spectre de Berkovich $\cM(\Z)$.}\label{fig:MZ}
\end{figure}

\medbreak

Les parties compactes et connexes de~$\cM(\Z)$ peuvent \'egalement s'exprimer comme des spectres. Plus pr\'ecis\'ement, soit~$V$ une partie compacte et connexe de~$\cM(\Z)$. Notons $S_{V}$ l'ensemble des \'el\'ements de~$\Z$ qui ne s'annulent pas sur~$V$. Posons $\cK(V) := S_{V}^{-1} \, \Z$ et notons~$\cB(V)$ le s\'epar\'e compl\'et\'e de~$\cK(V)$ pour la semi-norme uniforme~$\nm_{V}$ sur~$V$. (Nous retrouverons ces d\'efinitions plus tard dans un cadre plus g\'en\'eral, \cf~d\'efinition~\ref{def:fracrat} et notation~\ref{nota:BV}). Le spectre $\cM(\cB(V))$ s'identifie alors naturellement \`a~$V$, d'apr\`es \cite[proposition~3.1.16]{A1Z} (ce qui signifie que le compact~$V$ est spectralement convexe au sens de la d\'efinition~\ref{def:spconvexe}). \index{Partie!spectralement convexe}

Pour proposer un exemple plus concret, consid\'erons la partie~$N$ de~$\cM(\Z)$ qui est l'enveloppe convexe des valeurs absolues normalis\'ees, autrement dit
\begin{align*}
N &= \Big\{ x\in \cM(\Z) : \forall p\in \Pc,\ |p(x)| \ge \frac1p\Big\}\\
& = \intff{a_{0},a_{\infty}} \cup \bigcup_{p \in \Pc} \intff{a_{0},a_{p}},
\end{align*}
o\`u $\Pc$ d\'esigne l'ensemble des nombres premiers. On a alors $S_{N} = \Z \setminus \{0\}$, $\cK(N) =  \Q$, $\nm_{N} = \max \{ \va_{\infty}, \va_{p} : p\in \Pc\}$ et $\cB(N) = \Q$. 

L'ensemble~$N$ est repr\'esent\'e \`a la figure~\ref{fig:compactMZ}. Il nous semble m\'eriter consid\'eration et poss\`eder des propri\'et\'es int\'eressantes. On peut, par exemple, remarquer que la s\'erie exponentielle est d\'efinie et convergente dans un disque relatif de rayon strictement positif au-dessus de~$N$.

\begin{figure}[h!]
\centering
\begin{tikzpicture}
\draw (0,0) -- ({10*cos(-2*pi/10 r)},{10*sin(-2*pi/10 r)}) ;
\foreach \x [count=\xi] in {-1,0,...,17}
\draw (0,0) -- ({5.5*cos(\x*pi/10 r)/(\xi+1)},{5.5*sin(\x*pi/10 r)/(\xi+1)}) ;
\foreach \x [count=\xi] in {-1,0,...,17}
\draw [dotted] ({5.5*cos(\x*pi/10 r)/(\xi+1)},{5.5*sin(\x*pi/10 r)/(\xi+1)}) -- ({10*cos(\x*pi/10 r)/(\xi+1)},{10*sin(\x*pi/10 r)/(\xi+1)}) ;

\foreach \x [count=\xi] in {-1,0,...,17}
\fill ({5.5*cos(\x*pi/10 r)/(\xi+1)},{5.5*sin(\x*pi/10 r)/(\xi+1)}) circle ({0.07/(sqrt(\xi)}) ;

\draw ({10.5*cos(-pi/5 r)-.35},{10.5*sin(-pi/5 r)+.05}) node{$a_{\infty}$} ;
\fill ({10*cos(-pi/5 r)},{10*sin(-pi/5 r)}) circle (0.07) ;

\fill ({2.75*cos(-pi/10 r)},{2.75*sin(-pi/10 r)}) circle ({0.07/(sqrt(2)}) ;
\draw ({2.75*cos(-pi/10 r)},{2.75*sin(-pi/10 r)}) node[below]{$a_2$} ;

\draw ({6*cos(pi/7 r)/5},{6*sin(pi/7 r)/5}) node{\tiny $a_p$} ;
\end{tikzpicture}
\caption{Le compact~$N$ de $\cM(\Z)$.}\label{fig:compactMZ}
\end{figure}
\end{exem}

\begin{exem}\label{ex:cdn}\index{Anneau!des entiers d'un corps de nombres|textbf}
Soit~$K$ un corps de nombres. Notons~$\Sigma_{\infty}$ l'ensemble des places infinies de~$K$, vues comme classes de conjugaison de plongements complexes. Alors, l'anneau des entiers~$A$ de~$K$ muni de la norme
\[\nm_{A} := \max_{\sigma \in \Sigma_{\infty}} (|\sigma(\wc)|_{\infty}) \]
est un anneau de Banach uniforme. Les descriptions et r\'esultats de l'exemple~\ref{ex:Z} valent encore dans ce cadre, \textit{mutatis mutandis}.

En ce qui concerne le spectre~$\cM(A)$, par exemple, il poss\`ede un point central~$a_{0}$, correspondant \`a la valeur absolue triviale, et une infinit\'e de branches qui en jaillissent~: une pour chaque id\'eal maximal~$\m$ de~$A$ (dont les points seront not\'es~$a_{\m}^\eps$ avec $\eps \in [0,+\infty]$) et une pour chaque place infinie~$\sigma$ de~$K$ (dont les points seront not\'es~$a_{\sigma}^\eps$ avec $\eps \in [0,1]$). Pour plus de pr\'ecisions, nous renvoyons le lecteur \`a~\cite[\S 3.1]{A1Z}.%
\nomenclature[Ha]{$a_{0}$}{point de $\cM(A)$ associ\'e \`a la valeur absolue triviale~$\va_{0}$}%
\nomenclature[Hb]{$a_{\m}^\eps$}{point de $\cM(A)$ associ\'e \`a une valeur absolue $\m$-adique, o\`u $\m$ est un id\'eal maximal de~$A$}%
\nomenclature[Hc]{$a_{\sigma}^\eps$}{point de $\cM(A)$ associ\'e \`a une valeur absolue $\sigma$-adique, o\`u $\sigma$ est une place de~$K$}
\end{exem}

\begin{exem}\label{ex:corpshybride}
\index{Corps!hybride|textbf}%
\nomenclature[Bd]{$k^\hyb$}{corps $k$ muni de la norme hybride}
Soit~$k$ un corps muni d'une valeur absolue non triviale~$\va$ (pour laquelle il n'est pas n\'ecessairement complet). 
Posons $\nm := \max(\va,\va_{0})$. Alors $(k,\nm)$ est un anneau de Banach uniforme. Nous l'appellerons \emph{corps hybride} associ\'e \`a~$k$ et le noterons~$k^\hyb$.

L'application 
\[\eps\in[0,1]\mapstoo \va^\eps\in\cM(k^\hyb)\] 
est un hom\'eomorphisme. 

Le seul point qui n'est pas imm\'ediat est la surjectivit\'e. Soit~$x$ un point de~$\cM(k^\hyb)$ associ\'e \`a une valeur absolue~$\va_{x}$ non triviale. Par d\'efinition du spectre, pour tout $a\in k$, on a $|a|_{x} \le |a|$. Le lemme~\ref{lem:critcomvalabs} assure alors que~$x$ appartient bien \`a l'image de l'application ci-dessus.

Cet exemple appara\^it notamment lorsque l'on souhaite \'etudier des d\'eg\'en\'erescences de familles de vari\'et\'es sur~$k$. Il a \'et\'e introduit par V.~Berkovich dans~\cite{BerW0} pour $k=\C$.
\end{exem}

\begin{exem}\label{ex:avd}\index{Anneau!de valuation discr\`ete|textbf}
Soit $R$ un anneau de valuation discr\`ete complet et~$v$ la valuation associ\'ee. Notons~$\m$ l'id\'eal maximal de~$R$. Soit $r \in \intoo{0,1}$ et posons $\va := r^{v(\wc)}$. Alors $(R,\va)$ est un anneau de Banach uniforme.

D\'efinissons la semi-norme
\[\fonction{\va^{+\infty}}{R}{\R_{\ge 0}}{a}{\begin{cases} 0 \textrm{ si } a \in \m~;\\ 1 \textrm{ sinon.}\end{cases}}\]
L'application 
\[\eps\in[1,+\infty]\mapstoo\va^\eps\in\cM(R)\] 
est alors un hom\'eomorphisme. 

Le seul point qui n'est pas imm\'ediat est la surjectivit\'e. Soit $x\in\cM(R)$. Le noyau~$\Ker(x)$ de la semi-norme~$\va_{x}$ associ\'ee \`a~$x$ est un id\'eal premier de~$R$. Il est donc \'egal soit \`a~$(0)$, soit \`a~$\m$.

Supposons que $\Ker(x) = (0)$. Alors la semi-norme~$\va_{x}$ est une valeur absolue et elle s'\'etend par multiplicativit\'e en une valeur absolue sur~$\Frac(R)$. Soit $a\in \Frac(R)$ tel que $|f| >1$. Alors, on a $f^{-1} \in R$, donc $|f^{-1}|_{x} \le |f^{-1}|$, d'o\`u $|f|_{x} \ge |f|$. Le lemme~\ref{lem:critcomvalabs} assure alors qu'il existe~$\eta \in \intof{0,1}$ tel que $\va = \va_{x}^{\eta}$.

Supposons que $\Ker(x) = \m$.  Alors $R/\Ker(x)$ est un corps valu\'e dont la valeur absolue est born\'ee par~1. Cette derni\`ere est donc triviale. On en d\'eduit que $\va_{x} = \va^{+\infty}$. Ceci conclut la preuve.
\end{exem}

\begin{exem}\label{ex:Dedekind}\index{Anneau!de Dedekind trivialement valu\'e|textbf}
Soit $R$ un anneau de Dedekind. Munissons-le de la valeur absolue triviale~$\va_{0}$. Alors $(R,\va_{0})$ est un anneau de Banach uniforme. Nous le noterons~$R^\triv$.%
\nomenclature[Bca]{$R^\triv$}{Anneau $R$ muni de la valeur absolue triviale}

Soit~$\p$ un id\'eal premier non nul de~$R$. Notons~$v_\p$ la valuation associ\'ee \`a cet id\'eal. Pour tout $r \in \intoo{0,1}$, posons $\va_{\p,r} := r^{v_\p(\wc)}$. D\'efinissons \'egalement la semi-norme
\[\fonction{\va_{\p,0}}{R}{\R_{\ge 0}}{a}{\begin{cases} 0 \textrm{ si } a \in \p~;\\ 1 \textrm{ sinon.}\end{cases}}\]
On a une application injective $\lambda_{\p} \colon r \in {[}{0,1}{[} \mapsto \va_{\p,r}\in\cM(R^{\triv})$. 

Soit $s \in \intoo{0,1}$. Posons $V_{\p,s} := \{ x \in \cM(R^\triv) : \forall a \in \p, |a(x)| \le s\}$. Notons~$R_{\p,s}$ le compl\'et\'e de~$R_{\p}$ par rapport \`a la valeur absolue~$\va_{\p,s}$. Le morphisme born\'e naturel $R\to R_{\p,s}$ induit un hom\'eomorphisme $\mu_{\p,s} \colon \cM(R_{\p,s}) \to V_{\p,s}$. Seule la surjectivit\'e requiert une preuve. Elle repose sur le fait que toute semi-norme associ\'ee \`a un point de~$V_{\p,s}$ est born\'ee par $\va_{\p,s}$ et s'\'etend donc \`a~$R_{\p,s}$. En utilisant l'exemple~\ref{ex:avd}, on en d\'eduit que l'application~$\lambda_{\p}$ induit un hom\'eomorphisme entre~$[0,s]$ et~$V_{\p,s}$. 

Posons $U_{\p} :=  \{ x \in \cM(R^\triv) : \forall a \in \p, |a(x)| < 1\}$. Puisque l'id\'eal~$\p$ est de type fini, on a $U_{\p} = \bigcup_{s \in  \intoo{0,1}} V_{\p,s}$. Le r\'esultat pr\'ec\'edent montre que~$\lambda_{\p}$ induit un hom\'eomorphisme entre $\intfo{0,1}$ et~$U_{\p}$.

Notons $\lambda \colon \bigsqcup_{\p \ne (0)} \intfo{0,1}\to \cM(R^\triv)$ l'application induite par les~$\lambda_{\p}$. C'est un hom\'eomorphisme sur son image. 
Pour tout point~$x$ de~$\cM(R^\triv)$, $\{a \in R : |a(x)| < 1\}$ est un id\'eal premier de~$R$, non nul si $x$ ne correspond pas \`a~$\va_{0}$. Par cons\'equent, l'image de~$\lambda$ est \'egale \`a $\cM(R^\triv) \setminus \{\va_{0}\}$. Puisque $\cM(R^\triv)$ est compact, il s'ensuit que $\cM(R^\triv)$ est hom\'eomorphe au compactifi\'e d'Alexandrov de~$\bigsqcup_{\p \ne (0)} \intfo{0,1}$, le point~$\va_{0}$ correspondant au point \`a l'infini.

Dans le cas o\`u $R=\Z$, on peut identifier le spectre de $\Z^{\triv}$ \`a une partie du spectre de $\Z$, comme repr\'esent\'e sur la figure~\ref{fig:MZtriv}. Cette identification reste valable dans le cas o\`u~$R$ est un anneau d'entiers de corps de nombres.

\begin{figure}[!h]
\centering
\begin{tikzpicture}
\draw [dotted] (0,0) -- ({10*cos(-2*pi/10 r)},{10*sin(-2*pi/10 r)}) ;
\foreach \x [count=\xi] in {-1,0,...,17}
\draw (0,0) --  ({10*cos(\x*pi/10 r)/(\xi+1)},{10*sin(\x*pi/10 r)/(\xi+1)}) ;
\foreach \x [count=\xi] in {-1,0,...,17}
\fill ({10*cos(\x*pi/10 r)/(\xi+1)},{10*sin(\x*pi/10 r)/(\xi+1)}) circle ({0.07/(sqrt(\xi)}) ;

\draw ({10.5*cos(-pi/5 r)},{10.5*sin(-pi/5 r)}) node{$a_{\infty}$} ;
\draw ({11*cos(-pi/10 r)/2},{11*sin(-pi/10 r)/2}) node{$a_2^{+\infty}$} ;
\draw ({12*cos(pi/5 r)/5+.1},{12*sin(pi/5 r)/5}) node{$a_p^{+\infty}$} ;
\end{tikzpicture}
\caption{Le spectre de Berkovich $\cM(\Z^\triv)$.}\label{fig:MZtriv}
\end{figure}
\end{exem}
\index{Spectre analytique!exemples|)}
\index{Anneau!de Banach!uniforme|(}

Nous allons maintenant formuler des d\'efinitions abstraites inspir\'ees par les caract\'eristiques topologiques des exemples pr\'ec\'edents.

\begin{nota}\label{not:Ix}%
\nomenclature[Fg]{$x^\eps$}{point associ\'e \`a la semi-norme~$\va_{x}^\eps$ sur~$\cA$}%
\nomenclature[Fh]{$I_{x}$}{ensemble des $\eps \in \R_{>0}$ pour lesquels $x^\eps$ est d\'efini}
Soit $x\in \cM(\cA)$. On note 
\[ I_{x} := \{ \eps \in \R_{>0} : \va_{x}^\eps \in \cM(\cA)\},\]
c'est-\`a-dire l'ensemble des nombres r\'eels $\eps>0$ tels que l'application $\va_{x}^\eps \colon \cA \to \R_{\ge 0}$ soit une semi-norme multiplicative born\'ee sur~$\cA$. L'ensemble~$I_{x}$ est un intervalle de~$\R_{>0}$.

Pour $\eps \in I_{x}$, on note~$x^\eps$ le point de~$\cM(\cA)$ associ\'e \`a~$\va_{x}^\eps$.
\end{nota}

\begin{defi}\label{def:lineaire}\index{Partie!lineaire@lin\'eaire|textbf}\index{Spectre analytique!lineaire@lin\'eaire|textbf}
Une partie~$V$ de~$\cM(\cA)$ est dite \emph{lin\'eaire} s'il existe une partie~$V^{\partial}$ de~$V$, un point~$x$ de~$V$ et un intervalle $I \subset I_{x}$ non r\'eduit \`a un point satisfaisant les propri\'et\'es suivantes~: 
\begin{enumerate}[i)]
\item $V^{\partial}$ est contenue dans le bord de~$V$ et $\sharp V^{\partial}\le 2$~; 
\item l'application 
\[\begin{array}{ccc}
I & \too & V\setminus V^{\partial}\\
\eps & \mapstoo & x^\eps
\end{array}\]
est un hom\'eomorphisme.
\end{enumerate}

Le spectre~$\cM(\cA)$ est dit \emph{lin\'eaire} s'il est une partie lin\'eaire de lui-m\^eme.
\end{defi}

\begin{exem}\index{Corps!hybride}\index{Anneau!des entiers relatifs $\Z$}\index{Anneau!des entiers d'un corps de nombres}
Les spectres des corps hybrides (\cf~exemple~\ref{ex:corpshybride}) et des anneaux de valuation discr\`ete (\cf~exemple~\ref{ex:avd}) sont lin\'eaires.
\end{exem}

\begin{defi}\label{def:ostrowski}\index{Partie!d'Ostrowski|textbf}\index{Spectre analytique!d'Ostrowski|textbf}
Une partie~$V$ de~$\cM(\cA)$ est dite \emph{d'Ostrowski} s'il existe un point $a_{0}$ de~$V$, un ensemble d\'enombrable non vide (\'eventuellement fini)~$\Sigma$  et une famille $(V_{\sigma})_{\sigma\in \Sigma}$ de parties disjointes de~$V\setminus \{a_{0}\}$ satisfaisant les propri\'et\'es suivantes~: 
\begin{enumerate}[i)]
\item $V = \bigcup_{\sigma \in \Sigma} V_{\sigma} \cup \{a_{0}\}$~;
\item pour tout $\sigma\in \Sigma$, $V_{\sigma}$ et $V_{\sigma} \cup \{a_{0}\}$ sont des parties lin\'eaires de~$\cM(\cA)$~;
\item l'ensemble des parties de la forme
\[\bigcup_{\sigma \in \Sigma'} V'_{\sigma} \cup \bigcup_{\sigma \in \Sigma\setminus\Sigma'} V_{\sigma},\]
o\`u $\Sigma'$ est un sous-ensemble fini de~$\Sigma$ et, pour tout $\sigma \in \Sigma'$, $V'_{\sigma}$ est un voisinage de~$a_{0}$ dans~$V_{\sigma} \cup \{a_{0}\}$, est une base de voisinages de~$a_{0}$ dans~$V$.
\end{enumerate}

Le spectre~$\cM(\cA)$ est dit \emph{d'Ostrowski} s'il est une partie d'Ostrowski de lui-m\^eme.
\end{defi}

\begin{exem}\index{Anneau!des entiers relatifs $\Z$}\index{Anneau!des entiers d'un corps de nombres}\index{Anneau!de Dedekind trivialement valu\'e}
Le spectre de~$\Z$ et les spectres des anneaux d'entiers de corps de nombres (\cf~exemple~\ref{ex:Z}) sont d'Ostrowski, ainsi que les spectres des anneaux de Dedekind trivialement valu\'es (\cf~exemple~\ref{ex:Dedekind}).
\end{exem}

\section{Espace affine analytique sur un anneau de Banach}\label{sec:espaceaffineanalytique}

On peut adapter la d\'efinition de spectre analytique pour d\'efinir des espaces affines analytiques. Soit~$(\cA,\|.\|_\cA)$ un anneau de Banach et soit~$n\in\N$. 

\begin{defi}[\protect{\cite[definition~1.5.1]{Ber1}}]\index{Espace analytique!affine|see{Espace affine analytique}}\index{Espace affine analytique|textbf}
L'\emph{espace affine analytique} de dimension~$n$ sur~$(\cA,\|.\|_\cA)$ est l'ensemble $\E{n}{\cA}$ des semi-normes multiplicatives sur $\cA[T_{1},\dotsc,T_{n}]$ born\'ees sur~$\cA$, c'est-\`a-dire l'ensemble des applications
\[\va \colon \cA[T_{1},\dotsc,T_{n}] \too \R_{\ge 0}\]
v\'erifiant les propri\'et\'es suivantes~:
\begin{enumerate}[i)]
\item $|0| = 0$ et $|1| = 1$;
\item $\forall P,Q \in \cA[T_{1},\dotsc,T_{n}],\ |P+Q| \le |P| + |Q|$ ;
\item $\forall P,Q \in \cA[T_{1},\dotsc,T_{n}],\ |PQ| = |P| \, |Q|$ ;
\item $\forall a\in \cA,\ |a| \le \|a\|_{\cA}$.
\end{enumerate}%
\nomenclature[Ia]{$\E{n}{\cA}$}{espace affine analytique de dimension~$n$ sur~$\cA$}
\nomenclature[J]{$\E{1}{\cA}$}{droite affine analytique sur un anneau de Banach~$\cA$}

On munit $\E{n}{\cA}$ de la topologie de la convergence simple.
\end{defi}

\begin{rema}
L'espace $\E{0}{\cA}$ n'est autre que $\cM(\cA)$. 
\end{rema}

Comme dans le cas de~$\cM(\cA)$, pour tout point~$x$ de~$\E{n}{\cA}$, on note $\va_{x}$ la semi-norme sur $\cA[T_{1},\dotsc,T_{n}]$ associ\'ee. On d\'efinit \'egalement un corps r\'esiduel compl\'et\'e~$\cH(x)$ et un morphisme d'\'evaluation
\[ \ev_{x} \colon \cA[T_{1},\dotsc,T_{n}] \too \cH(x).\]
\index{Corps!residuel complete@r\'esiduel compl\'et\'e|textbf}\index{Evaluation@\'Evaluation|textbf}
Pour $f \in  \cA[T_{1},\dotsc,T_{n}]$, on pose 
\[ f(x) := \ev_{x}(f) \in \cH(x).\]%
\nomenclature[Ib]{$\va_{x}$}{semi-norme sur~$\cA[T_{1},\dotsc,T_{n}]$ associ\'ee \`a point~$x$ de~$\E{n}{\cA}$}
\nomenclature[Iba]{$\cH(x)$}{corps r\'esiduel compl\'et\'e de~$x$}%
\nomenclature[Ieva]{$\ev_{x}$}{morphisme d'\'evaluation en~$x$ sur $\cA[T_{1},\dotsc,T_{n}]$}%
\nomenclature[Ifxa]{$f(x)$}{\'evaluation en~$x$ d'un \'el\'ement~$f$ de~$\cA[T_{1},\dotsc,T_{n}]$}

\begin{rema}\label{rem:M(tau)Antopo}
Soit $\tau \colon \cA \to \cA'$ un morphisme born\'e d'anneaux de Banach. Les constructions des lemmes~\ref{lem:M(tau)topo} et~\ref{lem:M(tau)xx'} se g\'en\'eralisent imm\'ediatement au cadre des espaces affines analytiques.

On d\'efinit ainsi une application continue
\[ \E{n}{\tau} \colon \E{n}{\cA'} \too \E{n}{\cA}\]
et, pour tout $x'\in \E{n}{\cA'}$, un morphisme de corps isom\'etrique
\[\tau_{x',x} \colon \cH(x) \too \cH(x'),\]
o\`u $x := \E{n}{\tau}(x')$.
\nomenclature[Iaa]{$\E{n}{\tau}$}{application $\E{n}{\cA'} \too \E{n}{\cA}$ induite par un morphisme d'anneaux de Banach $\tau \colon \cA \to \cA'$}\nomenclature[Ibb]{$\tau_{x,x'}$}{application $\cH(x) \to \cH(x')$ induite par un morphisme~$\tau$ d'anneaux de Banach}

Lorsque $\tau$ est le morphisme de compl\'etion $\cA \to \hatcA$ (\cf~remarque~\ref{rem:completionuniforme}), l'application associ\'ee 
\[\E{n}{\cA\kern-0.45em\hat{\phantom{\cA}}} \too \E{n}{\vphantom{\hat\cA}\cA}\]
est un hom\'eomorphisme qui induit des isomorphismes isom\'etriques entre les corps r\'esiduels compl\'et\'es.
\end{rema}

\begin{defi}\index{Partie!ultrametrique@ultram\'etrique|textbf}\index{Partie!archimedienne@archim\'edienne|textbf}\index{Point!ultrametrique@ultram\'etrique|textbf}\index{Point!archimedien@archim\'edien|textbf}
Soit $x\in \E{n}{\cA}$. On dit que le point $x$ est \emph{ultram\'etrique} si $\va_{x}$ est ultram\'etrique~:
\[ \forall a,b \in \cA,\ |a+b|_{x} \le \max(|a|_{x},|b|_{x}).\]
Dans le cas contraire, on dit que le point $x$ est \emph{archim\'edien}. 

Soit~$V$ une partie de~$\E{n}{\cA}$. On appelle \emph{partie ultram\'etrique} de~$V$ l'ensemble
\[V_{\um} := \{x \in V : x \textrm{ est ultram\'etrique}\}.\]
On appelle \emph{partie archim\'edienne} de~$V$ l'ensemble
\[V_{\arc} := \{x \in V : x \textrm{ est archim\'edienne}\}.\]
\end{defi}
\nomenclature[Ik]{$V_{\um}$}{partie ultram\'etrique d'une partie $V$ de $\E{n}{\cA}$}
\nomenclature[Il]{$V_{\arc}$}{partie archim\'edienne de $V$}

\begin{rema}\label{rem:umfermee}
On a $V = V_{\um} \sqcup V_{\arc}$. 

Un r\'esultat classique assure qu'un point $x \in \E{n}{\cA}$ est ultram\'etrique si, et seulement si,
\[ \forall n\in \Z,\ |n|_{x} \le 1.\]
Il s'ensuit que $V_{\um}$ et~$V_{\arc}$ sont respectivement ferm\'ee et ouverte dans~$V$. 
\end{rema}

Les possibilit\'es pour les corps r\'esiduels compl\'et\'es en les points archim\'ediens sont tr\`es limit\'ees. Il s'agit l\`a encore d'un r\'esultat classique (\cf~\cite[VI, \S 6, \no 6, th\'eor\`eme~2]{BourbakiAC57} par exemple). Rappelons que l'on note~$\va_{\infty}$ la valeur absolue usuelle sur~$\R$ ou~$\C$.

\begin{theo}\label{th:vaarchimedienne}\index{Corps!residuel complete@r\'esiduel compl\'et\'e!archimedien@archim\'edien}
Soit $(k,\va)$ un corps valu\'e archim\'edien complet. Alors, il existe un corps $K$ \'egal \`a~$\R$ ou~$\C$, un isomorphisme $j \colon k \simto K$ et un nombre r\'eel $\eps \in \intof{0,1}$ tel que 
\[ \forall a \in k,\ |a| = |j(a)|_{\infty}^\eps.\]
\end{theo}

\begin{nota}\label{not:epsilon} 
\nomenclature[Ig]{$\eps(x)$}{pour $x$ archim\'edien, \'el\'ement de~$\intof{0,1}$ tel que $\va_{x} = \va_{\infty}^{\eps(x)}$}
Pour tout $x \in (\E{n}{\cA})_{\arc}$, on note $\eps(x)$ l'unique \'el\'ement de~$\intof{0,1}$ tel que la valeur absolue canonique sur~$\cH(x)$ s'identifie \`a~$\va_{\infty}^{\eps(x)}$ \textit{via} l'isomorphisme entre~$\cH(x)$ et~$\R$ ou~$\C$ fourni par le th\'eor\`eme~\ref{th:vaarchimedienne}.
\end{nota}

\begin{lemm}\label{lem:epscontinue}
L'application $\eps \colon (\E{n}{\cA})_{\arc} \too \intof{0,1}$ est continue.
\end{lemm}
\begin{proof}
Pour tout $x \in (\E{n}{\cA})_{\arc}$, on a $\eps(x) = \log(|2(x)|)/\log(2)$. Or, l'application~$|2|$ est continue, par d\'efinition de la topologie. Le r\'esultat s'ensuit.
\end{proof}

%

\medbreak

Munissons maintenant $\E{n}{\cA}$ d'un faisceau structural.

\begin{nota}\index{Norme!uniforme} 
\nomenclature[Im]{$\nm_{V}$}{semi-norme uniforme sur une partie compacte~$V$ de~$\E{n}{\cA}$}
Pour toute partie compacte~$V$ de~$\E{n}{\cA}$, on note~$\nm_{V}$ la \emph{semi-norme uniforme} sur~$V$, c'est-\`a-dire la semi-norme d\'efinie par
\[\forall P \in \cA[T_{1},\dotsc,T_{n}],\ \|P\|_{V} = \max_{x\in V}(|P(x)|).\]
\end{nota}

\begin{defi}\label{def:fracrat}\index{Fraction rationnelle sans p\^oles|textbf}
\nomenclature[In]{$\cK(V)$}{fractions rationnelles sans p\^oles sur~$V$} 
Pour toute partie compacte~$V$ de~$\E{n}{\cA}$, posons 
\[ S_{V} := \{P \in \cA[T_{1},\dotsc,T_{n}] : \forall x \in V, |P(x)| \ne 0\}.\]
Les \'el\'ements de l'anneau 
\[ \cK(V) := S_{V}^{-1} \cA[T_{1},\dotsc,T_{n}]\]
sont appel\'ees \emph{fractions rationnelles sans p\^oles sur~$V$}. 
\end{defi}

La semi-norme~$\nm_{V}$ sur $\cA[T_{1},\dotsc,T_{n}]$ s'\'etend en une semi-norme sur~$\cK(V)$. 
Pour tout $x\in V$, l'application d'\'evaluation $\ev_{x} \colon \cA[T_{1},\dotsc,T_{n}] \to \cH(x)$ s'\'etend \`a~$\cK(V)$. Pour $f\in \cK(V)$, on pose $f(x) := \ev_{x}(f) \in \cH(x)$.\index{Evaluation@\'Evaluation|textbf}%
\nomenclature[Ievb]{}{morphisme d'\'evaluation en~$x$ sur $\cK(V)$}%
\nomenclature[Ifxb]{}{\'evaluation en $x$ d'un \'el\'ement $f$ de $\cK(V)$}

\begin{defi}[\protect{\cite[definition~1.5.1]{Ber1}}]\index{Espace affine analytique!faisceau structural|textbf}
Le \emph{faisceau structural} sur~$\E{n}{\cA}$ est le faisceau d'anneaux qui \`a tout ouvert~$U$ de~$\E{n}{\cA}$ associe l'ensemble des applications
\[ f \colon U \too \bigsqcup_{x\in U} \cH(x)\]
v\'erifiant les propri\'et\'es suivantes~:
\begin{enumerate}[i)]
\item pour tout $x\in U$, on a $f(x) \in \cH(x)$ ;
\item $f$ est localement limite uniforme de fractions rationnelles sans p\^oles, au sens o\`u, pour tout $x \in U$, il existe un voisinage compact~$V$ de~$x$ dans~$U$ et une suite $(R_{i})_{i \ge 0}$ d'\'el\'ements de~$\cK(V)$ qui converge vers~$f_{|V}$ pour~$\nm_{V}$.
\end{enumerate}
Le caract\`ere local de la d\'efinition assure que l'on obtient bien ainsi un faisceau. Nous le noterons $\cO_{\E{n}{\cA}}$ ou simplement~$\cO$.%
\nomenclature[Io]{$\cO_{\E{n}{\cA}}$}{faisceau structural sur $\E{n}{\cA}$}%
\nomenclature[Ifxc]{}{\'evaluation en~$x$ d'une section~$f$ de~$\cO_{\E{n}{\cA}}$}
\end{defi}


\begin{lemm}\label{lem:M(tau)Anfoncteur}
Soit $\tau \colon \cA \to \cA'$ un morphisme born\'e d'anneaux de Banach. 

Soit $U$ un ouvert de~$\cM(\cA)$ et soit $f \colon U \to \bigsqcup_{x\in U} \cH(x)$ une application. Posons $U' :=  (\E{n}{\tau})^{-1}(U)$ et notons 
\[(\E{n}{\tau})^\ast f \colon (\E{n}{\tau})^{-1}(U) \too \bigsqcup_{x'\in U} \cH(x')\] 
l'application obtenue en composant~$f$ par $\E{n}{\tau}$ et les $\tau_{x',x}$ (\cf~remarque~\ref{rem:M(tau)Antopo}). Si $f$ appartient \`a~$\cO(U)$, alors $(\E{n}{\tau})^\ast f$ appartient \`a~$\cO(U')$.

On obtient ainsi un foncteur contravariant $\E{n}{\bullet}$ de la cat\'egorie des anneaux de Banach muni des morphismes born\'es vers celle des espaces localement annel\'es.
\qed
\end{lemm}

\begin{rema}\label{rem:completionisomorphisme}
Consid\'erons le cas particulier du lemme~\ref{lem:M(tau)Anfoncteur} o\`u le morphisme~$\tau$ est le morphisme de compl\'etion $\cA \to \hatcA$. D'apr\`es la remarque~\ref{rem:M(tau)Antopo}, l'application induite entre les espaces affines analytiques pr\'eserve les espaces topologiques et les corps r\'esiduels compl\'et\'es. On en d\'eduit que le morphisme d'espaces localement annel\'es induit
\[\E{n}{\cA\kern-0.45em\hat{\phantom{\cA}}} \too \E{n}{\vphantom{\hat\cA}\cA}\]
est un isomorphisme.
\end{rema}

\medbreak

Rappelons que les espaces affines analytiques sur les corps valu\'es archim\'ediens complets sont tr\`es proches des espaces analytiques complexes usuels (\cf~\cite[\S 1.5.4]{Ber1}).

\begin{lemm}\label{lem:AnC}\index{Espace affine analytique!sur un corps archim\'edien}
Munissons~$\C$ et~$\R$ d'une puissance de la valeur absolue usuelle. On a alors des isomorphismes d'espaces localement annel\'es naturels
\[ \E{n}{\C} \simeq \C^n \textrm{ et } \E{n}{\R} \simeq \C^n/\Gal(\C/\R).\]
\qed
\end{lemm}

\begin{rema}\label{rem:gammaR}%
\nomenclature[Is]{$\gamma_{\R}$}{pour $U$ ouvert de la partie archim\'edienne de~$\E{n}{\cA}$, morphisme canonique $\R \to \cO(U)$}
Soit~$U$ un ouvert de~$\E{n}{\cA}$ contenu dans la partie archim\'edienne. On a alors un morphisme injectif canonique
\[\gamma_{\R} \colon \R \too \cO(U).\]
En effet, le morphisme $\Z \to \cA$ induit un morphisme $\Z \to \cO(U)$ et on v\'erifie que ce dernier s'\'etend \`a~$\Q$ puis \`a~$\R$.
\end{rema}

\begin{lemm}\label{lem:evaluationcontinueaffine}
Pour tout ouvert~$U$ de~$\E{n}{\cA}$ et tout \'el\'ement~$f$ de~$\cO(U)$, l'application 
\[\fonction{|f|}{U}{\R_{\ge 0}}{x}{|f(x)|}\]
est continue.
\end{lemm}
\begin{proof}
Si~$f$ est un polyn\^ome, le r\'esultat d\'ecoule de la d\'efinition de la topologie. On en d\'eduit le r\'esultat pour les fractions rationnelles sans p\^oles, puis pour des fonctions arbitraires, la continuit\'e \'etant pr\'eserv\'ee par limite uniforme.
\end{proof}

Le r\'esultat suivant d\'ecoule directement des d\'efinitions.

\begin{lemm}\index{Corps!residuel@r\'esiduel|textbf}%
\nomenclature[Ioa]{$\cO_{x}$}{anneau local en un point~$x$ de~$\cE{n}{\cA}$}%
\nomenclature[Iob]{$\m_{x}$}{id\'eal maximal de $\cO_{x}$}%
\nomenclature[Ioc]{$\kappa(x)$}{$ = \cO_{x}/\m_{x}$, corps r\'esiduel de $x$}%
Pour tout point~$x$ de~$\E{n}{\cA}$, le germe~$\cO_{x}$ est un anneau local. Son id\'eal maximal~$\m_{x}$ est constitu\'e des fonctions qui s'annulent au point~$x$. Le corps $\kappa(x) := \cO_{x}/\m_{x}$ est un sous-corps dense de~$\cH(x)$. 
\qed
\end{lemm}

\begin{defi}\label{def:projection}\index{Projection|textbf}%
\nomenclature[Iaproj1]{$\pi_{n}$}{projection $\pi_{n} \colon \E{n}{\cA}\to\cM(\cA)$}%
\nomenclature[Iaproj2]{$\pi_{n,m}$}{projection $\pi_{n,m} \colon \E{n}{\cA}\to \E{m}{\cA}$ sur les $m$~premi\`eres coordonn\'ees}%
On appelle \emph{projection sur la base} le morphisme d'espaces localement annel\'es
\[\pi_{n} \colon \E{n}{\cA}\too\cM(\cA)\]
induit par le morphisme $\cA \to \cA[T_{1},\dotsc,T_{n}]$.

Plus g\'en\'eralement, pour tout $m\in \cn{0}{n}$, on appelle \emph{projection sur les $m$~premi\`eres coordonn\'ees} le morphisme d'espaces localement annel\'es
\[\pi_{n,m} \colon \E{n}{\cA}\to \E{m}{\cA}\]
induit par le morphisme $\cA[T_{1},\dotsc,T_{m}] \to \cA[T_{1},\dotsc,T_{n}]$.
\end{defi}

On verra plus loin que les projections sont des morphismes analytiques (\cf~exemple~\ref{ex:structural}).

\medbreak

Le r\'esultat suivant s'obtient sans peine \`a partir des d\'efinitions.

\begin{lemm}\index{Projection!fibre d'une}
Pour tout $m\in \cn{0}{n}$ et tout $y \in \E{m}{\cA}$, on a un hom\'eomorphisme naturel
\[ \E{n-m}{\cH(y)} \simtoo \pi_{n,m}^{-1}(y). \]
\qed
\end{lemm}


Terminons par un r\'esultat topologique qui nous sera utile dans la suite. Un cas particulier (dans le cas o\`u $\cA$ est un anneau d'entiers de corps de nombres) figure dans~\cite[proposition~3.4.1]{A1Z} et la d\'emonstration s'\'etend sans peine au cas g\'en\'eral.

\begin{nota}%
\nomenclature[Ih]{$x^\eps$}{point associ\'e \`a la semi-norme $\va_{x}^\eps$}%
Soit $b\in \cM(\cA)$. Pour tout $x\in \pi_{n}^{-1}(b)$ et tout $\eps \in I_{b}$, l'application~$\va_{x}^\eps$ est une semi-norme multiplicative sur $\cA[T_{1},\dotsc,T_{n}]$ born\'ee sur~$\cA$. On note~$x^\eps$ le point de~$\E{n}{\cA}$ associ\'e.
\end{nota}

\begin{lemm}\label{lem:flot}
Soit $b\in \cM(\cA)$. L'application 
\[{\renewcommand{\arraystretch}{1.3}\begin{array}{ccc}
\pi_{n}^{-1}(b) \times I_{b} & \too & \E{n}{\cA}\\
(x,\eps) & \mapstoo & x^\eps
\end{array}}\]
r\'ealise un hom\'eomorphisme sur son image.
\qed
\end{lemm}

\section{Parties rationnelles et spectralement convexes}\label{sec:spconvexe}

Soit~$\cA$ un anneau de Banach et soit $n\in \N$.

\begin{defi}\index{Partie!rationnelle|textbf}\index{Partie!pro-rationnelle|textbf}
On dit qu’une partie compacte~$V$ de~$\E{n}{\cA}$ est \emph{rationnelle} s’il existe $P_{1},\dotsc,P_{p},Q \in \cA[T_{1},\dotsc,T_{n}]$ ne s’annulant pas simultanément sur $\E{n}{\cA}$ et $r_{1},\dotsc,r_{p} \in \R_{>0}$ tels que 
\[ V = \bigcap_{1\le i\le p}  \{x\in \E{n}{\cA}: |P_i(x)|\le  r_i \, |Q(x)|\} .\]

On dit qu’une partie compacte~$V$ de~$\E{n}{\cA}$ est \emph{pro-rationnelle} si elle est intersection de parties compactes rationnelles.
\end{defi}

\begin{exem}\label{ex:pointprorationnel}
Tout point de~$\E{n}{\cA}$ est une partie pro-rationnelle. En effet, pour tout $x \in \E{n}{\cA}$, on a
\[ \{x\} = \bigcap_{P \in \cA[T_{1},\dotsc,T_{n}]} \{y \in \E{n}{\cA} : |P(y)| = |P(x)|\}.\]
\end{exem}

Un calcul classique (\cf~\cite[proposition~7.2.3/7]{BGR}) montre qu'une intersection finie de parties compactes rationnelles est encore compacte rationnelle. Il suit alors de la d\'efinition de la topologie de~$\E{n}{\cA}$ que tout point poss\`ede une base de voisinages form\'ee de parties compactes rationnelles.

\begin{nota}\label{nota:BV}
\nomenclature[Ina]{$\cB(V)$}{s\'epar\'e compl\'et\'e de~$\cK(V)$ pour $\nm_{V}$}
Pour toute partie compacte non vide~$V$ de~$\E{n}{\cA}$, on note~$\cB(V)$ le s\'epar\'e compl\'et\'e de~$\cK(V)$ pour la semi-norme uniforme~$\nm_{V}$ sur~$V$.

C'est un anneau de Banach uniforme.
\end{nota}

\begin{exem}\label{ex:Bpoint}
Pour tout point~$x$ de~$\E{n}{\cA}$, on a $\cB(\{x\}) = \cH(x)$. 
\end{exem}

\begin{lemm}\label{lem:factorisationBV}
Soit~$V$ une partie compacte  de~$\E{n}{\cA}$. Soit~$\cB$ une $\cA$-alg\`ebre de Banach uniforme et soit $f \colon \cA[T_1,\ldots,T_n]\to\cB$ un morphisme de~$\cA$-alg\`ebres tel que l'image du morphisme induit $\varphi \colon \cM(\cB)\to\E{n}{\cA}$ soit contenue dans~$V$.
Alors, il existe un unique morphisme de $\cA$-alg\`ebres born\'e $\cB(V) \to \cB$ qui fait commuter le diagramme
\[\begin{tikzcd}
\cA[T_{1},\dotsc,T_{n}] \arrow[d] \arrow[r, "f"] & \cB. \\
\cB(V) \arrow[ru]
\end{tikzcd}\]
\end{lemm}
\begin{proof}

Soit $P \in S_{V}$. Pour tout $y\in \cM(\cB)$, on a $|f(P)(y)| = |P(\varphi(y))| \ne 0$ car $\varphi(y) \in V$. On en d\'eduit que $f(P)$ est inversible dans~$\cB$, d'apr\`es \cite[theorem~1.2.1]{Ber1}. Par cons\'equent, $f$ s'\'etend en un morphisme
\[f' \colon \cK(V) = S_{V}^{-1} \cA[T_{1},\dotsc,T_{n}] \too \cB,\]
et ce de fa\c{c}on n\'ecessairement unique.

Soient $P \in \cA[T_{1},\dotsc,T_{n}]$ et $Q\in S_{V}$. Pour tout $y\in \cM(\cB)$, on a
\[ \left|f'\left(\frac P Q\right)(y)\right| = \frac{|f(P)(y)|}{|f(Q)(y)|} = \frac{|P(\varphi(y))|}{|Q(\varphi(y))|} =\left|\frac P Q(\varphi(y))\right|. \]
Puisque~$\cB$ est uniforme et que l'image de~$\varphi$ est contenue dans~$V$, on en d\'eduit que $\|f'(P/Q)\| \le \|P/Q\|_{V}$.
Par cons\'equent, le morphisme~$f'$ se prolonge en un morphisme born\'e $\cB(V) \to \cB$, et ce de fa\c{c}on unique.
\end{proof}

Remarquons que, pour toute partie compacte~$V$ de~$\E{n}{\cA}$, le morphisme naturel $\cA[T_{1},\dotsc,T_{n}] \to \cB(V)$ induit un morphisme d'espaces annel\'es $\varphi_{V} \colon \cM(\cB(V))\to \E{n}{\cA}$.

\begin{nota}\index{Interieur@Int\'erieur|textbf}
\nomenclature[D]{$\mathring V$}{int\'erieur d'une partie $V$ d'un espace topologique}
Soient $T$ un espace topologique et $V$ une partie de~$T$. On note~$\mathring V$ l'int\'erieur de~$V$ dans~$T$.
\end{nota}

\begin{theo}[\protect{\cite[th\'eor\`eme~1.2.11]{A1Z}}]\label{thm:rationnel}
Soit~$V$ une partie compacte pro-rationnelle de~$\E{n}{\cA}$. Alors le morphisme $\varphi_{V}$ induit 
\begin{enumerate}[i)]
\item un hom\'eomorphisme $\cM(\cB(V)) \simto V$ ;
\item un isomorphisme d'espaces annel\'es $\varphi_{V}^{-1}(\mathring V) \simto \mathring V$.
\end{enumerate} 
\qed
\end{theo}

\begin{defi}\label{def:spconvexe}\index{Partie!spectralement convexe|textbf}
On dit qu'une partie compacte~$V$ de~$\E{n}{\cA}$ est \emph{spectralement convexe} si elle satisfait les conclusions du th\'eor\`eme~\ref{thm:rationnel}.
\end{defi}

\begin{prop}[\protect{\cite[remarque~1.2.13]{A1Z}}]\label{crit_spectral}
Une partie compacte~$V$ de~$\E{n}{\cA}$ est spectralement convexe si, et seulement si, l'image de~$\varphi_{V}$ est contenue dans~$V$.
\qed
\end{prop}

\begin{coro}
L'ensemble des parties compactes spectralement convexes de~$\E{n}{\cA}$ est stable par intersection.
\qed
\end{coro}

\`A l'aide du lemme~\ref{lem:factorisationBV}, nous en d\'eduisons une autre caract\'erisation des parties spectralement convexes. 

\begin{prop}\label{crit_spectral_pu}
Soit~$V$ une partie compacte  de~$\E{n}{\cA}$. Les propositions suivantes sont \'equivalentes~:
\begin{enumerate}[i)]
\item $V$ est spectralement convexe~;
\item pour toute $\cA$-alg\`ebre de Banach uniforme~$\cB$ et tout morphisme de~$\cA$-alg\`ebres $f \colon \cA[T_1,\ldots,T_n]\to\cB$, l'image du morphisme induit $\cM(\cB)\to\E{n}{\cA}$ est contenue dans~$V$ si, et seulement si, il existe un unique morphisme de $\cA$-alg\`ebres born\'e $\cB(V) \to \cB$ qui fait commuter le diagramme
\[\begin{tikzcd}
\cA[T_{1},\dotsc,T_{n}] \arrow[d] \arrow[r, "f"] & \cB. \\
\cB(V) \arrow[ru]
\end{tikzcd}\]
\end{enumerate}
\qed
\end{prop}

\begin{rema}\label{rem:affvsspconvexe}
Nous ne disposons pas de notion de domaine affino\"ide dans notre cadre, mais celle de partie spectralement convexe pourra parfois jouer ce r\^ole. La propri\'et\'e universelle~ii) qui figure dans la proposition~\ref{crit_spectral_pu} est proche de celle qui d\'efinit les domaines affino\"ides (\cf~\cite[definition~2.2.1]{Ber1}), avec cependant la diff\'erence notable que l'on consid\`ere uniquement des morphismes \`a valeurs dans des alg\`ebres de Banach uniformes. Notons que, lorsque l'anneau de Banach~$\cA$ est un corps valu\'e, de nombreuses parties spectralement convexes (ferm\'es de Zariski, points, etc.) ne sont pas n\'ecessairement des domaines affino\"ides.
\end{rema}

\begin{rema}\label{spectral}\label{spectralement_conv}\index{Projection!fibre d'une}
Soit $m\in \cn{0}{n}$ et consid\'erons le morphisme de projection sur les $m$~premi\`eres coordonn\'ees $\pi_{n,m} \colon \E{n}{\cA} \to \E{m}{\cA}$  (\cf~d\'efinition~\ref{def:projection}). Soit~$V$ un ensemble compact spectralement convexe de~$\E{m}{\cA}$. D'apr\`es \cite[proposition~1.2.15]{A1Z}, on a un hom\'eomorphisme naturel
\[\E{n-m}{\cB(V)} \simtoo \pi_{n,m}^{-1}(V)\]
et m\^eme un isomorphisme d'espaces annel\'es au-dessus de~$\mathring V$.

Ce r\'esultat se r\'ev\`ele tr\`es utile dans les d\'emonstrations par r\'ecurrence sur la dimension. Par exemple, si l'on souhaite \'etudier une propri\'et\'e locale en un point~$x$ de~$\E{n}{\cA}$, il permet de se ramener au cas d'un point $x'$ de $\E{1}{\cB}$, o\`u $\cB$ est un anneau de Banach de la forme~$\cB(V)$ pour un voisinage compact spectralement convexe de~$\pi_{n,n-1}(x)$ dans $\E{n-1}{\cA}$. 
\end{rema}

\section{Disques, couronnes et domaines polynomiaux}\label{sec:dcl}

Soit~$\cA$ un anneau de Banach. Posons $B :=\cM(\cA)$ et $X := \AunA$, avec coordonn\'ee~$T$. Notons $\pi \colon X \to B$ la projection sur la base.

\begin{nota}\index{Disque|textbf}\index{Couronne|textbf}\index{Domaine polynomial|textbf}%
\nomenclature[Jb]{$D_{V}(P,t)$}{domaine polynomial ouvert relatif au-dessus de $V$}%
\nomenclature[Jba]{$\overline{D}_{V}(P,t)$}{domaine polynomial ferm\'e relatif au-dessus de $V$}%
\nomenclature[Jc]{$C_{V}(P,s,t)$}{domaine polynomial ouvert relatif au-dessus de $V$}%
\nomenclature[Jd]{$\overline{C}_{V}(P,s,t)$}{domaine polynomial ferm\'e relatif au-dessus de $V$}%
Soit~$V$ une partie de~$B$. Soient $P\in \cA[T]$ et $s,t\in \R$. On pose
\begin{align*}
D_{V}(P,t) &:= \{x\in X : \pi(x)\in V, |P(x)| < t\}~;\\
\overline{D}_{V}(P,t) &:= \{x\in X : \pi(x)\in V, |P(x)| \le t\}~;\\
C_{V}(P,s,t) &:= \{x\in X : \pi(x)\in V, s < |P(x)| < t\}~;\\
\overline{C}_{V}(P,s,t) &:= \{x\in X : \pi(x)\in V, s \le |P(x)| \le t\}.
\end{align*}
Une partie de la forme $D_{V}(P,t)$ ou $C_{V}(P,s,t)$ sera appel\'ee \emph{domaine polynomial ouvert relatif au-dessus de~$V$}. Une partie de la forme $\overline D_{V}(P,t)$ ou $\overline C_{V}(P,s,t)$ sera appel\'ee \emph{domaine polynomial ferm\'e relatif au-dessus de~$V$}. 

Lorsque $P = T$, on s'autorise \`a supprimer le~$P$ de la notation et \`a \'ecrire simplement $D_{V}(t)$, $\oD_{V}(t)$, $C_{V}(s,t)$ ou $\overline{C}_{V}(s,t)$. On parlera de \emph{disque relatif} dans les deux premiers cas et de \emph{couronne relative} dans les deux derniers.
\nomenclature[Je]{$D_{V}(t)$}{disque ouvert relatif au-dessus de $V$}
\nomenclature[Jf]{$\overline{D}_{V}(t)$}{disque ferm\'e relatif au-dessus de $V$}
\nomenclature[Jg]{$C_{V}(s,t)$}{couronne ouverte relative au-dessus de $V$}
\nomenclature[Jh]{$\overline{C}_{V}(s,t)$}{couronne ferm\'ee relative au-dessus de $V$}
\end{nota}

Dans tout le texte, nous nous autoriserons \`a parler de sections d'un faisceau sur une partie qui n'est pas n\'ecessairement ouverte.

\begin{nota}\label{nota:surconvergent}\index{Faisceau!surconvergent}
\nomenclature[Ira]{$\cO(V)$}{fonctions surconvergentes sur un compact~$V$ de~$\E{n}{\cA}$}
\nomenclature[Ea]{$\cF(V)$}{sections d'un faisceau coh\'erent~$\cF$ sur une partie~$V$}
Soient~$T$ un espace topologique et~$\cF$ un faisceau d'ensembles sur~$T$. Consid\'erons l'espace \'etale associ\'e $(\widetilde{\cF},p)$, o\`u~$\widetilde{\cF}$ est un espace topologique et $p \colon \widetilde{\cF} \to T$ un hom\'eomorphisme local. Pour toute partie~$V$ de~$T$, on note $\cF(V)$ l'ensemble des sections continues de l'application~$p$ au-dessus de~$V$.  
\end{nota}

Lorsque~$V$ est ouvert, on retrouve l'ensemble~$\cF(V)$ usuel, \cf~\cite[II, section~1.2]{GodementTATF}. En g\'en\'eral, $\cF(V)$ s'interpr\`ete comme un ensemble de sections surconvergentes. Pr\'ecis\'ement, d'apr\`es \cite[II, corollaire~1 du th\'eor\`eme~3.3.1]{GodementTATF}, si $V$ poss\`ede une base de voisinages paracompacts, alors l'application
\[\colim_{U \supset V} \cF(U) \too \cF(V),\] 
o\`u $U$ parcourt l'ensemble des voisinages ouverts de~$V$ dans~$T$, est bijective. L'hypoth\`ese sur~$V$ est, par exemple, satisfaite lorsque $T$ est localement compact (par exemple un espace analytique) et $V$ compacte. Hormis le cas ouvert, c'est le cas que nous rencontrerons le plus souvent.

\begin{exem}\label{ex:seriessurconvergentes}
Supposons que $(\cA,\nm)$ est un corps valu\'e complet $(k,\va)$. Pour tout $t\in \R_{\ge 0}$, $\cO(\overline{D}_{\cM(k)}(t))$ est l'anneau des s\'eries \`a coefficients dans~$k$ de rayon de convergence strictement sup\'erieur \`a~$t$.
\end{exem}

\subsection{Alg\`ebres de disques et de couronnes}
\index{Disque!algebre@alg\`ebre d'un|(}
\index{Couronne!algebre@alg\`ebre d'une|(}

\begin{nota}\label{nota:ATt}
Soit~$(\cC,\nm)$ un anneau muni d'une norme sous-multiplicative.

Soit $t \in \R_{>0}$. On note $\cC\la |T| \le t\ra$ l'alg\`ebre constitu\'ee des s\'eries de la forme
\[\sum_{n\in \N} a_{n}\, T^n \in \cC[\![T]\!]\]
telles que la s\'erie $\sum_{n\in \N} \|a_{n}\|\, t^n$ converge et on la munit de la norme d\'efinie par 
\[\left\| \sum_{n\in \N} a_{n}\, T^n\right\|_{t} := \sum_{n\in \N} \|a_{n}\|\, t^n.\]
 Si~$\cC$ est complet, $\cC\la |T| \le t\ra$ est compl\`ete.

Soient $s,t \in \R_{>0}$ tels que $s \le t$. On note $\cC\la s \le |T| \le t\ra$ l'alg\`ebre constitu\'ee des s\'eries de la forme
\[\sum_{n\in \Z} a_{n}\, T^n \in \cC[\![T,T^{-1}]\!]\]
telles que la famille $(\|a_{n}\|\, \max(s^n,t^n))_{n\in \Z}$ soit sommable et on la munit de la norme d\'efinie par 
\[\left\| \sum_{n\in \Z} a_{n}\, T^n\right\|_{s,t} := \sum_{n\in \Z} \|a_{n}\|\, \max(s^n,t^n).\]
 Si~$\cC$ est complet, $\cC\la s\le |T| \le t\ra$ est compl\`ete.

Par commodit\'e, on pose $\cC\la 0 \le |T| \le t\ra := \cC\la |T| \le t\ra$ et $\nm_{0,t} := \nm_{t}$.%
\nomenclature[Ca]{$\cA\la \vert T\vert \le t \ra$}{s\'eries \`a coefficients dans~$\cA$ normalement convergentes sur $\oD(t)$}%
\nomenclature[Bhaa]{$\nm_{t}$}{norme 1 sur $\cA\la \vert T\vert \le t\ra$}%
\nomenclature[Cb]{$\cA\la s \le \vert T\vert \le t\ra$}{s\'eries de Laurent \`a coefficients dans~$\cA$ normalement convergentes sur $\oC(s,t)$}%
\nomenclature[Bhb]{$\nm_{s,t}$}{norme 1 sur $\cA\la s\le  \vert T\vert \le t\ra$}%
\end{nota}

\begin{prop}[\protect{\cite[proposition~2.1.1]{A1Z}}]\label{prop:spectreseries}\index{Norme!spectrale}\index{Norme!uniforme}\index{Norme!comparaison}
Supposons que~$\cA$ est uniforme. Pour $s,t \in \R_{\ge 0}$ tels que $0= s <t$ ou $0<s\le t$, on a un isomorphisme canonique
\[\cM(\cA\la s \le |T| \le t\ra) \simtoo \overline{C}(s,t).\]
En particulier, la norme spectrale de~$\nm_{s,t}$ s'identifie avec la norme uniforme sur $\overline{C}(s,t)$. 
\qed
\end{prop}

Pour \'eviter d'avoir \`a distinguer dans chaque \'enonc\'e les disques et les couronnes, introduisons une notation.

\begin{nota}%
\nomenclature[Bha]{$s\prec t$}{relation entre $s,t \in \R_{\ge0}$ satisfaite si $s=0$ ou $s <t$}
Pour $s,t \in \R_{\ge0}$, on note
\[s \prec t \textrm{ si } s =0 \textrm{ ou } s< t.\]
\end{nota}

Rappelons que nous avons d\'efini un anneau~$\cB(V)$ \`a la notation~\ref{nota:BV}.

\begin{prop}[\protect{\cite[lemme~2.1.2 et proposition~2.1.3]{A1Z}}]\label{prop:restrictionserie}\index{Norme!comparaison}
\index{Algebre@Alg\`ebre!d'un disque|see{Disque}}\index{Algebre@Alg\`ebre!d'une couronne|see{Couronne}}\index{Algebre@Alg\`ebre!d'un domaine polynomial|see{Domaine polynomial}}
Supposons que~$\cA$ est uniforme. 
Soient $s,u \in \R_{\ge 0}$ et $t,v \in \R_{>0}$ tels que $s\prec u \le v < t$. Soit $f = \sum_{n\in \Z} a_{n}\, T^n \in \cA\la s \le |T| \le t\ra$. Alors, on a
\[\forall n\in \Z,\ \|a_{n}\| \max(u^n,v^n) \le \|f\|_{\overline{C}(s,t)}\]
et
\[ \|f\|_{u,v} \le \left( \frac{s}{u-s} + \frac{t}{t-v}\right) \|f\|_{\overline{C}(s,t)},\]
avec la convention que $s/(u-s) = 0$ lorsque $s=0$.

En particulier, on a un morphisme injectif naturel
\[\cB(\overline{C}(s,t)) \too \cA\la u\le |T| \le v\ra.\]
\end{prop}

Rappelons un r\'esultat permettant de d\'ecrire les fonctions au voisinage des disques.

\begin{prop}[\protect{\cite[corollaire~2.8]{EtudeLocale}}]\label{prop:disqueglobal}
Soit $V$ une partie compacte de~$B$. Soit $t \in \R_{\ge 0}$. Alors, le morphisme de $\cA$-alg\`ebres naturel
\[\colim_{W\supset V,v > t} \cO(W)\la |T|\le v\ra \too \cO(\overline{D}_{V}(t)),\]
o\`u~$W$ parcourt l'ensemble des voisinages compacts de~$V$ dans~$B$, est un isomorphisme.
\qed
\end{prop}

On peut obtenir un analogue de ce r\'esultat pour les couronnes contenues dans la partie ultram\'etrique.

\begin{prop}\label{prop:bumCst}
Soit $b \in B_{\um}$. Soient $s,t \in \R_{> 0}$ tels que $s\le t$. Alors, le morphisme de $\cA$-alg\`ebres naturel
\[\colim_{V\ni b, u < s \le t< v} \cB(V)\la u\le |T|\le v\ra \too \cO(\overline{C}_{b}(s,t)),\]
o\`u~$V$ parcourt l'ensemble des voisinages compacts de~$b$ dans~$B$, est un isomorphisme.
\end{prop}
\begin{proof}
L'injectivit\'e \'etant claire, il suffit de d\'emontrer la surjectivit\'e. Soit $f \in \cO(\overline{C}_{b}(s,t))$. On peut supposer qu'il existe un voisinage ouvert~$U$ de~$b$ dans~$B$ et des nombres r\'eels $s_{1},t_{1}\in \R_{>0}$ v\'erifiant $s_{1}< s \le t < t_{1}$ tels que $f\in \cO(\overline{C}_{U}(s_{1},t_{1}))$.

Notons~$z$ l'unique point du bord de Shilov de la couronne $\overline{C}_{b}(t,t)$. Par d\'efinition du faisceau structural, il existe un voisinage compact~$W$ de~$z$ dans~$\overline{C}_{U}(s_{1},t_{1})$ et un \'el\'ement~$g$ de~$\cB(W)$ co\"incidant avec~$f$ sur~$W$.
L'int\'erieur de~$W$ contient une couronne de la forme $\overline{C}_{V}(u,u)$, o\`u~$V$ est un voisinage compact spectralement convexe de~$b$ dans~$U$ et $u$ un nombre r\'eel v\'erifiant $t< u <t_{1}$. D'apr\`es \cite[proposition~2.4]{EtudeLocale}, il existe un \'el\'ement~$h$ de $\cB(V)\la |T| = u\ra$ qui co\"incide avec~$g$ sur $\overline{C}_{V}(u,u)$. On peut \'ecrire
\[h = \sum_{n\in\Z} a_{n}\, T^n \in \cB(V)[\![ T,T^{-1} ]\!],\]
o\`u la famille $(\|a_{n}\|_{V}\, {u}^n)_{n\in \Z}$ est sommable. 

Soit~$b' \in V$ et raisonnons dans l'espace $\pi^{-1}(b') \simeq \E{1}{\cH(b')}$. Puisque la restriction de~$f$ appartient \`a $\cO(\overline{C}_{b'}(s_{1},t_{1}))$, elle est d\'eveloppable en s\'erie de Laurent sur~$\cH(b')$ et son d\'eveloppement n'est autre que
\[h(b') = \sum_{n\in\Z} a_{n}(b')\, T^n \in \Hs(b')[\![ T , T^{-1}]\!].\]
Soient $s_{2},t_{2} \in \R_{>0}$ tels que $s_{1}< s_{2}< s \le t < t_{2} < t_{1}$. La proposition~\ref{prop:restrictionserie} assure alors que, pour tout $n\in \Z$, on a 
\[|a_{n}(b')| \max(s_{2}^n,t_{2}^n)\le \left( \frac{s_{1}}{s_{2}-s_{1}} + \frac{t_{1}}{t_{1}-t_{2}} \right) \|f\|_{\overline{C}_{b'}(s_{1},t_{1})},\]
et donc 
\[\|a_{n}\|_{V} \max(s_{2}^n,t_{2}^n)\le \left( \frac{s_{1}}{s_{2}-s_{1}} + \frac{t_{1}}{t_{1}-t_{2}} \right) \|f\|_{\overline{C}_{V}(s_{1},t_{1})}.\]
On en d\'eduit que~$f$ appartient \`a l'image de $\cB(V)\la s_{3}\le |T|\le t_{3}\ra$ pour tous $s_{3},t_{3} \in \R_{>0} $ tels que $s_{2}< s_{3}< s \le t < t_{3} < t_{2}$.
\end{proof}

\begin{rema}
Un r\'esultat du m\^eme type vaut certainement encore dans le cas archim\'edien. Par exemple, pour d\'emontrer que tout \'el\'ement de $\cO(\overline{C}_{b}(s,t))$ poss\`ede un d\'eveloppement \`a valeurs dans~$\cO_{b}$, une strat\'egie naturelle consisterait \`a exprimer chaque coefficient du d\'eveloppement \`a l'aide de la fonction~$f$, autrement dit \`a \'ecrire une formule des r\'esidus en chaque fibre, et \`a montrer que le r\'esultat est une fonction analytique du param\`etre sur la base. Nous n'avons pas souhait\'e pousser plus loin ces consid\'erations, difficile \`a mettre en \oe uvre dans le cas d'une base $\cM(\cA)$ g\'en\'erale, mais reviendrons sur ce r\'esultat de d\'eveloppement plus tard, lorsque nous aurons la th\'eorie des espaces de Stein \`a notre disposition, \cf~proposition~\ref{prop:bCstStein}.

\end{rema}

En appliquant la proposition~\ref{prop:bumCst} pr\'ec\'edente fibre \`a fibre et en recollant, on obtient le r\'esultat suivant.

\begin{coro}\label{cor:couronneglobale}
Soit~$V$ une partie compacte de~$B_{\um}$. Soient $s,t\in\R_{> 0}$ tels que $s\le t$. Alors le morphisme de $\cA$-alg\`ebres naturel
\[\colim_{W\supset V,u < s \le t < v} \cO(W)\la u\le |T|\le v\ra \too \cO(\overline{C}_{V}(s,t)),\]
o\`u~$W$ parcourt l'ensemble des voisinages compacts de~$V$ dans~$B$, est un isomorphisme.
\end{coro}

Dans le cadre ultram\'etrique, on peut \'egalement obtenir des r\'esultats sur les compl\'et\'es des anneaux de sections globales sur les couronnes. Introduisons au pr\'ealable quelques d\'efinitions sp\'ecifiques.

\begin{nota}%
\nomenclature[Cc]{$\cA\{ \vert  T\vert  \le t\}$}{pour $\cA$ ultram\'etrique, s\'eries \`a coefficients dans~$\cA$ convergentes sur $\oD(t)$}
\nomenclature[Bhc]{$\nm_{t}$}{norme $\infty$ sur $\cA\{ \vert T\vert  \le t\}$}
\nomenclature[Cd]{$\cA\{ s \le \vert T\vert  \le t\}$}{pour $\cA$ ultram\'etrique, s\'eries de Laurent \`a coefficients dans~$\cA$ convergentes sur $\oC(s,t)$}
\nomenclature[Bhd]{$\nm_{s,t}$}{norme $\infty$ sur $\cA\{ s\le  \vert T\vert  \le t\}$}
 Supposons que la norme sur~$\cA$ est ultram\'etrique.
Soit $t \in \R_{>0}$. On note $\cA\{ |T| \le t\}$ l'alg\`ebre constitu\'ee des s\'eries de la forme
\[\sum_{n\in \N} a_{n}\, T^n \in \cA[\![T]\!]\]
telles que $\lim_{n \to +\infty} \|a_{n}\|\, t^n = 0$.
Elle est compl\`ete pour la norme d\'efinie par 
\[\left\| \sum_{n\in \N} a_{n}\, T^n\right\|_{t} := \max_{n\in \N} \|a_{n}\|\, t^n.\]

Soient $s,t \in \R_{>0}$ tels que $s \le t$. On note $\cA\{ s \le |T| \le t\}$ l'alg\`ebre constitu\'ee des s\'eries de la forme
\[\sum_{n\in \Z} a_{n}\, T^n \in \cA[\![T,T^{-1}]\!]\]
telles que $\lim_{n \to +\infty} \|a_{n}\|\, t^n = \lim_{n \to -\infty} \|a_{n}\|\, s^n = 0$.
Elle est compl\`ete pour la norme d\'efinie par 
\[\left\| \sum_{n\in \Z} a_{n}\, T^n\right\|_{s,t} := \max_{n\in \Z} \|a_{n}\|\, \max(s^n,t^n).\]
Par commodit\'e, on pose $\cA\{ 0 \le |T| \le t\} := \cA\{ |T| \le t\}$ et $\nm_{0,t} := \nm_{t}$.
\end{nota}

\begin{nota}
\nomenclature[Irb]{$\overline{\cO(V)}$}{s\'epar\'e compl\'et\'e de~$\cO(V)$ pour $\nm_{V}$}
Pour toute partie compacte~$W$ de~$B$ ou de~$X$, on note $\overline{\cO(W)}$ le s\'epar\'e compl\'et\'e de~$\cO(W)$ pour la semi-norme uniforme~$\nm_{W}$ sur~$W$.
\end{nota}

\begin{lemm}\label{lem:couronneum}
Soit~$V$ une partie compacte de~$B_{\um}$. Soient $s,t\in\R_{\ge 0}$ tels que $0=s<t$ ou $0<s\le t$. Le morphisme de $\cA$-alg\`ebres naturel 
\[\cO(V)[T] \too \cO(\overline{C}_{V}(s,t))\]
induit un isomorphisme isom\'etrique
\[\overline{\cO(V)}\{s\le |T| \le t\} \simtoo \overline{\cO(\overline{C}_{V}(s,t))}.\]
\end{lemm}
\begin{proof}
On se contentera de traiter le cas o\`u $s>0$. Soit $f \in \cO(\overline{C}_{V}(s,t))$. D'apr\`es le corollaire~\ref{cor:couronneglobale}, $f$~peut s'\'ecrire sous la forme d'une s\'erie $\sum_{n\in\Z} a_{n}\,T^n$ satisfaisant les conditions de convergence ad\'equates. 

La norme uniforme de~$f$ sur $\overline{C}_{V}(s,t)$ est la borne sup\'erieure des normes uniformes de~$f$ sur les fibres au-dessus des points de~$V$. Puisque ces points sont ultram\'etriques, on en d\'eduit que
\[\|f\|_{\overline{C}_{V}(s,t)} = \sup_{b\in V} \bigl(\max_{n\in \Z} \bigl(|a_{n}(b)|\, \max(s^n,t^n)\bigr)\bigr).\]
Puisque~$V$ est compacte, $\overline{C}_{V}(s,t)$ l'est aussi, donc la borne sup\'erieure pr\'ec\'edente est un maximum~:
\begin{align*} 
\|f\|_{\overline{C}_{V}(s,t)} &= \max_{b\in V} \bigl(\max_{n\in \Z} \bigl(|a_{n}(b)|\, \max(s^n,t^n)\bigr)\bigr)\\
&= \max_{n\in \Z} \bigl(\|a_{n}\|_{V}\, \max(s^n,t^n)\bigr).
\end{align*}
Cette derni\`ere norme n'est autre que la norme sur $\overline{\cO(V)} \{ s\le |T|\le t\}$, par d\'efinition. L'expression du compl\'et\'e de $\cO(\overline{C}_{V}(s,t))$ s'en d\'eduit.
\end{proof}
\index{Disque!algebre@alg\`ebre d'un|)}
\index{Couronne!algebre@alg\`ebre d'une|)}

\subsection{Normes sur les domaines polynomiaux}

Dans cette section, nous d\'emontrons un r\'esultat technique sur la norme des fonctions sur les domaines polynomiaux.

\begin{lemm}\label{lem:minorationetaar}
Soit $(k,\va)$ un corps valu\'e ultram\'etrique complet. 
Soient $\alpha\in k$ et $r \in \R_{>0}$. Pour tout $Q = \sum_{j=0}^d b_{j} T^j \in k[T]$, on a
\[\max_{0\le j\le d} (|b_{j}|) \le \max\left(\frac{|\alpha|}{r}, \frac1 r, 1\right)^d \,  \|Q\|_{\overline{D}(\alpha,r)} .\]
\end{lemm}
\begin{proof}
Traitons tout d'abord le cas d'un polyn\^ome de la forme $T-b$ avec $b\in k$. On a
\begin{align*}
\max(|b|, 1) &\le \max (|\alpha|, |b-\alpha|, 1)\\
&  \le\max\left(\frac{|\alpha|}{r}, \frac1 r, 1\right) \cdot  \max(|b-\alpha|,r)\\
&  \le  \max\left(\frac{|\alpha|}{r}, \frac1 r, 1\right) \cdot \|T-b\|_{\overline{D}(\alpha,r)}.
\end{align*}
Posons $C :=  \max(\frac{|\alpha|}{r}, \frac1 r, 1)$.

Soit $Q = \sum_{j=0}^d b_{j} T^j \in k[T]$. Pour d\'emontrer l'in\'egalit\'e souhait\'ee, on peut remplacer~$k$ par une extension valu\'ee compl\`ete et donc supposer que~$k$ est alg\'ebriquement clos. On peut \'egalement supposer que $b_{d}\ne 0$ (car $C\ge 1$). Il existe alors $\beta_{1},\dotsc,\beta_{d} \in k$ tels que 
\[Q = b_{d} \prod_{i=1}^d (T-\beta_{i}).\]
Il d\'ecoule alors des relations coefficients/racines et de l'in\'egalit\'e triangulaire ultram\'etrique que l'on a
\[ \max_{0\le j\le d} (|b_{j}|) \le |b_{d}|\,  \max_{I \subset \cn{1}{d}} \big( \prod_{i\in I} |\beta_{i}| \big) \le |b_{d}|\, \prod_{i=1}^d \max(|\beta_{i}|, 1).\]
En utilisant l'in\'egalit\'e d\'emontr\'ee au d\'ebut de la preuve, on obtient
\[ \max_{0\le j\le d} (|b_{j}|) \le |b_{d}| \,C^d \prod_{i=1}^d \|T-\beta_{i}\|_{\overline{D}(\alpha,r)} \le C^d \, \|Q\|_{\overline{D}(\alpha,r)}.\]
\end{proof}

\begin{prop}\label{prop:minorationDPsk}\index{Domaine polynomial!norme sur un|(}\index{Norme!sur un domaine polynomial|see{Domaine polynomial}}
Soit $(k,\va)$ un corps valu\'e ultram\'etrique complet. Soient $p\in \N^\ast$ et $P = \sum_{i=0}^p a_{i} T^i \in k[T]$ un polyn\^ome de degr\'e~$p$ diff\'erent de~$T^p$. Soit $s \in \R_{>0}$. Posons 
\[\sigma := \max_{0\le i\le p-1} \left(\big| \frac{a_{i}}{a_{p}}\big|^{\frac{1}{p-i}}\right) >0 \] 
et $\rho := \min ( \sigma , s \,\sigma^{1-p})$. Pour tout $Q = \sum_{j=0}^d b_{j} T^j \in k[T]$, nous avons 
\[\max_{0\le j\le d} (|b_{j}|) \le \max\left(\frac{\sigma}{\rho}, \frac1 \rho\right)^d \,  \|Q\|_{\overline{D}(P;s)}.\]
\end{prop}
\begin{proof}
On peut supposer que~$k$ est alg\'ebriquement clos. On peut \'egalement supposer que~$P$ est unitaire. Il existe alors $\alpha_{1},\dotsc,\alpha_{p} \in k$ tels que 
\[P = \prod_{i=1}^p (T-\alpha_{i}).\]
Quitte \`a r\'eordonner les~$\alpha_{i}$, on peut supposer que $|\alpha_{p}| = \max_{1\le i\le p} (|\alpha_{i}|)$. D'apr\`es \cite[proposition~3.1.2/1]{BGR} (ou la th\'eorie des polygones de Newton), on a 
\[|\alpha_{p}| = \max_{0\le i\le p-1} (|a_{i}|^{1/(p-i)}) = \sigma > 0.\] 
Pour tout $r \in (0,|\alpha_{p}|]$ et tout $x\in \overline{D}(\alpha_{p},r)$, on a
\[ |P(x)| =  \prod_{i=1}^{p} |(T-\alpha_{i})(x)| \le r \, \prod_{i=1}^{p-1} \max(|\alpha_{i} - \alpha_{p}|, r) \le r |\alpha_{p}|^{p-1}. \]
En particulier, pour $\rho = \min ( |\alpha_{p}| , s |\alpha_{p}|^{1-p})$, on a $\overline{D}(\alpha_{p},\rho) \subset \overline{D}(P;s)$ et donc $\|Q\|_{\overline{D}(\alpha_{p},\rho)} \le  \|Q\|_{\overline{D}(P;s)}$.

D'apr\`es le lemme~\ref{lem:minorationetaar}, on a
\[ \max_{0\le j\le d} (|b_{j}|) \le \max\left(\frac{|\alpha_{p}|}{\rho}, \frac1 {\rho}, 1\right)^d \,  \|Q\|_{\overline{D}(\alpha_{p},\rho)}\]
et le r\'esultat s'ensuit.
\end{proof}

\begin{coro}\label{coro:minorationDPsA}
Supposons que~$\cA$ est uniforme.
Soient $p\in \N^\ast$ et $P \in \cA[T]$
tel que, pour tout $b\in \cM(\cA)$, $P(b) \in \cH(b)[T]$ soit un polyn\^ome de degr\'e~$p$ diff\'erent de~$T^p$. Soit $s \in \R_{>0}$. Alors, il existe $K_{P} \in \R_{>0}$ tel que, pour tout $Q = \sum_{j=0}^d b_{j} T^j \in \cA[T]$, on ait
\[\max_{0\le j\le d} (\|b_{j}\|) \le K_{P}^d \,  \|Q\|_{\overline{D}(P;s)}.\]
\qed
\end{coro}

\begin{rema}\label{rem:Tn}
Si~$P=T^p$, on a $\overline{D}(T^p;s) = \overline{D}(s^{1/p})$ et le r\'esultat reste valable d'apr\`es la proposition~\ref{prop:restrictionserie}.
\end{rema}

\index{Domaine polynomial!norme sur un|)}

\section{Bases de voisinages}
\label{description_locale}

Soit~$\cA$ un anneau de Banach. Posons $B :=\cM(\cA)$ et $X := \AunA$, avec coordonn\'ee~$T$. Notons $\pi \colon X \to B$ la projection sur la base. Dans cette section, nous donnons une description pr\'ecise de bases de voisinages des points de~$X$.

\medbreak

Nous nous baserons sur la classification de Berkovich des points de la droite analytique sur un corps valu\'e ultram\'etrique complet alg\'ebriquement clos, \cf~\cite[section 1.4.4]{Ber1}, que nous rappelons ici

\begin{nota}
Soit $(\ell,\va)$ un corps valu\'e ultram\'etrique.

Posons $\ell^\circ := \{a\in \ell : |a|\le 1\}$, $\ell^{\circ\circ} := \{a\in \ell : |a|< 1\}$ et $\tilde\ell := \ell^\circ/\ell^{\circ\circ}$. 

Posons $|\ell^\ast| := \{|a| : a \in \ell^\ast\}$ 
et notons $|\ell^\ast|^\Q$ sa cl\^oture divisible. Elle est naturellement munie d'une structure de $\Q$-espace vectoriel. Si $\ell$ est alg\'ebriquement clos, on a $|\ell^\ast|^\Q = |\ell^\ast|$.
\end{nota}

\begin{defi}\label{def:typealgclos}
Soit $(k,\va)$ un corps valu\'e ultram\'etrique complet alg\'ebriquement clos. Soit $x\in \E{1}{k}$. Notons $s(x)$ le degr\'e de transcendance de $\widetilde{\cH(x)}$ sur $\tilde k$ et $t(x)$ la dimension du $\Q$-espace vectoriel $|\cH(x)^\ast|^\Q/|k^\ast|$. 
Le point~$x$ est dit
\begin{itemize}
\item de \emph{type~1} si $\cH(x) = k$ ;
\item de \emph{type~2} si $s(x)=1$ ;
\item de \emph{type~3} si $t(x) = 1$ ;
\item de \emph{type~4} sinon.
\end{itemize}
\end{defi}


La d\'efinition de type d'un point s'\'etend \`a un corps valu\'e ultram\'etrique complet quelconque par extension des scalaires.

\begin{defi}\label{def:typenonalgclos}
Soit $(k,\va)$ un corps valu\'e ultram\'etrique complet. Notons $\hat{\bar{k}}$ le compl\'et\'e d'une cl\^oture alg\'ebrique de~$k$. Consid\'erons l'application $\pi_{\hat{\bar{k}}/k} \colon \E{1}{\hat{\bar{k}}} \to \E{1}{k}$ induite par l'inclusion $k \subset \hat{\bar{k}}$. Elle est surjective.

Soit $x\in \E{1}{k}$ et soit $y \in \pi_{\hat{\bar{k}}/k}^{-1}(x)$. Pour $i \in \cn{1}{4}$, le point~$x$ est dit de \emph{type~$i$} si le point~$y$ est de type~$i$. Cela ne d\'epend pas des choix.
\end{defi}

\begin{rema}\label{rem:typepointtopologie}
Lorsque $(k,\va)$ est un corps valu\'e ultram\'etrique complet, la droite analytique $\E{1}{k}$ poss\`ede une structure d'arbre r\'eel. Les diff\'erentes sortes de points de l'arbre correspondent \`a diff\'erents types de points au sens de Berkovich. Pr\'ecis\'ement, les feuilles sont les points de type~1 et~4 et les points de branchement sont les points de type~2. (Une infinit\'e de branches en partent.) Les points de type~3 sont ceux qui restent~: exactement deux branches en partent.

Nous utiliserons d\'esormais sans plus de commentaires les propri\'et\'es classiques des points de la droite analytique sur un corps valu\'e ultram\'etrique complet.
\end{rema}

Ajoutons une d\'efinition g\'en\'erale.

\begin{defi}\label{def:rigidedroite}\index{Point!rigide|textbf}\index{Point!rigide!polynome minimal@polyn\^ome minimal|textbf}\index{Point!rigide!degre@degr\'e|textbf}
\nomenclature[Jq]{$\mu_{x}$}{polyn\^ome minimal d'un point rigide~$x$ de~$\E{1}{k}$ (dans $k[T]$)}
\nomenclature[Jqa]{$\deg(x)$}{degr\'e de $\mu_{x}$}
Soit $(k,\va)$ un corps valu\'e complet. On dit qu'un point~$x$ de~$\E{1}{k}$ est \emph{rigide} si l'extension $\cH(x)/k$ est finie. 

Dans ce cas, il existe un unique polyn\^ome irr\'eductible unitaire \`a coefficients dans~$k$ qui s'annule en~$x$. On l'appelle \emph{polyn\^ome minimal} de~$x$ et on le note~$\mu_{x}$. On appelle \emph{degr\'e} de~$x$ son degr\'e et on le note~$\deg(x)$.
\end{defi}

Commen\c{c}ons par rappeler des r\'esultats connus sur les bases de voisinages des points de la droite analytique sur un corps valu\'e ultram\'etrique complet.

\begin{lemm}\label{lem:bv23}\index{Voisinage!d'un point de type~2 ou~3}
Soit $(k,\va)$ un corps valu\'e ultram\'etrique complet. Soit~$x$ un point de type~2 ou~3 de~$\Aunk$. Notons~$\cE_{x}$ l'ensemble des composantes connexes relativement compactes de $\Aunk\setminus\{x\}$.
Pour tout \'el\'ement~$C$ de~$\cE_{x}$, choisissons un point rigide~$x_{C}$ dans~$C$. 
Pour toute partie finie~$F$ de~$\Es_{x}$, posons $P_{F} :=  \prod_{C\in F} \mu_{x_C}$.

Alors, les ensembles de la forme 
\[\overline{C}(P_{F},|P_{F}(x)| -\eps, |P_{F}(x)| +\eps),\]
o\`u~$F$ est un sous-ensemble fini de~$\cE_{x}$ et~$\eps$ un nombre r\'eel strictement positif, 
sont une base de voisinages connexes de~$x$ dans~$\Aunk$.
\end{lemm}
\begin{proof}
Le r\'esultat d\'ecoule du fait que tout voisinage de~$x$ contient enti\`erement toutes les composantes connexes de $\Aunk \setminus \{x\}$ \`a l'exception \'eventuelle d'un nombre fini d'entre elles et d'arguments bas\'es sur le graphe de variation de~$|P_{F}|$.
\end{proof}

\begin{lemm}\label{lem:bv14}\index{Voisinage!d'un point de type~1 ou~4}\index{Base de voisinages|see{Voisinage}}
Soit $(k,\va)$ un corps valu\'e ultram\'etrique complet. Tout point~$x$ de type~1 ou~4 de~$\Aunk$ poss\`ede un syst\`eme fondamental de voisinages ouverts connexes de la forme
\[\overline{D}(P,|P(x)| + \eps),\]
o\`u $P\in k[T]$ est un polyn\^ome irr\'eductible et~$\eps$ un nombre r\'eel strictement positif. 
\end{lemm}
\begin{proof}
Soit~$U$ un voisinage de~$x$ dans~$\Aunk$.
Soit~$\bar k$ une cl\^oture alg\'ebrique de~$k$ et notons~$\hat{\bar{k}}$ son compl\'et\'e. Notons $\pi \colon \E{1}{\hat{\bar{k}}} \to \E{1}{k}$ le morphisme de projection. Soit $x' \in \pi^{-1}(x)$. Le point~$x'$ est un point de type~1 ou~4 de~$\E{1}{\hat{\bar{k}}}$. Puisque~$\hat{\bar{k}}$ est alg\'ebriquement clos, le singleton~$\{x\}$ est l'intersection d'une famille de disques ferm\'es de rayons strictement positifs. On en d\'eduit qu'il existe $\alpha \in \hat{\bar{k}}$ et $r_{1}>0$ tel que $\overline{D}(\alpha,r_{1})$ contienne~$x'$ et soit contenu dans~$\pi^{-1}(U)$. Puisque~$\bar k$ est dense dans~$\hat{\bar{k}}$, on peut supposer que $\alpha\in \bar k$. Notons $P \in k[T]$ son polyn\^ome minimal. Il existe alors $s_{1}>0$ tel que $\pi(\overline{D}(\alpha,r_{1})) = \overline{D}(P,s_{1})$ et on a 
\[ x \in  \overline{D}(P,s_{1}) \subset U.\]
Le r\'esultat s'en d\'eduit, en remarquant que les ensembles de la forme~$\overline{D}(P,s_{1})$ sont toujours connexes (\cf~lemme~\ref{lem:DQsconnexe}).
\end{proof}

Pour passer \`a des bases plus g\'en\'erales que des corps valu\'es, nous aurons besoin de quelques r\'esultats techniques.

\begin{nota}%
\nomenclature[Bia]{$\nm_{\infty}$}{norme $\infty$ sur $\cA[S]$}
Soit $(\cC,\nm)$ un anneau de Banach. On note~$\nm_{\infty}$ la norme sur l'anneau des polyn\^omes~$\cC[T]$ d\'efinie par
\[ \| P(T)\|_{\infty} = \max_{0\le i\le d} (\|a_{i}\|)\]
pour tout $P(T) = \sum_{i=0}^d a_{i}\, T^i \in \cC[T]$.

Lorsque $(\cC,\nm)$ est de la forme~$(\cB(V),\nm_{V})$ (\cf~notation~\ref{nota:BV}), on note~$\nm_{V,\infty}$ la norme correspondante sur $\cB(V)[T]$.
\end{nota}

\begin{lemm}\label{lem:borneT}
Soit $(k,\va)$ un corps valu\'e complet. Soit $P \in k[T]$ unitaire de degr\'e $d\ge 1$. Alors, pour tout $x\in \Aunk$, on a 
\[ |T(x)| \le |P(x)| + d\, \|P\|_{\infty}.\]
\end{lemm}
\begin{proof}
Soit $x\in \Aunk$. \'Ecrivons $P(T) = \sum_{i=0}^d a_{i}\, T^i \in k[T]$. Supposons, par l'absurde, que l'on ait $|T(x)| > |P(x)| + d\, \|P\|_{\infty}$. 
Puisque $P$~est unitaire, on a, en particulier, $|T(x)| \ge 1$, d'o\`u 
\[\big|\sum_{i=0}^{d-1} a_{i}(x) \, T(x)^i\big| \le d\, \|P\|_{\infty} \, |T(x)|^{d-1}.\] 
On en d\'eduit que 
\begin{align*} 
|P(x)| &= \big| T(x)^d + \sum_{i=0}^{d-1} a_{i}(x) \, T(x)^i \big|\\
& \ge |T(x)|^d - d\, \|P\|_{\infty} \, |T(x)|^{d-1}\\
& \ge |T(x)|^{d-1} \, (|T(x)| -   d\, \|P\|_{\infty})\\
& > |P(x)|,
\end{align*}
et l'on aboutit ainsi \`a la contradiction souhait\'ee.
\end{proof}

\begin{lemm}\label{lem:DQsconnexe}\index{Domaine polynomial!connexe}
Soit $(k,\va)$ un corps valu\'e complet. Soit~$Q_{0} \in k[T]$ un polyn\^ome de degr\'e $d\ge 1$ qui est une puissance d'un polyn\^ome irr\'eductible. Soit $r > 0$. 
Alors, il existe $\eps>0$ tel que, pour tout $Q \in k[T]$ unitaire de degr\'e~$d$ tel que $\|Q-Q_{0}\|_{\infty} \le \eps$ et tout $s \ge r$, $\overline{D}(Q,s)$ soit connexe.
\end{lemm}
\begin{proof}
Soit $Q\in k[T]$. Le principe du maximum assure que toute composante connexe de $\overline{D}(Q,s)$ contient un z\'ero de~$Q$. Par cons\'equent, pour montrer que $\overline{D}(Q,s)$ est connexe, il suffit de montrer que l'ensemble des z\'eros de~$Q$ est contenu dans une partie connexe de~$\overline{D}(Q,s)$. 

Par hypoth\`ese, le polyn\^ome~$Q_{0}$ poss\`ede un seul z\'ero (\'eventuellement multiple) dans~$\Aunk$. Le raisonnement pr\'ec\'edent assure que, pour tout $t>0$, $\overline{D}(Q_{0},t)$ est connexe.

Soit $t \in \intoo{0,r}$. Par continuit\'e des racines, il existe $\eps>0$ tel que, pour tout $Q \in k[T]$ unitaire de degr\'e~$d$ tel que $\|Q-Q_{0}\|_{\infty} \le \eps$, tous les z\'eros de~$Q$ dans~$\Aunk$ soient contenus dans $\overline{D}(Q_{0},t)$. Quitte \`a diminuer~$\eps$, on peut, en outre, supposer que $\overline{D}(Q,r) \supset \overline{D}(Q_{0},t)$. La connexit\'e de $\overline{D}(Q,r)$, et de $\overline{D}(Q,s)$ pour tout $s\ge r$, s'ensuit.  
\end{proof}

Venons-en maintenant aux descriptions annonc\'ees des bases de voisinages des points de $X = \E{1}{\cA}$. Rappelons que, si $b$ est un point de $B = \cM(\cA)$, la fibre $X_{b} := \pi^{-1}(b)$ s'identifie \`a~$\E{1}{\cH(b)}$. 

\begin{prop}\label{prop:basevoisdim1rigide}\index{Voisinage!d'un point rigide}
Soient $b\in B$ et $x$ un point rigide de $X_{b}$. Notons $\mu_{x} \in \cH(b)[T]$ le polyn\^ome minimal de~$x$. Soit $e \in \N^\ast$. Soit $U$ un voisinage de~$x$ dans~$X$ et soient $\eta, \sigma\in \R_{>0}$.
Alors, il existe un voisinage spectralement convexe~$V_{0}$ de~$b$ dans~$B$, un polyn\^ome~$P_{0}$ \`a coefficients dans~$\cB(V_{0})$, unitaire, de m\^eme degr\'e que~$\mu_{x}^e$, tel que $\|P_{0}(b) - \mu_{x}^e\|_{b,\infty} \le \eta$ (o\`u $P_{0}(b)$ d\'esigne l'image de~$P_{0}$ dans $\cH(b)[T]$) et des nombres r\'eels $s_{1},s_{2},\eps \in \R_{> 0}$ avec $s_{1} < s_{2} \le \sigma$ qui satisfont les propri\'et\'es suivantes: pour tout voisinage compact~$V$ de~$b$ dans~$V_{0}$, pour tout $P \in \cB(V)[T]$ unitaire de m\^eme degr\'e que~$P_{0}$ tel que $\|P - P_{0}\|_{V,\infty} \le \eps$ et pour tout $s \in [s_{1},s_{2}]$, on a
\[x \in \overline{D}_V(P,s) \subset U\]
et $\overline{D}_b(P,s)$ est connexe.

En outre, on peut imposer l'une des conditions suppl\'ementaires suivantes~:\footnote{Nous ne pr\'etendons pas qu'il soit possible de les imposer simultan\'ement.}
\begin{enumerate}[i)]
\item $P_{0}(b)$ est s\'eparable, si $\mu_{x}$ est s\'eparable et $e=1$ ;
\item $P_{0}(b)$ est s\'eparable, si $\cH(b)$ n'est pas trivialement valu\'e ;
\item $P_{0}$ est un relev\'e fix\'e de $\mu_{x}^e$ \`a~$\cO_{B,b}[T]$, si $\mu_{x}^e \in \kappa(b)[T]$.\end{enumerate} 
\end{prop}
\begin{proof} 
On peut supposer que $U$ est ouvert. 
Notons~$d$ le degr\'e de~$\mu_{x}^e$. Posons $M := \|\mu_{x}^e\|_{b,\infty}$ et $R := 1+d(M+1)$.

Le sous-ensemble de la fibre~$X_{b}$ d\'efini par l'\'equation $\mu_{x}^e = 0$ est r\'eduit au point~$x$. L'ensemble $\overline{D}_{b}(0,R) \setminus U$ est donc un compact sur lequel~$|\mu_{x}^e|$ est born\'ee inf\'erieurement par une constante $m_{0}>0$. 
Soit $m \in \intoo{0,\min(m_{0},\sigma,1)}$. On d\'eduit du lemme~\ref{lem:borneT} que l'on a $\overline{D}_{b}(\mu_{x}^e,m) \subset U$.

Puisque~$\kappa(b)$ est dense dans~$\cH(b)$, il existe un polyn\^ome~$P_{0}$ unitaire de degr\'e~$d$ \`a coefficients dans~$\cO_{b}$ tel que l'on ait 
\[ \|P_{0}(b) - \mu_{x}^e\|_{b,\infty}  < \min \Big( \frac 1 2, \frac m 5 \, \big(\sum_{i=0}^{d-1} R^i \big)^{-1},\eta \Big).\]
On a alors $\|P_{0}\|_{b,\infty} \le M+1/2$ donc, d'apr\`es le lemme~\ref{lem:borneT}, $\overline{D}_{b}(P_{0},1) \subset \overline{D}_{b}(0,R)$. On en d\'eduit alors que $\overline{D}_{b}(P_{0}, 4m/5) \subset \overline{D}_{b}(\mu_{x}^e,m) \subset U$. 
En outre, $\overline{D}_{b}(P_{0},m/5)$ contient~$x$.

Il existe un voisinage compact spectralement convexe~$W$ de~$b$ dans~$B$ tel que $P_{0}$ soit \`a coefficients dans~$\cB(W)$ et satisfasse $\|P_{0}\|_{W,\infty} \le M+3/4$. L'ensemble $\overline{D}_{W}(P_{0}, 4m/5) \setminus U$ est compact et sa projection sur~$B$ est un compact~$L$ qui ne contient pas le point~$b$. Soit~$V_{0}$ un voisinage compact spectralement convexe de~$b$ dans~$B \setminus L$. Par construction, on a $\overline{D}_{V_{0}}(P_{0}, 4m/5) \subset U$.

Soit $\eps \in \intoo{0,1/4}$ tel que $\eps \sum_{i=0}^{d-1} R^i < m/5$. Soient~$V$ une partie compacte de~$V_{0}$ et $P \in \cB(V)[T]$ un polyn\^ome unitaire de degr\'e~$d$ tel que $\|P - P_{0}\|_{V,\infty} \le \eps$. D'apr\`es le lemme~\ref{lem:borneT}, on a alors $\overline{D}_{V}(P,1) \subset \overline{D}_{V}(0,R)$. Posons $s_{1} := 2m/5$ et $s_{2} := 3m/5$. Pour tout $s \in [s_{1},s_{2}]$, on a alors $\overline{D}_{V}(P, s) \subset \overline{D}_{V}(P_{0},4m/5) \subset U$ et $\overline{D}_{V}(P, s) \supset \overline{D}_{V}(P_{0}, m/5) \ni x$. La premi\`ere partie du r\'esultat s'ensuit. Le lemme~\ref{lem:DQsconnexe} (utilis\'e avec $k = \cH(b)$ et $r=s_{1}$) permet d'assurer la connexit\'e de $\overline{D}_{b}(P, s)$, quitte \`a diminuer~$\eps$.

\medbreak

Int\'eressons-nous maintenant aux conditions suppl\'ementaires que nous pouvons imposer.

$i$) Si $\mu_{x}$ est s\'eparable et $e=1$, tout polyn\^ome suffisamment proche de $\mu_{x}^e = \mu_{x}$ est encore s\'eparable, par continuit\'e du discriminant. Le r\'esultat s'ensuit.

$ii$) Si $\cH(b)$ est de valuation non triviale, tout polyn\^ome \`a coefficients dans~$\cH(b)$ peut \^etre approch\'e autant qu'on le souhaite par un polyn\^ome s\'eparable \`a coefficients dans~$\cH(b)$, et m\^eme~$\kappa(b)$. Le r\'esultat s'ensuit.

$ii$) Si $\mu_{x}^e \in \kappa(b)[T]$, pour tout relev\'e $P_{0}$ de~$\mu_{x}^e$ dans $\cO_{B,b}[T]$, on a $\| P_{0}(b) - \mu_{x}^e\|_{b,\infty} = 0$. Ce choix de $P_{0}$ est donc permis.
%
 \end{proof}

\begin{prop}\label{prop:basevoisdim1}\index{Voisinage}
Soient $b\in B$, $x$ un point non rigide de~$X_{b}$ et $U$ un voisinage de~$x$ dans~$X$. 
Alors, il existe un voisinage spectralement convexe~$V_{0}$ de~$b$ dans~$B$, un polyn\^ome~$P_{0}$ unitaire non constant \`a coefficients dans~$\cB(V_{0})$ et des nombres r\'eels $r_{1},r_{2},s_{1},s_{2},\eps \in \R_{> 0}$ avec $r_{2} < r_{1} < s_{1} < s_{2}$ v\'erifiant les propri\'et\'es suivantes~: pour tout voisinage compact~$V$ de~$b$ dans~$V_{0}$, pour tout $P \in \cB(V_{0})[T]$ tel que $\|P - P_{0}\|_{V_{0},\infty} \le \eps$, pour tous $r,s \in \R_{> 0}$ tels que $r_{2} \le r \le  r_{1}$ et $s_{1} \le s\le s_{2}$, on a
\[x \in \overline{C}_V(P,r,s) \subset U\]
et $\overline{C}_b(P,r,s)$ est connexe.

En outre, si~$\cH(b)$ est de valuation non triviale ou de caract\'eristique nulle, on peut imposer que~$P_{0}(b)$ soit s\'eparable.
\end{prop}
\begin{proof}
Remarquons que les hypoth\`eses entra\^inent que le point~$b$ est ultram\'etrique. Le r\'esultat d\'ecoule alors des lemmes~\ref{lem:bv23} et~\ref{lem:bv14} 
appliqu\'es avec $k = \cH(b)$ et d'arguments d'approximation similaires \`a ceux de la preuve de la proposition~\ref{prop:basevoisdim1rigide}. La d\'emonstration de la connexit\'e de~$C_{b}(P,r,s)$ repose sur l'in\'egalit\'e triangulaire ultram\'etrique, qui assure que~$C_{b}(P,r,s)$ reste inchang\'e si l'on remplace~$P$ par un polyn\^ome suffisamment proche.  
\end{proof}

\begin{coro}[\protect{\cite[corollaire~6.8]{EtudeLocale}}]\label{cor:projectionouverte}\index{Projection!ouverte}
Soient $n,m \in \N$ avec $n\ge m$. La projection $\pi_{n,m} \colon \E{n}{\cA} \to \E{m}{\cA}$ sur les $m$ premi\`eres coordonn\'ees est ouverte.
\end{coro}
\begin{proof}
La remarque~\ref{spectralement_conv} permet de se ramener au cas o\`u~$m=0$ et~$n=1$. Le r\'esultat d\'ecoule alors des propositions~\ref{prop:basevoisdim1rigide} et~\ref{prop:basevoisdim1}.
\end{proof}

\section{R\'esultats locaux}\label{sec:resultatslocaux}

Soit~$\cA$ un anneau de Banach et soit~$n\in \N$. 
Dans cette section, nous rappelons quelques r\'esultats locaux sur~$\E{n}{\cA}$ obtenus par le second auteur dans~\cite{EtudeLocale}. 
Posons $B := \Mc(\Ac)$.

\subsection{Th\'eor\`eme de division de Weierstra\ss}

En g\'eom\'etrie analytique classique, les th\'eor\`emes de division et de pr\'eparation de Weierstra\ss{} sont des outils fondamentaux. Ils interviennent notamment de fa\c{c}on essentielle dans la d\'emonstration de la coh\'erence du faisceau structural. Le second auteur a d\'emontr\'e des versions de ces th\'eor\`emes pour les espaces de Berkovich dans~\cite{EtudeLocale}. 

Dans cette section, nous nous consacrons au th\'eor\`eme de division. Nous reviendrons plus tard sur le th\'eor\`eme de pr\'eparation (\cf{} th\'eor\`eme~\ref{thm:preparationW}). Commen\c{c}ons par rappeler quelques d\'efinitions.

\begin{defi}\label{def:bordan}\index{Bord analytique|textbf}
Soit~$V$ un ensemble compact spectralement convexe de~$\E{n}{\cA}$. On dit qu'une partie ferm\'ee~$\Gamma$ de~$V$ est un \emph{bord analytique} de~$V$ si, pour tout~$f\in\cB(V)$, on a 
\[\|f\|_\Gamma=\|f\|_V.\]
\end{defi}

\begin{defi}\label{def:decent}\index{Point!ultrametrique typique@ultram\'etrique typique|textbf}\index{Point!decent@d\'ecent|textbf}\index{Partie!ultrametrique typique@ultram\'etrique typique|textbf}\index{Partie!decente@d\'ecente|textbf}
Un point $x$ de $\E{n}{\cA}$ est dit \emph{ultram\'etrique typique} s'il appartient \`a l'int\'erieur de la partie ultram\'etrique de~$\E{n}{\cA}$ et s'il poss\`ede une base de voisinages form\'ee d'ensembles compacts et spectralement convexes poss\'edant un bord analytique fini.

Un point $x$ de $\E{n}{\cA}$ est dit \emph{d\'ecent} si l'une (au moins) des trois conditions suivantes est satisfaite~:
\begin{enumerate}[i)]
\item $\cH(x)$ est de caract\'eristique nulle~;
\item $\cH(x)$ est de valuation non triviale~;
\item $x$ est ultram\'etrique typique.
\end{enumerate}

Une partie de $\E{n}{\cA}$ est dite ultram\'etrique typique (resp. d\'ecente) lorsque tous ses points le sont. 
\end{defi}

\begin{exem}\label{ex:tresdecent}\index{Anneau!de base}
\index{Corps!valu\'e}\index{Corps!hybride}\index{Anneau!des entiers relatifs $\Z$}\index{Anneau!des entiers d'un corps de nombres}\index{Anneau!de valuation discr\`ete}\index{Anneau!de Dedekind trivialement valu\'e}
Lorsque $\cA$ est l'un de nos exemples usuels \ref{ex:corpsvalue} \`a~\ref{ex:Dedekind}~: les corps valu\'es complets, l'anneau~$\Z$ et les anneaux d'entiers de corps de nombres, les corps hybrides, les anneaux de valuation discr\`ete, les anneaux de Dedekind trivialement valu\'es, le spectre analytique~$\cM(\cA)$ est d\'ecent.
\end{exem}

\begin{prop}\label{prop:typiqueAn}
Soit~$b\in B$ un point ultram\'etrique typique (resp. d\'ecent). Alors, tout point de~$\E{n}{\cA}$ situ\'e au-dessus de~$b$ est encore ultram\'etrique typique (resp. d\'ecent).
\qed
\end{prop}
\begin{proof}
Le cas ultram\'etrique typique fait l'objet de \cite[proposition~6.10]{EtudeLocale}. Le cas d\'ecent s'en d\'eduit, les propri\'et\'es d'\^etre de caract\'eristique nulle ou de valuation non triviale \'etant stables par extension de corps.
\end{proof}

Posons $X := \AunA$, avec coordonn\'ee~$S$. 
Notons $\pi \colon X \to B$ la projection sur la base. Rappelons que si $b$ est un point de~$B$, la fibre $X_{b} := \pi^{-1}(b)$ s'identifie \`a~$\E{1}{\cH(b)}$. Rappelons \'egalement que, si $x$ est un point rigide de~$X_{b}$, alors l'anneau local $\cO_{X_{b},x}$ est un anneau de valuation discr\`ete d'uniformisante~$\mu_{x}$ (\cf~d\'efinition~\ref{def:rigidedroite} pour la notation).

\begin{theo}[Division de Weierstra\ss, \protect{\cite[th\'eor\`eme~8.3]{EtudeLocale}}]\label{weierstrass}\index{Theoreme@Th\'eor\`eme!de division de Weierstra\ss}
Soit $b$ un point de~$B$. Soit~$x$ un point rigide de~$X_{b}$. Si~$\mu_{x}$ est ins\'eparable et si $b$ est trivialement valu\'e, supposons que~$b$ est ultram\'etrique typique. Soit~$G$ un \'el\'ement de l'anneau local~$\cO_{X,x}$. Supposons que son image dans l'anneau de valuation discr\`ete~$\cO_{X_{b},x}$ n'est pas nulle et notons~$n$ sa valuation.

Alors, pour tout~$F\in\cO_{X,x}$, il existe un unique couple~$(Q,R)\in \cO_{X,x}^2$ tel que 
\begin{enumerate}[i)]
\item $F=QG+R$ ;
\item $R\in\cO_{B,b}[S]$ est un polyn\^ome de degr\'e strictement inf\'erieur \`a~$n\deg(x)$.
\end{enumerate}
\qed
\end{theo}

On peut \'etendre ce th\'eor\`eme de fa\c con \`a diviser simultan\'ement en plusieurs points de la fibre. On renvoie \`a la preuve de \cite[th\'eor\`eme~5.5.3]{A1Z} pour les d\'etails\footnote{Dans cette r\'ef\'erence, le polyn\^ome~$G$ est suppos\'e d'une forme particuli\`ere, mais cela n'est pas utilis\'e dans la preuve. La condition~$(I_{G})$ peut \^etre remplac\'ee par la condition~$(D_{G})$ et \cite[corollaire~5.4]{EtudeLocale} assure qu'elle est toujours satisfaite. La condition~$(S)$ n'est pr\'esente que pour assurer que le th\'eor\`eme de division de Weierstra\ss{} s'applique.}.

\begin{coro}\label{cor:weierstrassgeneralise}\index{Theoreme@Th\'eor\`eme!de division de Weierstra\ss!g\'en\'eralis\'e}
Soit $b$ un point d\'ecent de~$B$. Soit~$G$ un polyn\^ome unitaire de degr\'e~$d$ \`a coefficients dans~$\cO_{B,b}$. Notons~$\{z_1,\dotsc,z_{t}\}$ l'ensemble des z\'eros de~$G$ dans~$X_{b}$. 

Alors, pour tout~$(f_1,\dotsc,f_t)\in\prod_{i=1}^t \cO_{X,z_i}$, il existe un unique \'el\'ement~$(r,q_1,\dotsc,q_t)$ de $\cO_{B,b}[T]\times\prod_{i=1}^t\cO_{X,z_i}$ v\'erifiant les propri\'et\'es suivantes~:
\begin{enumerate}[i)]
\item pour tout~$i\in\cn{1}{t}$, on a~$f_i = q_i G + r$ dans~$\cO_{X,z_i}$~;
\item le polyn\^ome~$r$ est de degr\'e strictement inf\'erieur \`a~$d$.
\end{enumerate}
\qed
\end{coro}

On d\'emontrera plus loin une version raffin\'ee de cet \'enonc\'e permettant de contr\^oler les normes du reste et du quotient (\cf~th\'eor\`eme~\ref{weierstrassam} et corollaire~\ref{cor:divisionGnormes}).

\medbreak

Le th\'eor\`eme de division de Weierstra\ss{} a de nombreuses cons\'equences importantes, par exemple pour l'\'etude des morphismes finis. En voici une. Soit $P(S) \in \cA[S]$ unitaire non constant. Le morphisme de $\cA$-alg\`ebres naturel 
\[\cA[T] \too \cA[T,S]/(P(S) -T) \simtoo \cA[S]\]
induit un morphisme d'espaces localement annel\'es $\varphi_{P} \colon \AunA \to \AunA$, o\`u nous notons~$S$ (resp.~$T$) la coordonn\'ee sur l'espace de d\'epart (resp. d'arriv\'ee). Soit~$V$ une partie de~$B$. Soient $s,t\in\R$ tels que $0=s<t$ ou $0<s\le t$. Nous avons alors $\varphi_{P}^{-1}(\overline{C}_{V}(s,t)) = \overline{C}_{V}(P,s,t)$.

\begin{theo}[\protect{\cite[th\'eor\`eme~8.8]{EtudeLocale}}]\label{thm:isolemniscate}\index{Domaine polynomial!algebre@alg\`ebre d'un|(}
Dans la situation pr\'ec\'edente, si $V$~est d\'ecente, alors le morphisme de $\cA$-alg\`ebres naturel
\[\cO(\overline{C}_{V}(s,t))[S]/(P(S)-T) \too \cO(\overline{C}_{V}(P,s,t))\]
est un isomorphisme.
\qed
\end{theo}

On en d\'eduit une version du corollaire~\ref{cor:couronneglobale} pour les domaines polynomiaux.

\begin{coro}\label{cor:lemniscateglobale}
Dans la situation pr\'ec\'edente, si~$V$ est ultram\'etrique et d\'ecente, alors, le morphisme de $\cA$-alg\`ebres naturel
\[\colim_{W\supset V,u\prec s \le t < v} \cO(W)\la u\le |T|\le v\ra[S]/(P(S)-T) \too \cO(\overline{C}_{V}(P,s,t)),\]
o\`u~$W$ d\'ecrit l'ensemble des voisinages compacts de~$V$ dans~$B$ sur lesquels les coefficients de~$P$ sont d\'efinis, est un isomorphisme.
\qed
\end{coro}

On va maintenant modifier l'\'enonc\'e pr\'ec\'edent de fa\c{c}on \`a faire intervenir des anneaux de Banach dans la limite inductive, ce qui se r\'ev\`elera utile par la suite.

\begin{lemm}\label{lem:isoBVcompacts}\index{Fonction!surconvergente}
Soient~$V$ une partie compacte de~$B$ et~$\cV$ une base de voisinages compacts de~$V$ dans~$B$. Alors, pour tout $W\in\cV$, le morphisme de restriction induit un morphisme de $\cA$-alg\`ebres
\[\varphi_{W} \colon \overline{\cO(W)} \too \cO(V).\]
Le morphisme de $\cA$-alg\`ebres naturel
\[\colim_{W\in \cV} \overline{\cO(W)} \too \cO(V)\]
qui s'en d\'eduit est un isomorphisme. 
\end{lemm}
\begin{proof}
Par d\'efinition, les sections du faisceau structural sur un ouvert~$U$ sont localement des limites uniformes de fractions rationnelles sans p\^oles. On en d\'eduit que le morphisme $\varphi_{W} \colon \overline{\cO(W)} \to \cO(V)$ est bien d\'efini. Le second r\'esultat est imm\'ediat.
\end{proof}

\begin{coro}\label{cor:lemniscateglobaleBanach}
Dans la situation pr\'ec\'edente, si~$V$ est ultram\'etrique et d\'ecente, alors, le morphisme de $\cA$-alg\`ebres naturel
\[\colim_{W\supset V,u\prec s \le t < v} \overline{\cO(W)}\la u\le |T|\le v\ra[S]/(P(S)-T) \too \cO(\overline{C}_{V}(P,s,t)),\]
o\`u~$W$ d\'ecrit l'ensemble des voisinages compacts de~$V$ dans~$B$ sur lesquels les coefficients de~$P$ sont d\'efinis, est un isomorphisme.
\qed
\end{coro}
\index{Domaine polynomial!algebre@alg\`ebre d'un|)}

\subsection{Changement de variables}
\label{comparaison_entre_les_normes}

Dans cette section, nous d\'emontrons un r\'esultat de changement de variables tr\`es utile en pratique. \'Etant donn\'e une fonction non nulle sur un espace affine, il permet d'assurer que celle-ci reste non nulle lorsque l'on sp\'ecialise toutes les variables sauf une, ce qui permet notamment d'appliquer le th\'eor\`eme de division de Weierstra\ss~\ref{weierstrass}. Dans le cadre des espaces sur un corps valu\'e ultram\'etrique complet, on retrouve le fait qu'une fonction non nulle peut \^etre transform\'ee en une fonction distingu\'ee par rapport \`a une variable, par le biais d'un automorphisme (\cf~\cite[proposition~5.2.4/1]{BGR}). Notre r\'esultat figure d\'ej\`a dans \cite[lemme~9.15]{EtudeLocale}, mais l'auteur n'en a r\'edig\'e qu'une preuve succinte. Nous en proposons ici une d\'emonstration compl\`ete et diff\'erente.

\medbreak

Pour $n\in \N$, notons $T_{1},\dotsc,T_{n}$ les coordonn\'ees de~$\E{n}{\cA}$. Pour $\bu = (u_1,\dotsc,u_{n-1}) \in(\N^\ast)^{n-1}$, notons $\psi_{\bu} \colon \E{n}{\cA} \to \E{n}{\cA}$ l'automorphisme d'espaces localement annel\'es induit par le changement de variables 
\[\left\lbrace
\begin{array}{rcl}
T_1 &\longmapsto& T_1+T_n^{u_1}~;\\
&\vdots&\\
T_{n-1}&\longmapsto &T_{n-1}+T_n^{u_{n-1}}~;\\
T_n&\longmapsto& T_n.
\end{array}
\right.\]

\begin{lemm}\label{changement_variable}\index{Changement de variables}
Soient~$b$ un point de~$B$. Soient~$n\in\N$ et $x$ un point rigide de~$\pi_{n}^{-1}(b)$. Soit~$f$ un \'el\'ement de~$\cO_{\E{n}{\cA},x}$ dont  la restriction \`a~$\pi_{n}^{-1}(b)$ n'est pas nulle. Alors, il existe $\bu \in (\N^\ast)^{n-1}$ tel que l'image de $\psi_{\bu}^\#(f)$ dans $\pi_{n-1}^{-1}(\pi_{n-1}(\psi_{\bu}^{-1}(x)))$ ne soit pas nulle. 
\end{lemm}
\begin{proof}
Puisque le r\'esultat ne concerne que la restriction de la fonction~$f$ \`a~$\pi_{n}^{-1}(b)$, nous pouvons remplacer~$\cA$ par~$\cH(b)$, et donc supposer que~$\cA$ est un corps valu\'e complet. Nous le noterons d\'esormais~$K$.

Remarquons qu'il suffit de d\'emontrer le r\'esultat apr\`es extension des scalaires. Nous pouvons donc supposer que~$K$ est alg\'ebriquement clos. Le point~$x$ s'identifie alors \`a un \'el\'ement $(\alpha_{1},\dotsc,\alpha_{n})$ de~$K^n$.

Pour tout $n\in \N$, d\'efinissons maintenant un id\'eal~$I_{n}$ de~$\cO_{\E{n}{K},x}$. Posons $I_{0} := (0)$, $I_{1} := (0)$ et, pour tout $n\ge 2$,
\[I_{n} := \bigcap_{\bu \in \N^{n-1}} \big(T_1-\alpha_1-T_n^{u_{1}}+\alpha_{n}^{u_1},\dotsc,T_{n-1}-\alpha_{n-1}-T_n^{u_{n-1}}  + \alpha_{n}^{u_{n-1}}\big).\]
Rappelons que, d'apr\`es la proposition~\ref{prop:disqueglobal}, l'anneau local~$\cO_{\E{n}{K},x}$ s'identifie \`a un sous-anneau de $K\llbracket T_1-\alpha_1,\dotsc,T_n - \alpha_{n}\rrbracket$.

Nous allons montrer que $I_{n} = (0)$ par r\'ecurrence sur~$n$. Si $n$ vaut~0 ou~1, le r\'esultat vaut par d\'efinition. Soit $n\in \N$ avec $n\ge 2$ et supposons avoir d\'emontr\'e que $I_{n-1} = (0)$. 

Fixons~$u_{n-1}\in\N^\ast$. Puisque $T_{n} - \alpha_{n}$ divise $T_{n}^{u_{n-1}} - \alpha_{n}^{u_{n-1}}$, on a un isomorphisme naturel
\begin{align*} 
& K\llbracket T_1-\alpha_1,\dotsc,T_n-\alpha_n\rrbracket /(T_{n-1}-\alpha_{n-1} - T_{n}^{u_{n-1}} + \alpha_{n}^{u_{n-1}})\\ 
\simeq \ & K\llbracket T_1-\alpha_1,\dotsc,T_{n-2}-\alpha_{n-2},T_{n}-\alpha_{n}\rrbracket.
\end{align*}
L'image de~$I_{n}$ par cet isomorphisme s'identifie \`a~$I_{n-1}$ (en effectuant le changement de variable $T_{n} \mapsto T_{n-1}  - \alpha_{n-1} +\alpha_{n}$), et est donc nulle, par hypoth\`ese de r\'ecurrence. On en d\'eduit que~$I_{n}$ est contenu dans l'id\'eal de $ K\llbracket T_1-\alpha_1,\dotsc,T_n-\alpha_n\rrbracket$ engendr\'e par $T_{n-1}-\alpha_{n-1} - T_{n}^{u_{n-1}} -\alpha_{n}^{u_{n-1}}$.

Puisque le r\'esultat pr\'ec\'edent vaut pour tout $u_{n-1} \in \N^\ast$, que tous les $T_{n-1}-\alpha_{n-1}-T_{n}^{u_{n-1}} -\alpha_{n}^{u_{n-1}}$ sont irr\'eductibles et non associ\'es, et que $K\llbracket T_1-\alpha_1,\dotsc,T_n-\alpha_n\rrbracket$ est factoriel, on en d\'eduit que $I_{n}=0$.

On peut maintenant conclure. Puisque~$f$ n'est pas nulle, il existe $\bu = (u_1,\dotsc,u_{n-1})\in (\N^\ast)^{n-1}$ tel que~$f$ n'appartienne pas \`a l'id\'eal
\[(T_1-\alpha_1- T_n^{u_{1}} + \alpha_{n}^{u_1},\dotsc,T_{n-1}-\alpha_{n-1}- T_{n}^{u_{n-1}} + \alpha_{n}^{u_{n-1}})\]
de $\cO_{\E{n}{K},x}$. On en d\'eduit que $\psi_{\bu}^\sharp(f)$ 
n'appartient pas \`a l'id\'eal
\[(T_1-\alpha_1 + \alpha_{n}^{u_1},\dotsc,T_{n-1}-\alpha_{n-1} + \alpha_{n}^{u_{n-1}})\]
de $\cO_{\E{n}{K},\psi_{\bu}^{-1}(x)}$, ce qui signifie exactement que~$\psi_{\bu}^\sharp(f)$ ne s'annule pas en~$\psi_{\bu}^{-1}(x)$.
\end{proof}

\subsection{Anneaux de Banach de base} Nous introduisons ici une classe d'anneaux de Banach sur lesquels nous travaillerons dans la suite.

Posons $X:= \E{n}{\cA}$. Remarquons que, par d\'efinition du faisceau structural, pour tout point~$x$ de~$X$, on a un isomorphisme 
\[\colim_{V \ni x} \cB(V) \simtoo \cO_{x},\]
o\`u $V$ parcourt l'ensemble des voisinages compacts de~$x$.

\begin{defi}\label{def:B-defini}\index{B-definie@$\cB$-d\'efinie|see{Fonction}}\index{Fonction!B-definie@$\cB$-d\'efinie|textbf}
Soient~$x\in X$ et $V$ un voisinage compact de~$x$. On dit qu'un \'el\'ement~$f$ de~$\cO_{x}$ est \emph{$\cB$-d\'efini} sur~$V$ s'il appartient \`a l'image du morphisme naturel $\cB(V) \to \cO_{x}$.
\end{defi}

%

\begin{defi}\index{Base de voisinages!fine|textbf}
Soient~$T$ un espace topologique, $t$ un point de~$T$ et $\cU$ une base de voisinages de~$t$. On dit que~$\cU$ est \emph{fine} si elle contient une base de voisinages de chacun de ses \'el\'ements.
\end{defi}

\begin{defi}[\protect{\cite[d\'efinition~9.1]{EtudeLocale}}]\index{Ideal@Id\'eal!B-fortement de type fini@$\cB$-fortement de type fini|textbf}\index{Ideal@Id\'eal!$\cB$-syst\`eme de g\'en\'erateurs forts|textbf}
Soient~$x$ un point de~$X$ et~$\cV$ une base de voisinages fine de~$x$ form\'ee d'ensembles compacts et spectralement convexes. On dit qu'un id\'eal~$I$ de~$\cO_{x}$ est \emph{$\cB$-fortement de type fini} relativement \`a~$\cV$ s'il existe~$f_1,\dotsc,f_p$ appartenant \`a~$I$ tels que
\begin{enumerate}[i)]
\item pour tous~$V\in\cV$ et~$i\in \{1,\dotsc,p\}$, $f_i$ est~$\cB$-d\'efini sur~$V$.
\item pour tout voisinage compact~$U$ de~$x$, il existe une famille $(K_{V,U})_{V\in\cV}$ de~$\R_{>0}$ telle que, pour tout \'el\'ement~$f$ de~$I$ qui est $\cB$-d\'efini sur~$U$ et tout \'el\'ement~$V$ de~$\cV$ contenu dans~$\mathring U$, il existe des \'el\'ements~$a_1,\ldots,a_p$ de~$\cB(V)$ tels que 
\[\begin{cases}
f=a_1f_1+\cdots+a_pf_p \textrm{ dans }\cB(V)~;\\
\forall i\in \{1,\dotsc,p\}, \|a_i\|_V\leq K_{V,U}\, \|f\|_U.
\end{cases}\]
\end{enumerate}
Une famille~$(f_1,\ldots,f_p)$ v\'erifiant les propri\'et\'es pr\'ec\'edentes est appel\'ee \emph{$\cB$-syst\`eme de g\'en\'erateurs fort} de l'id\'eal~$I$ relativement \`a~$\cV$. On dit \'egalement qu'elle \emph{engendre $\cB$-fortement} l'id\'eal~$I$ relativement \`a~$\cV$.
\end{defi}

\begin{defi}[\protect{\cite[d\'efinition~9.3]{EtudeLocale}}]\index{Anneau!fortement regulier@fortement r\'egulier|textbf}\index{Anneau!fortement de valuation discrete@fortement de valuation discr\`ete|textbf}\index{Corps!fort|textbf}
Soient~$x$ un point de~$X$ et~$\cV$ une base de voisinages fine de~$x$ form\'ee d'ensembles compacts et spectralement convexes. Soit $d\in \N$. On dit que l'anneau local~$\cO_{x}$ est \emph{fortement r\'egulier} de dimension~$d$ relativement \`a~$\cV$ si 
\begin{enumerate}[i)]
\item $\cO_{x}$ est noeth\'erien de dimension de Krull~$d$~;
\item il existe des \'el\'ements $f_1,\dotsc,f_d$ de~$\m_x$ tels que $(f_1,\ldots,f_d)$ engendre~$\cB$-fortement l'id\'eal~$\m_x$ relativement \`a~$\cV$.
\end{enumerate}

On dit que l'anneau local~$\Oc_{x}$ est un \emph{corps fort} (resp. \emph{fortement de valuation discr\`ete}) s'il est fortement r\'egulier de dimension~0 (resp.~1).
\end{defi}

\begin{rema}\label{rem:corpsfortprolan}
La propri\'et\'e d'\^etre un corps fort s'apparente \`a une propri\'et\'e de prolongement analytique. 
La d\'efinition requiert qu'une fonction d\'efinie sur~$U$ et nulle au voisinage de~$x$ soit nulle sur tout voisinage de~$x$ appartenant \`a~$\cV$ et inclus dans~$\mathring U$. 

R\'eciproquement, supposons que $X$ satisfait le principe du prolongement analytique (au sens o\`u une fonction nulle au voisinage d'un point l'est encore sur tout ouvert connexe contenant ce point) et que tous les \'el\'ements de~$\cV$ sont connexes. Alors le fait que l'anneau~$\cO_{x}$ soit un corps implique que c'est un corps fort pour la base de voisinages~$\cV$.

Nous reviendrons plus pr\'ecis\'ement sur la notion de prolongement analytique \`a la d\'efinition~\ref{def:prolongementanalytique} et les \'enonc\'es qui la suivent.

\end{rema}

\begin{defi}[\protect{\cite[d\'efinition~9.5]{EtudeLocale}}]\label{def:basique}\index{Anneau!de base|textbf}\index{Anneau!basique|see{de base}}
On dit que l'anneau de Banach~$\cA$ est \emph{de base} 
ou \emph{basique} 
si tout point~$b$ de $B = \cM(\cA)$ poss\`ede une base fine de voisinages compacts et spectralement convexes~$\cV_b$ satisfaisant les propri\'et\'es suivantes~:
\begin{enumerate}[i)]
\item si $\cH(b)$ est de caract\'eristique non nulle et trivialement valu\'e, alors tout \'el\'ement de~$\cV_{b}$ est contenu dans $B_{\um}$ et poss\`ede un bord analytique fini~;
\item l'anneau local~$\cO_{B,b}$ est un corps fort ou un anneau fortement de valuation discr\`ete relativement \`a~$\cV_{b}$.
\end{enumerate}
La condition~i) entra\^ine que tout point de~$B$ est d\'ecent.
\end{defi}

\begin{exem}\label{ex:basique}\index{Anneau!de base}
\index{Corps!valu\'e}\index{Corps!hybride}\index{Anneau!des entiers relatifs $\Z$}\index{Anneau!des entiers d'un corps de nombres}\index{Anneau!de valuation discr\`ete}\index{Anneau!de Dedekind trivialement valu\'e}
Les exemples \ref{ex:corpsvalue} \`a~\ref{ex:Dedekind} donn\'es plus haut~: les corps valu\'es complets, l'anneau~$\Z$ et les anneaux d'entiers de corps de nombres, les corps hybrides, les anneaux de valuation discr\`ete, les anneaux de Dedekind trivialement valu\'es sont tous des anneaux de base. On renvoie \`a \cite[\S 3.1]{A1Z} pour des d\'etails dans le cas des anneaux d'entiers de corps de nombres.
\end{exem}

\subsection{Anneaux locaux et faisceau structural}

Dans cette section, nous \'enon\c cons les principaux r\'esultats obtenus dans~\cite{EtudeLocale}. Posons $X:= \E{n}{\cA}$, avec coordonn\'ees $T_{1},\dotsc,T_{n}$.

\subsubsection{Types de points}

Il sera utile de distinguer diff\'erents types de points dans les espaces affines.

\begin{defi}\index{Point!rigide epais@rigide \'epais|textbf}\index{Point!rigide epais@rigide \'epais|(}\index{Point!localement transcendant|textbf}\index{Point!purement localement transcendant|textbf}
Soient $b$ un point de~$B$ et $x$ un point de~$X$ au-dessus de~$b$. On dit que le point~$x$ est \emph{rigide \'epais} si $T_1(x),\ldots,T_n(x)$ sont alg\'ebriques sur~$\kappa(b)$. Dans le cas contraire, on dit que le point~$x$ est \emph{localement transcendant}. 

On dit que le point~$x$ est \emph{purement localement transcendant} si, pour tout $i \in \{1,\dotsc,n\}$, $\pi_{n,i}(x)$ est localement transcendant au-dessus de $\pi_{i,i-1}(\pi_{n,i}(x)) = \pi_{n,i-1}(x)$.
\end{defi}

Dans~\cite[d\'efinition~8.1]{EtudeLocale}, un point est dit rigide \'epais lorsque~$\kappa(x)$ est une extension finie de~$\kappa(b)$, mais c'est en r\'ealit\'e la condition que nous avons \'enonc\'ee qui est utilis\'ee (sous la forme de l'existence d'un polyn\^ome non nul \`a coefficients dans~$\cO_{b}$ dont l'image dans $\kappa(x)$ est nulle). 

\begin{prop}\label{rigide_\'epais}
Soient $b$ un point d\'ecent de~$B$ et $x$ un point de~$X$ au-dessus de~$b$. Alors, $x$ est rigide \'epais au-dessus de~$b$ si, et seulement si, $\kappa(x)$ est une extension finie de~$\kappa(b)$.

En particulier, la notion de point rigide \'epais est ind\'ependante du choix des coordonn\'ees $T_{1},\dotsc,T_{n}$.

Soit $m\in \cn{0}{n}$. Alors $x$ est rigide \'epais au-dessus de~$b$ si, et seulement si, $x$ est rigide \'epais au-dessus de~$\pi_{n,m}(x)$ et $\pi_{n,m}(x)$~est rigide \'epais au-dessus de~$b$.
\end{prop}
\begin{proof}
Si $\kappa(x)/\kappa(b)$ est finie, alors $x$ est rigide \'epais. R\'eciproquement, supposons que~$x$ est rigide \'epais. Le cas $n=0$ est trivial et une r\'ecurrence permet ensuite de se ramener au cas $n=1$. Posons $T := T_{1}$. Par hypoth\`ese, il existe un polyn\^ome $G \in \cO_{b}[T]$ non nul dont l'image dans~$\kappa(x)$ est nulle. Le th\'eor\`eme de division de Weierstra\ss~\ref{weierstrass} assure alors que~$\kappa(x)$ est engendr\'e, en tant que $\kappa(b)$-espace vectoriel, par un nombre fini de puissances de~$T(x)$. Le r\'esultat s'ensuit.
\end{proof}

\begin{rema}\index{Point!purement localement transcendant}
L'analogue de la proposition~\ref{rigide_\'epais} tombe en d\'efaut pour les points localement transcendants, m\^eme dans le cas o\`u l'anneau de base~$\cA$ est un corps valu\'e. On renvoie \`a \cite[\S 5.2.7]{TemkinTranscendence} pour un contre-exemple d\^u \`a M.~Temkin, inspir\'e par un r\'esultat de M.~Matignon et M.~Reversat dans~\cite{MatignonReversatSousCorpsFermes}.
\end{rema}

\begin{rema}\label{rem:rigeptrans}\index{Point!rigide epais@rigide \'epais}\index{Point!purement localement transcendant}
Soient $b$ un point de~$B$ et $x$ un point de~$X$ au-dessus de~$b$. 
Alors, quitte \`a permuter les coordonn\'ees $T_{1},\dotsc,T_{n}$, il existe $m_{x} \in\cn{0}{n}$ tel que les conditions suivantes soient satisfaites~:
\begin{enumerate}[i)]
\item $\pi_{n,m_{x}}(x)$ est purement localement transcendant au-dessus de~$\pi_{n}(x)$~;
\item $x$ est rigide \'epais au-dessus de~$\pi_{n,m_{x}}(x)$.
\end{enumerate}

Pour $m\in \N^\ast$ et $i_{1},\dotsc,i_{m} \in \cn{0}{n}$, notons $x_{i_{1},\dotsc,i_{m}}$ la projection du point~$x$ sur l'espace affine $\E{m}{\cA}$ de coordonn\'ees $T_{i_{1}},\dotsc,T_{i_{m}}$. L'entier $m_x$ peut \^etre d\'efini comme la borne sup\'erieure de l'ensemble des entiers~$m$ pour lesquels il existe $i_{1},\dotsc,i_{m} \in \cn{0}{n}$ tels que 
\begin{itemize}
\item $x_{i_{1}}$ est transcendant au-dessus de~$b$ ; 
\item pour tout $j\in \cn{2}{m}$, $x_{i_{1},\dotsc,i_{j}}$ est transcendant au-dessus de~$x_{i_{1},\dotsc,i_{j-1}}$.
\end{itemize}
\end{rema}

\subsubsection{Points rigides \'epais d'une droite relative}

Dans cette section, nous supposons que $n=1$ et notons $T=T_{1}$. Soit~$b$ un point de~$B$ et $x$ de~$X$ au-dessus de~$b$. 

\begin{nota}
\index{Point!rigide epais@rigide \'epais!polynome minimal@polyn\^ome minimal|textbf}\index{Point!rigide epais@rigide \'epais!polynome minimal@polyn\^ome minimal|(}%
\nomenclature[Jr]{$\mu_{\kappa,x}$}{polyn\^ome minimal \'epais d'un point rigide \'epais~$x$ de $\E{1}{\cA}$ (dans $\kappa(b)[T]$ si $x$ est au-dessus de $b \in \cM(\cA)$)}
Supposons que~$x$ est un point rigide \'epais. On appelle \emph{polyn\^ome minimal \'epais} de~$x$ le polyn\^ome minimal de~$T(x)$ sur~$\kappa(b)$ et on le note $\mu_{\kappa,x}(T)$. 
\end{nota}

Rappelons que l'on note $\mu_{x}(T)$ le polyn\^ome minimal de~$T(x)$ sur~$\cH(b)$ (\cf~d\'efinition~\ref{def:rigidedroite}). Dans~ \cite[d\'efinition~8.2]{EtudeLocale}, l'auteur affirme que $\mu_{\kappa,x} = \mu_{x}$, mais cet \'enonc\'e est incorrect. 
En revanche, on peut d\'emontrer que le polyn\^ome minimal sur~$\kappa(b)$ est une puissance de celui sur~$\cH(b)$.

\begin{lemm}\label{lem:PTp}
Soit~$K$ un corps de caract\'eristique~$p>0$. Soient $P \in K[X]$ un polyn\^ome irr\'eductible s\'eparable et $s\in \N$. Alors il existe un polyn\^ome irr\'eductible $Q \in K[X]$ et $r \in \N$ tels que
\[P(T^{p^s}) = Q(T)^{p^r}.\]
\end{lemm}
\begin{proof}
On peut supposer que~$P$ est unitaire. Soit~$\bar K$ une cl\^oture alg\'ebrique de~$K$. Par hypoth\`ese, il existe des \'el\'ements $\alpha_{1},\dotsc,\alpha_{d} \in \bar K$, deux \`a deux distincts, tels que l'on ait 
\[P(T) = \prod_{i=1}^d (T-\alpha_{i}) \textrm{ dans } \bar K[T].\]
Pour tout $i \in \{1,\dotsc,d\}$, $\alpha_{i}$ poss\`ede une unique racine $p^s$-\`eme dans~$\bar K$. Notons-la~$\alpha_{i}^{1/p^s}$. On a alors
\[P(T^{p^s}) = \prod_{i=1}^d (T^{p^s}-\alpha_{i}) = \prod_{i=1}^d (T-\alpha_{i}^{1/p^s})^{p^s}  \textrm{ dans } \bar K[T].\]
Notons~$Q(T)$ le polyn\^ome minimal de~$\alpha_{1}^{1/p^s}$ sur~$K$. 

Soit~$R(T)$ un facteur irr\'eductible de~$P(T^{p^s})$ dans~$K[T]$. Il existe $i \in \{1,\dotsc,d\}$ tel que $R(\alpha_{i}^{1/p^s}) = 0$. En posant $S(T) = R(T^{p^s})$, on a $S(\alpha_{i}) = 0$. Puisque~$P(T)$ est irr\'eductible, $P(T)$ divise~$S(T)$ et on a donc $S(\alpha_{1})=0$. On en d\'eduit que $R(\alpha_{1}^{1/p^s})=0$, donc que $Q(T)$ divise $R(T)$, puis que $R(T) = Q(T)$.

On a montr\'e que tous les facteurs irr\'eductibles de~$P(T^{p^s})$ sont \'egaux \`a~$Q(T)$. Par cons\'equent, il existe $m\in \N$ tel que $P(T^{p^s}) = Q(T)^m$. En consid\'erant les d\'ecompositions des polyn\^omes dans~$\bar K[T]$, on montre que~$m$ est n\'ecessairement une puissance de~$p$.
\end{proof}

\begin{lemm}\label{lem:polmin}\index{Point!rigide!polynome minimal@polyn\^ome minimal}\index{Point!rigide epais@rigide \'epais!polynome minimal@polyn\^ome minimal}
Supposons que~$x$ est un point rigide \'epais.  Alors il existe $r\in \N$ tel que 
\[\mu_{\kappa,x}(T) = \mu_{x}(T)^{e^r},\]
o\`u $e$ d\'esigne l'exposant caract\'eristique de~$\kappa(b)$.
\end{lemm}
\begin{proof}
On suit ici la preuve de \cite[corollaire~5.4]{EtudeLocale}.
Il existe un polyn\^ome irr\'eductible et s\'eparable $P \in \kappa(b)[T]$ et un entier $s\in\N$ tels que $\mu_{\kappa,x}(T) = P(T^{e^s})$. 

D'apr\`es~\cite[th\'eor\`eme~5.2]{EtudeLocale}, le corps~$\kappa(b)$ est hens\'elien, donc, d'apr\`es \cite[VI, \S 8, Exercices 14a et 12b]{BourbakiAC57} ou \cite[Proposition~2.4.1]{Ber2}, le polyn\^ome~$P$ reste irr\'eductible dans~$\cH(b)[T]$. D'apr\`es le lemme~\ref{lem:PTp}, il existe un polyn\^ome irr\'eductible $Q \in \cH(b)[T]$ et un entier $r \in \N$ tel que $P(T^{e^s}) = Q(T)^{e^r}$. Le r\'esultat s'ensuit.
\end{proof}

\begin{defi}\label{def:epaisseur}\index{Point!rigide epais@rigide \'epais!epaisseur@\'epaisseur|textbf}
Avec les notations du lemme~\ref{lem:polmin}, on appelle \emph{\'epaisseur du point~$x$} l'entier
\[\delta(x) := e^r \in\N^{\ast},\]
o\`u $e$ d\'esigne l'exposant caract\'eristique de~$\kappa(b)$.
\end{defi}
\index{Point!rigide epais@rigide \'epais!polynome minimal@polyn\^ome minimal|)}
\index{Point!rigide epais@rigide \'epais|)}

Nous pouvons maintenant \'enoncer un th\'eor\`eme de pr\'eparation de Weierstra\ss{} g\'en\'eralisant \cite[th\'eor\`eme~8.6]{EtudeLocale} et valable pour tout point rigide \'epais d'une droite relative.

\begin{theo}[Pr\'eparation de Weierstra\ss]\label{thm:preparationW}\index{Theoreme@Th\'eor\`eme!de pr\'eparation de Weierstra\ss}
Supposons que~$x$ est un point rigide \'epais. Si~$\mu_{x}$ est ins\'eparable et si $b$ est trivialement valu\'e, supposons que~$b$ est ultram\'etrique typique. Soit~$G$ un \'el\'ement de l'anneau local $\cO_{X,x}$. Supposons que son image dans~$\cO_{X_{b},x}$ n'est pas nulle et notons~$n$ sa valuation $\mu_{x}$-adique. Alors $n$ est multiple de~$\delta(x)$ et il existe un unique couple $(\Omega,E) \in \cO_{X,x}^2$ tel que
\begin{enumerate}[i)]
\item $\Omega \in \cO_{B,b}[T]$ est un polyn\^ome unitaire de degr\'e $n \deg(\mu_{x})$ dont l'image $\Omega(b)$ dans $\kappa(b)[T]$ satisfait $\Omega(b) = \mu_{\kappa,x}^{n/\delta(x)}$ ;
\item $E$ est inversible dans $\cO_{X,x}$ ;
\item $G = \Omega E$.
\end{enumerate}
\end{theo}
\begin{proof}
Posons $d := \deg(\mu_{x})$ et $\delta := \delta(x)$. Le polyn\^ome~$\mu_{\kappa,x}$ est alors de degr\'e $d\delta$. Choisissons un relev\'e~$M$ de~$\mu_{\kappa,x}$ unitaire de degr\'e~$d\delta$ dans~$\cO_{B,b}[T]$. Lorsque l'on parlera de valuation, il s'agira de la valuation $\mu_{x}$-adique.

Effectuons la division euclidienne de~$n$ par~$m$~: il existe $a\in \N$ et $b\in \cn{0}{\delta-1}$ tels que $n = a \delta + b$. D'apr\`es le th\'eor\`eme de division de Weierstra\ss{} \ref{weierstrass} appliqu\'e \`a~$G$ et~$M^a$ , il existe $Q \in \cO_{X,x}$ et $R\in \cO_{B,b}[T]$ de degr\'e strictement inf\'erieur \`a~$ad\delta$ tels que $G = QM^a + R$. Puisque le degr\'e de~$R$ est strictement inf\'erieur \`a~$ad\delta$, sa valuation est soit infinie, soit strictement inf\'erieure \`a~$a\delta$. Puisqu'elle ne peut \^etre strictement inf\'erieure aux valuations de~$G$ et de~$QM^a$, on en d\'eduit qu'elle est infinie. En d'autres termes, on a $R(b) = 0$.

Le raisonnement qui pr\'ec\`ede montre que la valuation de~$Q$ est \'egale \`a~$b$. Supposons, par l'absurde, que $b>0$. Alors, en appliquant le th\'eor\`eme de division de Weierstra\ss{} \`a~$M$ et \`a~$Q$, on obtient un reste dans $\cO_{B,b}[T]$ qui s'annule en~$x$ et dont le degr\'e est strictement inf\'erieur \`a~$bd$, et donc \`a $d\delta = \deg(\mu_{\kappa,x})$. On aboutit donc \`a une contradiction, ce qui d\'emontre que $b=0$. On en d\'eduit que~$\delta$ divise~$n$ (et que~$Q$ est inversible dans~$\cO_{X,x}$).

En appliquant, \`a pr\'esent, le th\'eor\`eme de division de Weierstra\ss{} \`a~$M^a$ et~$G$, on obtient une \'egalit\'e de la forme $M^a = Q'G+R'$. En raisonnant sur les valuations comme pr\'ecedemment, on montre que~$Q'$ est inversible dans~$\cO_{X,x}$ et que~$R'(b)=0$. Les fonctions $\Omega := M^a - R'$ et $E := Q'^{-1}$ satisfont alors les propri\'et\'es de l'\'enonc\'e.

Il reste \`a d\'emontrer l'unicit\'e du couple~$(\Omega,E)$ satisfaisant les propri\'et\'es de l'\'enonc\'e. Partant d'un tel couple, $R := \Omega - M^{n/\delta}$ est un \'el\'ement de~$\cO_{B,b}[T]$ de degr\'e strictement inf\'erieur \`a~$nd$ et on a $M^{n/\delta} = E^{-1}G - R$. L'unicit\'e du couple~$(\Omega,E)$ se d\'eduit donc de l'\'enonc\'e d'unicit\'e dans le th\'eor\`eme de division de Weierstra\ss.
\end{proof}

\subsubsection{R\'esultats}

Revenons au cas o\`u $X = \E{n}{\cA}$, avec $n\in \N$ arbitraire. Le r\'esultat qui suit permet de comprendre l'anneau local en un point de~$X$ en fonction de celui en sa projection sur la base~$B$. C'est une combinaison de \cite[corollaire~9.11 et th\'eor\`emes 9.17 et~9.18]{EtudeLocale}.

\begin{theo}\label{rigide}\index{Anneau!noetherien@noeth\'erien}\index{Anneau!fortement regulier@fortement r\'egulier}\index{Anneau!fortement de valuation discrete@fortement de valuation discr\`ete}\index{Corps!fort}
Supposons que~$\cA$ est basique. Soient~$b$ un point de~$B$ et~$x$ un point de~$X$ au-dessus de~$b$.
Alors l'anneau local~$\cO_{\E{n}{\cA},x}$ est noeth\'erien et fortement r\'egulier de dimension 
\[\dim(\cO_{X,x})=\dim(\cO_{B,b})+n-m_{x},\]
o\`u $m_{x}$ est d\'efini comme dans la remarque~\ref{rem:rigeptrans}.

Si~$x$ est puremement localement transcendant au-dessus de~$b$ (c'est-\`a-dire si $m_{x}=n$) et si $\cO_{B,b}$ est un corps fort (resp. un anneau fortement de valuation discr\`ete d'uniformisante~$\varpi_b$), alors~$\cO_{X,x}$ est un corps fort (resp. un anneau fortement de valuation discr\`ete d'uniformisante~$\varpi_b$). 
\qed
\end{theo}

Int\'eressons-nous, \`a pr\'esent, au principe du prolongement analytique. Nous adopterons la d\'efinition suivante, qui reste maniable en l'absence de connexit\'e locale.

\begin{defi}[\protect{\cite[d\'efinition~11.1]{EtudeLocale}}]\label{def:prolongementanalytique}\index{Prolongement analytique|textbf}\index{Prolongement analytique|(}
Soit $S$ un espace localement annel\'e. Soit $s\in S$. On dit que \emph{$S$ satisfait le principe du prolongement analytique en~$s$} si, pour tout ouvert~$U$ de~$S$ contenant~$s$ et tout \'el\'ement~$f$ de~$\cO_{S}(U)$ dont l'image dans~$\cO_{S,s}$ n'est pas nulle, il existe un voisinage~$V$ de~$s$ dans~$U$ tel que, pour tout $t\in V$, l'image de~$f$ dans~$\cO_{S,t}$ n'est pas nulle.

On dit que \emph{$S$ satisfait le principe du prolongement analytique} s'il le satisfait en tout point.
\end{defi}

\begin{rema}\label{rem:prolanversions}
Soit~$S$ un espace localement annel\'e. Pour tout ouvert~$U$ de~$S$ et tout \'el\'ement~$f$ de~$\cO_{S}(U)$, posons 
\[N_{U,f} := \{s \in U : f \ne 0 \textrm{ dans } \cO_{S,s}\}.\] 
C'est une partie ferm\'ee de~$U$.

L'espace~$S$ satisfait le principe du prolongement analytique au sens de la d\'efinition~\ref{def:prolongementanalytique} si, et seulement si, pour tout ouvert~$U$ de~$S$ et tout \'el\'ement~$f$ de~$\cO_{S}(U)$, l'ensemble $N_{U,f}$ est ouvert. 

La version classique du principe du prolongement analytique stipule que, pour tout ouvert connexe~$U$ de~$S$ et tout \'el\'ement~$f$ de~$\cO_{S}(U)$, s'il existe $s\in U$ tel que l'image de~$f$ dans~$\cO_{S,s}$ soit nulle, alors $f$ est nulle dans~$\cO_{S}(U)$. En d'autres termes, l'ensemble $N_{U,f}$ est soit vide, soit \'egal \`a~$U$. On observe que cet \'enonc\'e d\'ecoule du principe du prolongement analytique au sens de la d\'efinition~\ref{def:prolongementanalytique}.

Ce raisonnement permet de montrer que, dans le cas o\`u $S$ est localement connexe, les deux versions du prolongement analytique sont \'equivalentes.

\end{rema}


\begin{rema}\label{rem:prolongementanalytiquecorpsavd}
Soit $S$ un espace localement annel\'e et soit $s\in S$. 

Si $\cO_{S,s}$ est un corps, alors $S$ satisfait le principe du prolongement analytique en~$s$.

Si $\cO_{S,s}$ est un anneau de valuation discr\`ete d'uniformisante~$\pi$, alors $S$ satisfait le principe du prolongement analytique en~$s$ si, et seulement si, il existe un voisinage ouvert~$U$ de~$s$ dans~$S$ tel que, pour tout $t\in U$, l'image de~$\pi$ dans~$\cO_{S,t}$ ne soit pas nulle.
\end{rema}

\begin{exem}\index{Prolongement analytique}
Si~$\cA$ est l'un des anneaux cit\'e dans les exemples \ref{ex:corpsvalue} \`a~\ref{ex:Dedekind}, alors $\cM(\cA)$ satisfait le principe du prolongement analytique.
\end{exem}

\begin{prop}[\protect{\cite[corollaire~11.5]{EtudeLocale}}]\label{prop:prolongementanalytique}\index{Prolongement analytique}
Supposons que~$\cA$ est basique. Si~$\cM(\cA)$ satisfait le principe du prolongement analytique, alors il en va de m\^eme pour~$\E{n}{\cA}$.
\qed
\end{prop}
\index{Prolongement analytique|)}

\'Enon\c cons finalement un r\'esultat de coh\'erence. Il sera utilis\'e \`a de nombreuses reprises dans ce travail. 

\begin{theo}[\protect{\cite[th\'eor\`eme~11.9]{EtudeLocale}}]\label{coherent}\index{Faisceau!coherent@coh\'erent}
Supposons que~$\cA$ est basique et que~$\cM(\cA)$ satisfait le principe du prolongement analytique. Alors, le faisceau structural de~$\E{n}{\cA}$ est coh\'erent.
\end{theo}

\chapter[Cat\'egorie des espaces analytiques~: d\'efinitions]{Cat\'egorie des espaces analytiques~: d\'efinitions}
\label{def_cat}

Ce chapitre est consacr\'e \`a la construction de la cat\'egorie des espaces analytiques sur un anneau de Banach. Fixons un anneau de Banach $(\cA,\nm)$.

Dans la section~\ref{sec:catAan}, nous d\'efinissons les espaces $\cA$-analytiques et les morphismes analytiques entre ces espaces. Nous proc\'edons en plusieurs \'etapes en commen\c{c}ant par les ouverts d'espaces affines analytiques sur~$\cA$, en poursuivant avec leurs ferm\'es analytiques (appel\'es mod\`eles locaux $\cA$-analytiques), avant de traiter le cas g\'en\'eral. La cat\'egorie correspondante est not\'ee~$\cAAn$. 

Nous attirons l'attention du lecteur sur le fait que, lorsque $\cA=k$ est un corps valu\'e ultram\'etrique complet, la d\'efinition propos\'ee ne permet pas de retrouver tous les espaces $k$-analytiques d\'efinis par V.~Berkovich dans~\cite{Ber1,Ber2}, mais seulement les espaces sans bords.

Dans la section~\ref{sec:An_A}, nous d\'efinissons la cat\'egorie~$\An_\cA$ des espaces analytiques au-dessus de~$\cA$. Plus grosse que la premi\`ere, elle autorise des morphismes entre espaces sur des anneaux de Banach de base diff\'erents (mais n\'eanmoins reli\'es par un morphisme born\'e). Elle nous permettra, dans un chapitre ult\'erieur, d'effectuer des extensions des scalaires ou encore de munir les fibres des morphismes de structures analytiques. 

Finalement, dans la section~\ref{sec:immersion}, nous introduisons la notion d'immersion d'espaces $\cA$-analytiques et en d\'emontrons quelques propri\'et\'es \'el\'ementaires, analogues \`a celles dont on dispose dans d'autres cadres.


\section[Espaces $\cA$-analytiques]{Cat\'egorie des espaces $\cA$-analytiques}\label{sec:catAan}

Nous d\'efinissons ici les notions d'espace analytique et de morphisme d'espaces analytiques au-dessus d'un anneau de Banach~$\cA$. Ces d\'efinitions ne requi\`erent  pas de propri\'et\'es particuli\`eres sur~$\cA$.

\subsection{Morphismes entre ouverts d'espaces affines analytiques}

\index{Morphisme!analytique|see{Morphisme analytique}}
\index{Morphisme analytique|(}

\begin{defi}\label{def:morphismeouvertaffine}\index{Morphisme analytique|textbf}
Soient~$U$ et~$V$ des ouverts d'espaces affines analytiques sur~$\cA$.
Un \emph{morphisme analytique} de~$U$ dans~$V$ est un morphisme d'espaces localement annel\'es
$\varphi \colon U \to V$
v\'erifiant la condition suivante~: pour toute partie compacte~$U'$ de~$U$ et toute partie compacte~$V'$ de~$V$ telles que $\varphi(U') \subset V'$, le morphisme~$\cO_V(V')\to\cO_U(U')$ induit par $\varphi^\sharp$ est contractant (\textit{i.e.} pour tout~$f\in\cO_V(V')$, on a $\|\varphi^\sharp(f)\|_{U'}\leq\|f\|_{V'}$, o\`u $\nm_{U'}$ et $\nm_{V'}$ d\'esignent les normes uniformes sur~$U'$ et~$V'$).
\end{defi}


Nous pouvons caract\'eriser ces morphismes de la fa\c con suivante.

\begin{prop}\label{isomouvert}\index{Morphisme!de corps r\'esiduels}
Soient $U$ et~$V$ des ouverts d'espaces affines analytiques sur~$\cA$. 
Soit~$\mor{\varphi}: (U,\cO_U)\to(V,\cO_V)$ un morphisme d'espaces localement annel\'es. Les deux conditions suivantes sont \'equivalentes~:
\begin{enumerate}[i)]
\item le morphisme~$\mor{\varphi}$ est un morphisme analytique ;
\item pour tout $x \in U$, l'extension de corps $\kappa(\varphi(x))\to\kappa(x)$ induite par~$\varphi^\sharp$ est isom\'etrique.
\end{enumerate}
\end{prop}
\begin{proof}
$i) \implies ii)$ 
Soit~$a\in\kappa(\varphi(x))$. Soit~$f\in\cO_{V,\varphi(x)}$ un relev\'e de~$a$. Soit~$V'$ un voisinage compact de~$\varphi(x)$ dans~$V$ sur lequel~$f$ est d\'efinie. Soit~$U'$ un voisinage compact de~$x$ dans~$\varphi^{-1}(V')$. Par hypoth\`ese, on a 
\[\|\varphi^\sharp(f)\|_{U'}\leq\|f\|_{V'}\]
et, en particulier, 
\[|\varphi^\sharp(f)(x)|\leq\|f\|_{V'}.\]
En outre, on a l'\'egalit\'e
\[ |a| = |f(\varphi(x))|=\lim_{V'\ni \varphi(x)}\|f\|_{V'},\] 
o\`u~$V'$ parcourt les voisinages compacts de~$\varphi(x)$ dans~$V$ sur lesquels~$f$ est d\'efini. On en d\'eduit que 
\[|\varphi^\sharp(f)(x)|\leq |f(\varphi(x))|,\]
puis que 
\[|\varphi^\sharp(f)(x)| = |f(\varphi(x))|\]
en distinguant selon que~$a$ est nul, auquel cas l'\'egalit\'e \'evidente, ou non nul, auquel cas le r\'esultat d\'ecoule de l'in\'egalit\'e pour~$a^{-1}$.

\medbreak

$ii) \implies i)$ 
Soient~$x\in U$, $V'$ un voisinage compact de~$\varphi(x)$ dans~$V$ et~$U'$ un voisinage compact de~$x$ dans~$\varphi^{-1}(V')$. Soit $f\in \cO_{V}(V')$. Par hypoth\`ese, pour tout~$x'\in U'$, on a 
\[|\varphi^\sharp(f)(x')|  = |f(\varphi(x'))| \le \|f\|_{V'}.\] 
Le r\'esultat s'en d\'eduit.
\end{proof}

Il est utile de noter que les morphismes analytiques induisent \'egalement des morphismes entre anneaux du type~$\cB(W)$. 
Remarquons que, si $W$ est une partie compacte d'un espace affine, tout \'el\'ement de~$\cO(W)$ (au sens surconvergent) induit, par d\'efinition, une fonction
\[W \too \bigsqcup_{x\in W} \cH(x),\] 
et de m\^eme pour tout \'el\'ement de~$\cB(W)$. On peut donc comparer des \'el\'ements de~$\cO(W)$ et~$\cB(W)$.


\begin{lemm}\label{lem:morphismeB}\index{Fonction!B-definie@$\cB$-d\'efinie}
Soient~$U$ et~$V$ des ouverts d'espaces affines analytiques sur~$\cA$ et $\varphi \colon U \to V$ un morphisme analytique. Soit $x \in U$ et soit~$V'$ un voisinage compact de~$\varphi(x)$ dans~$V$. Alors il existe un voisinage compact~$U'$ de~$x$ dans~$U$ tel que $\varphi(U')\subset \mathring{V}'$ et, pour tout $F \in \cB(V')$, $\varphi^\sharp(F_{| \mathring{V}'})_{|U'} \in \cB(U')$.
\end{lemm}
\begin{proof}
Soit $n\in \N$ (resp. $m\in \N$) tel que $U$ (resp. $V$) soit un ouvert de l'espace affine $\E{n}{\cA}$ (resp. $\E{m}{\cA}$). Notons $T=(T_{1},\dotsc,T_{n})$ (resp. $S=(S_{1},\dotsc,S_{m})$) les coordonn\'ees sur ce dernier. Posons $\varphi^\sharp(S) := (\varphi^\sharp(S_{1}),\dotsc,\varphi^\sharp(S_{m}))$.

Soit $U'$ un voisinage compact de~$x$ dans~$U$ tel que $\varphi(U') \subset \mathring{V}'$. Soit $M\in \R_{>0}$ tel que, pour tout $i \in \cn{1}{m}$, on ait $\|\varphi^\sharp(S_{i})\|_{U'} < M$. Quitte \`a restreindre~$U'$, on peut supposer que, pour tout $i \in \cn{1}{m}$, il existe une suite de fractions rationnelles $(A_{i,p}(T))_{p\in \N}$ sans p\^oles sur~$U'$ qui converge uniform\'ement vers~$\varphi^\sharp(S_{i})$ sur~$U'$. On peut \'egalement supposer que, pour tout $p\in \N$, on a $\|A_{i,p}(T)\|_{U'} < M$. Posons $A_{p}(T) = (A_{1,p}(T),\dotsc,A_{m,p}(T))$. 


Soit $F \in \cB(V')$. Il existe une suite de fractions rationnelles $\big(B_{q}(S) = \frac{P_{q}(S)}{Q_{q}(S)}\big)_{q\in \N}$ sans p\^oles sur~$V'$ qui converge uniform\'ement vers~$F$ sur~$V'$. 

Soit $N \in \N$. Il existe $q_{N} \in \N$ tel que
\[ \Big\|F - \frac{P_{q_{N}}(S)}{Q_{q_{N}}(S)} \Big\|_{V'} \le \frac1{2^{N+1}},\]
et donc
\[ \Big\|\varphi^\sharp(F) - \varphi^\sharp\Big(\frac{P_{q_{N}}(S)}{Q_{q_{N}}(S)}\Big) \Big\|_{U'} = 
\Big\|\varphi^\sharp(F) - \frac{P_{q_{N}}(\varphi^\sharp(S))}{Q_{q_{N}}(\varphi^\sharp(S))} \Big\|_{U'} 
\le  \frac1{2^{N+1}}.\]

Il existe $C_{N}\in \R_{>0}$ tel que, pour tout corps valu\'e~$(K,\va)$ tel qu'il existe un morphisme born\'e $\psi \colon \cA \to K$ et tous $a_{1},\dotsc,a_{m},b_{1},\dotsc,b_{m} \in K$ tels que $\max_{1\le i\le m}(|b_{i} - a_{i}|) \le M$, on ait
\[ |\psi(P_{q_{N}})(b) - \psi(P_{q_{N}})(a)| \le C_{N} \, \max_{1\le i\le m}(|b_{i} - a_{i}|) \]
et
\[ |\psi(Q_{q_{N}})(b) - \psi(Q_{q_{N}})(a)| \le C_{N} \, \max_{1\le i\le m}(|b_{i} - a_{i}|) .\]

Par hypoth\`ese, $Q_{q_{N}}(S)$ ne s'annule pas sur~$V'$. Par compacit\'e de~$V'$, on a
\[ m_{N} := \inf(\{ |Q_{q_{N}}(S)(y)| : y \in V'\}) > 0.\]
La borne inf\'erieure de la valeur absolue de $\varphi^\sharp(Q_{q_{N}}(S)) = Q_{q_{N}}(\varphi^\sharp(S))$ sur~$U'$ est donc sup\'erieure \`a~$m_{N}$,
donc il existe $p_{N} \in \N$ tel que, pour tout $p\ge p_{N}$,
\[  \inf(\{ |Q_{q_{N}}(A_{p})(z)| : z \in U'\})  \ge \frac12 \, m_{N} >0.\]
On en d\'eduit qu'il existe $p'_{N} \ge p_{N}$ tel que 
\[ \Big\|\frac{P_{q_{N}}(\varphi^\sharp(S))}{Q_{q_{N}}(\varphi^\sharp(S))} - \frac{P_{q_{N}}(A_{p'_{N}}(T))}{Q_{q_{N}}(A_{p'_{N}}(T))} \Big\|_{U'} \le  \frac1{2^{N+1}},\]
et donc  
\[\Big\|\varphi^\sharp(F) - \frac{P_{q_{N}}(A_{p'_{N}}(T))}{Q_{q_{N}}(A_{p'_{N}}(T))} \Big\|_{U'} 
\le  \frac1{2^{N}}.\]
Le r\'esultat s'ensuit.
\end{proof}

Concluons avec un exemple classique et fondamental.

\begin{exem}\label{exemple_morphisme}\index{Morphisme analytique!vers un espace affine}
Soient $n,m \in \N$. Soient~$U$ un ouvert de~$\E{m}{\cA}$ et $f_{1},\dotsc,f_{n} \in \cO(U)$. Nous allons expliquer comment construire un morphisme analytique $\varphi \colon U \to \E{n}{\cA}$ \`a partir de ces donn\'ees. Plus pr\'ecis\'ement, si l'on d\'esigne par $T_{1},\dotsc,T_{n}$ les coordonn\'ees sur~$\E{n}{\cA}$, le morphisme~$\varphi$ satisfera la condition suivante~:
\[ \forall i\in \cn{1}{n},\ \varphi^\sharp(T_{i})=f_{i}.\]

Commen\c cons par d\'efinir l'application ensembliste sous-jacente. Soit $x\in U$. Notons~$\varphi(x)$ le point de $\E{n}{\cA}$ associ\'e \`a la semi-norme multiplicative
\[\renewcommand{\arraystretch}{1.2}
\begin{array}{ccc}
\cA[T_1,\ldots,T_n] & \too & \R_{\ge 0}\\
P(T_{1},\dotsc,T_{n}) & \mapstoo & |P(f_{1}(x),\dotsc,f_{n}(x))|
\end{array}.\]
On d\'efinit ainsi une application $\varphi \colon x \in U \mapsto \varphi(x) \in \E{n}{\cA}$.
Sa continuit\'e d\'ecoule du lemme~\ref{lem:evaluationcontinueaffine}.

Il suit \'egalement de la d\'efinition que, pour tout point~$x$ de~$U$, on a un morphisme de corps isom\'etrique
\[ \varphi_{x}^\sharp \colon \cH(\varphi(x))\too\cH(x).\]

Construisons, \`a pr\'esent, l'application~$\varphi^\sharp$. Soient~$V$ un ouvert de~$\E{n}{\cA}$ et $g$ un \'el\'ement de~$\cO(V)$. Par d\'efinition, ce dernier est une application
\[g \colon V\too\displaystyle\coprod_{y\in V}\cH(y)\] 
telle que
\begin{enumerate}[i)]
\item pour tout $y\in V$, on a $g(y)\in\cH(y)$~;
\item pour tout~$y\in V$, il existe un voisinage compact~$V_{y}$ de~$y$ dans~$V$ tel que~$g$ soit limite uniforme sur~$V_{y}$ d'une suite de fractions rationnelles sans p\^oles. 
\end{enumerate}

Posons 
\[\renewcommand{\arraystretch}{1.5}
\fonction{\varphi^\sharp(g)}{\varphi^{-1}(V)}{\displaystyle\coprod_{x\in \varphi^{-1}(V)}\cH(x)}{x}{\varphi_{x}^\sharp(g(\varphi(x)))}.\]
Montrons que~$\varphi^\sharp(g)$ est un \'el\'ement de~$\cO(\varphi^{-1}(V))$. 

Soit~$x\in \varphi^{-1}(V)$. Posons $y :=\varphi(x)$. Il existe un voisinage compact~$V_{y}$ de~$y$ dans~$V$ et une suite de fractions rationnelles $\left(\frac{P_{i}}{Q_{i}}\right)_{i\in\N}$ sans p\^oles sur~$V_{y}$ qui converge uniform\'ement vers~$g$ sur~$V_{y}$. Soit~$U_{x}$ un voisinage compact de~$x$ dans~$\varphi^{-1}(V_{y})$. Consid\'erons la suite $\left(\frac{P_i(f_1,\ldots,f_n)}{Q_i(f_1,\ldots,f_n)}\right)_{i\in\N}$ d'\'el\'ements de~$\cO(U_{x})$. Pour tout $x' \in U_{x}$, on a 
\[ \left| \frac{P_i(f_1,\ldots,f_n)(x')}{Q_i(f_1,\ldots,f_n)(x')} - \varphi^\sharp(g)(x')\right| = \left| \frac{P_i(\varphi(x'))}{Q_i(\varphi(x'))} - g(\varphi(x')) \right|.\]
On en d\'eduit que la suite~$\left(\frac{P_i(f_1,\ldots,f_n)}{Q_i(f_1,\ldots,f_n)}\right)_{i\in\N}$ converge uniform\'ement vers~$\varphi^\sharp(g)$ sur~$U_{x}$. Par cons\'equent, $\varphi^\sharp(g)$ appartient bien \`a~$\cO(\varphi^{-1}(V))$.

Remarquons que, pour tout $x \in \varphi^{-1}(V)$, on a 
\begin{equation}\label{eq:phifg} \tag{*}
|\varphi^\sharp(g)(x)| = |g(\varphi(x))|.
\end{equation}

Cette \'egalit\'e entra\^ine que le morphisme $\varphi^\sharp \colon \cO_{\varphi(x)}\to\cO_{x}$ est un morphisme local d'anneaux locaux. Pour le d\'emontrer, il suffit d'utiliser le fait que, dans chacun de ces anneaux locaux, l'id\'eal maximal est constitu\'e des \'el\'ements qui s'annulent au point consid\'er\'e. 

Finalement, nous avons d\'efini un morphisme d'espaces localement annel\'es $\varphi \colon U\to \E{n}{\cA}$. D'apr\`es~\eqref{eq:phifg}, c'est un morphisme analytique. Par construction, pour tout $i\in \cn{1}{n}$, on a $\varphi^\sharp(T_{i})=f_{i}$.

Ajoutons que le morphisme~$\varphi$ que nous venons de construire est l'unique morphisme qui satisfasse les conditions $\varphi^\sharp(T_{i})=f_{i}$. Cette propri\'et\'e d\'ecoule directement des d\'efinitions. Nous renvoyons au d\'ebut de la preuve de la proposition~\ref{morphsec} pour les d\'etails. (La condition d'anneau de base g\'eom\'etrique qui figure son \'enonc\'e n'intervient pas dans le cadre de notre exemple.)
\end{exem}

\begin{exem}\label{ex:structural}\index{Morphisme!structural|textbf}\index{Projection}
Lorsque $n\le m$ et qu'on applique le r\'esultat de l'exemple~\ref{exemple_morphisme} avec $U = \E{m}{\cA}$ et, pour tout $i\in \cn{1}{n}$, $f_{i} = T_{i}$, on obtient un morphisme analytique
\[ \pi_{m,n} \colon \E{m}{\cA} \too \E{n}{\cA}\]
qui n'est autre que la projection sur les $n$ premi\`eres coordonn\'ees (\cf~d\'efinition~\ref{def:projection}).

Dans le cas particulier o\`u $n=0$, on obtient un morphisme analytique
\[ \pi_{m} \colon \E{m}{\cA} \too \cM(\cA),\]
qui n'est autre que la projection sur la base (\cf~d\'efinition~\ref{def:projection}). On l'appelle \'egalement \emph{morphisme structural}.

\end{exem}

\index{Morphisme analytique|)}

\subsection{Mod\`eles locaux $\cA$-analytiques}

\index{Modele local analytique@Mod\`ele local analytique|(}

\begin{defi}\label{def:fermeanalytique}\index{Ferme analytique@Ferm\'e analytique|textbf}
Soit~$U$ un ouvert de~$\E{n}{\cA}$. Un \emph{ferm\'e analytique} de~$U$ est un quadruplet~$(X,\cO_X,j,\cI)$, o\`u~$\cI$ est un faisceau coh\'erent d'id\'eaux de~$\cO_U$, $X$~est le support du faisceau~$\cO_U/\cI$, $j \colon X \to U$ l'inclusion canonique et $\cO_X$ le faisceau $j^{-1} (\cO_U/\cI)$. Le couple~$(X,\cO_X)$ est un espace localement annel\'e.

\end{defi}

\begin{rema}\label{rem:fermesanalytiques}
Soient~$(X,\cO_X,j,\cI)$ un ferm\'e analytique d'un ouvert~$U$ de~$\E{n}{\cA}$ et~$Y$ un ouvert de~$X$. Il existe alors un ouvert~$V$ de~$U$ tel que~$X\cap V=Y$ et $Y$~h\'erite naturellement d'une structure de ferm\'e analytique de~$V$, le faisceau d'id\'eaux associ\'e \'etant~$\cI_{|V}$.
\end{rema}

\begin{defi}\index{Immersion!fermee@ferm\'ee|textbf}
Soit~$(X,\cO_X,j,\cI)$ un ferm\'e analytique d'un ouvert~$U$ de~$\E{n}{\cA}$. On a un morphisme naturel $j^\sharp \colon \cO_{U} \to j_{\ast}\cO_{X} = \cO_{U}/\cI$. Le couple~$\mor{j}$ d\'efinit un morphisme d'espaces localement annel\'es de~$(X,\cO_X)$ dans~$(U,\cO_U)$ appel\'e \emph{immersion ferm\'ee} de~$X$ dans~$U$.

Ce morphisme est injectif et induit des isomorphismes isom\'etriques entre les corps r\'esiduels.
\end{defi}

\begin{defi}\label{def:modelelocal}\index{Modele local analytique@Mod\`ele local analytique|textbf}\index{Morphisme!structural|textbf}
Un \emph{mod\`ele local $\cA$-analytique} est la donn\'ee d'un entier~$n$, d'un ouvert~$U$ de~$\E{n}{\cA}$ et d'un ferm\'e analytique~$X$ de~$U$.

Le morphisme $\E{n}{\cA} \to \cM(\cA)$ de l'exemple~\ref{ex:structural} induit un morphisme d'espaces localement annel\'es $\pi \colon X \to \cM(\cA)$ appel\'e \emph{morphisme structural}.
\end{defi}

\index{Morphisme analytique|(}
\begin{defi}\label{def:morphismefermeouvert}\index{Morphisme analytique|textbf}
Soient $X$ et $Y$ des mod\`eles locaux $\cA$-analytiques. Par d\'efinition, il existe des immersions ferm\'ees $j_{U} \colon X \to U$ et $j_{V} \colon Y \to V$, o\`u $U$ et $V$ sont des ouverts d'espaces analytiques sur~$\cA$.

Un \emph{morphisme analytique} de~$X$ dans~$Y$ est un morphisme d'espaces localement annel\'es
$\varphi \colon X \to Y$
v\'erifiant la condition suivante : pour tout point~$x$ de~$X$, il existe un voisinage ouvert~$U'$ de~$j_{U}(x)$ dans~$U$ et un morphisme analytique 
$\tilde \varphi \colon U' \to V$
tels que le diagramme
\[\begin{tikzcd}
U' \arrow[r, "\tilde{\varphi}"] &V\\
j_{U}^{-1}(U') \arrow[u, hook, "j_{U}"] \arrow[r, "\varphi"'] & \arrow[u, hook, "j_{V}"] Y
\end{tikzcd}\]
commute.
\end{defi}

\begin{exem}\label{ex:structuralmodele}\index{Morphisme!structural}
Le morphisme structural $\pi \colon X \to \cM(\cA)$ est un morphisme analytique.
\end{exem}

\begin{rema}
Soit $\varphi \colon X \to Y$ un morphisme analytique entre mod\`eles locaux $\cA$-analytiques. Soient~$X'$ et~$Y'$ des ouverts de~$X$ et~$Y$ respectivement tels que $\varphi(X') \subset Y'$. Alors, $\varphi$ induit par restriction un morphisme analytique $X' \to Y'$.
\end{rema}

\begin{lemm}\label{lem:localmodele}
La notion de morphisme analytique entre mod\`eles locaux $\cA$-analytiques est locale \`a la source et au but. 

Soit $\varphi \colon X \to Y$ un morphisme d'espaces localement annel\'es entre mod\`eles locaux $\cA$-analytiques. Alors $\varphi$ est un morphisme analytique si, et seulement si, pour tout point $x$ de~$X$ et tout point~$y$ de~$Y$, il existe un voisinage ouvert~$X'$ de~$x$ et un voisinage ouvert~$Y'$ de~$y$ tels que $\varphi(X') \subset Y'$ et le morphisme $X' \to Y'$ induit par~$\varphi$ soit un morphisme analytique.
\qed
\end{lemm}

\begin{lemm}\label{lem:compositionmodele}
Les morphismes compatibles entre mod\`eles locaux $\cA$-analytiques analytiques se composent.

Soient $\varphi \colon X \to Y$ et $\psi \colon Y \to Z$ des morphismes analytiques entre mod\`eles locaux $\cA$-analytiques. Alors $\psi \circ \varphi$ est un morphisme analytique.
\qed
\end{lemm}

\begin{prop}\label{prop:isometrie}\index{Morphisme!de corps r\'esiduels}
Soit $\varphi \colon X \to Y$ un morphisme entre mod\`eles locaux $\cA$-analytiques. Pour tout point~$x$ de~$X$, l'extension de corps $\kappa(\varphi(x)) \to \kappa(x)$ induite par le morphisme~$\varphi^\sharp$ est isom\'etrique.

\end{prop}
\begin{proof}
On se ram\`ene imm\'ediatement au cas o\`u~$X$ et~$Y$ sont des ouverts d'espaces affines analytiques. Le r\'esultat d\'ecoule alors de la proposition~\ref{isomouvert}.
\end{proof}

\index{Morphisme analytique|)}
\index{Modele local analytique@Mod\`ele local analytique|)}

\subsection{Espaces $\cA$-analytiques}
\index{Espace analytique|(}

\begin{defi}\label{def:carte}\index{Carte analytique|textbf}\index{Atlas analytique|textbf}
Soit~$X$ un espace localement annel\'e. 

Une \emph{carte $\cA$-analytique} de~$X$ est la donn\'ee d'un ouvert~$U$ de~$X$, d'un mod\`ele local $\cA$-analytique~$Z_{U}$ et d'un isomorphisme d'espaces localement annel\'es~$j_U \colon U\xrightarrow[]{\sim} Z_U$. 

 Un \emph{atlas $\cA$-analytique} de~$X$ est un ensemble $\{(U,Z_{U},j_{U})\}_{U \in \cU}$ de cartes $\cA$-analytiques de~$X$ tel que l'ensemble $\{U\}_{U\in \cU}$ forme un recouvrement ouvert de~$X$ et que, pour tous $U,U' \in \cU$, le morphisme 
\[j_{U}\circ j_{U'}^{-1} \colon j_{U'}(U\cap U') \too j_{U}(U\cap U')\]
soit un isomorphisme analytique (entre mod\`eles locaux $\cA$-analytiques).

Deux atlas $\cA$-analytiques sur~$X$ sont dits \emph{compatibles} si leur r\'eunion est encore un atlas $\cA$-analytique de~$X$. Cela d\'efinit une relation d'\'equivalence sur l'ensemble des atlas $\cA$-analytiques d'un espace localement annel\'e. 
\end{defi}

\begin{rema}
Il existe un unique atlas maximal compatible avec un atlas donn\'e sur un espace $\cA$-analytique. On l'obtient en consid\'erant la r\'eunion de tous les atlas avec lesquels il est compatible.
\end{rema}

\begin{defi}\label{def:espaceanalytique}\index{Espace analytique|textbf}\index{Morphisme!structural|textbf}
Un \emph{espace $\cA$-analytique} est un espace localement annel\'e muni d'une classe d'\'equivalence d'atlas $\cA$-analytiques.

Les morphismes structuraux des cartes d'un espace $\cA$-analytique~$X$ se recollent en un morphisme d'espaces localement annel\'es $\pi \colon X \to \cM(\cA)$, encore appel\'e \emph{morphisme structural}.
\end{defi}

Lorsque nous parlerons d'un atlas d'un $\cA$-espace analytique, il s'agira d'un \'el\'ement de cette classe d'\'equivalence.

\begin{rema}\label{remarque_base_espace_analytique}
Soit~$X$ un espace $\cA$-analytique. Tout ouvert~$U$ de~$X$ h\'erite par restriction d'une structure d'espace $\cA$-analytique.
\end{rema}

\begin{rema}\label{rem:proplocales}
Il d\'ecoule de la d\'efinition \ref{def:espaceanalytique} que les propri\'et\'es locales des espaces affines $\cA$-analytiques se transf\`erent aux espaces $\cA$-analytiques g\'en\'eraux. Par exemple, tout espace $\cA$-analytique est localement compact (et donc localement s\'epar\'e).
\end{rema}

\begin{rema}\label{rem:kansansbords}
Lorsque $\cA = k$ est un corps valu\'e complet, la d\'efinition~\ref{def:espaceanalytique} ne permet pas de retrouver tous les espaces $k$-analytiques d\'efinis par V.~Berkovich dans~\cite{Ber1,Ber2}, mais seulement les espaces sans bords.
\end{rema}

Soit~$X$ un espace $\cA$-analytique. \`A tout point~$x$ de~$X$, on peut associer un corps r\'esiduel~$\kappa(x)$. En choisissant une carte de~$X$ contenant~$x$, on peut munir~$\kappa(x)$ d'une valeur absolue. D'apr\`es la proposition~\ref{prop:isometrie}, celle-ci ne d\'epend pas du choix de la carte.
\nomenclature[Ka]{$\kappa(x)$}{corps r\'esiduel d'un point~$x$ d'un espace $\cA$-analytique}

\begin{defi}\index{Corps!residuel@r\'esiduel|textbf}\index{Corps!residuel complete@r\'esiduel compl\'et\'e|textbf}\index{Evaluation@\'Evaluation|textbf}%
\nomenclature[Kb]{$\cH(x)$}{corps r\'esiduel compl\'et\'e du point~$x$}%
\nomenclature[Kc]{$f(x)$}{\'evaluation en $x$ d'une section $f$ de $\cO_{X}$}
Soit~$X$ un espace $\cA$-analytique. Soit~$x \in X$. On munit le corps r\'esiduel~$\kappa(x)$ de la valeur absolue provenant de n'importe quelle carte de~$X$ contenant~$x$. On note~$\cH(x)$ le compl\'et\'e correspondant.

Pour tout $f\in \cO_{X}(X)$, on note~$f(x)$ l'image de~$f$ dans~$\cH(x)$.
\end{defi}

\begin{lemm}\label{lem:evaluationcontinue}\index{Evaluation@\'Evaluation}
Soit~$X$ un espace $\cA$-analytique. Pour tout $f\in\cO_X(X)$, l'application 
\[\fonction{|f|}{X}{\R_{\ge 0}}{x}{|f(x)|}\]
est continue.
\end{lemm}
\begin{proof}
On se ram\`ene au cas d'un mod\`ele local, puis au cas d'un espace affine analytique, et on conclut alors par le lemme~\ref{lem:evaluationcontinueaffine}.
\end{proof}

\index{Morphisme analytique|)}

\begin{defi}\label{def:morphismegeneral}\index{Morphisme analytique|textbf}
\nomenclature[Ke]{$\Hom_{\cAAn}(X,Y)$}{pour $X,Y$ espaces $\cA$-analytiques, ensemble des morphismes analytiques de~$X$ dans~$Y$}
Soient~$X$ et~$Y$ deux espaces $\cA$-analytiques. Un \emph{morphisme analytique} de~$X$ dans~$Y$ est une application $\varphi \colon X \to Y$ v\'erifiant la propri\'et\'e suivante~: pour tout $x\in X$, il existe une carte~$U$ de~$X$ contenant~$x$ et une carte~$V$ de~$Y$ telles que $\varphi^{-1}(V)$ soit un ouvert de~$X$ et le morphisme
\[ Z_U \cap j_{U}(\varphi^{-1}(V)) \too Z_V\] 
induit par $j_V\circ\varphi \circ j_U^{-1}$ soit un morphisme analytique (entre mod\`eles locaux $\cA$-analytiques). 

On note $\Hom_{\cAAn}(X,Y)$ l'ensemble des morphismes analytiques de~$X$ dans~$Y$.
\end{defi}

%

\begin{rema}\label{atlas}
Un morphisme analytique $\varphi \colon X \to Y$ est un morphisme d'espaces localement annel\'es. 

La d\'efinition donn\'e est abusive en ce qu'elle ne pr\'ecise pas les atlas choisis sur~$X$ et~$Y$ (au sein de la classe d'\'equivalence fix\'ee). On v\'erifie cependant qu'elle ne d\'epend pas de ces choix, car la notion de morphisme analytique entre mod\`eles locaux $\cA$-analytiques est locale (\cf~lemme~\ref{lem:localmodele}).

\end{rema}

\begin{exem}\label{ex:structuralAanalytique}\index{Morphisme!structural}
Le morphisme structural $\pi \colon X \to \cM(\cA)$ est un morphisme analytique.
\end{exem}

Nous disposons de r\'esultats analogues \`a ceux des lemmes~\ref{lem:localmodele} et~\ref{lem:compositionmodele} et de la proposition~\ref{prop:isometrie}.

\begin{prop}\label{prop:proprietemorphismes}\index{Morphisme!de corps r\'esiduels}

\

\begin{enumerate}[i)]
\item La notion de morphisme analytique d'espaces $\cA$-analytiques est locale \`a la source et au but (au sens du lemme \ref{lem:localmodele}).

\item Les morphismes compatibles d'espaces $\cA$-analytiques se composent (au sens du lemme \ref{lem:compositionmodele}).

\item Les morphismes d'espaces $\cA$-analytiques induisent des plongements isom\'etriques entre les corps r\'esiduels (au sens de la proposition~\ref{prop:isometrie}).
\end{enumerate}
\qed
\end{prop}

\begin{defi}\index{Categorie@Cat\'egorie!des espaces A-analytiques@des espaces $\cA$-analytiques|textbf}
\nomenclature[Kd]{$\cAAn$}{cat\'egorie des espaces~$\cA$-analytiques}
La \emph{cat\'egorie des espaces $\cA$-analytiques} est la cat\'egorie dont les objets sont les espaces~$\cA$-analytiques et dont les morphismes sont les morphismes analytiques. Nous noterons~$\cAAn$ cette cat\'egorie.
\end{defi}


\index{Espace analytique|)}
\index{Morphisme analytique|)}

\section[Espaces analytiques au-dessus de~$\cA$]{Cat\'egorie des espaces analytiques au-dessus de~$\cA$}\label{sec:An_A}
\index{Espace analytique!au-dessus de~$\cA$|(}

Nous d\'efinissons ici une cat\'egorie plus grande que~$\cAAn$. Elle nous permettra de proc\'eder \`a des extensions des scalaires.

\begin{defi}\label{def:espaceaudessusdeA}\index{Espace analytique!au-dessus de~$\cA$|textbf}
Un \emph{espace analytique au-dessus de~$\cA$} est la donn\'ee d'un morphisme born\'e d'anneaux de Banach~$\cA\to\cB$ et d'un espace~$\cB$-analytique. 
\end{defi}

Remarquons que l'espace~$\cB$-analytique est alors muni une structure d'espace localement~$\cA$-annel\'e.

\medbreak

Soit $\tau \colon \cB\to\cB'$ un morphisme born\'e de~$\cA$-alg\`ebres de Banach. Notons $\cM(\tau) \colon \cM(\cB') \to \cM(\cB)$ le morphisme d'espaces localement annel\'es associ\'e (\cf~lemme~\ref{lem:M(tau)Anfoncteur}).

Nous allons maintenant d\'efinir une notion de morphisme d'espaces analytiques au-dessus de~$\tau$, allant d'un espace~$\cB'$-analytique vers un espace~$\cB$-analytique. Nous proc\'ederons de la m\^eme mani\`ere que pour les morphismes d'espaces~$\cA$-analytiques en commen\c cant par d\'efinir les morphismes entre ouvert d'espaces affines analytiques. 

\index{Morphisme analytique!au-dessus d'un morphisme d'anneaux de Banach|(}
\begin{defi}\label{defi:morphismeaudessusdefouvertaffine}\index{Morphisme analytique!au-dessus d'un morphisme d'anneaux de Banach|textbf}
Soient~$U$ un ouvert d'un espace affine analytique sur~$\cB'$ et $V$ un ouvert d'un espace affine analytique sur~$\cB$. Un \emph{morphisme analytique de~$U$ dans~$V$ au-dessus de~$\tau$} est un morphisme  d'espaces localement annel\'es $\varphi \colon U\to V$ qui fait commuter le diagramme
\[\begin{tikzcd}
U \arrow[r, "\varphi"] \arrow[d, "\pi"]& V \arrow[d, "\pi"]\\
\cM(\cB')  \arrow[r, "\cM(\tau)"] & \cM(\cB)
\end{tikzcd},\]
o\`u les morphismes verticaux sont les morphismes structuraux, 
et qui satisfait la condition suivante~: pour toute partie compacte~$U'$ de~$U$ et toute partie compacte~$V'$ de~$V$ telles que $\varphi(U') \subset V'$, le morphisme $\cO_V(V')\to\cO_U(U')$ induit par~$\varphi^\sharp$ est contractant (\textit{i.e.} pour tout~$a\in\cO_V(V')$, on a $\|\varphi^\sharp(a)\|_{U'}\leq\|a\|_{V'}$, o\`u $\nm_{U'}$ et $\nm_{V'}$ d\'esignent les normes uniformes sur~$U'$ et~$V'$).
\end{defi}


\begin{rema}
Dans la situation de la d\'efinition pr\'ec\'edente, le morphisme~$\tau$ munit~$U$ d'une structure d'espace localement $\cB$-annel\'e et le morphisme analytique~$\varphi$ respecte cette structure.
\end{rema}

\begin{exem}\label{ex:tildefn}
Soit~$n\in\N$. D'apr\`es le lemme~\ref{lem:M(tau)Anfoncteur}, le morphisme~$\tau$ induit un morphisme d'espaces localement annel\'e $\E{n}{\tau} \colon \E{n}{\cB'} \to \E{n}{\cB}$. C'est un morphisme analytique au-dessus de~$\tau$.
\end{exem}

\begin{defi}\label{defi:morphismeaudessusdeffermeouvert}\index{Morphisme analytique!au-dessus d'un morphisme d'anneaux de Banach|textbf}
Soient $X$ un ferm\'e analytique d'un ouvert~$U$ d'un espace affine analytique sur~$\cA$ et~$Y$ un  ferm\'e analytique d'un ouvert~$V$ d'un espace affine analytique sur~$\cA$. Notons $j_{U} \colon X \to U$ et $j_{V} \colon Y \to V$ les immersions ferm\'ees associ\'ees.

Un \emph{morphisme analytique de~$X$ dans~$Y$ au-dessus de~$\tau$} est un morphisme d'espaces localement annel\'es
$\varphi \colon X \to Y$
v\'erifiant la condition suivante : pour tout point~$x$ de~$X$, il existe un voisinage ouvert~$U'$ de~$j_{U}(x)$ dans~$U$ et un morphisme analytique 
$\tilde \varphi \colon U' \to V$ au-dessus de~$\tau$ 
tels que le diagramme
\[\begin{tikzcd}
U' \arrow[r, "\tilde{\varphi}"] &V\\
j_{U}^{-1}(U') \arrow[u, hook, "j_{U}"] \arrow[r, "\varphi"] & \arrow[u, hook, "j_{V}"] Y
\end{tikzcd}\]
commute.
\end{defi}

\begin{defi}\label{defi:morphismeaudessusdef}\index{Morphisme analytique!au-dessus d'un morphisme d'anneaux de Banach|textbf}%
\nomenclature[Ki]{$\Hom_{\An_{\cA},\tau}(X,Y)$}{pour $X$ espace $\cB$-analytique, $Y$ espace $\cB'$-analytique et $\tau\colon \cB \to \cB'$ morphisme d'anneaux de Banach, ensemble des morphismes analytiques de~$X$ dans~$Y$ au-dessus de~$\tau$}
Soient~$X$ un espace $\cB'$-analytique et~$Y$ un espace $\cB$-analytique. 
Un \emph{morphisme analytique de~$X$ dans~$Y$ au-dessus de~$\tau$} est une application $\varphi \colon X \to Y$ v\'erifiant la propri\'et\'e suivante~: pour tout $x\in X$, il existe une carte~$U$ de~$X$ contenant~$x$ et une carte~$V$ de~$Y$ telles que $\varphi^{-1}(V)$ soit un ouvert de~$X$ et le morphisme
\[ Z_U \cap j_{U}(\varphi^{-1}(V)) \too Z_V\] 
induit par $j_V\circ\varphi \circ j_U^{-1}$ soit un morphisme analytique au-dessus de~$\tau$. 


On note $\Hom_{\An_{\cA},\tau}(X,Y)$ l'ensemble des morphismes analytiques de~$X$ dans~$Y$ au-dessus de~$\tau$.
\end{defi}

\begin{rema}
Comme dans le cas des morphismes entres espaces $\cA$-analytiques, les morphismes analytiques au-dessus de~$\tau$ sont des morphismes d'espaces localement annel\'es.

\end{rema}

\begin{rema}
Dans la situation de la d\'efinition pr\'ec\'edente, on a un diagramme commutatif 
\[\begin{tikzcd}
X \arrow[r, "\varphi"] \arrow[d]& Y \arrow[d]\\
\cM(\cB')  \arrow[r, "\cM(\tau)"] & \cM(\cB)
\end{tikzcd}.\]
En particulier, le morphisme analytique~$\varphi$ respecte les structures d'espaces localement $\cB$-annel\'es de~$X$ (induite par le morphisme $\tau \colon \cB \to \cB'$) et~$Y$. 
\end{rema}

\begin{rema}
Supposons que le morphisme~$\tau$ est le morphisme identit\'e $\id \colon \cB \to \cB$ (et donc que $\cB' = \cB$). La notion de morphisme analytique au-dessus de~$\id$ co\"incide alors avec celle de morphisme analytique entre espaces $\cB$-analytiques. 
\end{rema}

Les propri\'et\'es des morphismes entre espaces $\cA$-analytiques se transposent \emph{mutatis mutandis} \`a notre nouveau cadre.

\begin{prop}\label{prop:propmorphismeaudessusdef}\index{Morphisme!de corps r\'esiduels}
\

\begin{enumerate}[i)]
\item La notion de morphisme analytique au-dessus de~$\tau$ est locale \`a la source et au but (au sens du lemme \ref{lem:localmodele}).

\item Les morphismes analytiques au-dessus de~$\tau$ induisent des plongements isom\'etriques entre les corps r\'esiduels (au sens de la proposition~\ref{prop:isometrie}).
\end{enumerate}
\qed

\end{prop}

La propri\'et\'e de composition prend une forme un peu diff\'erente.

\begin{lemm}
Soient $\tau \colon \cB\to\cB'$ et $\tau' \colon \cB'\to\cB''$ deux morphismes born\'es de~$\cA$-alg\`ebres de Banach. Soient $\varphi \colon X\to Y$ un morphisme analytique au-dessus de~$\tau'$ et $\varphi' \colon Y\to Z$ un morphisme analytique au-dessus de~$\tau$. Alors le morphisme~$\varphi'\circ\varphi \colon X\to Z$ est un morphisme analytique au-dessus de~$\tau'\circ \tau$. 
\qed
\end{lemm}

\begin{defi}\index{Categorie@Cat\'egorie!des espaces analytiques au-dessus de~$\cA$|textbf}%
\nomenclature[Kj]{$\Hom_{\An_{\cA}}(X,Y)$}{pour $X,Y$ espaces analytiques au-dessus de~$\cA$, ensemble des morphismes analytiques de $X$ dans $Y$}%
\nomenclature[Kh]{$\An_{\cA}$}{cat\'egorie des espaces analytiques au-dessus de~$\cA$}
La \emph{cat\'egorie des espaces analytiques au-dessus de~$\cA$} est d\'efinie comme suit. Ses objets sont les espaces analytiques au-dessus de~$\cA$. Pour toutes $\cA$-alg\`ebres de Banach~$\cB$ et~$\cB'$, tout espace $\cB'$-analytique~$X$ et tout espace $\cB$-analytique~$Y$, l'ensemble des morphismes de~$X$ dans~$Y$ est 
\[\Hom_{\An_{\cA}}(X,Y) := \bigcup_{\tau}  \Hom_{\An_{\cA},\tau}(X,Y),\]
o\`u $\tau$~parcourt l'ensemble des morphismes born\'es de $\cA$-alg\`ebres de Banach de~$\cB$ dans~$\cB'$. 

On note~$\An_{\cA}$ la cat\'egorie des espaces analytiques au-dessus de~$\cA$.
\end{defi}

\begin{exem}\label{ex:morphismex}\index{Morphisme analytique!associe a un point@associ\'e \`a un point}%
\nomenclature[Kl]{$\gamma_{x}$}{morphisme analytique $\cM(\cH(x)) \to X$ associ\'e \`a un point $x$ d'un espace $\cA$-analytique $X$}
Soit~$X$ un espace $\cA$-analytique. \`A tout point~$x$ de~$X$, on peut associer canoniquement un morphisme $\gamma_{x} \colon \cM(\cH(x)) \to X$ d'espaces analytiques au-dessus de~$\cA$  dont l'image est~$\{x\}$. 

Pour le construire, on peut raisonner localement et donc supposer que~$X$ est un ferm\'e analytique d'un ouvert de~$\E{n}{\cA}$. Le morphisme $\ev_{x} \colon \cA[T_{1},\dotsc,T_{n}] \to \cH(x)$ d'\'evaluation en~$x$ induit alors le morphisme
\[ \gamma_{x} \colon \cM(\cH(x)) \too X \lhook\joinrel\too \E{n}{\cA}\]
recherch\'e. C'est un morphisme au-dessus du morphisme $\cA \to \cH(x)$ obtenu en restreignant le morphisme~$\ev_{x}$. 
\end{exem}

\index{Morphisme analytique!au-dessus d'un morphisme d'anneaux de Banach|)}
\index{Espace analytique!au-dessus de~$\cA$|)}

\section{Immersions d'espaces analytiques}\label{sec:immersion}
\index{Immersion|(}

Revenons maintenant \`a la cat\'egorie des espaces~$\cA$-analytiques~$\cAAn$. 

\begin{defi}\index{Ferme analytique@Ferm\'e analytique|textbf}
Soit~$(X,\cO_{X})$ un espace $\cA$-analytique. Un \emph{ferm\'e analytique} de~$X$ est un quadruplet~$(Z,\cO_Z,j,\cI)$, o\`u~$\cI$ est un faisceau coh\'erent d'id\'eaux de~$\cO_X$, $Z$~le lieu des z\'eros de~$\cI$, $j \colon Z \to X$ l'inclusion canonique et $\cO_Z$ le faisceau $j^{-1} (\cO_X/\cI)$. 

L'espace~$(Z,\cO_Z)$ est naturellement muni d'une structure d'espace $\cA$-analytique, obtenue en restreignant les cartes des atlas de~$X$.
\end{defi}

\begin{rema}\label{rem:supportfaisceau}\index{Faisceau!annulateur d'un}\index{Faisceau!support d'un}%
\nomenclature[Eb]{$\textrm{Ann}(\cF)$}{annulateur de~$\cF$}
Soient~$X$ un espace $\cA$-analytique et~$\cF$ un faisceau coh\'erent sur~$X$. Alors l'annulateur~$\textrm{Ann}(\cF)$ de~$\cF$ est un faisceau d'id\'eaux coh\'erents dont le lieu des z\'eros co\"incide avec le support de~$\cF$ (\cf~\cite[Annex, \S 4, 5]{Gr-Re2}). En particulier, le support de~$\cF$ est naturellement muni d'une structure de ferm\'e analytique de~$X$.
\end{rema}

\begin{defi}\index{Immersion|textbf}\index{Immersion!ouverte|textbf}\index{Immersion!fermee@ferm\'ee|textbf}
Soit~$i  \colon X\to Y$ un morphisme d'espaces $\cA$-analytiques. On dit que~$i$ est
\begin{itemize}
\item une \emph{immersion ouverte} s'il induit un isomorphisme entre~$X$ et un ouvert de~$Y$ (qui h\'erite naturellement d'une structure d'espace analytique d'apr\`es la remarque \ref{remarque_base_espace_analytique}) ;
\item une \emph{immersion ferm\'ee} s'il induit un isomorphisme entre~$X$ et un ferm\'e analytique de~$Y$ ;
\item une \emph{immersion} s'il existe une immersion ouverte~$i_{1}$ et une immersion ferm\'ee~$i_{2}$ telles que $i = i_1\circ i_2$.
\end{itemize}
\end{defi}

\'Enon\c cons quelques propri\'et\'es des immersions.

\begin{lemm}
Soient $i \colon X \to Y$ une immersion ferm\'ee d'espaces $\cA$-analytiques. Soit~$\cI$ le noyau du morphisme $\cO_{Y} \to i_{\ast}\cO_{X}$. C'est un faisceau d'id\'eaux coh\'erent. Notons~$Z$ le ferm\'e analytique de~$Y$ qu'il d\'efinit et $j \colon Z \to Y$ le morphisme d'inclusion. Alors, il existe un unique isomorphisme d'espaces $\cA$-analytiques $\psi \colon X \to Z$ tel que $i = j \circ \psi$.
\qed
\end{lemm}

\begin{lemm}
Soit $\varphi \colon X \to Y$ un morphisme d'espaces $\cA$-analytiques. Soit $U$ un ouvert de~$Y$ et notons $i_{U} \colon U \to Y$ l'immersion ouverte associ\'ee. Supposons que $\varphi(X) \subset U$. Alors, le morphisme~$\varphi$ se factorise par~$i_{U}$.
\end{lemm}
\begin{proof}
Le r\'esultat d\'ecoule du fait que la notion de morphisme analytique est locale au but, \cf~proposition~\ref{prop:proprietemorphismes}, i).
\end{proof}

\begin{prop}
La compos\'ee d'un nombre fini d'immersions ouvertes (resp. immersions ferm\'ees, resp. immersions) est une immersion ouverte (resp. immersion ferm\'ee, resp. immersion).
\end{prop}
\begin{proof}
Il suffit de d\'emontrer le r\'esultat pour la compos\'ee de deux morphismes. Le r\'esultat pour les immersions ouvertes et ferm\'ees est imm\'ediat.

Soient $i \colon X \to Y$ et $j \colon Y \to Z$ des immersions. On suit la preuve pour le cas des sch\'emas, \cf~\cite[\href{https://stacks.math.columbia.edu/tag/02V0}{lemma~02V0}]{stacks-project}. \'Ecrivons $i = i_1\circ i_2$ et $j = j_1\circ j_2$ o\`u $i_{2} \colon X \to U$ et $j_{2} \colon Y \to V$ sont des immersions ferm\'ees et $i_{1} \colon U \to Y$ et $j_{1} \colon V \to Z$ sont des immersions ouvertes. Il existe un ouvert~$V'$ de~$V$ tel que $U = j_{2}^{-1}(V')$. On peut alors \'ecrire~$j \circ i$ comme la compos\'ee de l'immersion ouverte 
\[V' \too V \too Z\]
et de l'immersion ferm\'ee
\[X \too U = j_{2}^{-1}(V') \too V'.\]
\end{proof}

\begin{lemm}\label{lem:immersionfermeelocann}
Soient $\varphi \colon X \to Y$ un morphisme d'espaces $\cA$-analytiques et $i \colon Y' \to Y$ une immersion ferm\'ee d'espaces $\cA$-analytiques. Le morphisme~$\varphi$ se factorise par~$i$ en tant que morphisme d'espaces $\cA$-analytiques si, et seulement si, il se factorise par~$i$ en tant que morphisme d'espaces localement annel\'es. 
\end{lemm}
\begin{proof}
Si~$\varphi$ se factorise par~$i$ en tant que morphisme d'espaces $\cA$-analytiques, il se factorise par~$i$ en tant que morphisme d'espaces localement annel\'es, par d\'efinition.

D\'emontrons l'implication r\'eciproque. Supposons qu'il existe un morphisme d'espaces localement annel\'es $\varphi' \colon X \to Y'$ tel que $\varphi = i \circ \varphi'$. On veut montrer que~$\varphi'$ est un morphisme d'espaces analytiques. Par d\'efinition, $Y'$ est isomorphe \`a~$Z$ dans la cat\'egorie des espaces $\cA$-analytiques. On peut donc supposer que $Y'=Z$ et que $i \colon Z \to X$ est le morphisme d'inclusion.

La question \'etant locale, on peut supposer que~$X$ et~$Y$ sont des ferm\'es analytiques d'ouverts~$U$ et~$V$ d'espaces affines et que~$\varphi$ se rel\`eve en un morphisme $\tilde{\varphi} \colon U\to V$. L'espace~$Z$ est alors un ferm\'e analytique de~$V$ et le morphisme~$\tilde{\varphi}$ rel\`eve~$\varphi'$. Le r\'esultat s'ensuit.
\end{proof}

\begin{prop}\label{immersion_ferm\'e}
Soit $\varphi \colon X \to Y$ un morphisme d'espaces $\cA$-analytiques. 
Soit $i \colon Y' \to Y$ une immersion ferm\'ee d'espaces $\cA$-analytiques. Soit~$\cI$ un faisceau coh\'erent d'id\'eaux d\'efinissant le ferm\'e analytique $Z := i(Y')$ de~$Y$. Les propositions suivantes sont \'equivalentes~:
\begin{enumerate}[i)]
\item le morphisme~$\varphi$ se factorise par~$i$, en tant que morphisme d'espaces analytiques~;
\item l'application $\varphi^\ast \cI \to \varphi^\ast\cO_{Y} = \cO_{X}$ est nulle.
\end{enumerate}
En outre, lorsqu'elle sont satisfaites, le morphisme $\varphi' \colon X \to Y'$ tel que $\varphi = i\circ  \varphi'$ est unique.
\end{prop}
\begin{proof}
Dans la cat\'egorie des espaces localement annel\'es, le r\'esultat est d\'emontr\'e dans
\cite[\href{https://stacks.math.columbia.edu/tag/01HP}{lemma~01HP}]{stacks-project}. Notre r\'esultat s'en d\'eduit en utilisant le lemme~\ref{lem:immersionfermeelocann}.
\end{proof}

\begin{prop}\label{immersion_ferm\'e2}
Soit~$i \colon X \to Y$ une immersion d'espaces $\cA$-analytiques. Les propositions suivantes sont \'equivalentes~:
\begin{enumerate}[i)]
\item le morphisme $i$ est une immersion ferm\'ee ;
\item l'image de~$i$ est ferm\'ee dans~$Y$.
\end{enumerate}
De m\^eme, les propositions suivantes sont \'equivalentes~:
\begin{enumerate}[i')]
\item le morphisme $i$ est une immersion ouverte ;
\item l'image de~$i$ est ouverte dans~$Y$.
\end{enumerate}
\end{prop}
\begin{proof}
L'implication $i) \implies ii)$ est \'evidente. D\'emontrons $ii) \implies i)$. On suit la preuve de~\cite[\href{https://stacks.math.columbia.edu/tag/01IQ}{lemma~01IQ}]{stacks-project} (pour les sch\'emas). Supposons que $i(X)$ est ferm\'ee dans~$Y$. Puisque le morphisme~$i$ est une immersion, il induit un isomorphisme avec son image et il suffit donc de montrer que $i(X)$ est un ferm\'e analytique de~$Y$. Par d\'efinition, il existe un ouvert~$U$ de~$Y$ et un faisceau d'id\'eaux coh\'erent~$\cI$ sur~$U$ tel que $i(X)$ soit le ferm\'e analytique de~$U$ associ\'e \`a l'id\'eal~$\cI$.

Remarquons que l'on a 
\[\cI_{|U \setminus i(X)} = \cO_{|U \setminus i(X)} = \big(\cO_{Y \setminus i(X)}\big)_{|U \setminus i(X)}.\]
On peut donc d\'efinir un faisceau d'id\'eaux coh\'erent~$\cJ$ en recollant~$\cI$ et~$\cO_{Y \setminus i(X)}$. Le faisceau~$\cO/\cJ$ est alors support\'e sur~$U$ et sa restriction \`a~$U$ est \'egale \`a~$\cO_{U}/\cI$. Le r\'esultat s'en d\'eduit.

\medbreak

L'\'equivalence $i') \Longleftrightarrow ii')$ d\'ecoule directement des d\'efinitions.
\end{proof}

\begin{prop}\label{crit\`ere_local_immersion}
Soit $i \colon X \to Y$ un morphisme d'espaces $\cA$-analytiques. Les propositions suivantes sont \'equivalentes~:
\begin{enumerate}[i)]
\item le morphisme~$i$ est une immersion~;
\item pour tout~$y\in i(X)$, il existe un voisinage ouvert de~$V$ de~$y$ dans~$Y$ tel que le morphisme $i^{-1}(V) \to V$ induit par~$i$ soit une immersion ferm\'ee.
\end{enumerate}
\end{prop}
\begin{proof}
L'implication $i) \implies ii)$ d\'ecoule des d\'efinitions. D\'emontrons $ii) \implies i)$. Supposons que, pour tout $y \in i(X)$, il existe un voisinage~$V_y$ de~$y$ dans~$Y$ tel que~$i$ induise un isomorphisme entre~$i^{-1}(V_y)$ et un ferm\'e analytique~$Z_y$ de~$V_y$ d\'efini par un faisceau d'id\'eaux coh\'erent~$\cI^y$. 

Posons $U := \bigcup_{y\in i(X)} V_{y}$. C'est un ouvert de~$Y$ sur lequel les faisceaux~$\cI^{y}$ se recollent en un faisceau d'id\'eaux coh\'erent~$\cI$. On v\'erifie que le morphisme~$i$ induit un isomorphisme entre~$X$ et le ferm\'e analytique de~$Y$ d\'efini par~$\cI$.
\end{proof}

\index{Immersion|)}

\chapter[Topologie des anneaux de fonctions]{Quelques r\'esultats topologiques sur les anneaux de fonctions analytiques}
\label{analyse}

Le but principal de ce chapitre est de d\'emontrer un r\'esultat de fermeture des id\'eaux du faisceau structural. Il sera essentiel, par la suite, dans la d\'emonstration des propri\'et\'es de la cat\'egorie des espaces analytiques (\cf~proposition~\ref{morphsec}, dont d\'ecoule le th\'eor\`eme~\ref{thm:analytification}) ou pour obtenir des propri\'et\'es d'annulation cohomologiques sur des espaces ouverts (\cf~th\'eor\`eme~\ref{th:Bouvert}).

Dans la section~\ref{sec:normesquotients}, nous d\'emontrons des r\'esultats techniques permettant de comparer des normes (norme r\'esiduelle, norme quotient, etc.) sur des anneaux de Banach. Nous commen\c{c}ons par rappeler quelques d\'efinitions et r\'esultats issus de~\cite{A1Z}, et les adaptons ensuite afin de pouvoir les appliquer \`a d'autres anneaux. 

La section~\ref{sec:divWnormes} contient le c\oe ur technique de ce chapitre~: la d\'emonstration d'une version raffin\'ee du th\'eor\`eme de division de Weierstra\ss~\ref{weierstrass} dans laquelle on garde un contr\^ole sur les normes du reste et du quotient. 

Dans la section~\ref{sec:fermeture}, nous d\'emontrons finalement un r\'esultat de fermeture pour les id\'eaux du faisceau structural ou, plus g\'en\'eralement, les sous-modules d'une puissance finie du faisceau structural. Pour ce faire, nous devons imposer \`a l'anneau de base au-dessus duquel nous travaillons de satisfaire certaines conditions, que nous regroupons sous le vocable d'anneau de base g\'eom\'etrique. Nos exemples usuels (corps valu\'es, anneaux d'entiers de corps de nombres, corps hybrides, anneaux de valuation discr\`ete, anneaux de Dedekind trivialement valu\'es) satisfont ces conditions.

\section{Normes sur les quotients}\label{sec:normesquotients}

Soit $(\Ac,\nm)$ un anneau de Banach. Posons $B := \Mc(\Ac)$. Soit~$G(S)$ un polyn\^ome unitaire non constant \`a coefficients dans~$\cA$~:
\[ G = S^d + \sum_{i=0}^{d-1} g_i \, S^i  \in\cA[S]. \]
Dans cette section, nous \'etudierons plusieurs normes et semi-normes dont on peut munir le quotient $\Ac[S]/(G(S))$.

\subsection{G\'en\'eralit\'es}\label{sec:generalitesnormes}

\index{Norme!comparaison|(}

Cette section est reprise de \cite[\S 5.2]{A1Z}. 

Nous noterons encore~$\nm$ la norme sur~$\Ac^d$ d\'efinie par 
\[ \fonction{\nm}{\cA^d}{\R_{\ge 0}}{(a_0,\ldots,a_{d-1})}{\max(\|a_{0}\|,\dotsc,\|a_{d-1}\|)}.\]
Nous avons un isomorphisme 
\[ \fonction{n_{G}}{\cA^d}{\cA[S]/(G(S))}{(a_0,\ldots,a_{d-1})}{\displaystyle\sum^{d-1}_{i=0} a_i\, S^i}.\]

\begin{defi}\index{Norme!divisorielle|textbf}%
\nomenclature[Bja]{$\nm_{\div}$}{norme divisorielle sur $\Ac[S]/(G(S))$, h\'erit\'ee de~$\cA^d$ par division euclidienne}
On appelle \emph{norme divisorielle} sur $\Ac[S]/(G(S))$ la norme~$\nm_{\div}$ d\'efinie par
\[ \|F\|_{\div} := \|n_{G}^{-1}(F)\| \]
pour $F \in \cA[S]/(G(S))$.
\end{defi}

Introduisons maintenant les semi-normes r\'esiduelles. Pour $v\in \R_{>0}$, consid\'erons la norme~$\nm_{v}$ sur $\Ac[S]$ d\'efinie par 
\[ \|F\|_{v} := \max_{0\le i\le n} (\|a_{i}\|\, v^i) \]
pour $F = \displaystyle\sum_{i=0}^n a_{i}\, S^i \in \cA[S]$.%
\nomenclature[Bib]{$\nm_{v}$}{norme $\infty$ de rayon~$v$ sur $\Ac[S]$}

\begin{defi}\index{Norme!residuelle@r\'esiduelle|textbf}%
\nomenclature[Bjb]{$\nm_{v,\res}$}{norme r\'esiduelle $\Ac[S]/(G(S))$ induite par~$\nm_{v}$}
Pour~$v \in \R_{>0}$, on appelle \emph{semi-norme r\'esiduelle de rayon~$v$} la semi-norme $\nm_{v,\res}$ sur $\cA[S]/(G(S))$ d\'efinie par
\[ \|F\|_{v, \res} := \inf \big( \|F'\|_{v} \ \colon \ F' \in \cA[S], F' = F \mod G \big)\]
pour $F \in \cA[S]/(G(S))$.
\end{defi}

Soit $v_{1}$ un \'el\'ement de~$\R_{>0}$ satisfaisant l'in\'egalit\'e
\begin{equation*}
\sum_{i=0}^{d-1} \|g_i\|\, v_{1}^{i-d} \le \frac{1}{2}.
\end{equation*}

\begin{prop}[\protect{\cite[th\'eor\`eme~5.2.1, lemmes~5.2.2 et~5.2.3]{A1Z}}]\label{prop:equivalencedivres}\index{Norme!residuelle@r\'esiduelle}\index{Norme!divisorielle}
Soit $v \ge v_{1}$. On a les propri\'et\'es suivantes :
\begin{enumerate}[i)]
\item la semi-norme~$\|.\|_{v,\res}$ sur~$\cA[S]/(G(S))$ est une norme ;
\item l'anneau~$\cA[S]/(G(S))$ muni de $\nm_{v,\res}$ est un anneau de Banach ;
\item pour tout~$F\in\cA[S]/(G(S))$ on a les in\'egalit\'es :
\[v^{-d+1}\|F\|_{v,\res}\leq\|F\|_{\div}\leq 2\|F\|_{v,\res}.\]
\end{enumerate}

En particulier, les normes~$\nm_{\div}$ et les normes~$\nm_{v,\res}$ pour $v\ge v_{1}$ sont toutes \'equivalentes.
\qed
\end{prop}

Les constantes que nous donnons ne figurent pas dans l'\'enonc\'e de~\cite[th\'eor\`eme 5.2.1]{A1Z}, mais sa preuve les fournit. 

\begin{defi}\index{Norme!spectrale|textbf}%
\nomenclature[Bjd]{$\nm_{\sp}$}{semi-norme spectrale de $\nm_{v,\res}$ sur $\cA[S]/(G(S))$ pour $v$ assez grand}
On appelle \emph{semi-norme spectrale} sur $\cA[S]/(G(S))$ la semi-norme~$\nm_{\sp}$ obtenue comme semi-norme spectrale de $\nm_{v_{1}, \res}$. Elle est ind\'ependante du choix de~$v_{1}$. 
\end{defi}

Notons $Z_{G}$ le ferm\'e de Zariski de~$\E{1}{\Ac}$ d\'efini par l'\'equation $G=0$. Le morphisme d'anneaux de Banach $\Ac \to \Ac[S]/(G(S))$ induit un morphisme d'espaces localement annel\'es
\[ \varphi_{G} \colon Z_{G} \too B .\]%
\nomenclature[Jya]{$Z_{G}$}{ferm\'e de Zariski de~$\E{1}{\Ac}$ d\'efini par $G \in \cA[S]$}%
\nomenclature[Jyb]{$\varphi_{G}$}{morphisme $Z_{G} \simto \cM(\cA)$}

Soit $v_{2}$ un \'el\'ement de~$\R_{>0}$ satisfaisant l'in\'egalit\'e
\begin{equation*}\label{eq:v02}
v_{2} \ge \max_{0\le i\le d-1}(\|g_{i}\|^{1/(d-i)}).
\end{equation*}

\begin{lemm}[\protect{\cite[lemme~5.2.4]{A1Z}}]\label{lem:ZGspectreASG}\index{Norme!residuelle@r\'esiduelle}\index{Norme!uniforme}
Pour tout $v\ge v_{2}$, le disque~$\overline{D}_{B}(v)$ contient toutes les racines de~$G$ dans~$\AunA$ et le spectre de l'anneau~$\cA[S]/(G(S))$ muni de $\nm_{v,\res}$ s'identifie \`a $Z_{G}$. 

En particulier, d'apr\`es le lemme~\ref{lem:spuni}, la semi-norme spectrale~$\nm_{\sp}$ s'identifie \`a la semi-norme uniforme sur~$Z_{G}$.
\qed
\end{lemm}

\subsection{Condition $(\cB N_{G})$}\label{sec:BNG}
\index{Condition!BNG@$(\cB N_{G})$|(}

Nous introduisons ici une condition permettant de comparer normes r\'esiduelles et normes spectrales dans le cadre de la section~\ref{sec:generalitesnormes}. 
Il s'agit d'un outil technique qui sera pour nous crucial dans la suite. 


Cette section est principalement reprise de \cite[\S 5.3]{A1Z} et \cite[\S 6]{EtudeLocale}. Signalons un changement de notation~: la condition~$(N_{G})$ de~\cite{EtudeLocale} sera not\'ee ici~$(\cB N_{G})$, en vue d'une g\'en\'eralisation \`a la section~\ref{sec:ONG}.\index{Condition!NG@$(N_{G})$}

\begin{defi}[\protect{\cite[d\'efinition~4.1]{EtudeLocale}}]\index{Condition!BNG@$(\cB N_{G})$}\index{Norme!residuelle@r\'esiduelle}\index{Norme!spectrale}
On dit qu'une partie compacte~$U$ de~$B$ \emph{satisfait la condition~$(\cB N_G)$} si elle est spectralement convexe et s'il existe $v_{U,0} \in \R_{>0}$ tel que, pour tout $v \ge v_{U,0}$, la semi-norme~$\nm_{U,v,\res}$ sur~$\cB(U)[S]/(G(S))$ soit \'equivalente \`a la semi-norme spectrale.
\end{defi}

On peut supposer que~$v_{U,0}$ satisfait les in\'egalit\'es 
\[ \sum_{i=0}^{d-1} \|g_i\|_{U}\, v_{U,0}^{i-d} \le \frac{1}{2} \textrm{ et } v_{U,0} \ge \max_{0\le i\le d-1}\big(\|g_{i}\|_{U}^{1/(d-i)}\big).\]
Nous nous placerons toujours sous cette hypoth\`ese dans la suite.

\begin{prop}[\protect{\cite[proposition~5.3.3]{A1Z}}]\index{Norme!residuelle@r\'esiduelle}\index{Norme!divisorielle}\index{Norme!uniforme}\index{Condition!BNG@$(\cB N_{G})$}
Soit~$U$ une partie compacte de~$B$ qui satisfait la condition~$(\Bc N_{G})$. Alors, pour tout $v\ge v_{U,0}$, 
les normes $\nm_{U,\div}$ et $\nm_{U,v,\res}$ sur~$\Bc(U)[T]/(G(T))$ sont \'equivalentes \`a la norme uniforme sur~$\varphi_{G}^{-1}(U)$. De plus, le morphisme de $\cA$-alg\`ebres de Banach naturel
\[ \Bc(U)[T]/(G(T)) \too \Bc(\varphi_{G}^{-1}(U)), \]
o\`u $\varphi_{G}^{-1}(U)$ est vu comme partie de~$\E{1}{\Ac}$, est un isomorphisme bi-born\'e.
\qed
\end{prop}

Nous rappelons maintenant quelques conditions suffisantes pour que la condition~$(\Bc N_{G})$ soit satisfaite. En pratique, nous utiliserons la premi\`ere dans la partie ultram\'etrique de l'espace et la seconde dans la partie de caract\'eristique nulle.


\begin{prop}[\protect{\cite[corollaire~6.3 et proposition~6.14]{EtudeLocale}}]\label{prop:CNBNG}\index{Condition!BNG@$(\cB N_{G})$}
\index{Bord analytique}\index{Resultant@R\'esultant}%
\nomenclature[Jz]{$\Res$}{r\'esultant}
Soit~$U$ une partie compacte et spectralement convexe de~$B$. Supposons que~$U$ poss\`ede un bord analytique~$\Gamma_{U}$ satisfaisant l'une ou l'autre des conditions suivantes~:
\begin{enumerate}
\item $\Gamma_{U}$ est fini et, pour tout $\gamma \in \Gamma_{U}$, $G(\gamma)$ est sans facteurs multiples;
\item la fonction $|\Res(G,G')|$, o\`u $\Res(G,G')$ d\'esigne le r\'esultant des polyn\^omes~$G$ et~$G'$, est born\'ee inf\'erieurement sur~$\Gamma_{U}$ par une constante strictement positive.
\end{enumerate}
Alors, $U$ satisfait la condition~$(\Bc N_{G})$. 
\qed
\end{prop}

Dans le cas particulier o\`u le polyn\^ome~$G$ est de degr\'e~1, la condition~$(\Bc N_{G})$ est toujours v\'erifi\'ee (ce que l'on peut d\'emontrer soit directement, soit \`a l'aide de la proposition \ref{prop:CNBNG}). \`A l'aide des r\'esultats de ce paragraphe, nous obtenons l'\'enonc\'e suivant.  Notons qu'il serait ici plus simple de le d\'emontrer directement.

\begin{prop}\label{prop:isodegre1}\index{Norme!divisorielle}\index{Norme!uniforme}
Supposons que le polyn\^ome~$G$ est de degr\'e~1. Alors, pour toute partie compacte spectralement convexe~$U$ de~$B$, la norme~$\nm_{U,\div}$ sur~$\Bc(U)[T]/(G(T))$ est \'equivalente \`a la norme uniforme sur~$\varphi_{G}^{-1}(U)$ et le morphisme de $\cA$-alg\`ebres de Banach naturel 
\[\Bc(U) \simtoo \Bc(U)[T]/(G(T)) \too \Bc(\varphi_{G}^{-1}(U)) \]
est un isomorphisme bi-borné.

En particulier, le morphisme naturel $Z_{G} \to B$ est un isomorphisme d'espaces annel\'es. 
\qed
\end{prop}

Les r\'esultats pr\'ec\'edents peuvent s'appliquer dans le cadre d'endomorphismes polynomiaux de la droite. \index{Endomorphisme de la droite}\index{Condition!BNPS-T@$(\cB N_{P(S)-T})$|(}

Soit~$W$ une partie compacte et spectralement convexe de~$B$. Soit~$P$ un polyn\^ome unitaire non constant \`a coefficients dans~$\cB(W)$. Le morphisme naturel 
\[\varphi^\sharp_{P} \colon \cB(W)[T] \too \cB(W)[T,S]/(P(S) - T) \simtoo \cB(W)[S]\]%
\nomenclature[Jy]{$\varphi_{P}$}{endomorphisme de $\E{1}{\cA}$ induit par $P \in \cA[T]$}
induit un endomorphisme~$\varphi_{P}$ de~$\E{1}{\cB(W)}$. Pour toute partie compacte spectralement convexe~$U$ de~$\E{1}{\cB(W)}$ (avec coordonn\'ee~$T$), notons~$Z_{U,P(S)-T}$ le ferm\'e de Zariski de $\E{1}{\cB(U)}$ (avec coordonn\'ee~$S$) d\'efini par l'\'equation $P(S)-T = 0$. D'apr\`es la proposition~\ref{prop:isodegre1}, il s'identifie \`a~$\varphi_{P}^{-1}(U)$. 

En particulier, pour tous $r,s \in \R_{\ge 0}$ tels que $0 = r< s$ ou $0< r\le s$, si $U = \overline{C}_{W}(r,s)$, $Z_{U,P(S)-T}$ s'identifie \`a $\overline{C}_{W}(P;r,s)$.

\begin{coro}[\protect{\cite[proposition~7.1]{EtudeLocale}}]\label{coro:BNPST}\index{Couronne!algebre@alg\`ebre d'une}\index{Disque!algebre@alg\`ebre d'un}\index{Domaine polynomial!algebre@alg\`ebre d'un}\index{Norme!residuelle@r\'esiduelle}\index{Norme!uniforme}\index{Norme!divisorielle}
Soient $W,P,r,s$ comme ci-dessus.
Supposons que $\overline{C}_{W}(r,s)$ satisfait la condition $(\cB N_{P(S)-T})$. Alors, pour tout $w\ge v_{0}$, les normes $\nm_{\overline{C}_{W}(r,s),\div}$ et $\nm_{\overline{C}_{W}(r,s),w,\res}$ sur~$\Bc(\overline{C}_{W}(r,s))[S]/(P(S)-T)$ sont \'equivalentes 
\`a la norme uniforme sur $\overline{C}_{W}(P;r,s)$. De plus, le morphisme de $\cA$-alg\`ebres de Banach
\[ \cB(\overline{C}_{W}(r,s))[S]/(P(S) - T) \too \cB(\overline{C}_{W}(P;r,s))\]
est un isomorphisme bi-born\'e.
\qed
\end{coro}

On peut d\'eduire de la proposition~\ref{prop:CNBNG} des conditions permettant d'assurer que certaines couronnes $\overline{C}_{W}(r,s)$ satisfont la condition $(\cB N_{P(S)-T})$.

\begin{prop}[\protect{\cite[proposition~6.5 et lemme~6.16]{EtudeLocale}}]\label{prop:CNBNPST}
Soient $W,P,r,s$ comme ci-dessus.
\begin{enumerate}[i)]
\item Supposons que $W$ est contenue dans la partie ultram\'etrique de~$B$ et poss\`ede un bord analytique fini. Alors $\overline{C}_{W}(r,s)$ satisfait la condition $(\cB N_{P(S)-T})$.
\item Supposons que $P(b) \in \cH(b)[T]$ est s\'eparable. Soient $r,r',s,s' \in \R_{\ge 0}$ tels que $r \prec r'$ et $0 < s' < s$. Alors, il existe $r_{1},r_{2},s_{1},s_{2} \in \R_{\ge 0}$ v\'erifiant $r \prec  r_{1}\prec r_{2} \prec r'$ et $s' < s_{1} < s_{2}< s$, un voisinage compact~$V_{0}$ de~$b$ dans~$B$ et $\eps\in \R_{>0}$ tels que, pour toute partie compacte spectralement convexe~$V$ de~$V_{0}$ contenant~$b$, tout $u \in[r_{1},r_{2}]$ et tout $v \in[s_{1},s_{2}]$, tout $Q \in \cB(V)[T]$ tel que $\|Q-P\|_{V,\infty} \le \eps$, la couronne $\overline{C}_{V}(u,v)$ satisfasse la condition $(\cB N_{Q(S)-T})$.
\end{enumerate}
\qed
\end{prop}

L'\'enonc\'e de \cite[lemme~6.16]{EtudeLocale} est moins pr\'ecis que~ii), mais sa preuve fournit cette version raffin\'ee.

\begin{coro}\label{cor:BVfineNPST}\index{Voisinage}
Soient $b\in B$ et~$x$ un point de~$\E{1}{\cA}$ (avec coordonn\'ee~$S$) au-dessus de~$b$. Soit~$\cV_{b}$ une base fine de voisinages compacts spectralement convexes de~$b$ dans~$B$. Si~$\cH(b)$ est de caract\'eristique non nulle et trivialement valu\'e, supposons que tout \'el\'ement de~$\cV_{b}$ poss\`ede un bord analytique fini. Alors $x$ poss\`ede une base fine~$\cV_{x}$ de voisinages compacts spectralement convexes de la forme 
\[\overline{C}_{V}(P,r,s),\]
avec $V \in \cV_{b}$, $P \in \cB(V)[S]$ unitaire non constant et $r,s \in \R_{\ge 0}$. En outre, on peut supposer que, pour tous $V,P,r,s$ tels que $\overline{C}_{V}(P,r,s)$ appartienne \`a~$\cV_{x}$,
\begin{enumerate}[i)]
\item $\overline{C}_{b}(P,r,s)$ est connexe~;
\item il existe un voisinage~$V_{0}$ de~$V$ tel que $P \in \cB(V_{0})[S]$ et $r_{0},s_{0} \in \R_{\ge 0}$ avec $r_{0}\prec r$ et $s_{0}>s$ tels que, pour tout $V' \in \cV_{b}$ contenu dans~$V_{0}$
et tous $r',s' \in \R_{\ge 0}$ v\'erifiant $r_{0}\prec r' \prec r$ et $s_{0}> s'>s$, $\overline{C}_{V'}(P',r',s')$ appartienne \`a~$\cV_{x}$~;
\item la couronne $\overline{C}_{V}(r,s)$ (avec coordonn\'ee~$T$) satisfait la propri\'et\'e~$(\cB N_{P(S)-T})$.
\end{enumerate}
\end{coro}
\begin{proof}
Commen\c cons par une remarque topologique qui nous permettra d'assurer que les bases de voisinages que l'on construit sont fines. Soient~$V$ une partie compacte de~$B$, $P \in \cO(V)[S]$ et $r,s \in \R_{\ge 0}$. Soit~$U$ un voisinage de $\overline{C}_{V}(P,r,s)$. Alors, pour tout voisinage~$V'$ de~$V$ assez petit et tous $r',s' \in \R_{\ge 0}$ avec $r'\prec r$ assez proche de~$r$ et $s'>s$ assez proche de~$s$, on a $\overline{C}_{V'}(P,r',s') \subset U$.

L'\'enonc\'e d\'ecoule alors des propositions~\ref{prop:basevoisdim1rigide} et \ref{prop:CNBNPST}, si $x$~est rigide dans la fibre au-dessus de~$b$, et des propositions~\ref{prop:basevoisdim1} et \ref{prop:CNBNPST} sinon.
\end{proof}

Pr\'ecisons le r\'esultat dans le cas des points rigides \'epais. 

\begin{coro}\label{cor:BVfineNPSTrigideepais}\index{Voisinage!d'un point rigide \'epais}\index{Point!rigide epais@rigide \'epais!voisinage|see{Voisinage}}
Soient $b\in B$ et~$x$ un point rigide \'epais de~$\E{1}{\cA}$ (avec coordonn\'ee~$S$) au-dessus de~$b$. Notons $d$ le degr\'e de $\mu_{\kappa,x} \in \kappa(b)[S]$. Soit~$P_{0}$ un relev\'e unitaire de degr\'e~$d$ de~$\mu_{\kappa,x}$ dans $\cO_{B,b}[S]$. Soit~$\cV_{b}$ une base fine de voisinages compacts spectralement convexes de~$b$ dans~$B$ sur lesquels~$P_{0}$ est $\cB$-d\'efinie. 
Alors $x$ poss\`ede une base fine~$\cV_{x}$ de voisinages compacts spectralement convexes de la forme 
\[\overline{D}_{V}(P_{0},s),\]
avec $V \in \cV_{b}$ et $s \in \R_{> 0}$. En outre, on peut supposer que pour tous $V,s$ tels que $\overline{D}_{V}(P_{0},s)$ appartienne \`a~$\cV_{x}$,
\begin{enumerate}[i)]
\item $\overline{D}_{b}(P_{0},s)$ est connexe~;
\item il existe un voisinage~$V_{0}$ de~$V$ et $s_{0} \in \R_{\ge 0}$ avec $s_{0}>s$ tels que, pour tout $V' \in \cV_{b}$ tel que $V \subset V' \subset V_{0}$ et tout $s' \in \intoo{s,s_{0}}$, $\overline{D}_{V'}(P',s')$ appartienne \`a~$\cV_{x}$.
\end{enumerate}
Si $\mu_{\kappa,x}$ est s\'eparable ou si tout \'el\'ement de~$\cV_{b}$ poss\`ede un bord analytique fini, on peut \'egalement supposer que
\begin{enumerate}
\item[iii)] le disque  $\overline{D}_{V}(s)$ (avec coordonn\'ee~$T$) satisfait la propri\'et\'e~$(\cB N_{P_{0}(S)-T})$.
\end{enumerate}
Si $\cH(b)$ est ultram\'etrique et non trivialement valu\'e, on peut \'egalement supposer que
\begin{enumerate}
\item[iii')] il existe $Q \in \cB(V)[S]$ unitaire de degr\'e~$d$ tel que $Q(b) \in \cH(b)[T]$ soit s\'eparable et un intervalle ouvert~$I$ de~$\R_{>0}$ contenant~$s$ tels que, pour tout voisinage compact spectralement convexe~$W$ de~$b$ dans~$V$ et tout $s' \in I$, on ait $\overline{D}_{W}(Q,s') = \overline{D}_{W}(P_{0},s')$ et le disque  $\overline{D}_{W}(s)$ (avec coordonn\'ee~$T$) satisfasse la propri\'et\'e~$(\cB N_{Q(S)-T})$. 
\end{enumerate}
\end{coro}
\begin{proof}
La premi\`ere partie de l'\'enonc\'e d\'ecoule du lemme~\ref{lem:polmin} et de la proposition~\ref{prop:basevoisdim1rigide}, ainsi que les points~i) et~ii). Les points~iii) et~iii') suivent alors de la proposition~\ref{prop:CNBNPST}.
\end{proof}
\index{Condition!BNG@$(\cB N_{G})$|)}
\index{Condition!BNPS-T@$(\cB N_{P(S)-T})$|)}

\subsection{Condition $(\Oc N_{G})$}\label{sec:ONG}
\index{Condition!ONG@$(\cO N_{G})$|(}

Dans le chapitre~\ref{chap:Stein} de ce texte, consacr\'e aux espaces de Stein, nous aurons besoin d'utiliser des r\'esultats similaires \`a ceux de la section~\ref{sec:BNG} en rempla\c cant l'alg\`ebre~$\Bc(U)$ par $\overline{\Oc(U)}$. Cela nous conduit \`a modifier certaines d\'efinitions. Il sera \'egalement important de ne plus imposer \`a~$U$ d'\^etre spectralement convexe. Nous convervons encore le cadre de la section~\ref{sec:generalitesnormes}.

Soit~$U$ une partie compacte de~$B$. Le morphisme canonique $\varphi_{G} \colon Z_{G} \to B$ induit un morphisme
\[\psi_{G,U} \colon \Oc(U)[T]/(G(T)) \too \Oc_{Z_{G}}(\varphi_{G}^{-1}(U)).\]%
\nomenclature[Jyc]{$\psi_{G,U}$}{morphisme naturel $\Oc(U)[T]/(G(T)) \to \Oc_{Z_{G}}(\varphi_{G}^{-1}(U))$}

\begin{defi}\label{def:NG}\index{Condition!ONG@$(\cO N_{G})$|textbf}\index{Norme!residuelle@r\'esiduelle}\index{Norme!uniforme}
On dit que le compact~$U$ satisfait la condition~$(\Oc N_{G})$ s'il existe $v_{U,0} \in \R_{>0}$ tel que, pour tout $v\ge v_{U,0}$, la norme r\'esiduelle $\nm_{U,v,\res}$ sur $\Oc_{X}(U)[T]/(G(T))$ soit \'equivalente \`a $\|\psi_{G,U}(\wc)\|_{\varphi_{G}^{-1}(U)}$.
\end{defi}

On peut supposer que~$v_{U,0}$ satisfait les in\'egalit\'es 
\[ \sum_{i=0}^{d-1} \|g_i\|_{U}\, v_{U,0}^{i-d} \le \frac{1}{2} \textrm{ et } v_{U,0} \ge \max_{0\le i\le d-1}(\|g_{i}\|_{U}^{1/(d-i)}).\]
\emph{Dans la suite de cette section, nous nous placerons toujours sous cette hypoth\`ese.}

\begin{lemm}\label{lem:eqnormespectrale}\index{Norme!residuelle@r\'esiduelle}\index{Norme!spectrale}
Supposons que~$U$ satisfait la condition~$(\Oc N_{G})$. Alors, pour tout $v\ge v_{U,0}$, la norme r\'esiduelle $\nm_{U,v,\res}$ sur $\overline{\Oc(U)}[T]/(G(T))$ est \'equivalente \`a sa norme spectrale.

En outre, si le morphisme canonique $U \to \Mc(\overline{\Oc(U)})$ est bijectif, cette condition est \'equivalente \`a la condition~$(\Oc N_{G})$.
\end{lemm}
\begin{proof}
Consid\'erons l'espace $\hat{U} = \Mc(\overline{\Oc(U)})$ et le ferm\'e $\hat Z_{G}$ d\'efini par $G=0$ dans $\E{1}{\hat U}$. Notons $\hat{\varphi}_{G} \colon \hat Z_{G} \to \hat V$ et 
\[\hat \psi_{G,U} \colon \overline{\Oc(U)}[T]/(G(T)) \too \Oc_{\hat Z_{G}}(\hat Z_{G})\] 
les morphismes canoniques. On a alors un diagramme cart\'esien
\[\begin{tikzcd}
\varphi_{G}^{-1}(U) \arrow[r, hook] \arrow[d] & \hat Z_{G} \arrow[d]\\
U \arrow[r, hook, ] & \hat{U}
\end{tikzcd}.\]

Soit~$v\ge v_{U,0}$. Pour tout $f\in \Oc_{X}(U)[T]/(G(T))$, on a 
\[\|\psi_{G,U}(f)\|_{\varphi_{G}^{-1}(U)} \le \|\hat\psi_{G,U}(f)\|_{\hat Z_{G}} \le  \|f\|_{U, v, \res},\]
la derni\`ere in\'egalit\'e provenant du fait que l'on a un morphisme born\'e 
\[\overline{\Oc_{X}(U)}[T] \too \Oc_{\hat Z_{G}}(\hat Z_{G})\]
puisque $\overline{D}_{\overline{\Oc(U)}}(v)$ contient toutes les racines de~$G(T)$ dans~$\E{1}{\overline{\Oc(U)}}$, d'apr\`es le lemme~\ref{lem:ZGspectreASG}.

Or, d'apr\`es le lemme~\ref{lem:ZGspectreASG} encore, $\hat{Z}_{G}$ s'identifie naturellement au spectre de l'alg\`ebre $\overline{\Oc_{X}(U)}[T]/(G(T))$, donc $\|\hat\psi_{G,U}(\wc)\|_{\hat Z_{G}}$ n'est autre que la norme spectrale de $\nm_{U, v, \res}$. Le r\'esultat d\'ecoule maintenant directement de la condition~$(\Oc N_{G})$.

Sous l'hypoth\`ese de l'assertion finale, la fl\`eche horizontale inf\'erieure est bijective. On en d\'eduit que la fl\`eche sup\'erieure l'est aussi, ce qui permet d'identifier $\|\psi_{G,U}(\wc)\|_{\varphi_{G}^{-1}(U)}$ \`a la norme spectrale de $\nm_{U, v, \res}$ et donc de conclure.
\end{proof}

\begin{rema}
Pour toute partie compacte~$U$ de~$B$, l'anneau $\Bc(U)$ s'envoie isom\'etriquement dans~$\overline{\Oc(U)}$. Il suit alors du lemme~\ref{lem:eqnormespectrale} que, si $U$ est spectralement convexe et satisfait la condition~$(\Oc N_{G})$, il satisfait aussi la condition~$(\Bc N_{G})$.
\end{rema}

On reprend maintenant les \'enonc\'es de \cite[\S 6.1]{EtudeLocale} en indiquant les modifications \`a effectuer.
\index{Bord analytique!fort|(}

\begin{defi}[\protect{\cf~\cite[d\'efinition~6.1]{EtudeLocale}}]\index{Bord analytique!fort|textbf}
Soit~$U$ une partie compacte de~$B$. On dit qu'une partie~$\Gamma$ de~$U$ est un bord analytique fort de~$U$ si, pour tout $f\in\Oc(U)$, on a
\[\|f\|_{U} = \|f\|_{\Gamma}.\]
\end{defi}

\begin{rema}\label{rem:bordanvsanfort}
La condition de bord analytique fort porte sur tous les \'el\'ements de~$\cO(U)$, alors que celle de bord analytique (\cf~d\'efinition~\ref{def:bordan}) fait intervenir ceux de~$\cB(U)$, qui proviennent de sections globales sur~$B$.
\end{rema}

\begin{prop}[\protect{\cf~\cite[proposition~6.2]{EtudeLocale}}]\label{prop:NGunion}
Supposons qu'il existe des parties $V_{1},\dotsc,V_{r}$ de~$U$ telles que
\begin{enumerate}[i)]
\item pour tout $i\in\cn{1}{r}$, $V_{i}$ soit compacte et satisfasse la condition~$(\Oc N_{G})$~;
\item la r\'eunion $\Gamma_{U} := \bigcup_{1\le i\le r} V_{i}$ soit un bord analytique fort de~$U$.
\end{enumerate}
Alors, $U$ satisfait \'egalement la condition~$(\Oc N_{G})$. 

En outre, si $U$~est d\'ecente,
alors la partie $\Gamma_{U,G} := \varphi_{G}^{-1}(\Gamma_{U})$ est un bord analytique fort de $\varphi_{G}^{-1}(U)$.
\end{prop}
\begin{proof}
La d\'emonstration de la premi\`ere assertion de~\cite[proposition~6.2]{EtudeLocale} utilise pr\'ecis\'ement la condition~$(N_{G})$ (c'est-\`a-dire $(\Bc N_{G})$ avec nos notations) sous la forme que nous avons \'enonc\'ee dans la d\'efinition~\ref{def:NG}. Il n'y a donc que des modifications \'evidentes \`a effectuer.

En reprenant la d\'emonstration de la seconde assertion de~\cite[proposition~6.2]{EtudeLocale}, on montre que, pour tout $f\in \Oc(U)[T]/(G(T))$, on a
\[\|\psi_{G,U}(f)\|_{\varphi_{G}^{-1}(U)} = \max_{x\in \Gamma_{U,G}} (|\psi_{G,U}(f)(x)|).\]
La d\'ecence de~$U$ et le th\'eor\`eme de division de Weierstra\ss{} \ref{weierstrass} (sous la forme du th\'eor\`eme~\ref{thm:isolemniscate} adapt\'e au polyn\^ome~$G$)
assurent qu'on a un isomorphisme
\[\Oc(U)[T]/(G(T)) \simtoo \Oc_{Z_{G}}(\varphi_{G}^{-1}(U)).\] 
Le r\'esultat s'en d\'eduit.
\end{proof}

Le corollaire~6.3 de~\cite{EtudeLocale} s'adapte imm\'ediatement et nous l'omettrons. Le lemme~6.4 a d\'ej\`a \'et\'e repris par le lemme~\ref{lem:couronneum}. Les r\'esultats techniques n\'ecessaires sont maintenant tous \`a notre disposition. Nous \'enon\c{c}ons leurs cons\'equences sans plus de commentaires.

\begin{prop}[\protect{\cf~\cite[propositions~6.5 et~6.6]{EtudeLocale}}]\label{prop:bordanfortfini}\index{Domaine polynomial}\index{Condition!ONPS-T@$(\cO N_{P(S)-T})$}
Soit~$U$ une partie compacte et d\'ecente de~$B_{\um}$ qui poss\`ede un bord analytique fort fini. Soit $P(S) \in \Oc(U)[S]$ unitaire non constant et soient $r,s\in\R$ v\'erifiant $0=r<s$ ou $0<r\le s$. Alors, le domaine polynomial $\overline{C}_{U}(P;r,s)$ poss\`ede un bord analytique fort fini.

En outre, pour tout $Q(T) \in \Oc(U)[T]$ unitaire non constant, le domaine polynomial $\overline{C}_{U}(Q;r,s)$ satisfait la condition $(\Oc N_{P(S)-T})$.
\qed
\end{prop}
\index{Bord analytique!fort|)}

\begin{defi}[\protect{\cf~\cite[D\'efinition~6.9]{EtudeLocale}}]\index{Point!ultrametrique tres typique@ultram\'etrique tr\`es typique|textbf}
Un \emph{point}~$x$ de~$B$ est dit \emph{ultram\'etrique tr\`es typique} s'il appartient \`a l'int\'erieur de~$B_{\um}$ et poss\`ede un syst\`eme fondamental de voisinages compacts, spectralement convexes admettant un bord analytique fort fini.
\end{defi}

\begin{defi}\label{defi:honnete}\index{Point!tres decent@tr\`es d\'ecent|textbf}\index{Partie!tres decente@tr\`es d\'ecente|textbf}
On dit qu'un point~$x$ d'un espace analytique est \emph{tr\`es d\'ecent} si l'une (au moins) des trois conditions suivantes est satisfaite~:
\begin{enumerate}[i)]
\item $\cH(x)$ est de caract\'eristique nulle~;
\item $\cH(x)$ est de valuation non triviale~;
\item $x$ est ultram\'etrique tr\`es typique.
\end{enumerate}

On dit qu'une partie d'un espace analytique est tr\`es d\'ecente lorsque tous ses points le sont.
\end{defi}

\begin{exem}\label{ex:tresdecent}\index{Anneau!de base}
\index{Corps!valu\'e}\index{Corps!hybride}\index{Anneau!des entiers relatifs $\Z$}\index{Anneau!des entiers d'un corps de nombres}\index{Anneau!de valuation discr\`ete}\index{Anneau!de Dedekind trivialement valu\'e}
Lorsque $\cA$ est l'un de nos exemples usuels \ref{ex:corpsvalue} \`a~\ref{ex:Dedekind}~: les corps valu\'es complets, l'anneau~$\Z$ et les anneaux d'entiers de corps de nombres, les corps hybrides, les anneaux de valuation discr\`ete, les anneaux de Dedekind trivialement valu\'es, le spectre analytique~$\cM(\cA)$ est tr\`es d\'ecent.
\end{exem}

\begin{coro}[\protect{\cf~\cite[proposition~6.10]{EtudeLocale}}]
Soit~$b\in B$ un point ultram\'etrique tr\`es typique (resp. tr\`es d\'ecent). Soit $n\in\N$. Alors, tout point de~$\E{n}{\cA}$ situ\'e au-dessus de~$b$ est encore ultram\'etrique tr\`es typique (resp. tr\`es d\'ecent).
\qed
\end{coro}

Le corollaire qui suit est nouveau.

\begin{coro}\label{cor:BVcompactN}\index{Condition!ONPS-T@$(\cO N_{P(S)-T})$}
Soit~$U$ une partie compacte de~$B_{\um}$ dont tout point est ultram\'etrique tr\`es typique. Alors, le compact~$U$ poss\`ede un syst\`eme fondamental de voisinages compacts~$\cV$ tel que, pour tout $V\in \cV$, pour tout $P(S) \in \cO(V)[S]$ unitaire non constant et pour tous $r,s\in\R$ v\'erifiant $0=r<s$ ou $0<r\le s$, la couronne $\overline{C}_{V}(r,s)$ satisfasse la condition~$(\cO N_{P(S)-T})$.
\end{coro}
\begin{proof}
Soit~$W$ un voisinage de~$U$. Puisque tout point de~$U$ poss\`ede un voisinage compact dans~$W$ admettant un bord analytique fort fini, par compacit\'e de~$U$, il existe un voisinage compact~$V$ de~$U$ dans~$W$ qui est union finie de compacts admettant un bord analytique fort fini. On conclut alors \`a l'aide des propositions~\ref{prop:bordanfortfini} et~\ref{prop:NGunion}.
\end{proof}

On peut \'egalement reprendre les \'enonc\'es de \cite[\S 6.2]{EtudeLocale}.

\begin{defi}[\protect{\cf~\cite[d\'efinition~6.12]{EtudeLocale}}]\index{Condition!ORG@$(\cO R_{G})$|textbf}\index{Bord analytique!fort}\index{Resultant@R\'esultant}
On dit que le compact~$U$ satisfait la condition $(\cO R_{G})$ s'il poss\`ede un bord analytique fort sur lequel la fonction $|\Res(G,G')|$ est born\'ee inf\'erieurement par un nombre r\'eel strictement positif.
\end{defi}

\begin{prop}[\protect{\cf~\cite[proposition~6.14]{EtudeLocale}}]\label{prop:RG}
Toute partie compacte de~$B$ qui satisfait la condition~$(\cO R_{G})$ satisfait aussi la condition~$(\Oc N_{G})$. 
\end{prop}
\index{Condition!ONG@$(\cO N_{G})$|)}
\index{Norme!comparaison|)}

\section{Division de Weierstra\ss{} avec contr\^ole sur les normes}\label{sec:divWnormes}

Dans cette section, nous d\'emontrons une version raffin\'ee du th\'eor\`eme de division de Weierstra\ss{} (\cf~th\'eor\`eme~\ref{weierstrass}) et du th\'eor\`eme de Weierstra\ss{} g\'en\'eralis\'e (\cf~corollaire~\ref{cor:weierstrassgeneralise}) comprenant un contr\^ole des normes du reste et du quotient.

Soit $(\Ac,\nm)$ un anneau de Banach. Posons $B := \Mc(\Ac)$ et $X := \E{1}{\cA}$ avec coordonn\'ee~$S$. Notons $\pi \colon X \to B$ le morphisme de projection. Pour tout point~$b$ de~$B$, posons $X_{b} := \pi^{-1}(b) \simeq \E{1}{\cH(b)}$. 

Rappelons qu'\`a tout point rigide~$x$ de~$\Aunk$, on associe un polyn\^ome minimal~$\mu_{x}$ et un degr\'e~$\deg(x)$ (\cf~d\'efinition~\ref{def:rigidedroite}) et que l'anneau local $\cO_{\Aunk,x}$ en ce point est un anneau de valuation discr\`ete d'uniformisante~$\mu_{x}$.

\begin{theo}\label{weierstrassam}\index{Theoreme@Th\'eor\`eme!de division de Weierstra\ss!avec contr\^ole des normes}
Soient $b\in B$ et~$x$ un point rigide de~$X_{b}$. Soit~$\cV_b$ une base de voisinages compacts spectralement convexes de~$b$ dans~$B$. Si $\cH(b)$ est trivialement valu\'e et $\mu_{x}$ ins\'eparable, supposons que tous les \'el\'ements de~$\cV_{b}$ sont inclus dans~$B_{\um}$ et poss\`edent un bord analytique fini. 

Soit~$U_{0}$ un voisinage compact de~$x$ dans~$X$ et soit~$G$ un \'el\'ement de~$\cB(U_{0})$. Supposons que son image dans $\cO_{X_{b},x}$ n'est pas nulle et notons $n \in \N$ la valuation~$\mu_{x}$-adique de cette derni\`ere. Soit~$\eta \in \R_{>0}$.

Alors, il existe un voisinage~$V_{0}$ de~$b$ dans~$B$, un polyn\^ome~$P_{0}$ \`a coefficients dans~$\cB(V_{0})$, unitaire, de m\^eme degr\'e que~$\mu_{x}$, tel que $\|P_{0}(b) - \mu_{x}\|_{b,\infty}\le \eta$ et des nombres r\'eels $s_{1},s_{2},\eps \in \R_{>0}$ avec $s_{1} < s_{2}$ v\'erifiant les propri\'et\'es suivantes~: pour tout $V \in \cV_{b}$ contenu dans~$V_{0}$, tout $P \in \cB(V)[T]$ unitaire de m\^eme degr\'e que~$P_{0}$ tel que $\|P-P_{0}\|_{V,\infty} \le \eps$, on a 
\[x \in \overline{D}_V(P;s_{1}) \subset \overline{D}_V(P;s_{2}) \subset U_{0}\] 
et, pour tout $F \in \cB(\overline{D}_V(P;s_{2}))$ et tout $s \in \intoo{s_{1},s_{2}}$, il existe un unique couple $(Q_{s},R_{s}) \in \cB(\overline{D}_V(P;s))^2$ v\'erifiant 
\begin{enumerate}[i)]
\item $F = Q_{s}G+R_{s}$ dans $\cB(\overline{D}_V(P;s))$;
\item $R_{s} \in \cB(V)[S]$ est un polyn\^ome de degr\'e strictement inf\'erieur \`a~$n\deg(x)$.
\end{enumerate}
En outre, pour tout $s\in \intoo{s_{1},s_{2}}$, il existe $C_{s} \in \R_{>0}$ tel que, pour tout $F \in \cB(\overline{D}_V(P;s_{2}))$, avec les notations pr\'ec\'edentes, 
on ait
\[\|Q_{s}\|_{\overline{D}_V(P;s)} \leq C_{s}\, \|F\|_{\overline{D}_V(P;s_{2})} \textrm{ et } \|R_{s}\|_{\overline{D}_V(P;s)}\leq C_{s}\,\|F\|_{\overline{D}_V(P;s_{2})}.\]
Pour $F \in \cB(\overline{D}_V(P;s_{2}))$ fix\'e, les \'el\'ements $Q_{s}$ et~$R_{s}$ sont ind\'ependants de $s$, au sens o\`u, pour tous $s,s' \in   \intoo{s_{1},s_{2}}$ avec $s'\le s$, on a
\[ (Q_{s})_{| \overline{D}_V(P;s')} = Q_{s'} \textrm{ et } (R_{s})_{| \overline{D}_V(P;s')} = R_{s'},\]
la notation $(\wc)_{| \overline{D}_V(P;s')}$ d\'esignant l'image par le morphisme canonique $\cB(\overline{D}_V(P;s) \to \cB(\overline{D}_V(P;s'))$.
\end{theo}
\begin{proof}
Posons $d:=\deg(x)$.
Il existe~$\alpha_0,\ldots,\alpha_{d-1}\in\cH(b)$ tels que 
\[\mu_{x}(S)=S^d+ \sum_{i=0}^{d-1}\alpha_iS^i.\]
Puisque la valuation~$\mu_{x}$-adique de~$G$ est \'egale \`a~$n$, il existe un \'el\'ement inversible~$H$ de~$\cO_{X_{b},x}$ tel que $G=H \mu_{x}^n$. Quitte \`a restreindre~$U_{0}$, on peut supposer que $H$ et $H^{-1}$ sont d\'efinies sur $U_{0} \cap X_{b}$.

Posons $N := \|\mu_{x}\|_{b,\infty}$. Soit~$v \in \R_{>0}$ tel que
\[(|\alpha_{0}| + 1) v^{-d} + \sum_{i=1}^{d-1}|\alpha_i| v^{i-d} <  \frac{1}{2}.\]

Soient~$V_{0}, P_{0}, s_{1}, s_{2}, \eps$ satisfaisant les conclusions de la proposition~\ref{prop:basevoisdim1rigide} avec $U=U_{0}$ et $\sigma=1$. En particulier, on a $s_{2} \le 1$. On peut \'egalement supposer que~$P_{0}$ est de degr\'e~$d$. D'apr\`es les points i) et ii) \'enonc\'es \`a la fin de la proposition~\ref{prop:basevoisdim1rigide}, si~$\mu_{x}$ est s\'eparable ou si $\cH(b)$ n'est pas trivialement valu\'e, on peut supposer que~$P_{0}(b)$ est s\'eparable. Sinon, par hypoth\`ese, tous les \'el\'ements de~$\cV_{b}$ sont inclus dans~$B_{\um}$ et poss\`edent un bord analytique fini. Dans tous les cas, d'apr\`es la proposition~\ref{prop:CNBNPST}, quitte \`a diminuer~$V_{0}$, $s_{2}$ et~$\eps$ et augmenter~$s_{1}$, on peut supposer que, pour toute partie spectralement convexe~$V$ de~$V_{0}$ contenant~$b$, tout $s \in [s_{1},s_{2}]$ et tout $P \in \cB(V)[T]$ tel que $\|P-P_{0}\|_{V,\infty} \le \eps$, le disque $\overline{D}_{V}(s)$ satisfasse la condition $(\cB N_{P(S)-T})$. 

Soit $P \in \cB(V)[T]$ unitaire de degr\'e~$d$ tel que $\|P-P_{0}\|_{V,\infty} \le \eps$. \'Ecrivons $P$ sous la forme $P= S^d + \sum_{i=0}^{d-1} p_i  S^i \in \cB(V_{0})[S]$. On peut \'egalement supposer que 
\begin{equation}\label{eq1}
\|P(b) - \mu_{x}\|_{b,\infty} \le \max \left(1, \frac{s_{1}^n \, v^{-(n+1) d + n + 1}}{8 n (N+1)^{n-1}} \right)
\end{equation}
et 
\begin{equation*}
(|p_{0}(b)| + 1) v^{-d} + \sum_{i=1}^{d-1} |p_{i}(b)|\, v^{i-d} < \frac12.
\end{equation*}
Par cons\'equent, il existe un voisinage compact spectralement convexe~$V$ de~$b$ contenu dans~$V_{0}$ sur lequel on a 
\begin{equation}\label{eq2}
(\|p_{0}\|_{V} + 1) v^{-d} + \sum_{i=1}^{d-1} \|p_{i}\|_{V}\, v^{i-d} \le \frac12.
\end{equation}

Notons \'egalement que l'\'equation~\eqref{eq1} entra\^ine
\begin{equation}\label{eq:LN+2}
\|P(b)\|_{b,\infty} \le \|P(b) - \mu_{x}\|_{b,\infty} + \|\mu_{x}\|_{b,\infty} \le N+1.
\end{equation}

Soit $s'_{2} \in \intoo{s_{1},s_{2}}$. Les choix effectu\'es assurent que $\overline{D}_{V}(s_{2})$ satisfait la condition $(\cB N_{P(S)-T})$. D'apr\`es le corollaire~\ref{coro:BNPST}, le morphisme naturel
\[ \cB(\overline{D}_{V}(s_{2}))[S]/(P(S) - T) \too \cB(\overline{D}_{V}(P;s_{2}))\]
est donc un isomorphisme. En outre, d'apr\`es la proposition~\ref{prop:restrictionserie}, tout \'el\'ement de $\cB(\overline{D}_{V}(s_{2}))$ induit par restriction un \'el\'ement de $\cB(V)\la |T| \le s'_{2}\ra$. Quitte \`a remplacer~$s_{2}$ par~$s'_{2}$, on peut donc supposer que~$G$ appartient \`a $\cB(V)\la |T| \le s_{2}\ra[S]/(P(S)-T)$.

Soit~$W$ une partie compacte de~$V$ et $t \in (0,1]$. D'apr\`es la proposition~\ref{prop:equivalencedivres} appliqu\'ee avec $\cA = \cB(W)\la |T| \le t\ra$ et $G(S) = P(S) -T$, pour tout $F\in \cB(W)\la |T|\le t\ra[S]/(P(S)-T)$, on a
\begin{equation}\label{eq3}
v^{-d+1} \|F\|_{\cB(W)\la|T|\le t\ra,v,\res} \le \|F\|_{\cB(W)\la|T|\le t\ra,\div}\le 2\|F\|_{\cB(W)\la|T|\le t\ra,v,\res}.
\end{equation}

Par la suite, on notera~$\|.\|_{W,t,v,\res}$ la norme~$\|.\|_{\cB(W)\langle|T|\leq t\rangle,v,\res}$ et~$\|.\|_{W,t,\div}$ la norme~$\|.\|_{\cB(W)\langle|T|\leq t\rangle,\div}$. 

D'apr\`es la proposition~\ref{prop:disqueglobal} et \cite[corollaire~7.4]{EtudeLocale}, les morphismes naturels
\[\underset{t>s_{2}}\colim\;\cH(b)\langle|T|\leq t\rangle \too \cO_{\Aunb}(\overline{D}_{b}(s_{2}))\]
et
\[\cO_{\Aunb}(\overline{D}_{b}(s_{2}))[S]/(P(S)-T) \too \cO_{\Aunb}(\overline{D}_{b}(P;s_{2}))\]
sont des isomorphismes. 

Puisque $\overline{D}_{b}(P;s_{2}) \subset U_{0}\cap X_{b}$, il existe $t>s_{2}$ tel que~$H^{-1}$ poss\`ede un repr\'esentant dans $\cH(b)\langle|T|\leq t\rangle[S]/(P(S)-T)$. Ce repr\'esentant peut lui-m\^eme \^etre approch\'e avec une pr\'ecision arbitraire pour la norme $\nm_{b,t,v,\res}$ par un \'el\'ement de $\cB(W)\langle|T|\leq t\rangle[S]/(P(S)-T)$ pour un voisinage compact~$W$ de~$b$ suffisamment petit. Quitte \`a r\'etr\'ecir~$V$, on peut donc supposer qu'il existe un \'el\'ement~$K$ de~$\cB(V)\langle|T|\leq s_{2}\rangle[S]/(P(S)-T)$  tel que 
\begin{equation}
\|K(b)G(b)-P^n\|_{b,s_{2},v,\res} < \frac{s_{1}^n}{8v^{d-1}}.
\label{eq7}
\end{equation}

On a alors
\begin{equation}
\begin{array}{rcl}
\|\mu_{x}^n-P(b)^n\|_{b,s_{2},v,\res}&\leq&\|\mu_{x}-P(b)\|_{b,s_{2},v,\res} \, \|\sum_{i=0}^{n-1} \mu_{x}^i P(b)^{n-1-i}\|_{b,s_{2},v,\res}\\
&\leq&\|\mu_{x}-P(b)\|_{b,s_{2},v,\res}\, n\max(\|\mu_{x}\|_{b,s_{2},v,\res},\|P(b)\|_{b,s_{2},v,\res})^{n-1}\\
&\overset{(\ref{eq3})}\leq&n v^{n(d-1)}\|\mu_{x}-P(b)\|_{b,s_{2},\div}\,\max(\|\mu_{x}\|_{b,s_{2},\div},\|P(b)\|_{b,s_{2},\div})^{n-1}\\
&\overset{\eqref{eq:LN+2}}\leq&n v^{n(d-1)}\|\mu_{x}-P(b)\|_{b,s_{2},\div}\,(N+1)^{n-1}\\
&\overset{\eqref{eq1}}\leq &\frac{s_{1}^nv^{-d+1}}{8}.
\end{array}
\label{eq9}
\end{equation}

On en d\'eduit que
\begin{equation}
\|K(b)G(b)-P(b)^n\|_{b,s_{2},v,\res}\overset{(\ref{eq7})+(\ref{eq9})}\le \frac{v^{-d+1}s_{1}^n}{4}.
\label{eq10}
\end{equation}
Quitte \`a restreindre~$V$, on peut donc supposer que
\begin{equation}
\|KG-P^n\|_{V,s_{2},v,\res}\leq\frac{v^{-d+1}s_{1}^n}{4}.
\label{eq12}
\end{equation}

Soit $s\in \intoo{s_{1},s_{2}}$. Tout \'el\'ement~$\varphi$ de $\cB(V)\langle |T|\leq s\rangle [S]/(P(S)-T)$ poss\`ede un unique repr\'esentant de la forme 
\[\varphi=\sum_{i=0}^{d-1}(\alpha_i(\varphi)T^n+\beta_i(\varphi))S^i,\]
o\`u les~$\alpha_i(\varphi)$ sont des \'el\'ements de~$\cB(V)\langle|T|\leq s\rangle$ et les~$\beta_i(\varphi)$ des \'el\'ements de~$\cB(V)[T]$ de degr\'e strictement inf\'erieur \`a~$n$. Posons 
\[\alpha(\varphi) := \sum_{i=0}^{d-1}\alpha_i(\varphi)S^i \text{ et } \beta(\varphi) := \sum_{i=0}^{d-1} \beta_i(\varphi)S^i.\]
On a alors
\[\varphi=\alpha(\varphi)T^n+\beta(\varphi)\]
et les deux in\'egalit\'es suivantes :
\begin{equation}\label{eq:alphabeta}
\|\alpha(\varphi)\|_{V,s,\div}\leq s^{-n}\|\varphi\|_{V,s,\div}\text{ et }\|\beta(\varphi)\|_{V,s,\div}\leq \|\varphi\|_{V,s,\div}.
\end{equation}
Remarquons que, dans $\cB(V)\langle |T|\leq s\rangle [S]/(P(S)-T)$, on peut \'ecrire~$\beta(\varphi)$ comme un polyn\^ome en~$S$ de degr\'e inf\'erieur \`a~$nd-1$. R\'eciproquement, tout polyn\^ome en~$S$ de degr\'e strictement inf\'erieur \`a~$nd$ peut s'\'ecrire sous la forme $\sum_{i=0}^{d-1} b_i S^i$ o\`u les~$b_{i}$ sont des polyn\^omes en~$T$ de degr\'e strictement inf\'erieur \`a~$n$.

Consid\'erons, \`a pr\'esent, l'endomorphisme 
\[\fonction{A_{s}}{\cB(V)\langle|T|\leq s\rangle[S]/(P(S)-T)}{\cB(V)\langle|T|\leq s\rangle[S]/(P(S)-T)}{\varphi}{\alpha(\varphi)KG+\beta(\varphi)}.\]
Pour tout~$\varphi\in \cB(V)\langle|T|\leq s\rangle[S]/(P(S)-T)$, on a
\begin{equation}
\begin{array}{rcl}
\|A_{s}(\varphi)-\varphi\|_{V,s,v,\res}&=&\|\alpha(\varphi)(KG-T^n)\|_{V,s,v,\res}\\
&\leq&\|\alpha(\varphi)\|_{V,s,v,\res}\, \|KG-P^n\|_{V,s,v,\res}\\
&\overset{(\ref{eq3})+(\ref{eq:alphabeta})}\leq&2v^{d-1}s^{-n}\,\|KG-P^n\|_{V,s,v,\res}\, \|\varphi\|_{V,s,v,\res}.
\end{array}
\label{eq14}
\end{equation}
On d\'eduit donc de l'in\'egalit\'e (\ref{eq12}) que la norme de l'op\'erateur~$A_{s}-\Id$ sur $\cB(V)\langle|T|\leq s\rangle[S]/(P(S)-T)$ est inf\'erieure \`a $1/2 (s_{1}/s)^n \le 1/2 < 1$. En particulier, $A_{s}$~est un isomorphisme d'espaces de Banach. Remarquons que l'on a
\begin{equation}\label{eq:normeAs-1}
\|A_{s}^{-1}\| \le 2.
\end{equation}

Les choix effectu\'es assurent que $\overline{D}_{V}(s_{2})$ satisfait la condition $(\cB N_{P(S)-T})$. D'apr\`es le corollaire~\ref{coro:BNPST}, le morphisme naturel
\[ \cB(\overline{D}_{V}(s_{2}))[S]/(P(S) - T) \too \cB(\overline{D}_{V}(P;s_{2}))\]
est un isomorphisme et il existe $w\ge v$ et $C' \in \R_{>0}$ tel que, pour tout $F \in \cB(\overline{D}_{V}(s_{2}))[S]/(P(S)-T)$, on ait
\begin{equation*}
\|F\|_{\overline{D}_{V}(s_{2}),w,\res}\leq C' \|F\|_{\overline{D}_{V}(P;s_{2})}
\end{equation*}
et donc 
\begin{equation*}
\|F\|_{\overline{D}_{V}(s_{2}),\div}\leq C'' \|F\|_{\overline{D}_{V}(P;s_{2})}
\end{equation*}
pour une certaine constante $C'' \in \R_{>0}$, d'apr\`es la proposition~\ref{prop:equivalencedivres}. D'apr\`es la proposition~\ref{prop:restrictionserie} tout \'el\'ement~$f$ de~$\cB(\overline{D}_{V}(s_{2}))$ induit naturellement par restriction un \'el\'ement de $\cB(V)\la |T|\le s\ra$ et on a
\begin{equation*}
\|f\|_{V,s} \le \frac{s_{2}}{s_{2}-s}\, \|f\|_{\overline{D}_{V}(s_{2})}
\end{equation*}
On en d\'eduit que tout \'el\'ement~$F$ de $\cB(\overline{D}_{V}(s_{2}))[S]/(P(S) - T)$ induit naturellement par restriction un \'el\'ement de $\cB(V)\la |T|\le s\ra[S]/(P(S)-T)$ et que l'on a
\begin{equation}\label{eq4}
\|F\|_{V,s,\div}\leq C'' \max \left(1, \big(\frac{s_{2}}{s_{2}-s}\big)^{d-1} \right)  \|F\|_{\overline{D}_{V}(P;s_{2})}.
\end{equation}

Soit $F \in \overline{D}_{V}(s_{2})$. Les \'el\'ements $Q_{s}^{\la\ra} := \alpha(A_{s}^{-1}(F)) K$ et $R_{s}^{\la\ra} := \beta(A_{s}^{-1}(F))$ de $\cB(V)\langle|T|\leq s\rangle[S]/(P(S)-T)$ satisfont alors l'\'egalit\'e $F = Q^{\la\ra}_{s}G+R^{\la\ra}_{s}$ dans $\cB(V)\langle|T|\leq s\rangle[S]/(P(S)-T)$. En outre, en combinant \eqref{eq:alphabeta}, \eqref{eq:normeAs-1} et~\eqref{eq4}, on montre qu'il existe une constante $C'_{s} \in \R_{>0}$, ind\'ependante de~$F$, telle que 
\[\|Q_{s}^{\la\ra}\|_{V,s,\div} \leq C'_{s}\, \|F\|_{\overline{D}_V(P;s_{2})} \textrm{ et } \|R_{s}^{\la\ra}\|_{V,s,\div}\leq C'\,\|F\|_{\overline{D}_V(P;s_{2})}.\]
Les choix effectu\'es assurent que $\overline{D}_{V}(s)$ satisfait la condition $(\cB N_{P(S)-T})$. On en d\'eduit que les images respectives~$Q_{s}$ et~$R_{s}$ de~$Q_{s}^{\la\ra}$ et~$R_{s}^{\la\ra}$ dans $\cB(\overline{D}_{V}(s))[S]/(P(S)-T) \simeq \cB(\overline{D}_{V}(P;s))$ satisfont les conditions de l'\'enonc\'e, y compris la partie sur l'\'egalit\'e des normes, pour une constante $C_{s} \in \R_{>0}$ bien choisie.

D\'emontrons maintenant l'unicit\'e du couple $(Q_{s},R_{s})$. Soient $Q'_{s}, R'_{s} \in \cB(\overline{D}_{V}(P;s'))$ satisfaisant les conditions de l'\'enonc\'e. On peut les consid\'erer dans $\cB(\overline{D}_{V}(s))[S]/(P(S)-T)$ et donc dans $\cB(V)\la|T|\le s'\ra[S]/(P(S)-T)$ pour $s' \in \intoo{s_{1},s}$, d'apr\`es la proposition~\ref{prop:restrictionserie}. La bijectivit\'e de l'op\'erateur~$A_{s'}$ assure que les images de~$Q'_{s}$ et~$R'_{s}$ dans $\cB(V)\la|T|\le s'\ra[S]/(P(S)-T)$ co\"incident avec celles de~$Q_{s}$ et~$R_{s}$. Puisque, d'apr\`es la proposition~\ref{prop:restrictionserie}, le morphisme $\cB(\overline{D}_{V}(s)) \to \cB(V)\la|T|\le s'\ra$ est injectif, le r\'esultat s'ensuit.

L'assertion d'ind\'ependance finale d\'ecoule de l'unicit\'e.
\end{proof}

\begin{rema}\label{rem:kappaseparable}
Il d\'ecoule du point iii) final de la proposition \ref{prop:basevoisdim1rigide} que, si $\mu_{x}$ est \`a coefficients dans~$\kappa(b)$ et s\'eparable, on peut choisir pour $P_{0}$ n'importe quel relev\'e de $\mu_{x}$ \`a $\cO_{B,b}[T]$.

\end{rema}

D\'emontrons \`a pr\'esent quelques cons\'equences de ce r\'esultat. Soit $G(T) \in \cA[T]$ un polyn\^ome unitaire de degr\'e~$d \ge 1$. Fixons un point~$b$ de~$B$. 

Supposons que $G(b)(T)$ est une puissance d'un polyn\^ome irr\'eductible dans~$\cH(b)[T]$. Le lieu d'annulation de~$G(b)$ d\'etermine alors un unique point rigide~$x$ de~$X_{b}$ et il existe un entier $n\in \N^\ast$ tel que $G(b) = \mu_{x}^n$. Posons $d:=\deg(x)$.

\begin{coro}\label{cor:divisionGirrednormes}
Pla\c cons-nous dans le cadre pr\'ec\'edent. Si~$\cH(b)$ est trivialement valu\'e et~$\mu_{x}$ ins\'eparable, supposons que $b$~est ultram\'etrique typique. 

Soit~$U$ un voisinage de~$x$ dans~$X$. Alors, il existe un voisinage compact spectralement convexe~$V$ de~$b$ dans~$B$, un polyn\^ome $P_{0} \in \cB(V)[T]$ unitaire non constant et des nombres r\'eels $s_{1},s_{2} \in \R_{>0}$ avec $s_{1} < s_{2}$ tels qu'on ait
\[ x \in \overline{D}_{V}(P_{0};s_{1}) \subset \overline{D}_{V}(P_{0};s_{2}) \subset U\] 
et, en posant, 
\[\psi \colon 
\begin{array}{ccc}
\Bc(V)^{nd} &\too& \Bc(\overline{D}_{V}(P_{0};s_2)),\\
(a_{0},\dotsc,a_{nd-1}) & \mapstoo & \sum_{i=0}^{nd-1} a_{i}\, T^i
\end{array}\]
les propri\'et\'es suivantes soient v\'erifi\'ees~:
\begin{enumerate}[i)]
\item il existe $K \in \R_{>0}$ tel que, pour tout $A \in \Bc(V)^{nd}$ et tout $s\in [s_{1},s_{2}]$, on ait $\|\psi(A)\|_{\overline{D}_{V}(P_{0};s)} \le K\, \|A\|_{V,\infty}$;
\item pour tout $s\in \intoo{s_{1},s_{2}}$ et pour tout $B \in \Bc(\overline{D}_{V}(P_{0};s_{2}))$, il existe un unique $A_s \in \Bc(V)^{nd}$ tel que $\psi(A_s) = B \mod G$ dans $\Bc(\overline{D}_{V}(P_{0};s))$. L'\'el\'ement $A_{s}$ est ind\'ependant de $s\in \intoo{s_{1},s_{2}}$. En outre, il existe $K'_{s} \in \R_{>0}$ (ind\'ependant de~$B$) tel que $\|A_s\|_{V,\infty} \le K'_{s} \, \|B\|_{\overline{D}_{V}(P_{0};s_{2})}$. 
\end{enumerate} 
En particulier, le morphisme naturel
\[\Oc_{b}[T]/(G(T)) \too \Oc_{x}/(G(T))\]
est un isomorphisme.
\end{coro}
\begin{proof}
Soit~$U_{0}$ un voisinage compact de~$x$ dans~$U$ tel que $G \in \cB(U_{0})$. Soit~$\cV_b$ une base de voisinages compacts spectralement convexes de~$b$ dans~$B$. Si $\cH(b)$ est trivialement valu\'e et $\mu_{x}$ ins\'eparable, supposons de plus que tous les \'el\'ements de~$\cV_{b}$ sont inclus dans~$B_{\um}$ et poss\`edent un bord analytique fini. 

$\bullet$ Supposons que $x\ne 0$.

Posons $d:=\deg(x)$ et \'ecrivons $\mu_{x}(T) = T^d + \sum_{i=0}^{d-1} a_{i} T^i$. Posons $I := \{i \in \cn{0}{d-1} : a_{i} \ne 0\}$. Puique $x\ne 0$, on a $I \ne \emptyset$. Soit $\eta \in \intoo{0,\min_{i \in I}(|a_{i}|)}$. 

Le th\'eor\`eme~\ref{weierstrassam} appliqu\'e avec $x$, $G$ et~$\eta$ fournit un voisinage~$V_{0}$ de~$b$ dans~$B$, un polyn\^ome $P_{0} \in \cB(V_{0})[T]$ unitaire de degr\'e~$d$ et des nombres r\'eels $s_{1},s_{2},\eps \in \R_{>0}$. 
Puisque $\|P_{0}(b) - \mu_{x}\| \le \eta$, $P_{0}(b)$ n'est pas une puissance de~$T$. 
Soit $V$ un \'el\'ement de~$\cV_{b}$ contenu dans~$V_{0}$. Identifions~$P_{0}$ \`a son image dans $\cB(V)[T]$.
Quitte \`a  restreindre~$V$, on peut supposer que, pour tout $b'\in B$, $P_{0}(b')$ n'est pas une puissance de~$T$. La propri\'et\'e~i) est imm\'ediate avec $K = nd \max(\|T\|^{nd-1}_{\overline{D}_{V}(P_{0};s_{2})},1)$. La propri\'et\'e~ii) d\'ecoule 
du th\'eor\`eme~\ref{weierstrassam} et du corollaire~\ref{coro:minorationDPsA}.

$\bullet$ Supposons que $x=0$.

D'apr\`es la remarque~\ref{rem:kappaseparable}, on peut choisir $P_{0} = T$ dans l'\'enonc\'e du th\'eor\`eme~\ref{weierstrassam}. 
Le m\^eme raisonnement que pr\'ec\'edemment s'applique alors, en utilisant la remarque~\ref{rem:Tn} au lieu du corollaire~\ref{coro:minorationDPsA}.
\end{proof}

Des arguments standards permettent de d\'emontrer un r\'esultat similaire sans hypoth\`ese sur~$G(b)(T)$. 
Nous renvoyons \`a la preuve de \cite[th\'eor\`eme~5.5.3]{A1Z} pour les d\'etails (dans un cadre o\`u les normes ne sont pas prises en compte).

Notons $G(b)(T) = \prod_{j=1}^t h_{j}(T)^{n_{j}}$ la d\'ecomposition en produit de polyn\^omes irr\'eductibles et unitaires de~$G(b)(T)$ dans~$\cH(b)[T]$. D'apr\`es~\cite[corollaire~5.4]{EtudeLocale}, il existe $H_{1},\dotsc,H_{t} \in \cO_{B,b}[T]$ unitaires tels que
\begin{enumerate}[i)]
\item $G = \prod_{j=1}^t H_{j}$ dans $\cO_{B,b}[T]$;
\item pour tout $j\in \cn{1}{t}$, on a $H_{j}(b) = h_{j}^{n_{j}}$.
\end{enumerate}
Il s'agit de la condition~$(D_{G})$ de \cite[d\'efinition~4.6]{EtudeLocale}, elle-m\^eme li\'ee \`a la condition $(I_{G})$ de \cite[d\'efinition~5.3.5]{A1Z} qui est pr\'esente dans l'\'enonc\'e de \cite[th\'eor\`eme~5.5.3]{A1Z}. Le r\'esultat pr\'ec\'edent est cons\'equence de l'hens\'elianit\'e du corps~$\kappa(b)$ (\cf~\cite[th\'eor\`eme~5.2]{EtudeLocale}).

Pour tout $j\in \cn{1}{t}$, notons $z_{j}$ le point rigide de~$X_{b}$ d\'etermin\'e par l'annulation de~$h_{j}$.

\begin{coro}\label{cor:divisionGnormes}\index{Theoreme@Th\'eor\`eme!de division de Weierstra\ss!g\'en\'eralis\'e avec contr\^ole des normes}
Pla\c cons-nous dans le cadre pr\'ec\'edent. 
Si~$\cH(b)$ est trivialement valu\'e et l'un des~$h_{j}$ est ins\'eparable, supposons que $b$~est ultram\'etrique typique. 

Pour tout $j \in \cn{1}{t}$, soit~$U_{j}$ un voisinage de~$z_{j}$ dans~$X$ sur lequel $H_{j}$~est d\'efini. Alors, il existe un voisinage compact spectralement convexe~$V$ de~$b$ dans~$B$, des polyn\^omes $P_{1,0},\dotsc,P_{t,0} \in \cB(V)[T]$ unitaires non constants et des nombres r\'eels $s_{1,1},s_{1,2},\dotsc,s_{t,1},s_{t,2} \in \R_{>0}$ avec $s_{j,1} < s_{j,2}$ pour tout~$j$ tels qu'on ait 
\[x\in \overline{D}_{V}(P_{j,0};s_{j,1}) \subset  \overline{D}_{V}(P_{j,0};s_{j,2}) \subset U_{j} \textrm{ pour tout }j\] 
et, en posant
\[\psi \colon 
\begin{array}{ccc}
\Bc(V)^d &\too& \prod_{j=1}^t \Bc(\overline{D}_{V}(P_{j,0};s_{j,2})),\\
(a_{0},\dotsc,a_{d-1}) & \mapstoo & (\sum_{i=0}^{d-1} a_{i}\, T^i,\dotsc,\sum_{i=0}^{d-1} a_{i}\, T^i)
\end{array}\]

les propri\'et\'es suivantes soient v\'erifi\'ees~:
\begin{enumerate}[i)]
\item il existe $K \in \R_{>0}$ tel que, pour tout $A \in \Bc(V)^d$ et tout $s \in \prod_{j=1}^t [s_{j,1},s_{j,2}]$, on ait 
\[\|\psi(A)\|_{s} \le K\, \|A\|_{V,\infty},\]
o\`u $\nm_{s}$ d\'esigne la norme uniforme sur $\prod_{j=1}^t\Bc(\overline{D}_{V}(P_{j,0};s_{j}))$~;
\item pour tout $s \in \prod_{j=1}^t \intoo{s_{j,1},s_{j,2}}$ et tout $B \in  \prod_{j=1}^t \Bc(\overline{D}_{V}(P_{j,0};s_{j,2}))$, il existe un unique $A_{s} \in \Bc(V)^d$ tel que $\psi(A_{s}) = B$ dans $\prod_{j=1}^t \Bc(\overline{D}_{V}(P_{j,0};s_{j}))/(H_{j})$. L'\'el\'ement~$A_{s}$ est ind\'ependant de $s \in \prod_{j=1}^t \intoo{s_{j,1},s_{j,2}}$. En outre, il existe $K'_{s} \in \R_{>0}$ (ind\'ependant de~$B$) tel que $\|A_{s}\|_{V,\infty} \le K'_{s} \, \|B\|_{s_{2}}$,
o\`u $\nm_{s_{2}}$ d\'esigne la norme uniforme sur $\prod_{j=1}^t \Bc(\overline{D}_{V}(P_{j,0};s_{j,2}))$.
\end{enumerate} 
En particulier, le morphisme naturel
\[\Oc_{b}[T]/(G(T)) \too \prod_{j=1}^t\Oc_{z_{j}}/(H_{j})\]
est un isomorphisme.
\end{coro}

Appliquons finalement ces r\'esultats au cas d'un endomorphisme polynomial de~$X$, de fa\c con \`a obtenir une version norm\'ee du th\'eor\`eme~\ref{thm:isolemniscate}. Soit $P\in \cA[T]$ un polyn\^ome unitaire de degr\'e $d\ge 1$. Le morphisme 
\[\begin{array}{ccc}
\Ac[T] &\too& \Ac[T]\\
T & \mapstoo & P(T)
\end{array}\]
induit un morphisme $\varphi_{P} \colon X\to X$. \index{Endomorphisme de la droite}

Remarquons que si $U$ et $V$ sont des parties compactes de~$X$ telles que $\varphi_{P}(U) \subset W$, alors nous avons un morphisme naturel $\cB(W) \to \cB(U)$ envoyant~$T$ sur~$P(T)$.
 
\begin{coro}\label{cor:phinormes} \index{Endomorphisme de la droite}
Pla\c cons-nous dans le cadre pr\'ec\'edent. Si~$\cH(b)$ est de caract\'eristique non nulle et trivialement valu\'e, supposons que $b$~est ultram\'etrique typique. 

Soit~$y \in X$. Notons $\varphi_{P}^{-1}(y) = \{x_{1},\dotsc,x_{t}\}$. Soit~$W$ un voisinage de~$y$ dans~$X$. Pour $j\in \cn{1}{t}$, soit $U_{j}$ un voisinage de~$x_{j}$ dans~$\varphi_{P}^{-1}(W)$. Alors, il existe un voisinage compact spectralement convexe~$V$ de~$y$ dans~$W$ et, pour tout $j\in \cn{1}{t}$, un voisinage spectralement convexe~$U'_{j}$ de~$x_{j}$ dans $\varphi_{P}^{-1}(V) \cap U_{j}$ tels que, 
en d\'efinissant
\[\chi_{W} \colon 
\begin{array}{ccc}
\Bc(W)^d &\too& \Bc(U_{1}) \times \dotsb \times \Bc(U_{t})\\
(a_{0}(T),\dotsc,a_{d-1}(T)) & \mapstoo & \disp \big(\sum_{i=0}^{d-1} \varphi_{P}^\sharp(a_{i})\, T^i,\dotsc,\sum_{i=0}^{d-1} \varphi_{P}^\sharp(a_{i})\, T^i\big)
\end{array}\]
et $\chi_{V} \colon \Bc(V)^d \to \prod_{j=1}^t \Bc(U'_{j})$ par la m\^eme formule, 
les propri\'et\'es suivantes soient v\'erifi\'ees~:
\begin{enumerate}[i)]
\item il existe $K \in \R_{>0}$ tel que, pour tout $A \in \Bc(W)^d$, on ait $\|\chi_{W}(A)\|_{U} \le K\, \|A\|_{W,\infty}$ et $\|\chi_{V}(A)\|_{U'} \le K\, \|A\|_{V,\infty}$, o\`u~$\nm_{U}$ et~$\nm_{U'}$ d\'esignent respectivement les normes infini sur $\prod_{i=1}^t \cB(U_{i})$ et $\prod_{i=1}^t \cB(U_{i}')$~; 
\item il existe $K' \in \R_{>0}$ tel que, pour tout $B \in  \prod_{j=1}^t \Bc(U_{j})$, il existe un unique $A \in \Bc(V)^d$ tel que $\chi_{V}(A) = B$ (dans $\prod_{j=1}^t \Bc(U'_{j})$) et $\|A\|_{V,\infty} \le K' \, \|B\|_{U}$.
\end{enumerate}  
En particulier, le morphisme $\chi \colon \Oc_{y}^d \to \prod_{j=1}^t \Oc_{x_{j}}$ d\'efini par la m\^eme formule que~$\chi_{W}$ est un isomorphisme.
\end{coro}
\begin{proof}
On se ram\`ene au corollaire~\ref{cor:divisionGnormes} en \'ecrivant comme d'habitude le morphisme~$\varphi_{P}$ comme celui provenant de la composition
\[\cA[T] \too \cA[T,S]/(P(S) - T) \simtoo \cA[S].\]
Notons $p_{1}, p_{2}$ les projections de~$\E{2}{\cA}$ sur $\AunA$. Notons~$\sigma$ la section de~$p_{2}$ envoyant $x$ sur $(P(x),x)$. 

\[\begin{tikzcd}
& \E{2}{\cA}\ar[ld, "p_{1}"] \ar[rd, "p_2"'] &\\
\AunA &&\AunA  \ar[ll, "\varphi_{P}"'] \ar[ul, bend right, "\sigma"']
\end{tikzcd}\]

Notons $Z$ le ferm\'e de Zariski de~$\E{2}{\cA}$ d\'efini par l'\'equation $P(S)-T = 0$. La projection~$p_{2}$ induit un isomorphisme entre~$Z$ et~$\AunA$ dont l'inverse est la section~$\sigma$.

D'apr\`es la proposition~\ref{prop:typiqueAn}, si~$b$ est ultram\'etrique typique, il en va de m\^eme pour~$y$ et les~$\sigma(x_{j})$. Nous pouvons donc appliquer le corollaire~\ref{cor:divisionGnormes} en rempla\c cant $\cA$ par un voisinage compact spectralement convexe de~$y$ dans~$W$, $T$ par~$S$, $G$ par $P(S)-T$ et chaque $U_{j}$ par un voisinage de~$\sigma(x_{j})$ dans $p_{2}^{-1}(U_{j})$ sur lequel~$H_{j}$ est d\'efini. On obtient l'\'enonc\'e d\'esir\'e en posant $U'_{j} = p_{2}(\overline{D}_{V}(P_{j,0},s_{j}))$ et en utilisant le fait que tirer en arri\`ere une fonction (ici par~$\sigma$ ou~$p_{2}$) diminue sa norme.
\end{proof}

\section{Fermeture des id\'eaux du faisceau structural}\label{sec:fermeture}

Dans cette section, nous d\'emontrons un r\'esultat de fermeture pour les faisceaux d'id\'eaux du faisceau structural d'un espace analytique. 
Nous commencerons par \'etudier quelques cas particuliers~: celui de l'id\'eal nul et celui de l'id\'eal engendr\'e par une puissance d'une uniformisante (en un point situ\'e au-dessus d'un point dont l'anneau local est fortement de valuation discr\`ete).

Soit $(\Ac,\nm)$ un anneau de Banach. Posons $B := \Mc(\Ac)$. 

\begin{lemm}\label{fortement_de_valuation_discr\`ete}
\index{Anneau!fortement de valuation discrete@fortement de valuation discr\`ete}
\index{Ideal@Id\'eal!B-fortement de type fini@$\cB$-fortement de type fini}
Soit~$b \in B$. Supposons que~$\cO_{B,b}$ est un anneau fortement de valuation discr\`ete relativement \`a une base fine~$\cV$ de voisinages compacts spectralement convexes de~$b$ dans~$B$. Soit~$\varpi_{b}$ une uniformisante de~$\cO_{B,b}$. Alors, pour tout~$v\in\N$, l'id\'eal~$(\varpi_b^v)\subset\cO_{B,b}$ est~$\cB$-fortement de type fini relativement \`a~$\cV$.
\end{lemm}
\begin{proof}
Nous allons d\'emontrer le r\'esultat par r\'ecurrence sur~$v$. Le cas $v=0$ est imm\'ediat. 

Supposons avoir d\'emontr\'e le r\'esultat pour $v\ge 0$. Soient~$U$ un voisinage compact de~$b$ et~$V$ un \'el\'ement de~$\cV$ tel que $V \subset \mathring U$. Puisque~$\cV$ est une base fine de voisinages, il existe un \'el\'ement~$V_1$ de~$\cV$ tel que $V_{1} \subset \mathring U$ et $V \subset \mathring V_1$. Soit $f \in \cB(U)$ dont l'image dans~$\cO_{B,b}$ appartient \`a l'id\'eal~$(\varpi_b^{v+1})$.
 
Par d\'efinition d'anneau fortement de valuation discr\`ete, il existe $g \in \cB(V_1)$ tel que $f = g \varpi_b$ sur~$\cB(V_1)$ et $\|g\|_{V_1}\leq K_{V_1,U}\,\|f\|_U$, o\`u~$K_{V_1,U}$ est la constante intervenant dans la d\'efinition d'anneau fortement de valuation discr\`ete. De plus, puisque~$\cO_{B,b}$ est int\`egre, le germe au voisinage de~$b$ correspondant \`a~$g$ est un multiple de~$\varpi_b^{v}$. 

Par hypoth\`ese de r\'ecurrence, il existe $h \in \cB(V)$ telle que~$h\varpi_b^{v}=g$ dans~$\cB(V)$ et $\|h\|_{V}\leq K_{V,V_1,v} \,\|g\|_{V_1}$ o\`u~$K_{V,V_1,v}$ est la constante intervenant dans la d\'efinition d'id\'eal~$\cB$-fortement de type fini de~$(\varpi_b^{v})$. Ainsi, on a l'\'egalit\'e~$h\pi_{b}^{v + 1}=f$ dans~$\cB(V)$ et l'in\'egalit\'e~$\|h\|_{V}\leq K_{V,V_1,v}\, K_{V_1,U}\,\|f\|_U$. 
 \end{proof}

\begin{lemm}\label{division_anneau_valuation_discr\`ete}
Soit~$b$ un point d\'ecent de~$B$. 
Supposons que $\cO_{B,b}$ est un anneau fortement de valuation discr\`ete. Soit~$\varpi_b$ une uniformisante de $\cO_{B,b}$. Soient~$x$ un point rigide \'epais de~$\E{n}{\cA}$ au-dessus de~$b$ et $v\in \N$. Alors, pour tout voisinage compact~$V$ de~$x$ dans~$\E{n}{\cA}$, il existe un voisinage compact spectralement convexe~$V'$ de~$x$ dans~$\mathring{V}$ et une constante~$K_{V',V,v} \in \R_{>0}$ tels que, pour tout $f \in \cB(V)$ dont l'image dans $\cO_{\bA^n_\cA,x}$ appartient \`a $(\varpi_b^v)$, il existe $g \in \cB(V')$ tel que 
\begin{enumerate}[i)]
\item $f=\varpi_b^vg$ dans $\cB(V')$ ;
\item $\|g\|_{V'}\leq K_{V',V,v}\,\|f\|_V$.
\end{enumerate}
\end{lemm}
\begin{proof}
Nous allons d\'emontrer ce r\'esultat par r\'ecurrence sur~$n$. Le cas~$n=0$ est donn\'e par le lemme \ref{fortement_de_valuation_discr\`ete}.

Soit~$n\ge 1$ et supposons que le r\'esultat est satisfait pour~$\E{n-1}{\cA}$. Notons~$x_{n-1}$ la projection de~$x$ sur ses~$n-1$ premi\`eres coordonn\'ees. 

Commen\c cons par d\'emontrer l'\'enonc\'e dans le cas o\`u~$x$ est le point~0 au-dessus du point~$x_{n-1}$, que nous noterons~$0_{x_{n-1}}$. Soient~$V$ un voisinage compact de~$x$ dans~$\E{n}{\cA}$ et $f \in \cB(V)$ dont l'image dans $\cO_{0_{x_{n-1}}}$ appartienne \`a $(\varpi_b^v)$. Il existe $g\in\cO_{0_{x_{n-1}}}$ tel que~$f=\varpi_b^vg$ dans~$\cO_{0_{x_{n-1}}}$.

Le point~$0_{x_{n-1},n}$ admet une base de voisinages de la forme~$\overline{D}_W(t)$, o\`u~$W$ parcourt une base de voisinages de~$x_{n-1}$. Soit~$W$ un voisinage compact spectralement convexe de~$x_{n-1}$ et~$t \in \R_{>0}$ tels que $\overline{D}_W(t)$ soit contenu dans~$\mathring{V}$, les fonctions~$f$, $g$ et~$\varpi_b$ soient~$\cB$-d\'efinies sur~$\overline{D}_W(t)$ et $f=\varpi_b^vg$ dans~$\cB(\overline{D}_W(t))$. Soit $s \in \intoo{0,t}$. D'apr\`es la proposition~\ref{prop:restrictionserie}, on a 
\[\|f\|_{W,s}\leq \frac{t}{t-s}\, \|f\|_{\overline{D}_W(t)} \textrm{ et }\|g\|_{W,s}\leq \frac{t}{t-s}\,\|g\|_{\overline{D}_W(t)},\]
 o\`u les normes \`a gauche des in\'egalit\'es sont les normes de~$f$ et~$g$ vues comme des \'el\'ements $\sum_{i\ge0}\alpha_iT^i$ et $\sum_{i\ge0}\beta_iT^i$ de $\cB(W)\langle |T|\leq s\rangle$. 
 Nous avons  
\[\sum_{i\ge0}\alpha_iT^i=\varpi_b^v \sum_{i\ge0}\beta_iT^i \textrm{ dans } \cB(W)\langle |T|\leq s\rangle,\]
ce qui implique que, pour tout~$i \ge 0$,~$\alpha_i=\varpi_b^v\beta_i$.

Par hypoth\`ese de r\'ecurrence, il existe un voisinage spectralement convexe~$W'$ de~$x_{n-1}$ dans~$\mathring{W}$, une constante $K_{W',W,v} \in \R_{>0}$ et, pour tout $i\ge 0$, un \'el\'ement~$\beta_{i,v}$ de~$\cB(W')$ tels que 
 \begin{enumerate}[i)]
 \item pour tout $i\ge 0$, $\|\beta_{i,v}\|_{W'}\leq K_{W',W,v}\, \|\alpha_i\|_W$ ; 
 \item $\sum_{i\ge0}\alpha_iT^i=\varpi_b^{v}\big(\sum_{i\ge0}\beta_{i,v}\, T^i\big)$ dans~$\cB(W')\langle |T|\leq s\rangle$.
 \end{enumerate}
Le voisinage $V' := \overline{D}_{W'}(s)$ satisfait donc les propri\'et\'es de l'\'enonc\'e.

\smallbreak
 
Passons maintenant au cas o\`u $x$ est un point rigide \'epais quelconque au-dessus de~$x_{n-1}$. Soit~$P(S)\in\kappa(x_{n-1})[S]$ le polyn\^ome minimal du point~$x$ au-dessus de~$x_{n-1}$. On note~$\tilde{P}\in\cO_{x_{n-1}}[S]$ un rel\`evement unitaire de~$P$. 
Soit~$\tilde{U}$ un voisinage compact spectralement convexe de~$x_{n-1}$ tel que chacun des coefficients de~$\tilde{P}$ soit~$\cB$-d\'efini sur~$\tilde{U}$. En notant $\varphi_P \colon \E{1}{\cB(\tilde U)}\to \E{1}{\cB(\tilde{U})}$ le morphisme d\'efini par le polyn\^ome~$\tilde P\in\cB(\tilde{U})[T]$, nous avons $\varphi_{P}^{-1}(0_{x_{n-1}}) = \{x\}$. Le r\'esultat se d\'eduit alors du cas du point~$0_{x_{n-1}}$ et du corollaire~\ref{cor:phinormes}

On note encore~$\varphi_P:\E{1}{\cB(\tilde U)}\to \E{1}{\cB(\tilde{U})}$ le morphisme d'espaces analytiques associ\'e au polyn\^ome~$\tilde P\in\cB(\tilde{U})[T]$. Il envoie le point~$x$ sur le point~$0_{x_{n-1}}$ et nous avons $\varphi_{P}^{-1}(0_{x_{n-1}}) = \{x\}$.
\end{proof}

La proposition suivante montre que, sous certaines conditions, la propri\'et\'e que l'id\'eal nul de l'anneau local en un point soit~$\cB$-fortement de type fini se transf\`ere aux points de la fibre.

\begin{prop}\label{prolongement_purement_localement_transcendant}
Soit~$b \in B$. Soit~$\cV_{b}$ une base fine de voisinages compacts spectralement convexes de~$b$ dans~$B$. Si~$\cH(b)$ est de caract\'eristique non nulle et trivialement valu\'e, supposons que tout \'el\'ement de~$\cV_{b}$ poss\`ede un bord analytique fini. Supposons que~$\cO_{B,b}$ soit un corps fort ou un anneau fortement de valuation discr\`ete relativement \`a~$\cV_b$ et que l'id\'eal nul de~$\cO_{B,b}$ soit $\cB$-fortement de type fini relativement \`a~$\cV_b$.

Soit~$x$ un point de~$\AunA$ au-dessus de~$b$. Alors, il existe une base fine de voisinages compacts spectralement convexes de~$x$ dans~$\E{1}{\cA}$ relativement \`a laquelle l'id\'eal nul de~$\cO_{\AunA,x}$ est $\cB$-fortement de type fini.
\end{prop}
\begin{proof}
Nous supposerons que~$\cO_{B,b}$ est un anneau fortement de valuation discr\`ete d'uniformisante~$\varpi_{b}$.
Le cas des corps forts se traite de mani\`ere similaire (et plus simple).

Soit~$\cV_{x}$ une base de voisinages de~$x$ satisfaisant les conclusions du corollaire~\ref{cor:BVfineNPST}. Soit~$U$ un voisinage compact de~$x$ et $f$ un \'el\'ement de~$\cB(U)$ dont l'image dans~$\cO_{x}$ est nulle. Soit~$\overline{C}_V(P,r,s)$ un \'el\'ement de~$\cV_x$ contenu dans~$\mathring U$. D'après le point~ii) du corollaire~\ref{cor:BVfineNPST}, il existe $r',s'\in \R_{\ge 0}$ avec $r'\prec r$ et $s'>s$ et un \'el\'ement~$V'$ de~$\cV_{b}$ contenant~$V$ tels que $\overline{C}_{V'}(P,r',s')$ appartienne \`a~$\cV_{x}$ et que l'on ait
\[\overline{C}_{V}(P,r,s)\subset\mathring{\overbrace{\overline{C}_{V'}(P,r',s')}}\subset \overline{C}_{V'}(P,r',s')\subset\mathring U.\]

On va montrer que~$f$ est nulle sur~$\overline{C}_V(P,r,s)$. D'après le point~i) du corollaire~\ref{cor:BVfineNPST}, $\overline{C}_{b}(P,r',s')$ est connexe, donc le principe du prolongement analytique assure que~$f$ y est nulle. D'apr\`es \cite[corollaire~9.8]{EtudeLocale}, $f$ est multiple de~$\pi_{b}$ sur~$\cB(\overline{C}_{V'}(P,r',s'))$. En r\'ep\'etant cet argument (en partant d'un $\overline{C}_{V''}(P,r'',s'')$ l\'eg\`erement plus grand), on montre que, pour tout~$n\in\N$, $f$~est multiple de~$\pi_{b}^n$ dans~$\cB(\overline{C}_{V'}(P,r',s'))$. La seconde partie de l'\'enonc\'e de \cite[corollaire~9.8]{EtudeLocale} assure alors que $f$ est nulle au voisinage de~$\overline{C}_{b}(P,r',s')$. 

D'après le point~iii) du corollaire~\ref{cor:BVfineNPST}, le morphisme naturel 
\[\cB(\overline{C}_{V'}(r',s'))[S]/(P(S)-T)\too\cB(\overline{C}_{V'}(P;r',s'))\]
est un isomorphisme. En utilisant la proposition~\ref{prop:restrictionserie}, on en d\'eduit un morphisme injectif naturel
\[\cB(\overline{C}_{V'}(P,r',s''))\too \cB(V')\langle r\leq|T|\leq s\rangle[S]/(P(S)-T).\]
L'image de~$f$ par ce morphisme poss\`ede un repr\'esentant de la forme $\sum_{i=0}^{d-1}\alpha_iS^i$, o\`u~$d$ est le degr\'e du polyn\^ome~$P$ et, pour tout $i \in \cn{0}{d-1}$, $\alpha_i = \sum_{k=-\infty}^{+\infty}a_{i,k}T^k$ est un \'el\'ement~$\cB(V')\langle r\leq|T|\leq s\rangle$. Puisque~$f$ est nulle au voisinage de~$\overline{C}_{b}(P,r',s')$, pour tous $i \in \cn{0}{d-1}$ et $k \in \Z$, $a_{i,k}$ est nul au voisinage de~$b$, et donc sur~$V'$ car l'id\'eal nul de~$\Oc_{B,b}$ est $\cB$-fortement de type fini relativement \`a~$\cV_b$. Le r\'esultat s'ensuit.
\end{proof}

On peut d\'emontrer un \'enonc\'e similaire sans hypoth\`eses sur~$\cO_{B,b}$ dans le cas des points rigides \'epais. 

\begin{prop}\label{prolongement_rigide}
Soit~$b \in B$. Soit~$\cV_{b}$ une base fine de voisinages compacts spectralement convexes de~$b$ dans~$B$. Supposons que l'id\'eal nul de~$\cO_{B,b}$ soit $\cB$-fortement de type fini relativement \`a~$\cV_b$. 

Soit~$x$ un point de~$\AunA$ rigide \'epais au-dessus de~$b$. Si~$\mu_{\kappa,x}$ est ins\'eparable et si $\cH(b)$ est trivialement valu\'e, supposons que tout \'el\'ement de~$\cV_{b}$ poss\`ede un bord analytique fini. Alors, il existe une base fine de voisinages compacts spectralement convexes de~$x$ dans~$\E{1}{\cA}$ relativement \`a laquelle l'id\'eal nul de~$\cO_{\AunA,x}$ est $\cB$-fortement de type fini.
\end{prop}
\begin{proof}
Posons $d :=\deg(\mu_{\kappa,x})$. Choisissons un relev\'e unitaire~$P_{0}$ de~$\mu_{\kappa,x}$ dans~$\Oc_{B,b}[T]$. Nous allons distinguer deux cas.

\medbreak

$\bullet$ Supposons que $\mu_{\kappa,x}$ est s\'eparable ou que tout \'el\'ement de~$\cV_{b}$ poss\`ede un bord analytique fini. 

Soit~$\cV_{x}$ une base de voisinages de~$x$ satisfaisant les conclusions du corollaire~\ref{cor:BVfineNPSTrigideepais} iii). 
Soient~$U$ un voisinage compact de~$x$, $\overline{D}_V(P_{0},s)$ un \'el\'ement de~$\cV_x$ inclus dans~$\mathring U$ et $f$ un \'el\'ement de $\cB(U)$ dont l'image dans~$\Oc_x$ est nulle. 

Les propri\'et\'es de~$\cV_x$ assurent que le morphisme naturel 
\[\cB(V)\langle |T|\leq s\rangle[S]/(P_{0}(S)-T) \too \cB(\overline{D}_{V}(P_{0},s))\]
est un isomorphisme. L'ant\'ec\'edent de~$f$ poss\`ede un unique repr\'esentant de la forme $\sum_{i=0}^{d-1}\alpha_iS^i$, avec $\alpha_i = \sum_{k=0}^{+\infty}a_{i,k}T^k\in\cB(V)\langle|T|\leq s\rangle$.

Pour tous~$i$ et~$k$, le germe de~$a_{i,k}$ en~$b$ est nul. En effet, puisque le germe de~$f$ en~$x$ est nul, il existe~$V'\subset V$ et~$s'\leq s$ tels que $\overline{D}_{V'}(P_{0},s')$ appartienne \`a~$\cV_{x}$ et $f$~soit nulle sur $\overline{D}_{V'}(P_{0},s')$. Le r\'esultat d\'ecoule alors du fait que le morphisme naturel 
\[\cB(V')\langle |T|\leq s'\rangle[S]/(P_{0}(S)-T) \too \cB(\overline{D}_{V'}(P_{0},s'))\]
est un isomorphisme.

Puisque l'id\'eal nul de~$\cO_{B,b}$ est $\cB$-fortement de type fini relativement \`a~$\cV_b$, pour tous~$i$ et~$k$, $a_{i,k}$ est nul sur~$V$. Le r\'esultat s'ensuit.

\medbreak

$\bullet$ Supposons que~$\cH(b)$ est ultram\'etrique et non trivialement valu\'e.

Soit~$\cV_{x}$ une base de voisinages de~$x$ satisfaisant les conclusions du corollaire~\ref{cor:BVfineNPSTrigideepais} iii'). Soient~$U$ un voisinage compact de~$x$, $\overline{D}_V(P_{0},s)$ un \'el\'ement de~$\cV_x$ inclus dans~$\mathring U$ et $f$ un \'el\'ement de $\cB(U)$ dont l'image dans~$\Oc_x$ est nulle. Par hypoth\`ese, il existe $Q \in \cB(V)[S]$ unitaire de degr\'e~$d$ et un intervalle ouvert~$I$ de~$\R_{>0}$ contenant~$t$ tels que, pour tout voisinage compact spectralement convexe~$V'$ de~$b$ dans~$V$ et tout $s' \in I$, on ait $\overline{D}_{V'}(Q,s') = \overline{D}_{V'}(P_{0},s')$ et le disque  $\overline{D}_{V'}(s')$ (avec coordonn\'ee~$T$) satisfasse la propri\'et\'e~$(\cB N_{Q(S)-T})$. 

Puisque~$\cH(b)$ n'est pas trivialement valu\'e, on peut trouver un polyn\^ome \`a coefficients dans~$\kappa(b)$ unitaire de degr\'e~$d$ et s\'eparable qui soit arbitrairement proche de~$P_{0}(b)$. Par cons\'equent, il existe un voisinage compact spectralement convexe~$V'$ de~$b$ dans~$V$, $Q' \in \cB(V')[S]$ unitaire de degr\'e~$d$ tel que $Q'(b)$ soit s\'eparable et $s_{0}, s_{1} \in \R_{>0}$ v\'erifiant $s_{0} < s < s_{1}$ tels que, pour tout $s' \in \intff{s_{0},s_{1}}$, on ait $\overline{D}_{V'}(Q',s') = \overline{D}_{V'}(P_{0},s')$. On peut \'egalement supposer que $f$ est nulle sur $\overline{D}_{V'}(Q',s_{0})$, que $\overline{D}_{V'}(Q',s_{1})$ soit contenu dans~$U$ et que $s_{1}\in I$. D'apr\`es la proposition~\ref{prop:CNBNPST}, ii), on peut supposer que $\overline{D}_{V'}(s_{0})$ et $\overline{D}_{V'}(s_{1})$ satisfont la condition $(\cB N_{Q'(S)-T})$. 

En utilisant le m\^eme argument que dans le cas pr\'ec\'edent, on montre que l'image de~$f$ dans $\overline{D}_{V'}(Q',s_{1})$ est nulle. Par cons\'equent, l'image de~$f$ dans $\overline{D}_{V'}(Q',s) = \overline{D}_{V'}(P_{0},s) = \overline{D}_{V'}(Q,s)$ est nulle. En r\'eutilisant le m\^eme argument, on montre que l'image de~$f$ dans $\overline{D}_{V}(Q,s) = \overline{D}_{V}(P_{0},s)$ est nulle. 
\end{proof}

Nous allons maintenant synth\'etiser les propositions \ref{prolongement_purement_localement_transcendant} et \ref{prolongement_rigide} dans le r\'esultat suivant.

\begin{prop}\label{prolongement}\index{Ideal@Id\'eal!B-fortement de type fini@$\cB$-fortement de type fini}
Soit~$b$ un point d\'ecent de~$B$.
Supposons que l'id\'eal nul de~$\cO_{B,b}$ est $\cB$-fortement de type fini. Alors, pour tout point~$x$ de~$\E{n}{\cA}$ au-dessus de~$b$, l'id\'eal nul de~$\cO_{\E{n}{\cA},x}$ est $\cB$-fortement de type fini.
\end{prop}
\begin{proof}
Soit~$x$ un point de~$\E{n}{\cA}$ au-dessus de~$b$. D'après la remarque~\ref{rem:rigeptrans}, quitte \`a permuter les coordonn\'ees, on peut supposer qu'il existe~$k\in \cn{0}{n}$ tel que la projection~$x_{k}$ de~$x$ sur les~$k$ premi\`eres soit purement localement transcendante au-dessus de~$b$ et tel que~$x$ soit rigide \'epais au-dessus de~$x_k$. Rappelons que, d'apr\`es la proposition~\ref{prop:typiqueAn}, si~$b$ est ultram\'etrique typique, alors tous les points au-dessus de~$b$ le sont \'egalement.

On montre tout d'abord que le point~$x_{k}$ satisfait les conclusions de l'\'enonc\'e \`a l'aide d'une r\'ecurrence utilisant la proposition~\ref{prolongement_purement_localement_transcendant}. On passe de~$x_{k}$ \`a~$x$ par une r\'ecurrence utilisant la proposition~\ref{prolongement_rigide}.
\end{proof}

Rappelons un r\'esultat technique qui nous sera utile.

\begin{lemm}[\protect{\cite[lemme~9.14]{EtudeLocale}}]\label{restriction_fibre_avd}\index{Anneau!fortement de valuation discrete@fortement de valuation discr\`ete}\index{Point!rigide epais@rigide \'epais}
Soit~$b$ un point d\'ecent de~$B$ tel que~$\Oc_{B,b}$ soit un anneau fortement de valuation discr\`ete d'uniformisante~$\varpi_{b}$. Soit $x\in \E{n}{\Ac}$ un point rigide \'epais au-dessus de~$b$. Alors, pour tout \'el\'ement non nul~$f$ de~$\Oc_{\E{n}{\Ac},x}$, il existe un unique couple $(v,g) \in \N \times \Oc_{\E{n}{\Ac},x}$ v\'erifiant les propri\'et\'es suivantes~:
\begin{enumerate}[i)]
\item $f = \varpi_{b}^v g$ dans~$\Oc_{\E{n}{\Ac},x}$~;
\item la restriction de~$g$ \`a~$\E{n}{\cH(b)}$ n'est pas nulle.
\end{enumerate} 
\qed
\end{lemm}

Adaptons maintenant ce r\'esultat (et sa preuve) au cas d'un corps fort.

\begin{lemm}\label{restriction_fibre}\index{Corps!fort}
Soit~$b$ un point d\'ecent de~$B$ tel que~$\Oc_{B,b}$ soit un corps fort. Soit $x\in \E{n}{\Ac}$ au-dessus de~$b$. Alors, pour tout \'el\'ement non nul~$f$ de~$\Oc_{\E{n}{\Ac},x}$, la restriction de~$f$  \`a~$\E{n}{\cH(b)}$ n'est pas nulle.
\end{lemm}
\begin{proof}
Notons $T_{1},\dotsc,T_{n}$ les coordonn\'ees sur~$\E{n}{\cA}$. D'après la remarque~\ref{rem:rigeptrans}, quitte \`a les permuter, on peut supposer que la projection~$x_{k}$ de~$x$ sur les~$k$ premi\`eres coordonn\'ees est purement localement transcendante au-dessus de~$b$ et que~$x$ est rigide \'epais au-dessus de~$x_k$. D\'emontrons l'\'enonc\'e par r\'ecurence sur l'entier~$n-k$.

Si~$n-k=0$, alors~$x$ est purement localement transcendant sur~$b$ et, d'apr\`es le th\'eor\`eme~\ref{rigide}, $\cO_{\bA^n_\cA,x}$ est un corps. Par cons\'equent, si~$f$ est non nul dans dans~$\Oc_{\E{n}{\Ac},x}$, alors $f(x)$ est non nul. Le r\'esultat s'ensuit.

Soit $l\in \N$ et supposons avoir d\'emontr\'e le r\'esultat pour $n-k = l$. Supposons que $n-k = l+1$. Notons~$x_{n-1}$  la projection de~$x$ sur les~$n-1$ premi\`eres coordonn\'ees. Par hypoth\`ese, $x$ est rigide \'epais au-dessus de~$x_{n-1}$. Consid\'erons son polyn\^ome minimal $\mu_{\kappa,x} \in \kappa(x_{n-1})[T_{n}]$ sur~$\kappa(x_{n-1})$ et choisissons $P \in \cO_{\E{n-1}{\cA},x_{n-1}}[T_n]$ un relev\'e unitaire de ce polyn\^ome. Notons $0_{x_{n-1}}$ le point de~$\E{n}{\cA}$ au-dessus de~$x_{n-1}$ d\'efini par $T_{n} = 0$. D'apr\`es le th\'eor\`eme~\ref{thm:isolemniscate} et le lemme~\ref{lem:polmin}, le morphisme naturel
\[ \cO_{\E{n}{\cA},0_{x_{n-1}}}[S]/(P(S)-T_n)\too \cO_{\E{n}{\cA},x}, \]
est un isomorphisme. Il suffit donc de d\'emontrer le r\'esultat pour~$0_{x_{n-1}}$. Or, d'apr\`es la proposition~\ref{prop:disqueglobal}, tout \'el\'ement de~$\cO_{\E{n}{\cA},0_{x_{n-1}}}$ poss\`ede un d\'eveloppement en s\'erie enti\`ere \`a coefficients dans~$\cO_{\E{n-1}{\cA},x_{n-1}}$. On conclut alors par l'hypoth\`ese de r\'ecurrence.
\end{proof}

\`A l'aide du th\'eor\`eme de division de Weierstra\ss{} avec contr\^ole sur les normes que nous avons d\'emontr\'e dans la partie pr\'ec\'edente (\cf~th\'eor\`eme~\ref{weierstrassam}), nous allons pouvoir montrer l'\'enonc\'e de fermeture des id\'eaux souhait\'e. Pour ce faire, nous allons nous placer dans un cadre l\'eg\`erement plus restrictif que celui des anneaux de base de la d\'efinition~\ref{def:basique}. 

\begin{defi}\index{Anneau!de base!geometrique@g\'eom\'etrique|textbf}
Soit~$(\cA,\nm)$ un anneau de Banach. On dit que~$\cA$ est un \emph{anneau de base g\'eom\'etrique}
si les conditions suivantes sont satisfaites~:
\begin{enumerate}[i)]
\item $\cA$ est un anneau de base tel que, pour tout $b\in B$, les \'el\'ements de~$\cV_{b}$ puissent \^etre choisis d'int\'erieur connexe~;
\item $\cM(\cA)$ satisfait le principe du prolongement analytique.
\end{enumerate}
\end{defi}

\begin{exem}\index{Corps!valu\'e}\index{Anneau!des entiers relatifs $\Z$}\index{Anneau!des entiers d'un corps de nombres}\index{Corps!hybride}\index{Anneau!de valuation discr\`ete}\index{Anneau!de Dedekind trivialement valu\'e}
Nos exemples usuels \ref{ex:corpsvalue} \`a~\ref{ex:Dedekind}~: les corps valu\'es, l'anneau~$\Z$ et les anneaux d'entiers de corps de nombres, les corps hybrides, les anneaux de valuation discr\`ete et les anneaux de Dedekind trivialement valu\'es sont tous des anneaux de base g\'eom\'etriques. 
\end{exem}

\begin{lemm}\label{lem:basegeometriqueideaux}\index{Ideal@Id\'eal!B-fortement de type fini@$\cB$-fortement de type fini}
Soit~$\cA$ un anneau de base g\'eom\'etrique. Alors, pour tout $b\in B$, tout id\'eal de~$\cO_{B,b}$ est~$\cB$-fortement de type fini relativement \`a~$\cV_{b}$.
\end{lemm}
\begin{proof}
Soit~$b\in B$. Supposons que $\cO_{B,b}$ est un corps fort. Les id\'eaux de~$\cO_{B,b}$ sont $\cO_{B,b}$ et~$(0)$. Par hypoth\`ese, il existe une base de voisinages compacts et spectralement convexes~$\cV_{b}$ de~$b$ tel que 0~engendre $\cB$-fortement l'id\'eal~$(0)$ relativement \`a~$\cV_{b}$. Il est imm\'ediat que 1~engendre $\cB$-fortement l'id\'eal~$\cO_{B,b}$ relativement \`a~$\cV_{b}$.

Supposons que $\cO_{B,b}$ est un anneau fortement de valuation discr\`ete. Fixons-en une uniformisante~$\varpi$. Les id\'eaux de~$\cO_{B,b}$ sont $\cO_{B,b}$, $(0)$ et les $(\varpi^n)$ pour $n\in \N^\ast$. Par hypoth\`ese, il existe une base de voisinages compacts, spectralement convexes et d'int\'erieur connexe~$\cV_{b}$ de~$b$ tel que $\varpi$~engendre $\cB$-fortement l'id\'eal~$(\varpi)$ relativement \`a~$\cV_{b}$. D'apr\`es le lemme~\ref{fortement_de_valuation_discr\`ete}, pour tout $n\ge 2$, $\varpi^n$~engendre $\cB$-fortement l'id\'eal~$(\varpi^n)$ relativement \`a~$\cV_{b}$. Comme pr\'ec\'edemment, il est \'evident que 1~engendre $\cB$-fortement l'id\'eal~$\cO_{B,b}$ relativement \`a~$\cV_{b}$. Il reste \`a traiter le cas de l'id\'eal nul. Puisque les \'el\'ements de~$\cV_{b}$ sont d'int\'erieur connexe, le fait que 0~engendre $\cB$-fortement l'id\'eal~$(0)$ relativement \`a~$\cV_{b}$ d\'ecoule du principe du prolongement analytique.
\end{proof}

Le r\'esultat qui suit g\'en\'eralise \cite[th\'eor\`eme~6.6.19]{A1Z} en dimension quelconque. La preuve que nous proposons est tr\`es largement inspir\'ee de celle de \cite[theorem II.D.2]{Gu-Ro}.

Pour $x \in \E{n}{\cA}$, $g=(g_1,\ldots,g_p)\in\cO_{x}^p$ et~$V$ voisinage compact de~$x$ dans~$\E{n}{\cA}$ sur lequel~$g$ est d\'efinie, on pose 
\[\|g\|_V := \max_{1\le i\le p} (\|g_{i}\|_{V}).\]

\begin{theo}\label{fermeture}\index{Ideal@Id\'eal!ferme@ferm\'e}\index{Sous-module|see{Id\'eal}}
Supposons que~$\cA$ est un anneau de base g\'eom\'etrique. Soient~$x$ un point de~$\E{n}{\cA}$, $p \ge 1$ un entier et $M$~un sous-module de~$\cO^p_{x}$ engendr\'e par des \'el\'ements $f_1,\dotsc,f_l$.

Alors, pour tout voisinage compact~$V$ de~$x$ dans~$\E{n}{\Ac}$, il existe un voisinage compact spectralement convexe~$V'$ de~$x$ dans~$V$ sur lequel les~$f_{i}$ sont $\cB$-d\'efinis et une constante $K_{V',V} \in \R_{>0}$ v\'erifiant la propri\'et\'e suivante~: pour tout \'el\'ement~$f$ de~$\cB(V)^p$ dont l'image dans~$\Oc_{x}^p$ appartient \`a~$M$, il existe~$a_1,\dotsc,a_{l}\in\cB(V')$ tels que
\begin{enumerate}[i)]
\item $\disp f = \sum_{i=1}^l a_i f_i$ dans~$\cB(V')^p$ ;
\item  pour tout $i\in \cn{1}{l}$, on a $\|a_i\|_{V'}\leq K_{V',V}\,\|f\|_V$.
\end{enumerate}
\end{theo}
\begin{proof}
Remarquons tout d'abord que que si l'\'enonc\'e vaut pour l'entier $p=1$ et pour un entier $p = p_{0}\ge 1$ donn\'e, alors il vaut encore pour $p=p_{0}+1$. En effet, consid\'erons un sous-module~$M$ de~$\cO^{p_{0}+1}_{x}$  engendr\'e par des \'el\'ements $f_1,\ldots,f_l\in \cO^{p_{0}+1}_{x}$. Notons $f'_1,\ldots,f'_l \in \Oc_{x}^{p_{0}}$ les projections respectives de $f_1,\ldots,f_l$ sur les $p_{0}$~premiers facteurs et~$M'$ le sous-module de~$\cO^{p_{0}}_{x}$ engendr\'e par $f'_{1},\dotsc,f'_{l}$. Le noyau $M''$ du morphisme $M \to M'$ s'identifie alors \`a un sous-module de~$\Oc_{x}$. Le r\'esultat pour le module~$M$ d\'ecoule ais\'ement de celui pour~$M'$ et~$M''$.

Une r\'ecurrence imm\'ediate montre qu'il suffit de traiter le cas o\`u $p=1$, cadre que l'on adopte d\'esormais.

Notons~$b$ la projection de~$x$ sur $B = \cM(\cA)$. D'après la remarque~\ref{rem:rigeptrans}, quitte \`a permuter les coordonn\'ees sur~$\E{n}{\cA}$, on peut supposer qu'il existe~$k\in\cn{0}{n}$ tel que la projection~$x_k$ de~$x$ sur les~$k$ premi\`eres coordonn\'ees soit purement localement transcendante au-dessus de~$b$ et que~$x$ soit rigide \'epais au-dessus de~$x_k$. D\'emontrons le r\'esultat par r\'ecurence sur l'entier~$n-k$. 
 
\medbreak
 
\noindent$\bullet$ \textit{Initialisation de la r\'ecurrence~: $n-k=0$. }

Dans ce cas, d'apr\`es le th\'eor\`eme~\ref{rigide}, l'anneau local~$\cO_{x}$ est un corps fort ou un anneau fortement de valuation discr\`ete. Soit~$V$ un voisinage compact de~$x$ dans~$\E{n}{\cA}$.

Supposons tout d'abord que~$\cO_{X,x}$ est un corps fort. Si tous les~$f_{i}$ sont nuls, le r\'esultat d\'ecoule de la d\'efinition de corps fort. Sinon, nous pouvons supposer que $f_{1} \ne 0$. Il existe alors un voisinage spectralement convexe~$V'$ de~$x$ dans~$V$ sur lequel $f_{1}$ et~$f_{1}^{-1}$ sont $\cB$-d\'efinies. Pour tout $f\in \cB(V)$, nous avons alors $f = (f f_{1}^{-1}) \, f_{1}$ dans $\cB(V')$ et $\|f f_{1}^{-1}\|_{V'} \le \|f_{1}^{-1}\|_{V'} \, \|f\|_{V}$, ce qui d\'emontre l'\'enonc\'e.

Supposons maintenant que~$\cO_{x}$ est un anneau fortement de valuation discr\`ete. Si $M=0$, le r\'esultat d\'ecoule de la proposition~ \ref{prolongement}. Si $M = \Oc_{x}$, l'un des~$f_{i}$ est inversible et le r\'esultat s'obtient comme dans le cas o\`u $\cO_{x}$~est un corps fort. Sinon, on peut se ramener au cas o\`u $l=1$ et $f_{1}$ est une puissance non triviale d'une uniformisante de~$\cO_{x}$. Le r\'esultat se d\'eduit alors du lemme~\ref{fortement_de_valuation_discr\`ete}.  
 
\medbreak

\noindent$\bullet$ \textit{\'Etape de r\'ecurrence.}

Supposons que le r\'esultat soit satisfait lorsque $n-k=N \ge 0$ (pour $p=1$ et donc pour tout $p\in \N^\ast$ d'apr\`es la remarque pr\'eliminaire) et d\'emontrons-le dans le cas o\`u $n-k=N+1$. 

\medbreak

\noindent$-$ \textit{Cas o\`u la restriction de~$f_1$ \`a~$\E{n-k}{\sH(x_k)}$ n'est pas nulle}

Soient~$x_{n-1}$ la projection de~$x$ sur les~$n-1$ premi\`eres coordonn\'ees. Puisque $n-k\ge 1$, le point~$x$ est rigide \'epais au-dessus de~$x_{n-1}$. Notons~$P\in\kappa(x_{n-1})[T]$ son polyn\^ome minimal. Soit~$W$ un voisinage spectralement convexe de~$x_{n-1}$ dans~$\E{n-1}{\cA}$ sur lequel les coefficients de~$P$ sont $\cB$-d\'efinis. Puisque l'\'enonc\'e est local, on peut se placer dans~$\E{1}{\cB(W)}$ pour le d\'emontrer.

Soit~$\varphi_P \colon \E{1}{\cB(W)}  \to \E{1}{\cB(W)}$ le morphisme induit par~$P$. En notant $0_{x_{n-1}}$ le point de~$\E{1}{\cB(W)}$ appartenant \`a la fibre au-dessus de~$x_{n-1}$ et \'egal \`a~0 dans cette fibre, on a $\varphi_{P}^{-1}(0_{x_{n-1}}) = \{x\}$. Le corollaire~\ref{cor:phinormes} permet de ramener la d\'emonstration du r\'esultat pour le point~$x$ (pour $p=1$) \`a celui pour le point~$0_{x_{n-1}}$ (pour $p$ \'egal au degr\'e de~$P$, et donc pour $p=1$ d'apr\`es la remarque pr\'eliminaire). On peut donc supposer que $x= 0_{x_{n-1}}$.

Par hypoth\`ese, la restriction de~$f_1$ \`a~$\E{n-k}{\sH(x_k)}$ n'est pas nulle. Le lemme~\ref{changement_variable} assure que, quitte \`a effectuer un changement de variables, on peut supposer que la restriction de~$f_1$ \`a~$\E{n-1}{\sH(x_k)}$ n'est pas nulle. Notons~$T$ la derni\`ere coordonn\'ee sur~$\E{n}{\cA}$. Alors $\Oc_{\E{1}{\sH(x_{n-1})},x}$ est un anneau de valuation discr\`ete d'uniformisante~$T$. Notons~$v$ la valuation de~$f_{1}$.

D'apr\`es le th\'eor\`eme de division de Weierstra\ss{} \ref{weierstrassam} (la version sans normes \cite[th\'eor\`eme 8.3]{EtudeLocale} suffirait), tout \'el\'ement~$f$ de~$\Oc_{x}$ peut s'\'ecrire de fa\c con unique sous la forme $f = f' f_{1} + \tilde f$ avec $f' \in \Oc_{x}$ et $\tilde f \in \Oc_{x_{n-1}}[T]$ de degr\'e inf\'erieur \`a~$v-1$.

Consid\'erons le $\Oc_{x_{n-1}}$-module $\tilde M :=  \cO_{x_{n-1}}[T]_{\leq v-1}\cap M$, autrement dit le module des restes de~$M$ dans la division de Weierstra\ss{} par~$f_{1}$. 
Par noetherianit\'e, $\tilde M$ est de type fini. Fixons des g\'en\'erateurs $\ti g_{1},\dotsc, \ti g_{m}$. Pour tout $i\in\cn{0}{m}$, il existe $a_{i,1},\dotsc,a_{i,l} \in \Oc_{x}$ tels que $\ti g_{i} = \sum_{j=1}^l a_{i,j}f_j$ dans~$\Oc_{x}$.

Soit~$V$ un voisinage compact de~$x$ dans~$\E{n}{\cA}$. D'apr\`es le th\'eor\`eme de division de Weierstra\ss{} \ref{weierstrassam}, il existe un voisinage compact~$W$ de~$x_{n-1}$ dans~$\E{n-1}{\cA}$, un polyn\^ome $P \in \cB(W)[T]$ et des nombres r\'eel $s,K \in \R_{>0}$ tels que $\overline{D}_{W}(P;s)$ soit contenu dans~$V$ et, pour tout \'el\'ement~$f$ de~$\cB(V)$ les propri\'et\'es suivantes soient satisfaites~: $f'$ est $\cB$-d\'efini sur $\overline{D}_{W}(P;s)$, les coefficients de $\ti f$ sont $\cB$-d\'efinis sur~$W$ et on a
\[ \|f'\|_{\overline{D}_{W}(P;s)} \le K\, \|f\|_{V} \textrm{ et } \|\ti f\|_{\overline{D}_{W}(P;s)} \le K\, \|f\|_{V}.\]
Puisque le point~$x$ est le point~$0$ au-dessus de~$x_{n-1}$, il poss\`ede une base de voisinages de la forme $\overline{D}_{W}(T;s)$ et on peut donc supposer que~$P = T$. On peut en outre supposer que tous les~$a_{i,j}$ sont $\Bc$-d\'efinis sur $\overline{D}_{W}(T;s)$.

Pour toute partie compacte~$U$ de~$\E{n}{\cA}$ et tout nombre r\'eel $t\in \R_{>0}$, consid\'erons l'isomorphisme
\[\fonction{\psi}{\cB(U)^{v}}{\cB(U)[T]_{\le v-1}}{(h_{0},\dotsc,h_{v-1})}{\disp\sum_{i=0}^{v-1} h_{i} T^i}.\]
Pour tous $h\in \cB(U)^{v}$ et $k\in \cB(U)[T]_{\le v-1}$, on a
\[\|\psi(h)\|_{\overline{D}_{U}(T;t)} \le v \max(1,t^{v-1}) \|h\|_{U}\] 
et
\[\|\psi^{-1}(k)\|_{U} \le \max(1,t^{1-v}) \, \|k\|_{\overline{D}_{U}(T;t)}.\] 
La premi\`ere in\'egalit\'e est imm\'ediate. Pour la seconde, on se ram\`ene imm\'ediatement \`a une situation un corps valu\'e complet, auquel cas elle est cons\'equence de la description explicite des normes des disques dans le cas ultram\'etrique ou de la formule de Cauchy dans le cas archim\'edien (\cf~\cite[lemme~2.1.2]{A1Z}).

L'hypoth\`ese de r\'ecurrence appliqu\'ee au module~$\ti M$, que l'on identifie \`a un sous-module de~$\Oc_{x_{n-1}}^{v}$ \textit{via} la restriction de~$\psi$, et \`a~$W$ fournit un voisinage~$W'$ de~$x_{n-1}$ dans~$\E{n-1}{\cA}$ et une constante~$K'$.

On peut maintenant d\'emontrer le r\'esultat. Soit $f \in \cB(V)$ dont l'image dans~$\Oc_{x}$ appartient \`a~$M$. On peut l'\'ecrire sous la forme $f = f' f_{1} + \sum_{j=1}^m b_{j} \ti g_{j}$ avec $f' \in \cB(\overline{D}_{W}(T;s))$, $b_{1},\dotsc,b_{m} \in \cB(W')$, 
\[\|f'\|_{\overline{D}_{W}(T;s)} \le K \|f\|_{V}\]
et, pour tout $j\in \cn{1}{m}$, 
\[\|b_{j}\|_{W'} \le K \max(1,t^{1-v}) \|f\|_{V}.\]
On conclut alors en r\'e\'ecrivant~$f$ sous la forme
\[f = \big( f' + \sum_{j=1}^m b_{j} a_{i,j}\big) f_{1} + \sum_{i=2}^l \big(\sum_{j=1}^m b_{j} a_{i,j}\big) f_{i}\]
et en choisissant $V' := \overline{D}_{W'}(T;s)$.

\medbreak

\noindent~$-$ \textit{Cas g\'en\'eral}

Si~$M=0$, le r\'esultat d\'ecoule de la proposition~\ref{prolongement}. On peut donc supposer que $M\ne 0$, et m\^eme que tous les~$f_{i}$ sont non nuls.

Si~$\cO_{b}$ est un corps fort, alors le lemme~\ref{restriction_fibre} assure que la restriction de~$f_1$ \`a~$\E{n-k}{\sH(x_k)}$ n'est pas nulle et on se ram\`ene au cas pr\'ec\'edent. On peut donc supposer que~$\Oc_{b}$ est un anneau fortement de valuation discr\`ete. Soit~$\varpi_{b}$ une uniformisante de~$\cO_{b}$.

D'apr\`es le lemme~\ref{restriction_fibre_avd}, pour tout $i\in \cn{1}{l}$, il existe $v_{i}\in \N$ et $g_{i} \in \Oc_{x}$ dont la restriction \`a~$\E{n-k}{\sH(x_k)}$ n'est pas nulle tels que $f_{i} = \varpi_{b}^{v_{i}} g_{i}$ dans~$\Oc_{x}$. Quitte \`a permuter les~$f_{i}$, on peut supposer que~$v_{1}$ est le minimum des~$v_{i}$. Le lemme~\ref{division_anneau_valuation_discr\`ete} permet de d\'eduire le r\'esultat pour l'id\'eal~$M$ de~$\Oc_{x}$ engendr\'e par les~$f_{i}$ du r\'esultat pour l'id\'eal engendr\'e par les $\pi_{b}^{-v_{1}}f_{i}$. Pour ce dernier, la restriction de la fonction $\pi_{b}^{-v_{1}}f_{1} = g_{1}$ \`a~$\E{n-k}{\sH(x_k)}$ n'est pas nulle et on se ram\`ene de nouveau au cas pr\'ec\'edent.
\end{proof}

\begin{coro}\label{limite}\index{Ideal@Id\'eal!ferme@ferm\'e}
Supposons que~$\cA$ est un anneau de base g\'eom\'etrique. Soient~$x$ un point de~$\E{n}{\cA}$, $p\ge 1$ un entier et $M$~un sous-module de~$\cO^p_{x}$. Soient $V$ un voisinage compact de~$x$ dans~$\E{n}{\cA}$ et $(g_n)_{n\in\N}$ une suite d'\'el\'ements de~$\cB(V)^p$ qui converge vers un \'el\'ement~$g$ de~$\cB(V)^p$. Si, pour tout $n\in \N$, l'image de~$g_{n}$ dans~$\cO_{x}^p$ appartient \`a~$M$, alors l'image de~$g$ dans~$\cO_{x}^p$ appartient \`a~$M$. 
\end{coro}
\begin{proof}
D'apr\`es le th\'eor\`eme~\ref{rigide}, $\cO_{x}$ est noeth\'erien, donc $M$~est de type fini. Soient $f_{1},\dotsc,f_{l}$ des g\'en\'erateurs de~$M$. D'apr\`es le th\'eor\`eme~\ref{fermeture}, il existe un voisinage compact~$V'$ de~$x$ dans~$V$ sur lequel les~$f_{i}$ sont $\cB$-d\'efinis et une constante $K \in \R_{>0}$ v\'erifiant la propri\'et\'e suivante~: pour tout \'el\'ement~$f$ de~$\cB(V)^p$ dont l'image dans~$\Oc_{x}^p$ appartient \`a~$M$, il existe~$a_1,\dotsc,a_{l}\in\cB(V')$ tels que
\begin{enumerate}[i)]
\item $\disp f = \sum_{i=1}^l a_i f_i$ dans~$\cB(V')^p$ ;
\item  pour tout $i\in \cn{1}{l}$, on a $\|a_i\|_{V'}\leq K\,\|f\|_V$.
\end{enumerate}
Soit~$(g_{n_k})_{k\in\N}$ une sous-suite de~$(g_n)_{n\in\N}$ telle que, pour tout~$k\in\N$, on ait l'in\'egalit\'e~$\|g_{n_{k+1}}-g_{n_k}\|_V\leq 1/2^k$. Pour tout $k\in \N$, il existe $a_{1,k},\dotsc,a_{l,k}\in\cB(V')$ tels que
\begin{enumerate}[i)]
\item $\disp g_{n_{k+1}} - g_{n_{k}} = \sum_{i=1}^l a_{i,k} f_i$ dans~$\cB(V')^p$ ;
\item  pour tout $i\in \cn{1}{l}$, on a $\|a_{i,k}\|_{V'}\leq \frac{K}{2^k}$.
\end{enumerate}
Pour tout $i \in \cn{1}{l}$, la suite $(\sum_{j=0}^k a_{i,j})_{k\in \N}$ est de Cauchy, et donc convergente. Notons~$s_{i} \in \cB(V')$ sa limite. On a alors 
\[g = g_{n_{0}} +    \sum_{i=1}^l s_{i} f_i \textrm{ dans } \cB(V')^p.\]
Le r\'esultat s'ensuit.
\end{proof}

\begin{rema}
La pr\'esence des anneaux $\cB(V)$ dans l'\'enonc\'e du corollaire~\ref{limite} n'est gu\`ere heureuse et peut le rendre difficile \`a utiliser en pratique. Au chapitre~\ref{chap:Stein}, gr\^ace \`a la th\'eorie des espaces de Stein, nous montrerons que le r\'esultat reste valable en rempla\c cant les~$\cB(V)$ par des~$\cO(U)$, \cf~corollaire~\ref{cor:limiteOU}.
\end{rema}
\chapter[Cat\'egorie des espaces analytiques~: propri\'et\'es]{Cat\'egorie des espaces analytiques~: propri\'et\'es}\label{catan}

Ce chapitre est consacr\'e \`a l'\'etude de la cat\'egorie des espaces analytiques sur un anneau de Banach. Fixons un anneau de Banach $(\cA,\nm)$. 

Dans la section~\ref{sec:analytification}, nous construisons un foncteur d'analytification de la cat\'egorie des sch\'emas localement de pr\'esentation finie sur~$\cA$ vers celle des espaces $\cA$-analytiques. L'ingr\'edient principal est l'existence d'une correspondance bijective entre les morphismes d'un espace $\cA$-analytique vers l'espace affine analytique~$\E{n}{\cA}$ et les $n$-uplets de fonctions globales sur~$X$. La d\'emonstration de ce r\'esultat repose de fa\c{c}on cruciale sur le th\'eor\`eme de fermeture des id\'eaux~\ref{fermeture}. Aussi devrons-nous souvent, dans cette section et tout au long du chapitre, supposer que~$\cA$ est un anneau de base g\'eom\'etrique.

Dans la section~\ref{sec:extensionB}, \'etant donn\'e un anneau de Banach~$\cB$ et un morphisme born\'e $\cA \to \cB$, nous construisons un foncteur~$\wc \ho{\cA} \cB$ d'extension des scalaires, d\'efini de la cat\'egorie des espaces $\cA$-analytiques vers celle des espaces $\cB$-analytiques. Dans la section~\ref{sec:produitsfibres}, nous d\'emontrons que la cat\'egorie des espaces $\cA$-analytiques admet des produits fibr\'es.

Nous d\'edions deux sections \`a des propri\'et\'es importantes des morphismes~: propret\'e dans la section~\ref{sec:propre} et s\'eparation dans la section~\ref{sec:morphismessepares}.

La section~\ref{sec:paspconv} est consacr\'ee aux parties spectralement convexes et, plus pr\'ecis\'ement, \`a la description de l'extension des scalaires d'une partie spectralement convexe et du produit fibr\'e de deux parties spectralement convexes. 
Nous expliquons \'egalement comment identifier une fibre d'un morphisme \`a un produit fibr\'e.


\section{Analytification de sch\'emas}
\label{exemple_esp}\label{sec:analytification}
\index{Analytification!|(}

La th\'eorie des sch\'emas peut fournir de nombreux exemples d'espaces analytiques, \emph{via} un foncteur d'analytification. Afin de construire ce foncteur, nous aurons besoin de la description attendue de l'ensemble des morphismes d'un espace analytique quelconque vers un espace affine analytique.

\begin{prop}\label{morphsec}\index{Morphisme analytique!vers un espace affine}
Soit~$n\in \N$. Soit~$f \colon \cB\to\cB'$ un morphisme born\'e de~$\cA$-alg\`ebres de Banach, o\`u $\cB'$ est un anneau de base g\'eom\'etrique. Soit~$X$ un espace $\cB'$-analytique. L'application 
\[\fonction{s_{X}}{\Hom_{\An_{\cA},f}(X,\E{n}{\cB})}{\Gamma(X,\cO_{X})^n}{\varphi}{(\varphi^\sharp(T_1),\dotsc,\varphi^\sharp(T_n))},\]
o\`u $T_{1},\dotsc,T_{n}$ sont les coordonn\'ees, sur~$\E{n}{\cB}$ est bijective.

En outre, la famille $(s_{X})_{X \in \Ob(\cB'-\An)}$ d\'efinit une transformation naturelle entre les foncteurs $\Hom_{\An_\cA,f}(\wc,\E{n}{\cB})$ et $\Gamma(\wc,\cO)^n$. 
\end{prop}
\begin{proof}
Commen\c cons par v\'erifier que la famille $s := (s_{X})$ d\'efinit bien une transformation naturelle. En effet, pour tout morphisme d'espaces $\cB'$-analytiques $\psi \colon Y \to X$, on a
\[s_{Y}(\varphi\circ\psi)=(\psi^\sharp(\varphi^\sharp(T_1)),\ldots,\psi^\sharp(\varphi^\sharp(T_n)))=\psi^\sharp(s_{X}(\varphi)).\]

\medbreak

Soit $X$ un espace $\cB'$-analytique. Montrons que $s_{X}$ est bijective.

$\bullet$ \textit{Injectivit\'e :} Soient $\varphi \colon X \to \E{n}{\cB}$ et $\varphi' \colon X \to \E{n}{\cB}$ des morphismes d'espaces analytiques au-dessus de~$f$ tels que l'on ait 
\[ (\varphi^\sharp(T_1),\ldots ,\varphi^\sharp(T_n))=(\varphi'^\sharp(T_1),....\varphi'^\sharp(T_n)). \]
Soit $x\in X$. Commen\c cons par montrer que~$\varphi(x)=\varphi'(x)$. Soit $P \in \cB[T_{1},\dotsc,T_{n}]$. Notons $f(P) \in \cB'[T_{1},\dotsc,T_{n}]$ le polyn\^ome obtenu en appliquant~$f$ \`a tous les coefficients de~$P$. D'apr\`es la proposition~\ref{prop:propmorphismeaudessusdef}, on a 
\begin{align*}
|P(T_{1}(\varphi(x)),\dotsc,T_{n}(\varphi(x)))| &= |f(P)(\varphi^\sharp(T_{1})(x),\dotsc,\varphi^\sharp(T_{n})(x))|\\
& = |f(P)(\varphi'^\sharp(T_{1}),\dotsc,\varphi'^\sharp(T_{n}))(x)|\\
& =|P(T_{1}(\varphi'(x)),\dotsc,T_{n}(\varphi'(x)))|. 
\end{align*} 
On en d\'eduit que $\varphi(x) = \varphi'(x)$.

Montrons, \`a pr\'esent, que, pour tout $y \in \E{n}{\cB}$, les morphismes $\varphi^\sharp \colon \cO_{\E{n}{\cB},y}\to\varphi_\ast(\cO_X)_y$ et $\varphi'^\sharp:\cO_{\E{n}{\cB},y}\to \varphi'_\ast(\cO_X)_y = \varphi_\ast(\cO_X)_y$ sont \'egaux. 

Soit $y \in \E{n}{\cB}$. Soient~$F\in \cO_{\E{n}{\cB},y}$ et~$V$ un voisinage compact de~$y$ dans~$\E{n}{\cB}$ sur lequel~$F$ est d\'efinie. Quitte \`a r\'eduire~$V$, on peut supposer que~$F$ est limite uniforme sur~$V$ d'une suite de fractions rationnelles~$\left(\frac{P_i}{Q_i}\right)_{i\in \N}$ sans p\^oles sur~$V$. 

Soit~$x\in\varphi^{-1}(\mathring V)$. Montrons que $\varphi^\sharp(F) = \varphi'^\sharp(F)$ dans~$\cO_{X,x}$. Par d\'efinition d'espace $\cB'$-analytique, $x$ poss\`ede un voisinage~$Z$ qui est un ferm\'e analytique d'un ouvert~$U$ d'un espace affine analytique sur~$\cB'$. Notons~$\cI$ le faisceau d'id\'eaux coh\'erent sur~$U$ d\'efinissant~$Z$ et $j \colon Z \to U$ l'immersion ferm\'ee correspondante.

La d\'efinition de morphisme assure qu'il existe un voisinage~$U'$ de~$j(x)$ dans~$U$ et un morphisme $\tilde\varphi \colon U' \to \E{n}{\cB}$ au-dessus de~$f$ tels que le diagramme
\[\begin{tikzcd} 
U' \arrow[r, "\tilde\varphi"] & \E{n}{\cB}\\
j^{-1}(U') \arrow[u, "j"] \arrow[ru, "\varphi_{|j^{-1}(U')}"']&
\end{tikzcd}\]
commute. Quitte \`a restreindre~$U'$, on peut supposer qu'il existe \'egalement un morphisme $\tilde\varphi' \colon U' \to \E{n}{\cB}$ au-dessus de~$f$ tel que le diagramme
\[\begin{tikzcd} 
U' \arrow[r, "\tilde\varphi'"] & \E{n}{\cB}\\
j^{-1}(U') \arrow[u, "j"] \arrow[ru, "\varphi'_{|j^{-1}(U')}"']&
\end{tikzcd}\]
commute. 

Par la d\'efinition de morphisme entre ouverts d'espaces affines, il existe un voisinage compact~$U''$ de~$x$ dans~$U'$ tel que $\tilde\varphi(U'') \subset V$ et, pour tout $G \in \cO_{\E{n}{\cB}}(V)$, $\|\tilde{\varphi}^\sharp(G)\|_{U''}\leq\|G\|_{V}$. On en d\'eduit que
\[\lim_{i \to +\infty} \left\|\tilde{\varphi}^\sharp(F) - \tilde{\varphi}^\sharp\left(\frac{P_i}{Q_i}\right) \right\|_{U''} = 0.\]
Quitte \`a restreindre $U''$, on peut supposer que le m\^eme r\'esultat vaut pour~$\tilde\varphi'$, et donc que 
\[\|\tilde{\varphi}^\sharp(F) - \tilde{\varphi}'^\sharp(F)\|_{U''} = 0.\]

Par hypoth\`ese, pour tout $k \in \cn{1}{n}$, $\tilde{\varphi}^\sharp(T_k)-\tilde{\varphi}'^\sharp(T_k)$ appartient \`a~$\cI(U'')$. On en d\'eduit que, pour~$i\in\N$, $\tilde{\varphi}^\sharp\left(\frac{P_i}{Q_i}\right)-\tilde{\varphi}'^\sharp\left(\frac{P_i}{Q_i}\right)$ appartient \`a~$\cI(U'')$. Le lemme~\ref{lem:morphismeB} et le corollaire~\ref{limite} assurent alors que l'\'el\'ement $\tilde{\varphi}^\sharp(F)-\tilde{\varphi}'^\sharp(F)$ de $\cO_{U'',j(x)}$ appartient \`a~$\cI_{j(x)}$ et donc que $\varphi^\sharp(F)=\varphi'^\sharp(F)$ dans $\cO_{X,x}$.

On a d\'emontr\'e que $\varphi^\sharp=\varphi'^\sharp$, et donc que $\varphi=\varphi'$ .

\medbreak

$\bullet$ \textit{Surjectivit\'e :} Soient $F_1,\ldots,F_n\in\cO_X(X)$. Puisque l'application~$s_{X}$ est injective, il suffit de v\'erifier que tout point~$x$ de~$X$ poss\`ede un voisinage ouvert~$Z$ tel que $(F_{1},\dotsc,F_{n})$ appartienne \`a l'image de~$s_{Z}$.

Soit $x$ un point de~$X$. Il poss\`ede un voisinage~$Z$ qui est un ferm\'e analytique d'un ouvert~$U$ d'un espace affine analytique sur~$\cB'$. Notons $j \colon Z \to U$ l'immersion ferm\'ee correspondante. Quitte \`a restreindre~$Z$ et~$U$, on peut supposer que $F_{1},\dotsc,F_{n} \in \cO_{Z}(Z)$ appartiennent \`a l'image du morphisme $j^\sharp \colon \cO_U(U)\to\cO_Z(Z)$. Choisissons des relev\'es $F'_1,\ldots,F'_n \in\cO_{U}(U)$ de ces \'el\'ements.

D'apr\`es l'exemple~\ref{exemple_morphisme}, il existe un morphisme d'espaces $\cB'$-analytiques $\psi \colon U \to \E{n}{\cB'}$ tel que, pour tout $i \in \cn{1}{n}$, on ait $\psi^\sharp(T_{i}) = F'_{i}$. Posons $\varphi := \tilde f_{n} \circ \psi \circ j$, o\`u $\tilde f_{n} \colon \E{n}{\cB'} \to \E{n}{\cB}$ est le morphisme d\'efini dans l'exemple~\ref{ex:tildefn}. C'est un morphisme analytique au-dessus de~$f$ qui satisfait $s_{Z}(\varphi) = (F_{1},\dotsc,F_{n})$.
\end{proof}

\begin{coro}\label{cor:factorisationouvert}
Soit~$f \colon \cB\to\cB'$ un morphisme born\'e de~$\cA$-alg\`ebres de Banach, o\`u $\cB'$ est un anneau de base g\'eom\'etrique. Soient $P_{1},\dotsc,P_{k} \in \cB[T_1,\ldots,T_n]$ et $r_{1},s_{1},\dotsc,r_{k},s_{k} \in \R$. Posons 
\[U:=\bigcap_{i=1}^k\{y\in\E{n}{\cB} : r_i<|P_i(y)|<s_i\}.\]

Soit~$X$ un espace $\cB'$-analytique. Alors, l'application~$s_{X}$ induit une bijection entre $\Hom_{\An_{\cA},f}(X,U)$ et l'ensemble 
\[\{(f_{1},\dotsc,f_{n}) \in \cO(X)^n : 
\forall x\in X, \forall i\in\cn{1}{k},\  r_i<|P_i(f_1,\ldots,f_n)(x)|<s_i.\}\] 
\end{coro}
\begin{proof}
Le r\'esultat provient du fait que, pour tout $x\in X$ et tout $i\in \cn{1}{k}$, on a 
\begin{align*}
|P_{i}(\varphi(x))| & = |P_{i}(T_{1}(\varphi(x)),\dotsc,T_{n}(\varphi(x)))|\\
& = |P_{i}(\varphi^\sharp(T_{1})(x),\dotsc,\varphi^\sharp(T_{n})(x))|.
\end{align*}
\end{proof}

Il d\'ecoule de la proposition~\ref{morphsec} que, sous des hypoth\`eses convenables, l'espace affine analytique~$\E{n}{\cA}$ repr\'esente le foncteur attendu.

\begin{coro}\label{cor:foncteurAn}\index{Espace affine analytique!foncteur repr\'esent\'e}\index{Espace affine analytique!morphisme|see{Morphisme analytique vers un espace affine}}
Supposons que~$\cA$ est un anneau de base g\'eom\'etrique. Pour tout $n\in \N$, l'espace~$\E{n}{\cA}$ repr\'esente le foncteur
\[\fonctionsp{\cAAn}{\Ens}{X}{\Gamma(X,\cO_{X})^n}.\]
\qed
\end{coro}

Nous pouvons maintenant construire un foncteur d'analytification. Nous noterons $\Hom_{\cA-\loc}(\wc,\wc)$ l'ensemble des morphismes dans la cat\'egorie des espaces localement $\cA$-annel\'es.%
\nomenclature[Kcz]{$\Hom_{\cA-\loc}(X,Y)$}{pour $X,Y$ espaces localement $\cA$-annel\'es, ensemble des morphismes d'espaces localement $\cA$-annel\'es de~$X$ dans~$Y$}

\begin{theo}\label{thm:analytification}\index{Analytification}
Supposons que~$\cA$ est un anneau de base g\'eom\'etrique.
Soit~$\cX$ un sch\'ema localement de pr\'esentation finie sur~$\cA$. Le foncteur
\[\fonction{\Phi_{\cX}}{\cAAn}{\Ens}{Y}{\Hom_{\cA-\loc}(Y,\cX)}\]
est repr\'esentable.
\end{theo}
\begin{proof}
Nous allons proc\'eder en plusieurs \'etapes.

\medbreak

$\bullet$ \textit{$\cX$ est un espace affine.}

Il existe $n\in \N$ tel que $\cX = \A^n_{\cA}$. Notons~$T_{1},\dotsc,T_{n}$ les coordonn\'ees sur~$\A^n_{\cA}$. Pour tout espace $\cA$-analytique~$Y$, l'application 
\[\begin{array}{ccc}
\Hom_{\cA-\loc}(Y,\A^n_{\cA}) & \to & \cO(Y)^n\\
\varphi & \mapsto & (\varphi^\sharp(T_{1}),\dotsc,\varphi^\sharp(T_{n}))
\end{array}\]
est alors une bijection (\cf~\cite[proposition~1.6.3]{EGAInew}). La proposition~\ref{morphsec} entra\^ine alors que~$\Phi_{\A^n_{\cA}}$ est repr\'esent\'e par~$\E{n}{\cA}$. 
 
 \medbreak

$\bullet$ \textit{$\cX$ est un sch\'ema affine de pr\'esentation finie sur~$\cA$.}

Il existe $n\in \N$ et un id\'eal de type fini~$I$ de $\cO(\A^n_{\cA})$ tel que~$\cX$ soit le sous-sch\'ema ferm\'e de~$\A^n_{\cA}$ d\'efinir par~$I$. L'id\'eal~$I$ engendre un faisceau d'id\'eaux coh\'erent~$\cI$ de~$\E{n}{\cA}$. Notons~$X$ le ferm\'e analytique de~$\E{n}{\cA}$ d\'efini par~$\cI$. Le point pr\'ec\'edent et la proposition~\ref{immersion_ferm\'e} assurent que le foncteur~$\Phi_{\cX}$ est repr\'esent\'e par~$X$.

\medbreak
 
$\bullet$ \textit{$\cX$ est un sch\'ema localement de pr\'esentation finie sur~$\cA$.}

Le sch\'ema~$\cX$ admet un recouvrement ouvert~$\{\cX_i\}_{i\in I}$ par des sch\'emas affines de pr\'esentation finie sur~$\cA$. Pour chaque $i\in I$, consid\'erons l'espace $\cA$-analytique~$X_{i}$ et le morphisme $\rho_{i} \colon X_{i} \to \cX_{i}$ associ\'es \`a~$\cX_{i}$ par la construction du point pr\'ec\'edent.

Pour tout~$i\in I$ et tout ouvert~$U$ de~$\cX_{i}$, on v\'erifie que $\rho_{i}^{-1}(U)$ repr\'esente le foncteur~$\Phi_{U}$. On d\'eduit de cette propri\'et\'e que les~$X_{i}$ se recollent. En effet, on peut identifier les intersections puisqu'elles repr\'esentent un m\^eme foncteur.

On v\'erifie que l'espace $\cA$-analytique~$X$ obtenu en recollant les~$X_{i}$ repr\'esente~$\cX$.
\end{proof}

\begin{defi}\label{def:analytification}\index{Analytification!d'un schema@d'un sch\'ema|textbf}\index{Analytification!d'un morphisme|textbf}\index{Morphisme analytique!analytifie@analytifi\'e|see{Analytification}}\index{Espace analytique|analytifie@analytifi\'e|see{Analytification}}%
\nomenclature[Kma]{$\cX^\an$}{analytifi\'e d'un sch\'ema~$\cX$}%
\nomenclature[Kmb]{$\rho_{\cX}$}{morphisme canonique $\cX^\an\to\cX$}%
\nomenclature[Kmc]{$\varphi^\an$}{analytifi\'e d'un morphisme de sch\'emas~$\varphi$}
Supposons que~$\cA$ est un anneau de base g\'eom\'etrique.
Soit~$\cX$ un sch\'ema localement de pr\'esentation finie sur~$\cA$. On appelle \emph{analytification} ou \emph{analytifi\'e} de~$\cX$ l'espace $\cA$-analytique qui repr\'esente~$\Phi_{\cX}$. On le note~$\cX^\an$ et on note $\rho_{\cX} \colon \cX^\an\to\cX$ le morphisme d'espaces localement $\cA$-annel\'es canoniquement associ\'e.

Soit $\varphi \colon \cX \to \cY$ un morphisme entre sch\'emas localement de pr\'esentation finie sur~$\cA$. La repr\'esentabilit\'e du foncteur~$\Phi_{\cY}$ assure l'existence d'un unique morphisme d'espace $\cA$-analytiques de~$\cX^\an$ dans~$\cY^\an$, not\'e~$\varphi^\an$, 
faisant commuter le diagramme
\[\begin{tikzcd}
\cX^\an \ar[d, "\varphi^\an"] \ar[r, "\rho_{\cX}"] & \cX \ar[d, "\varphi"]\\
\cY^\an \ar[r, "\rho_{\cY}"] & \cY
\end{tikzcd}.\]
Le morphisme $\varphi^\an \colon \cX^\an \to \cY^\an$ est appel\'e \emph{analytification} ou \emph{analytifi\'e} de~$\varphi$.
\end{defi}

\index{Analytification!|)}

\section{Extension des scalaires}\label{sec:extensionB}
\index{Extension des scalaires|(}

Soit $f\colon \cA \to \cB$ un morphisme d'anneaux de Banach born\'e, o\`u $\cB$ est un anneau de base g\'eom\'etrique. 

\begin{defi}\label{def:extensionscalaire}\index{Extension des scalaires|textbf}%
\nomenclature[Kna]{$X_{\cB}$}{extension des scalaires d'un espace $\cA$-analytique~$X$ \`a un anneau de Banach~$\cB$}%
\nomenclature[Knb]{$X \ho{\cA} \cB$}{extension des scalaires d'un espace $\cA$-analytique~$X$ \`a un anneau de Banach~$\cB$}%
Soit $X$ un espace $\cA$-analytique. On dit que \emph{$X$ s'\'etend \`a~$\cB$} si le foncteur 
\[\fonction{B_{X,\cB}}{\cB-\An}{\Ens}{Y}{\Hom_{\An_{\cA},f}(Y,X)}\]
est repr\'esentable. Dans ce cas, on appelle \emph{extension des scalaires de~$X$ \`a~$\cB$} l'espace $\cB$-analytique repr\'esentant~$B_{X,\cB}$. On le note~$X_{\cB}$ ou~$X \ho{\cA} \cB$ (par analogie avec le cas classique, \cf~\cite[\S 1.4]{Ber2}).
\end{defi}

Commen\c cons par un lemme g\'en\'eral.

\begin{lemm}\label{lem:extensionBouverts}
Soit $X$ un espace $\cA$-analytique qui s'\'etend \`a~$\cB$. Notons $\pi_{\cB} \colon X_{\cB} \to X$ le morphisme canonique.

\begin{enumerate}[i)]
\item Soit~$U$ un ouvert de~$X$. Alors $\pi_{\cB}^{-1}(U)$ est l'extension des scalaires de~$U$ \`a~$\cB$.

\item Soient~$F$ un ferm\'e de~$X$ d\'efini par un faisceau d'id\'eaux coh\'erent~$\cI$. Alors le ferm\'e analytique de~$X_\cB$ d\'efini par l'id\'eal image de ${\pi_{\cB}}^\ast \cI$ dans~$\cO_{X_{\cB}}$ est l'extension des scalaires de~$F$ \`a~$\cB$.

\item Soit $\iota \colon X' \to X$ une immersion (resp. immersion ouverte, resp. immersion ferm\'ee) d'espaces $\cA$-analytiques. Alors~$X'$ s'\'etend \`a~$\cB$ et le morphisme canonique $X'_{\cB} \to X_{\cB}$ est une immersion (resp. immersion ouverte, resp. immersion ferm\'ee) d'espaces $\cB$-analytiques. 
\end{enumerate}
\end{lemm}
\begin{proof}
Le point~i) d\'ecoule directement de la propri\'et\'e universelle de l'extension. Pour d\'emontrer le point~ii), on utilise la proposition~\ref{immersion_ferm\'e}. Le point~iii) se d\'eduit des deux premiers.
\end{proof}

Passons maintenant \`a des r\'esultats d'existence.

\begin{prop}\label{prop:extensionBaffine}\index{Espace affine analytique!extension des scalaires}
Soit~$n\in \N$. Alors l'extension des scalaires de~$\E{n}{\cA}$ \`a~$\cB$ est~$\E{n}{\cB}$ avec pour morphisme canonique $\tilde f_{n} \colon \E{n}{\cB} \to \E{n}{\cA}$, tel que défini dans l'exemple~\ref{ex:tildefn}.
\end{prop}
\begin{proof}
Le r\'esultat d\'ecoule de la proposition~\ref{morphsec}.
\end{proof}

\begin{theo}\label{thm:extensionB}
Tout espace $\cA$-analytique s'\'etend \`a~$\cB$.
\end{theo}
\begin{proof}
Soit~$X$ un espace $\cA$-analytique. Si~$X$ est un mod\`ele local $\cA$-analytique, alors le r\'esultat d\'ecoule de la proposition~\ref{prop:extensionBaffine} et du lemme~\ref{lem:extensionBouverts}. Le cas g\'en\'eral s'en d\'eduit en recouvrant~$X$ par des mod\`eles locaux et en recollant les diff\'erentes extensions.
\end{proof}

On peut consid\'erer des extensions des scalaires successives.

\begin{lemm}\label{lem:BB'}
Soit~$X$ un espace $\cA$-analytique. Soit $\cB'$ un anneau de base g\'eom\'etrique et soit $f'\colon \cB \to \cB'$ un morphisme d'anneaux de Banach born\'e.
Alors on a un isomorphisme canonique 
\[X\ho{\cA} \cB' \simtoo (X \ho{\cA} \cB)\ho{\cB} \cB'.\]
\end{lemm}
\begin{proof}
Lorsque~$X$ est un espace affine analytique, le r\'esultat d\'ecoule de la proposition~\ref{prop:extensionBaffine}. Celui pour les mod\`eles locaux analytiques s'en d\'eduit par le lemme~\ref{lem:extensionBouverts}, et, finalement, le cas g\'en\'eral, par recollement.
\end{proof}

On peut \'egalement combiner extension des scalaires et analytification.

\begin{lemm}\label{lem:anB}\index{Analytification!d'un schema@d'un sch\'ema}
Soit~$\cX$ un sch\'ema localement de pr\'esentation finie sur~$\cA$. Alors on a un isomorphisme d'espaces $\cB$-analytiques canonique 
\[\cX^\an \ho{\cA} \cB  \simtoo (\cX\otimes_{\cA} \cB)^\an.\]
\end{lemm}
\begin{proof}
Lorsque~$\cX$ est un espace affine, le r\'esultat d\'ecoule de la proposition~\ref{prop:extensionBaffine}. Celui pour les sch\'emas affines s'en d\'eduit par le lemme~\ref{lem:extensionBouverts} et la construction de l'analytifi\'e (deuxi\`eme point de la preuve du th\'eor\`eme~\ref{thm:analytification}). Le cas g\'en\'eral s'en d\'eduit  par recollement (troisii\`eme point de la preuve du th\'eor\`eme~\ref{thm:analytification}).
\end{proof}

Il est souvent utile de consid\'erer l'extension des scalaires d'un espace au corps r\'esiduel compl\'et\'e d'un de ses points.

\begin{lemm}\label{lem:morphismexHx}\index{Morphisme analytique!associe a un point@associ\'e \`a un point}
Soit~$X$ un espace $\cA$-analytique. Le morphisme $\lambda_{x} \colon \cM(\cH(x)) \to X$ de l'exemple~\ref{ex:morphismex} induit un morphisme
\[\lambda_{x} \ho{\cA}  \cH(x) \colon \cM(\cH(x)) \too X \ho{\cA} \cH(x) \]
qui est une immersion ferm\'ee. Son image~$x'$ est envoy\'ee sur~$x$ par le morphisme $X \ho{\cA} \cH(x)  \to X$ et le morphisme canonique $\cH(x) \to \cH(x')$ est un isomorphisme isom\'etrique.
\end{lemm}
\begin{proof}
L'image de~$\lambda_{x}$ \'etant le point~$x$, on peut localiser au voisinage de~$x$ pour d\'emontrer le r\'esultat. On peut donc supposer que~$X$ est un ferm\'e analytique d'un ouvert~$U$ d'un espace affine~$\E{n}{\cA}$. Notons $T_{1},\dotsc,T_{n}$ les coordonn\'ees sur~$\E{n}{\cA}$. Le morphisme~$\lambda_{x}$ est alors d\'efini comme celui associ\'e au morphisme d'\'evaluation $\cA[T_{1},\dotsc,T_{n}] \to \cH(x)$. 

Par extension des scalaires \`a~$\cH(x)$, on obtient un diagramme commutatif
\[\begin{tikzcd}
\cM(\cH(x)) \arrow[r] \arrow[dr]& X \arrow[r]  & U \arrow[r] & \E{n}{\cA} \\
& X\ho{\cA} \cH(x) \arrow[r] \arrow[u] & U\ho{\cA} \cH(x) \arrow[r] \arrow[u] & \E{n}{\cH(x)} \arrow[u]
\end{tikzcd}\]
dans lequel le morphisme $\cM(\cH(x)) \to  \E{n}{\cH(x)}$ est associ\'e au morphisme
\[ \begin{array}{ccc}
\cA[T_{1},\dotsc,T_{n}] & \too & \cH(x)\\
T_{i} & \mapstoo & T_{i}(x)
\end{array}.\]
En particulier, le morphisme $\cM(\cH(x)) \to  \E{n}{\cH(x)}$ est une immersion ferm\'ee d\'efinie par le faisceau d'id\'eaux~$\cI$ sur~$\E{n}{\cH(x)}$ engendr\'e par $T_{1}-T_{1}(x),\dotsc,T_{n}-T_{n}(x)$. Le morphisme  $\cM(\cH(x)) \to  X\ho{\cA} \cH(x)$ est donc encore  une immersion ferm\'ee, d\'efinie par le tir\'e en arri\`ere de~$\cI$ sur~$X$. 

La derni\`ere partie de l'\'enonc\'e est claire.
\end{proof}
\index{Extension des scalaires|)}

\section{Produits fibr\'es}\label{sec:produitsfibres}
\index{Produit!fibre@fibr\'e|(}

Dans cette section, nous supposerons que $\cA$ est un anneau de base g\'eom\'etrique. Nous allons d\'emontrer que la cat\'egorie~$\cAAn$ admet des produits fibr\'es. La strat\'egie adopt\'ee est la m\^eme qu'\`a la section~\ref{sec:extensionB} pour l'extension des scalaires.

\medskip

Commen\c cons par deux lemmes g\'en\'eraux.

\begin{lemm}\label{lem:produitouverts}
Soient $\varphi\colon X \to Z$ et $\psi \colon Y \to Z$ des morphismes d'espaces $\cA$-analytiques. 
Supposons que le produit fibr\'e de~$X$ et~$Y$ au-dessus de~$Z$ existe dans la cat\'egorie~$\cAAn$. Notons~$P$ ce produit et $p_{X} \colon P \to X$ et $p_{Y}\colon P \to Y$ les deux projections canoniques. 

\begin{enumerate}[i)]
\item Soient~$U$ un ouvert de~$X$ et~$V$ un ouvert de~$Y$. Alors $p_{X}^{-1}(U) \cap p_{Y}^{-1}(V)$ est le produit fibr\'e de~$U$ et~$V$ au-dessus de~$Z$ dans~$\cAAn$.

\item Soient~$F$ un ferm\'e de~$X$ d\'efini par un faisceau d'id\'eaux coh\'erent~$\cI$ et~$G$ un ferm\'e de~$Y$ d\'efini par un faisceau d'id\'eaux coh\'erent~$\cJ$. Alors le ferm\'e analytique de~$P$ d\'efini par l'id\'eal somme des images de ${p_{X}}^\ast \cI$ et ${p_{Y}}^\ast \cJ$ dans~$\cO_{P}$ est le produit fibr\'e de~$F$ et~$G$ au-dessus de~$Z$ dans~$\cAAn$.
\end{enumerate}
\end{lemm}
\begin{proof}
Le point~i) d\'ecoule directement de la propri\'et\'e universelle du produit. Pour d\'emontrer le point~ii), on utilise la proposition~\ref{immersion_ferm\'e}. 
\end{proof}

L'\'enonc\'e qui suit d\'ecoule du pr\'ec\'edent.

\begin{lemm}\label{lem:changementbaseimmersion}
Soit~$Z$ un espace $\cA$-analytique. Soient~$X$ et~$Y$ des espaces $\cA$-analytiques au-dessus de~$Z$. Supposons qu'il existe un produit fibr\'e~$P$ de~$X$ et~$Y$ au-dessus de~$Z$ dans~$\cAAn$. Soit $\iota \colon X' \to X$ une immersion (resp. immersion ouverte, resp. immersion ferm\'ee) d'espaces $\cA$-analytiques au-dessus de~$Z$. Alors, il existe un produit fibr\'e~$P'$ de~$X'$ et~$Y$ au-dessus de~$Z$ dans~$\cAAn$ et le morphisme canonique $P' \to P$ est une immersion (resp. immersion ouverte, resp. immersion ferm\'ee). 
\qed
\end{lemm}

Nous allons maintenant d\'emontrer des r\'esultats d'existence de produits. 

\begin{prop}\label{prop:produitaffines}\index{Espace affine analytique!produit d'}
Soient $n,m\in \N$. Alors $\E{n+m}{\cA}$ est le produit de~$\E{n}{\cA}$ et~$\E{m}{\cA}$ dans la cat\'egorie~$\cAAn$.
\end{prop}
\begin{proof}
Le r\'esultat d\'ecoule de la proposition~\ref{morphsec}.
\end{proof}

Puisque chaque espace $\cA$-analytique poss\`ede un morphisme canonique vers~$\cM(\cA)$, le produit dans la cat\'egorie $\cAAn$ co\"incide avec le produit fibr\'e au-dessus de~$\cM(\cA)$. Nous pourrons donc utiliser le lemme~\ref{lem:produitouverts} dans le contexte des produits.

\begin{coro}\label{cor:produitmodeles}\index{Modele local analytique@Mod\`ele local analytique!produit de}
Soient $n,m\in \N$. Soient~$X$ un ferm\'e analytique d'un ouvert de~$\E{n}{\cA}$ et $Y$~un ferm\'e analytique d'un ouvert de~$\E{m}{\cA}$. Alors le produit de~$X$ et~$Y$ existe dans la cat\'egorie~$\cAAn$ et c'est un ferm\'e analytique d'un ouvert de~$\E{n+m}{\cA}$.
\qed
\end{coro}

\begin{coro}\label{cor:produit}\index{Produit|textbf}\index{Espace analytique!produit d'|see{Produit}}
La cat\'egorie~$\cAAn$ admet des produits finis.
\end{coro}
\begin{proof}
Il suffit de montrer que le produit de deux espaces existe. Soient~$X$ et~$Y$ des espaces $\cA$-analytiques. Si ce sont des mod\`eles locaux $\cA$-analytiques, le r\'esultat d\'ecoule du corollaire~\ref{cor:produitmodeles}. Le cas g\'en\'eral s'en d\'eduit en recouvrant les espaces par des mod\`eles locaux et en recollant les diff\'erents produits. 
\end{proof}

\begin{nota}%
\nomenclature[Ko]{$X\times_{\cA} Y$}{produit d'espaces $\cA$-analytiques~$X$ et~$Y$}
Soient~$X$ et~$Y$ des espaces $\cA$-analytiques. Nous noterons $X\times_{\cA} Y$ leur produit dans la cat\'egorie $\cAAn$.
\end{nota}

L'existence de produits va nous permettre de d\'emontrer l'existence de produits fibr\'es.

\begin{prop}\label{prop:produitfibreZaffine}
Soient $\varphi\colon X \to Z$ et $\psi \colon Y \to Z$ des morphismes d'espaces $\cA$-analytiques. 
Notons $p_{X} \colon X\times_{\cA} Y\to X$ et $p_{Y} \colon X\times_{\cA} Y \to Y$ les deux morphismes de projection. Supposons que~$Z$ soit un ferm\'e analytique d'un ouvert de~$\E{n}{\cA}$. Notons~$T_{1},\dotsc,T_{n}$ les coordonn\'ees sur~$\E{n}{\cA}$ et, pour tout $i\in \cn{1}{n}$, posons 
\[ f_{i} := (\varphi\circ p_{X})^\sharp(T_{i}) \textrm{ et } g_{i} := (\psi\circ p_{Y})^\sharp(T_{i})
\textrm{ dans } \cO(X\times_{\cA} Y) .\]
Alors, le ferm\'e analytique de~$X\times_{\cA} Y$ d\'efini par l'id\'eal $(f_{1}-g_{1},\dotsc,f_{n}-g_{n})$ est le produit fibr\'e de~$X$ et~$Y$ au-dessus de~$Z$.
\end{prop}
\begin{proof}
Cela d\'ecoule de la proposition~\ref{immersion_ferm\'e}.
\end{proof}

\begin{theo}\label{produit_fibr\'e}\index{Produit!fibre@fibr\'e|textbf}\index{Espace analytique!produit fibr\'e d'|see{Produit fibr\'e}}
Pour tous morphismes d'espaces $\cA$-analytiques $\varphi\colon X \to Z$ et $\psi \colon Y \to Z$, le produit fibr\'e de~$X$ et~$Y$ au-dessus de~$Z$ existe et est naturellement un ferm\'e analytique de~$X\times_{\cA} Y$.

En particulier, la cat\'egorie~$\cAAn$ admet des produits fibr\'es finis.
\end{theo}
\begin{proof}
Le r\'esultat s'obtient en recouvrant~$Z$ par des mod\`eles locaux, en appliquant la proposition~\ref{prop:produitfibreZaffine} au-dessus de chacun d'eux, puis en recollant les produits fibr\'es obtenus.
\end{proof}

\begin{nota}
\nomenclature[Kproduitfibre1]{$X\times_{Z} Y$}{produit fibr\'e d'espaces $\cA$-analytiques~$X$ et~$Y$ au-dessus d'un espace $\cA$-analytique~$Z$}
Soient $\varphi\colon X \to Z$ et $\psi \colon Y \to Z$ des morphismes d'espaces $\cA$-analytiques. Nous noterons $X\times_{Z} Y$ le produit fibr\'e de~$X$ et~$Y$ au-dessus de~$Z$ dans la cat\'egorie $\cAAn$ (en sous-entendant les morphismes~$\varphi$ et~$\psi$).
\end{nota}

En utilisant le corollaire~\ref{cor:produitmodeles}, nous obtenons le r\'esultat suivant.

\begin{coro}\label{cor:produitfibremodeles}\index{Modele local analytique@Mod\`ele local analytique!produit fibr\'e de}
Soient $n,m\in \N$. Soient~$X$ un ferm\'e analytique d'un ouvert de~$\E{n}{\cA}$ et $Y$~un ferm\'e analytique d'un ouvert de~$\E{m}{\cA}$. Soient $\varphi\colon X \to Z$ et $\psi \colon Y \to Z$ des morphismes d'espaces $\cA$-analytiques. Alors le produit fibr\'e $X\times_{Z}Y$ est un ferm\'e analytique d'un ouvert de~$\E{n+m}{\cA}$.
\qed
\end{coro}

%
%
%

En utilisant les propri\'et\'es universelles, on v\'erifie que le produit fibr\'e commute \`a l'extension des scalaires.

\begin{lemm}\label{lem:produitifbreextensionscalaires}\index{Extension des scalaires}
Soient $\varphi\colon X \to Z$ et $\psi \colon Y \to Z$ des morphismes d'espaces $\cA$-analytiques. 
Soit $f\colon \cA \to \cB$ un morphisme d'anneaux de Banach born\'e, o\`u $\cB$~est un anneau de base g\'eom\'etrique. Alors, on a un isomorphisme canonique
\[(X\times_{Z}Y)\ho{\cA} \cB \simeq (X\ho{\cA} \cB)\times_{Z\ho{\cA} \cB}(Y\ho{\cA} \cB).\]
\qed
\end{lemm}

On peut \'egalement combiner produit fibr\'e et analytification.

\begin{lemm}\label{lem:produitan}\index{Analytification!d'un schema@d'un sch\'ema}
Soient $f \colon \cX \to \cZ$ et $g\colon \cY \to \cZ$ des morphismes de sch\'emas localement de pr\'esentation finie sur~$\cA$. Alors on a un isomorphisme d'espaces $\cA$-analytiques canonique 
\[(\cX\times_{\cZ}\cY)^\an \simtoo \cX^\an \times_{\cZ^\an} \cY^\an.\]
\end{lemm}
\begin{proof}
La propri\'et\'e universelle du produit fibr\'e permet de construire un morphisme canonique $(\cX\times_{\cZ}\cY)^\an \to \cX^\an \times_{\cZ^\an} \cY^\an$. 

En utilisant la propri\'et\'e universelle de l'analytification et le fait que les produits fibr\'es dans la cat\'egorie des sch\'emas sont \'egalement des produits fibr\'es dans la cat\'egorie des espaces localement annel\'es (\cf~\cite[\href{https://stacks.math.columbia.edu/tag/01JN}{Tag 01JN}]{stacks-project}), 
on construit un morphisme canonique $\cX^\an \times_{\cZ^\an} \cY^\an \to (\cX\times_{\cZ}\cY)^\an$. 

On v\'erifie que ces morphismes sont inverses l'un de l'autre, ce qui permet de conclure. 
\end{proof}

En utilisant l'extension des scalaires, 
on peut obtenir un \'enonc\'e plus g\'en\'eral d'existence de produits fibr\'es. Il se d\'emontre directement \`a l'aide des r\'esultats d\'ej\`a obtenus.

\begin{prop}\label{prop:produitfibreAnA}\index{Extension des scalaires}
Soit~$\varphi \colon X\to Z$ un morphisme d'espaces~$\cA$-analytiques. Soit $f\colon \cA \to \cB$ un morphisme d'anneaux de Banach born\'e, o\`u $\cB$~est un anneau de base g\'eom\'etrique.
Soient~$Y$ un espace~$\cB$-analytique et $\psi \colon Y\to Z$ un morphisme d'espaces analytiques au-dessus de~$f$. Alors l'espace $\cB$-analytique $(X\ho{\cA} \cB)\times_{Z\ho{\cA} \cB} Y$ repr\'esente le foncteur qui \`a tout espace analytique~$T$ au-dessus de~$\cB$ associe l'ensemble des diagrammes commutatifs d'espaces analytiques au-dessus de~$\cA$ de la forme
\[\begin{tikzcd}
T\ar[r]\ar[d]&X\ar[d]\\
Y\ar[r]&Z 
\end{tikzcd}.\]
\qed
\end{prop}

Dans le cadre de la proposition pr\'ec\'edente, nous adopterons parfois la notation simplifi\'ee
\[ X \times_{Z} Y := (X\ho{\cA} \cB)\times_{Z\ho{\cA} \cB} Y.\]
\nomenclature[Kproduitfibre2]{$X\times_{Z} Y$}{produit fibr\'e d'un espace $\cA$-analytique~$X$ et d'un espace $\cB$-analytique~$Y$ au-dessus d'un espace $\cA$-analytique~$Z$}

\medbreak

D\'efinissons maintenant les morphismes diagonaux.\index{Morphisme analytique!diagonal|(}

\begin{defi}\label{def:diagonale}\index{Morphisme analytique!diagonal|textbf}%
\nomenclature[Kqb]{$\Delta_{X/Y}$}{pour un espace $\cA$-analytique~$X$ au-dessus d'un espace $\cA$-analytique~$Y$, morphisme diagonal $X\to X\times_{Y} X$}%
\nomenclature[Kqa]{$\Delta_{X}$}{pour un espace $\cA$-analytique~$X$, morphisme diagonal $X\to X\times_{\cA} X$}%
Soient $X$, $Y$ des espaces $\cA$-analytiques et $\varphi \colon X\to Y$ un morphisme. On note $p_{1} \colon X\times_{Y} X \to X$ et $p_{2} \colon X\times_{Y} X \to X$ les deux projections canoniques. On appelle \emph{morphisme diagonal} l'unique morphisme $\Delta_{X/Y} \colon X\to X\times_{Y} X$ qui fait commuter le diagramme
\[\begin{tikzcd}
X \arrow[rd, "\Delta_{X/Y}"] \arrow[rrd, bend left, "\id_{X}"] \arrow[rdd, bend right, "\id_{X}"'] & &\\
&X\times_{Y} X \arrow[r, "p_{2}"] \arrow[d, "p_{1}"']& X \arrow[d, "\varphi"]\\
& X \arrow[r, "\varphi"]& Y
\end{tikzcd}.\]
Dans le cas o\`u $Y = \cM(\cA)$ et~$\varphi$ est le morphisme structural, on pose $\Delta_{X} := \Delta_{X/\cM(\cA)}$. 
\end{defi}

\begin{prop}\label{prop:immersiondiagonalemodelelocal}\index{Immersion!fermee@ferm\'ee}
Soient $X$, $Y$ des espaces $\cA$-analytiques et $\varphi \colon X\to Y$ un morphisme. Supposons que $X$ est un mod\`ele local $\cA$-analytique. Alors, le morphisme diagonal $\Delta_{X/Y} \colon X\to X\times_{Y} X$ est une immersion ferm\'ee.
\end{prop}
\begin{proof}
Par hypoth\`ese, $X$ s'identifie \`a un ferm\'e analytique d'un ouvert de~$\E{n}{\cA}$. Notons $T_{1},\dotsc,T_{n}$ les coordonn\'ees sur l'espace~$\E{n}{\cA}$ et identiquement leur restriction \`a~$X$. Posons $\Delta := \Delta_{X/Y}$. 

Soit $\cI$ le faisceau d'id\'eaux de~$X \times_{Y} X$ engendr\'e par les sections $p_{1}^\ast T_{j} - p_{2}^\ast T_{j}$ avec $j \in\cn{1}{m}$, $F$ le ferm\'e analytique de~$X\times_{Y} X$ qu'il d\'efinit et $\iota \colon F \to X \times_{Y} X$ l'immersion ferm\'ee associ\'ee. Pour tout $i\in \cn{1}{m}$, on a 
\begin{align*}
\Delta^\ast (p_{1}^\ast T_{j} - p_{2}^\ast T_{j}) &= \Delta^\ast p_{1}^\ast T_{j}  - \Delta^\ast p_{2}^\ast T_{j}\\
&= T_{j} - T_{j}\\
&=0
\end{align*}
donc, d'apr\`es la proposition~\ref{immersion_ferm\'e}, il existe un morphisme $\psi \colon X \to F$ tel que $\Delta= \iota \circ \psi$. 

Montrons que~$\psi$ est un isomorphisme, d'inverse $p_{1} \circ \iota$. Cela entra\^inera que~$\Delta$ est une immersion ferm\'ee, comme d\'esir\'e. Remarquons, tout d'abord, que l'on a 
\[p_{1} \circ \iota \circ \psi = p_{1} \circ \Delta = \id_{X}\]
par d\'efinition. 

Calculons maintenant  $\psi \circ p_{1} \circ \iota$. D'apr\`es le corollaire~\ref{cor:produitfibremodeles}, $X\times_{Y} X$ est un ferm\'e analytique d'un ouvert de~$\E{2n}{\cA}$. Les fonctions $p_{1}^\ast T_{1},\dotsc, p_{1}^\ast T_{n},p_{2}^\ast T_{1},\dotsc, p_{2}^\ast T_{n}$ forment un syst\`eme de coordonn\'ees sur~$\E{2n}{\cA}$. Pour tout $i\in \cn{1}{n}$, on a 
\[ (\psi \circ p_{1} \circ \iota)^\ast \iota^\ast p_{1}^\ast T_{i} = \iota^\ast p_{1}^\ast \psi^\ast \iota^\ast p_{1}^\ast T_{i} = \iota^\ast p_{1}^\ast T_{i},\]
et, pour  tout $j\in \cn{1}{m}$, on a
\[ (\psi \circ p_{1} \circ \iota)^\ast \iota^\ast p_{2}^\ast T_{j} = \iota^\ast p_{1}^\ast \psi^\ast \iota^\ast p_{2}^\ast T_{j} =  \iota^\ast p_{1}^\ast T_{j} = \iota^\ast p_{2}^\ast T_{j},\]
par d\'efinition de~$F$. On d\'eduit alors de la proposition~\ref{morphsec} que $\psi \circ p_{1} \circ \iota = \id_{X}$.
\end{proof}

\begin{coro}\label{cor:immersiondiagonale}\index{Immersion}
Soient $X$, $Y$ des espaces $\cA$-analytiques et $\varphi \colon X\to Y$ un morphisme. Alors, le morphisme diagonal $\Delta_{X/Y} \colon X\to X\times_{Y} X$ est une immersion.
\end{coro}
\begin{proof}
D'apr\`es la proposition~\ref{crit\`ere_local_immersion}, il suffit de montrer que tout point~$z$ de~$X\times_{Y} X$ poss\`ede un voisinage~$W$ tel que le morphisme $\Delta_{X/Y}^{-1}(W)\to W$ induit par~$\Delta$ soit une immersion ferm\'ee. La question est donc locale sur~$X\times_{Y} X$. D'apr\`es le lemme~\ref{lem:produitouverts}, on peut donc restreindre~$X$ et supposer que c'est un mod\`ele local $\cA$-analytique. Le r\'esultat d\'ecoule alors de la proposition~\ref{prop:immersiondiagonalemodelelocal}.
\end{proof}

\index{Morphisme analytique!diagonal|)}

\index{Produit!fibre@fibr\'e|)}

\section{Morphismes propres}\label{sec:propre}

Commen\c cons par rappeler la d\'efinition d'application propre entre espaces topologiques (\cf~\cite[I, \S 10]{BourbakiTG14}).

\index{Application!propre|(}\index{Morphisme analytique!propre|(}

\begin{defi}\index{Application!propre|textbf}\index{Morphisme analytique!propre|textbf}\index{Morphisme!topologique|see{Application}}
Une application $f \colon T\to T'$ entre espaces topologiques est dite \emph{propre} si elle est continue et si, pour tout espace topologique~$T''$, l'application $f \times \id_{T''} \colon T \times T'' \to T' \times T''$ est ferm\'ee. 

Un morphisme d'espaces $\cA$-analytiques est dit \emph{propre} 
si l'application induite entre les espaces topologiques sous-jacents est propre au sens pr\'ec\'edent.
\end{defi}

\begin{rema}\label{rem:proprelocalaubut}
La propret\'e est une notion locale au but, \cf~\cite[I, \S 10, \no 1, proposition~3 b)]{BourbakiTG14}. 
\end{rema}

\begin{rema}\label{rem:proprekan}
Dans le cas o\`u $\cA = k$ est un corps valu\'e ultram\'etrique complet, V.~Berkovich d\'efinit les morphismes propres comme les morphismes qui sont propres au sens topologique et sans bord (\cf~\cite[p.~50, apr\`es le corollaire~3.1.6]{Ber1}). Puisque nous ne consid\'erons ici que des espaces sans bords (\cf~remarque~\ref{rem:kansansbords}), cette derni\`ere condition est toujours satisfaite et il n'est donc pas n\'ecessaire de l'imposer.
\end{rema}

\'Enon\c cons quelques propri\'et\'es des applications propres vis-\`a-vis de la composition.

\begin{lemm}[\protect{\cite[I, \S 10, \no 1, proposition~5]{BourbakiTG14}}]\label{lem:compositionpropre}
Soient $f\colon T \to T'$ et $g\colon T' \to T''$ des applications continues entre espaces topologiques.
\begin{enumerate}[i)]
\item Si $f$ et $g$ sont propres, alors $g\circ f$ est propre.
\item Si $g\circ f$ est propre et $f$ est surjective, alors $g$ est propre.
\item Si $g\circ f$ est propre et $g$ est injective, alors $f$ est propre.
\item Si $g\circ f$ est propre et $T'$ est s\'epar\'e, alors $f$ est propre.
\end{enumerate} 
\end{lemm}

La propret\'e peut essentiellement s'exprimer en termes de compacit\'e. Nous disposons en particulier d'un crit\`ere utile pour d\'emontrer la propret\'e.

\begin{lemm}\label{lem:criterepropre}
Soit $f \colon T\to T'$ une application continue entre espaces topologiques. 
\begin{enumerate}[i)]
\item Si $f$ est propre, alors pour tout partie quasi-compacte~$K$ de~$T$, l'image r\'eciproque $f^{-1}(K)$ est quasi-compacte.
\item Si tout point~$t$ de~$T'$ poss\`ede un voisinage compact~$V$ dont l'image r\'eciproque~$f^{-1}(V)$ est compacte, alors $f$~est propre.
\end{enumerate}
\end{lemm}
\begin{proof}
i) d\'ecoule de \cite[I, \S 10, \no 2, proposition~6]{BourbakiTG14}.

ii) d\'ecoule de la remarque~\ref{rem:proprelocalaubut} et de \cite[I, \S 10, \no 2, corollaire~2 du lemme~2]{BourbakiTG14}.
\end{proof}
\index{Application!propre|)}

Rappelons que, pour tout $n\in \N$, nous avons d\'efini un morphisme $\tilde{f}_n \colon \E{n}{\cB} \to \E{n}{\cA}$ dans l'exemple~\ref{ex:tildefn}.

\begin{prop}\label{prop:tildefnpropre}
Soit $n\in \N$. Le morphisme~$\tilde{f}_{n}$ est propre.
\end{prop}
\begin{proof}
Notons~$T_{1},\dotsc,T_{n}$ les coordonn\'ees sur~$\E{n}{\cA}$ et~$\E{n}{\cB}$. Pour tout $r \in \R_{>0}$, notons $\oD_{\cM(\cA)}(r)$ et $\oD_{\cM(\cB)}(r)$ les polydisques ferm\'es de polyrayon~$(r,\dotsc,r)$ au-dessus de~$\cM(\cA)$ et~$\cM(\cB)$ respectivement.

Soit~$K$ une partie compacte de~$\E{n}{\cA}$. Il existe $r\in \R_{>0}$ tel que $K \subset \overline{D}_{\cM(\cA)}(r)$. Or, on a 
\[ \tilde{f}_{n}^{-1}(\overline{D}_{\cM(\cA)}(r)) = \overline{D}_{\cM(\cB)}(r), \]
qui est une partie compacte de~$\E{n}{\cB}$. Puisque~$\tilde f_{n}$ est continue, $\tilde{f}_{n}^{-1}(K)$ est ferm\'ee dans un compact de~$\E{n}{\cB}$, et donc compacte. On d\'eduit alors du lemme~\ref{lem:criterepropre}, ii) que~$\tilde f_{n}$ est propre. 
\end{proof}

\begin{coro}\label{stabilit\'e_propre_extension_scalaire}\index{Extension des scalaires}
Pour tout espace $\cA$-analytique~$X$, le morphisme d'extension des scalaires $X \ho{\cA} \cB \to X$ est propre. 
\end{coro}
\begin{proof}
D'apr\`es le lemme~\ref{lem:criterepropre}, il suffit de montrer que tout point~$x$ de~$X$ poss\`ede un voisinage compact dont l'image r\'eciproque est compacte.

Soit~$x\in X$. Quitte \`a remplacer~$X$ par un voisinage de~$x$, on peut supposer que~$X$ est un sous-espace analytique d'un ouvert~$U$ de l'espace affine~$\E{n}{\cA}$. La proposition~\ref{extension_scalaire_spectral} assure que le morphisme $\tilde{f}_{n} \colon \E{n}{\cB}\to\E{n}{\cA}$ est propre. 
D'apr\`es le lemme~\ref{lem:extensionBouverts}, on a $U \ho{\cA} \cB = \tilde{f}_{n}^{-1}(U)$ et $X \ho{\cA} \cB = \tilde{f}_{n}^{-1}(X)$. Le r\'esultat s'en d\'eduit.
\end{proof}

\begin{prop}\label{prop:propreAB}\index{Extension des scalaires}
Soit~$\varphi \colon X\to Y$ un morphisme propre d'espaces $\cA$-analytiques. Soit $f \colon \cA \to \cB$ un morphisme d'anneaux de Banach born\'e, o\`u~$\cB$ est un anneau de base g\'eom\'etrique. Alors le morphisme $\varphi_{\cB} \colon X \ho{\cA} \cB \to Y\ho{\cA} \cB$, d\'eduit de~$\varphi$ par extension des scalaires \`a~$\cB$, est propre.
\end{prop}
\begin{proof}
La propret\'e \'etant locale au but, d'apr\`es le lemme~\ref{lem:extensionBouverts}, on peut supposer que~$Y$ est un ferm\'e analytique d'un ouvert de~$\E{n}{\cA}$. D'apr\`es le lemme~\ref{lem:extensionBouverts} et la proposition~\ref{prop:extensionBaffine}, $Y\ho{\cA}\cB$ est alors un ferm\'e analytique d'un ouvert de~$\E{n}{\cB}$. En particulier, $|Y\ho{\cA}\cB|$ est s\'epar\'e.

On a un diagramme commutatif
\[\begin{tikzcd}
X \ho{\cA}\cB \arrow[r, "p_{X}"] \arrow[d, "\varphi_{\cB}"] & X  \arrow[d, "\varphi"] \\
Y\ho{\cA} \cB \arrow[r, "p_{Y}"] & Y
\end{tikzcd}.\]
D'apr\`es le corollaire~\ref{stabilit\'e_propre_extension_scalaire}, le morphisme~$p_{X}$ est propre. On d\'eduit alors du lemme~\ref{lem:compositionpropre} que $\varphi \circ p_{X} = p_{Y} \circ \varphi_{\cB}$ est propre, puis que $\varphi_{\cB}$ est propre, comme d\'esir\'e.
\end{proof}

Comme on s'y attend, la propret\'e est \'egalement stable par changement de base. Nous le d\'emontrerons \`a la proposition~\ref{stabilite_propre}.

\index{Morphisme analytique!propre|)}

\section{Parties spectralement convexes}\label{sec:paspconv}

Dans cette section, nous \'etudions le comportement des parties spectralement convexes par extension des scalaires et produit fibr\'e.

\medbreak

Commen\c cons par l'extension des scalaires. Soient~$\cB$ un anneau de Banach et $f \colon \cA \to \cB$ un morphisme born\'e. On suppose que~$\cA$ et~$\cB$ sont des anneaux de base g\'eom\'etriques.

\begin{nota}\index{Norme!tensorielle}\index{Produit tensoriel compl\'et\'e}%
\nomenclature[Bka]{$\nm_{\cM \otimes_{\cA} \cN}$}{pour $\cM,\cN$ des $\cA$-alg\`ebres de Banach, semi-norme tensorielle sur $ \cM \otimes_{\cA} \cN$}%
\nomenclature[Bkb]{$\cM \ho{\cA} \cN$}{s\'epar\'e compl\'et\'e de~$\cM \otimes_{\cA} \cN$ pour la semi-norme~$\nm_{\cM \otimes_{\cA} \cN}$}%
\nomenclature[Bkc]{$\cM \hosp{\cA} \cN$}{s\'epar\'e compl\'et\'e de~$\cM \otimes_{\cA} \cN$ pour la semi-norme spectrale associ\'ee \`a~$\nm_{\cM \otimes_{\cA} \cN}$}
Soient~$(\cM,\nm_{\cM})$ et~$(\cN,\nm_{\cN})$ deux $\cA$-alg\`ebres de Banach. On d\'efinit une semi-norme~$\nm_{\cM \otimes_{\cA} \cN}$ sur $\cM \otimes_{\cA} \cN$ par la formule suivante~: pour tout $x \in \cM \otimes_{\cA} \cN$, 
\[ \|x\|_{\cM \otimes_{\cA} \cN} = \inf\big(\big\{ \sum_{i\in I} \|m_{i}\|_{\cM}\, \|n_{i}\|_{\cN}, \ x = \sum_{i\in I} m_{i}\otimes n_{i}\big\}\big).\]
On note $\cM \ho{\cA} \cN$ le s\'epar\'e compl\'et\'e de~$\cM \otimes_{\cA} \cN$ pour la semi-norme~$\nm_{\cM \otimes_{\cA} \cN}$ et $\cM \hosp{\cA} \cN$ le s\'epar\'e compl\'et\'e pour la semi-norme spectrale associ\'ee.
\end{nota}

Rappelons que, pour tout $n\in \N$, nous avons d\'efini un morphisme $\tilde{f}_n \colon \E{n}{\cB} \to \E{n}{\cA}$ dans l'exemple~\ref{ex:tildefn}.

\begin{prop}\label{extension_scalaire_spectral}\index{Partie!spectralement convexe}\index{Produit tensoriel compl\'et\'e}
Soit $n\in \N$ et soit~$V$ une partie compacte spectralement convexe de~$\E{n}{\cA}$. Alors $\tilde{f}_{n}^{-1}(V)$ est compacte et spectralement convexe et le morphisme canonique
\[ \cB(V)\hosp{\cA} \cB \simtoo \cB(\tilde{f}_{n}^{-1}(V))\]
est un isomorphisme isom\'etrique.
\end{prop}
\begin{proof}
Notons~$T_{1},\dotsc,T_{n}$ les coordonn\'ees sur~$\E{n}{\cA}$ et~$\E{n}{\cB}$. 
%

D'apr\`es la proposition~\ref{prop:tildefnpropre} et le lemme~\ref{lem:criterepropre}, i), $\tilde{f}_{n}^{-1}(V)$ est compacte. 
Le morphisme $\cA[T_{1},\dotsc,T_{n}] \to \cB[T_{1},\dotsc,T_{n}]$ induit un morphisme $\cK(V) \to \cK(\tilde{f}_{n}^{-1}(V))$ et un morphisme born\'e d'anneaux de Banach $\cB(V) \to \cB(\tilde{f}_{n}^{-1}(V))$. Ces morphismes s'ins\'erent dans un diagramme commutatif
\[\begin{tikzcd}
\cA[T_{1},\dotsc,T_{n}] \arrow[r] \arrow[d] & \cB[T_{1},\dotsc,T_{n}] \arrow[d]\\
\cB(V) \arrow[r] & \cB(\tilde{f}_{n}^{-1}(V))\\
\end{tikzcd}.\]
Puisque~$V$ est spectralement convexe, l'image de l'application $\cM(\cB(V)) \to \E{n}{\cA}$ est \'egale \`a~$V$. On en d\'eduit que l'image de l'application $\cM(\cB(\tilde{f}_{n}^{-1}V)) \to \E{n}{\cB}$ est contenue dans~$\tilde{f}_{n}^{-1}(V)$, et donc que $\tilde f_{n}^{-1}(V)$ est spectralement convexe, d'apr\`es la proposition~\ref{crit_spectral}.

Passons maintenant \`a la d\'emonstration de l'isomorphisme. Le morphisme de $\cA$-alg\`ebres $\cA[T_1,\ldots,T_n]\to\cB(V)$ induit un morphisme de $\cB$-alg\`ebres $\cB[T_1,\ldots,T_n]\to \cB(V) \hosp{\cA} \cB$. Consid\'erons le morphisme induit $\cM( \cB(V) \hosp{\cA} \cB) \to \E{n}{\cB}$.

D'apr\`es la proposition \ref{crit_spectral_pu}, il suffit de d\'emontrer les deux propri\'et\'es suivantes~:
\begin{enumerate}[i)]
\item l'image du morphisme $\cM(\cB(V) \hosp{\cA} \cB)\to\E{n}{\cB}$ est contenue dans~$\tilde{f}_{n}^{-1}(V)$ ;
\item pour toute~$\cB$-alg\`ebre de Banach uniforme~$\cC$ et tout morphisme de $\cB$-alg\`ebres $\cB[T_1,\dotsc,T_n]\to\cC$ tel que l'image du morphisme induit $\cM(\cC)\to\E{n}{\cB}$ soit contenue dans~$\tilde{f}_{n}^{-1}(V)$, le morphisme~$\cB[T_1,\dotsc,T_n]\to\cC$ se factorise de mani\`ere unique par $\cB(V) \hosp{\cA} \cB$. 
\end{enumerate}

Pour d\'emontrer~i), il suffit de v\'erifier que l'image du morphisme $\cM(\cB(V) \hosp{\cA} \cB)\to\E{n}{\cA}$ est contenue dans~$V$. Or le morphisme de $\cA$-alg\`ebres $\cA[T_1,\ldots,T_n]\to\cB(V) \hosp{\cA} \cB$ 
se factorise par~$\cB(V)$. Puisque~$V$ est spectralement convexe, le r\'esultat s'ensuit.

D\'emontrons maintenant~ii). Soient~$\cC$ une~$\cB$-alg\`ebre de Banach uniforme et $\cB[T_1,\ldots,T_n]\to\cC$ un morphisme de~$\cB$-alg\`ebres tel que le morphisme induit $\cM(\cC)\to\E{n}{\cB}$ se factorise par~$\tilde{f}_{n}^{-1}(V)$. 
Le morphisme~$\cM(\cC)\to\E{n}{\cA}$ se factorise alors par~$V$. Puisque~$V$ est spectralement convexe, la proposition~\ref{crit_spectral_pu} assure que le morphisme $\cA[T_1,\ldots,T_n]\to\cC$ se factorise par~$\cB(V)$, et ce de fa\c con unique. En tensorisant par~$\cB$ au-dessus de~$\cA$, on en d\'eduit une factorisation de~$\cB[T_1,\ldots,T_n]\to\cC$ par~$\cB(V)\otimes_{\cA}\cB$, elle aussi unique. Puisque l'alg\`ebre~$\cC$ est compl\`ete et uniforme, le morphisme $\cB(V)\otimes_{\cA}\cB \to \cC$  se factorise de mani\`ere unique par $\cB(V) \hosp{\cA} \cB$, ce qui cl\^ot la d\'emonstration.
\end{proof}

On peut, comme de coutume, exprimer la fibre d'un morphisme au-dessus d'un point \`a l'aide d'un produit fibr\'e.

\begin{prop}\label{preimage}\index{Morphisme analytique!fibre d'un}\index{Morphisme analytique!associe a un point@associ\'e \`a un point}\index{Produit!fibre@fibr\'e}
 Soient~$\varphi\colon X\to Y$ un morphisme d'espaces~$\cA$-analytiques et~$y$ un point de~$Y$. Consid\'erons le morphisme $\lambda_{y} \colon \cM(\cH(y)) \to Y$ de l'exemple~\ref{ex:morphismex} et le produit fibr\'e 
 \[X\times_Y \cM(\cH(y)) := (X \ho{\cA} \cH(y)) \times_{Y \ho{\cA} \cH(y)} \cM(\cH(y))\] 
 de la proposition~\ref{prop:produitfibreAnA}.  Alors, la projection $p_{X} \colon X\times_Y \cM(\cH(y)) \to X$ induit un hom\'eomorphisme 
 \[X\times_Y \cM(\cH(y)) \simtoo \varphi^{-1}(y).\]
 \end{prop}
\begin{proof}
Apr\`es extension des scalaires de~$\cA$ \`a~$\cH(y)$, on obtient le morphisme 
\[\varphi_{\cH(y)} := \varphi \ho{\cA} \cM(\cH(y)) \colon X \ho{\cA} \cM(\cH(y)) \too Y \ho{\cA} \cM(\cH(y)).\]
D'apr\`es le lemme~\ref{lem:morphismexHx}, le morphisme $\cM(\cH(y)) \to Y \ho{\cA} \cM(\cH(y))$ induit par~$\lambda_{y}$ est une immersion ferm\'ee et, en notant~$y'$ l'unique point de son image, on a un isomorphisme canonique $\cH(y) \xrightarrow[]{\sim}\cH(y')$. On d\'eduit alors du  lemme~\ref{lem:produitouverts} que le morphisme 
\[(X \ho{\cA} \cM(\cH(y))) \times_{Y \ho{\cA} \cM(\cH(y))} \cM(\cH(y)) \too \varphi_{\cH(y)}^{-1}(y')\]
induit par la premi\`ere projection est un hom\'eomorphisme. Il suffit donc de montrer que l'application $\varphi_{\cH(y)}^{-1}(y') \to \varphi^{-1}(y)$ induite par le morphisme $\pi_{\cH(y)} \colon X \ho{\cA} \cH(y) \to X$ est un hom\'eomorphisme. Le morphisme $\pi_{\cH(y)}$ est continu et propre, 
d'apr\`es le corollaire~\ref{stabilit\'e_propre_extension_scalaire}. Par cons\'equent, il suffit de montrer que l'application $\varphi_{\cH(y)}^{-1}(y') \to \varphi^{-1}(y)$ est bijective.

Quitte \`a restreindre~$Y$, on peut supposer qu'il s'envoie par une immersion dans un espace affine analytique~$\E{m}{\cA}$, et m\^eme qu'il co\"incide avec cet espace affine. Notons $T_{1},\dotsc,T_{m}$ les coordonn\'ees sur~$\E{m}{\cA}$. 

On peut \'egalement raisonner localement sur~$X$. Quitte \`a restreindre~$X$, on peut supposer que c'est un ferm\'e analytique 
d'un ouvert~$U$ d'un espace affine~$\E{n}{\cA}$. 
On peut \'egalement supposer que le morphisme~$\varphi$ s'\'etend en un morphisme $\tilde \varphi \colon X \to Y = \E{m}{\cA}$.

Soit~$V$ un voisinage spectralement convexe de~$x$. Soit $z \in V\cap \varphi^{-1}(y)$. Le morphisme
\[\cA[T_1,\ldots,T_{n'}]\too \cB(V) \too \cH(z)\]
se factorise par~$\cH(y)$ et induit donc un morphisme
\[\sigma_{z} \colon \cB(V) \hosp{\cA} \cH(y) \too \cH(z),\]
par la propri\'et\'e universelle du produit tensoriel. Or, d'apr\`es la proposition~\ref{extension_scalaire_spectral}, $\tilde f_{n}^{-1}(V)$ est spectralement convexe et on a un isomorphisme $\cB(\tilde f_{n}^{-1}(V)) \simeq \cB(V) \hosp{\cA} \cH(y)$. Par cons\'equent, le morphisme~$\sigma_{z}$ d\'efinit un point~$\sigma(z)$ de~$\tilde f_{n}^{-1}(V)$. On v\'erifie qu'il appartient \`a~$\varphi_{\cH(y)}^{-1}(y')$ et que $\pi_{\cH(y)}(\sigma(z)) = z$. On v\'erifie \'egalement que, pour tout point~$t$ de $\tilde f_{n}^{-1}(V) \cap \varphi_{\cH(y)}^{-1}(y')$, on a $\sigma(\pi_{\cH(y)}(t)) = t$. Ceci conclut la preuve.
\end{proof}

Passons maintenant \`a l'\'etude du produit fibr\'e.

\begin{prop}\label{produit_spectralement_convexe}\index{Partie!spectralement convexe}\index{Produit!fibre@fibr\'e}\index{Produit tensoriel compl\'et\'e}
Soient $n,m,N\in\N$. Soit~$U'$ (resp.~$V'$, resp.~$W'$) un ouvert de~$\E{n}{\cA}$ (resp.~$\E{m}{\cA}$, resp.~$\E{N}{\cA}$). Soit~$U$ (resp.~$V$, resp.~$W$) un ensemble compact spectralement convexe de~$U'$ (resp.~$V'$, resp.~$W'$). Soient $\varphi \colon U' \to W'$ et $\psi \colon V'\to W'$ des morphismes d'espaces $\cA$-analytiques tels que $\varphi(U) \subset W$ et $\psi(V) \subset W$. On a naturellement
\[ U'\times_{W'} V' \subset U'\times_{\cA} V' \subset \E{n}{\cA}\times_{\cA} \E{m}{\cA} \simeq \E{n+m}{\cA}.\]
Notons $p_{U'} \colon U' \times_{W'} V' \to U'$ et $p_{V'} \colon U' \times_{W'} V' \to V'$ les deux projections. 

Alors, l'ensemble
\[ U \times_{W} V := p_{U'}^{-1}(U) \cap p_{V'}^{-1}(V),\]
est une partie compacte spectralement convexe de~$\E{n+m}{\cA}$ et le morphisme canonique
\[\cB(V) \hosp{\cB(W)} \cB(U) \simtoo \cB(U\times_{W} V).\]
est un isomorphisme isom\'etrique. 
\end{prop}
\begin{proof}
La suite d'inclusions suivie d'un isomorphisme provient du th\'eor\`eme~\ref{produit_fibr\'e}, du lemme~\ref{lem:produitouverts} et de la proposition~\ref{prop:produitaffines}.

Notons $T_{1},\dotsc,T_{n}$ (resp. $S_{1},\dotsc,S_{m}$, resp. $Z_{1},\dotsc,Z_{N}$) les coordonn\'ees sur~$\E{n}{\cA}$ (resp.~$\E{m}{\cA}$, resp.~$\E{N}{\cA}$). 
Notons $p_{\cA,U'} \colon U' \times_{\cA} V' \to U'$ et $p_{\cA,V'} \colon U' \times_{\cA} V' \to V'$ les deux projections. Posons
\[ U \times_{\cA} V := p_{\cA,U'}^{-1}(U) \cap p_{\cA,V'}^{-1}(V).\]
Pour tout $i\in \cn{1}{N}$, posons 
\[ f_{i} := (\varphi\circ p_{\cA,U'})^\sharp(Z_{i}) \textrm{ et } g_{i} := (\psi\circ p_{\cA,V'})^\sharp(Z_{i})
\textrm{ dans } \cO(U'\times_{\cA} V') .\]
D'apr\`es la proposition~\ref{prop:produitfibreZaffine}, $U'\times_{W'} V' = U'\times_{\E{N}{\cA}} V'$ est le ferm\'e analytique de~$U'\times_{\cA} V'$ d\'efini par l'id\'eal $(f_{1}-g_{1},\dotsc,f_{N}-g_{N})$.

Montrons, tout d'abord, que $U\times_{W} V$ est compact. Puisque c'est un ferm\'e de~$U\times_{\cA} V$, il suffit de montrer que ce dernier l'est. Puisque~$U$ et~$V$ sont compacts, il existe $r_{1},\dotsc,r_{n},s_{1},\dotsc,s_{m}\in \R_{>0}$ tels que $U \subset \overline{D}_{\cM(\cA)}(r_{1},\dotsc,r_{n})$ et~$V \subset \overline{D}_{\cM(\cA)}(s_{1},\dotsc,s_{m})$. Alors, $U \times_{\cA} V$ est une partie ferm\'ee de~$\E{n+m}{\cA}$ qui est contenue dans le compact $\overline{D}_{\cM(\cA)}(r_{1},\dotsc,r_{n},s_{1},\dotsc,s_{m})$. Elle est donc compacte.

Montrons, maintenant, que~$U\times_{W} V$ est spectralement convexe. Le morphisme $\cA[T_1,\dotsc,T_n]\to \cB(U\times_{W} V)$ s'\'etend en un morphisme born\'e d'alg\`ebres de Banach~$\cB(U)\to\cB(U\times_{W} V)$. Par cons\'equent, l'image du morphisme $\cM(\cB(U\times_{W} V)) \to \E{n+m}{\cA}$ induit par $\cA[T_1,\dotsc,T_n,S_{1},\dotsc,S_{m}]\to \cB(U\times_{W} V)$ est contenue dans~$p_{U'}^{-1}(U)$. En raisonnant de fa\c con similaire, on montre que cette image est \'egalement contenue dans~$p_{V'}^{-1}(V)$, et donc dans~$U\times_{\cA} V$. En outre, pour tout $i\in \cn{1}{N}$, on a $f_{i} - g_{i} = 0$ sur $U'\times_{W'} V'$ et donc dans $\cB(U\times_{W} V)$. On en d\'eduit que l'image du morphisme $\cM(\cB(U\times_{W} V)) \to \E{n+m}{\cA}$ est contenue dans $U\times_{W} V$, ce qui entra\^ine que~$U\times_{W} V$ est spectralement convexe, d'apr\`es la proposition~\ref{spectralement_conv}.

Passons finalement \`a la d\'emonstration de l'isomorphisme. \`A partir des morphismes $\cA[T_1,\dotsc,T_n]\to \cB(U)$ et $\cA[S_1,\dotsc,S_m]\to \cB(V)$, on construit un morphisme
\[ \cA[T_1,\dotsc,T_n,S_{1},\dotsc,S_{m}] \too \cB(U) \hosp{\cB(W)} \cB(V).\]
Consid\'erons le morphisme induit $\cM(\cB(U) \hosp{\cB(W)} \cB(V)) \to \E{n+m}{\cA}$. D'apr\`es la proposition~\ref{crit_spectral_pu}, il suffit de d\'emontrer les deux propri\'et\'es suivantes~:
\begin{enumerate}[i)]
\item l'image du morphisme $\cM(\cB(U) \hosp{\cB(W)} \cB(V))\to\E{n+m}{\cA}$ est contenue dans~$U \times_{W}V$ ;
\item pour toute~$\cA$-alg\`ebre de Banach uniforme~$\cC$ et tout morphisme de $\cA$-alg\`ebres $\cA[T_1,\dotsc,T_{n},S_{1},\dotsc,S_{m}]\to\cC$ tel que l'image du morphisme induit $\cM(\cC)\to\E{n+m}{\cA}$ soit contenue dans~$U \times_{\cB(W)}V$, le morphisme~$\cA[T_1,\dotsc,T_{n},S_{1},\dotsc,S_{m}]\to\cC$ se factorise de mani\`ere unique par $\cB(U) \hosp{\cB(W)} \cB(V)$. 
\end{enumerate}
On y parvient en raisonnant d'une fa\c con similaire \`a celle adopt\'ee dans la d\'emonstration de la proposition~\ref{extension_scalaire_spectral}.
\end{proof}

\'Enon\c cons une cons\'equence du r\'esultat pr\'ec\'edent concernant les espaces topologiques sous-jacents aux produits fibr\'es d'espaces analytiques. 

\begin{prop}\label{produit_fibr\'e_propre}\index{Produit!fibre@fibr\'e}\index{Application!propre}
Soient $\varphi \colon X\to Z$ et $\psi \colon Y\to Z$ des morphismes d'espaces $\cA$-analytiques. Notons~$|X|$, $|Y|$ et~$|Z|$ les espaces topologiques sous-jacents \`a~$X$, $Y$ et~$Z$ respectivement. L'application naturelle 
\[|X\times_Z Y|\too|X|\times_{|Z|} |Y|\] 
est continue, propre et surjective.
\end{prop}
\begin{proof}
L'application $\Pi \colon |X\times_Z Y|\to|X|\times_{|Z|} |Y|$ est continue par construction. Montrons qu'elle est surjective. Soient $x \in X$, $y\in Y$ et $z\in Z$ tels que $z = \varphi(x) = \psi(y)$. 
Quitte \`a restreindre les espaces, on peut supposer que~$X$, $Y$, et~$Z$ sont des ferm\'es analytiques d'espaces affines. Notons~$T$ l'ensemble des points de~$X\times_{Z} Y$ qui se projettent respectivement sur~$x$ et~$y$ par les premi\`ere et deuxi\`eme projections. D'apr\`es les exemples~\ref{ex:pointprorationnel} et~\ref{ex:Bpoint} et le th\'eor\`eme~\ref{thm:rationnel}, tout point~$x$ d'un espace affine est spectralement convexe et on a $\cB(\{x\}) = \cH(x)$. La proposition~\ref{produit_spectralement_convexe} assure donc que~$T$ est en bijection avec $\cM(\cH(x)\hosp{\cH(z)} \cH(y))$. Or $\cH(x)\hosp{\cH(z)} \cH(y)$ n'est pas nulle, donc, d'apr\`es le th\'eor\`eme~\ref{th:MAnonvide}, $T$~n'est pas vide. On en d\'eduit que l'application~$\Pi$ est surjective.

Montrons finalement que l'application $\Pi \colon |X\times_Z Y| \to |X|\times_{|Z|} |Y|$ est propre. D'apr\`es le lemme~\ref{lem:criterepropre}, il suffit de montrer que tout point de $|X|\times_{|Z|} |Y|$ poss\`ede un voisinage compact dont l'image r\'eciproque par~$\Pi$ est compacte. Soient~$t$ un point de~$|X|\times_{|Z|} |Y|$. Il existe $x \in X$, $y\in Y$ et $z\in Z$ tels que $z = \varphi(x) = \psi(y)$ et $t$ soit l'image du point $(x,y)$ de $|X|\times |Y|$. Soit~$U$ (resp.~$V$, resp.~$W$) un voisinage de~$x$ dans~$X$ (resp.~$y$ dans~$Y$, resp.~$z$ dans~$Z$) qui est un ferm\'e analytique d'un ouvert~$U'$ (resp.~$V'$, resp.~$W'$) d'un espace affine analytique. On peut supposer que $\varphi(U) \subset W$ et que~$\varphi$ s'\'etend en un morphisme $\tilde{\varphi} \colon U' \to W'$. De m\^eme, on peut supposer que $\psi(V) \subset W$ et que~$\psi$ s'\'etend en un morphisme $\tilde{\psi} \colon V' \to W'$. Soit~$W_{0}$ un voisinage compact de~$z$ dans~$W$ qui est spectralement convexe dans~$W'$. Soit~$U_{0}$ (resp.~$V_{0}$) un voisinage compact de~$x$ dans~$U$ (resp.~$y$ dans~$V$) qui est spectralement convexe dans~$U'$ (resp.~$V'$) et tel que $\varphi(U_{0}) \subset W_{0}$ (resp. $\psi(V_{0}) \subset W_{0}$). Il suffit de montrer que l'image r\'eciproque de $U_{0} \times_{W_{0}} V_{0}$ par~$\Pi$ est compacte. Cela d\'ecoule de la proposition~\ref{produit_spectralement_convexe}. 
\end{proof}

Nous pouvons maintenant d\'emontrer la stabilit\'e des morphismes propres par changement de base.

\begin{prop}\label{stabilite_propre}\index{Produit!fibre@fibr\'e}\index{Morphisme analytique!propre}
Soit~$\varphi \colon X\to Y$ un morphisme propre d'espaces $\cA$-analytiques. Soit $\psi \colon Z\to Y$ un morphisme d'espaces $\cA$-analytiques. Alors le morphisme $\varphi_{Z} \colon X \times_{Y} Z \to Z$, d\'eduit de~$\varphi$ par changement de base \`a~$Z$, est propre. 
\end{prop}
\begin{proof}
L'application entre espaces topologiques $|\varphi| \colon |X| \to |Y|$ est propre. Par d\'efinition, la propret\'e est pr\'eserv\'ee par changement de base, donc l'application $|X| \times_{|Z|} |Y| \to |Z|$ induite par~$|\varphi|$ et~$|\psi|$ est propre. En outre, d'apr\`es la proposition~\ref{produit_fibr\'e_propre}, l'application naturelle $|X\times_YZ|\to|X|\times_{|Y|}|Z|$ est propre. Le lemme~\ref{lem:compositionpropre} assure que la compos\'ee~$|\varphi_{Z}|$ des applications pr\'ec\'edentes est propre.
\end{proof}

\section{Morphismes s\'epar\'es}\label{sec:morphismessepares}
\index{Morphisme analytique!separe@s\'epar\'e|(}

Cette section est consacr\'ee \`a la notion de morphisme s\'epar\'e.
Pr\'ecisons que les espaces analytiques que nous consid\'erons, qui s'obtiennent pas des recollements g\'en\'eraux de sous-espaces d'espaces affines, ne sont pas n\'ecessairement topologiquement s\'epar\'es.

\begin{defi}\label{def:separe}\index{Morphisme analytique!separe@s\'epar\'e|textbf}\index{Immersion!fermee@ferm\'ee}\index{Espace analytique!separe@s\'epar\'e|textbf}
Soient $X$, $Y$ des espaces $\cA$-analytiques et $\varphi \colon X\to Y$ un morphisme. 
On dit que le morphisme~$\varphi$ est \emph{s\'epar\'e} si le morphisme diagonal $\Delta_{X/Y} \colon X \to X\times_{Y} X$ est une immersion ferm\'ee.

On dit que l'espace $\cA$-analytique~$X$ est \emph{s\'epar\'e} si son morphisme structural $\pi \colon X \to \cM(\cA)$ est s\'epar\'e.
\end{defi}

\begin{exem}
D'apr\`es la proposition~\ref{prop:immersiondiagonalemodelelocal}, tout mod\`ele local est s\'epar\'e (ainsi que tout morphisme dont la source est un mod\`ele local).
\end{exem}

\begin{prop}\label{stabilite_separe}\index{Extension des scalaires}\index{Produit!fibre@fibr\'e}
Soit~$\varphi \colon X\to Y$ un morphisme s\'epar\'e d'espaces $\cA$-analytiques. 

\begin{enumerate}[i)]
\item Soit $f \colon \cA \to \cB$ un morphisme d'anneaux de Banach born\'e, o\`u~$\cB$ est un anneau de base g\'eom\'etrique. Alors le morphisme $\varphi_{\cB} \colon X \ho{\cA} \cB \to Y\ho{\cA} \cB$, d\'eduit de~$\varphi$ par extension des scalaires \`a~$\cB$, est s\'epar\'e.
\item Soit $\psi \colon Z\to Y$ un morphisme d'espaces $\cA$-analytiques. Alors le morphisme $\varphi_{Z} \colon X \times_{Y} Z \to Z$, d\'eduit de~$\varphi$ par changement de base \`a~$Z$, est s\'epar\'e. 
\end{enumerate}
\end{prop}
\begin{proof}
i) Par hypoth\`ese, le morphisme diagonal $\Delta_{X/Y} \colon X \to X\times_{Y} X$ est une immersion ferm\'ee. Par la propri\'et\'e universelle de l'extension des scalaires, l'extension de~$\Delta_{X/Y}$ \`a~$\cB$ s'identifie au morphisme diagonal
\[\Delta_{X\ho{\cA} \cB/Y\ho{\cA} \cB} \colon X \ho{\cA} \cB \too (X\ho{\cA} \cB)\times_{Y\ho{\cA} \cB} (X\ho{\cA} \cB).\]
D'apr\`es le lemme~\ref{lem:extensionBouverts}, iii), c'est encore une immersion ferm\'ee. Le r\'esultat s'ensuit.

ii) Le r\'esultat se d\'emontre comme i), en utilisant le lemme~\ref{lem:extensionBouverts} au lieu du lemme~\ref{lem:changementbaseimmersion}, 
\end{proof}

Int\'eressons-nous maintenant \`a la version topologique de la s\'eparation.

\begin{defi}\index{Application!separee@s\'epar\'ee|textbf}
Soient~$T$, $T'$ des espaces topologiques. Une application $f \colon T\to T'$ est dite \emph{s\'epar\'ee} si l'image du morphisme diagonal $T \to T\times_{T'} T$ est ferm\'ee.
\end{defi}


La caract\'erisation suivante se d\'emontre sans difficult\'es.

\begin{lemm}
Soient~$T$, $T'$ des espaces topologiques. Une application $f \colon T\to T'$ est s\'epar\'ee si, et seulement si, pour tout $t'\in T'$ et tous $t_{1} \ne t_{2} \in f^{-1}(t')$, il existe des voisinages~$U_{1}$ de~$t_{1}$ et et~$U_{2}$ de~$t_{2}$ dans~$T$ tels que $U_{1} \cap U_{2} = \emptyset$.

En particulier, si~$T$ est s\'epar\'e, alors toute application $f \colon T \to T'$ est s\'epar\'ee.
\qed
\end{lemm}

%

\begin{rema}\label{rem:separelocalaubut}
La s\'eparation est une propri\'et\'e locale au but.
\end{rema}

\begin{lemm}\label{lem:compositionsepare}
Soient $f\colon T \to T'$ et $g\colon T' \to T''$ des applications continues entre espaces topologiques.
\begin{enumerate}[i)]
\item Si $f$ et $g$ sont s\'epar\'ees, alors $g\circ f$ est s\'epar\'ee.
\item Si $g\circ f$ est s\'epar\'ee, alors $f$ est s\'epar\'ee.
\end{enumerate} 
\qed
\end{lemm}



Les notions de s\'eparation analytique et topologique sont reli\'ees.

\begin{prop}\label{prop:separeantop}
Soit $\varphi \colon X\to Y$ un morphisme d'espaces $\cA$-analytiques. Alors, $\varphi$ est s\'epar\'e si, et seulement si, le morphisme d'espaces topologiques sous-jacent $|\varphi| \colon |X| \to |Y|$ est s\'epar\'e. 
\end{prop}
\begin{proof}
Supposons que $\varphi$ est s\'epar\'e. Le morphisme diagonal $\Delta_{X/Y}$ est alors une immersion ferm\'ee. En particulier, son image est ferm\'ee. Consid\'erons le diagramme commutatif
\[\begin{tikzcd}
\vert X \vert \arrow[r, "\vert\Delta_{X/Y}\vert"] \arrow[rd, "\Delta_{\vert X\vert/\vert Y\vert}"'] & \vert X\times_{Y} X\vert \arrow[d, "\Pi"]\\
& \vert X\vert\times_{\vert Y\vert} \vert X\vert
\end{tikzcd}.\]
On a 
\[ \Im(\Delta_{|X|/|Y|}) = \Pi( \Im(|\Delta_{X/Y}|)).\]
D'apr\`es la proposition~\ref{produit_fibr\'e_propre}, l'application~$\Pi$ est propre, donc ferm\'ee. Par cons\'equent, $\Im(\Delta_{|X|/|Y|})$ est ferm\'ee.

R\'eciproquement, supposons que~$|\varphi|$ est s\'epar\'e. Nous allons montrer que $\Delta_{X/Y}$ est une immersion ferm\'ee. D'apr\`es le corollaire~\ref{cor:immersiondiagonale}, $\Delta_{X/Y}$ est une immersion, donc, d'apr\`es la proposition~\ref{immersion_ferm\'e2}, il suffit de montrer que son image est ferm\'ee.

Soit $t \in (X\times_{Y} X) \setminus \Delta_{X/Y}(X)$. Alors, avec les notations de la d\'efinition~\ref{def:separe}, on a $p_{1}(t) \ne p_{2}(t)$. Puisque $|\varphi|$ est s\'epar\'e, il existe un voisinage ouvert~$U_{1}$ de~$p_{1}(t)$ dans~$X$ et un voisinage ouvert~$U_{2}$ de~$p_{2}(t)$ dans~$X$ tels que $U_{1}\cap U_{2} = \emptyset$. L'ouvert $p_{1}^{-1}(U_{1}) \cap p_{2}^{-1}(U_{2})$ de~$X\times_{Y} X$ contient alors~$t$ sans rencontrer $\Delta_{X/Y}(X)$. On en d\'eduit que $(X\times_{Y} X) \setminus \Delta_{X/Y}(X)$ est ouvert, et donc que $\Delta_{X/Y}(X)$ est ferm\'e.
\end{proof}

Terminons en consid\'erant les graphes de morphismes.

\begin{defi}\index{Morphisme analytique!graphe d'un}%
\nomenclature[Kr]{$\Gamma_{\varphi}$}{graphe d'un morphisme d'espaces $\cA$-analytiques~$\varphi$}
Soient $X$, $Y$, $Z$ des espaces $\cA$-analytiques et $\varphi \colon X\to Y$ un morphisme au-dessus de~$Z$.  Notons $p_X \colon X\times_{Z} Y\to X$ et $p_Y \colon X\times_{Z} Y\to Y$ les deux projections naturelles. On appelle \emph{graphe du morphisme~$\varphi$} l'unique morphisme $\Gamma_{\varphi} \colon X \to X\times_{Z} Y$ qui fait commuter le diagramme
\[\begin{tikzcd}
X \arrow[rd, "\Gamma_{\varphi}"] \arrow[rrd, bend left, "\varphi"] \arrow[rdd, bend right, "\id_{X}"'] & &\\
&X\times_{Z} Y \arrow[r, "p_{Y}"] \arrow[d, "p_{X}"']& Y \arrow[d]\\
& X \arrow[r]& Z.
\end{tikzcd}\]
\end{defi}

\begin{prop}\label{prop:grapheimmersion}\index{Morphisme analytique!graphe d'un} \index{Immersion}\index{Immersion!fermee@ferm\'ee}
Soient $X$, $Y$, $Z$ des espaces $\cA$-analytiques et $\varphi \colon X\to Y$ un morphisme au-dessus de~$Z$. Alors, le morphisme $\Gamma_{\varphi} \colon X \to X\times_{Z} Y$ est une immersion. 

Si le morphisme $Y \to Z$ est s\'epar\'e, alors $\Gamma_{\varphi}$ est une immersion ferm\'ee.
\end{prop}
\begin{proof}
On a le diagramme cart\'esien suivant dans la cat\'egorie des espaces $\cA$-analytiques~:
\[\begin{tikzcd}
X \arrow[r, "\Gamma_{\varphi}"] \arrow[d,  "\varphi"] & X \times_{Z} Y \arrow[d, "{(\id_{X},\varphi)}"] \\
Y \arrow[r, "\Delta_{Y/Z}"] & Y\times_{Z} Y.
\end{tikzcd}\]
D'apr\`es le corollaire~\ref{cor:immersiondiagonale}, $\Delta_{Y/Z}$ est une immersion. Si le morphisme $Y \to Z$ est s\'epar\'e, c'est m\^eme une immersion ferm\'ee, par d\'efinition. Le r\'esultat d\'ecoule maintenant du lemme~\ref{lem:changementbaseimmersion}.
\end{proof}

\index{Morphisme analytique!separe@s\'epar\'e|)}


\chapter{\'Etude des morphismes finis}\label{chap:fini}\index{Morphisme analytique!fini|(}

Ce chapitre est consacr\'e aux morphismes finis et \`a leurs applications.

Dans la section~\ref{sec:defpropfini}, nous d\'efinissons les morphismes finis d'espaces analytiques en \'etablissons quelques propri\'et\'es \'el\'ementaires. Nous nous int\'eressons ensuite aux analogues de certains r\'esultats classiques dans notre contexte. 

Dans la section~\ref{sec:imagedirectefinie}, nous montrons que l'image directe d'un faisceau coh\'erent par un morphisme fini est un faisceau coh\'erent (\cf~th\'eor\`eme~\ref{thm:fini}). La preuve fournit \'egalement un crit\`ere de finitude locale~: pour qu'un morphisme soit fini au voisinage d'un point, il faut et il suffit que ce point soit isol\'e dans sa fibre (\cf~th\'eor\`eme~\ref{thm:locfini}). Nous nous attendons, plus g\'en\'eralement, \`a ce que la coh\'erence soit pr\'eserv\'ee par image directe par un morphisme propre. Nous pr\'evoyons de traiter cette question dans un prochain travail.

Dans la section~\ref{sec:stabilitemorphismesfinis}, nous montrons que les morphismes finis sont stables par changement de base (\cf~proposition~\ref{stabilite_fini}).

Dans la section~\ref{sec:chgtbasefini}, nous \'etablissons une propri\'et\'e de changement de base fini pour les faisceaux coh\'erents. \'Etant donn\'e un carr\'e cart\'esien correspondant au produit fibr\'e d'un morphisme fini par un morphisme quelconque, les op\'erations de tirer en arri\`ere puis pousser en avant ou de pousser en avant puis tirer en arri\`ere fournissent des faisceaux coh\'erents isomorphes (\cf~corollaire~\ref{thm:formule_changement_base}). Le m\^eme r\'esultat vaut en rempla\c{c}ant ci-dessus le morphisme quelconque par une extension des scalaires (\cf~corollaire~\ref{thm:formule_extension_scalaire}). 

Dans la section~\ref{sec:Ruckert}, nous d\'emontrons une variante du Nullstellensatz, dite Nullstellensatz de R\"uckert (\cf~th\'eor\`eme~\ref{thm:Ruckert} et corollaire~\ref{cor:IVJ}), \'etablissant un lien entre fonctions nulles en tout point et fonctions nilpotentes. Elle est utilis\'ee dans la section finale~\ref{sec:ouverture} pour montrer des crit\`eres d'ouverture de morphismes finis, par exemple le fait qu'un morphisme fini et plat est ouvert (\cf~corollaire~\ref{plat}).

\medbreak

Fixons~$(\cA,\nm)$ un anneau de base g\'eom\'etrique. Posons $B:=\cM(\cA)$.

\section{D\'efinition et premi\`eres propri\'et\'es}\label{sec:defpropfini}

Dans cette section, nous d\'efinissons les morphismes finis et rappelons quelques r\'esultats classiques.

\begin{defi}\label{def:fini}\index{Application!finie|textbf}\index{Application!quasi-finie|textbf}\index{Application!separee@s\'epar\'ee}
Une application $f$ entre espaces topologiques est dite \emph{quasi-finie} si ses fibres sont finies. Elle est dite \emph{finie} si elle est, en outre, continue, s\'epar\'ee et ferm\'ee.

\end{defi}

\begin{rema}\label{rem:finilocalaubut}
La finitude est une propri\'et\'e locale au but.
\end{rema}

\begin{lemm}\label{lem:voisinagefibre}\index{Application!fermee@ferm\'ee}
Soit $f \colon T \to T'$ une application continue et ferm\'ee entre espaces topologiques. Soient $t' \in T'$ et~$\cV$ une base de voisinages de~$t'$ dans~$T'$. Alors 
\[\{ f^{-1}(V) : V \in \cV\}\]
est une base de voisinages de~$f^{-1}(t')$ dans~$T$.
\qed
\end{lemm}

Le r\'esultat suivant est une cons\'equence directe de \cite[I, \S 10, \no 2, th\'eor\`eme~1]{BourbakiTG14}.

\begin{prop}\label{prop:finipropre}\index{Application!propre}
Toute application finie entre espaces topologiques est propre.

En particulier, une application quasi-finie est finie si, et seulement si, elle est s\'epar\'ee et propre.
\qed
\end{prop}

\'Enon\c cons quelques propri\'et\'es des applications finies vis-\`a-vis de la composition. 

\begin{lemm}\label{lem:compositionfini}
Soient $f\colon T \to T'$ et $g\colon T' \to T''$ des applications continues entre espaces topologiques.
\begin{enumerate}[i)]
\item Si $f$ et $g$ sont quasi-finies (resp. finies), alors $g\circ f$ est quasi-finie (resp. finie).
\item Si $g\circ f$ est quasi-finie et $f$ est surjective, alors $g$ est quasi-finie.
\item Si $g\circ f$ et quasi-finie (resp. finie) et $g$ est injective, alors $f$ est quasi-finie (resp. finie).
\end{enumerate} 
\end{lemm}
\begin{proof}
Les r\'esultats se d\'emontrent ais\'ement pour la propri\'et\'e de quasi-finitude.
En utilisant la proposition~\ref{prop:finipropre}, on se ram\`ene \`a d\'emontrer les propri\'et\'es correspondantes pour les applications s\'epar\'ees et propres, et l'on conclut par les lemmes~\ref{lem:compositionpropre} et~\ref{lem:compositionsepare}.
\end{proof}

\begin{defi}\index{Morphisme analytique!fini|textbf}\index{Morphisme analytique!quasi-fini|textbf}\index{Morphisme analytique!fini en un point|textbf}\index{Morphisme analytique!quasi-fini|see{Morphisme analytique fini}}\index{Morphisme analytique!fini en un point|see{Morphisme analytique fini}}\index{Morphisme analytique!separe@s\'epar\'e}
Soit $\varphi \colon X \to Y$ un morphisme d'espaces $\cA$-analytiques. On dit que~$\varphi$ est  \emph{quasi-fini} (resp. \emph{fini}) si l'application induite entre les espaces topologiques sous-jacents est quasi-finie (resp. finie). 

Pour tout point~$x$ de~$X$, on dit que~$\varphi$ est \emph{fini en~$x$} s'il existe un voisinage~$U$ de~$x$ dans~$X$ et un voisinage~$V$ de~$\varphi(x)$ dans~$Y$ tel que $\varphi(U)\subset V$ et le morphisme $\varphi_{U,V} \colon U\to V$ induit par~$\varphi$ soit fini.
\end{defi}

\begin{lemm}\label{lem:compositionfinienx}
Soient $\varphi \colon X \to Y$ et $\psi \colon Y \to Z$ des morphismes d'espaces $\cA$-analytiques. Soit~$x \in X$. Si $\varphi$ et fini en~$x$ et $\psi$ est fini en~$\varphi(x)$, alors $\psi \circ \varphi$ est fini en~$x$.
\qed
\end{lemm}


\begin{exem}\label{ex:fini}\index{Endomorphisme de la droite}
Soit $P \in \cA[T]$ un polyn\^ome unitaire non constant. Alors le morphisme $\varphi_{P} \colon \E{1}{\cA} \to \E{1}{\cA}$ d\'efini par~$P$ (\cf~exemple~\ref{exemple_morphisme}) est fini.

En effet, $\varphi_{P}$ est clairement quasi-fini et s\'epar\'e, puisque $ \E{1}{\cA}$ est s\'epar\'e. En outre, pour tout $r\in \R_{\ge 0}$, l'ensemble
\[\varphi_{P}^{-1}(\overline{D}_{B}(r)) = \overline{D}_{B}(P,r)\]
est compact (\cf~\cite[Corollaire~1.1.12]{A1Z}). Par cons\'equent, d'apr\`es le lemme~\ref{lem:criterepropre}, le morphisme~$\varphi_{P}$ est propre. 
\end{exem}

L'\'enonc\'e suivant se d\'eduit sans difficult\'es du lemme~\ref{lem:voisinagefibre}. 

\begin{prop}\label{prop:fetoileenbasexact}
Soit $\varphi \colon X \to Y$ un morphisme fini s\'epar\'e d'espaces $\cA$-analytiques. 

\begin{enumerate}[i)]
\item Pour tout faisceau de~$\cO_{X}$-modules~$\cF$ et tout $y\in Y$, le morphisme naturel
\[(\varphi_{\ast}\cF)_{y} \too \prod_{x\in \varphi^{-1}(y)} \cF_{x}\]
est un isomorphisme de~$\cO_{Y,y}$-modules.

\item Le foncteur~$\varphi_{\ast}$  de la cat\'egorie des~$\cO_{X}$-modules dans celle des $\cO_{Y}$-modules est exact.
\end{enumerate}
\qed
\end{prop}

Rappelons qu'un faisceau de groupes ab\'eliens sur un espace topologique poss\`ede une r\'esolution flasque (par exemple, la r\'esolution canonique de Godement d\'efinie dans \cite[II, \S 4.3]{GodementTATF}). L'existence de ces r\'esolutions permet de d\'emontrer le r\'esultat suivant (\cf~par exemple \cite[I, \S 1, theorem~5]{Gr-Re} pour les d\'etails).\footnote{La structure analytique est ici totalement superflue et le r\'esultat vaut encore pour un morphisme fini entre espaces topologiques et un faisceau de groupes ab\'eliens.}

\begin{theo}\label{th:Hqfini}
Soit $\varphi \colon X \to Y$ un morphisme fini d'espaces $\cA$-analytiques. Pour tout faisceau de~$\cO_{X}$-modules~$\cF$ et tout $q\in \N$, on a un isomorphisme naturel
\[ H^q(X,\cF) \simeq H^q(Y,\varphi_{\ast} \cF).\]
\qed
\end{theo}

\section{Images directes des faisceaux coh\'erents}\label{sec:imagedirectefinie}

Le but de cette section est de montrer que si~$\varphi$ est un morphisme fini d'espaces analytiques, le foncteur~$\varphi_*$ pr\'eserve les faisceaux coh\'erents. 

\begin{theo}\label{thm:fini}\index{Morphisme analytique!fini!image d'un faisceau par un}\index{Faisceau!coherent@coh\'erent}
\index{Faisceau!image d'un|see{Immersion ferm\'ee et Morphisme fini}}
Soient~$\varphi \colon X\to Y$ un morphisme fini d'espaces $\cA$-analytiques et~$\cF$ un faisceau coh\'erent sur~$X$. Alors, le faisceau~$\varphi_*\cF$ est coh\'erent.
\end{theo}

D'un point de vue intuitif, cela signifie que la finitude topologique entra\^ine la finitude alg\'ebrique. Le point-cl\'e consiste \`a montrer qu'un morphisme fini ressemble localement \`a la projection d'un espace affine sur un sous-espace de coordonn\'ees, à composition par des immersions fermées près, \`a la source et au but (cf~lemme~\ref{lem:finiprojection}). Notre strat\'egie de preuve reprend celle de \cite[\S 3.1]{Gr-Re2}.

Les m\^emes m\'ethodes permettent de d\'emontrer un crit\`ere qui relie la propri\'et\'e d'\^etre fini au voisinage d'un point \`a celle d'\^etre isol\'e dans sa fibre (\cf~th\'eor\`eme~\ref{thm:locfini}).

\medbreak

Nous commen\c cons par traiter plusieurs cas particuliers du th\'eor\`eme~\ref{thm:fini}. Le premier d\'ecoule directement des d\'efinitions. 

\begin{lemm}\label{lem:immersionfermeecoherence}\index{Immersion!fermee@ferm\'ee!image d'un faisceau par une}
Soit $\iota \colon X \to X'$ une immersion ferm\'ee d'espaces $\cA$-analytiques. Alors, un faisceau de~$\cO_{X}$-modules~$\cF$ est coh\'erent si, et seulement si, $\iota_{\ast}\cF$ est coh\'erent. 
\qed
\end{lemm} 

\begin{lemm}\label{cas_particulier_coh\'erence}
Posons $X := \E{1}{\cA}$ avec coordonn\'ee~$T$ et notons $\pi \colon X \to B$ le morphisme structural. Soit~$\omega\in\cA[T]$ un polyn\^ome unitaire de degr\'e~$d \ge 1$. Notons~$Z$ le ferm\'e analytique d\'efini par~$\omega$. Alors,
\begin{enumerate}[i)]
\item le morphisme~$\pi_{|Z} \colon Z \to B$ est un morphisme fini ;
\item le morphisme
\[\begin{array}{ccc}
\cO_B^d & \too & (\pi_{Z})_{*} \cO_{Z}\\
(a_{0},\dotsc,a_{d-1}) & \mapstoo & \disp \sum_{i=0}^{d-1} a_{i}\, T^i
\end{array}\] 
est un isomorphisme~;
\item pour tout faisceau coh\'erent~$\cF$ sur~$Z$, le faisceau~$(\pi_{|Z})_*\cF$ est un faisceau coh\'erent sur~$B$.
\end{enumerate}
\end{lemm}
\begin{proof}
i) Puisque~$\omega$ est unitaire, le ferm\'e analytique~$Z$ est compact (\cf~\cite[corollaire~1.1.12]{A1Z}). Par cons\'equent, d'apr\`es le lemme~\ref{lem:criterepropre}, le morphisme~$\pi_{|Z}$ est propre. En outre, pour tout point~$b$ de~$B$, l'image r\'eciproque $\pi_{|Z}^{-1}(b)$ s'identifie \`a l'ensemble des racines du polyn\^ome unitaire~$\omega(b)$, et elle est donc finie. En d'autres termes, $\pi_{|Z}$ est quasi-fini. Finalement, puisque $X$ est s\'epar\'e, $\pi_{|Z}$ est s\'epar\'e. On en d\'eduit que~$\pi_{|Z}$ est fini.

\medbreak

ii) Soit~$b \in B$. Puisque~$\omega$ est unitaire de degr\'e~$d$, le morphisme
\[\begin{array}{ccc}
\cO_{B,b}^d & \too & \cO_{B,b}[T]/(\omega)\\
(a_{0},\dotsc,a_{d-1}) & \mapstoo & \disp \sum_{i=0}^{d-1} a_{i}\, T^i
\end{array}\]
est un isomorphisme. D'apr\`es le corollaire~\ref{cor:weierstrassgeneralise}, on a un isomorphisme naturel $\cO_{B,b}[T]/(\omega) \simto \prod_{i=1}^t \cO_{z_{i}}$, o\`u $z_{1},\dotsc,z_{t}$ sont les z\'eros de~$\omega(b)$ dans~$\pi^{-1}(b)$, autrement dit les points de~$\pi_{|Z}^{-1}(b)$. Le r\'esultat d\'ecoule alors de l'isomorphisme naturel $((\pi_{|Z})_{\ast}\cO_{Z})_{b} \simto  \prod_{i=1}^t \cO_{z_{i}}$ de la proposition~\ref{prop:fetoileenbasexact}.

\medbreak

iii) Soit~$\cF$ un faisceau coh\'erent sur~$Z$. Soit $b\in B$. Notons $z_{1},\dotsc,z_{t}$ les \'el\'ements de~$\pi_{|Z}^{-1}(b)$. Pour tout $i\in\cn{1}{t}$, ou peut trouver un voisinage ouvert~$U_{i}$ de~$z_i$ dans~$X$ et une suite exacte 
\[\cO_{U_{i}}^{m_i}\too \cO_{U_{i}}^{n_i}\too \cF_{|U_{i}}\too 0.\]
D'apr\`es le lemme~\ref{lem:voisinagefibre}, on peut supposer que les~$U_{i}$ sont disjoints et qu'il existe un voisinage ouvert~$V$ de~$b$ tel que $\pi_{|Z}^{-1}(V) = \bigcup_{1\le i\le t} U_{i}$. D'apr\`es la proposition~\ref{prop:fetoileenbasexact}, la suite 
\[(\pi_{|Z})_*\cO_{U_{i}}^{m_i}\too (\pi_{|Z})_*\cO_{U_{i}}^{n_i}\too (\pi_{|Z})_*\cF_{|U_{i}} \too 0\]
est encore exacte et, d'apr\`es~ii), le faisceau $(\pi_{|Z})_*\cO_{U_{i}}$ est coh\'erent. Le r\'esultat s'ensuit.
\end{proof}

\'Enon\c cons maintenant deux lemmes qui nous seront utiles \`a plusieurs reprises pour effectuer des raisonnements par r\'ecurrence.

\begin{lemm}\label{lem:finiprojection}
Soit~$\varphi \colon X\to Y$ un morphisme propre d'espaces $\cA$-analytiques et soit~$x \in X$. Posons $y := \varphi(x)$ et supposons que $\varphi^{-1}(y) =\{x\}$. 

Alors, il existe un voisinage ouvert~$U$ de~$x$ dans~$X$, un voisinage ouvert~$V$ de~$y$ dans~$Y$ tel que $\varphi^{-1}(V) = U$, des entiers $k,l \in \N$, un ouvert~$U'$ de~$\E{k}{\cA}$, un ouvert~$V'$ de~$\E{l}{\cA}$, des immersions ferm\'ees $\iota_{U} \colon U \to U'\times_{\cA} V'$ et $\iota_{V} \colon V \to V'$ tels que, en notant $p_{V'} \colon U'\times_{\cA} V' \to V'$ la seconde projection,
on ait $\iota_{V} \circ \varphi_{|U} = p_{V'} \circ \iota_{U}$.
\end{lemm}
\begin{proof}
Par d\'efinition d'espace analytique, le point~$x$ poss\`ede un voisinage ouvert~$U_{0}$ qui peut s'envoyer par une immersion ferm\'ee~$i_{U_{0}}$ dans un ouvert~$U_{0}'$ de~$\E{k}{\cA}$ pour un certain $k\in\N$. D'apr\`es le lemme~\ref{lem:voisinagefibre}, il existe un voisinage ouvert~$V$ de~$y$ tel que $\varphi^{-1}(V) \subset U_{0}$. Quitte \`a restreindre~$V$, on peut supposer qu'il existe une immersion ferm\'ee~$i_{V}$ dans un ouvert~$V'$ de~$\E{l}{\cA}$ pour un certain $l\in\N$. Posons $U :=\varphi^{-1}(V)$. Il existe alors une immersion ferm\'ee~$i_{U}$ dans un ouvert~$U'$ de~$\E{k}{\cA}$. Notons $\psi \colon U \to V$ le morphisme induit par~$\varphi$.

D'apr\`es la proposition~\ref{prop:grapheimmersion}, le morphisme $\Gamma_{\psi} \colon U \to U\times_{\cA} V$ est une immersion ferm\'ee. D'apr\`es le lemme~\ref{lem:produitouverts}, nous avons \'egalement une immersion ferm\'ee $j \colon U \times_{\cA} V \to U'\times_{\cA} V'$. En posant $\iota_{U} := j \circ \Gamma_{\psi}$ et $\iota_{V} := i_{V}$, on obtient le r\'esultat.
\end{proof}

\begin{lemm}\label{lem:recurrenceAn}
Soient $l,n \in \N$ avec $n>l$. Pour $m\in \cn{0}{n}$, notons $\varrho_{m} \colon \E{n}{\cA} \to \E{m}{\cA}$ le morphisme de projection sur les $m$~derni\`eres coordonn\'ees. Notons~$T$ la premi\`ere coordonn\'ee de~$\E{n}{\cA}$. 
Soit~$U$ un ferm\'e analytique d'un ouvert~$W'$ de~$\E{n}{\cA}$ et $V$ un ferm\'e analytique d'un ouvert~$V'$ de~$\E{l}{\cA}$ tels que $\varrho_{l}(U) = V$. 
Soient $x \in U$ et $y\in V$ tels que $\varrho_{l}^{-1}(y) \cap U=\{x\}$. 

Alors il existe un voisinage ouvert~$W'_{n-1}$ de~$\varrho_{n-1}(x)$ dans~$\varrho_{n-1}(W')$ et un polyn\^ome $\omega \in \cO(W'_{n-1})[T]$ unitaire non constant tel que $U \cap \varrho_{n-1}^{-1}(W'_{n-1})$ soit un ferm\'e analytique du ferm\'e analytique~$Z_{\omega}$ de~$\varrho_{n-1}^{-1}(W'_{n-1})$ d\'efini par~$\omega$. 
\end{lemm}
\begin{proof}
Posons $x_{n-1}:=\varrho_{n-1}(x)$. Puisque $q_{l}^{-1}(y) \cap U=\{x\}$, on a et $\varrho_{n-1}^{-1}(x_{n-1}) \cap U= \{x\}$. Par cons\'equent, il existe $g \in \cO_{\E{n}{\cA},x}$ appartenant \`a l'id\'eal de~$U$, qui s'annule en~$x$ mais n'est pas identiquement nul sur $\varrho_{n-1}^{-1}(x_{n-1})$. 
Quitte \`a restreindre~$W'$, nous pouvons supposer que $g \in \cO(W')$ et que~$U$ est contenu dans le lieu d'annulation de~$g$ sur~$W'$. Posons $W'_{n-1} := \varrho_{n-1}(W')$. D'apr\`es le corollaire~\ref{cor:projectionouverte}, c'est un ouvert de~$\E{n-1}{\cA}$. D'apr\`es le th\'eor\`eme de pr\'eparation de Weierstra\ss{} \ref{thm:preparationW}, quitte \`a restreindre encore~$W'$, on peut supposer qu'il existe $e\in \cO(W')$ inversible et $\omega \in \cO(W'_{n-1})[T]$ unitaire tel que $g = e \omega$. Posons $d := \deg(\omega)$. Les propri\'et\'es de~$g$ assurent que $d\ge 1$.

Notons $Z_{\omega}$ le ferm\'e analytique de~$\varrho_{n-1}^{-1}(W'_{n-1})$ d\'efini par~$\omega$. D'apr\`es le lemme \ref{cas_particulier_coh\'erence}, le morphisme $\varrho_{n-1}' \colon Z_{\omega}\to W'_{n-1}$ induit par~$\varrho_{n-1}$ est fini. Par cons\'equent, quitte \`a r\'eduire~$W'_{n-1}$, et remplacer $W'$ par $W' \cap \varrho_{n-1}^{-1}(W'_{n-1})$, on peut supposer que~$Z_{\omega}$ est une r\'eunion finie d'ouverts disjoints et que l'un de ces ouverts est~$Z_{\omega} \cap W'$. Dans cette situation, $Z_{\omega} \cap W'$ est un ferm\'e analytique de~$Z_{\omega}$ et, par cons\'equent, $U$ aussi.
\end{proof}

D\'emontrons un autre cas particulier du th\'eor\`eme~\ref{thm:fini}, celui o\`u le morphisme provient d'une projection d'un espaces affine analytique sur un sous-espace de coordonn\'ees.

\begin{lemm}\label{lem:pousseprojection}
Soient~$\varphi \colon X\to Y$ un morphisme d'espaces $\cA$-analytiques. Soient $x\in X$ et $y\in Y$ tels que $\varphi^{-1}(y) = \{x\}$. 
Supposons qu'il existe des entiers $n,l$ avec $n\ge l$, des ouverts~$W'$ de~$\E{n}{\cA}$ et $V'$ de~$\E{l}{\cA}$ et des immersions ferm\'ees $\iota_{X} \colon X \to W'$ et $\iota_{Y} \colon Y \to V'$ tels que, en notant $\varrho_{l} \colon \E{n}{\cA} \to \E{l}{\cA}$ la projection sur les $l$~derni\`eres coordonn\'ees, on ait $\varrho_{l}(W') = V'$ et le diagramme
\[\begin{tikzcd}
X \arrow[d, "\varphi"] \arrow[r, "\iota_{X}"] & W' \arrow[d, "\varrho_{l, |W'}"]\\
Y \arrow[r, "\iota_{Y}"] & V'
\end{tikzcd}\]
commute. 

Alors, il existe un voisinage ouvert~$Y'$ de~$y$ dans~$Y$ tel que
\begin{enumerate}[i)]
\item le morphisme $\psi \colon \varphi^{-1}(Y') \to Y'$ le morphisme induit par~$\varphi$ soit fini~;
\item pour tout faisceau coh\'erent~$\cF$ sur~$ \varphi^{-1}(Y')$, le faisceau $\psi_{\ast} \cF$ soit coh\'erent.
\end{enumerate}
\end{lemm}
\begin{proof}
D\'emontrons le r\'esultat par r\'ecurrence sur $n \ge l$. Si $n=l$, on a $W'=V'$, donc $\varrho_{l}=\id$ et $\iota_{Y} \circ \varphi = \iota_{X}$. Par cons\'equent, on a $(\iota_{Y})_{\ast} \varphi_{\ast} \cF = (\iota_{X})_{\ast} \cF$. Le r\'esultat d\'ecoule alors des lemmes~\ref{lem:compositionfini} et~\ref{lem:immersionfermeecoherence}.

Supposons maintenant que $n> l$ et que le r\'esultat est vrai pour~$n-1$. Notons $\varrho_{n-1} \colon \E{n}{\cA} \to \E{n-1}{\cA}$ la projection sur les $n-1$ derni\`eres coordonn\'ees et~$T$ la premi\`ere coordonn\'ee de~$\E{n}{\cA}$. D'apr\`es le lemme~\ref{lem:recurrenceAn}, il existe un voisinage ouvert~$W'_{n-1}$ de~$\varrho_{n-1}(x)$ dans~$\varrho_{n-1}(W')$ et un polyn\^ome $\omega \in \cO(W'_{n-1})[T]$ unitaire non constant tel que $X \cap \varrho_{n-1}^{-1}(W'_{n-1})$ soit un ferm\'e analytique du ferm\'e analytique~$Z_{\omega}$ de~$\varrho_{n-1}^{-1}(W'_{n-1})$ d\'efini par~$\omega$. Quitte \`a remplacer $W'$ par $W' \cap \varrho_{n-1}^{-1}(W'_{n-1})$ et les autres espaces en cons\'equence,
on peut supposer que $W'_{n-1} = \varrho_{n-1}(W')$.

Par hypoth\`ese, l'immersion ferm\'ee~$\iota_{X}$ peut se factoriser sous la forme $\iota_{X} = j_{Z_{\omega}} \circ j_{X}$, o\`u $j_{X} \colon X \to Z_{\omega}$ et $j_{Z_{\omega}} \colon Z_{\omega} \to \varrho_{n-1}^{-1}(W'_{n-1})$ sont deux immersions ferm\'ees. On a $\varrho_{n-1} \circ \iota_{X} = q_{n-1,|Z_{\omega}} \circ j_{X}$. D'apr\`es le lemme~\ref{cas_particulier_coh\'erence}, c'est un morphisme fini et, d'apr\`es les lemmes~\ref{lem:immersionfermeecoherence} et~\ref{cas_particulier_coh\'erence}, pour tout faisceau coh\'erent~$\cF$ sur~$X$, le faisceau 
\[(\varrho_{n-1})_{\ast}(\iota_{X})_{\ast}\cF = (q_{n-1,|Z_{\omega}})_{\ast} (j_{X})_{\ast} \cF\] 
est un faisceau coh\'erent sur~$W_{n-1}$.

Le raisonnement pr\'ec\'edent appliqu\'e avec le faisceau structural montre que $(\varrho_{n-1})_{\ast}(\iota_{X})_{\ast}\cO_{X}$ est coh\'erent. Son support $S := \varrho_{n-1}(\iota(X))$ est donc un ferm\'e analytique de~$W'_{n-1}$, d'apr\`es la remarque~\ref{rem:supportfaisceau}. En particulier, il est muni d'une structure d'espace $\cA$-analytique. Notons $\varrho'_{l} \colon \E{n-1}{\cA} \to \E{l}{\cA}$ la projection sur les $l$~derni\`eres coordonn\'ees et $\varrho' \colon S \to Y$ le morphisme induit. On a $(\varrho')^{-1}(y) = \{\varrho_{n-1}(x)\}$. Par hypoth\`ese de r\'ecurrence, quitte \`a restreindre~$S$ et~$Y$, on peut supposer que~$\varrho'$ est fini et que, pour tout faisceau coh\'erent~$\cG$ sur~$S$, le faisceau $(\varrho')_{\ast} \cG$ est coh\'erent. Le morphisme $\varphi = \varrho' \circ \varrho_{n-1} \circ \iota_{X}$ est alors fini et, pour tout faisceau coh\'erent~$\cF$ sur~$X$, le faisceau 
\[ \varphi_{\ast}\cF = 
(\varrho')_{\ast}(\varrho_{n-1})_{\ast}(\iota_{X})_{\ast}\cF\]
est coh\'erent.
\end{proof}

Tous les ingr\'edients sont maintenant en place pour d\'emontrer le r\'esultat de coh\'erence annonc\'e en toute g\'en\'eralit\'e. 



\begin{proof}[D\'emonstration du th\'eor\`eme~\ref{thm:fini}]
La coh\'erence \'etant une notion locale, il suffit de montrer le r\'esultat au voisinage de tout point de~$Y$. Soit $y\in Y$. 

Si~$y$ n'appartient pas \`a l'image de~$\varphi$, puisque~$\varphi$ est ferm\'e, il existe un voisinage~$V$ de~$y$ tel que~$\varphi^{-1}(V)$ soit vide. Par cons\'equent, $\varphi_*\cF$ est nul, et donc coh\'erent, au voisinage de~$y$.

Supposons que~$y$ appartient \`a l'image de~$\varphi$. Notons $x_{1},\dotsc,x_{t}$ ses ant\'ec\'edents par~$\varphi$. Puisque~$\varphi$ est fini, d'apr\`es le lemme~\ref{lem:voisinagefibre}, il existe un voisinage~$V$ de~$y$ tel que $\varphi^{-1}(V)$ puisse s'\'ecrire comme union disjointe finie d'ouverts $\bigsqcup_{i=1}^t U_{i}$, o\`u, pour tout $i\in \cn{1}{t}$, $U_{i}$ contient~$x_{i}$. Pour tout $i\in \cn{1}{t}$, notons $\varphi_{i} \colon U_{i} \to V$ le morphisme induit par~$\varphi$. Il est encore fini.

Il suffit de montrer que, pour tout $i\in \cn{1}{t}$, il existe un voisinage ouvert~$V_{i}$ de~$y$ dans~$V$ tel que, en notant $\psi_{i} \colon \varphi_{i}^{-1}(V_{i}) \to V_{i}$ le morphisme induit par~$\varphi_{i}$, le faisceau $(\psi_{i})_{\ast} \cF$ soit coh\'erent. En effet, dans ce cas, on peut choisir un voisinage ouvert~$W$ de~$y$ contenu dans tous les~$V_{i}$ et on a alors
\[ (\varphi_{\ast}\cF)_{|W} = \prod_{i=1}^t ((\psi_{i})_{\ast}\cF)_{|W},\]
ce qui permet de conclure. 

Soit $i\in \cn{1}{t}$. Puisque~$\varphi_{i}$ est fini, il est propre, d'apr\`es la proposition~\ref{prop:finipropre}, et le point~$x_{i}$ est isol\'e dans sa fibre. On peut donc appliquer le lemme~\ref{lem:finiprojection} avec~$\varphi_{i}$ et~$x_{i}$. On se trouve alors dans la situation du lemme~\ref{lem:pousseprojection}, qui permet de conclure.
\end{proof}

\begin{coro}\label{cor:finianneauxlocaux}
Soient~$\varphi \colon X\to Y$ un morphisme fini d'espaces $\cA$-analytiques. Soit $x\in X$. Alors
\begin{enumerate}[i)] 
\item le morphisme $\cO_{Y,\varphi(x)} \to \cO_{x}$ induit par~$\varphi$ fait de $\cO_{x}$ un $\cO_{Y,\varphi(x)}$-module de type fini ;
\item les extensions de corps $\kappa(x)/\kappa(\varphi(x))$ et $\cH(x)/\cH(\varphi(x))$ induites par~$\varphi$ sont finies.
\end{enumerate}
\end{coro}
\begin{proof}
D'apr\`es le th\'eor\`eme~\ref{thm:fini}, le faisceau $\varphi_{*}\cO_{X}$ est coh\'erent. Le point~i) d\'ecoule alors de la proposition~\ref{prop:fetoileenbasexact}, i). Le point~ii) s'en d\'eduit.
\end{proof}

\begin{coro}\label{cor:imagefini}\index{Ferme analytique@Ferm\'e analytique}\index{Morphisme analytique!fini!image d'un ferm\'e par un}
Soient~$\varphi \colon X\to Y$ un morphisme fini d'espaces $\cA$-analytiques. Alors $\varphi(X)$ est un ferm\'e analytique de~$Y$.
\end{coro}
\begin{proof}
D'apr\`es le th\'eor\`eme~\ref{thm:fini}, le faisceau $\varphi_{*}\cO_{X}$ est coh\'erent. D'apr\`es la remarque~\ref{rem:supportfaisceau}, son support, qui n'est autre que~$\varphi(X)$, est un ferm\'e analytique de~$Y$.
\end{proof}

D\'emontrons finalement un crit\`ere de finitude locale.

\begin{theo}\label{thm:locfini}\index{Morphisme analytique!fini en un point}\index{Morphisme analytique!quasi-fini}
Soient~$\varphi \colon X\to Y$ un morphisme d'espaces $\cA$-analytiques et~$x \in X$. Le morphisme~$\varphi$ est fini en~$x$ si, et seulement si, le point~$x$ est isol\'e dans sa fibre~$\varphi^{-1}(\varphi(x))$. 
\end{theo}
\begin{proof}
Suppsosons que~$\varphi$ est fini en~$x$. Alors, il existe un voisinage~$U$ de~$x$ dans~$X$ tel que $\varphi^{-1}(\varphi(x)) \cap U$ est fini. On en d\'eduit que~$x$ est isol\'e dans $\varphi^{-1}(\varphi(x))$.

\medbreak

Supposons que~$x$ est isol\'e dans~$\varphi^{-1}(\varphi(x))$. Quitte \`a restreindre~$X$, on peut supposer que $\varphi^{-1}(\varphi(x)) = \{x\}$. Dans ce cas, le lemme~\ref{lem:finiprojection} s'applique. On se trouve alors dans la situation du lemme~\ref{lem:pousseprojection}, qui permet de conclure.
\end{proof}

\section{Stabilit\'e des morphismes finis}\label{sec:stabilitemorphismesfinis}

Dans cette section, nous d\'emontrons que les morphismes finis sont stables par changement de base.

\begin{prop}\label{stabilite_fini}\index{Extension des scalaires}\index{Produit!fibre@fibr\'e}
Soit~$\varphi \colon X\to Y$ un morphisme d'espaces $\cA$-analytiques. Soit $x\in X$.

\begin{enumerate}[i)]
\item Soit $f \colon \cA \to \cB$ un morphisme d'anneaux de Banach born\'e, o\`u~$\cB$ est un anneau de base g\'eom\'etrique. Notons $\varphi_{\cB} \colon X \ho{\cA} \cB \to Y\ho{\cA} \cB$ le morphisme d\'eduit de~$\varphi$ par extension des scalaires \`a~$\cB$. Si $\varphi$ est quasi-fini (resp. fini, resp. fini en~$x$), alors $\varphi_{\cB}$ est quasi-fini (resp. fini, resp. fini en tout point au-dessus de~$x$). 
\item Soit $\psi \colon Z\to Y$ un morphisme d'espaces $\cA$-analytiques. Notons $\varphi_{Z} \colon X \times_{Y} Z \to Z$ le morphisme d\'eduit de~$\varphi$ par changement de base \`a~$Z$. Si $\varphi$ est quasi-fini (resp. fini, resp. fini en~$x$), alors $\varphi_{Z}$ est quasi-fini (resp. fini, resp. fini en tout point au-dessus de~$x$).
\end{enumerate}
\end{prop}
\begin{proof}
i) Supposons que~$\varphi$ est quasi-fini. Soit $z\in Y\ho{\cA} \cB$. Posons $y := p_{Y}(z)$. L'exemple~\ref{ex:morphismex} fournit un morphisme $\lambda_{z} \colon \cM(\cH(z)) \to Y\ho{\cA} \cB$. D'apr\`es la proposition~\ref{preimage}, il suffit de montrer que l'ensemble sous-jacent \`a l'espace
\[(X\ho{\cA} \cB)\times_{Y\ho{\cA} \cB} \cM(\cH(z)) :=  \big((X\ho{\cA} \cB)\ho{\cB} \cH(z)\big) \times_{(Y\ho{\cA} \cB) \ho{\cB} \cH(z)} \cM(\cH(z))\] 
est fini. Puisque le diagramme 
\[\begin{tikzcd}
Y \ho{\cA}\cB \arrow[r, "p_{Y}"] \arrow[d] & Y  \arrow[d] \\
\cM(\cB) \arrow[r] & \cM(\cA)
\end{tikzcd}\]
commute, en utilisant les lemmes~\ref{lem:BB'} et~\ref{lem:produitifbreextensionscalaires}, on obtient une suite d'isomorphismes
\begin{align*}
&(X\ho{\cA} \cB)\times_{Y\ho{\cA} \cB} \cM(\cH(z))\\
\simeq\ & (X\ho{\cA} \cH(z)) \times_{Y\ho{\cA} \cH(z)} \cM(\cH(z))\\
\simeq\ & \big((X\ho{\cA} \cH(y))\ho{\cH(y)} \cH(z)\big) \times_{(Y\ho{\cA} \cH(y)) \ho{\cH(y)} \cH(z)} \cM(\cH(z))\\
\simeq\ &  \big( (X\ho{\cA} \cH(y)) \times_{Y\ho{\cA}  \cH(y) } \cM(\cH(y)) \big) \ho{\cH(y)} \cH(z).
\end{align*}

D'apr\`es la proposition~\ref{preimage}, l'ensemble sous-jacent \`a $F := (X\ho{\cA} \cH(y)) \times_{Y\ho{\cA}  \cH(y) } \cM(\cH(y))$ est hom\'eomorphe \`a~$\varphi^{-1}(y)$. Par cons\'equent, $F$ est fini et la seconde projection $F \to \cM(\cH(y))$ est un morphisme fini (et donc s\'epar\'e) d'espaces $\cH(y)$-analytiques. 

D'apr\`es le lemme~\ref{lem:extensionBouverts}, pour montrer que $F \ho{\cH(y)} \cH(z)$ est fini, on peut supposer que~$F$ est r\'eduit \`a un point, disons~$a$. D'apr\`es le corollaire~\ref{cor:finianneauxlocaux}, l'extension de corps $\cH(a)/\cH(y)$ est finie.

L'espace~$F$ \'etant r\'eduit \`a un point, on peut l'identifier \`a un ferm\'e analytique d'un ouvert de~$\E{n}{\cH(y)}$. D'apr\`es la proposition~\ref{prop:extensionBaffine} et le lemme~\ref{lem:extensionBouverts}, l'ensemble sous-jacent \`a~$F \ho{\cH(y)} \cH(z)$ s'identifie alors \`a~$\tilde{f}_{n}^{-1}(F)$. Puisque~$F$ est une partie compacte spectralement convexe de~$\E{n}{\cH(y)}$, d'apr\`es la proposition~\ref{extension_scalaire_spectral}, $\tilde{f}_{n}^{-1}(F)$ est une partie compacte spectralement convexe de~$\E{n}{\cH(z)}$ et on a 
\[ \cB(\tilde{f}_{n}^{-1}(F)) \simeq \cB(F) \hosp{\cH(y)} \cH(z) = \cH(a) \ho{\cH(y)} \cH(z).\]
Puisque l'extension $\cH(a)/\cH(y)$ est finie, $\cH(a) \ho{\cH(y)} \cH(z)$ est une $\cH(z)$-alg\`ebre finie et 
\[ \tilde{f}_{n}^{-1}(F) = \cM(\cB(\tilde{f}_{n}^{-1}(F))) \simeq \cM(\cH(a) \ho{\cH(y)} \cH(z))\]
est finie, ce qui termine la d\'emonstration.

\medbreak

Le cas fini se d\'eduit des cas quasi-fini, s\'epar\'e (\cf~proposition~\ref{stabilite_separe}) et propre (\cf~proposition~\ref{prop:propreAB}). Le cas fini en~$x$ se ram\`ene au cas fini par les propri\'et\'es de l'extension des scalaires (\cf~lemme~\ref{lem:extensionBouverts}).

\medbreak

ii) Supposons que~$\varphi$ est quasi-fini. Soit $z\in Z$. Posons $y := \psi(z)$. L'exemple~\ref{ex:morphismex} fournit un morphisme $\lambda_{z} \colon \cM(\cH(z)) \to Z$. D'apr\`es la proposition~\ref{preimage}, il suffit de montrer que l'ensemble sous-jacent \`a l'espace
\[(X\times_YZ)\times_{Z} \cM(\cH(z)) :=  \big((X\times_YZ)\ho{\cA} \cH(z)\big) \times_{Z\ho{\cA} \cH(z)} \cM(\cH(z))\] 
est fini. En utilisant le lemme~\ref{lem:produitifbreextensionscalaires}, on obtient la suite d'isomorphismes canoniques
\begin{align*}
& \big((X\times_YZ)\ho{\cA} \cH(z)\big) \times_{Z\ho{\cA} \cM(\cH(z))} \cM(\cH(z))\\ 
\simeq\ & \big((X\ho{\cA} \cH(z))\times_{Y\ho{\cA} \cH(z)}(Z\ho{\cA} \cH(z))\big) \times_{Z\ho{\cA} \cH(z)} \cM(\cH(z))\\
\simeq\ &  (X\ho{\cA} \cH(z))\times_{Y\ho{\cA} \cH(z)} \cM(\cH(z)).
\end{align*}
D'apr\`es la proposition~\ref{preimage}, ce dernier espace est hom\'eomorphe \`a une fibre du morphisme $\varphi_{\cH(z)} \colon  X\ho{\cA} \cH(z) \to Y\ho{\cA} \cH(z)$ d\'eduit de~$\varphi$ par extension des scalaires \`a~$\cH(z)$. D'apr\`es~i), ce morphisme est fini et le r\'esultat s'ensuit.

\medbreak

Les cas fini et fini en~$x$ se d\'eduisent des cas quasi-fini, s\'epar\'e (\cf~proposition~\ref{stabilite_separe}) et propre (\cf~proposition~\ref{stabilite_propre}), comme dans~i).
\end{proof}

\section{Changement de base fini}\label{sec:chgtbasefini}

Le but de cette section est de d\'emontrer une propri\'et\'e de changement de base pour l'image directe d'un faisceau par un morphismes fini. \'Enon\c cons pr\'ecis\'ement ce dont il s'agit.

\begin{defi}\label{def:changementdebase}
Soit~$\varphi \colon X\to Y$ un morphisme fini d'espaces $\cA$-analytiques. Nous dirons que $\varphi$ \emph{satisfait la propri\'et\'e de changement de base} si, pour tout morphisme d'espaces $\cA$-analytiques $\psi \colon Z\to Y$ et tout faisceau de $\cO_{X}$-modules~$\cF$, le morphisme naturel 
\[\psi^*\varphi_*\cF\too \chi_*\rho^*\cF\] 
est un isomorphisme, o\`u~$\chi$ et~$\rho$ sont d\'efinis par le diagramme cart\'esien
\[\begin{tikzcd}
Z\times_Y X \arrow[r, "\rho"] \arrow[d, "\chi"] & X \arrow[d, "\varphi"] \\
Z  \arrow[r, "\psi"] & Y
\end{tikzcd}\ .\]
\end{defi}

%

Nous utiliserons \'egalement une version locale de cette propri\'et\'e.

\begin{defi}
Soit $\varphi \colon X\to Y$ un morphisme d'espaces $\cA$-analytiques. Soit $x\in X$ et supposons que~$\varphi$ est fini en~$x$. On dit que $\varphi$ \emph{satisfait la propri\'et\'e de changement de base en~$x$} si, pour tout morphisme d'espaces $\cA$-analytiques $\psi \colon Z \to Y$ et tout $z\in \psi^{-1}(\varphi(x))$, 
le morphisme naturel 
\[\cO_{X,x}\otimes_{\cO_{Y,\varphi(x)}}\cO_{Z,z} \too \prod_{t\in\rho^{-1}(x)  \cap \chi^{-1}(z)}\cO_{X\times_{Y}Z,t},\]
avec les notations de la d\'efinition~\ref{def:changementdebase}, est un isomorphisme. 
\end{defi}

Remarquons qu'il d\'ecoule de la stabilit\'e par changement de base des morphismes finis (\cf~proposition~\ref{stabilite_fini}) que l'ensemble $\rho^{-1}(x)  \cap \chi^{-1}(z)$ est fini.

\begin{lemm}\label{lem:cbflocal}
Soit $\varphi \colon X\to Y$ un morphisme fini d'espaces $\cA$-analytiques. Supposons que $\varphi$ satisfasse la propri\'et\'e de changement de base en tout point de~$X$. Alors, $\varphi$ satisfait la propri\'et\'e de changement de base. 

\end{lemm}
\begin{proof}
Soit $\psi \colon Z\to Y$ un morphisme d'espaces $\cA$-analytiques. Posons $T := Z\times_{Y} X$. Soit $\cF$ un faisceau de $\cO_{X}$-modules.
Il suffit de montrer que, pour tout $z\in Z$, le morphisme
\[(\psi^*\varphi_*\cF)_{z}\too (\chi_*\rho^*\cF)_{z}\]
est un isomorphisme.

Soit~$z\in Z$. Posons $y:=\psi(z)$. On a des isomorphismes naturels
\[(\psi^*\varphi_*\cF)_z\simeq\bigg(\prod_{x\in\varphi^{-1}(y)}\cF_x\bigg)\otimes_{\cO_{Y,y}}\cO_{Z,z}\]
et 
\[(\chi_*\rho^*\cF)_z\simeq \prod_{t\in\chi^{-1}(z)} \big(\cF_{\rho(t)}\otimes_{\cO_{X,\rho(t)}}\cO_{T,t}\big).\]
Or, par hypoth\`ese, pour tout~$x\in \varphi^{-1}(y)$, le morphisme naturel 
\[\cF_x\otimes_{\cO_{Y,y}}\cO_{Z,z} \too \prod_{t\in\rho^{-1}(x) \cap \chi^{-1}(z)}\cF_{x}\otimes_{\cO_{X,x}}\cO_{T,t}\]
est un isomorphisme. Le r\'esultat s'en d\'eduit.
\end{proof}

%



D\'emontrons maintenant la propri\'et\'e de changement de base locale. Commen\c cons par deux lemmes.

\begin{lemm}\label{lem:cbfcomposition}
Soient $\varphi_{1} \colon X_{1}\to X_{2}$ et $\varphi_{2} \colon X_{2} \to X_{3}$ des morphismes d'espaces $\cA$-analytiques. Soit $x_{1}\in X_{1}$. Supposons que $\varphi_{1}$ est fini en~$x_{1}$ et que $\varphi_{2}$ est fini en~$\varphi_{1}(x_{1})$. Si $\varphi_{1}$ et $\varphi_{2}$ satisfont la propri\'et\'e de changement de base respectivement en~$x_{1}$ et~$\varphi_{1}(x_{1})$, alors $\varphi_{2} \circ \varphi_{1}$ satisfait la propri\'et\'e de changement de base en~$x_{1}$. 
\qed
\end{lemm}


\begin{lemm}\label{lem:cbfYY'}
Soient $\varphi_{1} \colon X_{1}\to X_{2}$ et $\varphi_{2} \colon X_{2} \to X_{3}$ des morphismes d'espaces $\cA$-analytiques. Soit~$x_{1}\in X_{1}$. Supposons que $\varphi_{1}$ est fini en~$x_{1}$ et que $\varphi_{2}$ est fini en~$\varphi_{1}(x_{1})$. Si $\varphi_{2} \circ \varphi_{1}$ satisfait la propri\'et\'e de changement de base en~$x_{1}$, alors $\varphi_{1}$ aussi.
\end{lemm}
\begin{proof}
Soit $\psi \colon Z \to X_{2}$ un morphisme d'espaces $\cA$-analytiques. On a un isomorphisme canonique 
\[Z \times_{X_{2}} X_{1} \simtoo Z \times_{X_{3}} X_{1},\]
o\`u les morphismes d\'efinissant le produit fibr\'e de droite sont $\varphi_{2} \circ \psi$ et $\varphi_{2} \circ \varphi_{1}$. De m\^eme, pour tout $z \in \psi^{-1}(\varphi_{1}(x_{1}))$, on a un isomorphisme canonique
\[\cO_{X_{1},x_{1}} \otimes_{\cO_{X_{2},\varphi_{1}(x_{1})}}\cO_{Z,z} \simtoo \cO_{X_{1},x_{1}} \otimes_{\cO_{X_{3},\varphi_{2}(\varphi_{1}(x_{1}))}}\cO_{Z,z}.\]
Le r\'esultat s'en d\'eduit.
\end{proof}


Traitons maintenant deux cas particuliers de la propri\'et\'e de changement de base locale.

\begin{lemm}\label{lem:cbfomega}
Soit~$Y$ un espace $\cA$-analytique et $\omega\in\cO(Y)[T]$ un polyn\^ome unitaire non constant. Consid\'erons~$\E{1}{\cA}$ avec coordonn\'ee~$T$ et notons $\pi \colon Y\times_{\cA}\E{1}{\cA} \to Y$ le morphisme de projection. Notons $Z_{\omega}$ le ferm\'e analytique de~$Y\times_{\cA}\E{1}{\cA} $ d\'efini par~$\omega$. Soit $x\in Z_{\omega}$ tel que $\pi_{|Z_{\omega}}^{-1}(\pi_{|Z_{\omega}}(x)) = \{x\}$. Alors $\pi_{|Z_{\omega}}$ satisfait la propri\'et\'e de changement de base en~$x$.
\end{lemm}
\begin{proof}
Soit $\psi \colon Z \to Y$ un morphisme d'espaces $\cA$-analytiques. Soit $z\in Z$ tel que $\psi(z) = \pi(x)$. Posons $y := \pi(x)$. Consid\'erons le carr\'e cart\'esien
\[\begin{tikzcd}
T := Z\times_Y Z_{\omega} \arrow[r, "\rho"] \arrow[d, "\chi"] & Z_{\omega} \arrow[d, "\pi_{|Z_{\omega}}"] \\
Z  \arrow[r, "\psi"] & Y
\end{tikzcd}\ .\]
D'apr\`es le lemme~\ref{lem:produitouverts}, le produit fibr\'e~$T$ s'identifie au ferm\'e analytique de~$Z\times_{\cA} \E{1}{\cA}$ d\'efini par le polyn\^ome $\psi^\sharp(\omega) \in \cO(Z)[T]$. D'apr\`es la proposition~\ref{prop:fetoileenbasexact} et le lemme~\ref{cas_particulier_coh\'erence}, on a des isomorphismes naturels
\[ \prod_{t\in \chi^{-1}(z)} \cO_{T,t} \simeq (\chi_*\cO_{T})_z \simeq  \cO_{Z,z}[T]/(\psi^{\sharp}(\omega))\]
et 
\[\cO_{Z_{\omega},x} \simeq \big((\pi_{|Z_{\omega}})_*\cO_{Z_{\omega}}\big)_{y} \simeq \cO_{Y,y}[T]/(\omega).\]
On en d\'eduit que $\cO_{Z_{\omega},x} \otimes_{\cO_{Y,y}} \cO_{Z,z} \simeq \prod_{t\in \chi^{-1}(z)} \cO_{T,t} $. Puisque $\pi_{|Z_{\omega}}^{-1}(y) = \{x\}$, on a $\chi^{-1}(z) \subset \rho^{-1}(x)$. Le r\'esultat s'en d\'eduit.
\end{proof}

\begin{lemm}\label{lem:cbfprojection}
Soient~$\varphi \colon X\to Y$ un morphisme d'espaces $\cA$-analytiques. Soient $x\in X$ et $y\in Y$ tels que $\varphi^{-1}(y) = \{x\}$. 
Supposons qu'il existe des entiers $n,l$ avec $n\ge l$, des ouverts~$W'$ de~$\E{n}{\cA}$ et $V'$ de~$\E{l}{\cA}$ et des immersions ferm\'ees $\iota_{X} \colon X \to W'$ et $\iota_{Y} \colon Y \to V'$ tels que, en notant $\varrho_{l} \colon \E{n}{\cA} \to \E{l}{\cA}$ la projection sur les $l$~derni\`eres coordonn\'ees, on ait $\varrho_{l}(W') = V'$ et le diagramme
\[\begin{tikzcd}
X \arrow[d, "\varphi"] \arrow[r, "\iota_{X}"] & W' \arrow[d, "\varrho_{l, |W'}"]\\
Y \arrow[r, "\iota_{Y}"] & V'
\end{tikzcd}\]
commute. Alors $\varphi$ satisfait la propri\'et\'e de changement de base en~$x$.
\end{lemm}
\begin{proof}
D\'emontrons par r\'ecurrence sur~$n\ge l$ que $\varphi$ satisfait la propri\'et\'e de changement de base en~$x$.

Si $n=l$, alors $V' = W'$ et $q_{l} = \id$. D'apr\`es le lemme~\ref{lem:cbfYY'}, on peut remplacer~$\varphi$ par sa compos\'ee avec~$\iota_{Y}$ et donc supposer que $Y=V'$ et que $\iota_{Y} =\id$. On a alors $\varphi = \iota_{X}$. Soit $\psi \colon Z\to Y$ un morphisme d'espaces $\cA$-analytiques et reprenons les notations du d\'ebut de la section. D'apr\`es le lemme~\ref{lem:produitouverts}, le morphisme~$\chi$ est alors le morphisme d'inclusion d'un ferm\'e analytique d\'efini par le m\^eme id\'eal que~$X$ dans~$W'$. Pour tout $z\in Z$ tel que $\psi(z) = \varphi(x)$, la fibre $\chi^{-1}(z)$ contient donc un unique point, disons~$t$, et nous avons $\cO_{T,t} = \cO_{Z,z} \otimes_{\cO_{W',\varphi(x)}} \cO_{X,x}$. Le r\'esultat s'en d\'eduit.

Supposons que $n >l$ et que le r\'esultat est satisfait pour~$n-1$. Notons $\varrho_{n-1} \colon \E{n}{\cA} \to \E{n-1}{\cA}$ le morphisme de projection sur les $n-1$ derni\`eres coordonn\'ees et~$T$ la premi\`ere coordonn\'ee de~$\E{n}{\cA}$. D'apr\`es le lemme~\ref{lem:recurrenceAn}, il existe un voisinage ouvert~$W'_{n-1}$ de~$\varrho_{n-1}(x)$ dans~$\varrho_{n-1}(W')$ et un polyn\^ome $\omega \in \cO(W'_{n-1})[T]$ unitaire non constant tel que $X \cap \varrho_{n-1}^{-1}(W'_{n-1})$ soit un ferm\'e analytique du ferm\'e analytique~$Z_{\omega}$ de~$\varrho_{n-1}^{-1}(W'_{n-1})$ d\'efini par~$\omega$. Quitte \`a remplacer $W'$ par $W' \cap \varrho_{n-1}^{-1}(W'_{n-1})$ et les autres espaces en cons\'equence, on peut supposer que $W'_{n-1} = \varrho_{n-1}(W')$.
 
D'apr\`es le lemme~\ref{cas_particulier_coh\'erence} et le corollaire~\ref{cor:imagefini}, $X_{n-1} := \varrho_{n-1}(X)$ est un ferm\'e analytique de~$W'_{n-1}$. En particulier, on peut le munir d'une structure d'espace $\cA$-analytique. Notons $\varrho'_{n-1} \colon X \to X_{n-1}$ le morphisme induit par~$\varrho_{n-1}$. D'apr\`es le lemme~\ref{lem:cbfomega}, $\varrho'_{n-1}$ satisfait la propri\'et\'e de changement de base en~$x$. Notons $\varrho_{n-1,l} \colon \E{n-1}{\cA} \to \E{l}{\cA}$ le morphisme de projection sur les $l$~derni\`eres coordonn\'ees et $\varrho'_{n-1,l} \colon X_{n-1} \to Y$ le morphisme induit. Par hypoth\`ese de r\'ecurrence, $\varrho'_{n-1,l}$ satisfait la propri\'et\'e de changement de base en~$\varrho_{n-1}(x)$. Le r\'esultat d\'ecoule alors du lemme~\ref{lem:cbfcomposition}.
\end{proof}

D\'emontrons, \`a pr\'esent, le r\'esultat de changement de base en toute g\'en\'eralit\'e.

\begin{theo}\label{thm:changementdebaselocal}\index{Changement de base fini|see{Morphisme analytique fini}}\index{Morphisme analytique!fini!changement de base d'un}\index{Produit!fibre@fibr\'e}
Soient $\varphi \colon X\to Y$ et $\psi \colon Z\to Y$ des morphismes d'espaces $\cA$-analytiques. Soit $x\in X$ et supposons que $\varphi$ est fini en~$x$. Consid\'erons le carr\'e cart\'esien
\[\begin{tikzcd}
X\times_YZ \arrow[r, "\rho"] \arrow[d, "\chi"] & X \arrow[d, "\varphi"] \\
Z  \arrow[r, "\psi"] & Y
\end{tikzcd}\ .\]
Alors, le morphisme naturel 
\[\cO_{X,x}\otimes_{\cO_{Y,\varphi(x)}}\cO_{Z,z} \too \prod_{t\in\rho^{-1}(x)  \cap \chi^{-1}(z)}\cO_{X\times_YZ,t}\]
est un isomorphisme. 
\end{theo}
\begin{proof}
Le r\'esultat \'etant local en~$x$, quitte \`a restreindre~$X$ et~$Y$, on peut supposer que~$\varphi^{-1}(\varphi(x)) = \{x\}$. On peut alors appliquer le lemme~\ref{lem:finiprojection} pour se retrouver dans la situation du lemme~\ref{lem:cbfprojection}, qui permet de conclure.
\end{proof}

\begin{coro}\label{thm:formule_changement_base}\index{Changement de base fini|see{Morphisme analytique fini}}\index{Morphisme analytique!fini!changement de base d'un}\index{Produit!fibre@fibr\'e}
Soient $\varphi \colon X\to Y$ et $\psi \colon Z\to Y$ des morphismes d'espaces $\cA$-analytiques. Supposons que~$\varphi$ est fini. Alors, avec les notations du th\'eor\`eme~\ref{thm:changementdebaselocal}, pour tout faisceau de $\cO_{X}$-modules~$\cF$, le morphisme naturel 
\[\psi^*\varphi_*\cF\too \chi_*\rho^*\cF\] 
est un isomorphisme.
\end{coro}
\begin{proof}
Le r\'esultat d\'ecoule du lemme~\ref{lem:cbflocal} et du th\'eor\`eme~\ref{thm:changementdebaselocal}.
\end{proof}

%

En utilisant les m\^emes arguments, on peut d\'emontrer un r\'esultat similaire en rempla\c cant le changement de base par une extension des scalaires. 

\begin{theo}\label{thm:extensionscalaireslocal}\index{Morphisme analytique!fini!changement de base d'un}\index{Extension des scalaires}
Soient $\varphi \colon X\to Y$ et $\psi \colon Z\to Y$ des morphismes d'espaces $\cA$-analytiques. Soit $x\in X$ et supposons que $\varphi$ est fini en~$x$. Soit $f\colon \cA \to \cB$ un morphisme d'anneaux de Banach born\'e, o\`u $\cB$ est un anneau de base g\'eom\'etrique. Consid\'erons le carr\'e cart\'esien
\[\begin{tikzcd}
X\ho{\cA} \cB \arrow[r, "\pi_{X}"] \arrow[d, "\varphi_{\cB}"] & X \arrow[d, "\varphi"] \\
Y\ho{\cA} \cB  \arrow[r, "\pi_{Y}"] & Y
\end{tikzcd}\ .\]
Alors, le morphisme naturel 
\[\cO_{X,x}\otimes_{\cO_{Y,\varphi(x)}}\cO_{Y\ho{\cA}\cB,z} \too \prod_{t\in\rho^{-1}(x)  \cap \chi^{-1}(z)}\cO_{X\ho{\cA} \cB,t}\]
est un isomorphisme. 
\qed
\end{theo}

\begin{coro}\label{thm:formule_extension_scalaire}\index{Morphisme analytique!fini!changement de base d'un}\index{Extension des scalaires}
Soient $\varphi \colon X\to Y$ et $\psi \colon Z\to Y$ des morphismes d'espaces $\cA$-analytiques. Supposons que~$\varphi$ est fini. Alors, avec les notations du th\'eor\`eme~\ref{thm:extensionscalaireslocal}, pour tout faisceau de $\cO_{X}$-modules~$\cF$, le morphisme naturel 
\[(\pi_{Y})^*\varphi_*\cF\too (\varphi_{\cB})_*(\pi_{X})^*\cF\] 
est un isomorphisme.
\qed
\end{coro}


\section{Nullstellensatz de R\"uckert}\label{sec:Ruckert}

Dans cette section, nous d\'emontrons une version analytique du Nullstellensatz, connue en g\'eom\'etrie analytique complexe sous le nom de Nullstellensatz de R\"uckert. 
Dans sa premi\`ere version, il s'agit de montrer que si une fonction s'annule sur le support d'un faisceau coh\'erent, alors une puissance de cette fonction annule le faisceau. Nous suivons ici la strat\'egie de \cite[\S 3.2 et \S 4.1.5]{Gr-Re2}.

Commen\c{c}ons par d\'emontrer plusieurs cas particuliers du r\'esultat. Comme de coutume, nous noterons $T_{1},\dotsc,T_{n}$ les coordonn\'ees sur~$\E{n}{\cA}$. Rappelons la d\'efinition~\ref{def:projection} concernant les projections.

\begin{lemm}\label{lem:Ruckertloctr}
Soient~$x$ un point de~$U$ et~$U$ un voisinage ouvert de~$x$ dans~$\E{n}{\cA}$. 
Supposons que $x$~est purement localement transcendant au-dessus de $b := \pi_{n}(x)$. Soient~$\cF$ un faisceau coh\'erent sur~$U$ et $f\in\cO_{\E{n}{\cA}}(U)$ une fonction analytique nulle en tout point du support de~$\cF$. Alors, il existe~$d\in\N$ tel que~$f^d\cF_x=0$.
\end{lemm}
\begin{proof}
Si l'image de~$f$ dans~$\cO_{\E{n}{\cA},x}$ est nulle, alors le r\'esultat vaut avec $d=1$. Supposons donc que l'image de~$f$ dans~$\cO_{\E{n}{\cA},x}$ n'est pas nulle.

D'apr\`es le th\'eor\`eme~\ref{rigide}, si~$\cO_{B,b}$ est un corps fort (resp. un anneau fortement de valuation discr\`ete), alors il en va de m\^eme pour $\cO_{\E{n}{\cA},x}$. 

Supposons que $\cO_{\E{n}{\cA},x}$ est un corps fort. Alors, on a $f(x) \ne 0$, donc $x$ n'appartient pas au support de~$\cF$ et on a donc $\cF_{x}= 0$. 

Supposons que $\cO_{\E{n}{\cA},x}$ est un anneau fortement de valuation discr\`ete. Choisissons-en une uniformisante~$\varpi_{x}$. La fonction~$f$ peut alors s'\'ecrire de fa\c con unique sous la forme $f = h \varpi_x^l$ dans~$\cO_{\E{n}{\cA},x}$, o\`u $h$ est un \'el\'ement inversible de~$\cO_{\E{n}{\cA},x}$ et $l$ un entier non nul.

Par hypoth\`ese, le support de~$\cF$ est contenu dans le ferm\'e analytique d\'efini par~$f$. Au voisinage de~$x$, il est donc contenu dans le ferm\'e analytique d\'efini par~$\varpi_{x}$. En particulier, d'apr\`es le corollaire~\ref{cor:projectionouverte}, ce support n'est donc pas un voisinage de~$x$. Or le support de~$\cF$ co\"incide avec le lieu des z\'eros de son id\'eal annulateur, qui est coh\'erent. On en d\'eduit qu'il existe $g \in \cO_{\E{n}{\cA},x}$ non nul tel que $g\cF_x$ soit nul. Puisque $\cO_{\E{n}{\cA},x}$ est un anneau de valuation discr\`ete d'uniformisante~$\varpi_{x}$, on en d\'eduit qu'il existe~$v\in\N$ tel que~$\varpi_x^v \cF_x = 0$. Soit~$d\in \N$ tel que~$dl\geq v$. On a alors  
\[f^d\cF_x = h^d\varpi_x^{dl}\cF_x = h^d\varpi_x^{dl-v}(\varpi_x^v\cF_x)=0.\]
\end{proof}

\begin{lemm}\label{lem:Ruckertprojectioncorps}
Soient $k,n \in \N$ avec $n>k$. 
Soient~$U$ un ouvert de~$\E{n}{\cA}$ et $\cF$ un faisceau coh\'erent sur~$U$ dont le support est contenu dans le ferm\'e analytique d\'efini par~$T_{k+1}$. Soient~$b \in B$ tel que $\cO_{B,b}$ soit un corps fort. Soit $y \in \E{k}{\cA}$ tel que $\pi_{k}(y) = b$ et $y$ soit purement localement transcendant au-dessus de~$b$. Notons~$0_{y} \in \E{n}{\cA}$ le point~0 de la fibre $\pi_{n,k}^{-1}(y)$. Alors, il existe~$d\in\N$ tel que~$T_{k+1}^d \cF_{0_{y}}=0$.
\end{lemm}
\begin{proof}
D'apr\`es le th\'eor\`eme~\ref{rigide}, $\cO_{\E{k}{\cA},y}$ est un corps fort.

Par hypoth\`ese, le support de~$\cF$ est inclus dans le ferm\'e analytique de~$\E{n}{\cA}$ d\'efini par~$T_{k+1}$. En particulier, le support de~$\cF$ n'est pas un voisinage de~$0_{y}$. Or le support de~$\cF$ co\"incide avec le lieu des z\'eros de son id\'eal annulateur, qui est coh\'erent. On en d\'eduit qu'il existe $g \in \cO_{\E{n}{\cA},0_{y}}$ non nul tel que $g\cF_{0_{y}}$ soit nul. Quitte \`a restreindre~$U$, on peut supposer que~$g\cF = 0$. 

Notons~$Z_{g}$ le ferm\'e analytique de~$U$ d\'efini par~$g$ et $j\colon Z_{g} \to U$ l'immersion ferm\'ee associ\'ee. Le faisceau~$\cF$ poss\`ede une structure naturelle de $\cO_{U}/(g)$-module. Par cons\'equent, on a $j_{\ast}j^\ast\cF = \cF$.

D\'emontrons maintenant le r\'esultat par r\'ecurrence sur $n \ge k+1$. 

\medbreak

$\bullet$ Supposons que $n=k+1$. 

D'apr\`es la proposition~\ref{prop:disqueglobal} (appliqu\'ee avec $t=0$), $g$ peut s'\'ecrire sous la forme 
\[ g =  T_{k+1}^d \sum_{i\ge 0} \alpha_i \, T_{k+1}^i,\]
o\`u les~$\alpha_{i}$ sont des \'el\'ements de~$\cO_{\E{k}{\cA},y}$ avec $\alpha_{0} \ne 0$. Puisque $\cO_{\E{k}{\cA},y}$ est un corps fort, $\alpha_{0}$ est inversible et il en va de m\^eme pour $ \sum_{i\ge 0} \alpha_i \, T_{k+1}^i$. L'\'egalit\'e $g\cF_{0_{y}} = 0$ entra\^ine donc $T_{k+1}^d \cF_{0_{y}} = 0$.

\medbreak

$\bullet$ Supposons que $n > k+1$ et que l'\'enonc\'e est d\'emontr\'e pour $n-1$.

D'apr\`es la proposition~\ref{restriction_fibre}, la restriction de~$g$ \`a~$\pi_{n,k}^{-1}(y)$ n'est pas nulle. Posons $0_{n-1,y} := \pi_{n,n-1}(0_{y})$. D'apr\`es le lemme~\ref{changement_variable}, quitte \`a effectuer un changement de variables, on peut supposer que l'image de~$g$ dans $\cO_{\pi^{-1}_{n,n-1},0_{n-1,y}}$ n'est pas nulle. Par cons\'equent, quitte \`a restreindre~$U$, on peut supposer que $Z_{g} \cap \pi_{n,n-1}^{-1}(0_{n-1,y}) = \{0_{y}\}$. 

D'apr\`es le th\'eor\`eme~\ref{thm:locfini}, quitte \`a r\'eduire~$U$, on peut supposer qu'il existe un voisinage ouvert~$V$ de~$0_{n-1,y}$ dans~$\E{n-1}{\cA}$ tel que $\pi_{n,n-1}(U) \subset V$ et le morphisme $\pi'_{n,n-1} \colon U \cap Z_{g}\to V$ induit par~$\pi_{n,n-1}$ soit fini. D'apr\`es le th\'eor\`eme~\ref{thm:fini}, $(\pi'_{n,n-1})_*j^\ast\cF$ est un faisceau coh\'erent sur~$V$. Son support est contenu dans le ferm\'e analytique d\'efini par~$T_{k+1}$ donc, par hypoth\`ese de r\'ecurrence, il existe~$d\in \N$ tel que 
\[T_{k+1}^d \big((\pi'_{n,n-1})_*j^\ast\cF\big)_{0_{n-1,y}} = T_{k+1}^d j^\ast\cF_{0_{y}} = T_{k+1}^d \cF_{0_{y}}  = 0.\] 
\end{proof}

\begin{lemm}\label{lem:Ruckertprojectionavd}
Soient $k,n \in \N$ avec $n>k$. 
Soient~$U$ un ouvert de~$\E{n}{\cA}$ et $\cF$ un faisceau coh\'erent sur~$U$ dont le support est contenu dans le ferm\'e analytique d\'efini par~$T_{k+1}$. Soient~$b \in B$ tel que $\cO_{B,b}$ soit un anneau fortement de valuation discr\`ete. Soit $y \in \E{k}{\cA}$ tel que $\pi_{k}(y) = b$ et $y$ soit purement localement transcendant au-dessus de~$b$. Notons $0_{y} \in \E{n}{\cA}$ le point~0 de la fibre $\pi_{n,k}^{-1}(y)$. Alors, il existe~$d\in\N$ tel que~$T_{k+1}^d \cF_{0_{y}}=0$.
\end{lemm}
\begin{proof}
Notons $\varpi$ une uniformisante de~$\cO_{B,b}$. D'apr\`es le th\'eor\`eme~\ref{rigide}, $\cO_{\E{k}{\cA},y}$ est encore un anneau fortement de valuation discr\`ete d'uniformisante~$\varpi$. Quitte \`a restreindre~$U$, on peut supposer que~$\varpi$ est d\'efinie sur~$U$. Notons~$Z_{\varpi}$ le ferm\'e analytique de~$U$ d\'efini par~$\varpi$ et $j \colon Z_{\varpi} \to U$ l'immersion ferm\'ee associ\'ee. 

Il suffit de montrer qu'il existe $v,d\in \N$ tels que $\varpi^vT_{k+1}^d\cF_{0_{y}}=0$. Si $v=0$, on obtient directement le r\'esultat souhait\'e. Supposons que $v \ge 1$. Quitte \`a restreindre~$U$, on peut supposer que $\varpi^v T_{k+1}^d\cF=0$. Le faisceau $\varpi^{v-1}T_{k+1}^d\cF$ est alors naturellement muni d'une structure de~$\cO_{U}/(\varpi)$-module et on a donc $j_{\ast} j^\ast (\varpi^{v-1}T_{k+1}^d\cF) = \varpi^{v-1}T_{k+1}^d\cF$. 

Or $\cO_{Z,y} = \cO_{\E{n}{\cA},y}/(\varpi)$ est un corps fort. D'apr\`es le lemme~\ref{lem:Ruckertprojectioncorps}, il existe donc~$d'$ tel que 
\[T_{k+1}^{d'} j^\ast (\varpi^{v-1}T_{k+1}^d\cF)_{0_{y}} = j^\ast (\varpi^{v-1}T_{k+1}^{d+d'}\cF)_{0_{y}} = \varpi^{v-1}T_{k+1}^{d+d'}\cF_{0_{y}} =0.\]
En r\'ep\'etant le proc\'ed\'e, on montre finalement qu'il existe $d''\in \N$ tel que $T_{k+1}^{d''}\cF_{0_{y}} = 0$, comme attendu.

\medbreak

Passons maintenant \`a la d\'emonstration du r\'esultat. Par hypoth\`ese, le support de~$\cF$ est inclus dans le ferm\'e analytique de~$\E{n}{\cA}$ d\'efini par~$T_{k+1}$. En particulier, le support de~$\cF$ n'est pas un voisinage de~$0_{y}$. Or le support de~$\cF$ co\"incide avec le lieu des z\'eros de son id\'eal annulateur, qui est coh\'erent. On en d\'eduit qu'il existe $g \in \cO_{\E{n}{\cA},0_{y}}$ non nul tel que $g\cF_{0_{y}}$ soit nul. 

D\'emontrons le r\'esultat par r\'ecurrence sur $n \ge k+1$. 

\medbreak

$\bullet$ Supposons que $n=k+1$. 

D'apr\`es la proposition~\ref{prop:disqueglobal} (appliqu\'ee avec $t=0$), il existe un voisinage ouvert~$V$ de~$y$ dans~$\E{n}{k}$ tel que~$g$ puisse s'\'ecrire sous la forme 
\[ g =  T_{k+1}^d \sum_{i\ge 0} \alpha_i \, T_{k+1}^i,\]
o\`u les~$\alpha_{i}$ sont des \'el\'ements de~$\cO(V)$ avec $\alpha_{0} \ne 0$ dans $\cO_{\E{k}{\cA},y}$. Posons 
$\tilde g :=  \sum_{i\ge 0} \alpha_i \, T_{k+1}^i.$

Si $\alpha_{0}(y) \ne 0$, alors $\tilde{g}$ est inversible dans~$\cO_{\E{k+1}{\cA},0_{y}}$ et on en d\'eduit que $T_{k+1}^d \cF_{0_{y}} = 0$. On supposera donc d\'esormais que~$\alpha_0(x_k)=0$. D'apr\`es le lemme~\ref{restriction_fibre_avd}, il existe~$v\in\N$ et $h \in \cO_{\E{k+1}{\cA},0_y}$ tels que~$\tilde{g}=\varpi^v h$ dans~$\cO_{\E{k+1}{\cA},0_y}$ et la restriction de~$h$ \`a $\pi_{k+1,k}^{-1}(y)$ ne soit pas nulle. On a alors
\[ h \varpi^vT_{k+1}^d\cF_{0_{y}}=0.\]
Quitte \`a restreindre~$U$, nous pouvons supposer que $h \varpi^vT_{k+1}^d\cF = 0$.

Si $h(0_{y}) \ne 0$, alors on a $\varpi^vT_{k+1}^d\cF_{0_{y}}=0$. On supposera donc d\'esormais que $h(0_{y}) = 0$. D'apr\`es la proposition~\ref{prop:disqueglobal}, quitte \`a restreindre~$V$, on peut \'ecrire~$h$ sous la forme 
\[ h = \sum_{i\ge 0} \beta_i \, T_{k+1}^i,\]
o\`u les~$\beta_{i}$ sont des \'el\'ements de~$\cO(V)$. Remarquons que $\beta_{0}(y) = 0$, autrement dit, $\beta_{0}$ est divisible par~$\varpi$. Le support du faisceau $\varpi^vT_{k+1}^d\cF$ est contenu dans l'ensemble
\begin{align*}
\{h = 0\} \cap \{T_{k+1} = 0\} & = \{\beta_{0} = 0\} \cap \{T_{k+1} = 0\} \\
&= \{\varpi=0\}\cap\{T_{k+1}=0\}.
\end{align*}
En particulier, le support du faisceau $\varpi^vT_{k+1}^d\cF$ est contenu dans~$\{\varpi=0\}$. 

Puisque la restriction de~$h$ \`a~$\pi_{k+1,k}^{-1}(y)$ n'est pas nulle, quitte \`a restreindre~$U$, on peut supposer que l'ensemble des z\'eros de~$h$ dans~$\pi_{k+1,k}^{-1}(y)\cap U$ est r\'eduit \`a~$0_{y}$. D'apr\`es le th\'eor\`eme~\ref{thm:locfini}, quitte \`a r\'eduire~$U$ et~$V$, on peut supposer que $\pi_{k+1,k}(U) \subset V$ et que le morphisme $\pi'_{k+1,k} \colon U \cap \{h=0\} \to V$ induit par~$\pi_{k+1,k}$ est fini. D'apr\`es le th\'eor\`eme~\ref{thm:fini}, le faisceau 
\[(\pi_{k+1,k})_*(\varpi^v T_{k+1}^d \cF) = (\pi'_{k+1,k})_*(\varpi^v T_{k+1}^d \cF)_{|\{h=0\}}\] 
est un faisceau coh\'erent sur~$V$. 
 
Puisque le support de~$\varpi^vT_{k+1}^d\cF$ est contenu dans~$\{\varpi=0\}$, celui de~$(\pi_{k+1,k})_*(\varpi^vT_{k+1}^d\cF)$ l'est aussi. D'apr\`es le lemme~\ref{lem:Ruckertloctr}, il existe~$v'\in\N$ tel que
\[\varpi^{v'}(\pi_{k+1,k})_*(\varpi^vT_{k+1}^d\cF)_{y} = (\pi_{k+1,k})_*(\varpi^{v+v'}T_{k+1}^d\cF)_{y}=0.\] 
On en d\'eduit que $\varpi^{v+v'}T_{k+1}^d\cF_{0_{y}}=0$.

\medbreak

$\bullet$ Supposons que $n > k+1$ et l'\'enonc\'e est d\'emontr\'e pour $n-1$.

D'apr\`es le lemme~\ref{restriction_fibre_avd}, il existe $v\in \N$ et $h\in\cO_{\E{n}{\cA},0_{y}}$ dont la restriction \`a~$\pi_{n,k}^{-1}(y)$ n'est pas nulle tels que $g = \varpi^v h$. Quitte \`a remplacer~$g$ par~$h$ et~$\cF$ par~$\varpi^v \cF$, on peut donc supposer que la restriction de~$g$ \`a~$\pi_{n,k}^{-1}(y)$ n'est pas nulle. On peut maintenant conclure en reprenant la fin de la d\'emonstration du lemme~\ref{lem:Ruckertprojectioncorps}.
\end{proof}

D\'emontrons maintenant un lemme technique permettant de ramener l'\'etude d'un point rigide \'epais \`a celle du point rationnel~0. 

\begin{lemm}\label{lem:translation}\index{Point!rigide epais@rigide \'epais}
Soit~$x\in\E{n}{\cA}$. 
Soit $k\in \cn{0}{n}$. Posons $x_{k} := \pi_{n,k}(x)$ et notons~$0_{x_{k}}$ le point de~0 de la fibre $\pi_{n,k}^{-1}(x_{k}) \simeq \E{n-k}{\cH(x_{k})}$.

Supposons que~$x$ est rigide \'epais au-dessus de~$x_{k}$. Alors, il existe un voisinage ouvert~$U_{k}$ de~$x_{k}$ dans~$\E{k}{\cA}$, un voisinage ouvert~$U$ de~$x$ dans~$\E{n}{\cA}$, un voisinage ouvert~$V$ de~$0_{x_{k}}$ dans~$\E{n}{\cA}$ tel que $\pi_{n,k}(U) = \pi_{n,k}(V) = U_{k}$ et un morphisme fini $\varphi \colon U \to V$ tel que l'on ait $\varphi^{-1}(0_{x_k})=\{x\}$ et $\pi_{n,k} \circ \varphi = \pi_{n,k}$.
\end{lemm}
\begin{proof}
Posons~$l:=n-k$. D\'emontrons le r\'esultat par r\'ecurrence sur~$l$. Dans le cas o\`u~$l=0$, c'est imm\'ediat.

Supposons que~$l\ge 1$ et que le r\'esultat est satisfait pour~$l-1$. 
Posons $x_{n-1} := \pi_{n-1}(x)$. 
Puisque~$x$ est rigide \'epais au-dessus de~$x_k$, il l'est au-dessus de~$x_{n-1}$. Consid\'erons le polyn\^ome minimal~$\mu_{\kappa,x} \in \kappa(x_{n-1})[T_{n}]$ et relevons-le en un polyn\^ome unitaire $M \in\cO_{\E{n-1}{\cA},x_{n-1}}[T_{n}]$. Soit~$U$ un voisinage compact spectralement convexe de~$x_{n-1}$ sur lequel tous les coefficients de~$M$ sont $\cB$-d\'efinis. Notons $\varphi_{M} \colon \E{1}{\cB(U)} \to  \E{1}{\cB(U)}$ le morphisme d'espaces $\cB(U)$-analytiques d\'efini par~$M$ (\cf~exemple~\ref{exemple_morphisme}). Identifions $\E{1}{\cB(U)}$ \`a~$\pi_{n,n-1}^{-1}(U)$ et notons~$0_{x_{n-1}}$ le point de~0 de la fibre $\pi_{n,n-1}^{-1}(x_{n-1})$. Le morphisme~$\varphi_{M}$ est fini, par l'exemple~\ref{ex:fini}, et satisfait $\varphi_{M}^{-1}(0_{x_{n-1}}) = \{x\}$, car $\mu_{\kappa,x}$ est une puissance du polyn\^ome minimal de~$x$ sur~$\cH(x_{n-1})$, d'apr\`es le lemme~\ref{lem:polmin}. Remarquons que, puisque $\varphi_{M}$ est un morphisme d'espaces $\cB(U)$-analytiques, il satisfait $\pi_{n,n-1} \circ \varphi_{M} = \pi_{n,n-1}$.

Notons~$0_{n-1,x_{k}}$ le point de~0 de la fibre $\pi_{n-1,k}^{-1}(x_{k})$. 
Par hypoth\`ese de r\'ecurrence, il existe un voisinage ouvert~$U_{k}$ de~$x_{k}$ dans~$\E{k}{\cA}$, un voisinage ouvert~$U_{n-1}$ de~$x_{n-1}$ dans~$\E{n-1}{\cA}$, un voisinage ouvert~$V_{n-1}$ de~$0_{n-1,x_{k}}$ dans~$\E{n-1}{\cA}$ tel que $\pi_{n-1,k}(U_{n-1}) = \pi_{n-1,k}(V_{n-1}) = U_{k}$ et un morphisme fini $\psi \colon U_{n-1} \to V_{n-1}$ tel que l'on ait $\psi^{-1}(0_{n-1,x_k})=\{x_{n-1}\}$ et $\pi_{n-1,k} \circ \psi = \pi_{n-1,k}$. On peut supposer que $U_{n-1} \subset U$.

Notons $\psi_{n} \colon \pi_{n,n-1}^{-1}(U_{n-1}) = U_{n-1} \times_{\cA} \AunA \to V_{n-1} \times_{\cA} \AunA = \pi_{n,n-1}^{-1}(V_{n-1})$ le morphisme d\'eduit de~$\psi$ par le changement de base $\AunA \to \cM(\cA)$. D'apr\`es la proposition~\ref{stabilite_fini}, c'est encore un morphisme fini. Il v\'erifie $\psi_{n}^{-1}(0_{x_{k}}) = \{0_{x_{n-1}}\}$. 

Notons $\varphi'_{M} \colon \pi_{n,n-1}^{-1}(U_{n-1}) \to \pi_{n,n-1}^{-1}(U_{n-1})$ le morphisme induit par~$\varphi_{M}$. Le morphisme compos\'e $\psi_{n} \circ \varphi'_{M}$ satisfait alors les propri\'et\'es de l'\'enonc\'e.
\end{proof}

Nous pouvons maintenant d\'emontrer l'analogue du Nullstellensatz de R\"uckert en toute g\'en\'eralit\'e.

\begin{theo}\label{thm:Ruckert}\index{Nullstellensatz}\index{Theoreme@Th\'eor\`eme!Nullstellensatz}\index{Faisceau!support d'un}
Soient $X$ un espace $\cA$-analytique, $\cF$ un faisceau coh\'erent sur~$X$ et $f\in\cO(X)$ une fonction analytique nulle en tout point du support de~$\cF$. Alors, pour tout $x\in X$, il existe~$d\in\N$ tel que~$f^d\cF_x=0$.
\end{theo}
\begin{proof}
Soit $x\in X$. Le r\'esultat \'etant local au voisinage de~$X$, on peut supposer qu'il existe une immersion ferm\'ee $j \colon X \to U$, o\`u $U$ est un ouvert de~$\E{n}{\cA}$. Quitte \`a restreindre~$X$ et~$U$, on peut supposer qu'il existe $g\in \cO(U)$ tel que $j^\sharp(g) = f$. Le faisceau $j_{\ast} \cF$ est coh\'erent, son support est contenu dans le lieu d'annulation de~$g$ et on a $(j_{\ast}\cF)_{j(x)} = \cF_{x}$. On peut donc supposer que~$X$ est un ouvert~$U$ de~$\E{n}{\cA}$. 


D'apr\`es la remarque~\ref{rem:rigeptrans}, quitte \`a permuter les variables, on peut supposer qu'il existe~$k\in\cn{0}{n}$ tel que $x_k := \pi_{n,k}(x)$ soit purement localement transcendant au-dessus de~$\pi_{n}(x)$ et que~$x$ soit rigide \'epais au-dessus de~$x_k$. 


Si $n=k$, alors le r\'esultat d\'ecoule du lemme~\ref{lem:Ruckertloctr}. Supposons donc que $n>k$.
Notons $\AunA(T_{n+1})$ la droite affine analytique au-dessus de~$\cA$ munie de la coordonn\'ee~$T_{n+1}$. Notons $F \colon U \to \AunA(T_{n+1})$ le morphisme associ\'e \`a $f \in \cO_{\E{n}{\cA}}(U)$ par la proposition~\ref{morphsec}. D'apr\`es la proposition~\ref{prop:grapheimmersion}, le morphisme $\Gamma_{F} \colon U \to U\times_{\cA} \AunA(T_{n+1})$ est une immersion ferm\'ee. Identifions le produit $U \times_{\cA} \AunA(T_{n+1})$  \`a $\pi_{n+1,n}^{-1}(U)$.

Consid\'erons maintenant le faisceau coh\'erent~$(\Gamma_{F})_*\cF$. Soit $y \in \pi_{n+1,n}^{-1}(U)$ tel que $((\Gamma_{F})_*\cF)_y \ne 0$. Alors~$y$ poss\`ede un unique ant\'ec\'edent~$z$ par~$\Gamma_{F}$ et on a
\[ \cF_{z} \simeq ((\Gamma_{F})_*\cF)_y = 0.\]
On en d\'eduit que
\[ T_{n+1}(y) = \Gamma_{F}^\sharp(f)(y) = f(z) = 0.\]
Par cons\'equent, le support de~$(\Gamma_{F})_*\cF$ est contenu dans le ferm\'e analytique de~$\E{n+1}{\cA}$ d\'efini par~$T_{n+1}$.  
Il suffit maintenant de montrer qu'il existe~$d\in\N$ tel que
\[T_{n+1}^d ((\Gamma_{F})_*\cF)_{\Gamma_{F}(x)} \simeq ((\Gamma_{F})_*(f^d\cF))_x =0.\]

Pour la suite du raisonnement, par commodit\'e, \'echangeons les variables~$T_{k+1}$ et~$T_{n+1}$. On a alors $\pi_{n+1,k+1}(\Gamma_{F}(x)) = 0_{k+1,x_{k}}$, o\`u $0_{k+1,x_{k}} \in \E{k+1}{\cA}$ d\'esigne le point~0 de la fibre $\pi_{k+1,k}^{-1}(x_{k})$.

Notons $0_{n+1,x_{k}} \in \E{n+1}{\cA}$ le point~0 de la fibre $(\pi_{n+1,k})^{-1}(x_{k})$. D'apr\`es le lemme~\ref{lem:translation}, il existe un voisinage~$U'$ de~$\Gamma_{F}(x)$ dans~$\E{n+1}{\cA}$, un voisinage ouvert~$V$ de~$0_{n+1,x_{k}}$ dans~$\E{n+1}{\cA}$ tel que $\pi_{n+1,k+1}(U') = \pi_{n+1,k+1}(V)$ et un morphisme fini $\varphi \colon U' \to V$ tel que l'on ait $\varphi^{-1}(0_{n+1,x_k})=\{\Gamma_{F}(x)\}$ et $\pi_{n+1,k+1} \circ \varphi = \pi_{n+1,k+1}$.

D'apr\`es le th\'eor\`eme~\ref{thm:fini}, $\varphi_*(\Gamma_{F})_{\ast}\cF$ est un faisceau coh\'erent. Puisque~$\varphi$ commute \`a~$\pi_{n+1,k+1}$, on a $\varphi^\sharp(T_{k+1}) = T_{k+1}$. Par cons\'equent, le support de~$\varphi_*(\Gamma_{F})_{\ast}\cF$ est contenu dans le ferm\'e analytique de~$\E{n+1}{\cA}$ d\'efini par~$T_{k+1}$.  En outre, on a
\[(\varphi_*(\Gamma_{F})_{\ast}\cF)_{\varphi(\Gamma_{F}(x))} \simeq ((\Gamma_{F})_{\ast}\cF)_{\Gamma_{F}(x)}.\]
Par cons\'equent, on peut supposer que $x = 0_{n+1,x_{k}}$ et que $f = T_{k+1}$. Le r\'esultat d\'ecoule alors des lemmes~\ref{lem:Ruckertprojectioncorps} et~\ref{lem:Ruckertprojectionavd}.
\end{proof}

\begin{coro}\label{cor:nilpotent}\index{Fonction!nilpotente}\index{Nilpotent|see{Fonction nilpotente}}
Soient $X$ un espace $\cA$-analytique. Soit~$f$ un \'el\'ement de~$\cO(X)$ qui s'annule identiquement sur~$X$. Alors, pour tout $x\in X$, l'image de~$f$ dans~$\cO_{X,x}$ est nilpotente.
\qed
\end{coro}

\begin{nota}\label{nota:VI}%
\nomenclature[Eib]{$V(\cJ)$}{lieu des z\'eros d'un faisceau d'id\'eaux coh\'erent~$\cJ$}%
\nomenclature[Eia]{$\sqrt{\cJ}$}{radical d'un faisceau d'id\'eaux coh\'erent~$\cJ$}%
\nomenclature[Eic]{$\cI(Z)$}{faisceau d'id\'eaux des fonctions s'annulant en tout point d'un ferm\'e analytique~$Z$ d'un espace $\cA$-analytique}
Soit $X$ un espace $\cA$-analytique.

Soit~$\cJ$ un faisceau d'id\'eaux sur~$X$. On note $V(\cJ)$ le lieu des z\'eros de~$\cJ$ et $\sqrt{\cJ}$ le radical de~$\cJ$.

Soit~$Z$ un ferm\'e analytique de~$X$. On note $\cI(Z)$ le faisceau d'id\'eaux des fonctions s'annulant en tout point de~$Z$.
\end{nota}

\begin{coro}\label{cor:IVJ}\index{Nullstellensatz}\index{Theoreme@Th\'eor\`eme!Nullstellensatz}\index{Fonction!nilpotente}\index{Faisceau!radical d'un}\index{Faisceau!lieu des z\'eros d'un}\index{Ferme analytique@Ferm\'e analytique!faisceau associ\'e \`a un}
Soit $X$ un espace $\cA$-analytique. Soit~$\cJ$ un faisceau d'id\'eaux coh\'erent sur~$X$. Alors, on a 
\[\cI(V(\cJ)) = \sqrt{\cJ}.\]

En particulier, $\cI(X)$ co\"incide avec le faisceau d'id\'eaux~$\sqrt{0}$ form\'e par les \'el\'ements nilpotents.
\end{coro}
\begin{proof}
Soit $x\in X$. 

Pour tout $f \in (\sqrt{\cJ})_{x}$, il existe $d \in \N$ tel que $f^d \in \cJ_{x}$. On en d\'eduit que $f \in \cI(V(\cJ))_{x}$.

Soit $f \in \cI(V(\cJ))_{x}$. Le lieu des z\'eros de~$\cJ$ s'identifie au support du faisceau $\cO/\cJ$. Par cons\'equent, au voisinage de~$x$, $f$ s'annule sur ce support. D'apr\`es le th\'eor\`eme~\ref{thm:Ruckert}, il existe $d\in \N$ tel que $f^d \cO_{x}/\cJ_{x} = 0$, autrement dit, $f^d \in \cJ_{x}$.

La derni\`ere partie du r\'esultat s'obtient en appliquant la premi\`ere au faisceau $\cJ = 0$.
\end{proof}

\section{Crit\`eres d'ouverture de morphismes finis}\label{sec:ouverture}
\index{Morphisme analytique!ouvert|(}

Nous allons maintenant utiliser les r\'esultats obtenus dans les sections pr\'ec\'edentes de ce chapitre pour formuler des conditions assurant que les morphismes finis soient ouverts.

\begin{defi}\index{Espace analytique!reduit@r\'eduit|textbf}
Soient~$X$ un espace $\cA$-analytique. On dit que l'espace~$X$ est \emph{r\'eduit en un point } $x$ de~$X$ si l'anneau local~$\cO_{X,x}$ est r\'eduit. On dit que l'espace~$X$ est \emph{r\'eduit} s'il est r\'eduit en tout point.
\end{defi}

On peut d'ores et d\'ej\`a remarquer que, pour tout~$n\in\N$, l'espace affine analytique~$\E{n}{\cA}$ est r\'eduit. En effet, d'apr\`es le th\'eor\`eme~\ref{rigide}, ses anneaux locaux sont r\'eguliers.
\index{Espace affine analytique!reduit@r\'eduit}

Le r\'esultat suivant est inspir\'e de \cite[3.3.2, Criterion of Openness]{Gr-Re2}.

\begin{prop}\label{prop:ouvert}
Soit~$\varphi \colon X\to Y$ un morphisme d'espaces~$\cA$-analytiques. Soit $x\in X$ et supposons que $\varphi$ est fini en~$x$. Supposons que $Y$ est r\'eduit en~$\varphi(x)$ et que $\cO_{X,x}$ est un $\cO_{Y,\varphi(x)}$-module sans torsion. Alors, le morphisme~$\varphi$ est ouvert en~$x$.
\end{prop}
\begin{proof}
Le r\'esultat \'etant local en~$x$, on peut supposer que~$\varphi$ est fini. Il suffit de montrer que $\varphi(X)$ contient un voisinage de~$\varphi(x)$. Notons $x_{1} := x, x_{2},\dotsc,x_{n}$ les \'el\'ements de $\varphi^{-1}(\varphi(x))$. D'apr\`es le lemme~\ref{lem:voisinagefibre}, il existe un voisinage ouvert~$V$ de~$\varphi(x)$ tel que~$\varphi^{-1}(V)$ soit une union disjointe d'ouverts $\bigsqcup_{i=1}^n U_i$, avec, pour tout~$i\in\cn{1}{n}$, $x_i \in  U_i$. Quitte \`a remplacer $U$ par $U_{1}$, on peut supposer que $\varphi^{-1}(\varphi(x))=\{x\}$.


Posons $\cF := \varphi_*\cO_{U}$. On a alors $\cF_{\varphi(x)} = \cO_{U,x}$. D'apr\`es le th\'eor\`eme~\ref{thm:fini}, c'est un faisceau coh\'erent, et son support $\varphi(U)$ est donc un ferm\'e analytique de~$V$. Si ce ferm\'e n'est pas un voisinage de~$\varphi(x)$, alors il existe un voisinage~$V'$ de~$\varphi(x)$ dans~$V$ et $f\in \cO(V')$ non nulle qui s'annule identiquement sur $V' \cap \varphi(U)$. D'apr\`es le th\'eor\`eme~\ref{thm:Ruckert}, il existe~$d\in\N$ tel que~$f^d \cF_{\varphi(x)}=0$. Puisque~$Y$ est r\'eduit, l'image de~$f^d$ dans~$\cO_{Y,\varphi(x)}$ n'est pas nulle, donc  $\cF_{\varphi(x)} = \cO_{U,x}$ est de torsion, et on aboutit \`a une contradiction.
\end{proof}

\begin{defi}\label{def:plat}\index{Morphisme analytique!plat}
Soit~$\varphi \colon X\to Y$ un morphisme d'espaces $\cA$-analytiques. On dit que le morphisme~$\varphi$ est \emph{plat en un point} $x$ de~$X$ si le morphisme d'anneaux $\varphi^\sharp \colon \cO_{Y,\varphi(x)} \to \cO_{X,x}$ est plat. On dit que le morphisme~$\varphi$ est \emph{plat} s'il est plat en tout point de~$X$.
\end{defi}


\index{Morphisme analytique!fini!et plat|(}

Les morphismes finis et plats sont stables par changement de base.

\begin{prop}\label{changement_base_plat}\index{Produit!fibre@fibr\'e}\index{Extension des scalaires}
Soit $\varphi \colon X \to Y$ un morphisme d'espaces $\cA$-analytiques. Soit $x\in X$ et supposons que $\varphi$ est fini et plat en~$x$.

\begin{enumerate}[i)]
\item Soit $\psi \colon Z \to Y$ un morphisme d'espaces $\cA$-analytiques. Alors le morphisme $\varphi_{Z} \colon X\times_{Y} Z \to Z$ obtenu par changement de base \`a~$Z$ est fini et plat en tout point au-dessus de~$x$.

\item Soit $f\colon \cA \to \cB$ un morphisme d'anneaux de Banach born\'e, o\`u $\cB$ est un anneau de base g\'eom\'etrique. Alors le morphisme $\varphi_{\cB} \colon X\ho{\cA}\cB \to Y\ho{\cA}\cB$ obtenu par extension des scalaires \`a~$\cB$ est fini et plat en tout point au-dessus de~$x$.
\end{enumerate}
\end{prop}
\begin{proof}
Nous nous contenterons de d\'emontrer le point~i), l'autre s'obtenant de fa\c con similaire. Soit $t \in X\times_{Y}Z$ un point au-dessus de~$x$. D'apr\`es la proposition~\ref{stabilite_fini}, le morphisme~$\varphi_{Z}$ est fini en~$t$.

Il reste \`a montrer que l'anneau~$\cO_{X\times_{Y}Z,t}$ est plat sur~$\cO_{Z,\varphi_{Z}(t)}$. Par hypoth\`ese, $\cO_{X,x}$ est plat sur~$\cO_{Y,\varphi(x)}$, donc $\cO_{X,x}\otimes_{\cO_{Y,\varphi(x)}}\cO_{Z,\varphi_{Z}(t)}$ est plat sur~$\cO_{Z,\varphi_{Z}(t)}$. Or, d'apr\`es le th\'eor\`eme~\ref{thm:changementdebaselocal}, $\cO_{X\times_{Y}Z,t}$ s'identifie \`a un sous-module de $\cO_{X,x}\otimes_{\cO_{Y,\varphi(x)}}\cO_{Z,\varphi_{Z}(t)}$. Le r\'esultat s'ensuit.
\end{proof}

Les morphismes finis et plats sont ouverts.

\begin{coro}\label{plat}\index{Endomorphisme de la droite}
Soit $\varphi \colon X\to Y$ un morphisme d'espaces $\cA$-analytiques. Soit $x\in X$. Supposons que $\varphi$ est fini et plat en~$x$. 
Alors, $\varphi$ est ouvert en~$x$. 

En particulier, pour tout polyn\^ome~$P$ \`a coefficients dans~$\cA$, le morphisme $\varphi_{P} \colon \AunA \to \AunA$ induit par~$P$ (\cf~exemple~\ref{ex:fini}) est fini et plat, et donc ouvert.
\end{coro}
\begin{proof}
Rappelons que l'on d\'esigne par~$\cI(Y)$ le faisceau d'id\'eaux sur~$Y$ constitu\'e des fonctions qui s'annulent en tout point (\cf~notation~\ref{nota:VI}). Posons $y :=\varphi(x)$. D'apr\`es le th\'eor\`eme~\ref{rigide}, $\cO_{Y,y}$ est noeth\'erien, donc~$\cI(Y)_{y}$ est engendr\'e par un nombre fini d'\'el\'ements $g_{1},\dotsc,g_{m} \in\cO_{Y,y}$. Quitte \`a restreindre~$Y$ (et remplacer $X$ par sa pr\'eimage par~$\varphi$), on peut supposer que les~$g_j$ sont d\'efinis sur~$Y$ et s'annulent en tout point.


Notons~$\widetilde Y$ le ferm\'e analytique de~$Y$ d\'efini par le faisceau d'id\'eaux~$\cJ$ engendr\'e par les~$g_{j}$ et $r \colon \wti Y \to Y$ l'immersion ferm\'ee correspondante. Puisque les~$g_{j}$ s'annulent identiquement sur~$Y$, $r$ est un hom\'eomorphisme. Consid\'erons le diagramme cart\'esien
\[\begin{tikzcd}
X \times_{Y} \wti Y  \ar[d,"\tilde\varphi"] \ar[r, "r'"] & X \ar[d,"\varphi"]\\
\wti Y \ar[r,"r"] & Y
\end{tikzcd}.\]
Tout comme~$r$, le morphisme~$r'$ est un hom\'eomorphisme. En effet, d'apr\`es le lemme~\ref{lem:produitouverts}, le produit fibr\'e $X \times_{Y} \wti Y$ s'identifie au ferm\'e analytique d\'efini par l'id\'eal engendr\'e par $\varphi^\ast \cJ$, qui est encore constitu\'e de fonctions s'annulant en tout point. 

Puisque $r$ et~$r'$ sont des hom\'eomorphismes, il suffit de montrer que $\tilde\varphi$ est ouvert en~${r'}^{-1}(x)$. D'apr\`es la proposition~\ref{changement_base_plat}, i), $\tilde\varphi$ est fini est plat en~${r'}^{-1}(x)$. En particulier, le $\cO_{\wti Y,r^{-1}(y)}$-module $\cO_{X \times_{Y} \wti Y,{r'}^{-1}(x)}$ est libre, et donc sans torsion. Or, on a 
\[\cO_{\widetilde Y,r^{-1}(y)} = \cO_{Y,y}/\cJ_y = \cO_{Y,y}/\cI(Y)_y.\] 
Par cons\'equent, d'apr\`es le corollaire~\ref{cor:IVJ}, l'espace~$\wti Y$ est r\'eduit en~$r^{-1}(y)$. Le r\'esultat d\'ecoule alors de la proposition~\ref{prop:ouvert}.

\medbreak

D\'emontrons la seconde partie du r\'esultat. On a d\'ej\`a vu dans l'exemple~\ref{ex:fini} que le morphisme~$\varphi_{P}$ est fini. Consid\'erons une copie~$X_{T}$ de~$\AunA$ avec coordonn\'ee~$T$ et une autre~$X_{S}$ avec coordonn\'ee~$S$. On peut identifier~$X_{S}$ au ferm\'e analytique de~$\E{2}{\cA}$ avec coordonn\'ees $T,S$ d\'efini par $P(S)-T$. Notons $\pi \colon \E{2}{\cA} \to X_{T}$ le morphisme de projection. Le morphisme~$\varphi_{P}$ s'identifie alors \`a~$\pi_{|X_{S}}$ et sa platitude d\'ecoule du lemme~\ref{cas_particulier_coh\'erence} et de la proposition~\ref{prop:fetoileenbasexact}, i).
 \end{proof}

Nous pouvons maintenant pr\'eciser la conclusion du lemme~\ref{lem:translation}.

\begin{lemm}\label{lem:translation'}\index{Point!rigide epais@rigide \'epais}
Dans le lemme~\ref{lem:translation}, on peut imposer au morphisme $\varphi$ d'\^etre fini et plat, et donc ouvert.
\end{lemm}
\begin{proof}
Il suffit de reprendre la d\'emonstration du lemme~\ref{lem:translation} en utilisant le corollaire~\ref{plat} et le fait que les morphismes finis et plats sont stables par changement de base, d'apr\`es la proposition~\ref{changement_base_plat}, et par composition. 
\end{proof}

\index{Morphisme analytique!fini!et plat|)}
\index{Morphisme analytique!ouvert|)}
\index{Morphisme analytique!fini|)}


\chapter{Structure locale des espaces analytiques}\label{chap:structurelocale}

Dans ce chapitre, nous d\'emontrons quelques r\'esultats sur la structure locale des espaces analytiques et en tirons des applications.

Dans la section~\ref{sec:decomposition}, nous d\'emontrons, qu'au voisinage d'un point, un espace analytique peut s'\'ecrire de fa\c{c}on unique comme union de ferm\'es analytiques int\`egres en ce point (\cf~proposition~\ref{decomp}). Il s'agit d'un r\'esultat de d\'ecomposition en composantes irr\'eductibles dans lequel on ne prend en compte que l'espace topologique sous-jacent.

Dans la section~\ref{sec:critereouverture}, nous d\'emontrons un crit\`ere permettant d'assurer qu'un morphisme ouvert en un point reste ouvert au voisinage de ce point (\cf~proposition~\ref{ouvert2}).

La section~\ref{sec:normalisation} contient une variante locale du lemme de normalisation de Noether dans notre contexte (\cf~th\'eor\`eme~\ref{proj}). Elle assure qu'\'etant donn\'e un point en lequel un espace analytique est int\`egre, il existe un voisinage de ce point pouvant s'envoyer par un morphisme fini et ouvert vers un ouvert d'un espace affine.

Dans la section~\ref{sec:dimlocale}, nous d\'efinissons la dimension locale d'un espace analytique (\cf~d\'efinition~\ref{def:dimlocale}). Nous d\'emontrons, au passage, un r\'esultat d'int\'er\^et ind\'ependant~: un morphisme structural plat est ouvert (\cf~proposition~\ref{prop:morphismestructuralouvert}).

Finalement, dans les deux derni\`eres sections, nous comparons certaines propri\'et\'es d'un sch\'ema ou d'un morphisme de sch\'emas \`a celles de son analytifi\'e. Nous commen\c{c}ons, dans la section~\ref{sec:schan1}, par des r\'esultats assez directs~: pour un morphisme, surjectivit\'e, propret\'e et finitude se transf\`erent \`a l'analytifi\'e (\cf~proposition~\ref{stabilit\'e_analytification1}), pour un faisceau, analytification et image directe par un morphisme fini commutent (\cf~proposition~\ref{formule_analytification}). 

La section~\ref{sec:schan2} est bas\'ee sur un r\'esultat plus difficile~: la platitude du morphisme d'analytification (\cf~th\'eor\`eme~\ref{platitude_analytification}). Nous la d\'emontrons pour une classe plus restrictive d'anneaux de base, celle des anneaux de Dedekind analytiques (\cf~d\'efinition~\ref{def:aDa}). Nos exemples usuels (corps valu\'es, anneaux d'entiers de corps de nombres, corps hybrides, anneaux de valuation discr\`ete, anneaux de Dedekind trivialement valu\'es) appartiennent \`a cette classe. Elle nous permet notamment de comparer la dimension d'un sch\'ema \`a celle de son analytifi\'e (\cf~th\'eor\`eme~\ref{thm:dimXXan}). 

\medbreak

Soit $(\cA,\nm)$ un anneau de base g\'eom\'etrique. 
Posons $B:=\cM(\cA)$.

\section[Composantes irr\'eductibles]{D\'ecomposition locale des espaces analytiques en composantes irr\'eductibles}\label{sec:decomposition}

Dans cette section, nous suivons de pr\`es \cite[\S 4.1.3]{Gr-Re2}. Commen\c cons par d\'efinir les espaces analytiques localement int\`egres \'evoqu\'es dans l'introduction.

\index{Espace analytique!integre@int\`egre|(}
\begin{defi}\index{Espace analytique!integre@int\`egre|textbf}
Soit~$X$ un espace $\cA$-analytique. On dit que l'espace~$X$ est \emph{int\`egre en un point } $x$ de~$X$ si l'anneau local~$\cO_{X,x}$ est intègre. On dit que l'espace~$X$ est \emph{localement int\`egre} s'il est int\`egre en tout point.
\end{defi}

Pour tout~$n\in\N$, l'espace affine~$\E{n}{\cA}$ est localement int\`egre, puisque tous ses anneaux locaux sont r\'eguliers, d'apr\`es le th\'eor\`eme~\ref{rigide}.
\index{Espace affine analytique!integre@int\`egre}

\medbreak

Rappelons que l'on note~$\cI(X)$ le faisceau d'id\'eaux des fonctions s'annulant en tout point de~$X$, $\sqrt{0}$ celui des \'el\'ements nilpotents de~$\cO_{X}$ (\cf~notation~\ref{nota:VI}) et que ces deux faisceaux co\"incident, d'apr\`es le corollaire~\ref{cor:IVJ}.

\index{Espace analytique!irreductible@irr\'eductible|(}
\begin{defi}\index{Espace analytique!irreductible@irr\'eductible|textbf}
Soit~$X$ un espace $\cA$-analytique. On dit que l'espace~$X$ est \emph{irr\'eductible en un point} $x$ de~$X$ si~$\cI(X)_x$ est un id\'eal premier de~$\cO_{X,x}$.
\end{defi}

\begin{lemm}\label{lem:integreirreductible}
Soient~$X$ un espace $\cA$-analytique et $x \in X$. Si~$X$ est int\`egre en~$x$, alors $X$ est irr\'eductible en~$x$.
\end{lemm}
\begin{proof}
Supposons que~$X$ est int\`egre en~$x$. Alors, l'anneau~$\cO_{X,x}$ est int\`egre et $\cI(X)_{x} = (\sqrt{0})_{x} =0$ est un id\'eal premier de~$\cO_{X,x}$.
\end{proof}

\begin{rema}
Il existe des espaces analytiques irr\'eductibles en un point mais non int\`egres en ce point. C'est le cas du ferm\'e analytique de~$\AunA$, avec coordonn\'ee~$T$, d\'efini par le polyn\^ome~$T^2$.
\end{rema}

Le r\'esultat qui suit permet cependant d'affirmer qu'un espace analytique irr\'eductible en un point n'est pas tr\`es loin d'\^etre int\`egre en ce point.

\begin{prop}\label{representant}
Soient~$X$ un espace $\cA$-analytique et~$x\in X$. Supposons que~$X$ est irr\'eductible en~$x$. Alors, il existe un voisinage~$V$ de~$x$ dans~$X$ et un ferm\'e analytique~$\widetilde V$ de~$V$ de m\^eme ensemble sous-jacent tels que~$\widetilde V$ soit int\`egre en~$x$.
\end{prop}
\begin{proof}
Par hypoth\`ese, l'id\'eal~$\cI(X)_x$ de~$\cO_{X,x}$ est premier. D'apr\`es le th\'eor\`eme~\ref{rigide}, $\cO_{X,x}$ est noeth\'erien, donc~$\cI(X)_{x}$ est engendr\'e par un nombre fini d'\'el\'ements $f_{1},\dotsc,f_{n} \in\cO_{X,x}$. Soit~$V$ un voisinage de~$x$ dans~$X$ sur lequel tous les~$f_i$ sont d\'efinis. Quitte \`a restreindre~$V$, on peut supposer que les~$f_{i}$ s'annulent en tout point de~$V$.

Notons~$\widetilde V$ le sous-espace analytique de~$V$ d\'efini par le faisceau d'id\'eaux~$\cJ$ engendr\'e par les~$f_i$. Ensemblistement, on a $\widetilde V = V$. Qui plus est, on a 
\[\cO_{\widetilde V,x} = \cO_{X,x}/\cJ_x = \cO_{X,x}/\cI(X)_x,\] 
donc $\widetilde V$ est int\`egre en~$x$.
\end{proof}
\index{Espace analytique!irreductible@irr\'eductible|)}

D\'emontrons maintenant un \'enonc\'e local de d\'ecomposition d'un espace analytique en ferm\'es analytiques int\`egres. Commen\c cons par rappeler un r\'esultat d'alg\`ebre commutative. Il se d\'eduit, par exemple de \cite[IV, \S 2, \no 2, th\'eor\`eme~1]{BourbakiAC14} (pour l'existence) et \cite[IV, \S 2, \no 3, proposition~4]{BourbakiAC14} (pour l'unicit\'e).

\begin{theo}\label{noeth}
Soient~$R$ un anneau noeth\'erien et~$I$ un id\'eal de~$R$ tel que~$R/I$ est r\'eduit. Alors, il existe une unique famille finie d'id\'eaux premiers~$\p_1,\dotsc,\p_s$ de~$R$ v\'erifiant, pour tous $i, j \in \cn{1}{s}$ avec $i\ne j$, $\p_i\not\subset \p_j$ et telle que
\[I=\bigcap_{i=1}^s \p_i.\]
\end{theo}

\begin{lemm}\label{lem:Jintegre}
Soient~$X$ un espace $\cA$-analytique et $V$ un ferm\'e analytique de~$X$ d\'efini par un faisceau d'id\'eaux~$\cJ$ de~$\cO_{X}$. Si~$V$ est int\`egre en un point~$x$, alors on a $\cJ_{x} = \cI(V)_{x}$.
\end{lemm}
\begin{proof}
Soit $x\in V$ tel que~$V$ est int\`egre en~$x$. Alors, $\cJ_{x}$ est un id\'eal premier de~$\cO_{X,x}$. D'apr\`es le corollaire~\ref{cor:IVJ}, on a donc $\cI(V)_{x} = \sqrt{\cJ}_{x} = \cJ_{x}$.
\end{proof}

\begin{prop}\label{decomp}
Soient~$X$ un espace $\cA$-analytique et~$x\in X$. Alors, il existe un voisinage ouvert~$V$ de~$x$ et des ferm\'es analytiques $V_{1},\dotsc,V_{s}$ de~$V$ contenant~$x$ et int\`egres en~$x$ v\'erifiant, pour tous $i, j \in \cn{1}{s}$ avec $i\ne j$, $V_{i}\not\subset V_j$ et tels que
\[V=\bigcup_{i=1}^s V_i.\]

De plus, cette d\'ecomposition est unique au sens suivant~: s'il existe un voisinage ouvert~$V'$ de~$x$ et des ferm\'es analytiques $V'_{1},\dotsc,V'_{t}$ de~$V'$ satisfaisant les m\^emes propri\'et\'es, alors on a $s=t$ et il existe un voisinage ouvert~$W$ de~$x$ dans $V\cap V'$ et une permutation~$\sigma$ de $\cn{1}{s}$ tels que
\[\forall i\in \cn{1}{s},\ V_{i} \cap W = V'_{\sigma(i)} \cap W.\]
\end{prop}
\begin{proof}
$\bullet$ \emph{Existence}

Soient $\p_{1},\dotsc,\p_{s}$ les id\'eaux premiers associ\'es \`a l'id\'eal~$\cI(X)_{x}$ de~$\cO_{X,x}$ par le th\'eor\`eme~\ref{noeth}. Pour tout $i\in \cn{1}{s}$, soient $f_{i,1},\dotsc,f_{i,k_{i}}$ des g\'en\'erateurs de~$\p_{i}$. Soit~$V$ un voisinage ouvert de~$x$ sur lequel tous les~$f_{i,j}$ sont d\'efinis. Pour tout $i\in \cn{1}{s}$, on note~$\cI_{i}$ le faisceau d'id\'eaux de~$\cO_{V}$ engendr\'e par $f_{i,1},\dotsc,f_{i,k_{i}}$ et $V_{i}$ le ferm\'e analytique de~$V$ qu'il d\'efinit.

Pour tout $i\in \cn{1}{s}$, on a, par d\'efinition, $\cO_{V_{i},x} = \cO_{X,x}/\cI_{i,x} = \cO_{X,x}/\p_{i}$, donc $V_{i}$ est int\`egre en~$x$. Pour tous $i,j \in \cn{1}{s}$ avec $i\ne j$, on a $\p_{i} \not\subset \p_{j}$ et on en d\'eduit que $V_{j} \not\subset V_{i}$. Quitte \`a restreindre~$V$, on peut supposer que, pour tout $i\in \cn{1}{s}$ et tout $j\in \cn{1}{k_{i}}$, on a $f_{i,j} \in \cI(X)(V)$. On a alors $V=\bigcup_{i=1}^s V_i$.

\medbreak

$\bullet$ \emph{Unicit\'e}

Soient $V'$ un voisinage ouvert de~$x$ et $V'_{1},\dotsc,V'_{t}$ v\'erifiant les propri\'et\'es de l'\'enonc\'e. Pour tout $i\in \cn{1}{t}$, notons~$\cI'_{i}$ le faisceau d'id\'eaux de~$V'$ d\'efinissant~$V'_{i}$. 

D'apr\`es le lemme~\ref{lem:Jintegre}, pour tout $i\in \cn{1}{t}$, on a $\cI(V'_{i})_{x} = \cI'_{i,x}$ et c'est un id\'eal premier de~$\cO_{X,x}$. Pour tous $i,j\in \cn{1}{t}$ avec $i\ne j$, on a $\cI(V'_{i})_{x} \not\subset \cI(V'_{j})_{x}$. En outre, puisque $V' = \bigcup_{i=1}^t V'_i$, on a $\cI(X)_{x} = \bigcap_{i=1}^t \cI(V'_{i})_{x}$. Le th\'eor\`eme~\ref{noeth} assure alors que la famille $(\cI(V'_{i})_{x})_{1\le i\le t}$ est une permutation de  la famille $(\cI(V_{i})_{x})_{1\le i\le s}$. Le r\'esultat s'en d\'eduit.
\end{proof}

\begin{defi}\label{def:composanteslocales}\index{Composantes locales|textbf}
Soient~$X$ un espace $\cA$-analytique et~$x\in X$. Les ferm\'es analytiques $V_{1},\dotsc,V_{s}$ de la proposition~\ref{decomp} sont appel\'es \emph{composantes locales de~$x$ en~$X$}.
\end{defi}
\index{Espace analytique!integre@int\`egre|)}

\section{Un crit\`ere d'ouverture}\label{sec:critereouverture}

Dans cette section, nous d\'emontrons un r\'esultat technique permettant d'assurer qu'un morphisme fini et ouvert en un point le reste sur un voisinage. Nous suivons le raisonnement de~\cite[\S 3.3.3]{Gr-Re2}.

\begin{nota}\index{Faisceau!torsion d'un}\index{Faisceau!dual d'un}%
\nomenclature[Ec]{$\cF^\tors$}{sous-faisceau des sections de torsion de~$\cF$}%
\nomenclature[Ed]{$\cF^\vee$}{dual de~$\cF$}
Soient $X$ un espace $\cA$-analytique et~$\cF$ un faisceau de $\cO_{X}$-modules. 

On note $\cF^\tors$ le sous-faisceau de~$\cF$ form\'e des sections de torsion.

On note $\cF^\vee := \sHom_{\cO_{X}}(\cF,\cO_{X})$ le dual de~$\cF$.
\end{nota}

\begin{lemm}\label{lem:Ftors}\index{Faisceau!coherent@coh\'erent}
Soient $X$ un espace $\cA$-analytique localement int\`egre et~$\cF$ un faisceau coh\'erent sur~$X$. Alors on a 
\[\cF^\tors = \sKer(\cF \to \cF^{\vee\vee}).\]

En particulier, le faisceau~$\cF^\tors$ est coh\'erent.
\end{lemm}
\begin{proof}
Il suffit de montrer que, pour tout $x\in X$, on a $\cF^\tors_{x} = \Ker(\cF_{x} \to \cF^{\vee\vee}_{x})$.

Soit $x\in X$. Soit $f\in \cF^\tors_{x}$. Par d\'efinition, il existe $a\in \cO_{X,x} \setminus \{0\}$ tel que $af = 0$ dans~$\cF_{x}$. Soit $\varphi \in \cF^\vee_{x}$. On a $a \varphi(f) = \varphi(af) = 0$ dans $\cO_{X,x}$, donc $\varphi(f) = 0$ car $\cO_{X,x}$ est int\`egre. On en d\'eduit que l'image de~$f$ dans~$\cF^{\vee\vee}_{x}$ est nulle. Nous avons donc montr\'e que  $\cF^\tors_{x} \subset \Ker(\cF_{x} \to \cF^{\vee\vee}_{x})$.

Soit $f\in \cF_{x} \setminus \cF^{\tors}_{x}$. Posons $\cG := \cF/\cF^\tors$. Posons $V_{x} := (\cO_{X,x}\setminus \{0\})^{-1} \cG_{x}$. C'est un espace vectoriel de dimension finie sur~$\Frac{\cO_{X,x}}$ et le morphisme naturel $\cG_{x} \to V_{x}$ est injectif. Notons~$f'$ (resp.~$f'')$ l'image de~$f$ dans~$\cG_{x}$ (resp. $V_{x}$). L'\'el\'ement~$f''$ n'\'etant pas nul, il existe un morphisme $\Frac{\cO_{X,x}}$-lin\'eaire $\varphi''_{x} \colon V_{x} \to  \Frac{\cO_{X,x}}$ tel que $\varphi_{x}(f'') \ne 0$. Quitte \`a multiplier par un \'el\'ement convenable de~$\cO_{X,x}$, on obtient un morphisme $\cO_{X,x}$-lin\'eaire $\varphi'_{x} \colon \cG_{x} \to \cO_{X,x}$ tel que $\varphi'_{x}(f') \ne 0$. Le faisceau~$\cG$ \'etant de type fini, le morphisme~$\varphi'_{x}$ est induit par un morphisme $\varphi \colon \cG \to \cO_{X}$ d\'efini sur un voisinage de~$x$. On en d\'eduit que $f \notin \Ker(\cF_{x} \to \cF^{\vee\vee}_{x})$. Ceci cl\^ot la d\'emonstration.
\end{proof}

Venons-en maintenant au r\'esultat annonc\'e.

\begin{prop}\label{ouvert2}\index{Morphisme analytique!ouvert}\index{Espace analytique!integre@int\`egre}
Soit~$\varphi \colon X\to Y$ un morphisme fini entre espaces~$\cA$-analytiques, avec $Y$ localement int\`egre. Soit~$x\in X$. Si~$\varphi$ est ouvert en~$x$ et~$X$ est int\`egre en~$x$, alors il existe un voisinage ouvert~$U$ de~$x$ et un voisinage ouvert~$V$ de~$\varphi(x)$ tels que~$\varphi(U)\subset V$ et le morphisme $\varphi_{U,V} \colon U\to V$ induit par~$\varphi$ soit fini et ouvert.
\end{prop}
\begin{proof}
D'apr\`es le lemme~\ref{lem:voisinagefibre}, quitte \`a restreindre~$X$, on peut supposer que $\varphi^{-1}(\varphi(x)) = \{x\}$. Posons $\cF := \varphi_*(\cO_X)$. D'apr\`es le th\'eor\`eme~\ref{thm:fini}, c'est un faisceau coh\'erent. D'apr\`es la proposition~\ref{prop:ouvert}, il suffit de montrer qu'il existe un voisinage ouvert~$V$ de~$\varphi(x)$ dans~$Y$ tel que pour tout~$y\in V$, $\cF_y$ soit un~$\cO_{Y,y}$-module sans torsion. D'apr\`es le lemme~\ref{lem:Ftors}, $\cF^\tors$ est coh\'erent. Par cons\'equent, son support est ferm\'e. Il suffit donc de montrer que $\cF_{\varphi(x)}$ est un~$\cO_{Y,\varphi(x)}$-module sans torsion. 

Montrons tout d'abord que le morphisme $\colon \cO_{Y,\varphi(x)} \to \cO_{X,x}$ induit par~$\varphi^\sharp$ est injectif. Soit~$f\in \cO_{Y,\varphi(x)}$ dont l'image~$\varphi^\sharp(f)$ dans~$\cO_{X,x}$ est nulle. Il existe un voisinage~$U$ de~$x$ dans~$X$ tel que~$\varphi^\sharp(f)$ est nulle sur~$U$. En utilisant la proposition~\ref{prop:proprietemorphismes}, on en d\'eduit que~$f$ est nulle en tout point de~$\varphi(U)$. Puisque~$\varphi$ est ouverte en~$x$, $f$ appartient donc \`a~$\cI(Y)_{\varphi(x)}$, donc~$f$ est nilpotente dans~$\cO_{Y,\varphi(x)}$. Or~$Y$ est localement int\`egre donc $f$ est nulle dans~$\cO_{Y,\varphi(x)}$.

Puisque $\varphi^{-1}(\varphi(x)) = \{x\}$, on a $\cF_{\varphi(x)}\simeq \cO_{X,x}$. Or, par hypoth\`ese, $\cO_{X,x}$ est int\`egre. Le r\'esultat s'ensuit. 
\end{proof}

\section{Normalisation de Noether locale}\label{sec:normalisation}

Le but de cette section est de d\'emontrer un analogue local du lemme de normalisation de Noether, valable en un point en lequel l'espace est int\`egre.

\medbreak

Rappelons la d\'efinition~\ref{def:projection} concernant les projections.
Pour tout $b\in B$, notons $0_{b,l} \in \E{n}{\cA}$ le point~0 de la fibre~$\pi_{l}^{-1}(b)$. Commen\c cons par un cas particulier.

\begin{lemm}\label{lem:finivoisinagefibre}
Soient~$X$ un espace $\cA$-analytique. Notons $\pi \colon X \to B$ le morphisme structural. Soient~$b \in B$, $x \in \pi^{-1}(b)$, $n \in \N$, $V_{0}$ un voisinage ouvert de~$0_{b,n}$ dans~$\E{n}{\cA}$ et $\psi\colon X \to V_{0}$ un morphisme fini tel que $\psi(x) = 0_{b,n}$. Alors, il existe un voisinage ouvert~$W$ de~$b$ dans~$B$, un voisinage ouvert~$U$ de~$x$ dans~$X$, un entier $l\le n$, un voisinage ouvert~$V$ de~$0_{b,l}$ dans~$\E{l}{\cA}$ et un morphisme fini $\varphi\colon U \to V$ tels que 
\begin{enumerate}[i)]
\item $\pi(U) = \pi_{l}(V) = W$ et $\pi_{l} \circ \varphi = \pi$ ;
\item $\varphi(U)$ est un voisinage de~$0_{b,l}$ dans~$\pi_{l}^{-1}(b)$.
\end{enumerate}
\end{lemm}
\begin{proof}
D\'emontrons le r\'esultat par r\'ecurrence sur~$n$. Si $n=0$, le r\'esultat est satisfait avec $l=0$ puisque $\pi_{0}^{-1}(b) = \{b\}$.

Supposons que $n\ge 1$ et que le r\'esultat est vrai pour~$n-1$. Si $\psi(X)$ est un voisinage de~$0_{b,n}$ dans~$\pi_{n}^{-1}(b)$, alors $\psi$ satisfait toutes les propri\'et\'es requises. Supposons maintenant que tel n'est pas le cas.

D'apr\`es le corollaire~\ref{cor:imagefini}, $\psi(X)$ est un ferm\'e analytique de~$V_{0}$ et $\psi(X) \cap \pi_{n}^{-1}(b)$ est strictement contenu dans $V_{0} \cap \pi_{n}^{-1}(b)$. On en d\'eduit qu'il existe un voisinage ouvert~$V$ de~$0_{b,n}$ dans~$V_{0}$ et un \'el\'ement~$f$ de~$\cO(V)$ qui soit nul sur $V \cap \psi(X)$ mais pas identiquement nul sur $V \cap \pi_{n}^{-1}(b)$.  Posons $U := \psi^{-1}(V)$ et $W:= \pi_{n}(V)$. D'apr\`es le corollaire~\ref{cor:projectionouverte}, $W$~est un ouvert de~$B$. D'apr\`es le lemme~\ref{changement_variable}, il existe un isomorphisme $\sigma \colon\pi_{n}^{-1}(W) \to \pi_{n}^{-1}(W)$ tel que $\sigma(0_{b,n}) = 0_{b,n}$ et la restriction de~$\sigma^\sharp(f)$ \`a~$\pi_{n,n-1}^{-1}(b)$ ne soit pas nulle. Le point~$0_{b,n}$ est alors isol\'e dans $\sigma(\psi(U)) \cap \pi_{n,n-1}^{-1}(b)$. 

Posons $\varphi := \pi_{n,n-1} \circ \sigma \circ \psi$. On a alors $\varphi(x) = 0_{b,n-1}$ et le point~$x$ est isol\'e dans~$\varphi^{-1}(0_{b,n-1})$. D'apr\`es le th\'eor\`eme~\ref{thm:locfini}, le morphisme~$\varphi$ est donc fini en~$x$. Puisqu'il est \`a valeurs dans un ouvert de~$\E{n-1}{\cA}$, l'hypoth\`ese de r\'ecurrence permet de conclure.
\end{proof}

\begin{theo}\label{proj}\index{Espace analytique!integre@int\`egre}\index{Theoreme@Th\'eor\`eme!de normalisation de Noether}\index{Lemme!de normalisation de Noether}
Soit~$X$ un espace $\cA$-analytique. Notons $\pi \colon X \to B$ le morphisme structural. Soit~$x$ un point de~$X$ en lequel~$\cO_{X,x}$ est int\`egre. Posons $b := \pi(x)$ et notons $\pi^\sharp_{b,x} \colon \cO_{B,b}\to\cO_{X,x}$ le morphisme induit par~$\pi^\sharp$.

Si $\pi^\sharp_{b,x}$ est injectif, alors il existe un voisinage ouvert~$U$ de~$x$ dans~$X$, un entier $n\in \N$, un ouvert~$V$ de~$\E{n}{\cA}$ et un morphisme fini ouvert $\varphi \colon U\to V$.

Si $\pi^\sharp_{b,x}$ n'est pas injectif, alors il existe un voisinage ouvert~$U$ de~$x$ dans~$X$, un entier $n\in \N$, un ouvert~$V$ de~$\E{n}{\cH(b)}$ et un morphisme fini ouvert $\varphi \colon U\to V$.
\end{theo}
\begin{proof}
Puisque la question est locale, on peut supposer que~$X$ est un ferm\'e analytique d'un ouvert~$\tilde U$ de~$\E{n}{\cA}$. D'après la remarque~\ref{rem:rigeptrans}, quitte \`a permuter les variables, on peut supposer qu'il existe $k\in \cn{0}{n}$ tel que $x$ soit rigide \'epais au-dessus de $x_{k} :=\pi_{n,k}(x)$ et que $x_{k}$ soit purement localement transcendant au-dessus de~$b := \pi_{n}(x)$.

Pour tout $m\ge k$, notons $0_{x_{k},m} \in \E{m}{\cA}$ le point~0 de la fibre $\pi_{m,k}^{-1}(x_{k}) \simeq \E{m-k}{\cH(x_{k})}$. D'apr\`es le lemme~\ref{lem:translation}, il existe un voisinage~$W$ de~$x_{k}$ dans~$\E{k}{\cA}$, un voisinage ouvert~$U'$ de~$x$ dans~$\tilde U$ et un voisinage ouvert~$V_{0}$ de~$0_{x_{k},n}$ dans~$\E{n}{\cA}$ v\'erifiant $\pi_{n,k}(U') = \pi_{n,k}(V_{0}) = W$ et un morphisme fini $\psi \colon U' \to V_{0}$ tels que l'on ait $\psi^{-1}(0_{x_{k},n})=\{x\}$ et $\pi_{n,k} \circ \psi = \pi_{n,k}$. 

D'apr\`es le lemme~\ref{lem:finivoisinagefibre} appliqu\'e \`a~$\psi_{|X}$, quitte \`a restreindre~$W$, on peut supposer qu'il existe un voisinage ouvert~$U$ de~$x$ dans~$X$, un entier $l\ge k$, un voisinage ouvert~$V$ de~$0_{x_{k},l}$ dans~$\E{l}{\cA}$ et un morphisme fini $\varphi\colon U \to V$ tels que 
$\pi_{n,k}(U) = \pi_{l,k}(V) = W$, $\pi_{l,k} \circ \varphi = \pi_{n,k}$ et $\varphi(U)$ est un voisinage de~$0_{x_{k},l}$ dans~$\pi_{l,k}^{-1}(b)$. En outre, en utilisant le lemme~\ref{lem:voisinagefibre}, on montre que, quitte \`a restreindre~$U$, on peut supposer que $\varphi^{-1}(0_{x_{k},l}) = \{x\}$. Posons $y := 0_{x_{k},l}$.

D'apr\`es le th\'eor\`eme~\ref{thm:fini}, le faisceau $\varphi_{\ast} \cO_{U}$ est coh\'erent. Notons~$\cI$ le noyau du morphisme $\varphi^\sharp \colon \cO_{V} \to \varphi_{\ast} \cO_{U}$. L'ensemble~$\varphi(U)$ est l'ensemble sous-jacent au ferm\'e analytique de~$V$ d\'efini par~$\cI$. Nous consid\'ererons d\'esormais~$\varphi(U)$ comme un espace analytique muni de cette structure.

Remarquons que~$\cI_{y}$ est un id\'eal premier de~$\cO_{V,y}$. En effet, puisque $\varphi^{-1}(y) = \{x\}$, $\cI_{y}$ est le noyau du morphisme $\cO_{V,y}\to\cO_{U,x}$ induit par~$\varphi^\sharp$. On en d\'eduit que $\cO_{\varphi(U),y}\simeq\cO_{V,y}/\cI_{y}$ s'injecte dans~$\cO_{U,x}$, qui est suppos\'e int\`egre. Le r\'esultat s'ensuit.

Nous allons \`a pr\'esent traiter s\'epar\'ement les deux cas pr\'esent\'es dans l'\'enonc\'e.

\smallbreak

$\bullet$ Supposons que le morphisme~$\pi^\sharp_{b,x}$ est injectif.

Puisque~$U$ est int\`egre en~$x$ et que~$V$ est localement int\`egre, d'apr\`es la proposition~\ref{ouvert2}, il suffit de montrer que~$\varphi$ ouvert en~$x$. 
Puisqu'on a $\varphi^{-1}(y) = \{x\}$, d'apr\`es le lemme~\ref{lem:voisinagefibre}, pour tout voisinage~$U'$ de~$x$, $\varphi(U')$ est un voisinage de~$y$ dans~$\varphi(U)$.  Il suffit donc de montrer que~$\varphi(U)$ est un voisinage de~$y$ dans~$V$. 

Raisonnons par l'absurde et supposons que~$\varphi(U)$ n'est pas un voisinage de~$y$ dans~$V$. Dans ce cas, $\cI_{y}$ n'est pas nul. Soit $f \in \cI_{y} \setminus\{0\}$. 

Supposons que~$\cO_{B,b}$ est un corps fort. D'apr\`es le th\'eor\`eme~\ref{rigide}, $\cO_{\E{k}{\cA},x_k}$ est encore un corps fort. D'apr\`es le lemme~\ref{restriction_fibre}, l'image de~$f$ dans~$\cO_{\pi_{l,k}^{-1}(x_k),y}$ n'est pas nulle. Or l'image de~$f$ dans~$\cO_{\varphi(U),y}$ est nulle et $\varphi(U)$ est un voisinage de~$y$ dans $\pi_{l,k}^{-1}(x_k)$. On aboutit \`a une contradiction.

Supposons que $\cO_{B,b}$ est un anneau de valuation discr\`ete. Fixons-en une uniformisante~$\pi_{b}$. D'apr\`es le th\'eor\`eme~\ref{rigide}, $\cO_{\E{k}{\cA},x_k}$ est encore un anneau fortement de valuation discr\`ete d'uniformisante~$\pi_{b}$. D'apr\`es le lemme~\ref{restriction_fibre_avd}, il existe $v\in \N$ et $g\in\cO_{V,y}$ tels que la restriction de~$g$ \`a~$\pi_{l,k}^{-1}( x_{k})$ ne soit pas nulle et $f=\pi_b^v g$ dans~$\cO_{V,y}$. 

Puisque le morphisme $\pi^\sharp_{b,x}\colon \cO_{B,b}\to\cO_{X,x}$ est injectif, le morphisme $\cO_{B,b}\to\cO_{\varphi(U),y} \simeq\cO_{V,y}/\cI_{y}$ l'est aussi. Par cons\'equent, l'image de~$\pi_{b}$ dans $\cO_{V,y}$ n'appartient pas \`a~$\cI_{y}$. Puisque~$\cI_{y}$ est premier, on en d\'eduit que $g\in \cI_{y}$. On en d\'eduit une contradiction, comme pr\'ec\'edemment.

\smallbreak

$\bullet$ Supposons que le morphisme~$\pi^\sharp_{b,x}$ n'est pas injectif.

Ce cas ne peut se produire que lorsque~$\cO_{B,b}$ est un anneau de valuation discr\`ete. Fixons une uniformisante~$\pi_{b}$ de~$\cO_{B,b}$. Notons~$V_{b}$ le ferm\'e analytique de~$V$ d\'efini par~$\pi_{b}$. Rappelons que le morphisme $\cO_{\varphi(U),y} \simeq \cO_{V,y}/\cI_{y}   \to\cO_{U,x}$ induit par~$\varphi^\sharp$ est injectif. Puisque $\pi^\sharp_{b,x}\colon \cO_{B,b}\to\cO_{U,x}$ n'est pas injectif, l'image de~$\pi_{b}$ dans~$\cO_{V,y}$ appartient \`a~$\cI_{y}$. Par cons\'equent,  quitte \`a r\'etr\'ecir~$U$ et~$V$, on peut supposer que~$\varphi(U)$ est un ferm\'e analytique de~$V_b$. Le morphisme~$\varphi$ se factorise donc par un morphisme $\varphi' \colon U\to V_b$. On peut alors appliquer \`a~$\varphi'$ le m\^eme raisonnement que pr\'ec\'edemment, dans le cas des corps forts, pour montrer qu'il est ouvert en~$x$. On conclut alors par la proposition~\ref{ouvert2}.
\end{proof}

\begin{rema}
Dans le cas classique des espaces sur un corps valu\'e, on peut d\'emontrer un r\'esultat similaire \`a celui du th\'eor\`eme~\ref{proj} sans condition sur l'espace~$X$ et construire un morphisme ouvert en~$x$ (qui sera donc ouvert au voisinage de~$x$ si~$X$ est irr\'eductible, par exemple). Notre preuve fournit d'ailleurs ce r\'esultat lorsque~$\cO_{B,b}$ est un corps.

En revanche, cet \'enonc\'e plus g\'en\'eral tombe en d\'efaut lorsque~$\cO_{B,b}$ est un anneau de valuation discr\`ete, comme on le v\'erifie sur l'exemple du ferm\'e analytique~$X$ de~$\E{1}{\Z}$, avec coordonn\'ee~$T$, d\'efini par l'\'equation $2T=0$ et du point $x$ d\'efini par $T(x) = 2(x) = 0$.
\end{rema}

\begin{rema}
L'entier~$n$ pr\'esent dans l'\'enonc\'e du th\'eor\`eme~\ref{proj} ne d\'epend que de~$x$. Pour s'en convaincre, il suffit de remarquer que, dans tous les cas, on a un morphisme fini et ouvert $U \cap \pi^{-1}(b) \to \E{n}{\cH(b)}$. 
On est alors ramen\'es au cas classique des espaces sur un corps valu\'e complet et on v\'erifie que l'entier~$n$ n'est autre que la dimension de $U \cap \pi^{-1}(b)$ en~$x$.
\end{rema}

\section{Dimension locale d'un espace analytique}\label{sec:dimlocale}

Le but de cette section est de d\'efinir la dimension d'un espace analytique en un point. 

Commen\c cons par d\'emontrer un r\'esultat utile d'ouverture des morphismes plats. Nous remercions Mattias Jonsson qui nous a sugg\'er\'e de l'inclure.

\begin{prop}\label{prop:morphismestructuralouvert}\index{Morphisme analytique!plat}\index{Morphisme analytique!ouvert}\index{Morphisme!structural}
Soit~$X$ un espace~$\cA$-analytique. Notons $\pi \colon X \to B$ le morphisme structural. Soit $x\in X$. Si le morphisme $\cO_{B,\pi(x)} \to \cO_{X,x}$ induit par~$\pi^\sharp$ est plat, alors le morphisme~$\pi$ est ouvert au point~$x$.
\end{prop}
\begin{proof}
Supposons que le morphisme $\cO_{B,\pi(x)} \to \cO_{X,x}$ induit par~$\pi^\sharp$ est plat. Soit~$U$ un voisinage ouvert de~$x$ dans~$X$. Posons $b:=\pi(x)$. Montrons que~$\pi(U)$ est un voisinage de~$b$. 

D'apr\`es la proposition~\ref{decomp}, quitte \`a r\'etr\'ecir le voisinage~$U$, on peut supposer qu'il s'\'ecrit sous la forme $U = \bigcup_{i=1}^s U_{i}$, o\`u les~$U_{i}$ sont des ferm\'es analytiques de~$U$ contenant~$x$ et int\`egres en~$x$. Il suffit de montrer qu'il existe $i\in \cn{1}{s}$ tel que $\pi(U_{i})$ soit un voisinage de~$b$. 

Montrons, tout d'abord, qu'il existe $i\in \cn{1}{s}$ tel que le morphisme naturel $\cO_{B,b}\to\cO_{U_i,x}$ soit injectif. Si~$\cO_{B,b}$ est un corps, le r\'esultat est \'evident. Supposons donc que~$\cO_{B,b}$ est un anneau fortement de valuation discr\`ete. Soit~$\varpi_{b}$ une uniformisante de cet anneau. Il suffit de montrer qu'il existe $i\in \cn{1}{s}$ tel que l'image de~$\varpi_{b}$ dans $\cO_{U_i,x}$ ne soit pas nulle. Supposons, par l'absurde, que tel ne soit pas le cas. Alors $\varpi_{b}$ s'annule en tout point d'un voisinage de~$x$ dans~$U$, donc son image dans~$\cO_{U,x}$ est nilpotente, d'apr\`es le corollaire~\ref{cor:nilpotent}. Ceci contredit la platitude du morphisme $\cO_{B,b}\to\cO_{U,x}$.

Soit~$i_0 \in \cn{1}{s}$ tel que le morphisme $\cO_{B,b}\to\cO_{U_{i_0},x}$ est injectif. Puisque~$U_{i_0}$ est int\`egre en~$x$, le th\'eor\`eme~\ref{proj} assure que, quitte \`a r\'etr\'ecir~$U_{i_0}$, il existe un morphisme fini ouvert $\varphi \colon U_{i_0}\to V$, o\`u~$V$ est un ouvert d'un espace affine analytique~$\E{n}{\cA}$. Or, d'apr\`es le corollaire~\ref{cor:projectionouverte}, le morphisme de projection $\E{n}{\cA}\to B$ est ouvert. Le r\'esultat s'en d\'eduit.
\end{proof}

Venons-en maintenant aux questions de dimension. Nous nous permettrons d'utiliser la th\'eorie classique de la dimension des espaces analytiques sur un corps valu\'e complet (\cf~\cite[\S~2.3]{Ber1} ou \cite[\S~1]{variationdimension}).
\index{Dimension!en un point|see{Espace analytique}}\index{Espace analytique!dimension d'un|(}

\begin{defi}\index{Espace analytique!dimension d'un|textbf}
\nomenclature[Ksa]{$\dim_{x}(X/\cM(\cA))$}{dimension relative d'un espace $\cA$-analytique~$X$ en un point~$x$}
Soit~$X$ un espace $\cA$-analytique. Notons $\pi \colon X \to B$ le morphisme structural. Soit $x\in X$ tel que $x$ soit int\`egre en~$X$. On appelle \emph{dimension relative de~$X$ en~$x$} la quantit\'e
\[ \dim_{x}(X/B) := \dim_{x}\big(\pi^{-1}(\pi(x))\big).\]
\end{defi}

Avant de d\'efinir la dimension locale d'un espace, introduisons encore une nouvelle notion.

\begin{defi}\label{def:vertical}\index{Espace analytique!vertical en un point|textbf}
Soit~$X$ un espace $\cA$-analytique. Notons $\pi \colon X \to B$ le morphisme structural. On dit que l'espace~$X$ est \emph{vertical en~$x$} si~$\pi^{-1}(\pi(x))$ est un voisinage de~$x$ dans~$X$.
\end{defi}

\begin{nota}%
\nomenclature[Ksb]{$v_{x}(X)$}{$ =0$ si $X$ est vertical en~$x$, 1 sinon}
Soient~$X$ un espace $\cA$-analytique et $x\in X$. On pose 
\[ v_{x}(X) :=
\begin{cases} 
0 \textrm{ si } X \textrm{ est vertical en } x~;\\ 
1 \textrm{ sinon.}
\end{cases}\] 
\end{nota}

\begin{defi}\index{Point!isolable|textbf}
Un point~$b$ de~$B$ est dit \emph{isolable} s'il existe un voisinage ouvert~$V$ de~$b$ dans~$B$ et $h\in \cO(V)$ tels que
\[ \{ c\in V : h(c)=0\} = \{b\}.\]
\end{defi}

\begin{rema}\label{rem:isolable}
Soit $b\in B$. Si $b$ est isol\'e dans~$B$, alors il est isolable. Si~$b$ n'est pas isol\'e dans~$B$, alors il ne peut \^etre isolable que si $\cO_{B,b}$ n'est pas un corps, auquel cas c'est un anneau de valuation discr\`ete.

Dans nos exemples standard \ref{ex:corpsvalue} \`a \ref{ex:Dedekind}, la r\'eciproque de l'énoncé précédent vaut également. Ainsi, dans $\cM(\Z)$ (\cf~exemple~\ref{ex:Z}), les points isolables sont-ils exactement les points~$a_{p}^{+\infty}$, avec $p$ premier, chacun de ces points \'etant isol\'e par l'\'equation $p=0$.

Notons cependant que, dans le cas général, il existe des points non isolables dont l'anneau local est de valuation discr\`ete. L'anneau $\C_{\hyb}\la \vert T\vert\le 1 \ra$ fournit un exemple de cette situation. D'apr\`es la proposition~\ref{prop:spectreseries}, le spectre~$\cM(\C_{\hyb}\la \vert T\vert\le 1 \ra)$ s'identifie au disque unit\'e ferm\'e au-dessus de~$\cM(\C_{\hyb})$. Le point~0 de $\cM(\C_{\hyb}\la \vert T\vert\le 1 \ra)$ au-dessus de la valeur absolue triviale (ou de toute autre) n'est pas isolable, mais l'anneau local en ce point est un anneau de valuation discr\`ete.
\end{rema}

\begin{lemm}\label{lem:vertical}\index{Morphisme analytique!plat}\index{Morphisme!structural}
Soit~$X$ un espace $\cA$-analytique. Notons $\pi \colon X \to B$ le morphisme structural. Soit $x\in X$ tel que $X$ soit irr\'eductible en~$x$ et $b:= \pi(x)$ soit isolable. Notons $\pi_{b,x}^\sharp \colon \cO_{B,b} \to \cO_{X,x}$ le morphisme induit par~$\pi^\sharp$.

Si $\cO_{B,b}$ est un corps, alors $X$ est vertical en~$x$.

Si $\cO_{B,b}$ est un anneau de valuation discr\`ete, alors les propositions suivantes sont \'equivalentes~:
\begin{enumerate}[i)]
\item $X$ est vertical en~$x$~;
\item le morphisme $\pi_{b,x}^\sharp$ n'est pas plat~;
\item le morphisme $\pi_{b,x}^\sharp$ n'est pas injectif~;
\item toute uniformisante de~$\cO_{B,b}$ est nilpotente dans~$\cO_{X,x}$.
\end{enumerate}
\end{lemm}
\begin{proof}
Supposons que~$\cO_{B,b}$ est un corps. D'apr\`es la remarque~\ref{rem:isolable}, le point~$b$ est isol\'e, donc $X$ est vertical en~$x$.

Supposons que $\cO_{B,b}$ est un anneau de valuation discr\`ete.

$i) \implies ii)$ Supposons, par l'absurde, que $\pi_{x,b}^\sharp$ est plat. Alors, d'apr\`es la proposition~\ref{prop:morphismestructuralouvert}, $\pi(\pi^{-1}(b)) = \{b\}$ est un voisinage de~$b$ et on aboutit \`a une contradiction.

$ii) \implies iii)$ Soit~$\varpi_{b}$ une uniformisante de $\cO_{B,b}$. Par hypoth\`ese, il existe un entier~$n\ge 1$ et un \'el\'ement $f$ de~$\cO_{X,x}$ non nul tels que $\varpi_{b}^n f = 0$. En d'autres termes, $\varpi_{b}^n f \in \cI(X)_{x}$. Or $\cI(X)_{x}$ est un id\'eal premier de~$\cO_{X,x}$, donc $\varpi_{b}^n =0$ dans~$\cO_{X,x}$. On en d\'eduit que $\pi_{b,x}^\sharp$ n'est pas injectif.

$iii) \implies iv)$ C'est imm\'ediat.

$iv) \implies i)$ Soit~$\varpi_{b}$ une uniformisante de~$\cO_{B,b}$. Puisque~$b$ est isolable, il existe un voisinage ouvert~$V$ de~$b$ dans~$B$ et $h\in \cO(V)$ tels que $\{ c\in V : h(c)=0\} = \{b\}$. Quitte \`a restreindre~$V$, on peut \'ecrire~$h$ sous la forme $h = \varpi_{b}^n g$ avec $n\ge 1$ et $g \in \cO(V)$ inversible. On a donc
\[ \{b\} = \{ c\in V : h(c)=0\} = \{ c\in V : \varpi_{b}(c)=0\}.\]
Puisque l'image de~$\varpi_{b}$ dans~$\cO_{X,x}$ est nilpotente, le ferm\'e analytique d\'efini par~$\varpi_{b}$ dans~$\pi^{-1}(V)$ est un voisinage de~$x$ dans~$X$. Or, ce ferm\'e n'est autre que~$\pi^{-1}(b)$.
\end{proof}

\begin{defi}\label{def:dimlocale}\index{Espace analytique!dimension d'un|textbf}
\nomenclature[Ksc]{$\dim_{x}(X)$}{dimension de~$X$ en~$x$}
Soit~$X$ un espace $\cA$-analytique. Soit $x\in X$ tel que $x$ soit irr\'eductible en~$X$. On appelle \emph{dimension de~$X$ en~$x$} la quantit\'e
\[ \dim_{x}(X) := \dim_{x}(X/B) + v_{x}(X).\]

Soit~$x\in X$. D\'ecomposons~$X$ au voisinage de~$x$ sous la forme $X = \bigcup_{i=1}^s V_{i}$, o\`u $V_{1},\dotsc,V_{s}$ sont int\`egres en~$x$ comme dans la proposition~\ref{decomp}. On appelle \emph{dimension de~$X$ en~$x$} la quantit\'e
\[ \dim_{x}(X) := \max_{1\le i\le s} \dim_{x}(V_{i}).\]
\end{defi}
\index{Espace analytique!dimension d'un|)}

\section[D'un sch\'ema \`a son analytifi\'e~1]{D'un sch\'ema \`a son analytifi\'e~: premi\`eres propri\'et\'es}\label{sec:schan1}
\index{Analytification!|(}

Dans cette section, nous nous int\'eressons aux propri\'et\'es d'un sch\'ema localement de pr\'esentation finie sur~$\cA$ qui sont pr\'eserv\'ees lorsque l'on passe \`a son analytifi\'e. 

\medbreak

Rappelons qu'\`a tout sch\'ema localement de pr\'esentation finie~$\cX$ sur~$\cA$, on associe un espace $\cA$-analytique~$\cX^\an$ et un morphisme d'espaces localement $\cA$-annel\'es $\rho_{\cX} \colon \cX^\an \to \cX$ (\cf~d\'efinition~\ref{def:analytification}).

\begin{lemm}\label{lem:Pnancompact}\index{Espace analytique!projectif}
Pour tout $n\in \N$, l'espace $\EP{n}{\cA}$ est compact.
\end{lemm}
\begin{proof}
Il suit de la construction de l'analytifi\'e (\cf~th\'eor\`eme~\ref{thm:analytification}) que $\EP{n}{\cA}$ peut \^etre construit \`a partir d'espaces affines~$\E{n}{\cA}$ par la prod\'edure de recollement habituelle. On v\'erifie directement que $\EP{n}{\cA}$ est alors recouvert par les images des disques unit\'es ferm\'es de ces espaces affines. Le r\'esultat s'ensuit, par la compacit\'e des disques.

\end{proof}

\begin{lemm}\label{lem:imageanalytification}\index{Morphisme analytique!image d'un}
Soient $\cX$ et~$\cY$ des sch\'emas localement de pr\'esentation finie sur~$\cA$ et $\varphi\colon \cX\to \cY$ un morphisme de sch\'emas. Soit $y\in \cY^{an}$. Munissons $\varphi^{-1}(\rho_Y(y))$ de sa structure de sch\'ema sur $\kappa(\rho_{Y}(y))$ canonique et $(\varphi^\an)^{-1}(y)$ de sa structure d'espace analytique sur $\cH(y)$ canonique (\cf~proposition~\ref{preimage}). On a alors un isomorphisme canonique d'espaces $\cH(y)$-analytiques
\[(\varphi^{-1}(\rho_Y(y))\otimes_{\kappa(\rho_Y(y))}\cH(y))^{an} \simeq (\varphi^\an)^{-1}(y).\]
En particulier, on a 
\[ \rho_{\cY}(\varphi^\an(\cX^\an)) = \varphi(\cX).\]
\end{lemm}
\begin{proof}
On a 
\[\varphi^{-1}(\rho_{Y}(y)) \simeq \cX \otimes_{\cY} \kappa(\rho_{Y}(y))\]
et
\[ (\varphi^\an)^{-1}(y) \simeq (\cX^\an \ho{\cA} \cH(y)) \times_{\cY^\an \ho{\cA} \cH(y)} \cM(\cH(y)).\]
L'isomorphisme d\'ecoule alors des propri\'et\'es universelles (\cf~d\'efinitions~\ref{def:analytification} et~\ref{def:extensionscalaire}). 

En particulier, $\varphi^{-1}(\rho_{\cY}(y))$ est vide si, et seulement si, $(\varphi^\an)^{-1}(y)$ l'est. L'\'egalit\'e finale s'en d\'eduit.
\end{proof}

%

\begin{prop}\label{stabilit\'e_analytification1}\index{Morphisme analytique!surjectif}\index{Morphisme analytique!propre}\index{Morphisme analytique!fini}\index{Immersion}
Soit $\varphi\colon \cX\to \cY$ un morphisme de sch\'emas localement de pr\'esentation finie sur~$\cA$. Si le morphisme~$\varphi$ est (1) une immersion (resp. une immersion ferm\'ee, resp. une immersion ouverte), (2) surjectif, (3) propre, (4) s\'epar\'e, (5) fini, alors le morphisme~$\varphi^{an}$ poss\`ede la m\^eme propri\'et\'e.
\end{prop}
\begin{proof}
(1) Les trois propri\'et\'es d\'ecoulent de la construction de l'analytifi\'e (\cf~th\'eor\`eme~\ref{thm:analytification}).

(2) Cette propri\'et\'e d\'ecoule du lemme~\ref{lem:imageanalytification}.

(3) Supposons que~$\varphi$ est propre. D'apr\`es le lemme de Chow (\cf~\cite[corollaire~5.6.2]{EGAII})), il existe un $\cY$-sch\'ema projectif~$\cX'$ et un morphisme de $\cY$-sch\'emas $\psi \colon \cX'\to \cX$ projectif et surjectif. D'apr\`es (2), $\psi^\an$ est surjectif. D'apr\`es le lemme~\ref{lem:compositionpropre}, il suffit donc de montrer le r\'esultat pour les morphismes projectifs.

Supposons que~$\varphi$ est projectif. Puisque la propret\'e est une notion locale au but, quitte \`a r\'etr\'ecir~$\cX$ et~$\cY$, on peut supposer que~$\varphi$ s'\'ecrit comme la compos\'ee d'une immersion ferm\'ee $\cX\to \P^n_\cA\times_{\cA} \cY$ et de la projection sur le second facteur $\P^n_{\cA}\times_{\cA} \cY\to \cY$. Remarquons qu'il d\'ecoule imm\'ediatement des propri\'et\'es universelles que l'analytification commute au produit. D'apr\`es~(1), $\cX^\an\to \EP{n}{\cA}\times_{\cA} \cY^\an$ est une immersion ferm\'ee, donc un morphisme propre. D'apr\`es le lemme~\ref{lem:Pnancompact}, $\EP{n}{\cA}$ est compact, donc la projection $\EP{n}{\cA} \to\cM(\cA)$ est propre. D'apr\`es la proposition~\ref{stabilite_propre}, la projection $\EP{n}{\cA}\times_{\cA} Y\to Y$ l'est encore. On conclut en utilisant le fait que la compos\'ee de deux morphismes propres est propre (\cf~lemme~\ref{lem:compositionpropre}).

(4) Cette propri\'et\'e d\'ecoule du point~(1) et du lemme~\ref{lem:produitan}.

(5) Supposons que~$\varphi$ est un morphisme fini. Alors, d'apr\`es~(3) et~(4), $\varphi^{an}$ est propre et s\'epar\'e. Il suffit donc de montrer que les fibres de~$\varphi^\an$ sont finies. Cela d\'ecoule de l'isomorphisme qui figure au d\'ebut de la preuve de~(2). 
\end{proof}

\begin{rema}
Il est tr\`es plausible que les implications d\'emontr\'ees \`a la proposition~\ref{stabilit\'e_analytification1} soient en fait des \'equivalences. Nous n'en dirons pas plus ici.
\end{rema}

Nous allons maintenant d\'emontrer des r\'esultats sur les faisceaux.

\begin{defi}\index{Analytification!d'un faisceau|textbf}\index{Faisceau!analytifie@analytifi\'e|see{Analytification}}%
\nomenclature[Ee]{$\cF^\an$}{analytifi\'e d'un faisceau coh\'erent~$\cF$ sur un sch\'ema}
Soit~$\cX$ un sch\'ema localement de pr\'esentation finie sur~$\cA$. Pour tout faisceau de $\cO_{\cX}$-modules~$\cF$, on appelle \emph{analytification} ou \emph{analytifi\'e} de~$\cF$ le faisceau de $\cO_{\cX^\an}$-modules d\'efini par
\[ \cF^\an := (\rho_{\cX})^\ast \cF.\]
\end{defi}


\begin{lemm}\index{Faisceau!coherent@coh\'erent}
Soit~$\cX$ un sch\'ema localement de pr\'esentation finie sur~$\cA$. Soit~$\cF$ un faisceau de $\cO_{\cX}$-modules. Si $\cF$ est de type fini (resp. coh\'erent), alors~$\cF^\an$ est de type fini (resp. coh\'erent).
\end{lemm}
\begin{proof}
Le foncteur $\cF \mapsto \cF^\an$ est exact \`a droite, donc il pr\'eserve les faisceaux de type fini. D'apr\`es le th\'eor\`eme~\ref{coherent}, le faisceau structural $\cO_{\cX^\an}$ est coh\'erent, donc les faisceaux coh\'erents sont exactement ceux de pr\'esentation finie. On en d\'eduit qu'il sont \'egalement pr\'eserv\'es. 
\end{proof}

\begin{prop}\label{formule_analytification}\index{Morphisme analytique!fini!image d'un faisceau par un}
Soit~$\varphi \colon \cX\to \cY$ un morphisme fini de sch\'emas localement de pr\'esentation finie sur~$\cA$. Soit~$\cF$ un faisceau de $\cO_{\cX}$-modules. Alors le morphisme naturel
\[(\varphi_*\cF)^{an}\too (\varphi^{an})_*\cF^{an}\]
est un isomorphisme.
\end{prop}
\begin{proof}
Consid\'erons le diagramme commutatif
\[\begin{tikzcd}
\cX^{\an} \ar[d, "\varphi^{\an}"] \ar[r, "\rho_{\cX}"] &\cX \ar[d,"\varphi"] \\
\cY^\an \ar[r, "\rho_{\cY}"]  &\cY.
\end{tikzcd}\]
Par le m\^eme raisonnement que dans la preuve du lemme~\ref{lem:cbflocal}, on v\'erifie qu'il suffit de d\'emontrer l'\'enonc\'e suivant, que l'on appellera \emph{\'enonc\'e local}~: pour tout $u\in \cX$ et tout $y\in \cY^\an$ tels que $\varphi(u) = \rho_{\cY}(y)$, le morphisme naturel 
\[\cO_{\cX,u}\otimes_{\cO_{\cY,\varphi(u)}}\cO_{\cY^\an,y} \too \prod_{x\in\rho_{\cX}^{-1}(u) \cap (\varphi^\an)^{-1}(y)}\cO_{\cX^\an, x}\]
est un isomorphisme. 

La question \'etant locale sur~$\cY$, on peut supposer que~$\cY$ est affine. Puisque~$\varphi$ est fini, $\cX$ est \'egalement affine. Il existe donc des $\cA$-alg\`ebres~$A$ et~$B$ telles que $X= \Spec(A)$ et $\cY = \Spec(B)$. En outre, $\varphi^\sharp \colon B \to A$ munit $A$ d'une structure de $B$-module de type fini. Soit~$B[T_1,\dotsc,T_k]\to A$ un morphisme de $B$-alg\`ebres surjectif. Montrons l'\'enonc\'e souhait\'e par r\'ecurrence sur~$k$. 

Supposons que~$k=0$. Alors $\varphi$ est une immersion ferm\'ee et il va de m\^eme pour~$\varphi^\an$, d'apr\`es la proposition \ref{stabilit\'e_analytification1}. Soient $y \in Y^\an$ et $u\in X$ tels que $\varphi(u) = \rho_{\cY}(y)$. D'apr\`es le lemme~\ref{lem:imageanalytification}, on a $y\in \varphi^{\an}(\cX^\an)$. Notons~$x\in \cX^\an$ l'unique point tel que $\varphi^\an(x)=y$. On a alors $\rho_{\cX}(x) = u$. Il d\'ecoule de la construction de l'analytifi\'e (\cf~th\'eor\`eme~\ref{thm:analytification}) que si~$\cX$ est d\'efini par un faisceau d'id\'eaux~$\cI$ de~$\cY$, alors $\cX^\an$ est d\'efini par le faisceau d'id\'eaux~$\cI\otimes_{\cO_{\cY}} \cO_{\cY^\an}$. On en d\'eduit que 
\[ \cO_{\cX,u} \otimes_{\cO_{\cY,\rho_{\cY}(y)}} \cO_{\cY^\an,y} \simeq \cO_{\cX^\an,x}  ,\] 
ce qui n'est autre que l'\'enonc\'e local souhait\'e.

\medbreak

Supposons que $k\ge 1$ et que l'\'enonc\'e est vrai pour~$k-1$. Notons $A_{k-1}$ l'image de $B[T_{1},\dotsc,T_{k-1}]$ par la surjection $B[T_1,\dotsc,T_k]\to A$ et posons $\cX_{k-1} := \Spec(A_{k-1})$. L'inclusion~$A_{k-1}\to A$ induit un morphisme de sch\'emas $\psi \colon X\to X_{k-1}$ et le morphisme~$\varphi$ se factorise sous la forme $\varphi = \psi' \circ \psi $ avec $\psi' \colon X_{k-1} \to Y$. Par hypoth\`ese de r\'ecurrence, $\psi'$ satisfait l'\'enonc\'e local. Il suffit de d\'emontrer que~$\psi$ le satisfait \'egalement.

Par construction, il existe une surjection $A_{k-1}[S]\to A$. Puisque~$A$ est un~$A_{k-1}$-module de type fini, il existe un polyn\^ome unitaire~$P\in A_{k-1}[S]$ tel que cette surjection se factorise par $A_{k-1}[S]/(P)\to A$. Puisque~$B$ est une $\cA$-alg\`ebre de pr\'esentation finie, il existe un morphisme de $\cA$-alg\`ebres surjectif $\cA[S_{1},\dotsc,S_{n}] \to A_{k-1}$. On notera encore~$P$ un rel\`evement unitaire de~$P$ \`a $\cA[S_1,\dotsc,S_n][S]$. Soit~$Z_{P}$ le sous-sch\'ema ferm\'e de~$\E{n+1}{\cA}$ d\'efini par~$P$. On a un diagramme commutatif
\[\begin{tikzcd}
X\ar[r]\ar[d, "\psi"] & Z_{P} \ar[d, "\pi_{n+1,n}"]\\
X_{k-1}\ar[r]&\E{n}{\cA},
\end{tikzcd}\]
o\`u les morphismes $X\to Z_{P}$ et~$X_k\to \E{n}{\cA}$ sont des immersions ferm\'ees et le morphisme $\pi_{n+1,n} \colon Z_{P}\to\E{n}{\cA}$ est la restriction \`a~$Z_{P}$ de la projection sur les~$n$ premiers facteurs. On v\'erifie qu'il suffit de d\'emontrer l'\'enonc\'e local pour le morphisme~$\pi_{n+1,n}$. Ce dernier d\'ecoule du lemme \ref{cas_particulier_coh\'erence}.
\end{proof}

\section[D'un sch\'ema \`a son analytifi\'e~2]{D'un sch\'ema \`a son analytifi\'e~: platitude et cons\'equences}\label{sec:schan2}

Soit $\cX$ un sch\'ema localement de pr\'esentation finie sur~$\cA$. L'objet de cette section est de d\'emontrer que le morphisme $\rho_{\cX} \colon \cX^\an \to \cX$ est plat. Il s'agit d'un r\'esultat essentiel pour les applications. En g\'eom\'etrie complexe, il est par exemple utilis\'e dans la preuve du th\'eor\`eme GAGA de J.~P.~Serre (\cf~\cite{GAGA}). 

Nous aurons besoin d'imposer des propri\'et\'es suppl\'ementaires sur l'anneau~$\cA$, permettant de relier ses localis\'es alg\'ebriques et analytiques. Pour ce faire, nous introduisons une nouvelle d\'efinition.

\begin{defi}\label{def:aDa}\index{Anneau!de Dedekind analytique|textbf}
On dit que $(\cA,\nm)$ est un \emph{anneau de Dedekind analytique} s'il satisfait les propri\'et\'es suivantes~:
\begin{enumerate}[i)]
\item $\cA$ est un anneau de Dedekind~;
\item $\cA$ est un anneau de base g\'eom\'etrique~;
\item le morphisme $\rho \colon \cM(\cA) \to \Spec(\cA)$ est surjectif~;
\item pour tout $a \in \Spec(\cA)$,
\begin{enumerate}
\item[$\alpha$)] si $\cO_{\Spec(\cA),a}$ est un corps, alors pour tout $b\in \rho^{-1}(a)$, $\cO_{\cM(\cA),b}$ est un corps~;
\item[$\beta$)] si $\cO_{\Spec(\cA),a}$ est un anneau de valuation discr\`ete, alors $\rho^{-1}(a)$ contient un unique point, disons~$b$, $\cO_{\cM(\cA),b}$ est un anneau de valuation discr\`ete et le morphisme $\cO_{\Spec(\cA),a} \to \cO_{\cM(\cA),b}$ induit par~$\rho^\sharp$ est le morphisme de compl\'etion.
\end{enumerate}
\end{enumerate}
\end{defi}

\begin{rema}\index{Point!isolable}
Supposons que~$\cA$ est un anneau de Dedekind analytique. Soit~$b \in \cM(\cA)$ tel que $\cO_{\cM(\cA),b}$ soit un anneau de valuation discr\`ete. Alors le point~$b$ est isolable. En effet, d'apr\`es iv), $\cO_{\Spec(\cA),\rho(b)}$ est un anneau de valuation discr\`ete, et on peut donc isoler~$\rho(b)$ dans~$\Spec(\cA)$ en utilisant une uniformisante~$\varpi$ de $\cO_{\Spec(\cA),\rho(b)}$. D'apr\`es~iv), $\rho^\sharp(\varpi)$ isole encore $b$ dans~$\cM(\cA)$. 
\end{rema}

\begin{exem}\index{Corps!valu\'e}\index{Anneau!des entiers relatifs $\Z$}\index{Anneau!des entiers d'un corps de nombres}\index{Corps!hybride}\index{Anneau!de valuation discr\`ete}\index{Anneau!de Dedekind trivialement valu\'e}
Nos exemples usuels \ref{ex:corpsvalue} \`a~\ref{ex:Dedekind}~: les corps valu\'es, l'anneau~$\Z$ et les anneaux d'entiers de corps de nombres, les corps hybrides, les anneaux de valuation discr\`ete et les anneaux de Dedekind trivialement valu\'es sont tous des anneaux de Dedekind analytiques.
\end{exem}

\emph{Nous supposerons d\'esormais que $\cA$ est un anneau de Dedekind analytique. }

\medbreak

Commen\c cons par un cas particulier du r\'esultat \`a d\'emontrer.

\begin{prop}\label{prop:affineplat}
Soit $n\in \N$. Pour tout $x\in \E{n}{\cA}$, le morphisme $\cO_{\A^n_{\cA},\rho_{\A^n_{\cA}}(x)} \to \cO_{\E{n}{\cA},x}$ induit par~$\rho_{\A^n_{\cA}}^\sharp$ est plat.
\end{prop}
\begin{proof}
Posons $\rho := \rho_{\A^n_{\cA}}$. D\'emontrons l'\'enonc\'e par r\'ecurrence sur~$n$. 

Supposons que~$n=0$. Soit $x\in \cM(\cA)$. Si $\cA_{\rho(x)}$ est un corps, alors le morphisme $\cO_{\Spec(\cA),\rho(x)} = \cA_{\rho(x)} \to \cO_{\cM(\cA),x}$ est plat. Si $\cA_{\rho(x)}$ est un anneau de valuation discr\`ete, alors le morphisme $\cO_{\Spec(\cA),\rho(x)} = \cA_{\rho(x)} \to \cO_{\cM(\cA),x}$ est le morphisme de compl\'etion. Il est donc encore plat.

\medbreak

Supposons que $n\ge 1$ et que le r\'esultat vaut pour~$n-1$. Soit~$x\in\E{n}{\cA}$. Notons~$\q$ l'id\'eal premier de~$\cA$ correspondant \`a la projection de~$\rho(x)$ sur~$\Spec(\cA)$. Le point~$\rho(x)$ correspond alors \`a un id\'eal premier~$\p_{x}$ de~$\Frac(\cA/\q)[T_1,\dotsc,T_{n}]$. 

Si $\p_{x} = (0)$ et $\q = (0)$, alors $\cO_{\A^{n}_\cA,\rho(x)}$ est un corps, donc $\cO_{\E{n}{\cA},x}$ est plat sur $\cO_{\A^{n}_\cA,\rho(x)}$. 

Si $\p_{x} = (0)$ et $\q \ne (0)$, alors $\cO_{\A^{n}_\cA,\rho(x)}$ est un anneau de valuation discr\`ete d'id\'eal maximal engendr\'e par~$\q$. D'apr\`es le corollaire~\ref{cor:projectionouverte}, le morphisme de projection $\E{n}{\cA} \to \cM(\cA)$ est ouvert, donc, pour tout~$m\in \N^\ast$, l'image de~$\q^m$ dans~$\cO_{\E{n}{\cA},x}$ n'est pas nulle. On en d\'eduit que le morphisme $\cO_{\A^{n}_\cA,\rho(x)}\to \cO_{\E{n}{\cA},x}$ est injectif. Or, d'apr\`es le th\'eor\`eme~\ref{rigide}, $\cO_{\E{n}{\cA},x}$ est int\`egre. Par cons\'equent, il est sans $\q$-torsion, et donc plat sur~$\cO_{\A^{n}_\cA,\rho(x)}$.

Supposons d\'esormais que $\p_{x} \ne (0)$. Notons $\p'_{x}$ l'id\'eal de~$\Frac(\cA/\q)[T_1,\dotsc,T_{n}]$ correspondant \`a~$\rho(x)$. Quitte \`a permuter les variables, on peut supposer qu'il existe~$k\in\cn{0}{n-1}$ tel que 
\[ \p'_{x} \cap \Frac(\cA/\q)[T_1,\dotsc,T_k] = \{0\} \]
et $\kappa(\rho(x))$ soit une extension finie de $\Frac(\cA/\q)(T_1,\dotsc,T_k)$. 

Notons $x_{n-1} \in \E{n-1}{\cA}$ la projection de~$x$ sur les $n-1$ premi\`eres coordonn\'ees. Son image canonique $\rho(x_{n-1})$ dans $\A^{n-1}_{\cA}$ est l'image de~$\rho(x)$ par la projection sur les $n-1$ premi\`eres coordonn\'ees. Par construction, le morphisme naturel $\kappa(\rho(x_{n-1}))[T_{n}]\to \kappa(\rho(x))$ n'est pas injectif. Notons $P \in \kappa(\rho(x_{n-1}))[T_{n}]$ un g\'en\'erateur unitaire de son noyau et~$d$ le degr\'e de ce polyn\^ome. Soient~$\tilde{P}$ un relev\'e de~$P$ \`a $\cO_{\A^{n-1}_{\cA},\rho(x_{n-1})}[T_{n}]$  unitaire de degr\'e~$d$. 
Soit $U=\Spec(B)$ un voisinage affine de~$\rho(x_{n-1})$ dans~$\A^{n-1}_{\cA}$ sur lequel tous les coefficients de~$\tilde{P}$ sont d\'efinis.

Proc\'edons comme dans la preuve du corollaire~\ref{cor:phinormes}. Consid\'erons le sch\'ema~$\A^2_{B}$ muni des coordonn\'ees~$T_{n}$ et~$T$. Notons $p_{1} \colon \A^2_{B} \to \A^1_{B}$ (resp. $p_{2} \colon \A^2_{B} \to \A^1_{B}$) le morphisme de projection sur~$\A^1_{B}$ muni de la coordonn\'ee~$T$ (resp.~$T_{n}$). Notons $\varphi_{\tilde P} \colon \A^1_{B} \to \A^1_{B}$ le morphisme induit par le morphisme naturel
\[B[T] \too B[T,T_{n}]/(\tilde P(T_{n}) - T) \simtoo B[T_{n}] \]
et $\sigma$ la section de $p_{2}$ induite par le morphisme de $B$-alg\`ebres
\[\begin{array}{ccc}
B[T,T_{n}]&\too&B[T_{n}]\\
T&\mapstoo&\tilde{P}(T_{n})\\
T_{n}&\mapstoo&T_{n}
\end{array}.\]
On a alors $p_{1} \circ \sigma = p_{2}$.

\[\begin{tikzcd}
& \A^2_{B}\ar[ld, "p_{1}"] \ar[rd, "p_2"'] &\\
\A^1_{B} &&\A^1_{B}  \ar[ll, "\varphi_{\tilde P}"'] \ar[ul, bend right, "\sigma"']
\end{tikzcd}\]

Notons $Z$ le ferm\'e de Zariski de~$\A^{2}_{\cA}$ d\'efini par $\tilde P(T_{n})-T$. La projection~$p_{2}$ induit un isomorphisme entre~$Z$ et~$\A^1_{\cA}$ dont l'inverse est la section~$\sigma$. 

Le morphisme
\[\fonction{\psi_{\tilde P}}{\cO^d_{\A^1_B}}{(p_1)_*\cO_Z}{(a_0(T),\dotsc,a_{d-1}(T))}{\displaystyle\sum_{i=0}^{d-1}a_i(T)T_{n}^i}\] 
est un isomorphisme, dont l'inverse est le reste de la division euclidienne par~$\tilde P(T_{n})-T$. En composant par l'isomorphisme $(p_1)_*\cO_Z \to (p_1)_*\sigma_{*}\cO_Z = (\varphi_{\tilde P})_{*} \cO_{\A^1_B}$ induit par~$\sigma^\sharp$, on obtient un isomorphisme
\[\fonction{\psi'}{\cO^d_{\A^1_B}}{(\varphi_{\tilde P})_{*} \cO_{\A^1_B}}{(a_0(T),\dotsc,a_{d-1}(T))}{\displaystyle\sum_{i=0}^{d-1}a_i(\tilde P (T_{n}))T_{n}^i}.\] 

Consid\'erons $\E{n}{\cA}$ avec coordonn\'ees $(T_{1},\dotsc,T_{n-1},S)$ et $\pi_{n,n-1} \colon \E{n}{\cA} \to \E{n-1}{\cA}$ la projection sur les $n-1$ premi\`eres coordonn\'ees. Notons $0_{x_{n-1}} \in \E{n}{\cA}$ le point~0 de la fibre $\pi_{n,n-1}^{-1}(x_{n-1}) \simeq \E{1}{\cH(x_{n-1})}$. Par construction, on a $\varphi_{\tilde P}^{-1}(\rho(0_{x_{n-1}})) = \{\rho(x)\}$ et le morphisme $\psi_{\tilde P}$ induit donc un isomorphisme entre de germes
\[ \cO^d_{\A^1_B,\rho(0_{x_{n-1}})} \to \big((\varphi_{\tilde{P}})_* \cO_{\A^1_B}\big)_{\rho(0_{x_{n-1}})}\simeq \cO_{\A^1_B,\rho(x)}. \]
En analytifiant la construction, on obtient un morphisme de germes 
\[ \cO^d_{\AunA\times_{\cA} U^\an,0_{x_{n-1}}} \to \big((\varphi^\an_{\tilde{P}})_* \cO_{\AunA\times_{\cA} U^\an}\big)_{0_{x_{n-1}}}, \]
et le corollaire~\ref{cor:phinormes} assure que c'est un isomorphisme. 

Or, $x$ est un ant\'ec\'edent de~$0_{x_{n-1}}$ par~$\varphi^{an}_{\tilde{P}}$, donc, d'apr\`es la proposition~\ref{prop:fetoileenbasexact}, $\cO_{\AunA\times_{\cA} U^{\an},x}$ est un facteur de $\big((\varphi^{\an}_{\tilde{P}})_*\cO_{\AunA\times_{\cA} U^{\an}}\big)_{0_{x_{n-1}}}$. Par cons\'equent, pour montrer que le morphisme $\cO_{\A^1_B,\rho(x)}\to\cO_{\AunA\times_{\cA} U^{\an},x}$ est plat, il suffit de montrer que le morphisme $\cO_{\A^1_B,\rho(x)}\to \big((\varphi^{\an}_{\tilde{P}})_*\cO_{\AunA\times_{\cA} U^{\an}}\big)_{0_{x_{n-1}}}$ l'est. Gr\^ace aux isomorphismes obtenus pr\'ec\'edemment, cela revient encore \`a montrer que le morphisme
\[\cO_{\A^1_B,\rho(0_{x_{n-1}})} \simeq \cO_{\A^{n}_\cA,\rho(0_{x_{n-1}})} \too \cO_{\E{n}{\cA},0_{x_{n-1}}} \simeq \cO_{\AunA\times_{\cA} U^{\an},0_{x_{n-1}}}\] 
est plat. 

Pour ce faire, il suffit de montrer que le morphisme induit entre les compl\'et\'es $T_{n}$-adiques est plat. Or, ces compl\'et\'es sont respectivement isomorphes \`a $\cO_{\A^{n}_{\cA},\rho(x_{n-1})}\llbracket T_{n}\rrbracket$ et $\cO_{\E{n}{\cA},x_{n-1}}\llbracket T_{n}\rrbracket$ (d'apr\`es la proposition~\ref{prop:disqueglobal} pour le second). L'hypoth\`ese de r\'ecurrence permet de conclure.
\end{proof}

\begin{theo}\label{platitude_analytification}\index{Morphisme analytique!plat}
Soit~$\cX$ un sch\'ema localement de pr\'esentation finie sur~$\cA$. Alors, pour tout $x\in \cX^\an$, le morphisme $\cO_{\cX,\rho_{\cX}(x)}\to\cO_{\cX^{\an},x}$ induit par~$\rho_{\cX}$ est plat.
\end{theo}
\begin{proof}
Quitte \`a localiser au voisinage de~$\rho(x)$, on peut supposer que~$\cX$ est un sous-sch\'ema ferm\'e d'un espace affine~$\A^n_{\cA}$. Par construction de l'analytifi\'e (\cf~th\'eor\`eme~\ref{thm:analytification}), on a un isomorphisme
\[\cO_{X,x} \simeq \cO_{\cX,\rho_{\cX}(x)} \otimes_{\cO_{\A^{n}_{\cA},\rho_{\cX}(x)}} \cO_{\E{n}{\cA},x}.  \]
Or, d'apr\`es la proposition~\ref{prop:affineplat}, $\cO_{\E{n}{\cA},x}$ est plat sur $\cO_{\A^{n}_{\cA},\rho_{\cX}(x)}$. Le r\'esultat s'en d\'eduit.  
\end{proof}

\begin{rema}\label{rem:pfp}
Rappelons qu'un morphisme local d'anneaux locaux qui est plat est automatiquement fid\`element plat (\cf~\cite[\href{https://stacks.math.columbia.edu/tag/00HR}{lemma 00HR}]{stacks-project} par exemple). Le morphisme du th\'eor\`eme~\ref{platitude_analytification} est donc fid\`element plat.
\end{rema}

\begin{coro}\index{Analytification!d'un faisceau!exactitude du foncteur d'}
Soit~$\cX$ un sch\'ema localement de pr\'esentation finie sur~$\cA$. Le foncteur qui \`a un faisceau de $\cO_{\cX}$-modules~$\cF$ associe le faisceau de $\cO_{\cX^\an}$-modules~$\cF^\an$ est exact et fid\`ele.
\qed
\end{coro}

\begin{coro}\label{cor:XYplat}\index{Morphisme analytique!plat}
Soit $\varphi \colon \cX\to \cY$ un morphisme entre sch\'emas localement de pr\'esentation finie sur~$\cA$. Soit $x\in \cX^\an$. Si le morphisme $\varphi^\an$ est plat en~$x$, alors le morphisme $\varphi$ est plat en~$\rho_{\cX}(x)$. 
\end{coro}
\begin{proof}
Le r\'esultat d\'ecoule du th\'eor\`eme~\ref{platitude_analytification} et de \cite[\href{https://stacks.math.columbia.edu/tag/039V}{lemma 039V}]{stacks-project}.
\end{proof}

Il est tr\`es vraisemblable que la r\'eciproque du corollaire~\ref{cor:XYplat} soit vraie, mais nous ne savons pas la d\'emontrer en g\'en\'eral. Nous \'enon\c cons cependant maintenant quelques cas particuliers int\'eressants. 


\begin{prop}\label{stabilit\'e_analytification2}\index{Faisceau!plat}\index{Faisceau!sans torsion}
Soit $\varphi \colon \cX\to \cY$ un morphisme fini entre sch\'emas localement de pr\'esentation finie sur~$\cA$. Soit~$\cF$ un faisceau coh\'erent sur~$\cX$ et soit $y\in \cY^\an$. 
\begin{enumerate}[i)]
\item Le $\cO_{\cY,\rho_{\cY}(y)}$-module $(\varphi_*\cF)_{\rho_{\cY}(y)}$ est sans torsion si, et seulement si, le $\cO_{\cY^\an,y}$-module $(\varphi^\an_*\cF^\an)_{y}$ est sans torsion.
\item Soit $u\in \varphi^{-1}(\rho_{\cY}(y))$. Le $\cO_{\cY,\rho_{\cY}(y)}$-module $\cF_{u}$ est  plat si, et seulement si, pour tout $x\in (\varphi^\an)^{-1}(y) \cap \rho_{\cX}^{-1}(u)$, le $\cO_{\cY^\an,y}$-module $\cF^\an_{x}$ est plat.
\end{enumerate}
\end{prop}
\begin{proof}
Pour all\'eger l'\'ecriture, on notera indiff\'eremment~$\rho$ les deux morphismes~$\rho_{\cX}$ et~$\rho_{\cY}$.

i) D'apr\`es le th\'eor\`eme \ref{platitude_analytification}, $\cO_{\cY^{an},y}$ est plat sur $\cO_{\cY,\rho(y)}$. Puisque~$\varphi$ est fini, $(\varphi_*\cF)_{\rho(y)}$ est un $\cO_{\cY,\rho(y)}$-module de pr\'esentation finie. D'apr\`es \cite[I, \S 2, \no 10, proposition~11]{BourbakiAC14} et la proposition \ref{formule_analytification}, on a des isomorphismes canoniques 
\begin{align*}
&\Hom_{\cO_{\cY,\rho(y)}}((\varphi_*\cF)_{\rho_{\cY}(y)},(\varphi_*\cF)_{\rho(y)}) \otimes_{\cO_{\cY,\rho(y)}} \cO_{\cY^{an},y}\\
\simeq\ & 
\Hom_{\cO_{\cY^{an},y}}((\varphi_*\cF)_{\rho(y)}\otimes_{\cO_{\cY,\rho(y)}} \cO_{\cY^{an},y},(\varphi_*\cF)_{\rho(y)}\otimes_{\cO_{\cY,\rho(y)}} \cO_{\cY^{an},y})\\
\simeq \ &\Hom_{\cO_{\cY^{an},y}}((\varphi^{\an}_*\cF^{\an})_y,(\varphi^{\an}_*\cF^{\an})_y).
\end{align*}
Or, l'id\'eal de torsion de $(\varphi_*\cF)_{\rho(y)}$ s'identifie au noyau du morphisme
\[\cO_{\cY,\rho(y)} \too \Hom_{\cO_{\cY,\rho(y)}}((\varphi_*\cF)_{\rho(y)},(\varphi_*\cF)_{\rho(y)})\] 
et de m\^eme pour $(\varphi^{\an}_*\cF^{\an})_y$. Le r\'esultat d\'ecoule donc de la fid\`ele platitude de $\cO_{\cY^{an},y}$ sur $\cO_{\cY,\rho(y)}$. 

ii) Quitte \`a restreindre~$\cX$, on peut supposer que $\varphi^{-1}(\varphi(u)) = \{u\}$. Dans ce cas, on a $(\varphi_{*}\cF)_{\rho(y)} \simeq \cF_{u}$ et, d'apr\`es la proposition~\ref{prop:fetoileenbasexact}, 
\[ (\varphi^\an_{*}\cF^\an)_{y} \simeq \prod_{x \in (\varphi^\an)^{-1}(y)} \cF^\an_{x}.\]
En outre, d'apr\`es la proposition \ref{formule_analytification}, on a un isomorphisme canonique
\[\cO_{\cY^{an},y}\otimes_{\cO_{\cY,\rho(y)}} (\varphi_*\cF)_{\rho(y)} \simeq (\varphi^{\an}_*\cF^{\an})_y.\] 
Le r\'esultat d\'ecoule alors de la fid\`ele platitude de $\cO_{\cY^{an},y}$ sur $\cO_{\cY,\rho(y)}$ par \cite[I, \S 3, \no 3, proposition~6]{BourbakiAC14}.
\end{proof}

\begin{prop}\label{prop:platan}\index{Morphisme analytique!plat}\index{Morphisme!structural}
Soit~$\cX$ un sch\'ema localement de pr\'esentation finie sur~$\cA$. Notons $\pi \colon \cX \to \Spec(\cA)$ le morphisme structural. Soit $x\in \cX^\an$. Le morphisme $\pi$ est plat en~$\rho_{\cX}(x)$ si, et seulement si, le morphisme $\pi^\an$ est plat en~$x$.
\end{prop}
\begin{proof}
Posons $\rho := \rho_{\cX}$. Si $\cO_{\Spec(\cA),\pi(\rho(x))}$ est un corps, alors $\cO_{\cM(\cA),\pi^\an(x)}$ l'est aussi et le r\'esultat est imm\'ediat.

Supposons que $\cO_{\Spec(\cA),\pi(\rho(x))}$ est un anneau de valuation discr\`ete. Fixons-en une uniformisante~$\varpi$. Par hypoth\`ese, $\cO_{\cM(\cA),\pi^\an(x)}$ est encore un anneau de valuation discr\`ete d'uniformisante~$\varpi$. Il suffit de montrer que~$\varpi$ divise~0 dans~$\cO_{\cX,\rho(x)}$ si, et seulement si, $\varpi$ divise~0 dans~$\cO_{\cX^{\an},x}$, autrement dit, que la multiplication par~$\varpi$ est injective dans~$\cO_{\cX,\rho(x)}$ si, et seulement si, elle l'est dans~$\cO_{\cX^{\an},x}$. Ce r\'esultat d\'ecoule de la fid\`ele platitude du morphisme $\cO_{\cX,\rho(x)}\to\cO_{\cX^{\an},x}$ (\cf~th\'eor\`eme~\ref{platitude_analytification}).
\end{proof}

D\'emontrons finalement un r\'esultat sur la dimension. Commen\c cons par un lemme pr\'eliminaire.

\begin{lemm}\label{lem:finiouvertcomposante}\index{Morphisme analytique!fini}\index{Morphisme analytique!ouvert}\index{Composantes locales}
Soit $\varphi \colon X \to Y$ un morphisme fini et ouvert d'espaces $\cA$-analytiques. Soit $x\in X$ et supposons que $Y$~est int\`egre en~$\varphi(x)$. Notons $U_{1},\dotsc,U_{s}$ les composantes locales de~$X$ en~$x$. Alors il existe $i\in \cn{1}{s}$, un voisinage ouvert~$U'_{i}$ de~$x$ dans~$U_{i}$ et un voisinage ouvert~$V$ de~$\varphi(x)$ dans~$Y$ tels que le morphisme $U'_{i} \to V$ induit par~$\varphi$ soit fini et ouvert en~$x$.
\end{lemm}
\begin{proof}
Par d\'efinition des composantes locales, il existe un voisinage ouvert~$U$ de~$x$ dans~$X$ tel que les~$U_{i}$ soit des ferm\'es analytiques de~$U$ dont la r\'eunion recouvre~$U$. Pour $i\in \cn{1}{s}$, notons~$\cJ_{i}$ l'id\'eal de~$U_{i}$. D'apr\`es les lemmes~\ref{lem:Jintegre} et~\ref{lem:integreirreductible}, $\cJ_{i,x}$ est un id\'eal premier de~$\cO_{X,x}$. D'apr\`es le corollaire~\ref{cor:nilpotent}, on a $\sqrt{(0)} =\bigcap_{i=1}^s \cJ_{i,x}$.

Remarquons que le morphisme $\varphi^\sharp_{x} \colon \cO_{Y,\varphi(x)} \to \cO_{X,x}$ induit par~$\varphi^\sharp$ est injectif. Puisque~$\varphi$ ouvert et que $\cO_{Y,\varphi(x)}$ est int\`egre, cela d\'ecoule du corollaire~\ref{cor:nilpotent}. On en d\'eduit que 
\[ (0) = (\varphi^\sharp_{x})^{-1}\big(\sqrt{(0)}\big) = \bigcap_{i=1}^s (\varphi^\sharp_{x})^{-1}(\cJ_{i,x}). \]
D'apr\`es la partie unicit\'e du th\'eor\`eme~\ref{noeth}, il existe $i\in \cn{1}{s}$ tel que $(\varphi^\sharp_{x})^{-1}(\cJ_{i,x}) = (0)$.

Par construction, le morphisme $\cO_{Y,\varphi(x)} \to \cO_{U_{i},x}$ induit par~$\varphi^\sharp$ est encore injectif. Quitte \`a restreindre~$U_{i}$, on peut supposer qu'il existe un voisinage ouvert~$V$ de~$\varphi(x)$ dans~$Y$ tel que le morphisme $\psi \colon U_{i} \to V$ induit par~$\varphi$ soit fini. Quitte \`a restreindre encore, on peut supposer que $\psi^{-1}(\psi(x)) = \{x\}$.  On a alors un isomorphisme $(\psi_{\ast} \cO_{U_{i}})_{\psi(x)} \simeq \cO_{U_{i},x}$ et la proposition~\ref{prop:ouvert} assure que~$\psi$ est ouvert en~$x$. 
\end{proof}

\begin{theo}\label{thm:dimXXan}\index{Espace analytique!dimension d'un|(}
Supposons que $\dim(\cA) = 1$. Soient~$\cX$ un sch\'ema localement de pr\'esentation finie sur~$\cA$ et~$x$ un point de~$\cX^{\an}$. On a $\dim_{x}(\cX^\an) = \dim_{\rho_{\cX}(x)}(\cX)$.
\end{theo}
\begin{proof}
Il suffit de d\'emontrer le r\'esultat pour chaque composante irr\'eductible de~$\cX$. On peut donc supposer que~$\cX$ est irr\'eductible. Quitte \`a le r\'eduire et \`a localiser au voisinage de~$\rho_{\cX}(x)$, on peut supposer que~$\cX$ est int\`egre et affine. Il existe alors une $\cA$-alg\`ebre int\`egre~$\cA'$ telle que $\cX = \Spec(\cA')$. Notons $\pi \colon \cX \to \Spec(\cA)$ le morphisme structural. Distinguons deux cas.

\smallbreak

$\bullet$ Supposons que le morphisme $\pi^\sharp \colon \cA \to \cA'$ n'est pas injectif.

Notons~$\p$ le noyau de~$\pi^\sharp$. Par hypoth\`ese, le ferm\'e d\'efini par~$\p$ dans~$\Spec(\cA)$ et~$\cM(\cA)$ est un singleton. Puisque l'espace~$\cX$ et son analytifi\'e se projettent sur ce singleton, on a $v_{x}(\cX^\an) =0$ et on est ramen\'es \`a la th\'eorie classique de la dimension sur un corps valu\'e. Le r\'esultat est alors bien connu.

\smallbreak

$\bullet$ Supposons que le morphisme $\pi^\sharp \colon \cA \to \cA'$ est injectif.

Alors le morphisme $\pi^\sharp \colon \cA \to \cA'$ est plat, puisque~$\cA$ est un anneau de Dedekind. En particulier, on a 
\[ \dim(\cA') = \dim (\cA'\otimes_{\cA} K) + \dim(\cA) = \dim (\cA'\otimes_{\cA} K) +1,\]
o\`u $K := \Frac(\cA)$.

Posons $n := \dim (\cA'\otimes_{\cA} K)$. 
Le lemme de normalisation de Noether assure qu'il existe un morphisme injectif $K[T_{1},\dotsc,T_{n}] \to \cA' \otimes_{\cA} K$ qui fait de~$\cA'$ un $K[T_{1},\dotsc,T_{n}]$-module fini. On en d\'eduit qu'il existe $f\in \cA \setminus\{0\}$ et un morphisme injectif $\cA_{f}[T_{1},\dotsc,T_{n}] \to \cA' \otimes_{\cA} \cA_{f}$ qui fait de~$\cA'$ un $\cA_{f}[T_{1},\dotsc,T_{n}]$-module fini. Le morphisme associ\'e $\varphi \colon \Spec(\cA') \times_{\cA} \Spec(\cA_f) \to \A^n_{\cA} \times_{\cA} \Spec(\cA_f)$ est alors un morphisme fini sans torsion. Notons~$B_{f}$ l'ouvert de $B = \cM(\cA)$ d\'efini par l'in\'egalit\'e $f\ne 0$. D'apr\`es la proposition~\ref{stabilit\'e_analytification2}, le morphisme induit $\varphi^\an \colon \cX^\an_{f} := \cX^{\an}\times_{\cA} B_{f}\to \E{n}{\cA} \times_{\cA} B_{f}$ est \'egalement fini et sans torsion. D'apr\`es la proposition~\ref{prop:ouvert}, il est ouvert.

Notons $U_{1},\dotsc,U_{s}$ les composantes locales de~$\cX^\an_{f}$ en~$x$ et $U := \bigcup_{i=1}^s U_{i}$. Quitte \`a restreindre~$U$ et les~$U_{i}$, on peut supposer qu'il existe un ouvert~$V$ de $\E{n}{\cA} \times_{\cA} B_{f}$ tel que le morphisme $\psi \colon U \to V$ induit par~$\varphi^\an$ soit fini. La th\'eorie de la dimension sur un corps valu\'e assure que, pour tout $i\in \cn{1}{s}$, on a 
\[ \dim_{x}(U_{i}) = \dim_{x}(U_{i}/B) + v_{x}(U_{i}) \le n + 1.\]
En outre, d'apr\`es le lemme~\ref{lem:finiouvertcomposante}, il existe $j\in \cn{1}{s}$ tel que le morphisme $\psi_{j} \colon U_{j} \to V$ induit par~$\psi$ soit fini et ouvert en~$x$. On en d\'eduit que $v_{x}(U_{j}) = 1$ et que $\dim_{x}(U_{j}) = n$, donc que $\dim_{x}(U_{j}) = n+1$. Ceci montre finalement que 
\[ \dim_{x}(\cX^\an) = \dim_{x}(U) = n+1 = \dim(A) = \dim_{\rho_{\cX}(x)}(\cX).\]
\end{proof}

\begin{rema}
Si $\cA = k$ est un corps valu\'e, le r\'esultat $\dim_{x}(\cX^\an) = \dim_{\rho_{\cX}(x)}(\cX)$ vaut encore, par la th\'eorie classique. 

En revanche, il tombe en d\'efaut lorsque $\cA = \C^\hyb$ (\cf~exemple~\ref{ex:corpshybride}). Dans ce cas, $\cX$ est un sch\'ema de pr\'esentation finie sur~$\C$. Posons $B=\cM(\C^\hyb)$. Soit $x\in \cX^\an$. Notons $b$ l'image de~$x$ par  le morphisme structural $\pi^\an \colon \cX^\an \to B$. L'anneau local $\cO_{B,b}$ s'identifie au corps~$\C$, donc le morphisme $\cO_{B,b} \to \cO_{\cX^\an,x}$ et plat et, d'apr\`es la proposition~\ref{prop:morphismestructuralouvert}, le morphisme $\pi^\an$ est ouvert en~$x$. On en d\'eduit que~$\cX^\an$ n'est pas vertical en~$x$, et donc que $v_{x}(\cX^\an) = 1$. Par ailleurs, d'apr\`es le lemme~\ref{lem:imageanalytification}, la fibre $(\pi^\an)^{-1}(b)$ s'identifie \`a l'analytifi\'e de~$\cX$ sur le corps valu\'e~$\cH(b)$ (qui n'est autre que le corps~$\C$ muni d'une valeur absolue de la forme $\va_{\infty}^\eps$ ou~$\va_{0}$), et la th\'eorie classique de la dimension assure que $\dim_{x}\big((\pi^\an)^{-1}(b)\big) = \dim_{\rho_{\cX}(x)}(\cX)$. On obtient finalement l'\'egalit\'e 
\[\dim_{x}(\cX^\an) = \dim_{\rho_{\cX}(x)}(\cX) +1.\]

\end{rema}
\index{Espace analytique!dimension d'un|)}

\index{Analytification!|)}

\chapter{Propri\'et\'es topologiques des espaces analytiques}\label{chap:topo}

Dans ce chapitre, nous \'etudions certaines propri\'et\'es topologiques des espaces analytiques. 

La section~\ref{sec:elastique} traite de topologie g\'en\'erale. Nous y introduisons la notion d'espace \'elastique, qui renforce celle d'espace connexe par arcs (\cf~d\'efinition~\ref{def:elastique}).

Dans la section~\ref{sec:cpa}, nous d\'emontrons que les espaces analytiques sont localement connexes par arcs (\cf~th\'eor\`eme~\ref{th:cpageneral}). Comme on s'y attend, nous \'etudions d'abord les espaces affines analytiques, cas dans lequel nous pouvons exhiber des sections explicites de la projection (\cf~lemme~\ref{lem:section}). Le cas g\'en\'eral s'y r\'eduit, par normalisation de Noether.

Pour finir, dans la section~\ref{sec:dimtop}, nous calculons la dimension topologique, plus pr\'ecis\'ement la dimension de recouvrement, des espaces affines analytiques, ainsi que des disques (\cf~th\'eor\`eme~\ref{th:dimensionetoile}). Le calcul de la dimension des fibres est classique, mais, les espaces que nous consid\'erons n'\'etant, en g\'en\'eral, pas m\'etrisables, le passage \`a l'espace total requiert quelques pr\'ecautions.

\medbreak

Soit~$(\cA,\nm)$ un anneau de base g\'eom\'etrique. Posons $B:=\cM(\cA)$.

\section{Espaces \'elastiques}\label{sec:elastique}
\index{Espace topologique!elastique@\'elastique|(}

Dans cette section, nous regroupons quelques r\'esultats de topologie g\'en\'erale qui vont nous \^etre utiles pour montrer que les espaces affines, puis les espaces analytiques quelconques, sont localement connexes par arcs. Nous introduisons notamment la notion d'espace \'elastique.

\medbreak

Commen\c cons par montrer un r\'esultat de mod\'eration du nombre de composantes connexes.

\begin{prop}\label{prop:top}\index{Espace topologique!connexe}\index{Application!finie et ouverte}
Soit $f \colon X \to Y$ une application finie et ouverte entre espaces topologiques. Supposons que~$Y$ est connexe. Alors $X$ poss\`ede un nombre fini de composantes connexes et la restriction de~$f$ \`a chacune d'elle est surjective.
\end{prop}
\begin{proof}
Soit~$x\in X$. Notons~$\cO\cF_x$ l'ensemble des parties ouvertes et ferm\'ees de~$X$ qui contiennent~$x$ et posons $OF_x:= \bigcap_{U\in\cO\cF_x} U$. Montrons que 
\[\bigcap_{U\in\cO\cF_x} f(U)=f(OF_x).\]

L'inclusion $f(OF_x)  \subset \bigcap_{U\in\cO\cF_x} f(U)$ est \'evidente. D\'emontrons l'inclusion r\'eciproque. Soit $y \in \bigcap_{U\in\cO\cF_x} f(U)$. Notons $x_{1},\dotsc,x_{n}$ les \'el\'ements de~$f^{-1}(y)$. Supposons, par l'absurde, que pour tout~$i\in\cn{1}{n}$, il existe~$U_i\in \cO\cF_x$ tel que~$x_i\notin U_i$. On a alors $(\bigcap_{i=1}^n U_i)\cap f^{-1}(y)= \emptyset$. Or, toute intersection finie de parties ouvertes et ferm\'ees est encore ouverte et ferm\'ee, donc $\bigcap_{i=1}^nU_{i} \in \cO\cF_x$, ce qui est absurde car~$y$ appartient \`a~$\bigcap_{U\in\cO\cF_x} f(U)$. Par cons\'equent, il existe~$i\in\cn{1}{n}$ tel que~$x_{i}$ appartienne \`a tous les \'el\'ements de~$\cO\cF_{x}$ et donc \`a~$OF_{x}$. On en d\'eduit que~$y$ appartient \`a~$f\left(OF_x\right)$.

Pour tout~$U\in\cO\cF_x$, $f(U)$ est ouvert, ferm\'e et non vide, donc \'egal \`a~$Y$, puisque~$Y$ est connexe. Ainsi, par ce qui pr\'ec\`ede, on en d\'eduit que $f(OF_x)=Y$.

\medbreak

Montrons maintenant que les ensembles~$OF_x$ forment une partition de~$X$. Soient $x,x' \in X$. 

Supposons que $x' \notin OF_{x}$. Alors il existe un ouvert ferm\'e~$U$ contenant~$x$ tel que~$x'$ n'appartienne pas \`a~$U$. Dans ce cas~$X\setminus U$ est un ouvert ferm\'e qui contient~$x'$ et ne contient pas~$x$. On en d\'eduit que $OF_x\cap OF_{x'} = \emptyset$.

Supposons que $x' \in OF_{x}$. Alors tout ouvert ferm\'e qui contient~$x$ contient~$x'$ donc $OF_{x'} \subset OF_x$. D'apr\`es ce qui pr\'ec\`ede, cela implique que $x \in OF_{x'}$, et finalement que $OF_x = OF_{x'}$. La propri\'et\'e annonc\'ee s'ensuit.

\medbreak

Soit~$y\in Y$. Pour tout $x\in X$, on a $f(OF_x)=Y$, donc il existe $z\in f^{-1}(y)$ tel que $z \in OF_{x}$, ce qui entra\^ine $OF_{z} = OF_{x}$. On en d\'eduit l'\'egalit\'e
\[ X = \bigsqcup_{z\in f^{-1}(y)} OF_{z}.\]

Soit $z\in f^{-1}(y)$. L'ensemble~$OF_{z}$ est intersection de ferm\'es, donc ferm\'e. En \'ecrivant~$OF_{z}$ comme le compl\'ementaire de la r\'eunion (finie) des~$OF_{z'}$, pour $z' \in f^{-1}(y) \setminus \{z\}$, on montre que~$OF_{z}$ est ouvert. 

Montrons que~$OF_{z}$ est connexe, ce qui permettra de conclure. Soient~$A$ et~$B$ des ouverts disjoints de~$OF_{z}$ de r\'eunion \'egale \`a~$OF_{z}$. On peut supposer que $z\in A$. Dans ce cas, on a $A \in \cO\cF_{z}$, donc $OF_{z} \subset A$, et finalement $OF_{z} = A$. Le r\'esultat s'ensuit.
\end{proof}

Nous souhaiterions disposer d'un r\'esultat analogue au pr\'ec\'edent mais concernant cette fois-ci la connexit\'e par arcs. Pour l'obtenir, nous aurons besoin d'imposer une condition plus forte sur l'espace d'arriv\'ee. Nous l'introduisons ici. 

\begin{defi}\label{def:elastique}\index{Espace topologique!elastique@\'elastique|textbf}\index{Espace topologique!connexe par arcs|textbf}\index{Chemin!elastique@\'elastique|textbf}\index{Chemin|textbf}
Soit~$X$ un espace topologique. 

Soient $x,y \in X$. On dit qu'une partie~$\ell$ de~$X$ est un \emph{chemin reliant $x$ \`a $y$} s'il existe une application continue surjective $\varphi \colon [0,1] \to \ell$ telle que $\varphi(0)=x$ et $\varphi(1)=y$.

Soit~$\ell$ un chemin reliant~$x$ \`a~$y$. Le chemin~$\ell$ est dit \emph{\'elastique en~$y$} si, pour tout voisinage~$U$ de~$\ell$, il existe un voisinage~$V$ de~$y$ tel que, pour tout~$y'\in V$, il existe un chemin de~$x$ \`a~$y'$ contenu dans~$U$. Le chemin~$\ell$ est dit \emph{\'elastique} s'il est \'elastique en~$x$ et en~$y$.


L'espace~$X$ est dit \emph{connexe par arcs} (resp. \emph{\'elastique}) si, pour tous $x,y \in X$, il existe un chemin (resp. chemin \'elastique) reliant~$x$ \`a~$y$.
\end{defi}

\begin{rema}\label{rem:definitionelastique}
\begin{enumerate}[i)]
\item Pour montrer qu'un espace~$X$ est \'elastique, il suffit de montrer qu'il est connexe par arcs et que, pour tout point~$x$ de~$X$, il existe un chemin d'origine~$x$ qui est \'elastique en~$x$. La preuve consiste en un simple argument bas\'e sur la concat\'enation des chemins.
\item Tout espace connexe et localement connexe par arcs est \'elastique. De m\^eme, tout espace contractile est \'elastique. La notion d'espace \'elastique est cependant plus g\'en\'erale. En effet, il existe des espaces connexes et localement connexes par arcs mais non contractiles (par exemple le cercle) ainsi que des espaces contractiles non localement connexes par arcs  (par exemple la r\'eunion des droites du plan r\'eel de pente rationnelle passant par~0).
\end{enumerate}
\end{rema}

\begin{exem}\label{exemple_non_\'elastique}
L'exemple qui suit montre que les notions d'espaces \'elastique et connexe par arcs sont distinctes. 

Notons~$G_{\sin(1/x)}$ le sous-ensemble de~$\R^2$ constitu\'e du graphe de~$\sin(1/x)$ sur~$\intof{0,1/\pi}$. Consid\'erons l'espace topologique~$X$ form\'e de la r\'eunion de l'adh\'erence $\overline{G_{\sin(1/x)}}$ de~$G_{\sin(1/x)}$ dans~$\R^2$ et d'un chemin~$\ell_{0}$ reliant le point $y_{0}$ de coordonn\'ees $(0,-1)$ au point~$x_{0}$ de coordonn\'ees $(1/\pi,0)$, de sorte que $\ell \cap \overline{G_{\sin(1/x)}} =\{x_{0},y_{0}\}$ (\cf~figure~\ref{fig:elastiquevscpa}). Cet espace est connexe par arcs. 

\begin{figure}[!h]
\centering
\labellist
\pinlabel~$\bullet$ at 3.5 27
\pinlabel~$y_{0}$ at -6 27
\pinlabel~$\bullet$ at 71 47
\pinlabel~$x_0$ at 79 47
\pinlabel~$\ell_0$ at 40 0
\endlabellist
\includegraphics[scale=1.5]{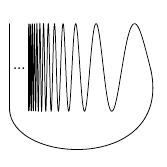}
\caption{Un espace connexe par arcs non \'elastique.}
\label{fig:elastiquevscpa}
\end{figure}

Soit~$\ell$ un chemin reliant~$x_{0}$ \`a~$y_{0}$. Tout voisinage de~$y_{0}$ contient un point de $G_{\sin(1/x)}$ avec une abscisse strictement positive tr\`es petite. Un tel point ne peut \^etre reli\'e \`a~$x_{0}$ dans aucun voisinage assez petit du chemin~$\ell$ (\cf~figure~\ref{fig:voisinagechemin}). 

 \pgfdeclareimage[interpolate=true,height=4cm]{nonelastique2}{nonelastique2}
 \begin{figure}[!h]
\centering
\labellist
\pinlabel~$\bullet$ at 1 21
\pinlabel~$y_{0}$ at -6 21
\pinlabel~$\bullet$ at 69 39.5
\pinlabel~$x_0$ at 77 40
\endlabellist
\includegraphics[scale=1.5]{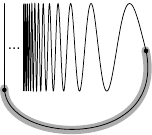}
\caption{Un voisinage d'un chemin de~$x_{0}$ \`a~$y_{0}$.}
\label{fig:voisinagechemin}
\end{figure}

Par cons\'equent, le chemin~$\ell$ n'est pas \'elastique en~$y_{0}$ et l'espace topologique~$X$ n'est donc pas \'elastique.

Ajoutons que des exemples de ce type peuvent appara\^itre dans l'\'etude des espaces de Berkovich g\'en\'eraux. En effet, tout espace compact~$K$ peut \^etre r\'ealis\'e comme spectre d'anneau de Banach~$\cA$, par exemple en choississant pour~$\cA$ l'alg\`ebre $\cC(K,\C)$ des fonctions continues de~$K$ dans~$\C$ munie de la norme uniforme. (Le corps~$\C$ est ici muni de la valeur absolue usuelle~$\va_{\infty}$.) Il est donc n\'ecessaire d'imposer des restrictions sur les anneaux~$\cA$ consid\'er\'es.
\end{exem}

La proposition suivante reprend \cite[lemma~3.2.5]{Ber1}. Nous en rappelons la d\'emonstration pour apporter quelques pr\'ecisions.

\begin{prop}\label{chemin}\index{Espace topologique!connexe par arcs}\index{Application!finie et ouverte}
Soit~$f \colon X \to Y$ une application finie et ouverte entre espaces topologiques. Supposons que~$X$ et~$Y$ sont connexes, que~$X$ est s\'epar\'e et que~$Y$ est \'elastique. Alors $X$~est connexe par arcs.
\end{prop}
\begin{proof}
Soit~$\ell$ un chemin de~$Y$. Soit~$\cU$ une base d'ouverts de~$\ell$ qui engendre sa topologie. Nous pouvons supposer que~$\cU$ est d\'enombrable. Pour tout $U \in \cU$, l'application $f^{-1}(U) \to U$ induite par~$f$ est encore finie et ouverte. La proposition~\ref{prop:top} assure alors que~$f^{-1}(U)$ a un nombre fini de composantes connexes. Celles-ci sont donc ouvertes. Notons~$\cC_{U}$ l'ensemble des composantes connexes de~$f^{-1}(U)$. L'ensemble $\cC := \bigcup_{U\in \cU} \cC_{U}$ est alors d\'enombrable et, d'apr\`es le lemme~\ref{lem:voisinagefibre}, il engendre la topologie de~$f^{-1}(\ell)$. 

D'apr\`es le lemme~\ref{lem:criterepropre}, $f^{-1}(\ell)$ est compact. D'apr\`es \cite[IX, \S 2, \no 9, proposition~16]{BourbakiTG510}
il est donc m\'etrisable. Puisque les \'el\'ements de~$\cC$ sont connexes, $f^{-1}(\ell)$ est localement connexe. D'apr\`es \cite[III, \S 2, \no 4, corollaire~2 de la proposition~11]{BourbakiTA14},
$f^{-1}(\ell)$ est donc localement connexe par arcs. En particulier, les composantes connexes de~$f^{-1}(\ell)$ sont compactes et connexes par arcs.

\medbreak

Pour d\'emontrer le r\'esultat, il suffit de montrer que les composantes connexes par arcs de~$X$ sont ouvertes. Soit $x\in X$. Posons $y := f(x)$. Soit $y_{0} \in Y$. Soit~$\ell$ un chemin reliant~$y_{0}$ \`a~$y$ qui soit \'elastique en~$y$. D'apr\`es le raisonnement pr\'ec\'edent, $f^{-1}(\ell)$ poss\`ede un nombre fini de composantes connexes, qui sont compactes et connexes par arcs. Notons-les $\Sigma_{1},\dotsc,\Sigma_{n}$. On peut supposer que $x \in \Sigma_{1}$.

Puisque~$X$ est s\'epar\'e, on peut trouver des voisinages $W_{1},\dotsc,W_{n}$ respectivement de $\Sigma_{1},\dotsc,\Sigma_{n}$ qui soient deux \`a deux disjoints. D'apr\`es le lemme~\ref{lem:voisinagefibre}, il existe un voisinage~$U$ de~$\ell$ tel que $f^{-1}(U) \subset \bigsqcup_{i=1}^n W_{i}$.

Soit~$V$ un voisinage ouvert de~$y$ dans~$U$ tel que, pour tout $y' \in V$, il existe un chemin reliant~$y_{0}$ \`a~$y$ contenu dans~$U$. Alors $f^{-1}(V) \cap W_{1}$ est inclus dans la composante connexe par arcs de~$x$. En effet, soit $x' \in f^{-1}(V) \cap W_{1}$. Il existe un chemin~$\ell'$ reliant~$y_{0}$ \`a~$y':=f(x')$ contenu dans~$U$. Notons~$\Sigma'$ la composante connexe de~$f^{-1}(\ell')$ contenant~$x'$. Elle est connexe par arcs et contenue dans $f^{-1}(U)$, donc dans $\bigsqcup_{i=1}^n W_{i}$, donc dans~$W_{1}$ puisqu'elle rencontre~$W_{1}$. On a donc
\[ \Sigma' \cap f^{-1}(y_{0}) \subset W_{1} \cap f^{-1}(y_{0}) = \Sigma_{1} \cap  f^{-1}(y_{0}) .\]
On en d\'eduit que $\Sigma' \cup \Sigma_{1}$ est connexe par arcs. Il existe donc un chemin reliant~$x$ \`a~$x'$. Le r\'esultat s'en d\'eduit.
\end{proof}

\begin{rema}\label{rem:revetementdegre2}
Il existe un espace topologique connexe par arcs (mais non \'elastique) admettant un rev\^etement de degr\'e~$2$ connexe mais non connexe par arcs. On peut, par exemple, consid\'erer le rev\^etement de l'espace d\'efini \`a l'exemple \ref{exemple_non_\'elastique} repr\'esent\'e \`a la figure~\ref{fig:revetementcpa}.

\begin{figure}[!h]
\centering
\includegraphics[scale=2.2]{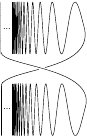}
\caption{Rev\^etement de degr\'e 2 connexe mais non connexe par arcs d'un espace connexe par arcs non \'elastique.}
\label{fig:revetementcpa}
\end{figure}

Cela permet de justifier l'introduction de la notion d'espace \'elastique.
\end{rema}

Dans le cas o\`u l'espace au but est seulement suppos\'e connexe par arcs, on dispose cependant du r\'esultat suivant.

\begin{prop}\label{prop:revetementcpa}\index{Espace topologique!connexe par arcs}\index{Application!finie et ouverte}
Soit~$f \colon X \to Y$ une application finie et ouverte entre espaces topologiques. Supposons que~$X$ et~$Y$ sont s\'epar\'es et que~$Y$ est connexe par arcs. Alors $X$ poss\`ede un nombre fini de composantes connexes par arcs et la restriction de~$f$ \`a chacune d'elles est surjective.
\end{prop}
\begin{proof}
Soient~$y_{1},y_{2} \in Y$. Par hypoth\`ese, il existe un chemin~$\ell$ reliant~$y_{1}$ et~$y_{2}$. Puisque~$Y$ est s\'epar\'e, d'apr\`es~\cite[III, \S 2, \no 9, proposition~18]{BourbakiTA14},
on peut supposer que~$\ell$ est hom\'eomorphe au segment~$[0,1]$. D'apr\`es la remarque~\ref{rem:definitionelastique} ii), l'espace topologique~$\ell$ est alors \'elastique. Le morphisme $\varphi^{-1}(\ell) \to \ell$ induit par~$\varphi$ est encore fini et ouvert. D'apr\`es la proposition~\ref{prop:top} et la proposition~\ref{chemin}, $\varphi^{-1}(\ell)$ a un nombre fini de composantes connexes et celles-ci se surjectent sur~$\ell$ par~$\varphi$ et sont connexes par arcs.

On en d\'eduit imm\'ediatement que les composantes connexes par arcs de~$X$ se surjectent sur~$Y$ par~$f$. Puisque les fibres de~$f$ sont finies, il ne peut y en avoir qu'un nombre fini.
\end{proof}

\begin{rema}
L'exemple de la remarque~\ref{rem:revetementdegre2} montre que les composantes connexes par arcs de~$X$ peuvent n'\^etre ni ferm\'ees ni ouvertes.
\end{rema}

Nous allons maintenant \'enoncer deux crit\`eres permettant de montrer l'\'elasticit\'e de certains espaces topologiques.

\begin{lemm}\label{lem:crit_elast}
Soient~$X$ un espace topologique et~$W$ un ouvert de~$X$. Supposons que
\begin{enumerate}[i)]
\item $W$ et~$F:=X\setminus W$ sont \'elastiques ;
\item l'ensemble $\partial F:=F\setminus \mathring F$ n'est pas vide ;
\item tout point~$x$ de~$\partial F$ admet une base de voisinages~$\cU_x$ dans~$X$ telle que pour tout $U\in\cU_x$ et tout $y\in W\cap U$, il existe un chemin trac\'e sur~$U$ reliant~$y$ \`a un point de~$F$.
\end{enumerate}
Alors l'espace topologique~$X$ est \'elastique.
\end{lemm}
\begin{proof}
Soit~$x \in X$. D'apr\`es la remarque~\ref{rem:definitionelastique}, il suffit de montrer qu'il existe un chemin d'origine~$x$ qui est \'elastique en~$x$. Si $x\in W$, la propri\'et\'e d\'ecoule du fait que~$W$ est ouvert dans~$X$ et \'elastique. Il en va de m\^eme si $x\in \mathring F$. On peut donc supposer que $x\in \partial F$. 

Choisissons un chemin~$\ell$ de~$F$ d'origine~$x$ et \'elastique en~$x$ en tant que chemin de~$F$. Notons $x_{0}\in F$ son autre extr\'emit\'e. 

Soit~$V$ un voisinage de~$\ell$ dans~$X$. Le point~$x$ poss\`ede un voisinage~$U'$ dans~$V$ tel que, pour tout~$x'\in U'\cap F$, il existe un chemin reliant~$x'$ \`a~$x_{0}$ dans~$V$. Soient~$U$ un voisinage de~$x$ dans~$U'$ appartenant \`a~$\cU_{x}$. Soit~$y \in U$. Si $y\in F$, on a d\'ej\`a vu qu'il existe un chemin reliant~$y$ \`a~$x_0$ dans~$V$. Si~$y\in W$, alors il existe un chemin dans~$U$ reliant~$y$ \`a un point~$x'$ de~$U\cap F$. Il existe alors un chemin reliant~$x'$ \`a~$x_0$ dans~$V$, et donc un chemin reliant~$y$ \`a~$x_0$ dans~$V$. Ceci montre que~$\ell$ est \'elastique en~$x$ en tant que chemin de~$X$. 
\end{proof}

\begin{lemm}\label{lem:crit_elast'}
Soient~$X$ un espace topologique et~$(X_i)_{i\in\N}$ une suite croissante (pour l'inclusion) de sous-espaces topologiques tels que~$\bigcup_{i\in\N}\mathring X_i=X$. Si, pour tout~$i\in\N$, l'espace~$X_i$ est \'elastique, alors l'espace~$X$ est lui aussi \'elastique.
\qed
\end{lemm}
\index{Espace topologique!elastique@\'elastique|)}

\section[Connexit\'e par arcs locale]{Connexit\'e par arcs locale}\label{sec:cpa}

Dans cette section, nous nous int\'eressons \`a la connexit\'e par arcs et la connexit\'e par arcs locale des espaces analytiques. Nous commencerons par le cas des espaces affines.

\medbreak

Soit $n\in \N$. Notons $\pi \colon \E{n}{\cA} \to B = \cM(\cA)$ le morphisme de projection. Fixons des coordonn\'ees $T_{1},\dotsc,T_{n}$ sur $\E{n}{\cA}$. Posons $\bT := (T_{1},\dotsc,T_{n})$.%
\nomenclature[Ita]{$\bT$}{$= (T_{1},\dotsc,T_{n})$, $n$-uplet de coordonn\'ees}

\begin{nota}%
\nomenclature[Itb]{$\br$}{$= (r_{1},\dotsc,r_{n}) \in \R_{\ge 0}^n$}%
\nomenclature[Itc]{$\bi$}{$= (i_{1},\dotsc,i_{n}) \in \N^n$}%
\nomenclature[Itd]{$\br^\bi$}{$= \prod_{m=1}^n r_{m}^{i_{m}}$}
Soit $A$ un anneau. Pour tout $\ba = (a_{1},\dotsc,a_{n}) \in A^n$ et tout $\bi = (i_{1},\dotsc,i_{n}) \in \N^n$, on pose
\[ \ba^\bi := \prod_{m=1}^n a_{m}^{i_{m}},\]
avec la convention que $a^0 = 1$, pour tout $a\in A$.
\end{nota}

\begin{nota}%
\nomenclature[Iua]{$\eta_{\br}$}{sur un corps valu\'e ultram\'etrique complet, unique point du bord de Shilov du polydisque centr\'e en~0 de polyrayon $\br$}
Soit $(k,\va)$ un corps valu\'e ultram\'etrique complet. Notons $\bT := (T_{1},\dotsc,T_{n})$ des coordonn\'ees sur~$\E{n}{k}$. Soit $\br = (r_1,\dotsc,r_n) \in \R_{\ge0}^n$. On note~$\eta_{\br}$ le point de~$\E{n}{k}$ associ\'e \`a la semi-norme multiplicative
\[{\renewcommand{\arraystretch}{1.3}\begin{array}{ccc}
k[\bT] & \too & \R_{\ge 0}\\
\disp \sum_{\bi \ge 0} a_{\bi}\, \bT^{\bi} & \mapstoo & \max_{\bi \ge 0} ( |a_{\bi}| \, \br^{\bi} ).
\end{array}}\]

Soit $b\in B_{\um}$. On note $\eta_{b,\br} \in \E{n}{\cA}$ le point~$\eta_{\br}$ de la fibre $\pi^{-1}(b) \simeq \E{n}{\cH(b)}$
\end{nota}

Rappelons que nous avons introduit une fonction $\eps \colon B_{\arc} \to \intof{0,1}$ (\cf~notation~\ref{not:epsilon}).

\begin{nota}
\nomenclature[Isz]{$\rho_{b}$}{pour $b$ point archim\'edien de~$\cM(\cA)$, plongement canonique $\R^n \to \pi_{n}^{-1}(b)$}
Soit $b\in B_{\arc}$. D'apr\`es le th\'eor\`eme~\ref{th:vaarchimedienne} et le lemme~\ref{lem:AnC}, la fibre $\pi^{-1}(b)$ s'identifie \`a~$\C^n$ si $\cH(b) = \C$ et \`a $\C^n/\Gal(\C/\R)$ si $\cH(b) = \R$ et, dans tous les cas, on a un plongement canonique 
\[ \rho_{b} \colon \R^n \too \pi^{-1}(b).\] 
\end{nota}

\begin{lemm}\label{lem:section}\index{Projection!section d'une}%
\nomenclature[Iuc]{$\sigma_{\br}$}{une section de la projection~$\pi_{n}$}
Soit $\br = (r_1,\dotsc,r_n) \in \R_{\ge0}^n$. L'application
\[\fonction{\sigma_{\br}}{B}{\E{n}{\cA}}{b}{
\begin{cases}
\eta_{b,\br} &\textrm{si } b\in B_{\um}~;\\
\rho_{b}(\br^{1/\eps(b)}) &\textrm{si } b\in B_{\arc}.
\end{cases}}\]
est une section de la projection $\pi \colon \E{n}{\cA} \to B$. Sa restriction \`a~$B_{\um}$ et sa restriction \`a~$B_{\arc}$ sont continues.

Notons $S := \{ i \in \cn{1}{n} : r_{i} \ne 0\}$. Soit~$b_{0} \in B$ tel que la famille $(r_{i})_{i\in S}$ soit libre dans le $\Q$-espace vectoriel $\R_{>0}/|\cH(b_{0})^\times|^\Q$. Alors $\sigma_{\br}$ est continue en~$b_{0}$.
\end{lemm}
\begin{proof}
Il d\'ecoule directement de la d\'efinition de~$\sigma_{\br}$ que c'est une section de~$\pi$. 

Rappelons que, par d\'efinition de la topologie de~$\E{n}{\cA}$, pour montrer que la restriction de~$\sigma_{\br}$ \`a une partie~$U$ de~$B$ est continue, il suffit de montrer que, pour tout $P \in \cA[\bT]$, l'application
\[ b\in U \mapsto |P(\sigma_{\br}(b))| \in \R_{\ge 0}\]
est continue. Soit $P = \sum_{\bi \ge 0} a_{\bi}\, \bT^{\bi}\in \cA[\bT]$.

Commen\c cons par traiter le cas de la restriction \`a~$B_{\um}$. Pour tout $b\in B_{\um}$, on a
\[ |P(\sigma_{\br}(b))| = \max_{\bi \ge0} (|a_{\bi}(b)| \, \br^\bi).\]
Le r\'esultat d\'ecoule donc de la continuit\'e des fonctions~$|a_{\bi}|$ (qui sont en nombre fini).

\'Etudions maintenant la restriction \`a~$B_{\arc}$. Rappelons que~$B_{\arc}$ est un ouvert de~$B$, d'apr\`es la remarque~\ref{rem:umfermee}. Soit~$b \in B_{\arc}$. Pour tout $b'\in B$, on a 
\begin{align*} 
& \big| |P(\sigma_{\br}(b'))| -|P(\sigma_{\br}(b))| \big|\\
=\ & \big| |P(\rho_{b'}(\br^{1/\eps(b')}))| - |P(\rho_{b}(\br^{1/\eps(b)})| \big|\\
\le\ &  \big| |P(\rho_{b'}(\br^{1/\eps(b')}))| - |P(\rho_{b'}(\br^{1/\eps(b)})| \big| +  \big| |P(\rho_{b'}(\br^{1/\eps(b)}))| - |P(\rho_{b}(\br^{1/\eps(b)})| \big|.
\end{align*}
\'Etudions d'abord le second terme de la somme. Rappelons que, d'apr\`es la remarque~\ref{rem:gammaR}, on a une injection canonique $\gamma_{\R} \colon \R \to \cO(B_{\arc})$. Pour tout $b'\in B$, on a
\[\big| |P(\rho_{b'}(\br^{1/\eps(b')}))| - |P(\rho_{b'}(\br^{1/\eps(b)})| \big| 
= \big| |P(\gamma_{\R}(\br^{1/\eps(b)}))(b')| - |P(\gamma_{\R}(\br^{1/\eps(b)}))(b)|  \big|\]
et cette quantit\'e tend donc vers~0 quand~$b'$ tend vers~$b$, d'apr\`es le lemme~\ref{lem:evaluationcontinueaffine}.

\'Etudions maintenant le premier terme de la somme. Il existe un voisinage~$V$ de~$b$ dans~$B_{\arc}$ et une constante $M\in \R$ tels que
\[ \forall \bi \ge 0, \forall b' \in V,\ |a_{\bi}(b')| \le M.\]
Soit $\alpha \in \R_{>0}$. D'apr\`es le lemme~\ref{lem:epscontinue}, la fonction~$\eps$ est continue. Quitte \`a restreindre~$V$, on peut donc supposer que
\[ \forall b' \in V,\ |\br^{1/\eps(b')} - \br^{1/\eps(b)}| \le \frac\alpha M.\]
On a alors
\begin{align*}
 \big| |P(\rho_{b'}(\br^{1/\eps(b')}))| - |P(\rho_{b'}(\br^{1/\eps(b)})| \big|
&\le \big| \sum_{\bi \ge 0} a_{\bi}(b') (\br^{1/\eps(b')} - \br^{1/\eps(b)})\big|\\
& \le d M \frac\alpha M \le d\alpha,
\end{align*}
o\`u~$d$ est le nombre de coefficients non nuls de~$P$.

On en d\'eduit que $|P(\sigma_{\br}(\wc))|$ est continue en~$b$, puis que~$\sigma_{\br}$ est continue sur~$B_{\arc}$.

\medbreak

D\'emontrons maintenant la derni\`ere partie du r\'esultat. L'hypoth\`ese de libert\'e entra\^ine que l'on a
\[ \{ x \in \pi^{-1}(b_{0}) : \forall i \in \cn{1}{n},\ |T_{i}(x)| = r_{i} \} = \{\sigma_{\br}(b_{0})\}.\]
Soit~$U$ un voisinage de~$\sigma_{\br}(b_{0})$. D'apr\`es l'\'egalit\'e pr\'ec\'edente (\cf~\cite[lemme~2.4.1]{A1Z} pour les d\'etails), il existe un voisinage~$V$ de~$b_{0}$ dans~$B$ et $s_{1},t_{1},\dotsc,s_{n},t_{n} \in \R_{\ge0}$ avec, pour tout $i\in \cn{1}{n}$, $s_{i} \prec r_{i} < t_{i}$, tels que 
\[ \{ x \in \pi^{-1}(V) : \forall i \in \cn{1}{n},\ s_{i} \prec |T_{i}(x)| < t_{i} \} \subset U.\]
Or, par d\'efinition, pour tout $b\in B$ et tout $i\in \cn{1}{n}$, on a $|T_{i}(\sigma_{\br}(b))| = r_{i}$. 
Par cons\'equent, $\sigma_{\br}^{-1}(U)$ contient~$V$. On en d\'eduit que~$\sigma_{\br}$ est continue en~$b_{0}$.
\end{proof}

Introduisons une d\'efinition qui nous permettra de tirer profit du lemme~\ref{lem:section}.

\begin{defi}\label{def:peumixte}\index{Chemin!pur|textbf}\index{Partie!peu mixte|textbf}
Soit $\varphi \colon [0,1] \to B$ une application continue. Le chemin correspondant \`a~$\varphi$ est dit \emph{pur} si $\varphi(\intoo{0,1})$ est enti\`erement contenu soit dans~$U_{\um}$ soit dans~$U_{\arc}$.

Une partie~$U$ de $B=\cM(\cA)$ est dite \emph{peu mixte} si
\begin{enumerate}[i)]
\item pour pour $b \in \partial U_{\um} = \partial U_{\arc}$, le $\Q$-espace vectoriel  $|\cH(b)^\times|^\Q$ est de codimension infinie dans~$\R_{>0}$~;
\item pour tous $x,y$ dans la m\^eme composante connexe par arcs de~$U$, il existe un chemin reliant~$x$ \`a~$y$ qui est concat\'enation d'un nombre fini de chemins purs.
\end{enumerate}
\end{defi}

\begin{rema}\label{rem:peumixte}
Le point~ii) de la d\'efinition de partie peu mixte est satisfait d\`es que $\partial U_{\um} = \partial U_{\arc}$ est discret.
\end{rema}

\begin{exem}\label{ex:peumixte}\index{Corps!valu\'e}\index{Anneau!des entiers relatifs $\Z$}\index{Anneau!des entiers d'un corps de nombres}\index{Corps!hybride}\index{Anneau!de valuation discr\`ete}\index{Anneau!de Dedekind trivialement valu\'e}
Si l'anneau de Banach~$\cA$ est l'un de nos exemples habituels (corps valu\'e, $\Z$ ou anneau d'entiers de corps de nombres, corps hybride, anneau de valuation discr\`ete, anneau de Dedekind trivialement valu\'e, \cf~exemples~\ref{ex:corpsvalue} \`a~\ref{ex:Dedekind}), alors toute partie de~$\cM(\cA)$ est peu mixte.
\end{exem}

Passons maintenant aux r\'esultats de connexit\'e. Nous allons, dans un premier temps, montrer que les polycouronnes relatives dont la base est connexe par arcs sont connexes par arcs. Nous utiliserons ensuite ce r\'esultat pour montrer que ces polycouronnes sont m\^eme \'elastiques. Finalement, nous en d\'eduirons que tout point d'un espace affine analytique admet une base de voisinages connexes par arcs.

Introduisons une notation pour les polycouronnes relatives.

\begin{nota}\label{nota:disquerelatif}%
\nomenclature[Iva]{$\E{n}{U}$}{espace affine relatif au-dessus de~$U$}%
\nomenclature[Ivaa]{$D_{U}(\bt)$}{polydisque ouvert relatif au-dessus de~$U$}%
\nomenclature[Ivb]{$\oD_{U}(\bt)$}{polydisque ferm\'e relatif au-dessus de~$U$}%
\nomenclature[Ivc]{$C_{U}(\bs,\bt)$}{polycouronne ouverte relative au-dessus de~$U$}%
\nomenclature[Ivd]{$\overline{C}_{U}(\bs,\bt)$}{polycouronne ferm\'ee relative au-dessus de~$U$}%
Soit $X$ un espace $\cA$-analytique. Consid\'erons l'espace affine analytique $\E{n}{\cA}$, avec coordonn\'ees $T_{1},\dotsc,T_{n}$, et posons $\E{n}{X} := X \times_{\cA} \E{n}{\cA}$. Notons $\pi\colon \E{n}{X} \to X$ la premi\`ere projection. Pour $i \in \cn{1}{n}$, on note encore~$T_{i}$ le tir\'e en arri\`ere de~$T_{i}$ sur~$\E{n}{X}$ par la seconde projection.

Soit~$U$ une partie de~$X$. On pose 
\[\E{n}{U} := \pi^{-1}(U).\]
Soient $\bs=(s_1,\ldots,s_n), \bt=(t_1,\ldots,t_n) \in \R^n$. On pose 
\begin{align*}
D_{U}(\bt) &:= \big\{x\in \E{n}{U} : \forall i\in\cn{1}{n}, |T_{i}(x)| < t_{i}\big\}~;\\
\overline{D}_{U}(\bt) &:= \big\{x\in \E{n}{U} : \forall i\in\cn{1}{n}, |T_{i}(x)| \le t_{i}\big\}~;\\
C_{U}(\bs,\bt) &:= \big\{x\in \E{n}{U} : \forall i\in\cn{1}{n}, s_{i} < |T_{i}(x)| < t_{i}\big\}~;\\
\overline{C}_{U}(\bs,\bt) &:= \big\{x\in \E{n}{U} : \forall i\in\cn{1}{n}, s_{i} \le |T_{i}(x)| \le t_{i}\big\}.
\end{align*}
\end{nota}


\begin{lemm}\label{lem:cpasurMk}
Soit $(k,\va)$ un corps valu\'e ultram\'etrique complet. Soient $\bs=(s_1,\ldots,s_n), \bt=(t_1,\ldots,t_n) \in \R^n$. 
Alors $C_{\cM(k)}(\bs,\bt)$ est connexe par arcs.
\end{lemm}
\begin{proof}
Nous allons distinguer trois cas, selon le type du corps~$k$ consid\'er\'e. 

\smallbreak

$\bullet$ Supposons que~$k=\C$.

Alors, d'apr\`es le lemme~\ref{lem:AnC}, $\E{n}{k}$ est isomorphe \`a~$\C^n$. Par cons\'equent, $C_{\cM(k)}(\bs,\bt)$ est une polycouronne complexe. Elle est donc connexe par arcs.

\smallbreak

$\bullet$ Supposons que~$k=\R$.

Alors, d'apr\`es le lemme~\ref{lem:AnC}, $\E{n}{k}$ est hom\'eomorphe \`a~$\C^n/\Gal(\C/\R)$. Par cons\'equent, $C_{\cM(k)}(\bs,\bt)$ est l'image d'une polycouronne complexe par une application continue. Elle est donc \'egalement connexe par arcs.

\smallbreak

$\bullet$ Supposons que~$k$ est ultram\'etrique.

On renvoie \`a~\cite[corollary~3.2.3]{Ber1}. Dans cette r\'ef\'erence, V.~Berkovich d\'emontre la connexit\'e par arcs du polydisque unit\'e, mais le m\^eme raisonnement fonctionne pour une polycouronne ferm\'ee quelconque. Le cas d'une polycouronne ouverte s'en d\'eduit en l'\'ecrivant comme union croissante de polycouronnes ferm\'ees. 
\end{proof}

\begin{prop}\label{connexe}\index{Disque!connexe par arcs}\index{Couronne!connexe par arcs}\index{Espace affine analytique!connexe par arcs}
Soit~$U$ une partie connexe par arcs de~$B$. Supposons que~$U$ est peu mixte. Soient $\bs=(s_1,\ldots,s_n), \bt=(t_1,\ldots,t_n) \in \R^n$. 
Alors $C_U(\bs,\bt)$ est connexe par arcs.

En particulier, $\pi^{-1}(U)$ est connexe par arcs.
\end{prop}
\begin{proof}
On peut supposer que, pour tout $i\in \cn{1}{n}$, on a $s_{i} < t_{i}$ et $t_{i}>0$. En effet, si tel n'est pas le cas, $C_{U}(\bs,\bt)$ est vide et le r\'esultat est imm\'ediat.

Soient $x,y \in C_U(\bs,\bt)$. Puisque~$U$ est connexe par arcs, il existe un chemin dans~$U$ reliant~$\pi(x)$ \`a~$\pi(y)$. Puisque~$U$ est peu mixte, on peut trouver un tel chemin~$\ell$ qui s'\'ecrive comme concat\'enation d'un nombre fini de chemins purs $\ell_{1},\dotsc,\ell_{m}$.

Soit $j \in \cn{1}{m}$. Soit $\varphi_{j} \colon [0,1] \to U$ l'application continue correspondant au chemin~$\ell_{j}$. Posons $a_{j} := \varphi_{j}(0)$ et $b_{j} := \varphi_{j}(1)$. Montrons que la fibre $\pi^{-1}(a_{j})$ peut-\^etre reli\'ee par un chemin \`a toute fibre au-dessus d'un point de $\varphi_{j}(\intoo{0,1})$. Distinguons trois cas.

$\bullet$ Supposons que $\varphi_{j}(\intoo{0,1}) \subset U_{\um}$.

Alors $a_{j} \in U_{\um}$. Soit $\br = (r_{1},\dotsc,r_{n}) \in \R_{>0}^n$ tel que, pour tout $i\in\cn{1}{n}$, on ait $s_{i} < r_{i} < t_{i}$. D'apr\`es le lemme~\ref{lem:section}, la restriction de la section~$\sigma_{\br}$  \`a~$\varphi_{j}(\intfo{0,1})$ est continue. Le r\'esultat s'ensuit.

$\bullet$ Supposons que $\varphi_{j}(\intoo{0,1}) \subset U_{\arc}$ et que $a_{j} \in U_{\arc}$.

On conclut par le m\^eme raisonnement que dans le cas pr\'ec\'edent.

$\bullet$ Supposons que $\varphi_{j}(\intoo{0,1}) \subset U_{\arc}$ et que $a_{j} \in U_{\um}$.

Alors $a_{j} \in \partial U_{\arc}$. Par hypoth\`ese, il existe $\br = (r_{1},\dotsc,r_{n}) \in \R_{>0}^n$ tel que, pour tout $i\in\cn{1}{n}$, on ait $s_{i} < r_{i} < t_{i}$, et $\br$ soit une famille libre de $\R_{>0}/|\cH(a_{j})^\times|^\Q$. D'apr\`es le lemme~\ref{lem:section}, la restriction de la section~$\sigma_{\br}$  \`a~$\varphi_{j}(\intfo{0,1})$ est encore continue. Le r\'esultat s'ensuit.

On a montr\'e que la fibre $\pi^{-1}(a_{j})$ peut-\^etre reli\'ee par un chemin \`a toute fibre au-dessus d'un point de $\varphi_{j}(\intoo{0,1})$. Le m\^eme raisonnement s'applique en rempla\c cant~$a_{j}$ par~$b_{j}$ et il en d\'ecoule que les fibres $\pi^{-1}(a_{j})$ et~$\pi^{-1}(b_{j})$ peuvent \^etre reli\'ees par un chemin. On en d\'eduit que fibres $\pi^{-1}(\pi(x))$ et~$\pi^{-1}(\pi(y))$ peuvent \^etre reli\'ees par un chemin. 

Pour d\'emontrer le r\'esultat, il suffit maintenant de d\'emontrer que les fibres de~$\pi$ en les points de~$U$ sont connexes par arcs. Or, pour tout $b\in U$, on a un hom\'eomorphisme $\pi^{-1}(b) \simeq \E{n}{\cH(b)}$, et le r\'esultat d\'ecoule donc du lemme~\ref{lem:cpasurMk}.

D\'emontrons maintenant la derni\`ere partie de l'\'enonc\'e. Posons $\bs := (-1,\dotsc,-1)$. On obtient le r\'esultat en \'ecrivant $\pi^{-1}(U)$ comme union croissante de polydisques ouverts $D_{U}(\bt) = C_{U}(\bs,\bt)$ et en utilisant la premi\`ere partie.
\end{proof}

Nous allons maintenant d\'eduire du r\'esultat pr\'ec\'edent la connexit\'e locale de~$\E{n}{\cA}$.

\begin{prop}\label{connexit\'e_fini_ouvert}\index{Voisinage}\index{Morphisme analytique!fini!et plat}\index{Couronne}
Supposons que~$B$ est localement connexe par arcs. Soit $x \in \E{n}{\cA}$. Le point~$x$ admet une base de voisinages ouverts~$\cV$ telle que, pour tout~$V\in\cV$, il existe un ouvert connexe par arcs~$U$ de~$B$, $\bs=(s_1,\dotsc,s_n), \bt=(t_1,\ldots,t_n) \in \R^n$ et un morphisme fini et plat $\varphi \colon V\to C_U(\bs,\bt)$.
\end{prop}
\begin{proof}
Nous allons d\'emontrer l'\'enonc\'e par r\'ecurrence sur~$n$. Si $n=0$, le r\'esultat d\'ecoule directement de la connexit\'e par arcs locale de~$B$.

Soit $n\ge 1$ et supposons avoir d\'emontr\'e le r\'esultat pour~$n-1$. Soit $x\in \E{n}{\cA}$. Notons $\pi_{n-1} \colon \E{n}{\cA} \to \E{n-1}{\cA}$ la projection sur les $n-1$ premi\`eres coordonn\'ees et posons $x_{n-1} := \pi_{n-1}(x_{n})$. Soit~$\cV_{n-1}$ une base de voisinages ouverts de~$x_{n-1}$ dans~$\E{n-1}{\cA}$ satisfaisant les conditions de l'\'enonc\'e. Notons~$T$ la derni\`ere coordonn\'ee de~$\E{n}{\cA}$.

Soit~$W$ un voisinage de~$x$. D'apr\`es les propositions~\ref{prop:basevoisdim1rigide} et~\ref{prop:basevoisdim1}, il existe $V \in \cV_{n-1}$, $s,t \in \R$ et $P \in \cO_{\E{n-1}{\cA}}(V)[T]$ tels que $C_{V}(P,s,t) \subset W$. Puisque $V \in \cV_{n-1}$, il existe un ouvert connexe par arcs~$U$ de~$B$, $\bs'=(s_1,\dotsc,s_{n-1}), \bt'=(t_1,\ldots,t_{n-1}) \in \R^{n-1}$ et un morphisme fini et plat $\varphi_{n-1} \colon V\to C_U(\bs',\bt')$. D'apr\`es le lemme~\ref{lem:produitouverts} et la proposition~\ref{prop:produitaffines}, $\pi_{n-1}^{-1}(V)$ et $\pi_{n-1}^{-1}(C_U(\bs',\bt'))$ s'identifient respectivement aux produits $V\times_{\cA}\AunA$ et $C_U(\bs',\bt')\times_{\cA} \AunA$. La proposition~\ref{changement_base_plat} assure donc que le morphisme 
\[\varphi_{n} \colon \pi_{n-1}^{-1}(V) \too \pi_{n-1}^{-1}(C_U(\bs',\bt'))\]
d\'eduit de~$\varphi_{n-1}$ par changement de base est fini et plat. Posons $\bs := (\bs',s)$ et $\bt:=(\bt',t)$. Les propri\'et\'es de finitude et de platitude \'etant locales au but, le morphisme $\varphi'_{n} \colon C_V(s,t)\to C_U(\bs,\bt)$ induit par~$\varphi_{n}$ est encore fini et plat. 

D'apr\`es la proposition~\ref{plat}, le morphisme $\varphi_{P} \colon \pi_{n-1}^{-1}(V) \to \pi_{n-1}^{-1}(V)$ induit par~$P$ est fini et plat. Sa restriction $\varphi'_{P} \colon C_V(P,s,t)\to C_V(s,t)$ l'est donc \'egalement. Ceci montre que le morphisme $\varphi := \varphi'_{n} \circ \varphi'_{P}$ satisfait les conditions de l'\'enonc\'e.
\end{proof}

En combinant les propositions~\ref{connexit\'e_fini_ouvert} et~\ref{connexe}, le corollaire~\ref{plat} et la proposition~\ref{prop:top}, on obtient le r\'esultat suivant.

\begin{coro}
Supposons que~$B$ est localement connexe par arcs et peu mixte. Alors, l'espace $\E{n}{\cA}$ est localement connexe.
\qed
\end{coro}

Nous allons maintenant montrer que l'ensemble~$C_U(\bs,\bt)$ est \'elastique. 

\begin{prop}\label{elast_partiel1}
Supposons que~$B_{\um}$ est localement connexe par arcs. Soit~$U$ un ouvert connexe de~$B_{\um}$. Soient $\bs=(s_1,\ldots,s_n), \bt=(t_1,\ldots,t_n) \in \R^n$. Alors, la polycouronne~$C_U(\bs,\bt)$ est \'elastique. 
\end{prop}
\begin{proof}
Montrons, par r\'ecurrence sur~$n$, que, pour tout ouvert connexe~$U$ de~$B_{\um}$, et tous $\bs=(s_1,\ldots,s_n), \bt=(t_1,\ldots,t_n) \in \R^n$, la polycouronne $C_U(\bs,\bt)$ est \'elastique. 

Pour $n=0$, le r\'esultat est vrai car tout ouvert connexe de~$B_{\um}$ est connexe par arcs et localement connexe par arcs, donc \'elastique, d'apr\`es la remarque~\ref{rem:definitionelastique}.

Soit $n\ge 1$ et supposons que le r\'esultat est vrai pour~$n-1$. Soit~$U$ un ouvert connexe de~$B_{\um}$. Soient $\bs=(s_1,\ldots,s_n), \bt=(t_1,\ldots,t_n) \in \R^n$. Posons $\bs' := (s_1,\ldots,s_{n-1}), \bt' := (t_1,\ldots,t_{n-1})$ et $U' := C_{U}(\bs',\bt')$. On peut supposer que, pour tout $i\in \cn{1}{n}$, on a $s_{i} < t_{i}$ et $t_{i}>0$. D'apr\`es le lemme~\ref{lem:crit_elast'}, il suffit de montrer que, pour tous $s'_{n},t'_{n} \in \R$ avec $s_{n} < s'_{n} < t'_{n} < t_{n}$ et $t'_{n}>0$, l'espace $\overline{D}_{U'}(s'_{n},t'_{n})$ est \'elastique. 

Soient $s'_{n},t'_{n} \in \R$ avec $s_{n} < s'_{n} < t'_{n} < t_{n}$ et $t'_{n}>0$. Soit $x \in \overline{D}_{U'}(s'_{n},t'_{n})$. Notons~$y$ sa projection sur~$U'$. 
Montrons qu'il existe un chemin reliant~$x$ \`a~$\eta_{y,t'_{n}}$ dans $\overline{D}_{U'}(s'_{n},t'_{n})$ qui est \'elastique en~$x$. Cela suffira pour conclure, d'apr\`es la remarque~\ref{rem:definitionelastique}.

Fixons une coordonn\'ee~$T$ sur $\overline{D}_{U}(s'_{n},t'_{n})$. Elle induit une coordonn\'ee sur $\overline{D}_{y}(s'_{n},t'_{n})$ que l'on note identiquement. D\'efinissons un ordre partiel~$\le$ sur $\overline{D}_{y}(s'_{n},t'_{n})$ de la fa\c con suivante~: pour tous $z_{1},z_{2} \in \overline{D}_{y}(s'_{n},t'_{n})$, on a $z_{1} \le z_{2}$ si
\[\forall P \in \cH(y)[T],\ |P(z_{1})| \le |P(z_{2})|.\]
Puisque~$\cO_{U',y}$ est dense dans~$\cH(y)$, on peut remplacer~$\cH(y)$ par~$\cO_{U',y}$ dans cette d\'efinition. Le point~$\eta_{y,t'_{n}}$ est le plus grand point de $\overline{D}_{y}(s'_{n},t'_{n})$ pour l'ordre partiel~$\le$. En outre, l'ensemble
\[ \{z \in \overline{D}_{y}(s'_{n},t'_{n}) : z \ge x\}\]
est hom\'eomorphe \`a un intervalle ferm\'e dont les extr\'emit\'es s'envoient sur~$x$ et~$\eta_{y,t'_{n}}$. Identifions-le \`a un chemin~$\ell$ reliant~$\eta_{y,t'_{n}}$ \`a~$x$ et montrons qu'il est \'elastique en~$x$.

Soit~$U''$ un voisinage compact de~$y$ dans~$U'$. Pour tout $\eps \in \R_{>0}$, tout voisinage~$W$ de~$y$ dans~$U''$ et tout $P\in\cO(W)[T]$, posons 
\[L_{\eps,W,P}:=\{z\in \overline{D}_{W}(s'_n,t'_n) : |P(z)| \ge |P(x)|-\epsilon\}.\]
C'est un voisinage de~$\ell_{x}$ dans $\overline{D}_{U''}(s'_n,t'_n)$, compact si~$W$ est compact. Remarquons que l'on a 
\[ \ell = \bigcap_{\eps>0} \bigcap_{W \ni y}\bigcap_{P\in\cO(W)[T]} L_{\eps,W,P},\]
o\`u $W$ d\'ecrit l'ensemble des voisinages compacts de~$y$ dans~$U'$.
 
Soit~$V$ un voisinage ouvert de~$\ell$ dans~$\overline{D}_{U''}(s'_n,t'_n)$. Alors $\overline{D}_{U''}(s'_n,t'_n) \setminus \ell$ est compact et on en d\'eduit qu'il existe un ensemble fini~$I$ et, pour chaque $i\in I$, un nombre r\'eel~$\eps_{i} >0$, un voisinage compact~$W_{i}$ de~$y$ dans~$U''$ et un polyn\^ome $P_{i} \in \cO(W_{i})[T]$ tels que
\[ \bigcap_{i\in I} L_{\eps_{i},W_{i},P_{i}} \subset V.\]
En utilisant la proposition \ref{connexit\'e_fini_ouvert}, le corollaire~\ref{plat}, l'hypoth\`ese de r\'ecurrence et la proposition~\ref{chemin}, on montre que $\bigcap_{i\in I} W_{i}$ contient un voisinage connexe par arcs~$W_{0}$ de~$y$. On a encore 
\[L := \bigcap_{i\in I} L_{\eps_{i},W_{0},P_{i}} \subset V.\]

Montrons que tout point~$x'$ de~$L$ peut-\^etre reli\'e \`a~$\eta_{y,t'_{n}}$ par un chemin contenu dans~$V$. Soit $x'\in L$. Notons~$y'$ sa projection sur~$U'$. L'ensemble 
\[\ell_{x'} := \{z \in \overline{D}_{y'}(s'_{n},t'_{n}) : z \ge x'\}\]
s'identifie \`a un chemin reliant~$x'$ \`a~$\eta_{y',t'_{n}}$ et ce chemin est contenu dans~$L$. Puisque~$W_{0}$ est connexe par arcs, il existe un chemin~$\ell_{y'}$ dans~$W_{0}$ reliant~$y'$ \`a~$y$. Son image $\sigma_{t'_{n}}(\ell_{y'})$ est un chemin reliant~$\eta_{y',t'_{n}}$ \`a~$\eta_{y,t'_{n}}$ et contenu dans~$L$. Le r\'esultat s'ensuit.
\end{proof}

\begin{prop}\label{elast_partiel2}
Supposons que~$B_{\arc}$ est localement connexe par arcs. Soit~$U$ un ouvert connexe de~$B_{\arc}$. Soient $\bs=(s_1,\ldots,s_n), \bt=(t_1,\ldots,t_n) \in \R^n$. Alors, la polycouronne~$C_U(\bs,\bt)$ est \'elastique. 
\end{prop}
\begin{proof}
Nous suivrons une strat\'egie similaire \`a celle de la preuve de la proposition~\ref{elast_partiel1}. Montrons, par r\'ecurrence sur~$n$, que, pour tout ouvert connexe~$U$ de~$B_{\arc}$, et tous $\bs=(s_1,\ldots,s_n), \bt=(t_1,\ldots,t_n) \in \R^n$, la polycouronne $C_U(\bs,\bt)$ est \'elastique. 

Pour $n=0$, le r\'esultat est vrai car tout ouvert connexe de~$B$ est \'elastique.

Soit $n\ge 1$ et supposons que le r\'esultat est vrai pour~$n-1$. Soit~$U$ un ouvert connexe de~$B_{\arc}$. Soient $\bs=(s_1,\ldots,s_n), \bt=(t_1,\ldots,t_n) \in \R^n$. On peut supposer que, pour tout $i\in \cn{1}{n}$, on a $s_{i} < t_{i}$ et $t_{i}>0$. Posons $\bs' := (s_1,\ldots,s_{n-1}), \bt' := (t_1,\ldots,t_{n-1})$ et $U' := C_{U}(\bs',\bt')$. D'apr\`es le lemme~\ref{lem:crit_elast'}, il suffit de montrer que, pour tous $s'_{n},t'_{n} \in \R$ avec $s_{n} < s'_{n} < t'_{n} < t_{n}$ et $t'_{n}>0$, l'espace $\overline{D}_{U'}(s'_{n},t'_{n})$ est \'elastique. 

Soient $s'_{n},t'_{n} \in \R$ avec $s_{n} < s'_{n} < t'_{n} < t_{n}$ et $t'_{n}>0$. Notons $\pi \colon \overline{D}_{U'}(s'_{n},t'_{n}) \to U'$ le morphisme de projection. Rappelons que, pour tout $r\in [s'_{n},t'_{n}] \cap \R_{\ge 0}$, nous en avons d\'efini une section continue~$\sigma_{r}$ au lemme~\ref{lem:section}.

Fixons une coordonn\'ee~$T$ sur $\overline{D}_{U'}(s'_{n},t'_{n})$. Soit $x \in \overline{D}_{U'}(s'_{n},t'_{n})$. Posons $y := \pi(x)$. L'ensemble
\[\{z \in \pi^{-1}(y) : |T(z)| = |T(x)|\} \]
est hom\'eomorphe \`a un cercle si $\cH(y) = \C$ et \`a un demi-cercle si $\cH(y) = \R$. C'est l'ensemble sous-jacent \`a un chemin~$\ell_{x}$ reliant~$x$ \`a~$\sigma_{|T(x)|}(y)$. 

Soit~$U''$ un voisinage compact de~$y$ dans~$U'$. Pour tout $\eps \in \R_{>0}$ et tout voisinage~$W$ de~$y$ dans~$U'$, posons 
\[L_{\eps,W}:=\big\{z\in \overline{D}_{W}(s'_n,t'_n) : \big||T(z)| - |T(x)|\big| \le \epsilon\big\}.\]
C'est un voisinage de~$\ell_{x}$ dans $\overline{D}_{U''}(s'_n,t'_n)$, compact si~$W$ est compact. 

Soit~$V$ un voisinage ouvert de~$\ell_{x}$ dans~$\overline{D}_{U''}(s'_n,t'_n)$. Par le m\^eme raisonnement que dans la preuve de la proposition~\ref{elast_partiel1}, on montre qu'il existe un nombre r\'eel $\eps >0$ et un voisinage connexe par arcs~$W$ de~$y$ dans~$U''$ tels que
\[ L_{\eps,W} \subset V.\]

Montrons que tout point~$x'$ de~$L_{\eps,W}$ peut-\^etre reli\'e \`a~$\sigma_{|T(x)|}(y)$ par un chemin contenu dans~$V$. Soit $x'\in L$. Notons~$y'$ sa projection sur~$W$. Puisque l'ensemble
\[\{z \in \pi^{-1}(y') : |T(z)| = |T(x')|\} \]
est contenu dans~$L_{\eps,W}$, il existe un chemin reliant~$x'$ \`a~$\sigma_{|T(x')|}(y')$ contenu dans~$L_{\eps,W}$. On peut \'egalement relier $\sigma_{|T(x')|}(y')$ \`a $\sigma_{|T(x)|}(y')$ dans~$L_{\eps,W}$, puis $\sigma_{|T(x)|}(y')$ \`a $\sigma_{|T(x)|}(y)$, toujours dans~$L_{\eps,W}$, gr\^ace \`a la continuit\'e de la section~$\sigma_{|T(x)|}$. Ceci termine la d\'emonstration.
\end{proof}

\begin{prop}\label{elast}\index{Couronne!elastique@\'elastique}\index{Disque!elastique@\'elastique}
Supposons que~$B$ et~$B_{\um}$ sont localement connexes par arcs et que $B$~est peu mixte.
Soit~$U$ un ouvert connexe de~$B$. Soient $\bs=(s_1,\ldots,s_n), \bt=(t_1,\ldots,t_n) \in \R^n$. Alors la polycouronne $C_U(\bs,\bt)$ est \'elastique. 
\end{prop}
\begin{proof}
D'apr\`es les propositions~\ref{elast_partiel1} et~\ref{elast_partiel2}, les polycouronnes $C_{U_{\arc}}(\bs,\bt)$ et $C_{U_{\um}}(\bs,\bt)$ sont \'elastiques. D'apr\`es le lemme~\ref{lem:crit_elast} appliqu\'e avec $W=C_{U_{\arc}}(\bs,\bt)$ et $F=C_{U_{\um}}(\bs,\bt)$, il suffit de montrer que tout point $x$ de $\partial C_{U_{\um}}(\bs,\bt)$ admet une base de voisinages~$\cV_x$ dans~$C_U(\bs,\bt)$ telle que, pour tout $V \in \cV_x$ et tout $y\in V\cap C_{U_\arc}(\bs,\bt)$, il existe un chemin dans~$V$ reliant~$y$ \`a un point de~$C_{U_{\um}}(\bs,\bt)$.

Soit $x \in \partial C_{U_{\um}}(\bs,\bt)$. D'apr\`es la proposition~\ref{connexit\'e_fini_ouvert}, $x$ admet une base de voisinages dans~$C_U(\bs,\bt)$ form\'ee d'ouverts~$V$ pour lesquels il existe un ouvert connexe~$U'$ de~$U$, $\bs'=(s'_1,\dotsc,s'_n), \bt'=(t'_1,\dotsc,t'_n) \in \R_{\ge0}^n$ avec $s'_{i}<t'_{i}$ pour tout $i\in\cn{1}{n}$ et un morphisme fini et plat $\varphi \colon V\to C_{U'}(\bs',\bt')$. D'apr\`es le corollaire~\ref{plat}, $\varphi$ est ouvert.

Montrons qu'un tel~$V$ satisfait la propri\'et\'e requise. D'apr\`es la proposition~\ref{connexe}, $C_{U'}(\bs',\bt')$ est connexe par arcs donc, d'apr\`es la proposition~\ref{prop:revetementcpa}, chaque composante connexe par arcs de~$V$ se surjecte sur $C_{U'}(\bs',\bt')$. En particulier, tout point de~$V$ peut-\^etre reli\'e \`a un point de $V_{\um}$.
\end{proof}

\begin{theo}\label{arbre}\index{Espace affine analytique!localement connexe par arcs}
Supposons que~$B$ et~$B_{\um}$ sont localement connexes par arcs et que $B$~est peu mixte. Alors, l'espace analytique~$\E{n}{\cA}$ est localement connexe par arcs.
\end{theo}
\begin{proof}
Soient~$x\in\E{n}{\cA}$ et~$W$ un voisinage de~$x$ dans~$\E{n}{\cA}$. D'apr\`es la proposition \ref{connexit\'e_fini_ouvert}, il existe un voisinage~$V$ de~$x$ dans~$W$ pour lesquel il existe un ouvert connexe~$U$ de~$B$, $\bs=(s_1,\dotsc,s_n), \bt=(t_1,\dotsc,t_n) \in \R_{\ge0}^n$ avec $s_{i}<t_{i}$ pour tout $i\in\cn{1}{n}$ et un morphisme fini et plat $\varphi \colon V\to C_{U}(\bs,\bt)$. D'apr\`es le corollaire~\ref{plat}, $\varphi$ est ouvert. Il suit des propositions~\ref{connexe} et~\ref{prop:top} que la composante connexe~$V_{x}$ de~$V$ contenant~$x$ est ouverte et se surjecte sur $C_{U}(\bs,\bt)$ par~$\varphi$. D'apr\`es les propositions~\ref{elast} et~\ref{chemin}, $V_{x}$ est connexe par arcs. Le r\'esultat s'ensuit.
\end{proof}

Nous g\'en\'eralisons finalement le r\'esultat de connexit\'e par arcs locale \`a des espaces analytiques quelconques. 

\begin{theo}\label{th:cpageneral}\index{Espace analytique!localement connexe par arcs}
Supposons que~$B$ et~$B_{\um}$ sont localement connexes par arcs et que $B$~est peu mixte. Alors tout espace $\cA$-analytique est localement connexe par arcs.
\end{theo}
\begin{proof}
Soit $X$ un espace $\cA$-analytique et soit $x\in X$. La proposition \ref{decomp} assure qu'il existe un voisinage ouvert~$V$ de~$x$ dans~$X$ et des ferm\'es analytiques $V_1,\dotsc,V_{s}$ de~$V$ contenant~$x$ et int\`egres en~$x$ tels que $V=\bigcup_{i=1}^s V_i$. Il suffit de montrer que, pour tout $i\in\cn{1}{s}$, $x$ admet une base de voisinages connexes par arcs chacun des~$V_{i}$. Il suffit donc de traiter le cas d'un espace int\`egre en~$x$. Par cons\'equent, on supposera d\'esormais que~$X$ est int\`egre en~$x$. 

Soit~$U$ un voisinage ouvert de~$x$ dans~$X$. On peut supposer qu'il est s\'epar\'e. Le th\'eor\`eme \ref{proj} assure que, quitte \`a r\'etr\'ecir~$U$, on peut supposer qu'il existe un morphisme fini ouvert $\varphi \colon U \to V$, o\`u~$V$ est un ouvert de $\E{n}{\cA}$ ou de~$\E{n}{\cH(b)}$ pour un certain $b\in B$. D'apr\`es le th\'eor\`eme~\ref{arbre}, $\E{n}{\cA}$ et~$\E{n}{\cH(b)}$ sont localement connexes par arcs. Quitte \`a r\'etr\'ecir~$V$ (et~$U$ en cons\'equence), on peut donc supposer que~$V$ est connexes par arcs et localement connexe par arcs. La remarque~\ref{rem:definitionelastique} assure que~$V$ est \'egalement \'elastique. 

D'apr\`es les propositions~\ref{prop:top} et~\ref{chemin} , $U$ poss\`ede un nombre fini de composantes connexes et celles-ci sont connexes par arcs. On en d\'eduit le r\'esultat souhait\'e.
\end{proof}

\section[Dimension topologique]{Dimension topologique}\label{sec:dimtop}

Dans cette section, nous calculons la dimension topologique des espaces affines et des disques sur des bases convenables.

Rappelons tout d'abord la d\'efinition de dimension de recouvrement. On rappelle qu'un espace topologique~$T$ est dit \emph{normal} si deux ferm\'es disjoints de~$T$ sont toujours contenus dans deux ouverts disjoints. Par exemple, un espace paracompact (s\'epar\'e) est normal. \index{Espace topologique!normal|textbf}

Un espace $\cA$-analytique peut ne pas \^etre normal (ni m\^eme s\'epar\'e). En revanche, d'apr\`es la remarque~\ref{rem:proplocales}, il est toujours localement compact (s\'epar\'e), donc localement normal.

\index{Espace topologique!dimension d'un|(}\index{Dimension!topologique|see{Espace topologique}}

\begin{defi}[\protect{\cite[definitions~1.6.6, 1.6.7]{Eng1}}]\index{Espace topologique!dimension d'un|textbf}%
\nomenclature[D]{$\dimr$}{dimension de recouvrement d'un espace topologique}
Soient~$T$ un espace topologique normal non vide. 
\begin{enumerate}[i)]
\item Soit~$\cU$ un ensemble non vide de parties de~$T$. On appelle \emph{ordre de~$\cU$} la borne sup\'erieure de l'ensemble des entiers~$n$ tel qu'il existe $n+1$ \'el\'ements distincts de~$\cU$ d'intersection non vide.
\item On appelle \emph{dimension de recouvrement de~$T$} la borne sup\'erieure de l'ensemble des entiers~$n$ tel que tout recouvrement fini ouvert de~$T$ admette un raffinement fini d'ordre inf\'erieur o\`u \'egal \`a~$n$. On la note $\dimr(X) \in \N \cup \{\infty\}$.
\end{enumerate}
\end{defi}

Dans la suite du texte, la dimension d'un espace sera toujours \`a prendre au sens de la dimension de recouvrement, m\^eme si nous omettrons parfois cette pr\'ecision.

\begin{rema}\label{rem:angelique}\index{Espace topologique!m\'etrisable}\index{Espace affine analytique!m\'etrisable}
La notion de dimension se comporte bien dans le cas des espaces m\'etrisables et s\'eparables (au sens o\`u ils poss\`edent un sous-ensemble d\'enombrable dense). Les espaces que nous consid\'erons ne jouissent malheureusement pas tous de ces propri\'et\'es. Par exemple, la droite analytique sur le corps~$\C$ muni de la valuation triviale n'est ni m\'etrisable, ni s\'eparable. Nous devrons donc manipuler la notion de dimension avec pr\'ecaution. 

Signalons cependant que, si~$\cA$ poss\`ede un sous-ensemble d\'enombrable dense (par exemple, si $\cA$ est~$\Z$ ou un anneau d'entiers de corps de nombres), alors tout espace affine analytique sur~$\cA$ est m\'etrisable et s\'eparable (\cf~\cite[th\'eor\`eme~3.5.1]{A1Z} et sa preuve). 

Ajoutons finalement que, m\^eme lorsqu'ils ne sont pas m\'etrisables, les espaces de Berkovich sur un corps valu\'e ultram\'etrique sont toujours ang\'eliques (\cf~\cite{Angie}). Ils se comportent donc comme des espaces m\'etriques \`a bien des \'egards. Par exemple, les caract\'erisations s\'equentielles usuelles de la fermeture ou de la compacit\'e sont valables. Nous ignorons si ce r\'esultat d'ang\'elicit\'e persiste pour tous les espaces analytiques globaux. 
\end{rema}

Rappelons quelques propri\'et\'es de la notion de dimension de recouvrement sur les espaces topologiques normaux.

\begin{theo}[\protect{\cite[theorem~3.1.3]{Eng1}}]\label{dimension_sous-espace}
Soient~$T$ un espace topologique normal et~$F$ une partie ferm\'ee de~$T$. Alors, on a
\[\dimr(F)\le \dimr(T).\]
\qed
\end{theo}

\begin{theo}[\protect{\cite[theorem~3.1.8]{Eng1}}]\label{dimension_union_d\'enombrable}
Soient~$T$ un espace topologique normal et~$(F_i)_{i\in I}$ une famille d\'enombrable de parties ferm\'ees de~$T$ telle que $T=\bigcup_{i\in I} F_i$. 
Alors, on a
\[\dimr(T)\leq \sup_{i\in I} (\dimr(F_{i})).\]
\qed
\end{theo}

\begin{theo}[\protect{\cite[theorem~3.2.13]{Eng1}}]\label{th:produit}
Soient~$T$ et~$S$ des espaces topologiques compacts non vides. Alors, on a
\[\dimr(T\times S)\le \dimr(T) + \dimr(S).\]
\qed
\end{theo}
\index{Espace topologique!dimension d'un|)}

\index{Espace affine analytique!dimension d'un|(}\index{Disque!dimension d'un|(}
Rappelons le r\'esultat suivant pour les espaces analytiques sur un corps valu\'e complet.

\begin{theo}\label{th:dimensioncorps}
Soit $(k,\va)$ un corps valu\'e complet. Soient $n\in \N$ et $\br \in \R_{>0}^n$. Alors, on a 
\begin{align*} 
\dimr(\E{n}{k}) & = \dimr(\oD_{k}(\br)) = \dimr(D_{k}(\br))\\ 
& = \begin{cases}
2n & \textrm{ si } (k,\va) \textrm{ est archim\'edien~;}\\
n & \textrm{ si } (k,\va) \textrm{ est ultram\'etrique.}
\end{cases}
\end{align*}
En outre, chacun des espaces $\E{n}{k}$, $\oD_{k}(\br)$ et~$D_{k}(\br)$ contient un pav\'e compact de m\^eme dimension.
\end{theo}
\begin{proof}
Si $(k,\va)$ est archim\'edien, alors, d'apr\`es le th\'eor\`eme~\ref{th:vaarchimedienne} et le lemme~\ref{lem:AnC}, $\E{n}{k}$ est hom\'eomorphe \`a~$\C^n$ ou $\C^n/\Gal(\C/\R)$. Le r\'esultat s'en d\'eduit. 

Supposons que $(k,\va)$ est ultram\'etrique. D'apr\`es~\cite[corollary~3.2.8]{Ber1}, on a $\dimr(\E{n}{k}) = n$.

D'apr\`es le th\'eor\`eme~\ref{dimension_sous-espace}, on a donc $\dimr(\oD_{k}(\br)) \le n$. Notons $\br = (r_{1},\dotsc,r_{n})$. Pour tout $(t_{1},\dotsc,t_{n}) \in \R_{>0}$, notons $\eta_{t_{1},\dotsc,t_{n}}$ l'unique point du bord de Shilov du disque $\oD_{k}(t_{1},\dotsc,t_{n})$. Pour tout $i\in \cn{1}{n}$, soient $s_{i} \in \intoo{0,r_{i}}$. L'application
\[ \begin{array}{cccc}
\varphi \colon & \disp\prod_{i=1}^n [s_{i},r_{i}] & \too & \oD_{k}(r_{1},\dotsc,r_{n})\\
& (t_{1},\dotsc,t_{n}) & \mapstoo & \eta_{t_{1},\dotsc,t_{n}}
\end{array}\]
r\'ealise un hom\'eomorphisme du pav\'e compact $\prod_{i=1}^n [s_{i},r_{i}]$ sur un ferm\'e de~$\oD_{k}(r_{1},\dotsc,r_{n})$. D'apr\`es le th\'eor\`eme~\ref{dimension_sous-espace}, on a donc 
\[ \dimr(\oD_{k}(\br)) \ge \dimr \Big(\prod_{i=1}^n [s_{i},r_{i}] \Big) = n.\]
Le r\'esultat s'ensuit.

Puisque le polydisque ouvert~$D_{k}(\br)$ contient un polydisque ferm\'e et peut s'\'ecrire comme union d\'enombrable de polydisques ferm\'es, le r\'esultat pr\'ec\'edent, joint aux th\'eor\`emes~\ref{dimension_sous-espace} et~\ref{dimension_union_d\'enombrable}, entra\^ine que $\dimr(D_{k}(\br)) = n$.
\end{proof}

En nous basant sur ce r\'esultat, nous allons pouvoir calculer la dimension de recouvrement d'espaces affines sur des bases plus g\'en\'erales. Nous utiliserons les notions de partie lin\'eaire et de partie d'Ostrowski introduites pr\'ec\'edemment (\cf~d\'efinitions~\ref{def:lineaire} et~\ref{def:ostrowski}).

\begin{theo}\label{th:dimensionetoile}\index{Partie!lineaire@lin\'eaire}\index{Partie!d'Ostrowski}
Soit~$V$ une partie lin\'eaire ou d'Ostrowski de~$B$. Soient $n\in \N$ et $\br \in \R_{>0}^n$. Notons $\pi \colon \E{n}{\cA} \to B$ la projection canonique. Alors, on a 
\begin{align*} 
\dimr(\pi^{-1}(V)) & = \dimr(\oD_{V}(\br)) = \dimr(D_{V}(\br))\\ 
& = \begin{cases}
2n+1 & \textrm{ si } V \cap B^\arc \ne \emptyset~;\\
n+1 & \textrm{ si } V \subset B^\um.
\end{cases}
\end{align*}
En outre, chacun des espaces $\pi^{-1}(V)$, $\oD_{V}(\br)$ et~$D_{V}(\br)$ contient un pav\'e compact de m\^eme dimension.
\end{theo}
\begin{proof}
Remarquons qu'il suffit de d\'emontrer le r\'esultat pour les polydisques ferm\'es. Ceci d\'ecoule des th\'eor\`emes~\ref{dimension_sous-espace} et~\ref{dimension_union_d\'enombrable} puisque les espaces affines et les polydisques ouverts contiennent et sont union d\'enombrables de tels disques. 

Posons 
\[ d = 
\begin{cases}
2n & \textrm{ si } V \cap B^\arc \ne \emptyset~;\\
n & \textrm{ si } V \subset B^\um. 
\end{cases}
\]

$\bullet$ Supposons que $V$ est lin\'eaire.

Par d\'efinition, il existe une partie~$V^{\partial}$ de~$V$, un point~$b$ de~$V$ et un intervalle $I \subset I_{b}$ non r\'eduit \`a un point satisfaisant les propri\'et\'es suivantes~: 
\begin{enumerate}[i)]
\item $V^{\partial}$ est contenue dans le bord de~$V$ et $\sharp V^{\partial}\le 2$~; 
\item l'application 
\[\begin{array}{ccc}
I & \too & V\setminus V^{\partial}\\
\eps & \mapstoo & b^\eps
\end{array}\]
est un hom\'eomorphisme.
\end{enumerate}

Supposons que $I$ est un segment et que $V^\partial = \emptyset$. Alors, d'apr\`es le lemme~\ref{lem:flot}, l'application
\[\fonction{\Phi}{\pi^{-1}(b) \times I}{\pi^{-1}(V)}{(x,\eps)}{x^\eps}\]
est un hom\'eomorphisme. 

Soit $\bs = (s_{1},\dotsc,s_{n}) \in \R_{>0}^n$. On a alors 
\[ \Phi( \oD_{b}(\bs) \times I) = \bigcup_{\eps \in I} \oD_{b^\eps}(s_{1}^\eps,\dotsc,s_{n}^\eps).\]

On peut choisir~$\bs$ de fa\c con que $\Phi( \oD_{b}(\bs) \times I) \subset \oD_{V}(\br)$. D'apr\`es le th\'eor\`eme~\ref{th:dimensioncorps}, $\oD_{b}(\bs)$ contient un pav\'e compact de dimension~$d$, donc $\oD_{b}(\bs) \times I$ contient un pav\'e compact de dimension~$d+1$. D'apr\`es le th\'eor\`eme~\ref{dimension_sous-espace}, on en d\'eduit que $\dimr(\oD_{V}(\br)) \ge d+1$.

On peut choisir \'egalement choisir~$\bs$ de fa\c con que $\oD_{V}(\br) \subset \Phi( \oD_{b}(\bs) \times I)$. D'apr\`es les th\'eor\`emes~\ref{th:dimensioncorps}, \ref{th:produit} et~\ref{dimension_sous-espace}, on a alors $\dimr(\oD_{V}(\br)) \le d+1$. Le r\'esultat s'ensuit.

\medbreak

Traitons maintenant le cas g\'en\'eral. Posons $F := \bigsqcup_{c \in V^\partial} \oD_{c}(\br)$. D'apr\`es le th\'eor\`eme~\ref{th:dimensioncorps}, $F$ est soit vide, soit de dimension~$d$. 

\'Ecrivons l'intervalle~$I$ comme une r\'eunion d\'enombrable $\bigcup_{n\in\N} I_{n}$, o\`u les~$I_{n}$ sont des segments non r\'eduits \`a un point.  Pour tout $n\in \N$, posons $V_{n} := \{b^\eps : \eps \in I_{n}\}$, $\oD_{n} := F \cup \oD_{V_{n}}(\br)$. D'apr\`es le cas pr\'ec\'edent, $\oD_{n}$ est de dimension~$d+1$ et contient un pav\'e de dimension~$d+1$. On en d\'eduit que $\oD_{V}(\br)$ contient un pav\'e de dimension~$d+1$ et que $\dimr(\oD_{V}(\br)) \ge d+1$, d'apr\`es le th\'eor\`eme~\ref{dimension_sous-espace}.

Pour $n\in\N$, posons $\oD'_{n} := F \sqcup \oD_{n}$. On a $\dimr(\oD'_{n}) = d+1$. Puisque, pour tout $n\in \N$, $\oD'_{n}$ est ferm\'e dans~$\oD_{V}(\br)$ et que $\oD_{V}(\br) = \bigcup_{n\in \N} \oD'_{n}$, le th\'eor\`eme~\ref{dimension_union_d\'enombrable} assure que $\dimr(\oD_{V}(\br)) \le d+1$. Le r\'esultat s'ensuit.

\medbreak

$\bullet$ Supposons que $V$ est d'Ostrowski.

Par hypoth\`ese, il existe un point $a_{0}$ de~$V$, un ensemble d\'enombrable non vide~$\Sigma$  et une famille $(V_{\sigma})_{\sigma\in \Sigma}$ de parties disjointes de~$V\setminus \{a_{0}\}$ satisfaisant les propri\'et\'es suivantes~: 
\begin{enumerate}[i)]
\item $V = \bigcup_{\sigma \in \Sigma} V_{\sigma} \cup \{a_{0}\}$~;
\item pour tout $\sigma\in \Sigma$, $V_{\sigma}$ et $V_{\sigma} \cup \{a_{0}\}$ sont des parties lin\'eaires de~$\cM(\cA)$~;
\item l'ensemble des parties de la forme
\[\bigcup_{\sigma \in \Sigma'} V'_{\sigma} \cup \bigcup_{\sigma \in \Sigma\setminus\Sigma'} V_{\sigma},\]
o\`u $\Sigma'$ est un sous-ensemble fini de~$\Sigma$ et, pour tout $\sigma \in \Sigma'$, $V'_{\sigma}$ est un voisinage de~$a_{0}$ dans~$V_{\sigma} \cup \{a_{0}\}$, est une base de voisinages de~$a_{0}$ dans~$V$.
\end{enumerate}


Soit $\sigma\in \Sigma$. Si $V \cap B^\arc \ne \emptyset$, on peut supposer que $V_{\sigma} \cap B^\arc \ne \emptyset$. La premi\`ere partie de la preuve assure que~$\oD_{V_{\sigma}}(\br)$ contient un pav\'e compact de dimension~$d+1$. Il en va donc de m\^eme pour~$\oD_{V}(\br)$. En particulier, d'apr\`es le th\'eor\`eme~\ref{dimension_sous-espace}, on a $\dimr(\oD_{V}(\br)) \ge d+1$.

\medbreak

D\'emontrons l'\'in\'egalit\'e r\'eciproque. Nous supposerons que~$\Sigma$ est infini et l'identifierons \`a~$\N$. Le cas fini se traite de fa\c{c}on similaire.

Pour tout $n \in \N$, $V_{n}$ est lin\'eaire, donc il existe une partie~$V_{n}^{\partial}$ de~$V$, un point~$b_{n}$ de~$V$ et un intervalle $I_{n} \subset I_{b_{n}}$ non r\'eduit \`a un point satisfaisant les propri\'et\'es suivantes~: 
\begin{enumerate}[i)]
\item $V_{n}^{\partial}$ est contenue dans le bord de~$V_{n}$ et $\sharp V_{n}^{\partial}\le 2$~; 
\item l'application 
\[\begin{array}{cccc}
p_{n} \colon &I_{n} & \too & V_{n}\setminus V_{n}^{\partial}\\
&\eps & \mapstoo & b_{n}^\eps
\end{array}\]
est un hom\'eomorphisme.
\end{enumerate}
\'Ecrivons~$I_{n}$ comme une r\'eunion d\'enombrable $\bigcup_{m\in \N} I_{n,m}$, o\`u les~$I_{n,m}$ sont des segments non r\'eduits \`a un point. 

Pour tout $n\in \N$, posons
\[ W_{n} := \{a_{0}\} \sqcup \bigsqcup_{m=0}^n \big(p_{m}(I_{m,n}) \sqcup V_{m}^\partial\big).\]
C'est une partie ferm\'ee de~$V$ qui est union disjointe finie de points et de parties lin\'eaires. D'apr\`es le th\'eor\`eme~\ref{th:dimensioncorps} et la premi\`ere partie de la preuve, on a $\dimr(\oD_{W_{n}}(\br)) \le d+1$. Or on a $\oD_{V}(\br) = \bigcup_{n\in \N} \oD_{W_{n}}(\br)$, donc, d'apr\`es le th\'eor\`eme~\ref{dimension_union_d\'enombrable}, $\dimr(\oD_{V}(\br)) \le d+1$. Ceci conclut la preuve.
\end{proof}

\index{Espace affine analytique!dimension d'un|)}\index{Disque!dimension d'un|)}

\chapter{Espaces de Stein}\label{chap:Stein}

Le but de ce chapitre est d'initier la th\'eorie des espaces de Stein sur un anneau de Banach en dimension sup\'erieure (le cas de la dimension~1 ayant d\'ej\`a fait l'objet de~\cite[chapitre~6]{A1Z}). Plus pr\'ecis\'ement, nous exhibons une famille d'espaces satisfaisant les conclusions des th\'eor\`emes~A et~B de Cartan~: la famille des polydisques ferm\'es (et de leurs ferm\'es analytiques). Rappelons que le th\'eor\`eme~A stipule que tout faisceau coh\'erent est engendr\'e par ses sections globales et le th\'eor\`eme~B que la cohomologie de tout faisceau coh\'erent est nulle en degr\'e strictement positif.\index{Theoreme@Th\'eor\`eme!A}\index{Theoreme@Th\'eor\`eme!B}

Dans la section~\ref{sec:CousinRunge}, nous d\'efinissons les notions de syst\`emes de Cousin et de Runge (\cf~d\'efinitions~\ref{def:Cousin} et~\ref{def:Runge}). Deux compacts \'etant donn\'es, il s'agit d'imposer des conditions reliant les fonctions sur l'intersection \`a des fonctions sur chacun des compacts. Dans ce cadre, nous d\'emontrons que si un faisceau est globalement engendr\'e sur chacun des compacts, il l'est encore sur leur r\'eunion (\cf~corollaire~\ref{cor:K-K+A}). Nous introduisons ensuite la notion, assez lourde, d'arbre de Cousin-Runge (\cf~d\'efinition~\ref{def:arbreCR}). Il s'agit d'un outil technique permettant de construire des syst\`emes de Cousin-Runge dans un disque relatif \`a partir de syst\`emes de Cousin-Runge sur la base. Il nous permettra d'effectuer des raisonnements par r\'ecurrence sur la dimension. 

Dans la section~\ref{sec:AB}, nous d\'emontrons les th\'eor\`emes~A et~B sur les polydisques ferm\'es relatifs (\cf~corollaires~\ref{cor:thA} et~\ref{cor:thB}). Pour ce faire, nous nous pla\c{c}ons sur une base que nous appelons de Stein (\cf~d\'efinition~\ref{def:basedeStein}). Les spectres de nos anneaux de Banach habituels (corps valu\'es, anneaux d'entiers de corps de nombres, corps hybrides, anneaux de valuation discr\`ete, anneaux de Dedekind trivialement valu\'es) en sont des exemples.

Dans la section~\ref{sec:affinoides}, nous d\'efinissons les espaces affino\"ides surconvergents (\cf~d\'efinition~\ref{def:affinoide}), par analogie avec la g\'eom\'etrie analytique rigide. De nombreuses propri\'et\'es restent valable dans notre cadre. Les exemples classiques, domaines de Weierstra\ss, domaines de Laurent et domaines rationnels s'adaptent sans peine (\cf~d\'efinition~\ref{def:domaines} et proposition~\ref{prop:domaines}). Nous d\'emontrons \'egalement des analogues du th\'eor\`eme d'acyclicit\'e de Tate et du th\'eor\`eme de Kiehl (\cf~th\'eor\`emes~\ref{th:sectionsglobalesaffinoide} et~\ref{th:affinoideAB}). Il s'agit de cons\'equences assez directes de nos th\'eor\`emes~A et~B sur des polydisques qu'il nous semble utile d'\'enoncer explicitement pour renforcer le parall\`ele avec la th\'eorie ultram\'etrique classique. 

Dans la courte section \ref{sec:complements}, nous utilisons les th\'eor\`emes~A et~B pour raffiner quelques r\'esultats ant\'erieurs, en particulier celui ayant trait \`a la fermeture des id\'eaux (\cf~corollaire~\ref{cor:limiteOU}). 

Dans la section~\ref{sec:Bouvert}, nous d\'emontrons que les polydisques ouverts relatifs, et certains espaces de nature similaire, satisfont le th\'eor\`eme~B (\cf~corollaire~\ref{cor:Bouvertdisque}). Nous suivons une strat\'egie classique~: exhaustion par des polydisques ferm\'es relatifs et passage \`a la limite. Le r\'esultat de fermeture des id\'eaux de la section pr\'ec\'edente y joue un r\^ole essentiel.

Dans la section finale~\ref{sec:noetherianite}, nous appliquons le r\'esultat d'annulation cohomologique sur les polydisques ferm\'es \`a l'\'etude des s\'eries arithm\'etiques convergentes, c'est-\`a-dire de s\'eries \`a coefficients entiers qui convergent sur un polydisque complexe. Dans~\cite{HarbaterConvergent}, D.~Harbater a d\'emontr\'e que certains anneaux naturels de s\'eries arithm\'etiques convergentes en une variable sont noeth\'eriens. Nous \'etendons son r\'esultat \`a des s\'eries en un nombre quelconque de variables (\cf~corollaire~\ref{cor:noetherienconcret}).

\medbreak

Soit $(\cA,\nm)$ un anneau de base g\'eom\'etrique. On pose $B := \cM(\cA)$. Pr\'ecisons que l'hypoth\`ese sur l'anneau ne sera peu utilis\'ee. Toute la section~\ref{sec:CousinRunge} reste, par exemple, valable pour un anneau de Banach arbitraire. Dans les sections~\ref{sec:AB} et \ref{sec:affinoides}, l'hypoth\`ese interviendra lorsque nous aurons besoin de la coh\'erence du faisceau structural, et uniquement pour cette raison. Dans les sections~\ref{sec:complements} et~\ref{sec:Bouvert}, en revanche, elle jouera un r\^ole important lorsqu'interviendra la fermeture des id\'eaux du faisceau structural, et il faudra m\^eme la compl\'eter par une hypoth\`ese suppl\'ementaire (\cf~d\'efinition~\ref{def:Badapte}). Nous avons choisi de supposer d\`es le d\'ebut que l'anneau de Banach $(\cA,\nm)$ est anneau de base g\'eom\'etrique afin de ne pas alourdir la r\'edaction.


\medbreak

Dans ce chapitre, nous travaillerons souvent avec des parties compactes d'espaces $\cA$-analytiques. Elles seront munies du faisceau structural surconvergent (\cf~notation~\ref{nota:surconvergent}), sans que nous le pr\'ecisions d\'esormais. 


\section{Syst\`emes de Cousin-Runge}\label{sec:CousinRunge}

\subsection{G\'en\'eralit\'es}\label{sec:CousinRungeGeneralites}
\index{Systeme@Syst\`eme|(}

Soit~$X$ un espace $\cA$-analytique.

Soient~$K^-$ et~$K^+$ deux parties compactes de~$X$. Posons $L := K^- \cap K^+$ et $M := K^- \cup K^+$. Dans cette section, nous introduisons diff\'erentes notions permettant de relier les sections globales d'un faisceau coh\'erent sur~$K^-$ et~$K^+$ \`a celles sur~$M$. Elles sont inspir\'ees de celles introduites dans \cite[\S 6.2.1]{A1Z}, tout en \'etant parfois diff\'erentes, de fa\c con \`a prendre en compte le cadre plus g\'en\'eral \'etudi\'e ici.%
\nomenclature[La]{$K^-$, $K^+$}{compacts d'un espace $\cA$-analytique}%
\nomenclature[Lb]{$L$}{intersection $K^-\cap K^+$ des compacts $K^-$ et $K^+$}%
\nomenclature[Lc]{$M$}{union $K^-\cup K^+$ des compacts $K^-$ et $K^+$}%

Sous nos conventions, les anneaux $\cO(K^+)$, $\cO(K^-)$ et $\cO(L)$ sont des anneaux de sections surconvergentes (\cf~notation~\ref{nota:surconvergent}). En g\'en\'eral, ce ne sont pas des anneaux de Banach, mais on peut chercher \`a les \'ecrire comme limites de suites de tels anneaux. On peut par exemple \'ecrire les sections sur un disque \`a l'aide des compl\'et\'es des sections sur des disques de rayons strictement plus grands, dans l'esprit de la proposition~\ref{prop:disqueglobal}. La notion de syst\`eme de Banach vient formaliser cette id\'ee, en imposant de plus une condition de compatibilit\'e entre les suites.

Nous travaillerons ici, sans plus le pr\'eciser, avec la cat\'egorie des anneaux de Banach dont les objets sont les anneaux de Banach et les morphismes sont les morphismes born\'es. 

\begin{defi}\label{def:systemedeBanachfin}\index{Systeme@Syst\`eme!de Banach|textbf}\index{Systeme@Syst\`eme!de Banach fort|textbf}%
\nomenclature[Ld]{$(\cB_{m}^-,\nm_{m}^-)$}{famille d'espaces de Banach associ\'ee \`a~$K^-$ dans un syst\`eme de Banach}%
\nomenclature[Le]{$(\cB_{m}^+,\nm_{m}^+)$}{famille d'espaces de Banach associ\'ee \`a~$K^+$ dans un syst\`eme de Banach}%
\nomenclature[Lf]{$(\cC_{m}^-,\nm_{m})$}{famille d'espaces de Banach associ\'ee \`a~$L$ dans un syst\`eme de Banach}%
Un \emph{syst\`eme de Banach} associ\'e \`a~$(K^-,K^+)$ dans~$X$ est la donn\'ee de syst\`emes inductifs d'anneaux de Banach
\begin{align*}
\cB^- &= ((\cB_{m}^-,\nm_{m}^-)_{m\in \N},(\varphi^-_{m',m})_{m' \ge m\, \in \N}),\\
\cB^+ &= ((\cB_{m}^+,\nm_{m}^+)_{m\in \N},(\varphi^+_{m',m})_{m' \ge m\, \in \N}),\\
\cC &= ((\cC_{m},\nm_{m})_{m\in \N},(\varphi_{m',m})_{m' \ge m\, \in \N})
\end{align*}
et de morphismes
\[(\psi^-_{m})_{m\in\N} \colon \cB^- \to \cC, \ (\psi^+_{m})_{m\in\N} \colon \cB^+ \to \cC\]
et 
\[(\rho^-_{m})_{m\in\N} \colon \cB^- \to \cO(K^-), \ (\rho^+_{m})_{m\in\N} \colon \cB^+ \to \cO(K^+), \ (\rho_{m})_{m\in\N} \colon \cC \to \cO(L)\]
tels que le morphisme
\[ \colim_{m\in \N} \cC_{m} \too \cO(L) \]
induit par les~$\rho_{m}$ soit un isomorphisme d'anneaux.

On dit que le syst\`eme de Banach est \emph{fort} si, de plus, pour tout voisinage compact~$V$ de~$L$, il existe $m_{V} \in\N$ satisfaisant la propri\'et\'e suivante~: pour tout $m\ge m_{V}$, il existe $C_{m} \in \R$ tel que, pour tout $f\in \cO(V)$, il existe $g\in \cC_{m}$ v\'erifiant
\[\begin{cases}
f = \rho_{m}(g) \textrm{ dans } \cO(L)~;\\
\|g\|_{m} \le C_{m}\, \|f\|_{V}.
\end{cases}\]
\end{defi}

Dans la suite, nous nous permettrons d'utiliser implicitement les morphismes $\psi^-_{m}, \psi^+_{m}, \rho^-_{m}, \rho^+_{m}, \rho_{m}$, sans que cela ne soit source de confusions. Nous sous-entendrons \'egalement souvent l'espace~$X$ ambiant.

\begin{rema}
La d\'efinition de syst\`eme de Banach reprend~\cite[d\'efinition~6.2.1]{A1Z}. La d\'efinition de syst\`eme de Banach fort est nouvelle.
\end{rema}

Introduisons la notion de syst\`eme de Cousin. L'id\'ee sous-jacente consiste \`a \'ecrire une section sur l'intersection~$L$ comme somme de sections sur les parties~$K^-$ et~$K^+$, avec un contr\^ole sur les normes.

\begin{defi}\label{def:Cousin}\index{Systeme@Syst\`eme!de Cousin|textbf}
Un \emph{syst\`eme de Cousin} associ\'e \`a~$(K^-,K^+)$ est un syst\`eme de Banach associ\'e \`a~$(K^-,K^+)$ pour lequel il existe $D \in\R$ satisfaisant la propri\'et\'e suivante~: pour tous $m\in\N$ et $f \in \cC_{m}$, il existe $f^-\in\cB_{m}^-$ et $f^+\in\cB_{m}^+$ tels que
\begin{enumerate}[i)]
\item $f = f^- + f^+$ dans $\cC_{m}$~;
\item $\|f^-\|^-_{m} \le D\, \|f\|_{m}$~;
\item $\|f^+\|^+_{m} \le D\, \|f\|_{m}$.
\end{enumerate}
\end{defi}

\begin{rema}
La d\'efinition de syst\`eme de Cousin reprend~\cite[d\'efinition~6.2.2]{A1Z}. 
\end{rema}

Donnons quelques exemples simples pour illustrer la d\'efinition.

\begin{exem}\label{ex:ZCousinum}\index{Anneau!des entiers relatifs $\Z$}
Consid\'erons le cas o\`u $\cA = \Z$ et $X=\cM(\Z)$. Reprenons les notations de l'exemple~\ref{ex:Z}. 

Soit $q$ un nombre premier et soit $\alpha \in \R_{>0}$. Posons $K^- := [a_{q}^\alpha,a_{q}^{+\infty}]$ et  $K^+ := \cM(\Z) \setminus \intof{a_{q}^\alpha,a_{q}^{+\infty}}$. On a $K^-\cap K^+ = \{a_{q}^\alpha\}$ et $K^-\cup K^+ = \cM(\Z)$. Cette situation est repr\'esent\'ee \`a la figure~\ref{fig:K-K+}. 

\begin{figure}[!h]
\centering
\begin{tikzpicture}
\foreach \x [count=\xi] in {-2,-1,...,17}
\draw (0,0) -- ({10*cos(\x*pi/10 r)/\xi},{10*sin(\x*pi/10 r)/\xi}) ;
\foreach \x [count=\xi] in {-2,-1,...,17}
\fill ({10*cos(\x*pi/10 r)/\xi},{10*sin(\x*pi/10 r)/\xi}) circle ({0.07/(sqrt(\xi)}) ;




\fill (2.9,0) circle ({0.07/(sqrt(3)}) ;
\draw (2.6,-0.3) node{$a_q^\alpha$} ;

\draw (2.9,0) to[out=100,in=0] (0,1.7);
\draw (0,1.7) to[out=180,in=90] (-1.4,0);
\draw (-1.4,0) to[out=270,in=135] (-0.8,-1);
\draw (-0.8,-1) to[out=-45,in=160] (8,-6.4);
\draw (8,-6.4) to[out=-20,in=-70] (8.7,-5.7);
\draw (8.7,-5.7) to[out=110,in=-45] (5.5,-1.5);
\draw (5.5,-1.5) to[out=135,in=-65] (2.9,0);

\draw (2.9,0) to[out=45,in=90] (3.7,0);
\draw (3.7,0) to[out=-90,in=-45] (2.9,0);

\draw (4.1,0.4) node{$K^-$} ;
\draw (-.8,-1.6) node{$K^+$} ;

\end{tikzpicture}
\caption{Un syst\`eme de Cousin de $\cM(\Z)$.}\label{fig:K-K+}
\end{figure}

Pour tout $m\in \N$, posons 
\[\begin{cases}
(\cB^-_{m},\nm^-_{m}) := (\cB(K^-),\nm_{K^-}) = (\Z_{q}, \va_{q}^\alpha)~;\\[2pt]  
(\cB^+_{m},\nm^+_{m}) := (\cB(K^+),\nm_{K^+}) = \big(\Z\big[\frac1q\big],\max(\va_{q}^\alpha,\va_{\infty})\big)~;\\[2pt]
(\cC_{m},\nm_{m}) := (\cH(a_{q}^\alpha), \va_{a_{q}^\alpha}) = (\Q_{q},\va_{q}^\alpha).
\end{cases}\]
Soit $f\in \cH(a_{q}^\alpha) = \Q_{q}$. Supposons que $f\ne 0$ et \'ecrivons-le sous la forme $f = \sum_{i\ge i_{0}} a_{i} \, q^i$ avec $i_{0} \in \Z$, $a_{i_{0}} \ne 0$ et, pour tout $i \ge i_{0}$, $a_{i} \in \cn{0}{q-1}$. 

Si $i_{0} \ge 0$, on pose 
\[ f^- := f \textrm{ et } f^+ := 0.\]
On a alors $f = f^-+f^+$, $\|f^-\|_{K^-} = |f|_{q}^\alpha$ et $\|f\|_{K^+} = 0$.

Si $i_{0} < 0$, on a $|f|_{q}^\alpha >1$ et on pose 
\[ f^- := \sum_{i\ge 0} a_{i} \, q^i \textrm{ et } f^+ := \sum_{i_{0}\le i <0} a_{i} \, q^i.\]
On a alors $f = f^-+f^+$, $\|f^-\|_{K^-} \le 1$ et $\|f\|_{K^+}  \le \sum_{i<0} (q-1)\, q^{-i} = 1$.

On a donc bien d\'efini un syst\`eme de Cousin associ\'e \`a $(K^-,K^+)$.
\end{exem}

\begin{exem}\label{ex:ZCousinarc}\index{Anneau!des entiers relatifs $\Z$}
Consid\'erons le cas o\`u $\cA = \Z$ et $X=\cM(\Z)$. Reprenons les notations de l'exemple~\ref{ex:Z}. 

Soit $\beta \in \intoo{0,1}$. Posons $K^- := [a_{\infty}^\beta,a_{\infty}]$ et  $K^+ := \cM(\Z) \setminus \intof{a_{\infty}^\beta,a_{\infty}}$. On a $K^-\cap K^+ = \{a_{\infty}^\beta\}$ et $K^-\cup K^+ = \cM(\Z)$.  

Pour tout $m\in \N$, posons 
\[\begin{cases}
(\cB^-_{m},\nm^-_{m}) := (\cB(K^-),\nm_{K^-}) = (\R, \max(\va_{\infty}^\beta,\va_{\infty}))~;\\
(\cB^+_{m},\nm^+_{m}) := (\cB(K^+),\nm_{K^+}) = (\Z,\va_{\infty}^\beta)~;\\
(\cC_{m},\nm_{m}) := (\cH(a_{\infty}^\beta), \va_{a_{\infty}^\beta}) = (\R,\va_{\infty}^\beta).
\end{cases}\]
Soit $f\in \cH(a_{\infty}^\beta) = \R$. 

Si $|f|_{\infty} \le 1$, on pose
\[ f^- := f \textrm{ et } f^+ := 0.\]
On a alors $f = f^-+f^+$, $\|f^-\|_{K^-} = |f|_{\infty}^\beta$ et $\|f\|_{K^+} = 0$.

Si $|f|_{\infty} > 1$, il existe $n\in \Z$ tel que $|n|_{\infty} \le |f|_{\infty}$ et $|f-n|_{\infty} < 1$ et on pose
\[ f^- := f-n \textrm{ et } f^+ := n.\]
On a alors $f = f^-+f^+$, $\|f^-\|_{K^-} \le 1$ et $\|f\|_{K^+} = |n|_{\infty}^\beta \le |f|_{\infty}^\beta$.

On a donc bien d\'efini un syst\`eme de Cousin associ\'e \`a $(K^-,K^+)$.
\end{exem}

Les exemples pr\'ec\'edents se g\'en\'eralisent.

\begin{exem}\label{ex:Cousin}
\index{Corps!valu\'e}
\index{Anneau!des entiers d'un corps de nombres}\index{Corps!hybride}\index{Anneau!de valuation discr\`ete}\index{Anneau!de Dedekind trivialement valu\'e}
Soit $\cA$ l'un de nos anneaux de Banach usuels~: corps valu\'e, anneau d'entiers de corps de nombres, corps hybride, anneau de valuation discr\`ete, anneau de Dedekind trivialement valu\'e (\cf~exemples~\ref{ex:corpsvalue} \`a~\ref{ex:Dedekind}). Posons $X := \cM(\cA)$.

Soit~$a$ un point de~$\cM(\cA)$ non associ\'e \`a la valuation triviale sur~$\cA$. Dans ce cas, il existe deux compacts~$K^-$ et~$K^+$ de~$\cM(\cA)$ tels que $K^- \cap K^+ = \{a\}$ et $K^- \cup K^+ = \cM(\cA)$. (Si $\cM(\cA) \setminus \{a\}$ poss\`ede deux composantes connexes, ce qui est le cas g\'en\'eral, $K^-$ et~$K^+$ sont les adh\'erences de ces composantes connexes.) 

On obtient alors un syst\`eme de Cousin fort associ\'e \`a~$(K^-,K^+)$ en posant, pour tout $m\in \N$, $(\cB^-_{m},\nm^-_{m}) := (\cB(K^-),\nm_{K^-})$,  $(\cB^+_{m},\nm^+_{m}) := (\cB(K^+),\nm_{K^+})$ et  $(\cC_{m},\nm_{m}) := (\cH(a),\va_{a})$. Le seul cas qui n'est pas imm\'ediat est celui des anneaux d'entiers de corps de nombres. On peut le d\'eduire du th\'eor\`eme des unit\'es de Dirichlet (\cf~\cite[lemme~6.3.2]{A1Z}).
\end{exem}

Introduisons la notion de syst\`eme de Runge. L'id\'ee sous-jacente consiste \`a \'ecrire les sections sur l'intersection \`a partir de sections globales, et plus pr\'ecis\'ement d'approcher une section sur l'intersection~$L$ par une section sur une partie ($K^-$ ou~$K^+$), apr\`es multiplication \'eventuelle par une section inversible sur l'autre ($K^+$ ou~$K^-$).

\begin{defi}\label{def:Runge}\index{Systeme@Syst\`eme!de Runge|textbf}
Un \emph{syst\`eme de Runge} associ\'e \`a~$(K^-,K^+)$ est un syst\`eme de Banach associ\'e \`a~$(K^-,K^+)$ tel que, pour tous $m \in \N$, $\eps >0$, $p\in\N^\ast$, $s_{1},\dotsc,s_{p}\in \cC_{m}$, les propri\'et\'es suivantes soient v\'erifi\'ees~: pour tout $\sigma\in \{-,+\}$, il existe $f \in \cB_{m}^\sigma$ inversible et $s'_{1},\dotsc,s'_{p}\in \cB_{m}^{-\sigma}$ tels que
\begin{enumerate}[i)]
\item $\lim_{m'\to +\infty} \|f\|_{m'} \, \|f^{-1}\|_{m'} = 1$~;
\item pour tout $i\in\cn{1}{p}$, on ait
\[\|fs_{i} - s'_{i}\|_{m} < \eps \|f\|_{m}.\]
\end{enumerate}
\end{defi}

\begin{rema}\label{rem:CR}
Dans le cadre de la d\'efinition pr\'ec\'edente, pour $m$ assez grand (d\'ependant de~$f$), on a 
\[\|fs_{i} - s'_{i}\|_{m} \, \|f^{-1}\|_{m} < \eps.\]
\end{rema}

Nous utiliserons souvent des combinaisons des d\'efinitions pr\'ec\'edentes dont le sens est clair~: syst\`eme de Cousin fort, syst\`eme de Cousin-Runge, syst\`eme de Cousin-Runge fort, etc.

\begin{rema}
Dans~\cite{A1Z} figure uniquement la d\'efinition de syst\`eme de Cousin-Runge (\cf~\cite[d\'efinition~6.2.8]{A1Z}). Elle est proche de celle que nous proposons ici, mais plus compliqu\'ee. 
\end{rema}

Reprenons les exemples de syst\`emes de Banach vus plus haut.

\begin{exem}\label{ex:ZCousinRungeum}\index{Anneau!des entiers relatifs $\Z$}
Reprenons l'exemple~\ref{ex:ZCousinum}. Soit $\eps \in \R_{>0}$ et soient $s_{1},\dotsc,s_{p} \in \cH(a_{q}^\alpha) = \Q_{q}$. 

Consid\'erons le cas $\sigma = -$. Puisque $\Z[1/q]$ est dense dans~$\Q_{q}$, il existe $s'_{1},\dotsc,s'_{p} \in \Z[1/q]$ tel que, pour tout $i\in \cn{1}{p}$, on ait $|s_{i} - s'_{i}| < \eps$. Le r\'esultat vaut donc avec $f=1$.

Consid\'erons le cas $\sigma = +$. Il existe $N \in \N$ tel que, pour tout $i\in \cn{1}{p}$, on ait $q^N s_{i} \in \Z_{q}$. Le r\'esultat vaut donc avec $f=q^N$ (qui est inversible dans~$\Z[1/q]$) et, pour tout $i\in \cn{1}{p}$, $s'_{i} = q^N s_{i}$.

On en d\'eduit que le syst\`eme de l'exemple~\ref{ex:ZCousinum} est de Runge.
\end{exem}

\begin{exem}\label{ex:ZCousinRungearc}\index{Anneau!des entiers relatifs $\Z$}
Reprenons l'exemple~\ref{ex:ZCousinarc}. Soit $\eps \in \R_{>0}$ et soient $s_{1},\dotsc,s_{p} \in \cH(a_{\infty}^\beta) = \R$. 

Consid\'erons le cas $\sigma = +$. Les anneaux sous-jacents \`a~$\cH(a_{\infty}^\beta)$ et $\cB(K^-)$ co\"incident et le r\'esultat vaut donc avec $f=1$ et, pour tout $i\in \cn{1}{p}$, $s'_{i} = s_{i}$.

Consid\'erons le cas $\sigma = -$. Il existe $M \in \N_{\ge 1}$ tel que $\eps \ge 1/M^\beta$. Pour tout $i\in \cn{1}{p}$, il existe $s'_{i} \in \Z$ tel que $|M s_{i} - s'_{i}|_{\infty} < 1$. Le r\'esultat vaut alors avec $f=M$ (qui est inversible dans~$\R$).

On en d\'eduit que le syst\`eme de l'exemple~\ref{ex:ZCousinarc} est de Runge.
\end{exem}

\begin{exem}\label{ex:CousinRunge}
\index{Corps!valu\'e}
\index{Anneau!des entiers d'un corps de nombres}\index{Corps!hybride}\index{Anneau!de valuation discr\`ete}\index{Anneau!de Dedekind trivialement valu\'e}

Les syst\`emes de Cousin de l'exemple~\ref{ex:Cousin} sont des syst\`emes de Cousin-Runge forts. Comme pr\'ec\'edemment, le seul cas difficile est celui des anneaux d'entiers de corps de nombres. On peut le d\'eduire du th\'eor\`eme d'approximation fort (\cf~\cite[lemme~6.3.3]{A1Z}).
\end{exem}

Rappelons la notation~\ref{nota:disquerelatif} pour les disques relatifs.

\begin{lemm}\label{lem:CRrelatif}\index{Systeme@Syst\`eme!sur un disque}\index{Disque!systeme@syst\`eme sur un|see{Syst\`eme}}
Soient $r_{1},\dotsc,r_{n} \in \R_{\ge 0}$. Soient~$K^-$ et~$K^+$ des parties compactes de~$X$. S'il existe un syst\`eme de Banach (resp. de Banach fort, resp. de Cousin, resp. de Runge) associ\'e \`a $(K^-,K^+)$, alors il existe un syst\`eme de Banach (resp. de Banach fort, de Cousin, resp. de Runge) associ\'e \`a $(\overline{D}_{K^-}(r_{1},\dotsc,r_{n}),\overline{D}_{K^+}(r_{1},\dotsc,r_{n}))$ dans $\E{n}{X}$.
\end{lemm}
\begin{proof}
Une r\'ecurrence imm\'ediate montre qu'il suffit de traiter le cas $n=1$. Posons $r=r_{1}$. 

Supposons tout d'abord qu'il existe un syst\`eme de Banach~$\Omega$ associ\'e au couple $(K^-,K^+)$. Soit $(R_{n})_{n\in\N}$ une suite d\'ecroissante d'\'el\'ements de $\intof{r,+\infty}$ de limite~$r$. Pour $n \in \N$, posons $\bar\cB_{n}^- = \cB_{n}^-\la |T|\le R_{n}\ra$, $\bar\cB_{n}^+ = \cB_{n}^+\la |T|\le R_{n}\ra$, $\bar\cC_{n} = \cC_{n}\la |T|\le R_{n}\ra$. Avec les morphismes \'evidents, on obtient un syst\`eme de Banach~$\bar \Omega$ associ\'e au couple $(\overline{D}_{K^-}(r),\overline{D}_{K^+}(r))$.
En effet, les autres propri\'et\'es \'etant imm\'ediates, il suffit de v\'erifier que le morphisme
\[\colim_{m\in\N} \bar\cC_{m} \too \cO(\overline{D}_{L}(r))\]
est un isomorphisme, ce qui d\'ecoule de la proposition~\ref{prop:disqueglobal}.

\smallbreak

Supposons, maintenant, que~$\Omega$ est un syst\`eme de Banach fort. Soit~$V$ un voisinage compact de~$\overline{D}_{L}(r)$ dans~$\E{1}{X}$. On peut supposer qu'il est de la forme $\overline{D}_{W}(s)$, o\`u~$W$ est un voisinage compact de~$L$ et~$s$ un nombre r\'eel strictement sup\'erieur \`a~$r$. Soit~$m_{W}$ un entier v\'erifiant les propri\'et\'es de la d\'efinition de syst\`eme de Banach fort pour le voisinage compact~$W$ de~$L$. Soit~$m_{V} \ge m_{W}$ tel que $R_{m_{V}} < s$. Soient $m\ge m_{V}$ et $f\in\cO(\overline{D}_{W}(s))$. D'apr\`es la proposition~\ref{prop:disqueglobal}, on peut \'ecrire~$f$ sous la forme 
\[f = \sum_{n\ge 0} a_{n}\, T^n \in \cO(W)\llbracket T \rrbracket,\]
o\`u la s\'erie $\sum_{n\ge 0} \|a_{n}\|_{W}\, s^n$ converge. Pour tout $n\in\N$, il existe $b_{n} \in \cC_{m}$ tel que 
\[\begin{cases}
a_{n} = b_{n} \textrm{ dans } \cO(L)~;\\
\|b_{n}\|_{m} \le C_{m}\, \|a_{n}\|_{W}.
\end{cases}\]
Consid\'erons la s\'erie
\[g = \sum_{n\ge 0} b_{n}\, T^n \in \cC_{m}\llbracket T \rrbracket.\]
Elle appartient \`a $\cC_{m}\la |T|\le s\ra$, donc \`a $\cC_{m}\la |T|\le R_{m}\ra$, et son image dans $\cO(\overline{D}_{L}(r))$ co\"{\i}ncide avec celle de~$f$. D'apr\`es la proposition~\ref{prop:restrictionserie}, on a 
\begin{align*}
\|g\|_{\cC_{m}\la |T|\le R_{m}\ra} &=  \sum_{n\ge 0} \|b_{n}\|_{m}\, R_{m}^n\\
&\le C_{m}\,  \sum_{n\ge 0} \|b_{n}\|_{W}\, R_{m}^n\\
&\le C_{m} \frac{s}{s-R_{m}}\, \|f\|_{\overline{D}_{W}(s)}.
\end{align*}
Ceci montre que~$\bar\Omega$ est un syst\`eme de Banach fort.

\smallbreak

Supposons, maintenant, que~$\Omega$ est un syst\`eme de Cousin. Par hypoth\`ese, il existe $D\in \R$ tel que, pour tous $m\in\N$ et $a\in \cC_{m}$, il existe $a^-\in \cB_{m}^-$ et $a^+\in \cB_{m}^+$ tels que
\[\begin{cases}
a = a^- + a^+ \textrm{ dans } \cC_{m}~;\\
\|a^-\|^-_{m} \le D\, \|a\|_{m}~;\\
\|a^-\|^+_{m} \le D\, \|a\|_{m}.
\end{cases}\]
Soient $m\in\N$ et $f \in \bar\cC_{m} = \cC_{m}\la |T|\le R_{m}\ra$. On peut l'\'ecrire sous la forme
\[f = \sum_{n\ge 0} a_{n}\, T^n.\]
Les s\'eries
\[f^- := \sum_{n\ge 0} a_{n}^-\, T^n \textrm{ et } f^+ := \sum_{n\ge 0} a_{n}^+\, T^n\]
d\'efinissent alors respectivement des \'el\'ements de~$\bar\cB_{m}^-$ et~$\bar\cB_{m}^+$ tels que
\[\begin{cases}
f = f^- + f^+ \textrm{ dans } \bar\cC_{m}~;\\
\|f^-\|^-_{m} \le D\, \|f\|_{m}~;\\
\|f^-\|^+_{m} \le D\, \|f\|_{m}.
\end{cases}\]
Ceci montre que~$\bar\Omega$ est un syst\`eme de Cousin.

\smallbreak

Supposons, finalement, que~$\Omega$ est un syst\`eme de Runge. Soient $m\in\N$, $\eps>0$ et $\sigma\in\{-,+\}$. Soient $p\in\N^\ast$ et $s_{1},\dotsc,s_{p}\in \bar\cC_{m}$. 

Pour tout $i\in\cn{1}{p}$, il existe 
\[ s'_{i} = \sum_{n=0}^{d_{i}} a'_{i,n} \, T^n \in \cC_{m}[T] \]
tel que $\|s'_{i}-s_{i}\|_{m} < \eps/2$. On peut supposer que les~$d_{i}$ sont tous \'egaux. Notons~$d$ leur valeur commune. Posons $C = \sum_{n=0}^d R_{m}^n$. D'apr\`es la propri\'et\'e de Runge pour~$\cC_{m}$, il existe $f\in \cB_{m}^\sigma$ inversible et, pour tout $(i,n) \in \cn{1}{p} \times \cn{0}{d}$, $a''_{i,n} \in \cB_{m}^{-\sigma}$ tels que 
\begin{enumerate}[i)]
\item $\lim_{m'\to +\infty} \|f\|_{m'} \, \|f^{-1}\|_{m'} = 1$~;
\item pour tout $(i,n) \in \cn{1}{p} \times \cn{0}{d}$, on ait
\[\|fa'_{i,n} - a''_{i,n}\|_{m} < \eps/(2C)\, \|f\|_{m}.\]
\end{enumerate}

Pour tout $i\in\cn{1}{p}$, posons 
\[ s''_{i} = \sum_{n=0}^{d_{i}} a''_{i,n} \, T^n \in \cC_{m}[T]. \]
On a alors 
\[ \|fs_{i} - s''_{i}\|_{\bar{\cC}_{m}} \le \|f(s_{i} - s'_{i})\|_{\bar{\cC}_{m}}+ \|fs'_{i} - s''_{i}\|_{\bar{\cC}_{m}} \le \eps \|f\|_{\bar{\cC}_{m}}.\]

De plus, pour tout $m'\ge m$ et tout \'el\'ement~$g$ de~$\cC_{m}$, on a $\|g\|_{\bar{\cC}_{m'}} = \|g\|_{m'}$. On en d\'eduit que 
\[\lim_{m'\to +\infty} \|f\|_{\bar{\cC}_{m'}} \, \|f^{-1}\|_{\bar{\cC}_{m'}} = 1.\]
Ceci montre que~$\bar\Omega$ est un syst\`eme de Runge.
\end{proof}

Rappelons que, dans le cadre des syst\`emes de Cousin, on dispose de l'analogue du lemme de Cartan.

\begin{theo}[\protect{\cite[th\'eor\`eme~6.2.7]{A1Z}}]\label{thm:lemmeCartan}\index{Lemme!de Cartan}\index{Theoreme@Th\'eor\`eme!de Cartan}
Soit $\Omega = (\cB^-_{m},\cB^+_{m},\cC_{m})_{m\in\N}$ un syst\`eme de Cousin.
Alors, il existe $\varepsilon>0$ tel que, pour tous $m\in\N$, $q\in\N^\ast$ et $A\in M_{q}(\cC_{m})$ v\'erifiant $\|A-I\|_{m} < \varepsilon$, il existe $C^- \in GL_{q}(\cB_{m}^-)$ et $C^+ \in GL_{q}(\cB_{m}^+)$ telles que
\[\begin{cases}
A = C^-\, C^+ \textrm{ dans } GL_{q}(\cC_{m})~;\\
\|C^--I\|_{m}^- \le 4 D\, \|A-I\|_{m}~;\\
\|C^+-I\|_{m}^+ \le 4 D\, \|A-I\|_{m},\\
\end{cases}\]
o\`u~$D$ est la constante apparaissant dans la d\'efinition de syst\`eme de Cousin.
\qed
\end{theo}

Nous souhaitons maintenant d\'emontrer que, pour un faisceau coh\'erent sur $M = K^-\cup K^+$, la propri\'et\'e d'\^etre globalement engendr\'e passe de~$K^-$ et~$K^+$ \`a~$M$. Le lemme qui suit est une version l\'eg\`erement modifi\'ee de \cite[lemme~6.2.9]{A1Z}, dont la preuve reste essentiellement inchang\'ee.

\begin{lemm}\label{lem:matriceapprochee}
Soit $\Omega = (\cB^-_{m},\cB^+_{m},\cC_{m})_{m\in\N}$ un syst\`eme de Cousin associ\'e \`a $(K^-,K^+)$. Soient~$\cF$ un faisceau de $\cO_{M}$-modules, $p,q\in \N^\ast$, $T^- \in \cF(K^-)^p$, $T^+ \in \cF(K^+)^q$, 
$U\in M_{p,q}(\cO(\cC_{m}))$ et $V \in M_{q,p}(\cO(\cC_{m}))$
tels que qu'on ait
\[\begin{cases}
T^- = U\,T^+ \textrm{ dans } \cF(L)^p~;\\
T^+=V\,T^- \textrm{ dans } \cF(L)^q.
\end{cases}\] 
Soit $\eps>0$ satisfaisant la conclusion du th\'eor\`eme~\ref{thm:lemmeCartan}. Supposons qu'il existe $m\in \N$ et $\bar U \in M_{p,q}(\cB_{m}^+)$ tels que
\[ \|\bar U - U\|_{m}\, \|V\|_{m} < \varepsilon.\]
Alors, il existe $S^-\in \cF(M)^p$ et $A^- \in GL_{p}(\cO(K^-))$ tels que
\[ S^- = A^-\, T^- \textrm{ dans } \cF(K^-)^p.\]
\end{lemm}
\begin{proof}
Posons $A = I + (\bar U -U) V \in M_{p}(\cC_{m})$. Nous avons alors $A\, T^- = \bar U\, T^+$ dans $\cF(L)^p$.

Nous avons \'egalement $\|A-I\|_{m} < \eps$, donc, par hypoth\`ese, il existe $C^- \in GL_{q}(\cB_{m}^-)$ et $C^+ \in GL_{q}(\cB_{m}^+)$ telles que $A = C^+\, C^-$ dans $GL_{p}(\cC_{m})$. On a alors 
\[C^-\, T^- = (C^+)^{-1}\, A\, T^- = (C^+)^{-1}\, \bar U\, T^+.\]
On peut donc d\'efinir un \'el\'ement~$S^-$ de~$\cF(M)^p$ par
\[\begin{cases}
S^- = C^-\, T^- \textrm{ sur } K^-~;\\
S^- = (C^+)^{-1}\, \bar U\, T^+ \textrm{ sur } K^+.
\end{cases}\]
Il v\'erifie la propri\'et\'e requise avec la matrice~$A^-$ qui est l'image de~$C^-$ dans $GL_{p}(\cO(K^-))$.
\end{proof}

Nous en d\'eduisons un analogue de \cite[th\'eor\`eme~6.2.10]{A1Z}.

\begin{theo}\index{Faisceau!globalement engendr\'e}
Soit $\Omega = (\cB^-_{m},\cB^+_{m},\cC_{m})_{m\in\N}$ un syst\`eme de Cousin-Runge associ\'e \`a $(K^-,K^+)$. Soit~$\cF$ un faisceau de $\cO_{M}$-modules. Supposons qu'il existe deux entiers~$p$ et~$q$, une famille $(t_{1}^-,\dotsc,t_{p}^-)$ d'\'el\'ements de~$\cF(K^-)$ et une famille $(t_{1}^+,\dotsc,t_{q}^+)$ d'\'el\'ements de~$\cF(K^+)$ dont les restrictions \`a~$L$ engendrent le m\^eme sous-$\cO(L)$-module de~$\cF(L)$. Alors, il existe $s_{1}^-,\dotsc,s_{p}^-,s_{1}^+,\dotsc,s_{q}^+ \in \cF(M)$, $A^- \in GL_{p}(\cO(K^-))$ et $A^+ \in GL_{q}(\cO(K^+))$ tels que 
\[\begin{pmatrix} s_{1}^-\\ \vdots \\ s_{p}^- \end{pmatrix} = A^- \begin{pmatrix} t_{1}^-\\ \vdots \\ t_{p}^- \end{pmatrix} \textrm{ dans } \cF(K^-)^p\]  
et
\[\begin{pmatrix} s_{1}^+\\ \vdots \\ s_{q}^+ \end{pmatrix} = A^+ \begin{pmatrix} t_{1}^+\\ \vdots \\ t_{q}^+ \end{pmatrix} \textrm{ dans } \cF(K^+)^q.\]  

En outre, on peut choisir $s_{1}^-,\dotsc,s_{p}^-,s_{1}^+,\dotsc,s_{q}^+$ de fa\c{c}on que, pour tout compact~$K^0$ contenu dans~$K^-$ (resp.~$K^+$) tel que $(t_{1}^-,\dotsc,t_{p}^-)_{|K^0}$ engendre le $\cO(K^0)$-module $\cF(K^0)$ (resp. $(t_{1}^+,\dotsc,t_{q}^+)_{|K^0}$ engendre le $\cO(K^0)$-module $\cF(K^0)$), la famille $(s_{1}^-,\dotsc,s_{p}^-,s_{1}^+,\dotsc,s_{q}^+)_{|K^0}$  engendre encore le $\cO(K^0)$-module $\cF(K^0)$.
\end{theo}
\begin{proof}
Posons 
\[T^- = \begin{pmatrix} t_{1}^-\\ \vdots \\ t_{p}^- \end{pmatrix} \textrm{ et } T^+ = \begin{pmatrix} t_{1}^+\\ \vdots \\ t_{q}^+ \end{pmatrix}.\]
Par hypoth\`ese, il existe $m\in\N$, $U = (u_{a,i}) \in M_{p,q}(\cC_{m})$ et $V = (v_{b,j}) \in M_{q,p}(\cC_{m})$ tels que l'on ait
\[\begin{cases}
T^- = U\,T^+ \textrm{ dans } \cF(L)^p~;\\
T^+=V\,T^- \textrm{ dans } \cF(L)^q.
\end{cases}\] 
Consid\'erons un nombre r\'eel~$\eps>0$ satisfaisant la conclusion du th\'eor\`eme~\ref{thm:lemmeCartan}. D'apr\`es la remarque~\ref{rem:CR}, il existe un entier $m' \ge m$, un \'el\'ement inversible~$f$ de~$\cB_{m'}^-$ et des \'el\'ements $\bar{u}_{a,i}$ de~$\cB_{m'}^+$, pour $(a,i) \in \cn{1}{p}\times\cn{1}{q}$, tels que, pour tous $(a,i) \in \cn{1}{p}\times\cn{1}{q}$ et $(b,j) \in \cn{1}{q}\times\cn{1}{p}$, on ait
\[\|f u_{a,i} - \bar{u}_{a,i}\|_{m'}\, \|f^{-1} v_{b,j}\|_{m'} \le \|f u_{a,i} - \bar{u}_{a,i}\|_{m'}\, \|f^{-1}\|_{m'}\, \| v_{b,j}\|_{m'}  < \eps.\]
On a alors 
\[\begin{cases}
fT^- = (fU)\, T^+ \textrm{ dans } \cF(L)^p~;\\
T^+ = (f^{-1}V)\, (fT^-) \textrm{ dans } \cF(L)^q
\end{cases}\]
et, en posant $\bar U = (\bar{u}_{a,i}) \in M_{p,q}(\cB_{m'}^+)$, 
\[\|\bar U - fU\|\, \|f^{-1}V\| < \eps.\]
On se trouve donc sous les hypoth\`eses du lemme~\ref{lem:matriceapprochee}. On en d\'eduit qu'il existe $S^- \in \cF(M)^p$ et $A^- \in GL_{p}(\cO(K^-))$ tels que
\[S^- = A^- \, f T^- \textrm{ dans } \cF(K^-)^p.\]
Le premier r\'esultat s'en d\'eduit en utilisant le fait que~$f$ est inversible sur~$K^-$. Le r\'esultat pour~$K^+$ se d\'emontre identiquement.

La remarque finale d\'ecoule de la construction et de l'inversibilit\'e des matrices~$A^-$ et~$A^+$.
\end{proof}

Nous en d\'eduisons finalement un analogue raffin\'e et corrig\'e de \cite[corollaire~6.2.11]{A1Z}. (L'hypoth\`ese de surjectivit\'e sur~$L$ a malencontreusement \'et\'e omise dans l'\'enonc\'e original.)

\begin{coro}\label{cor:K-K+A}\index{Faisceau!globalement engendr\'e}
Supposons qu'il existe un syst\`eme de Cousin-Runge fort associ\'e \`a $(K^-,K^+)$. Soit~$\cF$ un faisceau de $\cO_{M}$-modules. Supposons qu'il existe des morphismes surjectifs $\varphi^- \colon \cO^{p}_{K^-} \to \cF_{|K^-}$ et  $\varphi^+ \colon \cO^{q}_{K^+} \to \cF_{|K^+}$ tels que les morphismes induits $\cO^{p}(L) \to \cF(L)$ et  $\cO^{q}(L) \to \cF(L)$ soient surjectifs. Alors, $\cF$~est globalement engendr\'e sur~$M$.

Plus pr\'ecis\'ement, il existe un morphisme surjectif $\varphi \colon \cO^{p+q} \to \cF$ tel que, pour tout compact~$K^0$ contenu dans~$K^-$ (resp.~$K^+$) tel que le morphisme induit $\cO^{p}(K^0) \to \cF(K^0)$ (resp. $\cO^{q}(K^0) \to \cF(K^0)$) soit surjectif, le morphisme induit $\cO^{p+q}(K^0) \to \cF(K^0)$ soit surjectif.
\qed
\end{coro}
\index{Systeme@Syst\`eme|)}

\subsection{Alg\`ebres de domaines polynomiaux}

Nous pr\'esentons ici l'exemple principal de syst\`eme de Cousin-Runge fort que nous utiliserons dans la suite de ce texte.

Posons $X := \E{1}{\cA}$ et notons $\pi \colon X \to B$ le morphisme structural. Pour toute partie~$V$ de~$B$, on note $X_{V} := \pi^{-1}(V)$.

\begin{prop}\label{prop:systemeCR}\index{Systeme@Syst\`eme!sur un domaine polynomial}\index{Domaine polynomial!systeme@syst\`eme sur un|see{Syst\`eme}}\index{Systeme@Syst\`eme!de Runge}\index{Systeme@Syst\`eme!de Cousin}\index{Condition!ONPS-T@$(\cO N_{P(S)-T})$}
Soit~$V$ une partie compacte et d\'ecente de~$B_{\um}$. 
Soit $P \in \cO(V)[T]$ unitaire non constant. Soient $r,s\in\R$ tels que $0<r\le s$. Soient~$K^-$ et~$K^+$ deux parties compactes de~$X_{V}$ telles que 
\[\begin{cases}
K^- \subset \{x\in X_{V} : |P(x)|\le s\}\ ;\\
K^+ \subset \{x\in X_{V} : |P(x)|\ge r\}\ ;\\
K^- \cap K^+ = \{x\in X_{V} : r \le |P(x)| \le s\}.
\end{cases}\]
Alors, il existe un syst\`eme de Cousin associ\'e \`a $(K^-,K^+)$. Si $r=s$, alors il existe un syst\`eme de Cousin-Runge associ\'e \`a $(K^-,K^+)$.

Supposons, en outre, qu'il existe une suite d\'ecroissante $(W_{m})_{m\in \N}$ de parties compactes de~$B$ sur lesquels les coefficients de~$P$ sont d\'efinis et formant une base de voisinages de~$V$, une suite $(u_{m})_{m\in\N}$ d'\'el\'ements de~$\intoo{0,r}$ strictement croissante de limite~$r$ et une suite $(v_{m})_{m\in\N}$ d'\'el\'ements de~$\intoo{s,+\infty}$ strictement d\'ecroissante de limite~$s$ telles que, pour tout $m\in\N$, le compact $\overline{C}_{W_{m}}(u_{m},v_{m})$ satisfasse la condition $(\cO N_{P(S)-T})$. Alors, il existe un syst\`eme de Banach fort (et donc de Cousin fort, et m\^eme de Cousin-Runge fort si $s=r$) associ\'e \`a $(K^-,K^+)$. 
\end{prop}
\begin{proof}
Puisque~$K^+$ est compact, il existe un nombre r\'eel $t > s$ qui majore la fonction~$|P|$ sur~$K^+$. Nous pouvons supposer que $K^+ = \overline{C}_{V}(P;r,t)$ et $K^- = \overline{C}_{V}(P;s)$. Posons $L =: K^-\cap K^+ = \overline{C}_{V}(P;r,s)$.

Soient $(W_{m})_{m\in \N}$ une suite d\'ecroissante de parties compactes de~$B$ sur lesquels les coefficients de~$P$ sont d\'efinis et formant une base de voisinages de~$V$, $(u_{m})_{m\in\N}$ une suite d'\'el\'ements de~$\intoo{0,r}$ strictement croissante de limite~$r$ et $(v_{m})_{m\in\N}$ une suite d'\'el\'ements de~$\intoo{s,+\infty}$ strictement d\'ecroissante de limite~$s$. 

D'apr\`es la proposition~\ref{prop:equivalencedivres} appliqu\'ee avec $\cA = \overline{\cO(W_{0})}\la |T|\le t\ra$, il existe $w \in \R_{>0}$ tel que, pour tout entier~$m$, la semi-norme r\'esiduelle $\nm_{\overline{\cO(W_{m})}\la|T|\le v_{m}\ra,w,\res}$ d\'efinie sur $\overline{\cO(W_{m})}\la|T|\le v_{m}\ra[S]/(P(S)-T)$ soit une norme pour laquelle cet espace est complet. On notera $(\cB_{m}^-,\nm_{m}^-)$ cet anneau de Banach. Les m\^emes propri\'et\'es valent pour $\overline{\cO(W_{m})} \la u_{m} \le |T|\le t\ra[S]/(P(S)-T)$, que l'on notera $(\cB_{m}^+,\nm_{m}^+)$, et $\overline{\cO(W_{m})}\la u_{m}\le |T|\le v_{m}\ra[S]/(P(S)-T)$, que l'on notera $(\cC_{m},\nm_{m})$. On peut supposer que le m\^eme nombre r\'eel~$w$ convient pour tous ces anneaux.

Avec les morphismes \'evidents, on vient de d\'efinir un syst\`eme de Banach~$\Omega$ associ\'e \`a $(K^-,K^+)$. Le seul point d\'elicat \`a v\'erifier est le fait que le morphisme 
\[\colim_{m\in\N} \cC_{m} \too \cO(L)\]
soit un isomorphisme. C'est l'objet du corollaire~\ref{cor:lemniscateglobaleBanach}.

\smallbreak

D\'emontrons maintenant que~$\Omega$ est un syst\`eme de Cousin. Soient~$m\in \N$ et $f\in \cC_{m}$. On peut supposer que $f\ne 0$. Il existe alors un entier $d \in \N$ et un \'el\'ement $F = \sum_{i=0}^d a_{i} \, S^i$ de $\overline{\cO(W_{m})}\la u_{m}\le |T|\le v_{m}\ra[S]$ tel que $\|F\|_{\overline{\cO(W_{m})}\la u_{m}\le |T|\le v_{m}\ra,w} \le 2 \|f\|_{m}$. En utilisant la description explicite de $\overline{\cO(W_{m})}\la u_{m}\le |T|\le v_{m}\ra$ comme anneau de s\'eries de Laurent, on montre que, pour tout $i\in \cn{0}{d}$, il existe $a_{i}^- \in \overline{\cO(W_{m})}\la |T|\le v_{m}\ra$ et $a_{i}^+ \in \overline{\cO(W_{m})}\la u_{m}\le |T|\le t\ra$ tels que
\[\begin{cases}
a_{i} = a_{i}^-+a_{i}^+ \textrm{ dans } \overline{\cO(W_{m})}\la u_{m}\le |T|\le v_{m}\ra\ ;\\[2pt]
\|a_{i}^-\|_{\overline{\cO(W_{m})}\la |T|\le v_{m}\ra} \le \|a_{i}\|_{\overline{\cO(W_{m})}\la u_{m}\le |T|\le v_{m}\ra}\ ;\\[2pt]
\|a_{i}^+\|_{\overline{\cO(W_{m})}\la u_{m}\le |T|\le t\ra} \le \|a_{i}\|_{\overline{\cO(W_{m})}\la u_{m}\le |T|\le v_{m}\ra}.
\end{cases}\]

En posant $F^- = \sum_{i=0}^d a_{i}^- \,S^i \in\overline{\cO(W_{m})}\la |T|\le v_{m}\ra[S]$ et $F^+ = \sum_{i=0}^d a_{i}^+ \,S^i \in \overline{\cO(W_{m})}\la u_{m}\le |T|\le t\ra[S]$ et en notant~$f^-$ et~$f^+$ les classes r\'esiduelles respectives modulo $P(S)-T$, on obtient le r\'esultat d\'esir\'e avec la constante~$D=2$ (ind\'ependante de~$m$).

\smallbreak

Supposons que $r=s$ et d\'emontrons que~$\Omega$ est un syst\`eme de Runge. Soient $m\in\N$ et $\eps>0$. Soient $p\in\N^\ast$ et $s_{1},\dotsc,s_{p} \in \cC_{m}$. Pour tout $i\in\cn{1}{p}$, choisissons un \'el\'ement $S_{i} = \sum_{k=0}^d A_{i,k}\, S^k$ de $\overline{\cO(W_{m})}\la u_{m}\le |T|\le v_{m}\ra[S]$ qui rel\`eve~$s_{i}$. Remarquons qu'il est loisible de supposer l'entier~$d$ ind\'ependant de~$i$. 

Commen\c{c}ons par traiter le cas $\sigma=-$. Pour tous $i\in\cn{1}{p}$ et $k\in\cn{0}{d}$, il existe $A'_{i,k} \in \overline{\cO(W_{m})}[T,T^{-1}]$ tel que
\[ \|A_{i,k} - A'_{i,k}\|_{\overline{\cO(W_{m})}\la u_{m}\le |T|\le v_{m}\ra} \, w^k < \eps/(d+1).\]
Pour tout $i\in\cn{1}{p}$, on pose $S'_{i} = \sum_{k=0}^d A'_{i,k}\, S^k$ et on a alors 
\[ \| S_{i} - S'_{i}\|_{\overline{\cO(W_{m})}\la u_{m}\le |T|\le v_{m}\ra,w} < \eps.\]
Tout \'el\'ement~$S'_{i}$ se prolonge en un \'el\'ement de $\overline{\cO(W_{m})}\la u_{m}\le |T|\le t\ra[S]$ et d\'efinit donc un \'el\'ement~$s'_{i}$ de $\cB_{m}^+$ par passage au quotient. La propri\'et\'e d'approximation de l'\'enonc\'e est alors v\'erifi\'ee avec $f=1$. La condition 
\[\lim_{m'\to\infty} \|f\|_{m} \, \|f^{-1}\|_{m} = 1\]
est \'evidemment v\'erifi\'ee aussi.

Traitons maitenant le cas $\sigma=+$. Pour tous $i\in\cn{1}{p}$ et $k\in\cn{0}{d}$, notons 
\[A_{i,k} = \sum_{l\in\Z} a_{i,k}^l \,T^l \in \overline{\cO(W_{m})}\la u_{m}\le |T|\le v_{m}\ra.\]
Il existe un entier $l_{0}\le 0$ tel que, pour tous $i\in\cn{1}{p}$ et $k\in\cn{0}{d}$, on ait
\[\sum_{l\le l_{0} - 1} \|a_{i,k}^l\|_{W_{m}}\, \max(u_{m}^l,v_{m}^l)  < \frac{\eps}{(d+1) w^k}.\]
Posons $F = T^{-l_{0}} \in \overline{\cO(W_{m})}\la u_{m}\le |T|\le t\ra[S]$. C'est un \'el\'ement inversible. Notons~$f$ son image dans~$\cB_{m}^+$. 

Pour $i\in\cn{1}{p}$ et $k\in\cn{0}{d}$, posons 
\[A'_{i,k} = \sum_{l\ge l_{0}} a_{i,k}^l\, T^l.\]
On a
\[F A'_{i,k} = \sum_{l\ge 0} a_{i,k}^{l+l_{0}}\, T^l \in \overline{\cO(W_{m})}\la |T|\le v_{m}\ra.\]
Pour $i\in\cn{1}{p}$, posons 
\[S'_{i} = \sum_{k=0}^d F A'_{i,k}\, S^k \in \overline{\cO(W_{m})}\la |T|\le v_{m}\ra[S]\]
et notons~$s'_{i}$ son image dans~$\cB_{m}^-$. Pour tout $i\in\cn{1}{p}$, on a
\begin{align*}
\|f s_{i}-s'_{i}\|_{m} & \le \|f\|_{m}\, \|s_{i} - f^{-1}s'_{i}\|_{m}\\
& \le \|f\|_{m} \,  \left\| \sum_{k=0}^d (A_{i,k} - A'_{i,k}) \, S^{k}\right\|_{\overline{\cO(W_{m})}\la u_{m}\le |T|\le v_{m}\ra,w} \\
& <\eps\, \|f\|_{m}.
\end{align*}
Finalement, pour tout $m' \ge m$, on a
\begin{align*} 
\|f\|_{m'}\, \| f^{-1}\|_{m'} &\le \|T^{-l_{0}}\|_{\overline{\cO(W_{m'})}\la u_{m'}\le |T|\le v_{m'}\ra} \, \|T^{l_{0}}\|_{\overline{\cO(W_{m'})}\la u_{m'}\le |T|\le v_{m'}\ra}\\ 
& \le  \left(\frac{v_{m}'}{u_{m'}}\right)^{|l_{0}|}.
\end{align*}
Cette derni\`ere quantit\'e tend vers $s/r =1$ lorsque~$m'$ tend vers l'infini (et c'est le seul endroit o\`u l'\'egalit\'e $s=r$ est utilis\'ee).

\smallbreak

Finalement, supposons que, pour tout $m\in\N$, le compact $\overline{C}_{W_{m}}(u_{m},v_{m})$ satisfait la condition $(\cO N_{P(S)-T})$ et d\'emontrons que~$\Omega$ est un syst\`eme de Banach fort.

Soit~$U$ un voisinage compact de~$L$. Il contient un voisinage de la forme $\overline{C}_{W_{m_{0}}}(P;u_{m_{0}},v_{m_{0}})$. D'apr\`es la condition~$(\cO N_{P(S)-T})$, il existe $A\in\R$ tel que, pour tout $h\in\cO(\overline{C}_{W_{m_{0}}}(u_{m_{0}},v_{m_{0}}))[S]/(P(S)-T)$, on ait
\[\|h\|_{\cO(\overline{C}_{W_{m_{0}}}(u_{m_{0}},v_{m_{0}})),w,\res} \le A\, \|h\|_{\overline{C}_{W_{m_{0}}}(P;u_{m_{0}},v_{m_{0}}) } \le A\, \| h\|_{U}.\]

Soit $m\ge m_{0}+1$. Alors, d'apr\`es la proposition~\ref{prop:restrictionserie}, pour tout $k\in\overline{\cO(W_{m_{0}})}\la u_{m_{0}}\le|T|\le v_{m_{0}}\ra$, on a
\[\|k\|_{\overline{\cO(W_{m})}\la u_{m}\le|T|\le v_{m}\ra} \le \left(\frac{u_{m_{0}}}{u_{m}-u_{m_{0}}} + \frac{v_{m_{0}}}{v_{m_{0}}-v_{m}} \right) \, \|k\|_{\cO(\overline{C}_{W_{m_{0}}}(u_{m_{0}},v_{m_{0}}))}.\]

Soit $f\in \cO(U)$. D'apr\`es le corollaire~\ref{cor:lemniscateglobale}, il existe un \'el\'ement~$g$ de $\overline{\cO(W_{m})}\la u_{m}\le|T|\le v_{m}\ra[S]/(P(S)-T)$ dont l'image dans~$\cO(L)$ co\"{\i}ncide avec celle de~$f$. Les \'egalit\'es pr\'ec\'edentes montrent par ailleurs que l'on a
\begin{align*}
\|g\|_{\overline{\cO(W_{m})}\la u_{m}\le|T|\le v_{m}\ra,w,\res} &\le \left(\frac{u_{m_{0}}}{u_{m}-u_{m_{0}}} + \frac{v_{m_{0}}}{v_{m_{0}}-v_{m}} \right)\, \|g\|_{\cO(\overline{C}_{W_{m_{0}}}(u_{m_{0}},v_{m_{0}})),w,\res}\\
&\le A \left(\frac{u_{m_{0}}}{u_{m}-u_{m_{0}}} + \frac{v_{m_{0}}}{v_{m_{0}}-v_{m}} \right)\, \|f\|_{U}.
\end{align*}
Le r\'esultat s'ensuit.
\end{proof}

Dans la proposition~\ref{prop:systemeCR}, la difficult\'e r\'eside, bien s\^ur, dans la v\'erification de la condition suppl\'ementaire finale portant sur la propri\'et\'e $(\cO N_{P(S)-T})$. Remarquons que le corollaire~\ref{cor:BVcompactN} fournit un moyen d'y parvenir dans le cas o\`u tous les points du compact~$V$ sont ultram\'etriques tr\`es typiques.

Mentionnons \'egalement le r\'esultat suivant, particuli\`erement adapt\'e au voisinage de points dont le corps r\'esiduel est de caract\'eristique nulle ou de valuation non triviale.

\begin{lemm}\label{lem:conditionscar0}\index{Condition!ONPS-T@$(\cO N_{P(S)-T})$}\index{Resultant@R\'esultant}
Soient~$V$ une partie compacte de~$B_{\um}$ et $P \in \cO(V)[T]$ unitaire non constant. Soient $r,s \in \R$ tels que $0<r\le s$. Supposons que le polyn\^ome $R(T) = \Res_{S}(P(S)-T,P'(S))$ ne s'annule ni sur $\overline{C}_{V}(r,r)$, ni sur $\overline{C}_{V}(s,s)$. 

Alors, il existe un voisinage~$U$ de~$V$ sur lequel les coefficients de~$P$ sont d\'efinis et des nombres r\'eels $r_{0}\in\intoo{0,r}$ et $s_{0}\in\intoo{s,+\infty}$ tels que, pour toute partie compacte~$W$ de~$U$ contenant~$V$, tout $u\in\intff{r_{0},r}$ et tout $v\in\intff{s,s_{0}}$, le compact $\overline{C}_{W}(u,v)$ satisfasse la condition $(\cO N_{P(S)-T})$.

En particulier, la condition finale de la proposition~\ref{prop:systemeCR} est satisfaite.
\end{lemm}
\begin{proof}
Soit~$U$ un voisinage de~$V$ sur lequel les coefficients de~$P$ sont d\'efinis. Consid\'erons le polyn\^ome $R(T) = \Res_{S}(P(S)-T,P'(S)) \in \cO(U)[T]$. Par hypoth\`ese, il ne s'annule ni sur $\overline{C}_{V}(r,r)$, ni sur $\overline{C}_{V}(s,s)$. Quitte \`a restreindre~$U$, on peut supposer qu'il ne s'annule pas non plus sur $\overline{C}_{U}(r,r)$ et $\overline{C}_{U}(s,s)$, et m\^eme sur $\overline{C}_{U}(u,u)$ et $\overline{C}_{U}(v,v)$, pour tous $u \in [r_{0},r]$ et $v \in [s,s_{0}]$ avec $r_{0}<r$ et $s_{0}>s$ bien choisis. Le r\'esultat d\'ecoule alors de la proposition~\ref{prop:RG}.
\end{proof}

\subsection{Arbres de Cousin-Runge}
\index{Arbre|(}

Soit~$X$ un espace $\cA$-analytique. 

Nous introduisons maintenant une nouvelle notion qui sera utile pour effectuer des r\'ecurrences sur la dimension.

\begin{defi}\label{def:arbrebinaire}\index{Arbre!binaire|textbf}
\nomenclature[Ra]{$\cT$}{arbre binaire de compacts}%
\nomenclature[Rb]{$\cI(\cT)$}{ensemble des n\oe uds internes de~$\cT$}%
On appelle \emph{arbre binaire} un arbre fini~$\cT$ avec une racine~$r$ dont tout n\oe ud~$v$ qui n'est pas une feuille poss\`ede exactement deux fils~$v^-$ et~$v^+$. Dans la repr\'esentation graphique de~$\cT$, nous placerons $v^-$ et~$v^+$ respectivement \`a gauche et \`a droite de~$v$.

On appelle \emph{n\oe ud interne} de~$\cT$ tout n\oe ud qui n'est pas une feuille. On note~$\cI(\cT)$ l'ensemble de ces n\oe uds.
\end{defi}

\begin{defi}\label{def:arbreCR}\index{Arbre!binaire!de compacts|textbf}\index{Arbre!assez d'|textbf}\index{Arbre!a intersections simples@\`a intersections simples|textbf}\index{Arbre!de Cousin|textbf}\index{Arbre!adapt\'e \`a un recouvrement|textbf}%
\nomenclature[Rc]{$K(v)$}{compact \'etiquetant un n\oe ud~$v$ de~$\cT$}%
\nomenclature[Rd]{$\cF(\cT)$}{famille des compacts associ\'es aux feuilles de~$\cT$}%
Soit~$V$ une partie compacte de~$X$. On appelle \emph{arbre binaire de compacts} sur~$V$ un arbre binaire~$\cT$ dont chaque n{\oe}ud~$v$ est \'etiquet\'e par une partie compacte $K(v)$ de~$V$ et qui v\'erifie les propri\'et\'es suivantes~:
\begin{enumerate}[i)]
\item la racine~$r$ de~$\cT$ est \'etiquet\'ee par~$K(r) = V$~;
\item pour tout n\oe ud interne~$v$ de~$\cT$, on a 
\[K(v) = K(v^-) \cup K(v^+).\]
\end{enumerate}
En particulier, la famille~$\cF(\cT)$ des compacts associ\'es aux feuilles de~$\cT$ est un recouvrement de~$V$. \'Etant donn\'e un recouvrement~$\cU$ de~$V$, on dit que l'\emph{arbre}~$\cT$ est \emph{adapt\'e} \`a~$\cU$ si~$\cF(\cT)$ est plus fin que~$\cU$.

On appelle \emph{arbre de Cousin} (resp. Cousin-Runge, etc.) sur~$V$ tout arbre binaire de compacts sur~$V$ qui v\'erifie la propri\'et\'e suivante~:
\begin{enumerate}
\item[iii)] pour tout n{\oe}ud interne~$v$, il existe un syst\`eme de Cousin (resp. Cousin-Runge, etc.) associ\'e \`a $(K(v^-),K(v^+))$.
\end{enumerate}

Un arbre binaire de compacts~$\cT$ sur~$V$ est dit \emph{\`a intersections simples} si, pour tout n{\oe}ud~$v$ de~$\cT$ qui n'est pas une feuille, il existe une feuille~$f(v^-)$ du sous-arbre issu de~$v^-$ et une feuille~$f(v^+)$ du sous-arbre issu de~$v^+$ telles que 
\[K(v^-)\cap K(v^+) = K(f(v^-)) \cap K(f(v^+)).\] 

On dit que le compact~$V$ de~$X$ poss\`ede \emph{assez d'arbres} de Cousin (resp. Cousin-Runge, etc.) si, pour tout recouvrement ouvert~$\cU$ de~$V$, il existe un arbre de Cousin (resp. Cousin-Runge, etc.) sur~$V$ adapt\'e \`a~$\cU$.
\end{defi}


\begin{prop}\label{prop:CRimpliqueA}\index{Faisceau!globalement engendr\'e}
Soit~$V$ une partie compacte de~$X$. Soit $\br \in \R_{\ge 0}^n$. 
Soit~$\cF$ un faisceau de $\cO_{\overline{D}_{V}(\br)}$-modules de type fini. Supposons qu'il existe un arbre de Cousin-Runge fort \`a intersections simples~$\cT$ sur~$V$ et, pour toute feuille~$f$ de~$\cT$, un morphisme surjectif $\varphi_{f} \colon \cO_{\overline{D}_{K(f)}(\br)}^{n_{f}} \to \cF_{|\overline{D}_{K(f)}(\br)}$ v\'erifiant la propri\'et\'e suivante~: pour toute feuille~$g$ de~$\cT$, le morphisme $\cO^{n_{f}}(\overline{D}_{K(f)\cap K(g)}(\br)) \to \cF(\overline{D}_{K(f)\cap K(g)}(\br))$ induit par~$\varphi_{f}$ est surjectif.

Alors, le faisceau $\cF$ est globalement engendr\'e sur~$\overline{D}_{V}(\br)$.
\end{prop}
\begin{proof}
Modifions l'arbre binaire de compacts~$\cT$ en rempla\c{c}ant chaque \'etiquette~$K$ par $\overline{D}_{K}(\br)$. On obtient un arbre binaire de compacts~$\cT'$ sur~$\overline{D}_{V}(\br)$. Il est encore \`a intersections simples, et \'egalement de Cousin-Runge fort, d'apr\`es le lemme~\ref{lem:CRrelatif}. Pour tout n\oe ud~$v'$ de~$\cT'$, notons~$K'(v')$ son \'etiquette.

On souhaite, \`a pr\'esent, d\'emontrer que, pour tout n\oe ud~$v'$ de~$\cT'$, il existe un entier~$n_{v'}$ et un morphisme surjectif $\varphi_{v'} \colon \cO_{K'(v')}^{n_{v'}} \to \cF_{|K'(v')}$ qui induit un morphisme surjectif $\cO^{n_{v'}}(K'(f')\cap K'(g')) \to \cF(K'(f')\cap K'(g'))$, pour toute feuille~$f'$ de~$\cT'$ issue de~$v'$ et toute feuille~$g'$ de~$\cT'$. Il suffit, pour ce faire, d'appliquer le corollaire~\ref{cor:K-K+A} de fa\c{c}on r\'ep\'et\'ee en remontant dans l'arbre, l'hypoth\`ese de surjectivit\'e globale sur les intersections \'etant assur\'ee par la condition d'intersections simples sur~$\cT'$.

En particulier, lorsque~$v'$ est la racine de~$\cT'$, d'\'etiquette~$\overline{D}_{V}(\br)$, on obtient le r\'esultat souhait\'e. 
\end{proof}

\subsection{Au-dessus d'un corps ultram\'etrique}

Dans cette section, nous fixons un corps valu\'e ultram\'etrique complet $(k,\va)$. Nous allons construire des arbres de Cousin-Runge sur certaines parties de la droite analytique~$\E{1}{k}$.

\begin{defi}\index{Voisinage!elementaire@\'el\'ementaire|textbf}\index{Recouvrement elementaire@Recouvrement \'el\'ementaire|textbf}
Soient $x\in \Aunk$ un point de type~2 ou~3. On dit qu'un ouvert~$U$ de~$\Aunk$ est un \emph{voisinage \'el\'ementaire} de~$x$ dans~$\Aunk$ s'il existe un nombre r\'eel $\eps >0$, un ensemble fini~$F$ de composantes connexes relativement compactes de~$\Aunk\setminus\{x\}$ et, pour tout \'el\'ement~$C$ de~$F$, un point rigide~$x_{C}$ de~$C$ tels que 
\[U = C(P_{F},|P_{F}(x)| - \eps, |P_{F}(x)| + \eps),\]
o\`u $P_{F} := \prod_{C\in F} \mu_{x_{C}}$.

Soient $x\in \Aunk$ un point de type~1 ou~4. On dit qu'un ouvert~$U$ de~$\Aunk$ est un \emph{voisinage \'el\'ementaire} de~$x$ dans~$\Aunk$ s'il existe un nombre r\'eel $\eps >0$ et un polyn\^ome irr\'eductible $P\in k[T]$ tels que
\[U = D(P,|P(x)|+ \eps).\]

On dit qu'un recouvrement ouvert~$\cU$ de~$\Aunk$ est un \emph{recouvrement \'el\'ementaire} si tout \'el\'ement~$U$ de~$\cU$ est voisinage \'el\'ementaire de l'un de ses points.
\end{defi}

\begin{rema}\label{rem:raffinementelementaire}\index{Voisinage}
Les lemmes~\ref{lem:bv23} et~\ref{lem:bv14} assurent que tout point de~$\Aunk$ poss\`ede une base de voisinages \'el\'ementaires, et donc que tout recouvrement de~$\Aunk$ peut \^etre raffin\'e en un recouvrement \'el\'ementaire.
\end{rema}

\begin{defi}\label{defi:arbreelementaire}\index{Arbre!elementaire@\'el\'ementaire|textbf}\index{Arbre!bien decoupe@bien d\'ecoup\'e|textbf}%
\nomenclature[Re]{$G_{v}$}{ensemble des n{\oe}uds~$w\ne v$ situ\'es entre la racine et~$v$ tels que~$v$ soit \`a gauche de~$w$}%
\nomenclature[Rf]{$D_{v}$}{ensemble des n{\oe}uds~$w\ne v$ situ\'es entre la racine et~$v$ tels que~$v$ soit \`a droite de~$w$}%
\nomenclature[Rg]{$\cT(\cE)$}{arbre binaire de compacts \'el\'ementaire}%
Soient~$X$ un espace analytique et~$V$ une partie compacte de~$X$. Soient $P\in\cO(V)[T]$ et $a\in\R_{\ge 0}$. 

Soit~$\cT$ un arbre binaire. Pour tout n{\oe}ud $v$ de~$\cT$, notons~$G_{v}$ (resp.~$D_{v}$) l'ensemble des n{\oe}uds~$w\ne v$ situ\'es entre la racine et~$v$ et tels que~$v$ soit \`a gauche (resp. \`a droite) de~$w$, c'est-\`a-dire dans le sous-arbre issu du fils gauche (resp. droit), \cf~figure~\ref{fig:GvDv}.

\begin{figure}[!h]
\centering
\begin{tikzpicture}
\draw (-3.2,4.3) node[above right]{$\cT$} ;
\draw (-1.6,5) -- (-2.6,3.3);
\draw (-1.6,5)-- (-0.8,3.3);
\draw (-2.6,3.3)-- (-4,1.8);
\draw (-2.6,3.3)-- (-1.6,1.8);
\draw (-1.6,1.8)-- (-2.5,0.4);
\draw (-1.6,1.8)-- (-0.9,0.3);
\draw (-3,0.3) node[anchor=north west] {$v$};
\draw (2,3.6) node[anchor=north west] {$G_v$};
\draw (-5.7,3.6) node[anchor=north west] {$D_v$};
\draw [->] (1.8,3.2) to[out=120,in=-10] (-1.2,4.9);
\draw [->] (1.8,3.2) to[out=230,in=10]  (-1.3,1.8);
\draw [->] (-4.8,3.3) to[out=20,in=160] (-3,3.3);

\fill (-1.6,5) circle (2pt) ;
\fill (-2.6,3.3) circle (2pt) ;
\fill (-0.8,3.3) circle (2pt) ;
\fill (-4,1.8) circle (2pt) ;
\fill (-1.6,1.8) circle (2pt) ;
\fill (-2.5,0.4) circle (2pt) ;
\fill (-0.9,0.3) circle (2pt) ;

\end{tikzpicture}
\caption{Sous-ensembles $G_{v}$ et $D_{v}$.}\label{fig:GvDv}
\end{figure}

Soit $\cE = (P_{v},s_{v})_{v\in \cI(\cT)}$ une famille d'\'el\'ements de $\cO(V)[T] \times \R_{\ge 0}$ ou $\cB(V)[T] \times \R_{\ge 0}$  ind\'ex\'ee par l'ensemble~$\cI(\cT)$ des n{\oe}uds internes de~$\cT$. On d\'efinit alors un \emph{arbre binaire de compacts} $\cT(\cE)$ sur~$\overline{D}_{V}(P;a)$ dit \emph{\'el\'ementaire} en \'etiquetant tout n{\oe}ud~$v$ de~$\cT$ par
\begin{align*}
K(v) &= \overline{D}_{V}(P;a) \cap \bigcap_{w\in G_{v}}  \{x\in \E{1}{V} : |P_{w}(x)| \le s_{w}\}\\ 
&\qquad \cap \bigcap_{w\in D_{v}}  \{x\in \E{1}{V} : |P_{w}(x)| \ge s_{w}\}.
\end{align*}

On dit que l'\emph{arbre} $\cT(\cE)$ est \emph{bien d\'ecoup\'e} si, pour tout n\oe uds internes $v \ne v'$, on a
\[\{x\in \E{1}{V} : |P_{v}(x)| = s_{v}\} \cap \{x\in \E{1}{V} : |P_{v'}(x)| = s_{v'}\} = \emptyset.\]
\end{defi}

\begin{rema}\label{rem:biendecoupe}
Pour tout n{\oe}ud interne~$v$, on a
\[\begin{cases}
K(v^-) = K(v) \cap \{x\in \E{1}{V} : |P_{v}(x)|\le s_{v}\}\ ;\\
K(v^+) = K(v) \cap \{x\in \E{1}{V} : |P_{v}(x)|\ge s_{v}\}.
\end{cases}\]
En particulier, si~$\cT(\cE)$ est bien d\'ecoup\'e, on a
\[\begin{cases}
K(v^-) = \{x\in \E{1}{V} : |P_{v}(x)|\le s_{v}\}\ ;\\
K(v^-) \cap K(v^+) = \{x\in \E{1}{V} : |P_{v}(x)| = s_{v}\}.
\end{cases}\]
Nous retrouvons ainsi une partie des hypoth\`eses de la proposition~\ref{prop:systemeCR}. 

Remarquons encore que, si~$\cT(\cE)$ est bien d\'ecoup\'e, alors, pour tout n\oe ud interne~$v$, on a 
\[ \{x\in \E{1}{V} : |P_{v}(x)| = s_{v}\}  = K(f^d(v^-)) \cap K(f^d(v^+)),\]
o\`u $f^d(v^-)$ (resp. $f^d(v^+)$) d\'esigne la feuille la plus \`a droite du sous-arbre issu de~$v^-$ (resp. $v^+$). En particulier, $\cT(\cE)$ est \`a intersections simples. On peut \'egalement v\'erifier que, pour toutes feuilles $f\ne g$ de~$\cT(\cE)$, l'intersection $K(f) \cap K(g)$ est soit vide, soit de la forme $\{x\in \E{1}{V} : |P_{v}(x)| = s_{v}\}$ pour un certain n\oe ud interne~$v$.  
\end{rema}

\begin{defi}\label{defi:etoile}\index{Etoile@\'Etoile|textbf}
On dit qu'une partie~$E$ de~$\Aunk$ est une \emph{\'etoile} s'il existe un polyn\^ome irr\'eductible $P\in k[T]$, un nombre r\'eel~$a \ge 0$ et une famille finie~$\cC$ de composantes connexes de $\overline{D}(P,a) \setminus \{\eta\}$, o\`u~$\eta$ d\'esigne l'unique point du bord de~$\overline{D}(P,a)$, tels que 
\[E = \overline{D}(P,a) \setminus \bigcup_{C \in \cC} C.\] 
\end{defi}

\begin{exem}
Les disques et les couronnes, et plus g\'en\'eralement les parties de la forme $\overline{D}(P,s)$ et $\overline{C}(P;s,s)$, avec $P \in k[T]$ irr\'eductible et $s \in \R_{\ge 0}$, sont des \'etoiles.
\end{exem}

\begin{prop}\label{prop:CRcorps}\index{Etoile@\'Etoile!arbre sur une|see{Arbre}}\index{Arbre!sur une \'etoile}\index{Resultant@R\'esultant}
Soit~$E \subset \Aunk$ une \'etoile. Alors, pour toute famille d'ouverts~$\cU$ de~$\Aunk$ recouvrant~$E$, il existe un arbre binaire~$\cT$ et une famille $\cE = (P_{v},s_{v})_{v\in \cI(\cT)}$ d'\'el\'ements de $k[T]\times\R_{>0}$, les polyn\^omes~$P_{v}$ \'etant irr\'eductibles, index\'ee par l'ensemble des n{\oe}uds internes de~$\cT$ tels que l'arbre binaire de compacts \'el\'ementaire $\cT(\cE)$ sur~$E$ soit bien d\'ecoup\'e, adapt\'e \`a~$\cU$ et de Cousin-Runge fort.

En outre, si~$k$ est parfait ou de valuation non triviale, on peut supposer que, pour tout n{\oe}ud interne~$v$ de~$\cT$, le polyn\^ome $R_{v}(T) = \Res_{S}(P_{v}(S)-T,P_{v}'(S))$ ne s'annule pas sur $\overline{C}(s_{v},s_{v})$.
\end{prop}
\begin{proof}
Par hypoth\`ese, il existe un polyn\^ome irr\'eductible $P\in k[T]$, un nombre r\'eel~$a \ge 0$ et une famille finie~$\cC$ de composantes connexes de $\overline{D}(P,a) \setminus \{\eta\}$, o\`u~$\eta$ d\'esigne l'unique point du bord de~$\overline{D}(P,a)$, tels que  
\[E = \overline{D}(P,a) \setminus \bigcup_{C \in \cC} C.\] 

Soit~$\cU$ une famille d'ouverts de~$\Aunk$ recouvrant $E$.
D'apr\`es les lemmes~\ref{lem:bv23} et~\ref{lem:bv14}, on peut supposer que tout \'el\'ement de~$\cU$ est voisinage \'el\'ementaire de d'un point de~$E$. En outre, par compacit\'e de~$E$, on peut supposer que~$\cU$ est fini, de cardinal~$n \ge 1$. D\'emontrons le r\'esultat pour de tels recouvrements par r\'ecurrence sur~$n$. 

Traitons tout d'abord le cas o\`u $n=1$. Notons~$\cT$ l'arbre binaire r\'eduit \`a sa racine. En consid\'erant la famille vide~$\cE$, on obtient un arbre binaire de compacts~$\cT(\cE)$ sur~$E$ r\'eduit \`a sa racine, \'etiquet\'ee par~$E$. Les conditions de l'\'enonc\'e sont satisfaites.

Supposons d\'esormais que~$n\ge 2$ et que le r\'esultat est d\'emontr\'e pour les recouvrements \'el\'ementaires de cardinal inf\'erieur \`a~$n-1$. Soit un \'el\'ement~$U$ de~$\cU$ contenant~$\eta$. C'est un voisinage \'el\'ementaire d'un point~$x$ de~$E$. 

Si~$x$ est de type~1 ou~4, alors tout voisinage \'el\'ementaire de~$x$ contenant~$\eta$ contient $E$ tout entier, et nous sommes donc ramen\'es au cas $n=1$ d\'ej\`a trait\'e.

Nous pouvons donc supposer que~$x$ est de type~2 ou~3. Il existe alors un ensemble $F = \{C_{1},\dots,C_{p}\}$ de composantes connexes born\'ees de $\Aunk \setminus\{x\}$, pour tout $i\in\cn{1}{p}$, un polyn\^ome irr\'eductible~$P_{i}$ s'annulant en un point de~$C_{i}$ et deux nombres r\'eels $r,s$ avec $r < |P_{F}(x)| < s$ tels que $U = \mathring{D}(P_{F};r,s)$, o\`u $P_{F} = \prod_{1\le i\le p} P_{i}$. Remarquons que, puisque $U$ contient~$\eta$, on a 
\[U \cap E = \{y\in \Aunk : r < |P_{F}(y)| \le |P_{F}(\eta)|\} \cap E.\]

Si $r<0$, alors $\mathring{D}(P_{F};r,s) = \mathring{D}(P_{F},s)$ contient~$E$, et nous sommes de nouveau ramen\'es au cas $n=1$. Nous pouvons donc supposer que $r\ge 0$.

Soit $i\in \cn{1}{p}$. Le compact $C_{i} \setminus U$ est de la forme $\overline{D}(P_{i};a_{i})$ et il est recouvert par $\cU\setminus\{U\}$. Il existe $c_{i} >a_{i}$ tel que $\overline{D}(P_{i};c_{i})$ soit encore contenu dans~$C_{i}$ et recouvert par $\cU\setminus\{U\}$. 

Si $k$ est parfait, le polyn\^ome~$P_{i}$ est s\'eparable. Si~$k$ est de valuation non triviale, on peut perturber l\'eg\`erement les coefficients de~$P_{i}$ de fa\c{c}on \`a le rendre s\'eparable sans changer les disques $\overline{D}(P_{i};b_{i})$ pour $b_{i} \in [a_{i},c_{i}]$. Alors, le polyn\^ome $R_{i}(T) = \Res_{S}(P_{i}(S)-T,P'_{i}(S))$ n'est pas nul. Quitte \`a diminuer~$c_{i}$, on peut supposer qu'il ne s'annule pas sur $\overline{C}(b_{i},b_{i})$, pour tout $b_{i} \in \intoo{a_{i},c_{i}}$. 

D'apr\`es l'hypoth\`ese de r\'ecurrence appliqu\'ee \`a~$\overline{D}(P_{i},a_{i})$, il existe un arbre binaire~$\cT_{i}$ et une famille $\cE_{i} = (P_{v},s_{v})_{v\in \cI(\cT_{i})}$ d'\'el\'ements de $k[T]\times\R_{>0}$ tels que l'arbre binaire de compacts $\cT_{i}(\cE_{i})$ sur $\overline{D}(P_{i},a_{i})$ satisfasse les propri\'et\'es de l'\'enonc\'e (y compris la propri\'et\'e finale si~$k$ est parfait ou de valuation non triviale). Quitte \`a diminuer~$c_{i}$, on peut supposer que, pour tout $b_{i} \in \intff{a_{i},c_{i}}$, l'arbre binaire de compacts $\cT_{i}(\cE_{i})$ sur $\overline{D}(P_{i},b_{i})$ satisfait encore ces propri\'et\'es. Remarquons que, pour tout $b_{i} \in \intof{a_{i},c_{i}}$ et tout n\oe ud interne de~$\cT_{i}$, on a 
\begin{equation}\label{eq:biendecoupe}\tag{a}
\{x\in \E{1}{V} : |P_{v}(x)| = s_{v}\} \cap \{x\in \E{1}{V} : |P_{i}(x)| = b_{i}\} = \emptyset.
\end{equation}

Rappelons que $U = \mathring{D}(P_{F};r,s)$. Il existe donc $r' \in \intoo{r,|P_{F}(x)|}$ tel que, en posant 
\[V := \{y\in \Aunk : r' \le |P_{F}(y)| \le |P_{F}(z)|\} \cap E \subset U,\]
on ait, pour tout~$i \in \cn{1}{p}$, $C_{i} \setminus V \subset  \mathring{D}(P_{i};c_{i})$. 

Soit $i\in\cn{1}{p}$. Il existe $b_{i} \in \intoo{a_{i},c_{i}}$ tel que l'on ait 
\begin{equation}\label{eq:intersection}\tag{b}
\overline{D}(P_{i};b_{i}) \cap V = \{y\in\Aunk : |P_{i}(y)| = b_{i}\}.
\end{equation}
On a alors
\begin{equation}\label{eq:V}\tag{c}
V = E \cap \bigcap_{1\le i\le p} \{y\in\Aunk : |P_{i}(y)| \ge b_{i}\}.
\end{equation}

D\'efinissons maintenant un arbre binaire de compacts sur~$E$. Commen\c{c}ons par consid\'erer l'arbre \og peigne \fg{} \`a $2p+1$ n{\oe}uds avec les dents \`a gauche, \cf~figure~\ref{fig:peigne}. 
Pour l'obtenir, on part d'un arbre r\'eduit \`a sa racine et on effectue~$p$ fois l'op\'eration qui consiste \`a ajouter un fils gauche et un fils droit \`a la feuille la plus \`a droite. Notons~$\cT_{0}$ cet arbre. Ses n{\oe}uds internes forment un ensemble totalement ordonn\'e $v_{0} \ge \dotsb \ge v_{p-1}$, o\`u~$v_{0}$ d\'esigne la racine et~$v_{p-1}$ le p\`ere des feuilles les plus basses.

\begin{figure}[!h]
\centering
\begin{tikzpicture}
\foreach \r in {0,...,1}
{\draw ({(\r)*cos(30)},-{(\r)*sin(30)}) -- ({(\r+1)*cos(30)},-{(\r+1)*sin(30)}) ;}
\foreach \r in {0,...,2}
{\fill ({(\r)*cos(30)},-{(\r)*sin(30)}) circle (2pt) ;
\draw ({(\r)*cos(30)},-{(\r)*sin(30)}) node[above right]{$v_{\r}$} ;
\draw ({(\r)*cos(30)},-{(\r)*sin(30)}) -- ({(\r)*cos(30) - cos(60)},-{(\r)*sin(30)-sin(60)});
\fill ({(\r)*cos(30) - cos(60)},-{(\r)*sin(30)-sin(60)}) circle (2pt) ;
\draw ({(\r)*cos(30) - cos(60)},-{(\r)*sin(30)-sin(60)}) node[below left]{$v_{\r}^-$} ;
}
\draw[dashed]  ({(2)*cos(30)},-{(2)*sin(30)}) -- ({(2+2)*cos(30)},-{(2+2)*sin(30)}) ;
\foreach \r in {4}
{\fill ({(\r)*cos(30)},-{(\r)*sin(30)}) circle (2pt) ;
\draw ({(\r)*cos(30)},-{(\r)*sin(30)}) -- ({(\r+1)*cos(30)},-{(\r+1)*sin(30)}) ;
\draw ({(\r+1)*cos(30)},-{(\r+1)*sin(30)}) node[above right]{$v_{p-1}^+$} ;
\draw ({(\r)*cos(30)},-{(\r)*sin(30)}) -- ({(\r)*cos(30) - cos(60)},-{(\r)*sin(30)-sin(60)});
\fill ({(\r)*cos(30) - cos(60)},-{(\r)*sin(30)-sin(60)}) circle (2pt) ;
\draw ({(\r)*cos(30) - cos(60)},-{(\r)*sin(30)-sin(60)}) node[below left]{$v_{p-1}^-$} ;
}
\fill ({(5)*cos(30)},-{(5)*sin(30)}) circle (2pt) ;
\draw ({(4)*cos(30)},-{(4)*sin(30)}) node[above right]{$v_{p-1}$} ;
\draw (-0.5,-3.3) node[above right]{$\cT_{0}$} ;
\end{tikzpicture}
\caption{Un arbre peigne \`a $2p+1$ n\oe uds.}\label{fig:peigne}
\end{figure}

Consid\'erons maintenant la famille $\cE_{0} = (P_{i},b_{i})_{1\le i\le p}$ et l'arbre binaire de compacts \'el\'ementaire $\cT_{0}(\cE_{0})$ sur~$E$ qui lui est associ\'e. Ses feuilles sont $v_{0}^-$ (\'etiquet\'ee par $\overline{D}(P_{1},b_{1})$), \dots, $v_{p-1}^-$ (\'etiquet\'ee par $\overline{D}(P_{p},b_{p})$) et~$v_{p-1}^+$ (\'etiquet\'ee par~$V$, d'apr\`es~\eqref{eq:V}). Puisque, pour tout $i \in \{1,\dotsc,p\}$, $\overline{D}(P_{i},b_{i})$ est contenu dans~$C_{i}$, $\cT_{0}(\cE_{0})$ est bien d\'ecoup\'e.

D'apr\`es~\eqref{eq:intersection}, toutes les intersections $K(v^-)\cap K(v^+)$ sont de la forme $\overline{C}(P_{i};b_{i},b_{i})$, avec $i\in \cn{1}{p}$. On d\'eduit alors de la proposition~\ref{prop:systemeCR} que~$\cT_{0}(\cE_{0})$ est un arbre de Cousin-Runge fort sur~$E$. 

Construisons maintenant un autre arbre binaire de compacts~$\cT$ sur~$E$ \`a partir de~$\cT_{0}(\cE_{0})$ en rempla\c{c}ant, pour tout $i\in\cn{1}{p}$, la feuille~$v_{i-1}^-$, \'etiquet\'ee par $\overline{D}(P_{i},b_{i})$, par l'arbre~$\cT_{i}(\cE_{i})$. L'arbre~$\cT$ est encore un arbre de Cousin-Runge fort sur~$E$. On v\'erifie ais\'ement que c'est un arbre \'el\'ementaire, associ\'e \`a la famille obtenue en r\'eunissant~$\cE_{0}$ et les~$\cE_{i}$, pour $i\in\cn{1}{p}$.

L'arbre~$\cT$ est adapt\'e \`a~$\cU$ car chacune de ses feuilles est soit une feuille de l'un des~$\cT_{i}(\cE_{i})$, pour $i\in \cn{1}{p}$, soit~$V$. Le fait qu'il soit bien d\'ecoup\'e d\'ecoule du fait que $\cT_{0}(\cE_{0})$ est bien d\'ecoup\'e et de~\eqref{eq:biendecoupe}.
\end{proof}

\subsection{Au-dessus d'une base ultram\'etrique} Soit $X$ un espace $\cA$-analytique. 

Nous souhaitons montrer que la propri\'et\'e de poss\'eder assez d'arbres de Cousin-Runge se transf\`ere d'une partie de l'espace~$X$ \`a un disque ferm\'e relatif au-dessus de cette partie.

\begin{coro}\label{cor:CREx}\index{Arbre!sur une \'etoile}\index{Resultant@R\'esultant}
Soit~$x$ un point d\'ecent (resp. tr\`es d\'ecent) de~$X_{\um}$. Posons $\A^1_{X} := X \times_{\cA} \AunA$ et notons $\pi \colon \A^1_{X} \to X$ le morphisme de projection. Soit~$E$ une \'etoile de~$\pi^{-1}(x)$. Alors, pour toute famille d'ouverts~$\cU$ de~$\A^1_{X}$ recouvrant~$E$, il existe un arbre binaire~$\cT$ et une famille $\cE = (P_{v},s_{v})_{v\in \cI(\cT)}$ d'\'el\'ements de $\cO_{x}[T]\times\R_{>0}$ index\'ee par l'ensemble des n{\oe}uds internes de~$\cT$ tels que l'arbre binaire de compacts \'el\'ementaire $\cT(\cE)$ sur~$E$ soit bien d\'ecoup\'e, adapt\'e \`a~$\cU$ et de Cousin-Runge (resp. de Cousin-Runge fort). De plus, si $f \ne g$ sont des feuilles de~$\cT$, alors $K(f) \cap K(g)$ est une \'etoile de~$\pi^{-1}(x)$.

En outre, si~$\cH(x)$ est parfait ou de valuation non triviale, on peut supposer que, pour tout n{\oe}ud interne~$v$ de~$\cT$, le polyn\^ome $R_{v}(T) = \Res_{S}(P_{v}(S)-T,P_{v}'(S))$ ne s'annule pas sur $\overline{C}(s_{v},s_{v})$.
\end{coro}
\begin{proof}
D'apr\`es la proposition~\ref{prop:CRcorps} appliqu\'ee \`a~$E$, il existe un arbre binaire~$\cT$ et une famille $\cE = (P_{v},s_{v})_{v\in \cI(\cT)}$ d'\'el\'ements de $\cH(x)[T]\times\R_{>0}$, les $P_{v}$ \'etant irr\'eductibles, tels que l'arbre binaire de compacts \'el\'ementaire $\cT(\cE)$ sur~$E$ soit bien d\'ecoup\'e et adapt\'e \`a~$\cU$. Si~$\cH(x)$ est parfait ou de valuation non triviale, on peut supposer que, pour tout n{\oe}ud interne~$v$ de~$\cT$, le polyn\^ome $R_{v}(T) = \Res_{S}(P_{v}(S)-T,P_{v}'(S))$ ne s'annule pas sur $\overline{C}(s_{v},s_{v})$.

L'image de~$\cO_{X,x}$ dans~$\cH(x)$ \'etant dense, on peut modifier les polyn\^omes~$P_{v}$, o\`u~$v$ est un n{\oe}ud interne de~$\cT_{x}$, de fa\c{c}on que leurs coefficients appartiennent \`a~$\cO_{X,x}$ sans modifier les ensembles $\{x\in \E{1}{\cH(x)} : |P_{v}(x)| \bowtie s_{v}\}$, avec $\bowtie\, \in \{\le, \ge, =\}$. Les propri\'et\'es de l'arbre~$\cT_{x}(\cE_{x})$ sont alors pr\'eserv\'ees. Puisque les \'etiquettes n'ont pas \'et\'e modif\'ees, les intersections des \'etiquettes de feuilles distinctes sont encore soit vides, soit des \'etoiles (\cf~remarque~\ref{rem:biendecoupe}). Si~$\cH(x)$ est parfait ou de valuation non triviale, on peut \'egalement supposer que la condition de non annulation du r\'esultant \'enonc\'ee ci-dessus reste satisfaite. 

Pour tout n{\oe}ud interne~$v$ de~$\cT_{x}(\cE_{x})$, la proposition~\ref{prop:systemeCR} assure qu'il existe un syst\`eme de Cousin-Runge associ\'e \`a $(K(v^-),K(v^+))$. Si~$x$ est tr\`es d\'ecent, il s'agit m\^eme d'un syst\`eme de Cousin-Runge fort, d'apr\`es le corollaire~\ref{cor:BVcompactN} si~$x$ est tr\`es typique et le lemme~\ref{lem:conditionscar0} sinon. 
\end{proof}

\begin{coro}\label{cor:CRfibreum}\index{Arbre!sur un disque}\index{Disque!arbre sur un|see{Arbre}}
Soit~$V$ une partie compacte et d\'ecente de~$X_{\um}$ et soit $\br \in \R_{\ge 0}^n$. Si~$V$ poss\`ede assez d'arbres de Cousin (resp. de Cousin-Runge), alors il en va de m\^eme pour $\overline{D}_{V}(\br)$.
\end{coro}
\begin{proof}
En raisonnant par r\'ecurrence, on se ram\`ene imm\'ediatement au cas o\`u~$n=1$. Notons $r=r_{1}$. Supposons que~$V$ poss\`ede assez d'arbres de Cousin (resp. de Cousin-Runge). Soit~$\cU$ un recouvrement ouvert de~$\overline{D}_{V}(r)$. 

Soit~$x$ un point de~$X$. D'apr\`es le corollaire~\ref{cor:CREx} appliqu\'e \`a $\overline{D}_{x}(r)$, il existe un arbre binaire~$\cT_{x}$ et une famille $\cE_{x} = (P_{v},s_{v})_{v\in \cI(\cT_{x})}$ d'\'el\'ements de $\cO_{x}[T]\times\R_{>0}$ tels que l'arbre binaire de compacts \'el\'ementaire $\cT_{x}(\cE_{x})$ sur~$\overline{D}_{x}(r)$ soit bien d\'ecoup\'e et adapt\'e \`a~$\cU$. 

Consid\'erons un voisinage ouvert~$V_{x}$ de~$x$ dans~$V$ sur lequel les coefficients des polyn\^omes~$P_{v}$ sont d\'efinis. Soit~$W$ un voisinage compact de~$x$ dans~$V_{x}$. 
Pour tout n\oe ud interne~$v$ de~$\cT_{x}$, on notera encore~$P_{v}$ un repr\'esentant de~$P_{v}$ dans $\cO(W)[T]$. Consid\'erons la famille  $\cE_{W} = (P_{v},s_{v})_{v\in \cI(\cT_{x})}$ de $\cO(W)[T]\times\R_{>0}$ et l'arbre de compacts \'el\'ementaire $\cT_{x}(\cE_{W})$ associ\'e sur $\overline{D}_{W}(r)$. Puisque~$\cT_{x}(\cE_{x})$ est bien d\'ecoup\'e et adapt\'e \`a~$\cU$, quitte \`a restreindre~$V_{x}$, on peut \'egalement supposer que $\cT_{x}(\cE_{W})$ est bien d\'ecoup\'e et adapt\'e \`a~$\cU$. 

En outre, pour tout n{\oe}ud interne~$v$ de~$\cT_{x}(\cE_{W})$, d'apr\`es la remarque~\ref{rem:biendecoupe} et la proposition~\ref{prop:systemeCR}, il existe un syst\`eme de Cousin (resp. Cousin-Runge) associ\'e \`a $(K(v^-),K(v^+))$. 

La famille $\cV = (V_{x})_{x\in X}$ forme un recouvrement ouvert de~$V$. Puisque~$V$ poss\`ede assez d'arbres de Cousin (resp. Cousin-Runge), il existe un arbre de Cousin (resp. Cousin-Runge) $\cT$ sur~$V$ adapt\'e \`a~$\cV$. Consid\'erons l'arbre binaire de compacts~$\cT'$ obtenu en rempla\c{c}ant dans~$\cT$ chaque \'etiquette~$K$ par l'\'etiquette $\overline{D}_{K}(r)$. D'apr\`es le lemme~\ref{lem:CRrelatif}, $\cT'$ est encore un arbre de Cousin (resp. de Cousin-Runge).

Finalement, pour toute feuille~$f$ de~$\cT$, consid\'erons un point~$x_{f}$ de~$V$ tel que $K(f) \subset V_{x_{f}}$ et rempla\c{c}ons la feuille correspondante de~$\cT'$ par l'arbre~$\cT_{x_{f}}(\cE_{K(f)})$. Le r\'esultat est un arbre de Cousin (resp. de Cousin-Runge) sur~$\overline{D}_{X}(r)$ adapt\'e \`a~$\cU$. 
\end{proof}

\begin{coro}\label{cor:CRfortfibreum}\index{Arbre!sur un disque}
Soit~$V$ une partie compacte de~$X_{\um}$ et soit $\br \in \R_{\ge 0}^n$.
Supposons que l'une des conditions suivantes est satisfaite~:
\begin{enumerate}[i)]
\item tous les points de~$V$ sont ultram\'etriques tr\`es typiques~;
\item en tout point~$b$ de~$V$, le corps r\'esiduel compl\'et\'e~$\cH(b)$ est de caract\'eristique nulle ou de valuation non triviale.
\end{enumerate}
Si~$V$ poss\`ede assez d'arbres de Cousin forts (resp. de Cousin-Runge forts), alors il en va de m\^eme pour $\overline{D}_{V}(\br)$.
\end{coro}
\begin{proof}
Reprenons le d\'ebut de la d\'emonstration du corollaire~\ref{cor:CRfibreum} pour construire des arbres binaires de compacts~$\cT_{x}(\cE_{x})$ et~$\cT_{x}(\cE_{W})$. 

Soit~$v$ un n\oe ud interne de~$\cT_{x}(\cE_{W})$. Pour poursuivre, on a besoin d'un syst\`eme de Cousin fort ou Cousin-Runge fort associ\'e \`a $(K(v^-),K(v^+))$. Dans le cas~i), son existence d\'ecoule de la remarque~\ref{rem:biendecoupe}, la proposition~\ref{prop:systemeCR} et le corollaire~\ref{cor:BVcompactN}. 

Dans le cas~ii), remarquons tout d'abord que l'assertion finale du corollaire~\ref{cor:CREx} permet de supposer que le polyn\^ome $R_{v}(T) = \Res_{S}(P_{v}(S)-T,P_{v}'(S))$ ne s'annule pas sur~$\overline{C}_{x}(s_{v},s_{v})$. Quitte \`a restreindre~$V_{x}$, on peut supposer qu'ils ne s'annule pas non plus sur~$\overline{C}_{W}(s_{v},s_{v})$. L'existence d'un syst\`eme de Cousin fort ou de Cousin-Runge fort d\'ecoule alors de la remarque~\ref{rem:biendecoupe}, de la proposition~\ref{prop:systemeCR} et du lemme~\ref{lem:conditionscar0}. 

On peut alors conclure en reprenant la fin de la d\'emonstration du corollaire~\ref{cor:CRfibreum}.
\end{proof}
\index{Arbre|)}

\section{Th\'eor\`emes A et B sur les polydisques ferm\'es relatifs}\label{sec:AB}

Dans cette section, nous d\'emontrons que les polydisques ferm\'es relatifs sur des bases convenables satisfont les th\'eor\`emes~A et~B. Nous commencerons par traiter le cas o\`u la base est un point, archim\'edien d'abord, ultram\'etrique ensuite, avant de traiter le cas d'une base plus grande.

Rappelons que, suivant nos conventions, les diff\'erents disques et polydisques relatifs que nous rencontrons dans cette section sont munis de la structure surconvergente induite par leur plongement dans l'espace affine relatif. En particulier, la structure d'espace localement annel\'e d'un disque relatif au-dessus d'un point n'est pas, en g\'en\'eral, celle du disque sur un corps valu\'e qui lui est sous-jacent.

\subsection{Fibres archim\'ediennes}

Commen\c{c}ons par \'etablir les th\'eor\`emes~A et~B pour les fibres archim\'ediennes de disques de Berkovich. Ils d\'ecoulent assez directement de r\'esultats classiques dans le cadre des espaces analytiques complexes.

\begin{prop}\label{prop:C}
Munissons le corps~$\C$ muni de la valeur absolue~$\va_{\infty}^\eps$, avec $\eps\in\intof{0,1}$. Soit $\br \in\R_{\ge 0}^n$ et soit~$\cF$ un faisceau coh\'erent sur $\overline{D}_{\C}(\br)$. Alors, $\cF$ est globalement engendr\'e et, pour tout $q\ge 1$, on a
$H^q(\overline{D}_{\C}(\br),\cF)=0$.
\end{prop}
\begin{proof}
Le fait que la valeur absolue puisse \^etre une puissance diff\'erente de la valeur absolue usuelle n'influe pas sur le r\'esultat et on peut supposer que $\eps=1$, ce qui nous ram\`ene au cas des espaces analytiques complexes usuels. Il est alors bien connu que les disques ferm\'es sont des espaces de Stein et qu'ils satisfont les propri\'et\'es de l'\'enonc\'e. 
\end{proof}

Nous allons maintenant nous int\'eresser aux espaces analytiques sur~$\R$ (muni de la valeur absolue~$\va^\eps$, avec $\eps\in\intof{0,1}$) au sens de V.~Berkovich. Rappelons que l'extension des scalaires $\R \to \C$ permet d'associer \`a tout espace de Berkovich~$Y$ sur~$\R$ un espace analytique complexe~$Y_{\C}$ et un morphisme $\pi \colon Y_{\C} \to Y$, qui n'est autre que le quotient par la conjugaison complexe (\cf~lemme~\ref{lem:AnC}). En particulier, pour tout point~$y$ de~$Y$, la fibre $\pi^{-1}(y)$ contient un point si $\Hs(y)\simeq \R$ et deux points conjugu\'es si $\Hs(y)\simeq \C$. 

L'ingr\'edient cl\'e pour passer de~$\C$ \`a~$\R$ est le r\'esultat suivant, d\^u \`a Q.~Liu (\cf~\cite[proposition~2]{LiuContre-Exemple}) et \'enonc\'e sous cette forme dans \cite[th\'eor\`eme~6.11]{A1Z} (\`a une hypoth\`ese manquante pr\`es).

\begin{lemm}\label{lem:Liu}\index{Theoreme@Th\'eor\`eme!B}
Soit $\varphi \colon Z \to Y$ un morphisme surjectif d'espaces analytiques. Supposons que l'espace~$Z$ et les ferm\'es de Zariski de~$Y$ de support diff\'erent de~$Y$ satisfont le th\'eor\`eme~B. Supposons \'egalement qu'il existe un point~$y$ tel que~$\{y\}$ soit l'espace topologique sous-jacent \`a un ferm\'e analytique de~$Y$ et que la fibre $(\varphi_{\ast}\cO_{Z})_{y}$ soit un $\cO_{Y,y}$-module libre de rang fini. Alors, l'espace~$Y$ satisfait le th\'eor\`eme~B.
\qed
\end{lemm}

\begin{prop}\label{prop:R}
Consid\'erons le corps~$\R$ muni de la valeur absolue~$\va_{\infty}^\eps$, avec $\eps\in\intof{0,1}$. Soit $\br \in \R_{\ge 0}^n$ et soit~$\cF$ un faisceau coh\'erent sur $\overline{D}_{\R}(\br)$. Alors, $\cF$ est globalement engendr\'e et, pour tout $q\ge 1$, on a
$H^q(\overline{D}_{\R}(\br),\cF)=0$.
\end{prop}
\begin{proof}
Munissons $\overline{D}_{\R}(\br)$ et $\overline{D}_{\C}(\br)$ de la topologie dont les ferm\'es sont les ferm\'es analytiques. D'apr\`es \cite[proposition I,7]{Frisch}, l'espace topologique $\overline{D}_{\C}(\br)$ est noeth\'erien et il en va donc de m\^eme de l'espace $\overline{D}_{\R}(\br)$ sur lequel il se surjecte contin\^ument.

Supposons, par l'absurde, qu'il existe un ferm\'e analytique de $\overline{D}_{\R}(\br)$ ne satisfaisant pas le th\'eor\`eme~B. Par noeth\'erianit\'e, il en existe un qui soit minimal pour cette propri\'et\'e~; notons-le~$Y$. L'espace~$Y_{\C}$ est n\'ecessairement de dimension strictement positive et l'espace~$Y$ contient donc un point~$y$ donc le corps r\'esiduel est isomorphe \`a~$\C$. Le morphisme $\pi \colon Y_{\C} \to Y$ d\'efinit alors un isomorphisme d'un voisinage de chacun des deux ant\'ec\'edents de~$y$ vers un voisinage de~$y$. Par cons\'equent, la fibre $(\pi_{\ast}\cO_{Y_{\C}})_{y}$ est un $\cO_{Y,y}$-module libre de rang deux. D'apr\`es le lemme~\ref{lem:Liu}, $Y$ satisfait le th\'eor\`eme~B et nous aboutissons \`a une contradiction.

Il est par ailleurs classique que, sur un espace qui satisfait le th\'eor\`eme~B et dont tout point est l'espace topologique sous-jacent \`a un ferm\'e analytique, tout faisceau coh\'erent est globalement engendr\'e (\cf~\cite[IV, \S 1, theorem~2]{Gr-Re} ou~\cite[th\'eor\`eme~6.1.9]{A1Z} -- l\`a encore, il manque une hypoth\`ese).
\end{proof}

Soit $K$ le corps $\R$ ou~$\C$ muni de la norme $\nm_{\hyb} = \max(\va_{0},\va_{\infty})$, o\`u~$\va_{0}$ d\'esigne la valeur absolue triviale et~$\va_{\infty}$ la valeur absolue usuelle. Rappelons que le spectre~$\Mc(K)$ est constitu\'e des points associ\'es aux valeurs absolues~$\va_{\infty}^\eps$ pour $\eps\in[0,1]$ (en identifiant~$\va_{\infty}^0$ \`a~$\va_{0}$), \cf~exemple~\ref{ex:corpshybride}.

\begin{prop}\label{prop:fibrearc}
Soit $K$ le corps $\R$ ou~$\C$ muni de la norme $\nm_{\hyb}$. Posons $X=\Mc(K)$. Soit~$b$ un point de~$X$ associ\'e \`a une valeur absolue archim\'edienne (c'est-\`a-dire $\va_{\infty}^\eps$, avec $\eps\in\intof{0,1}$). Soit $\br \in \R_{\ge 0}^n$ et soit $\cF$~un faisceau coh\'erent sur $\overline{D}_{b}(\br)$. Alors, $\cF$ est globalement engendr\'e et, pour tout $q\ge 1$, on a 
$H^q(\overline{D}_{b}(\br),\cF)=0$.
\end{prop}
\begin{proof}
D'apr\`es \cite[proposition~3.4.6]{A1Z}, l'inclusion 
\[j_{b} \colon \overline{D}_{b}(\br) \lhook\joinrel\too \overline{D}_{X}(\br)\]
de la fibre dans l'espace total induit un isomorphisme d'espaces annel\'es
\[\bigl(\overline{D}_{b}(\br), j_{b}^{-1} \cO_{\overline{D}_{X}(\br)}\bigr) \simtoo (\overline{D}_{b}(\br),\cO_{\overline{D}_{b}(\br)}).\]
Il suffit donc de d\'emontrer les r\'esultats sur l'espace de gauche, qui est un espace analytique sur un corps valu\'e. Les propositions~\ref{prop:C} et~\ref{prop:R} permettent alors de conclure.
\end{proof}


\subsection{Fibres ultram\'etriques}

Soit~$X$ un espace $\cA$-analytique. 

\begin{prop}\label{prop:H1}
Soit~$V$ une partie compacte de~$X$ et soit $\br \in \R_{\ge 0}^n$. 

Supposons que~$V$ poss\`ede assez d'arbres de Cousin et que, pour tout point~$x$ de~$V$, on a $H^1(\oD_{x}(\br),\cO)=0$. 
Alors, on a $H^{1}(\oD_{V}(r),\cO) = 0$. 

En particulier, si~$V$ poss\`ede assez d'arbres de Cousin, on a $H^1(V,\cO)=0$.
\end{prop}
\begin{proof}
Rappelons que la catégorie des faisceaux en groupes abéliens sur~$\oD_{V}(\br)$, comme sur tout site, poss\`ede assez d'injectifs (\cf~\cite[\href{https://stacks.math.columbia.edu/tag/01DP}{Theorem 01DP}]{stacks-project}). Soit 
\[0 \too \cO \too \cI_{0} \xlongrightarrow[]{d} \cI_{1} \xlongrightarrow[]{d} \dotsc\]
une r\'esolution injective du faisceau~$\cO$ sur~$\oD_V(\br)$. Soit $\alpha\in \cI_{1}(\oD_{V}(\br))$ tel que $d\alpha=0$.

Soit~$x$ un point de~$V$. Par hypoth\`ese, on a $H^{1}(\oD_{x}(\br),\cO) = 0$, donc il existe un \'el\'ement $\beta_{x} \in \cI_{0}(\oD_{x}(\br))$ tel que $d \beta_{x} = \alpha_{|\oD_{x}(\br)}$. Par d\'efinition, il existe un voisinage~$U_{x}$ de~$\oD_{x}(\br)$ et un \'el\'ement $\gamma_{x}$ de $\cI_{0}(U_{x})$ dont la restriction \`a~$\oD_{x}(\br)$ est \'egale \`a~$\beta_{x}$. En outre, $d \gamma_{x}$ est \'egale \`a~$\alpha$ sur un voisinage de~$\oD_{x}(\br)$, que l'on peut supposer \^etre \'egal \`a~$U_{x}$. En d'autres termes, l'image de~$\alpha$ dans $H^{1}(U_{x},\cO)$ est nulle. Puisque~$\pi$ est propre, on peut supposer que~$U_{x}$ est de la forme $\oD_{V_{x}}(\br)$, o\`u~$V_{x}$ est un voisinage ouvert de~$x$ dans~$X$. 

Consid\'erons le recouvrement ouvert $\cV = (V_{x})_{x\in V}$ de~$V$. Par hypoth\`ese, il existe un arbre de Cousin~$\cT$ sur~$V$ adapt\'e \`a~$\cU$. Par construction, pour toute feuille~$f$ de~$\cT$, l'image de~$\alpha$ dans $H^1(\oD_{K(f)}(\br),\cO)$ est nulle.

En remontant pas \`a pas dans l'arbre, on d\'emontre plus g\'en\'eralement que, pour tout n{\oe}ud~$v$ de~$\cT$, l'image de~$\alpha$ dans $H^1(\oD_{K(v)}(\br),\cO)$ est nulle. En effet, supposons le r\'esultat d\'emontr\'e pour les deux fils~$v^-$ et~$v^+$ d'un n{\oe}ud~$v$ de~$\cT$. Par hypoth\`ese, il existe un syst\`eme de Cousin associ\'e \`a $(K(v^-),K(v^+))$. En particulier, le morphisme
\[\begin{array}{rcl}
 \cO(K(v^-))\times \cO(K(v^+)) &\too & \cO(K(v^-) \cap K(v^+))\\
 (f^-,f^+) & \mapstoo & f^--f^+
 \end{array} \]
est surjectif et il en va donc de m\^eme du morphisme
\[\begin{array}{rcl}
 \cO(\oD_{K(v^-)}(\br))\times \cO(\oD_{K(v^+)}(\br)) &\too & \cO(\oD_{K(v^-)}(\br) \cap\oD_{K(v^+)}(\br))\\
 (f^-,f^+) & \mapstoo & f^--f^+
 \end{array},\]
d'apr\`es le lemme~\ref{lem:CRrelatif}. Le r\'esultat d\'ecoule alors de la suite exacte longue de Mayer-Vietoris
\[\hspace{-170pt}
\begin{tikzcd}
\dotsb \longrightarrow H^{0}(\oD_{K(v^-)}(\br),\cO) \oplus H^{0}(\oD_{K(v^+)(\br)},\cO) 
\ar[out=0, in=180, looseness=2]{d}\\
H^0(\oD_{K(v^-)}(\br) \cap \oD_{K(v^+)}(\br),\cO) \longrightarrow H^{1}(\oD_{K(v)}(\br),\cO) \ar[out=0, in=180, looseness=2]{d}\\
 H^{1}(\oD_{K(v^-)}(\br),\cO) \oplus H^{1}(\oD_{K(v^+)}(\br),\cO) \longrightarrow \dotsb
\end{tikzcd}\]

Puisque la racine de~$\cT$ est \'etiquet\'ee par le compact~$V$, ceci d\'emontre que l'image de~$\alpha$ dans $H^1(\oD_{V}(\br),\cO)$ est nulle.

\smallbreak

La partie finale du r\'esultat d\'ecoule du cas~$n=0$ en remarquant que, pour tout point~$x$ de~$V$, on a $H^1(\{x\},\cO)=0$.
\end{proof}

\begin{lemm}\label{lem:Hq}
Soit~$V$ une partie compacte de~$X$ et soit $\br \in \R_{\ge 0}^n$. Soit~$\cF$ un faisceau de groupes ab\'eliens sur~$\oD_{V}(\br)$. Soit $q\ge 1$.

Supposons que~$V$ est de dimension de recouvrement~1 et que, pour tout point~$x$ de~$V$, on a $H^q(\oD_{x}(\br),\cF)=H^{q+1}(\oD_{x}(\br),\cF)=0$. 
Alors, on a $H^{q+1}(\oD_{V}(\br),\cF) = 0$. 
\end{lemm}
\begin{proof}
Consid\'erons le morphisme de projection $\pi \colon \E{n}{X} \to X$. Il induit un morphisme d'espaces localement annel\'es $\pi_{V,\br} \colon \oD_{V}(\br) \to V$, o\`u $\oD_{V}(\br)$ et $V$ sont munis des structures surconvergentes induites respectivement par~$\E{n}{X}$ et~$X$. Consid\'erons la suite spectrale de Leray
\[ E_{2}^{p,p'} = H^p(V,R^{p'} (\pi_{V,\br})_{\ast} \cF) \implies H^{p+p'}(\oD_{V}(\br),\cF). \]
Puisque~$V$ est de dimension~1, pour tout faisceau de groupes ab\'eliens~$\cG$ sur~$V$ et tout $p'\ge 2$, on a $H^{p'}(V,\cG)=0$. En particulier, seules les deux premi\`eres colonnes de la deuxi\`eme page de la suite spectrale peuvent contenir des termes non nuls. En outre, par hypoth\`ese, on a $R^q (\pi_{V,\br})_{\ast}\cF = R^{q+1} (\pi_{V,\br})_{\ast}\cF = 0$. Le r\'esultat s'en d\'eduit.
\end{proof}

%

\begin{coro}\label{cor:HqO}
Soit~$x$ un point d\'ecent de~$X_{\um}$.  Alors, pour tout $\br \in \R_{\ge 0}^n$ et tout $q\ge 1$, on a
\[H^q(\overline{D}_{x}(\br),\cO) = 0.\] 
\end{coro}
\begin{proof}
D\'emontrons le r\'esultat par r\'ecurrence sur~$n$. Pour $n=0$, on a bien $H^q(\{x\},\cO) = 0$ pour tout $q\ge 1$.

Supposons que le r\'esultat est vrai pour les disques de dimension~$n$ et d\'emontrons-le pour la dimension~$n+1$. Soit $\br = (r_{1},\dotsc,r_{n+1}) \in \R_{\ge 0}^{n+1}$.

D'apr\`es le corollaire~\ref{cor:CRfibreum}, le disque relatif~$\overline{D}_{x}(\br)$ poss\`ede assez d'arbres de Cousin. D'apr\`es la proposition~\ref{prop:H1}, on a donc $H^1(\overline{D}_{x}(\br),\cO)=0$.

Consid\'erons le morphisme de projection sur la derni\`ere coordonn\'ee $\pi \colon \overline{D}_{x}(\br) \to \overline{D}_{x}(r_{n+1})$. Posons $\br' := (r_{1},\dotsc,r_{n})$. Par hypoth\`ese de r\'ecurrence, pour tout point~$y$ de $\overline{D}_{x}(r_{n+1})$ et tout $q\ge 1$, on a $H^q(\oD_{y}(\br'),\cO)=0$. Le lemme~\ref{lem:Hq} assure alors que, pour tout $q\ge 2$, on a $H^q(\overline{D}_{x}(\br),\cO)=0$.
\end{proof}

%
%
%

\subsection{Cas g\'en\'eral}

Soit~$X$ un espace $\cA$-analytique.

Nous souhaitons montrer que les polydisques ferm\'es relatifs au-dessus de certaines parties de~$X$ d'un type particulier satisfont les th\'eor\`emes~A et~B. Introduisons-les maintenant. 

\begin{defi}\label{def:basedeStein}\index{Base de Stein|textbf}
Une \emph{base de Stein} de~$X$ est une partie compacte~$V$ de~$X$ qui satisfait les conditions suivantes~:
\begin{enumerate}[i)]
\item $V$ est de dimension de recouvrement inf\'erieure \`a~1~;
\item tout point ultram\'etrique de~$V$ est d\'ecent~; 
\item tout point archim\'edien de~$V$ est isol\'e ou poss\`ede un voisinage lin\'eaire dans~$V$~;
\item tout polydisque ferm\'e relatif sur~$V$ est de dimension de recouvrement finie~;
\item pour tout recouvrement ouvert~$\cU$ de~$V$, il existe un arbre de Cousin-Runge fort \`a intersections simples~$\cT_{\cU}$ sur~$V$ adapt\'e \`a~$\cU$ tel que, pour toutes feuilles $f\ne g$ de~$\cT_{\cU}$, le compact $K(f)\cap K(g)$ poss\`ede assez d'arbres de Cousin.
\end{enumerate}
\end{defi}

\begin{exem}\label{ex:basedeStein}
\index{Corps!valu\'e}\index{Anneau!des entiers relatifs $\Z$}\index{Anneau!des entiers d'un corps de nombres}\index{Corps!hybride}\index{Anneau!de valuation discr\`ete}\index{Anneau!de Dedekind trivialement valu\'e}

Soit $\cA$ l'un de nos anneaux de Banach usuels~: corps valu\'es, anneaux d'entiers de corps de nombres, corps hybrides, anneaux de valuation discr\`ete, anneaux de Dedekind trivialement valu\'es (\cf~exemples~\ref{ex:corpsvalue} \`a~\ref{ex:Dedekind}). Alors, toute partie compacte et connexe de~$\cM(\cA)$ est une base de Stein de~$\cM(\cA)$.

En effet, soit~$\cU$ un recouvrement ouvert de~$V$. En utilisant la description explicite de~$\cM(\cA)$, on montre que l'on peut raffiner~$\cU$ en un recouvrement compact fini~$\cV$ tel que l'intersection de deux parties quelconques de~$\cV$ soit vide ou r\'eduite \`a un point non associ\'e \`a la valuation triviale sur~$\cA$. On peut alors associer un syst\`eme de Cousin-Runge fort par restriction, \`a partir du proc\'ed\'e d\'ecrit dans les exemples~\ref{ex:Cousin} et~\ref{ex:CousinRunge}. Le r\'esultat s'en d\'eduit.
\end{exem}

\begin{lemm}\label{lem:Dxrlineaire}
Soit $x$ un point tr\`es d\'ecent de~$X_{\um}$. Alors, pour tout $r\in \R_{\ge 0}$, la partie compacte $\overline{D}_{x}(r)$ de~$\A^1_{X}$ est une base de Stein.
\end{lemm}
\begin{proof}
Soit $r\in \R_{\ge 0}$. Les points~i) et~iv) de la d\'efinition~\ref{def:basedeStein} d\'ecoulent du th\'eor\`eme~\ref{th:dimensioncorps}. Le point~ii) suit de la proposition~\ref{prop:typiqueAn}. Le point~iii) est trivialement v\'erifi\'e. 

D\'emontrons le point~v). Soit~$\cU$ un recouvrement ouvert de~$V$. D'apr\`es le corollaire~\ref{cor:CREx}, il existe un arbre binaire de compacts \'el\'ementaire~$\cT_{\cU}$ sur~$V$ qui soit bien d\'ecoup\'e, adapt\'e \`a~$\cU$ et de Cousin-Runge fort. D'apr\`es la remarque~\ref{rem:biendecoupe}, il est \`a intersections simples. Par construction, pour toutes feuilles $f\ne g$ de~$\cT_{\cU}$, le compact $K(f)\cap K(g)$ est une \'etoile, donc poss\`ede assez d'arbres de Cousin, d'apr\`es le corollaire~\ref{cor:CREx}.
\end{proof}

\begin{coro}\label{cor:HqVO}
Soit~$V$ une base de Stein de~$X$ et soit $\br \in \R_{\ge 0}^n$. Alors, pour tout $q\ge 1$, on a 
\[H^q(\overline{D}_{V}(\br),\cO) = 0.\] 
\end{coro}
\begin{proof}
D'apr\`es la proposition~\ref{prop:fibrearc} et le corollaire~\ref{cor:HqO}, pour tout point~$x$ de~$V$, on a $H^q(\overline{D}_{x}(\br),\cO) = 0$. Le r\'esultat d\'ecoule alors de la proposition~\ref{prop:H1} et du lemme~\ref{lem:Hq}.  
\end{proof}


Rappelons que, d'apr\`es le th\'eor\`eme~\ref{th:dimensionetoile}, si~$V$ est une partie lin\'eaire ou d'Ostrowski de~$X$, alors, tout disque ferm\'e relatif sur~$V$ est de dimension de recouvrement finie.

\begin{coro}\label{cor:HqVFresolution}
Soit~$V$ une base de Stein de~$X$ et soit $\br \in \R_{\ge 0}^n$. Soit~$\cF$ un faisceau de $\cO$-modules sur $\overline{D}_{V}(\br)$. Supposons qu'il existe une r\'esolution libre 
\[\cO^{N_{d}} \too \cO^{N_{d-1}} \too \dotsb \too \cO^{N_{1}} \too \cF \too 0,\] 
o\`u $d := \dimr(\overline{D}_{V}(\br))$. Alors, pour tout $q\ge 1$, on a 
\[H^q(\overline{D}_{V}(\br),\cF) = 0.\] 
\end{coro}
\begin{proof}
Pour $m \in \cn{1}{d}$, notons $\cR_{m}$ le noyau de $\cO^{N_{m}} \to \cF$. Par r\'ecurrence sur~$m$, en utilisant des suites exactes longues de cohomologie et le corollaire~\ref{cor:HqVO}, on d\'emontre que, pour tout $m \in \cn{1}{d}$ et tout $q\ge 1$, on a
\[H^{q}(\overline{D}_{V}(\br),\cF) = H^{q+m}(\overline{D}_{V}(\br),\cR_{m}).\] 
Pour $m=d$, ces groupes sont nuls car~$\overline{D}_{V}(\br)$ est de dimension de recouvrement~$d$.
\end{proof}


\begin{prop}\label{prop:thArecurrence}
Soit~$V$ une base de Stein de~$X$ et soit $\br \in \R_{\ge 0}^n$.

Supposons que, pour tout point~$x$ de~$V$, tout faisceau coh\'erent sur~$\oD_{x}(\br)$ soit globalement engendr\'e. Alors, tout faisceau coh\'erent sur~$\oD_{V}(\br)$ est globalement engendr\'e. 
\end{prop}
\begin{proof}
Soit~$\cF$ un faisceau coh\'erent sur~$\oD_{V}(\br)$.
Soit~$x$ un point de~$V$. Par hypoth\`ese, il existe un morphisme de $\cO_{\oD_{x}(\br)}$-modules surjectif $\varphi_{x} \colon \cO_{\oD_{x}(\br)}^{N_{x}} \to \cF_{|\oD_{x}(\br)}$. Ce dernier s'\'etend en un morphisme surjectif $\psi_{x} \colon \cO^{N_{x}} \to \cF$ d\'efini sur un voisinage de~$\oD_{x}(\br)$, que l'on peut supposer de la forme~$\oD_{V_{x}}(\br)$, o\`u~$V_{x}$ est un voisinage ouvert de~$x$ dans~$V$. Notons~$\cR$ le noyau de~$\psi_{x}$. Le faisceau~$\cO$ est coh\'erent d'apr\`es le th\'eor\`eme~\ref{coherent}, le faisceau~$\cF$ l'est aussi par hypoth\`ese, donc le faisceau~$\cR$ l'est \'egalement.

En appliquant de fa\c{c}on r\'ep\'et\'ee l'hypoth\`ese de l'\'enonc\'e, on montre le faisceau~$\cR_{|\oD_{x}(\br)}$ poss\`ede une r\'esolution libre 
\[\cO_{\oD_{x}(\br)}^{N_{d}} \too \cO_{\oD_{x}(\br)}^{N_{d-1}} \too \dotsb \too \cO_{\oD_{x}(\br)}^{N_{1}} \too \cR_{|\oD_{x}(\br)} \too 0,\] 
o\`u~$d$ est la dimension de recouvrement de~$\oD_{V}(\br)$. Cette r\'esolution s'\'etend en une r\'esolution de~$\cR$ sur un voisinage de~$\oD_{x}(\br)$, que l'on peut supposer de la forme~$\oD_{V'_{x}}(\br)$, o\`u~$V'_{x}$ est un voisinage ouvert de~$x$ dans~$V_{x}$.

D'apr\`es le corollaire~\ref{cor:HqVFresolution}, pour toute partie compacte~$W$ de~$V'_{x}$ qui poss\`ede assez d'arbres de Cousin, nous avons $H^1(\oD_{W}(\br),\cR)=0$. On en d\'eduit que le morphisme $\cO^{N_{x}}(\oD_{W}(\br)) \to \cF(\oD_{W}(\br))$ induit par~$\psi_{x}$ est surjectif. 

Consid\'erons le recouvrement~$\cV$ de~$V$ d\'efini par les~$V'_{x}$. Par hypoth\`ese, il existe un arbre de Cousin-Runge fort \`a intersections simples~$\cT$ adapt\'e \`a~$\cV$ tel que, pour toutes feuilles $f \ne g \in \cT$, le compact $K(f)\cap K(g)$ poss\`ede assez d'arbres de Cousin. Soit~$f$ une feuille de~$\cT$. Il existe un point~$x$ de~$V$ tel que~$V'_{x}$ contienne~$K(f)$. Pour toute feuille~$g$ de~$\cT$, le compact $K(f)\cap K(g)$ poss\`ede assez d'arbres de Cousin. Le raisonnement ci-dessus montre alors que le morphisme $\cO^{N_{x}}(\oD_{K(f)\cap K(g)}(\br)) \to \cF(\oD_{K(f) \cap K(g)}(\br))$ induit par~$\psi_{x}$ est surjectif. Le r\'esultat d\'ecoule alors de la proposition~\ref{prop:CRimpliqueA}.
\end{proof}

\begin{coro}\label{cor:thAum}
Soit~$x$ un point tr\`es d\'ecent de~$X_{\um}$ et soit $\br \in \R_{\ge 0}^n$. Alors, tout faisceau coh\'erent sur $\overline{D}_{x}(\br)$ est globalement engendr\'e.
\end{coro}
\begin{proof}
D\'emontrons le r\'esultat par r\'ecurrence sur~$n$. Pour~$n=0$, le r\'esultat est \'evident.

Supposons que le r\'esultat est vrai pour les disques de dimension~$n$ et d\'emontrons-le pour la dimension~$n+1$. Soit $\br=(r_{1},\dotsc,r_{n+1}) \in \R_{\ge 0}^{n+1}$. Soit~$\cF$ un faisceau coh\'erent sur $\overline{D}_{x}(\br)$. Consid\'erons le morphisme de projection sur la derni\`ere coordonn\'ee $\pi \colon \overline{D}_{x}(\br) \to \overline{D}_{x}(r_{n+1})$. Posons $\br':=(r_{1},\dotsc,r_{n})$. Par hypoth\`ese de r\'ecurrence, pour tout point~$y$ de $\overline{D}_{x}(r_{n+1})$, le faisceau~$\cF$ est globalement engendr\'e sur~$\oD_{y}(\br')$. En outre, d'apr\`es le lemme~\ref{lem:Dxrlineaire}, le disque $\overline{D}_{x}(r_{n+1})$ est une base de Stein. On conclut par la proposition~\ref{prop:thArecurrence}.
\end{proof}

%

\begin{coro}\label{cor:thA}\index{Faisceau!globalement engendr\'e}\index{Theoreme@Th\'eor\`eme!A}
Soit~$V$ une base de Stein de~$X$ et soit $\br \in \R_{\ge 0}^n$. Alors, tout faisceau coh\'erent sur~$\overline{D}_{V}(\br)$  est globalement engendr\'e.
\end{coro}
\begin{proof}
Soit~$\cF$ un faisceau coh\'erent sur~$\oD_{V}(\br)$. Pour tout point~$x$ de~$V$, le faisceau~$\cF$ est globalement engendr\'e sur~$\oD_{x}(\br)$. Cela d\'ecoule de la proposition~\ref{prop:fibrearc} si~$x$ est archim\'edien et du corollaire~\ref{cor:thAum} si~$x$ est ultram\'etrique. On conclut par la proposition~\ref{prop:thArecurrence}.
\end{proof}


\begin{coro}\label{cor:thB}\index{Theoreme@Th\'eor\`eme!B}
Soit~$V$ une base de Stein de~$X$ et soit $\br \in \R_{\ge 0}^n$. Alors, pour tout faisceau coh\'erent~$\cF$ sur~$\overline{D}_{V}(\br)$ et tout $q\ge 1$, on a $H^q(\overline{D}_{V}(\br),\cF)=0$.
\end{coro}
\begin{proof}
Soit~$\cF$ un faisceau coh\'erent sur~$\overline{D}_{V}(\br)$. En appliquant de fa\c con r\'ep\'et\'ee le corollaire~\ref{cor:thA}, on montre que le faisceau~$\cF$ poss\`ede des r\'esolutions libres de longueur arbitraire. Le r\'esultat d\'ecoule alors du corollaire~\ref{cor:HqVFresolution}. 
\end{proof}


\section{Affino\"ides surconvergents}\label{sec:affinoides}

Dans cette section, nous introduisons les espaces affino\"ides surconvergents et montrons qu'ils satisfont des propri\'et\'es analogues aux affino\"ides de la g\'eom\'etrie rigide tels que le th\'eor\`eme d'acyclicit\'e de Tate et th\'eor\`eme de Kiehl. 

\medbreak

Rappelons que toute partie compacte~$V$ de~$\E{n}{\cA}$ est munie de son faisceau de fonctions surconvergentes, \cf~notation~\ref{nota:surconvergent}.
Par exemple, si $V$ est un polydisque ferm\'e $\overline{D}_{B}(r_{1},\dotsc,r_{n})$, chaque section globale converge sur un disque $\overline{D}_{B}(r'_{1},\dotsc,r'_{n})$ avec $r'_{i} >r_{i}$ pour tout $i \in \cn{1}{n}$.
\index{Faisceau!surconvergent}\index{Fonction!surconvergente}

\begin{defi}\label{def:affinoide}\index{Espace affino\"ide!surconvergent|textbf}
On dit qu'un espace $\cA$-analytique~$X$ est \emph{$\cA$-affino\"ide surconvergent} s'il existe un polydisque ferm\'e $\overline{D}_{B}(\br)$, avec $\br\in\R_{\ge 0}^n$, et un faisceau d'id\'eaux coh\'erent~$\cI$ sur $\oD_{B}(\br)$ tels que~$X$ soit isomorphe au ferm\'e analytique de~$\oD_{B}(\br)$ d\'efini par~$\cI$.
\end{defi}

\begin{rema}\label{rem:MAsurcv}
Le spectre analytique $\cM(\cA)$ lui-m\^eme est un espace $\cA$-affino\"ide surconvergent, obtenu en choisissant le rayon~$\br$ vide et le faisceau d'id\'eaux~$\cI$ nul.
\end{rema}

\begin{rema}
Dans la th\'eorie classique sur un corps ultram\'etrique, on ne consid\`ere g\'en\'eralement que des polydisques dont tous les rayons sont strictement positifs. Cette restriction apparente ne nuit pas \`a la g\'en\'eralit\'e car un disque dont l'un des rayons est nul est isomorphe \`a un disque de dimension plus petite. Dans notre cadre, surconvergent et non plus convergent, les disques dont certains rayons sont nuls fournissent des espaces suppl\'ementaires, et nous ne souhaitons pas les exclure.
\end{rema}

En g\'eom\'etrie analytique rigide, on dispose de proc\'ed\'es classiques de constructions de domaines affino\"ides. Nous en proposons ici l'analogue.

\begin{defi}\label{def:domaines}\index{Domaine!de Weierstra\ss|textbf}\index{Domaine!de Laurent|textbf}\index{Domaine!rationnel|textbf}
Soient~$X$ un espace $\cA$-analytique et~$Y$ une partie de~$X$.

On dit que $Y$ est un \emph{domaine de Weierstra\ss} de~$X$ s'il existe $p\in \N$, $f_{1},\dotsc,f_{p} \in \cO(X)$ et $r_{1},\dotsc,r_{p} \in \R_{\ge 0}$ tels que
\[ Y = \{x\in X : \forall i\in \cn{1}{p},\ |f_{i}(x)| \le r_{i}\}.\]

On dit que $Y$ est un \emph{domaine de Laurent} de~$X$ s'il existe $p,q\in \N$, $f_{1},\dotsc,f_{p},g_{1},\dotsc,g_{q} \in \cO(X)$, $r_{1},\dotsc,r_{p} \in \R_{\ge 0}$ et $s_{1},\dotsc,s_{q} \in \R_{> 0}$ tels que
\[ Y = \{x\in X : \forall i\in \cn{1}{p},\ |f_{i}(x)| \le r_{i},\ \forall j\in \cn{1}{q},\ |g_{j}(x)| \ge s_{j}\}.\]


On dit que $Y$ est un \emph{domaine rationnel} de~$X$ s'il existe $p\in \N$, $f_{1},\dotsc,f_{p},g \in \cO(X)$ sans z\'eros communs sur~$X$ et $r_{1},\dotsc,r_{p} \in \R_{\ge 0}$ tels que
\[ Y = \{x\in X : \forall i\in \cn{1}{p},\ |f_{i}(x)| \le r_{i} |g(x)|\}.\]
\end{defi}

\begin{rema}
Dans le cadre classique sur un corps ultram\'etrique, seuls les rayons strictement positifs sont autoris\'es dans la d\'efinition des diff\'erents domaines. Dans notre cas, la th\'eorie fonctionne encore avec certains rayons nuls, gr\^ace au fait que nous consid\'erons des faisceaux structuraux surconvergents et non seulement convergents.
\end{rema}

\begin{prop}\label{prop:domaines}
Soit~$X$ un espace $\cA$-affino\"ide surconvergent. Les domaines de Weierstra\ss, les domaines de Laurent et les domaines rationnels de~$X$ sont des espaces $\cA$-affino\"ides surconvergents.
\end{prop}
\begin{proof}
Par hypoth\`ese, il existe un polydisque ferm\'e~$E$ sur~$B$ et un faisceau d'id\'eaux coh\'erent~$\cI$ sur $E$ tels que~$X$ soit isomorphe au ferm\'e analytique de~$E$ d\'efini par~$\cI$. 

Soit~$Y$ un  domaine de Weierstra\ss{} de~$X$. Alors, il existe $p\in \N$, $f_{1},\dotsc,f_{p} \in \cO(X)$ et $r_{1},\dotsc,r_{p} \in \R_{\ge 0}$ tels que
\[ Y = \{x\in X : \forall i\in \cn{1}{p},\ |f_{i}(x)| \le r_{i}\}.\]
Posons $E' := \oD_{E}(r_{1},\dotsc,r_{p})$ avec coordonn\'ees $T_{1},\dotsc,T_{p}$. C'est encore un polydisque ferm\'e sur~$B$. Notons $\pi \colon E' \to E$ la projection. Le faisceau d'id\'eaux 
\[\cI' := \pi^\ast \cI + (T_{1}-f_{1},\dotsc,T_{p}-f_{p})\, \cO_{E'}\] 
est coh\'erent sur~$E'$ et la projection~$\pi$ induit un isomorphisme entre le ferm\'e analytique de~$E'$ d\'efini par~$\cI'$ et~$Y$ (dont l'inverse est le morphisme associ\'e \`a $(f_{1},\dotsc,f_{p})$ par l'application bijective de la proposition~\ref{morphsec}). On en d\'eduit que $Y$ est un espace $\cA$-affino\"ide surconvergent.

Le cas des domaines de Laurent se traite de fa\c con similaire. Si~$Y$ est un domaine de Laurent donn\'e comme dans la d\'efinition~\ref{def:domaines}, on reprend la preuve pr\'ec\'edente en consid\'erant, cette fois-ci, le disque $E' := \oD_{E}(r_{1},\dotsc,r_{p},s_{1},\dotsc,s_{q})$ avec coordonn\'ees $T_{1},\dotsc,T_{p},S_{1},\dotsc,S_{q}$ et le faisceaux d'id\'eaux 
\[\cI' := \pi^\ast \cI + (T_{1}-f_{1},\dotsc,T_{p}-f_{p},S_{1}g_{1}-1,\dotsc,S_{q}g_{q}-1)\, \cO_{E'}.\]
Remarquons que les fonctions $g_{1},\dotsc,q_{q}$ sont inversibles sur~$Y$. On peut donc consid\'erer le morphisme de source~$Y$ associ\'e \`a $(f_{1},\dotsc,f_{p},g_{1}^{-1},\dotsc,g_{q}^{-1})$. Son image s'identifie au ferm\'e analytique~$Z'$ de~$E'$ d\'efini par~$\cI'$ et il fournit un inverse au morphisme $Z' \to Y$ induit par la projection $\pi\colon E' \to E$.

La m\^eme strat\'egie s'applique encore pour les domaines rationnels. Si~$Y$ est un domaine de Laurent donn\'e comme dans la d\'efinition~\ref{def:domaines}, on reprend la preuve en consid\'erant, cette fois-ci, le disque $E' := \oD_{E}(r_{1},\dotsc,r_{p})$ avec coordonn\'ees $T_{1},\dotsc,T_{p},S$ et le faisceaux d'id\'eaux 
\[\cI' := \pi^\ast \cI + (gT_{1}-f_{1},\dotsc,gT_{p}-f_{p})\, \cO_{E'}.\]
Puisque les fonctions $f_{1},\dotsc,f_{p},g$ sont suppos\'ees sans z\'eros communs sur~$X$, la fonction~$g$ ne peut s'annuler sur~$Y$ et elle y est donc inversible. On peut donc consid\'erer le morphisme de source~$Y$ associ\'e \`a $(f_{1}\,g^{-1},\dotsc,f_{p}\,g^{-1})$. Son image s'identifie au ferm\'e analytique~$Z'$ de~$E'$ d\'efini par~$\cI'$ et il fournit un inverse au morphisme $Z' \to Y$ induit par la projection $\pi\colon E' \to E$.
\end{proof}

\begin{coro}\label{cor:BVaffinoide}\index{Voisinage}
Soit~$X$ un espace $\cA$-analytique. Tout point de~$X$ poss\`ede une base de voisinages form\'ee d'espaces $\cA$-affino\"ides surconvergents. 
\end{coro}
\begin{proof}
Soit~$x \in X$. Le r\'esultat \`a d\'emontrer \'etant local, on peut supposer que~$X$ est un ferm\'e analytique d'un ouvert~$U$ d'un espace affine analytique~$\E{n}{\cA}$. Puisque tout ferm\'e analytique d'un espace $\cA$-affino\"ide surconvergent est encore $\cA$-affino\"ide surconvergent, il suffit de montrer le r\'esultat pour~$U$, et donc pour~$\E{n}{\cA}$. Or, par d\'efinition de la topologie de~$\E{n}{\cA}$, les domaines de Laurent forment une base de voisinages de tout point. Le r\'esultat d\'ecoule donc de la proposition~\ref{prop:domaines}.
\end{proof}

Passons maintenant aux analogues des r\'esultats classiques sur les affino\"ides. Le premier que nous \'enon\c cons s'apparente \`a une partie du th\'eor\`eme d'acyclicit\'e de Tate, affirmant que les sections globales sur un affino\"ide surconvergent sont celles auxquelles on s'attend d'apr\`es sa d\'efinition.

\begin{theo}\label{th:sectionsglobalesaffinoide}\index{Theoreme@Th\'eor\`eme!d'acyclicit\'e de Tate}
Supposons que $\cM(\cA)$ est une base de Stein. Soit $\br \in \R_{\ge 0}^n$ et soient $f_{1},\dotsc,f_{m} \in \cO_{\E{n}{\cA}}(\overline{D}_{B}(\br))$. 
Alors, le ferm\'e analytique~$X$ de $\overline{D}_{B}(\br)$ d\'efini par $f_{1},\dotsc,f_{m}$ est un espace $\cA$-affino\"ide surconvergent et on a un isomorphisme canonique
\[ \cO_{\E{n}{\cA}}(\overline{D}_{B}(\br))/(f_{1},\dotsc,f_{m}) \simtoo \cO_{X}(X).\]
\end{theo}
\begin{proof}
Notons $\cI$ le faisceau d'id\'eaux coh\'erent de~$\oD_{B}(\br)$ engendr\'e par~$(f_{1},\dotsc,f_{m})$. Posons $I := \cI(\oD_{B}(\br))$. D'apr\`es le corollaire~\ref{cor:thB}, on a $H^1(\oD_{B}(\br),\cI) =0$. En utilisant la suite exacte longue associ\'ee \`a la suite exacte courte $0 \to \cI \to \cO \to \cO/\cI\to 0$, on en d\'eduit que 
\[\cO_{X}(X) = (\cO/\cI)(\oD_{B}(\br)) \simeq \cO(\oD_{B}(\br))/I.\]

Le morphisme de faisceau 
\[\fonction{\varphi}{\cO^m}{\cO}{(a_{1},\dotsc,a_{m})}{\sum_{i=1}^m a_{i} f_{i}}\]
a pour image~$\cI$. Son noyau~$\cK$ est un faisceau coh\'erent sur~$\oD_{B}(\br)$. La suite exacte courte $0 \to \cK \to \cO^m \to \cI \to 0$ fournit une suite exacte longue 
\[ \dotsc \too \cO(\oD_{B}(\br))^m \xrightarrow[]{\varphi(\oD_{B}(\br))} \cI(\oD_{B}(\br)) \too H^1(\oD_{B}(\br),\cK) \too \dotsc\]
D'apr\`es le corollaire~\ref{cor:thB}, on a $H^1(\oD_{B}(\br),\cK) =0$, d'o\`u l'on d\'eduit que $I = \cI(\oD_{B}(\br)) = (f_{1},\dotsc,f_{m}) \cO(\oD_{B}(\br))$. Cela conclut la preuve.
\end{proof}

Le r\'esultat qui suit contient la partie manquante du th\'eor\`eme d'acyclicit\'e de Tate, ainsi que le th\'eor\`eme de Kiehl.

\begin{theo}\label{th:affinoideAB}\index{Theoreme@Th\'eor\`eme!d'acyclicit\'e de Tate}\index{Theoreme@Th\'eor\`eme!de Kiehl}\index{Theoreme@Th\'eor\`eme!A}\index{Theoreme@Th\'eor\`eme!B}
Supposons que $\cM(\cA)$ est une base de Stein. Soient~$X$ un espace $\cA$-affino\"ide surconvergent et~$\cF$ un faisceau coh\'erent sur~$X$. Alors, 
\begin{enumerate}[i)]
\item pour tout entier $q\ge 1$, on a $H^q(X,\cF) = 0$~;
\item le $\cO(X)$-module $\cF(X)$ est de type fini et engendre~$\cF$.
\end{enumerate}
\end{theo}
\begin{proof}
Par hypoth\`ese, il existe un polydisque ferm\'e~$\oD$ sur~$B$ et un faisceau d'id\'eaux coh\'erent~$\cI$ sur $\oD$ tels que~$X$ soit isomorphe au ferm\'e analytique de~$\oD$ d\'efini par~$\cI$. Notons $j \colon X \to \oD$ l'inclusion.

Soit~$\cF$ un faisceau coh\'erent. Son pouss\'e en avant $\varphi_{\ast}\cF$ est un faisceau coh\'erent sur~$\oD$ avec des groupes de cohomologie associ\'es identiques. Le point~i) d\'ecoule donc du corollaire~\ref{cor:thB}. D'apr\`es le corollaire~\ref{cor:thA}, $\varphi_{*}(\cF)(\oD)$ engendre~$(\varphi_{*}\cF)_{y}$ en tout point~$y$ de~$\oD$ et on en d\'eduit que $\cF(X) = \varphi_{*}(\cF)(\oD)$ engendre $\cF_{x} = (\varphi_{*}\cF)_{x}$ en tout point~$x$ de~$X$.

Il reste \`a montrer que $\cF(X)$ est de type fini sur~$\cO(X)$. Soit $x \in X$. On vient de montrer que $\cF_{x}$ est engendr\'e par~$\cF(X)$. Par cons\'equent, il existe un morphisme $\cO^{q_{x}} \to \cF$ qui est surjectif en~$x$, et donc sur un voisinage~$U_{x}$ de~$x$. Par compacit\'e de~$X$, on peut le recouvrir par un nombre fini de~$U_{x}$ et donc construire un morphisme de faisceaux surjectif $\varphi \colon \cO^q \to \cF$. Puisque $H^1(X,\sKer(\varphi))=0$, le morphisme $\cO(X)^q \to \cF(X)$ induit entre les sections globales est encore surjectif.

\end{proof}


\begin{coro}\index{Faisceau!coh\'erent}
Supposons que $\cM(\cA)$ est une base de Stein. Soit~$X$ un espace $\cA$-affino\"ide surconvergent. Alors le foncteur sections globales $\cF \mapsto \cF(X)$ induit une \'equivalence entre la cat\'egorie des faisceaux coh\'erents sur~$X$ et celle des modules de type fini sur~$\cO(X)$.
\qed 
\end{coro}

\section{Quelques compl\'ements}\label{sec:complements}

Dans cette section, nous utilisons la th\'eorie des espaces de Stein, que nous venons de d\'evelopper, pour raffiner ou \'etendre quelques r\'esultats ant\'erieurs.

Commen\c cons par des r\'esultats sur les alg\`ebres de couronnes. Nous pouvons maintenant \'etendre la proposition~\ref{prop:bumCst} au-del\`a du cas ultram\'etrique. 

\begin{prop}\label{prop:bCstStein}\index{Couronne!algebre@alg\`ebre d'une}
Soit $x\in \E{n}{\cA}$. Soient $s,t \in \R_{> 0}$ tels que $s\le t$. Notons $\overline{C}(s,t)$ la couronne relative au-dessus de $\E{n}{\cA}$, avec coordonn\'ee~$T$, et $\overline{C}_{x}(s,t)$ sa fibre au-dessus de~$x$. Alors, le morphisme naturel
\[\colim_{V\ni x, u < s \le t< v} \cB(V)\la u\le |T|\le v\ra \too \cO(\overline{C}_{x}(s,t)),\]
o\`u~$V$ parcourt l'ensemble des voisinages compacts de~$x$ dans~$\E{n}{\cA}$, est un isomorphisme.
\end{prop}
\begin{proof}
L'injectivit\'e du morphisme est claire. Pour montrer la surjectivit\'e, consid\'erons le disque relatif $\oD(t,s^{-1})$ de dimension~2 au-dessus de $\E{n}{\cA}$, avec coordonn\'ees~$T,S$, et son ferm\'e analytique~$Z$ d\'efini par l'id\'eal engendr\'e par~$ST-1$. Par le m\^eme raisonnement que dans la preuve de la proposition~\ref{prop:domaines}, on montre que $Z$ est isomorphe \`a $\overline{C}(s,t)$. En se restreignant aux fibres au-dessus de~$x$, on en d\'eduit un isomorphisme $Z_{x} \simeq \overline{C}_{x}(s,t)$.

D'apr\`es le corollaire~\ref{cor:thB}, les groupes de cohomologie sup\'erieurs de tout faisceau coh\'erent sur la fibre~$\oD_{x}(t,s^{-1})$ de $\oD(t,s^{-1})$ au-dessus de~$x$ sont nuls. En raisonnant comme dans la preuve du th\'eor\`eme~\ref{th:sectionsglobalesaffinoide}, on en d\'eduit un isomorphisme canonique $\cO(\oD_{x}(t,s^{-1}))/(ST-1) \simeq \cO(Z_{x})$.

On obtient finalement un diagramme commutatif
\[\begin{tikzcd}
\Big(\colim\limits_{V\ni x, u < s \le t< v} \cB(V)\la |T|\le v, |S| \le u^{-1}\ra\Big)/(ST-1) \ar[d]\ar[r]& \cO(\oD_{x}(t,s^{-1}))/(ST-1) \ar[d, "\sim" {anchor=south, rotate=90}]\\
\colim\limits_{V\ni x, u < s \le t< v} \cB(V)\la u\le |T|\le v\ra \ar[r]&  \cO(\overline{C}_{x}(s,t)),
\end{tikzcd}\]
o\`u la fl\`eche verticale de gauche est donn\'ee par l'\'evaluation en $S = T^{-1}$. D'apr\`es la proposition~\ref{prop:disqueglobal}, dans une version en plusieurs variables qui se d\'emontre sans difficult\'es suppl\'ementaires\footnote{On d\'emontre d'abord le r\'esultat dans le cas des rayons nuls, cf~\cite[th\'eor\`eme~2.4.8]{A1Z}, puis le cas g\'en\'eral en utilisant le principe du prolongement anlaytique, comme dans la preuve de \cite[corollaire~2.7]{EtudeLocale}}, la fl\`eche horizontale sup\'erieure est un isomorphisme. Le r\'esultat s'en d\'eduit.
\end{proof}

En appliquant la proposition pr\'ec\'edente fibre \`a fibre et en recollant, on obtient le r\'esultat suivant, \cf~corollaire~\ref{cor:couronneglobale}.

\begin{coro}\label{cor:couronneglobaleStein}\index{Couronne!algebre@alg\`ebre d'une}
Soit~$V$ une partie compacte de~$\E{n}{\cA}$. Soient $s,t\in\R_{> 0}$ tels que $s\le t$. Alors le morphisme naturel
\[\colim_{W\supset V,u < s \le t < v} \cO(W)\la u\le |T|\le v\ra \too \cO(\overline{C}_{V}(s,t)),\]
o\`u~$W$ parcourt l'ensemble des voisinages compacts de~$V$ dans~$\E{n}{\cA}$, est un isomorphisme.
\end{coro}

Nous allons maintenant \'enoncer un r\'esultat technique qui permet de remplacer, dans certains \'enonc\'es, des anneaux de la forme $\cB(V)$ par des anneaux de la forme $\cO(U)$, plus naturels. Au pr\'elable, introduisons une d\'efinition.

\begin{defi}\index{Fonction!B-definie@$\cB$-d\'efinie}
Soit $U$ un ouvert de~$\E{n}{\cA}$ et $V$ une partie compacte de~$U$. On dit qu'un \'el\'ement~$f$ de~$\cO(U)$ est \emph{$\cB$-d\'efini} sur~$V$ s'il existe un \'el\'ement~$F$ de~$\cB(V)$ tel que $f_{|V} = F_{|V}$, au sens o\`u
\[\forall y \in V, \ f(y) = F(y).\]
\end{defi}

\begin{rema}
Nous avons d\'ej\`a introduit la propri\'et\'e d'\^etre $\cB$-d\'efini pour les \'el\'ements d'un anneau local~$\cO_{x}$ \`a la d\'efinition~\ref{def:B-defini}. Si $U$ est un ouvert de~$\E{n}{\cA}$, $x$ un point de~$U$ et $V$ un voisinage compact de~$x$ dans~$U$, alors tout \'el\'ement~$f$ de~$\cO(U)$ qui est $\cB$-d\'efini sur~$V$ est \'egalement $\cB$-d\'efini en tant qu'\'el\'ement de~$\cO_{x}$.
\end{rema}

\begin{defi}\label{def:Badapte}\index{Point!$\cB$-adapt\'e|textbf}\index{Espace affine analytique!$\cB$-adapt\'ee|textbf}\index{Fonction!B-definie@$\cB$-d\'efinie}
On dit qu'un point $x$ de~$\E{n}{\cA}$ est \emph{$\cB$-adapt\'e} si, pour tout voisinage~$U$ de~$x$ dans~$\E{n}{\cA}$, il existe un voisinage compact~$V$ de~$x$ dans~$U$ tel que tout \'el\'ement de~$\cO(U)$ soit $\cB$-d\'efini sur~$V$.



On dit que $\E{n}{\cA}$ est $\cB$-adapt\'e si tous ses points sont $\cB$-adapt\'es.
%
%
\end{defi}

\begin{exem}\label{ex:Badapte}\index{Corps!valu\'e}\index{Anneau!des entiers relatifs $\Z$}\index{Anneau!des entiers d'un corps de nombres}\index{Corps!hybride}\index{Anneau!de valuation discr\`ete}\index{Anneau!de Dedekind trivialement valu\'e}
Lorsque l'anneau de Banach~$\cA$ est l'un de nos exemples habituels (corps valu\'e, $\Z$ ou anneau d'entiers de corps de nombres, corps hybride, anneau de valuation discr\`ete, anneau de Dedekind trivialement valu\'e, \cf~exemples~\ref{ex:corpsvalue} \`a~\ref{ex:Dedekind}), le spectre $\cM(\cA)$ est $\cB$-adapt\'e.
\end{exem}

\begin{prop}\label{prop:AnBadapte}
Supposons que $\cM(\cA)$ est $\cB$-adapt\'e. Alors, pour tout $n\in \N$, $\E{n}{\cA}$ est $\cB$-adapt\'e.
\end{prop}
\begin{proof}
D\'emontrons le r\'esultat par r\'ecurrence sur~$n$. Le cas $n=0$ est vrai par hypoth\`ese.

Supposons que le r\'esultat est vrai pour~$\E{n}{\cA}$ et d\'emontrons-le pour~$\E{n+1}{\cA}$. Consid\'erons $\pi \colon \E{n+1}{\cA} \to \E{n}{\cA}$ la projection sur la derni\`ere coordonn\'ee, not\'ee~$S$.

Soit $x\in \E{n+1}{\cA}$ et soit $U$ un voisinage de~$x$ dans~$\E{n+1}{\cA}$. D'apr\`es le corollaire~\ref{cor:BVfineNPST}, 
il existe un voisinage compact~$W_{0}$ de~$\pi(x)$ dans~$\E{n}{\cA}$, une base de voisinages compacts~$\cW_{\pi(x)}$ de~$\pi(x)$ dans~$W_{0}$, un polyn\^ome $P \in \cB(W_{0})[S]$ unitaire non constant et des nombres r\'eels $r,s \in \R_{\ge 0}$ tels que, pour tout $W \in \cW_{\pi(x)}$, $\overline{C}_{W}(P,r,s)$ soit un voisinage de~$x$ dans~$U$ et la couronne $\overline{C}_{W}(r,s)$ (avec coordonn\'ee~$T$) satisfasse la propri\'et\'e~$(\cB N_{P(S)-T})$.

Soit $W_{1}\in \cW_{\pi(x)}$. Par hypoth\`ese, il existe un voisinage compact~$W_{2}$ de~$\pi(x)$ dans~$W_{1}$ tel que tout \'el\'ement de~$\cO(W_{1})$ soit $\cB$-d\'efini sur~$W_{2}$. Quitte \`a restreindre~$W_{2}$, on peut supposer qu'il appartient \`a~$\cW_{\pi(x)}$.


Soit $a \in \cO(\overline{C}_{W_{1}}(r,s))$. D'apr\`es le corollaire~\ref{cor:couronneglobaleStein}, c'est l'image d'une s\'erie $a'$ de $\cO(W_{1})\la r\le |T|\le s\ra$ par l'application canonique. Par d\'efinition de~$W_{2}$, il existe une s\'erie  $A'$ de $\cB(W_{2})\la r\le |T|\le s\ra$ dont les coefficients co\"incident avec ceux de~$a'$ sur~$W_{2}$. Par cons\'equent, l'image~$A$ de~$A'$ dans $\cB(\overline{C}_{W_{2}}(r,s))$ co\"incide avec~$a$ sur $\overline{C}_{W_{2}}(r,s)$.

Soit $f\in \cO(U)$. Notons encore~$f$ son image dans $\cO(\overline{C}_{W_{1}}(P,r,s))$. 
D'apr\`es le th\'eor\`eme~\ref{thm:isolemniscate}, on a un isomorphisme
\[\varphi \colon \cO(\overline{C}_{W_{1}}(P,r,s)) \simtoo \cO(\overline{C}_{W_{1}}(r,s))[S]/(P(S)-T).\]
Repr\'esentons $\varphi(f)$ par un \'el\'ement~$g$ de $\cO(\overline{C}_{W_{1}}(r,s))[S]$. Le r\'esultat du paragraphe pr\'ec\'edent assure qu'il existe un \'el\'ement~$G$ de $\cB(\overline{C}_{W_{2}}(r,s))[S]$ dont la restriction \`a~$\oC_{W_{2}}(r,s)$ co\"incide avec~$g$. Notons $\bar G$ son image dans $\cB(\overline{C}_{W_{2}}(r,s))[S]/(P(S)-T)$. Puisque la couronne $\overline{C}_{W_{2}}(r,s)$ satisfait la propri\'et\'e~$(\cB N_{P(S)-T})$, on a un isomorphisme
\[\psi \colon \cB(\overline{C}_{W_{2}}(P,r,s)) \simtoo \cB(\overline{C}_{W_{2}}(r,s))[S]/(P(S)-T).\]
Par construction, $\psi^{-1}(\bar G)$ co\"incide avec~$f$ sur $\overline{C}_{W_{2}}(P,r,s)$. Le r\'esultat s'ensuit.


\end{proof}

Comme exemple d'application, \'enon\c cons une version plus g\'en\'erale du r\'esultat de fermeture des id\'eaux du faisceau structural, \cf~corollaire~\ref{limite}.

\begin{coro}\label{cor:limiteOU}\index{Ideal@Id\'eal!ferme@ferm\'e}
Supposons que $\cM(\cA)$ est $\cB$-adapt\'e. Soient~$x$ un point de~$\E{n}{\cA}$, $p\ge 1$ un entier et $M$~un sous-module de~$\cO^p_{x}$. Soient $U$ un voisinage de~$x$ dans~$\E{n}{\cA}$ et $(g_n)_{n\in\N}$ une suite d'\'el\'ements de~$\cO(U)^p$ qui converge vers un \'el\'ement~$g$ de~$\cO(U)^p$. Si, pour tout $n\in \N$, l'image de~$g_{n}$ dans~$\cO_{x}^p$ appartient \`a~$M$, alors l'image de~$g$ dans~$\cO_{x}^p$ appartient \`a~$M$. 
\qed
\end{coro}

\section{Th\'eor\`eme B sur certains ouverts}\label{sec:Bouvert}


Pour d\'emontrer des r\'esultats d'annulation cohomologique sur des ouverts \`a partir de r\'esultats sur des ferm\'es, nous nous inspirons de la strat\'egie utilis\'ee dans le cas complexe, et en particulier de la notion d'exhaustion de Stein, \cf~\cite[chapter~IV]{Gr-Re}. Signalons que le second auteur a d\'ej\`a mis en {\oe}uvre une strat\'egie similaire dans le cas de la droite affine dans~\cite[\S 6.6]{A1Z}.

Les changements \`a apporter dans notre cadre sont mineurs. Pour faciliter la lecture, nous allons cependant proc\'eder \`a quelques rappels et indiquer les grandes lignes du raisonnement.

\begin{defi}\index{Exhaustion|textbf}
Soit~$T$ un espace topologique. Une \emph{exhaustion de~$T$} est une suite de parties compactes $(K_{m})_{m\ge 0}$ de~$T$ telle que, pour tout $m\ge0$, 
$K_{m}$ soit incluse dans l'int\'erieur de~$K_{m+1}$ et 
\[T =  \bigcup_{m\ge 0} K_{m} .\]
\end{defi}

\begin{theo}[\protect{\cite[IV, \S 1, theorem~4]{Gr-Re}}]\index{Theoreme@Th\'eor\`eme!B}
Soient~$T$ un espace topologique et $(K_{m})_{m\ge 0}$ une exhaustion de~$T$. Soient~$\cF$ un faisceau de groupes ab\'eliens sur~$T$ et $q\ge 2$ un entier. Supposons que, pour tout $m\ge 0$, on a 
\[ H^{q-1}(K_{m},\cF) = H^{q}(K_{m},\cF) = 0.\]
Alors, on a 
\[ H^q(T,\cF) = 0.\]
\qed
\end{theo}

Ce r\'esultat ne pr\'esente pas de difficult\'e particuli\`ere. On voit cependant que l'annulation du~$H^1$ reste hors d'atteinte. Pour l'obtenir, on utilise une notion plus fine qui fait intervenir des conditions topologiques sur les espaces de sections.

\begin{defi}[\protect{\cite[IV, \S 1, definition~6]{Gr-Re}}]\index{Exhaustion!de Stein|textbf}
Soient~$T$ un espace topologique et~$\cF$ un faisceau de groupes ab\'eliens sur~$T$. Une \emph{exhaustion de Stein de~$T$ relativement \`a~$\cF$} est une exhaustion $(K_{m})_{m\ge 0}$ de~$T$ telle que 
\[ \textrm{pour tous } m\ge 0, q\ge 1, \ H^q(K_{m},\cF)=0\] 
et, pour tout $m\ge0$, il il existe une semi-norme~$\nm_{m}$ sur~$\cF(K_{m})$ v\'erifiant les propri\'et\'es suivantes~:
\begin{enumerate}[i)]
\item l'image du morphisme de restriction $\cF(T) \to \cF(K_{m})$ est dense pour~$\nm_{m}$~;
\item le morphisme de restriction $(\cF(K_{m+1}),\nm_{m+1}) \to (\cF(K_{m}),\nm_{m})$ est born\'e~;
\item le morphisme de restriction $(\cF(K_{m+1}),\nm_{m+1}) \to (\cF(K_{m}),\nm_{m})$ envoie toute suite de Cauchy sur une suite convergente~;
\item tout \'el\'ement~$s$ de~$\cF(K_{m+1})$ tel que $\|s\|_{m+1}=0$ est nul sur~$K_{m}$.
\end{enumerate}
\end{defi}

\begin{theo}[\protect{\cite[IV, \S 1, theorem~7]{Gr-Re}}]\label{thm:exhaustionStein}\index{Theoreme@Th\'eor\`eme!B}
Soient~$T$ un espace topologique et~$\cF$ un faisceau de groupes ab\'eliens sur~$T$. Soit $(K_{m})_{m\ge 0}$ une exhaustion de Stein de~$T$ relativement \`a~$\cF$. Alors, on a
\[ \textrm{pour tout } q\ge 1,\ H^q(T,\cF) = 0.\]
\qed
\end{theo}

Nous allons d\'emontrer l'annulation des groupes de cohomologie sup\'erieurs de certains ouverts de~$\E{n}{\cA}$, et notamment des disques ouverts relatifs, en exhibant des exhaustions de Stein. Nous suivrons de pr\`es \cite[\S 6.6]{A1Z}. La succession de lemmes est identique et nous nous contenterons de d\'emontrer ceux qui demandent des modifications. Nous utiliserons de fa\c con cruciale le r\'esultat de fermeture des modules locaux (\cf~corollaire~\ref{cor:limiteOU}). 

\begin{defi}\label{def:exhaustionuniverselle}\index{Exhaustion!de Stein!universelle|textbf}
Soit $U$ une partie de~$\E{n}{\cA}$. Une \emph{exhaustion de Stein universelle} de $U$ est une exhaustion $(K_{m})_{m\ge 0}$ de~$U$ telle que, pour tout $m\ge0$, les propri\'et\'es suivantes soient satisfaites~:
\begin{enumerate}[i)]
\item tout faisceau coh\'erent~$\cF$ sur~$U$ est globalement engendr\'e sur~$K_{m}$~;
\item pour tout faisceau coh\'erent~$\cF$ sur $U$ et tout entier $q\ge 1$, on a $H^q(K_{m},\cF) = 0$~;
\item l'image du morphisme de restriction $\cO(K_{m+1}) \to \cO(K_{m})$ est dense pour~$\nm_{K_{m}}$.
\end{enumerate}
\end{defi}

\begin{exem}\label{ex:exhaustiondisque}
Soit $V$ une base de Stein de~$B$. Soient $t_{1},\dotsc,t_{n} \in \R_{>0} \cup \{+\infty\}$. Pour $i\in \cn{1}{n}$, fixons une suite strictement croissante $(t_{i,m})_{m\ge 0}$ d'\'el\'ements de~$\R_{>0}$ qui tend vers~$t_{i}$. Alors $(\oD_{V}(t_{1,m},\dotsc,t_{n,m}))_{m\ge 0}$ est une exhaustion de Stein universelle de $D_{V}(t_{1},\dotsc,t_{n})$. Le point~i) d\'ecoule du corollaire~\ref{cor:thA}, le point~ii) du corollaire~\ref{cor:thB} et le point~iii) de la proposition~\ref{prop:disqueglobal}.
\end{exem}


Soit $U$ une partie de~$\E{n}{\cA}$ et supposons qu'elle poss\`ede une \emph{exhaustion de Stein universelle} $(K_{m})_{m\ge 0}$. Soit~$\cF$ un faisceau coh\'erent sur~$U$.

Soit $m\in \N$. Le faisceau~$\cF$ est globalement engendr\'e sur~$K_{m}$, donc il existe $l_{m} \in \N$ et un morphisme surjectif $\alpha_{m} \colon \cO_{K_{m}}^{l_{m}} \to \cF_{|K_{m}}$. Les propri\'et\'es d'annulation cohomologique entra\^inent que le morphisme induit
\[ \eps_{m} \colon \cO(K_{m})^{l_{m}} \too \cF(K_{m})\]
est surjectif.

Munissons $\cO(K_{m})^{l_{m}}$ de la norme d\'efinie par le maximum des normes des coordonn\'ees. On la note encore~$\nm_{K_{m}}$. On d\'efinit alors une semi-norme~$\nm_{m}$ sur~$\cF(K_{m})$ par 
\[\renewcommand{\arraystretch}{1.2}\fonction{\nm_{m}}{\cF(K_{m})}{\R_{\ge0}}{s}{\inf(\{\|t\|_{K_{m}} : t\in \eps_{m}^{-1}(s)\})}.\]

Consid\'erons les morphismes de restriction 
\[r_{m} \colon (\cO(K_{m+1})^{l_{m+1}},\nm_{K_{m+1}}) \too (\cO(K_{m})^{l_{m+1}},\nm_{K_{m}})\] 
et 
\[\rho_{m} \colon (\cF(K_{m+1}),\nm_{m+1}) \too (\cF(K_{m}),\nm_{m}).\] 
Le morphisme~$r_{m}$ est born\'e.

Les propri\'et\'es d'annulation cohomologique entra\^inent que le morphisme $\alpha_{m+1} \colon \cO_{K_{m+1}}^{l_{m+1}} \to \cF_{|K_{m+1}}$ induit un morphisme surjectif
\[ \eps'_{m} \colon \cO(K_{m})^{l_{m+1}} \too \cF(K_{m}).\]
On d\'efinit alors une nouvelle semi-norme~$\nm'_{m}$ sur~$\cF(K_{m})$ par 
\[\renewcommand{\arraystretch}{1.2}\fonction{\nm'_{m}}{\cF(K_{m})}{\R_{\ge0}}{s}{\inf(\{\|t\|_{K_{m}} : t\in {\eps'_{m}}^{-1}(s)\})}.\]

Consid\'erons maintenant le morphisme
\[ \sigma_{m} \colon (\cF(K_{m}),\nm'_{m}) \too (\cF(K_{m}),\nm_{m})\]
qui induit l'identit\'e sur~$\cF(K_{m})$. Il est born\'e.

\begin{lemm}[\protect{\cite[lemme~6.6.21]{A1Z}}]
Il existe un morphisme born\'e~$\eta_{m}$ qui fait commuter le diagramme
\[\begin{tikzcd}
(\cO(K_{m})^{l_{m+1}}, \nm_{K_{m}}) \ar[r, "\eps'_{m}"] \ar[d, "\eta_{m}"]& (\cF(K_{m}), \nm'_{m}) \ar[d, "\sigma_{m}"]\\
(\cO(K_{m})^{l_{m+1}}, \nm_{K_{m}})  \ar[r, "\eps_{m}"] & (\cF(K_{m}), \nm_{m})
\end{tikzcd}.\]
\qed
\end{lemm}

On a maintenant un diagramme commutatif
\[\begin{tikzcd}
(\cO(K_{m+1})^{l_{m+1}}, \nm_{K_{m}}) \ar[r, "\eps_{m+1}"] \ar[d, "r_{m}"]& (\cF(K_{m+1}), \nm_{m+1}) \ar[d, "\wc_{|K_{m}}"] \ar[dd, out = -30, in= 30, bend left= 70, "\rho_{m}"]\\
(\cO(K_{m})^{l_{m+1}}, \nm_{K_{m}}) \ar[r, "\eps'_{m}"] \ar[d, "\eta_{m}"]& (\cF(K_{m}), \nm'_{m}) \ar[d, "\sigma_{m}"]\\
(\cO(K_{m})^{l_{m+1}}, \nm_{K_{m}})  \ar[r, "\eps_{m}"] & (\cF(K_{m}), \nm_{m})
\end{tikzcd}\]

\begin{lemm}[\protect{\cite[lemme~6.6.22]{A1Z}}]
Le morphisme $\rho_{m}$ est born\'e.
\qed
\end{lemm}

Reprenons maintenant \cite[lemme~6.6.23]{A1Z} en apportant les modifcations n\'ecessaires \`a sa preuve.

\begin{lemm}
Le morphisme~$\rho_{m}$ est d'image dense.
\end{lemm}
\begin{proof}
Puisque le morphisme~$\sigma_{m}$ est surjectif et born\'e, il suffit de montrer que l'image du morphisme de restriction $\cF(K_{m+1}) \to \cF(K_{m})$ est dense pour la norme~$\nm'_{m}$.

Soit $s \in \cF(K_{m})$. Soit $\delta>0$. Puisque~$\eps'_{m}$ est surjectif, il existe $t\in \cO(K_{m})^{k_{m+1}}$ tel que $\eps'_{m}(t)=s$. Par hypoth\`ese, $\cO(K_{m+1})$ est dense dans $\cO(K_{m})$, donc il existe $t' \in \cO(K_{m+1})$ tel que $\|r_{m}(t') -t\|_{K_{m}} \le \delta$. On a alors $\|\eps_{m+1}(t')_{|K_{m}} - s\|'_{m} \le \delta$. 
\end{proof}

Reprenons de m\^eme \cite[lemme~6.6.24]{A1Z}.

\begin{lemm}
Supposons que $\cM(\cA)$ est $\cB$-adapt\'e. Soit $s\in \cF(K_{m+1})$ tel que $\|s\|_{m+1} = 0$. Alors, $s$ est nulle sur l'int\'erieur de~$K_{m+1}$ dans~$U$. En particulier, $s_{|K_{m}} = 0$.
\end{lemm}
\begin{proof}
Soit $t\in \eps_{m+1}^{-1}(s)$. Par hypoth\`ese, il existe une suite~$(t_{j})_{j\ge 0}$ d'\'el\'ements de~$\Ker(\eps_{m+1})$ telle que $\lim_{j \to \infty} \|t-t_{j}\|_{K_{m+1}}= 0$. En d'autres termes, $(t_{j})_{j\ge 0}$ converge uniform\'ement vers~$t$ sur~$K_{m+1}$.

Soit $x$ un point de l'int\'erieur de~$K_{m+1}$ dans $U$. D'apr\`es le corollaire~\ref{cor:limiteOU}, on a $t \in \sKer(\alpha_{m+1})_{x}$, donc l'image de $\eps_{m+1}(t)$ dans~$\cF_{x}$ est nulle. 


On a montr\'e que~$s$ est nulle sur l'int\'erieur de~$K_{m+1}$ dans~$U$, et, en particulier, sur~$K_{m}$.
\end{proof}

%
%

En utilisant les r\'esultats qui pr\'ec\'edent, les autres lemmes s'adaptent maintenant sans changements.

\begin{lemm}[\protect{\cite[lemme~6.6.25]{A1Z}}]
Supposons que $\cM(\cA)$ est $\cB$-adapt\'e. Soit $(s_{k})_{k\ge 0}$ une suite d'\'el\'ements de~$\cF(K_{m+1})$ qui est de Cauchy pour la semi-norme~$\nm_{m+1}$. Alors, il existe un \'el\'ement~$s$ de~$\cF(K_{m})$ tel que la suite $(\rho_{m}(s_{k}))_{k\ge 0}$ converge vers~$s$ pour la semi-norme~$\nm_{m}$.

De plus, si $s' \in \cF(K_{m})$ est une limite de la suite $(\rho_{m}(s_{k}))_{k\ge 0}$, elle co\"incide avec~$s$ sur l'int\'erieur de~$K_{m}$ dans~$U$ .
\qed
\end{lemm}

\begin{lemm}[\protect{\cite[lemme~6.6.25]{A1Z}}]
Supposons que $\cM(\cA)$ est $\cB$-adapt\'e. L'image du morphisme de restriction $\cF(U) \to \cF(K_{m})$ est dense pour la semi-norme~$\nm_{m}$.
\qed
\end{lemm}

Nous disposons maintenant de toutes les propri\'et\'es constitutives d'une exhaustion de Stein. R\'esumons le r\'esultat obtenu.

\begin{prop}\index{Exhaustion!de Stein}
Supposons que $\cM(\cA)$ est $\cB$-adapt\'e. Soit $U$ une partie de~$\E{n}{\cA}$. Soit $(K_{m})_{m\ge 0}$ une exhaustion de Stein universelle de~$U$. Alors, $(K_{m})_{m\ge 0}$ est une exhaustion de Stein de~$U$ relativement \`a tout faisceau coh\'erent~$\cF$ sur~$U$
\qed
\end{prop}

En appliquant le th\'eor\`eme~\ref{thm:exhaustionStein}, on obtient un r\'esultat d'annulation cohomologique.

\begin{theo}\label{th:Bouvert}\index{Theoreme@Th\'eor\`eme!B}
Supposons que $\cM(\cA)$ est $\cB$-adapt\'e. Soit $U$ une partie de~$\E{n}{\cA}$. Supposons qu'elle poss\`ede une exhaustion de Stein universelle. 
Alors, pour tout faisceau coh\'erent~$\cF$ sur~$U$ et tout entier $q\ge 1$, on a 
\[H^q(U,\cF)=0.\]
\qed
\end{theo}

L'exemple~\ref{ex:exhaustiondisque} permet d'appliquer ce r\'esultat dans le cas des disques ouverts relatifs.

\begin{coro}\label{cor:Bouvertdisque}\index{Theoreme@Th\'eor\`eme!B}
Supposons que $\cM(\cA)$ est $\cB$-adapt\'e. Soit $V$ une base de Stein de~$B$. Soient $t_{1},\dotsc,t_{n} \in \R_{>0} \cup \{+\infty\}$. Alors, pour tout faisceau coh\'erent~$\cF$ sur $D_{V}(t_{1},\dotsc,t_{n})$ et tout entier $q\ge 1$, on a 
\[H^q(D_{V}(t_{1},\dotsc,t_{n}),\cF)=0.\]
\qed
\end{coro}

De la m\^eme fa\c con que les propri\'et\'es des affino\"ides surconvergents se d\'eduisent de celles des disques ferm\'es, du corollaire~\ref{cor:Bouvertdisque} d\'ecoulent des r\'esultats d'annulation cohomologique pour d'autres espaces. Nous introduisons un peu de terminologie de fa\c con \`a pouvoir les \'enoncer plus commod\'ement.

\begin{defi}\label{def:polydisqueouvertgeneralise}\index{Disque!ouvert g\'en\'eralis\'e}
Un espace de la forme $D_{B}(t_{1},\dotsc,t_{n})$ avec $n\in \N$ et $t_{1},\dotsc,t_{n} \in \R_{>0}\cup \{+\infty\}$ est appel\'e \emph{polydisque ouvert g\'en\'eralis\'e sur~$B$}.
\end{defi}

\begin{exem}
Pour tout $n\in \N$, l'espace affine analytique~$\E{n}{\cA}$ est un polydisque ouvert g\'en\'eralis\'e.
\end{exem}

\begin{defi}\label{def:affinoideouvert}\index{Espace affino\"ide!ouvert|textbf}
On dit qu'un espace $\cA$-analytique~$X$ est \emph{$\cA$-affino\"ide ouvert} s'il existe un polydisque ouvert g\'en\'eralis\'e~$D$ sur~$B$ et un faisceau d'id\'eaux coh\'erent~$\cI$ sur~$D$ tel que~$X$ soit isomorphe au ferm\'e analytique de~$D$ d\'efini par~$\cI$.
\end{defi}

\begin{exem}
Soit~$\cX$ un sch\'ema affine de pr\'esentation finie sur~$\cA$. Alors son analytifi\'e~$\cX^\an$ est un espace $\cA$-affino\"ide ouvert, par construction (\cf~th\'eor\`eme~\ref{thm:analytification} et sa preuve).
\end{exem}

\begin{defi}\label{def:domainesouverts}\index{Domaine!de Weierstra\ss!ouvert|textbf}\index{Domaine!de Laurent!ouvert|textbf}\index{Domaine!rationnel!ouvert|textbf}
Soient~$X$ un espace $\cA$-analytique et~$Y$ une partie de~$X$.

On dit que $Y$ est un \emph{domaine de Weierstra\ss{} ouvert} de~$X$ s'il existe $p\in \N$, $f_{1},\dotsc,f_{p} \in \cO(X)$ et $r_{1},\dotsc,r_{p} \in \R_{>0}$ tels que
\[ Y = \{x\in X : \forall i\in \cn{1}{p},\ |f_{i}(x)| < r_{i}\}.\]

On dit que $Y$ est un \emph{domaine de Laurent ouvert} de~$X$ s'il existe $p,q\in \N$, $f_{1},\dotsc,f_{p},g_{1},\dotsc,g_{q} \in \cO(X)$ et $r_{1},\dotsc,r_{p},s_{1},\dotsc,s_{q} \in \R_{>0}$ tels que
\[ Y = \{x\in X : \forall i\in \cn{1}{p},\ |f_{i}(x)| < r_{i},\ \forall j\in \cn{1}{q},\ |g_{j}(x)| > s_{j}\}.\]

On dit que $Y$ est un \emph{domaine rationnel ouvert} de~$X$ s'il existe $p\in \N$, $f_{1},\dotsc,f_{p},g \in \cO(X)$ sans z\'eros communs sur~$X$ et $r_{1},\dotsc,r_{p} \in \R_{>0}$ tels que
\[ Y = \{x\in X : \forall i\in \cn{1}{p},\ |f_{i}(x)| < r_{i} |g(x)|\}.\]
\end{defi}

\begin{exem}
Une polycouronne est un domaine rationnel ouvert d'un espace affine analytique.
\end{exem}

Les deux r\'esultats qui suivent se d\'emontrent respectivement comme la proposition~\ref{prop:domaines} et le corollaire~\ref{cor:BVaffinoide}.

\begin{prop}\label{prop:domainesouverts}
Soit~$X$ un espace $\cA$-affino\"ide ouvert. Les domaines de Weierstra\ss, les domaines de Laurent et les domaines rationnels de~$X$ sont des espaces $\cA$-affino\"ides ouverts.
\qed
\end{prop}

\begin{coro}\label{cor:BVaffinoideouvert}\index{Voisinage}
Soit~$X$ un espace $\cA$-analytique. Tout point de~$X$ poss\`ede une base de voisinages form\'ee d'espaces $\cA$-affino\"ides ouverts.
\qed
\end{coro}

En suivant la m\^eme strat\'egie que dans la preuve du th\'eor\`eme~\ref{th:affinoideAB}, on d\'eduit le r\'esultat suivant du corollaire~\ref{cor:Bouvertdisque}. 

\begin{coro}\label{cor:thBaffinoideouvert}\index{Theoreme@Th\'eor\`eme!B}
Supposons que $\cM(\cA)$ est une base de Stein $\cB$-adapt\'ee. 
Soit~$X$ un espace $\cA$-affino\"ide ouvert. Alors, pour tout faisceau coh\'erent~$\cF$ sur~$X$ et tout entier $q\ge 1$, on a $H^q(X,\cF) = 0$.
\qed
\end{coro}

\begin{coro}\index{Theoreme@Th\'eor\`eme!B}
Supposons que $\cM(\cA)$ est une base de Stein $\cB$-adapt\'ee. Alors, tout espace $\cA$-analytique poss\`ede une base d'ouverts satisfaisant le th\'eor\`eme~B.
\qed
\end{coro}

\section[Noeth\'erianit\'e]{Noeth\'erianit\'e d'anneaux de s\'eries arithm\'etiques convergentes}\label{sec:noetherianite}
\index{Anneau!noetherien@noeth\'erien|(}
\index{Series arithmetiques convergentes@S\'eries arithm\'etiques convergentes|(}

Dans cette section finale, nous proposons une application de nos r\'esultats aux anneaux de s\'eries arithm\'etiques convergentes. Ceux-ci ont \'et\'e introduits par D.~Harbater dans le cadre de son \'etude du probl\`eme inverse de Galois (\cf~\cite{HarbaterConvergent,HarbaterGaloisCovers}). Un exemple typique est 
\[ \Z_{r^+}\llbracket T\rrbracket := \{f \in \Z\llbracket T\rrbracket : R_{\infty}(f) >r \},\]%
\nomenclature[Ck]{$\Z_{r^+}\llbracket T\rrbracket$}{sous-anneau de $\Z\llbracket T\rrbracket$ form\'e des s\'eries~$f$ telles que $R_{\infty}(f)>r$}
o\`u~$r$ est un nombre r\'eel positif et $R_{\infty}(f)$%
\nomenclature[Cja]{$R_{\infty}(f)$}{rayon de convergence complexe d'une s\'erie~$f$}
d\'esigne le rayon de convergence complexe de la s\'erie~$f$, c'est-\`a-dire son rayon de convergence en tant que fonction analytique complexe ($\C$ \'etant muni de la valeur absolue usuelle).

Nous n'\'etudierons pas ici les propri\'et\'es galoisiennes de ces anneaux. Nous renvoyons le lecteur int\'eress\'e par une approche g\'eom\'etrique de ces questions \`a l'article~\cite{Raccord}, dans lequel figure une preuve, bas\'ee sur la construction de rev\^etements de disques analytiques sur~$\Z$,  du fait que tout groupe fini peut se r\'ealiser comme groupe de Galois d'une extension r\'eguli\`ere de $\Z_{r^+}\llbracket T\rrbracket$, lorsque $r \in \intfo{0,1}$. 

En ce qui concerne les aspects alg\'ebriques des anneaux de s\'eries arithm\'etiques convergentes, on dispose du r\'esultat suivant.

\begin{theo}[\protect{\cite[theorem~1.8]{HarbaterConvergent}}]\label{th:Harbaternoetherien}
Pour tout $r\in \intfo{0,1}$, l'anneau $\Z_{r^+}\llbracket T\rrbracket$ est noeth\'erien, r\'egulier, factoriel et de dimension~2.
\qed
\end{theo}

Nous allons montrer que la propri\'et\'e de noeth\'erianit\'e s'\'etend aux anneaux de s\'eries arithm\'etiques convergentes  en plusieurs variables.

La preuve du th\'eor\`eme~\ref{th:Harbaternoetherien} que propose D.~Harbater est de nature tr\`es alg\'ebrique. La noeth\'erianit\'e, par exemple, est obtenue en exhibant, pour chaque id\'eal premier, un famille g\'en\'eratrice explicite. Une telle strat\'egie semble difficile \`a g\'en\'eraliser.

Notre approche est, au contraire, g\'eom\'etrique. Elle s'inspire de celle adopt\'ee par J.~Frisch dans le cadre d'anneaux de fonctions analytiques complexes.
\index{Disque!algebre@alg\`ebre d'un!noeth\'erianit\'e de l'|(}

\begin{theo}[\protect{\cite[th\'eor\`eme (I, 9)]{Frisch}}]\label{th:noetheriencomplexe}
Soit~$A$ une partie compacte semi-analytique et de Stein d'un espace analytique complexe~$X$. Alors, l'anneau $\cO(A)$ est noeth\'erien.

En particulier, pour tous $n\in \N$ et $\br \in \R_{\ge 0}^n$, l'anneau $\cO(\oD_{\C}(\br))$ est noeth\'erien.
\qed
\end{theo}

Signalons que Y.-T. Siu a, par la suite, obtenu une d\'emonstration simplifi\'ee d'un r\'esultat un peu plus g\'en\'eral dans~\cite{SiuNoetherianness}. Pour le cas particulier des polydiques, K.~Langmann a propos\'e une preuve plus courte dans~\cite{LangmannFrisch}. Nous nous sommes grandement inspir\'es de cette derni\`ere.

Il n'est gu\`ere difficile de passer du r\'esultat sur~$\C$ \`a un r\'esultat valable sur tout corps valu\'e archim\'edien complet.

\begin{coro}\label{th:noetherienarchimedien}
Soit $(k,\va)$ un corps valu\'e complet archim\'edien. Soient $n\in \N$ et $\br \in \R_{\ge0}^n$. Alors l'anneau $\cO(\oD_{k}(\br))$ est noeth\'erien.
\end{coro}
\begin{proof}
D'apr\`es le th\'eor\`eme~\ref{th:vaarchimedienne}, il existe $\eps \in \intof{0,1}$ tel que $(k,\va)$ soit isom\'etriquement isomorphe \`a $(\R,\va_{\infty}^\eps)$ ou $(\C,\va_{\infty}^\eps)$. Traitons s\'epar\'ement ces deux cas.

$\bullet$ Supposons que $(k,\va) = (\C,\va_{\infty}^\eps)$. L'\'el\'evation \`a la puissance~$1/\eps$ induit un isomorphisme entre $\cO(\oD_{k}(\br))$ et $\cO(\oD_{\C}(\br^{1/\eps}))$, o\`u~$\C$ est muni de la valeur absolue usuelle~$\va_{\infty}$ (\cf~\cite[proposition~1.3.10]{A1Z} pour un r\'esultat permettant de comparer chacun de ces anneaux \`a l'anneau d'un disque relatif sur l'espace hybride~$\cM(\C^\hyb) \setminus \{\va_{0}\}$). Le r\'esultat d\'ecoule alors du th\'eor\`eme~\ref{th:noetheriencomplexe}.

$\bullet$ Supposons que $(k,\va) = (\R,\va_{\infty}^\eps)$. Posons $(k',\va) = (\C,\va_{\infty}^\eps)$. Pour toute s\'erie $h = \sum_{\bi\ge 0} a_{\bi}\, \bT^\bi \in \C\llbracket \bT\rrbracket$, on pose $\bar h := \sum_{\bi\ge 0} \bar a_{\bi}\, \bT^\bi \in \C\llbracket \bT\rrbracket$. Remarquons que, pour tout \'el\'ement~$h$ de~$\cO(\oD_{k'}(\br))$, $\bar h$ appartient \`a~$\cO(\oD_{k'}(\br))$ et $h+\bar h$ \`a~$\cO(\oD_{k}(\br))$.

Soit $(f_{m})_{m\ge 0}$ une suite d'\'el\'ements de $\cO(\oD_{k}(\br))$. D'apr\`es le cas pr\'ec\'edent, l'anneau~$\cO(\oD_{k'}(\br))$ est noeth\'erien, donc il existe $M\ge 0$ tel que, pour tout $m> M$, on ait $f_{m} \in (f_{0},\dotsc,f_{M}) \, \cO(\oD_{k'}(\br))$. Soit $m> M$. Il existe $h_{0},\dotsc,h_{M} \in \cO(\oD_{k'}(\br))$ tels que $f_{m} = \sum_{j=0}^M h_{j}f_{j}$. On a 
\[ f_{m} = \bar f_{m} = \sum_{j=0}^M \bar h_{j}f_{j} = \sum_{j=0}^M \frac{h_{j} + \bar h_{j}}2\, f_{j},\]
donc $f_{m} \in (f_{0},\dotsc,f_{M}) \, \cO(\oD_{k}(\br))$. On en d\'eduit que $\cO(\oD_{k}(\br))$ est noeth\'erien.
\end{proof}

Un r\'esultat similaire vaut \'egalement dans le cas ultram\'etrique. Le cas du polydisque unit\'e a \'et\'e trait\'e par E.~Grosse--Kl\"onne dans~\cite[1.4]{GrosseKlonne}. Un r\'esultat g\'en\'eral proche de celui de J.~Frisch et de Y.-T.~Siu a \'et\'e obtenu par le second auteur dans~\cite[th\'eor\`eme~A.5]{PoineauConnexite}. Pour \'eviter d'introduire des notions techniques suppl\'ementaires, nous nous contenterons d'en \'enoncer la cons\'equence suivante.

\begin{theo}\label{th:noetherienultrametrique}
Soit $(k,\va)$ un corps valu\'e complet ultram\'etrique. Soient $n\in \N$ et $\br \in \R_{\ge0}^n$. Alors l'anneau $\cO(\oD_{k}(\br))$ est noeth\'erien.
\qed
\end{theo}

Soit~$K$ un corps de nombres et $A$~son anneau d'entiers. Notons~$\Sigma_{f}$ l'ensemble des places finies de~$K$ (vues comme des id\'eaux maximaux de~$A$) et $\Sigma_{\infty}$ l'ensemble des places infinies de~$K$ (vues comme des classes de conjugaison de plongements complexes). Posons $\Sigma :=\Sigma_{f} \cup \Sigma_{\infty}$. On utilisera les notations de l'exemple~\ref{ex:cdn} pour les points de~$\cM(A)$.

Soit~$\Sigma_{f,0}$ un sous-ensemble fini de~$\Sigma_{f}$. Posons $\Sigma_{0} := \Sigma_{f,0} \cup \Sigma_{\infty}$. Pour tout $\sigma \in \Sigma_{f,0}$, fixons $\eps_{\sigma} \in \intoo{0,+\infty}$ et, pour tout $\sigma \in \Sigma_{\infty}$, fixons $\eps_{\sigma} \in \intof{0,1}$. Posons 
\[ V := \bigcup_{\sigma \in \Sigma_{0}} \{a_{\sigma}^\eps : \eps \in [0,\eps_{\sigma}]\} \cup  \bigcup_{\sigma \in \Sigma \setminus \Sigma_{0}} \{a_{\sigma}^\eps : \eps \in [0,+\infty]\}.\] 
Notons~$A[1/\Sigma_{0}]$ l'ensemble des \'el\'ements de~$K$ qui sont entiers pour toutes les places de $\Sigma_{f} \setminus \Sigma_{f,0}$. D\'ecrivons maintenant les fonctions globales sur~$V$.

\begin{lemm}\label{lem:AOV}
L'application $A \to \cO(V)$ induit un isomorphisme $A[1/\Sigma_{0}] \simto \cO(V)$ et on a $\nm_{V} = \max_{\sigma \in \Sigma_{0}} (\va_{\sigma}^{\eps_{\sigma}})$.
\end{lemm}
\begin{proof}
C'est une cons\'equence des descriptions explicites d\'emontr\'ees dans \cite[\S 3.1.2]{A1Z}.
\end{proof}

Passons maintenant aux fonctions sur des polydisques relatifs.

\begin{nota}
Pour $n\in\N$ et $\br = (r_{1},\dotsc,r_{n}), \bs = (s_{1},\dotsc,s_{n}) \in \R_{\ge 0}^n$, on note $\br < \bs$ si, pour tout $i\in \cn{1}{n}$, $r_{i} < s_{i}$.%
\nomenclature[Itbb]{$\br < \bs$}{relation satisfaite si $r_{i} < s_{i}$ pour tout $i\in \cn{1}{n}$} 
\end{nota}

\begin{lemm}\label{lem:descriptionseriesarithmetiques}\index{Disque!fonction sur un}\index{Fonction!sur un disque}
Soient $n\in\N$ et $\br \in \R_{\ge 0}^n$. L'anneau $\cO(\oD_{V}(\br))$ est constitu\'e des s\'eries 
\[ f = \sum_{\bi \ge 0} f_{\bi}\, \bT^\bi  \in \cA\Big[\frac1{\Sigma_{0}}\Big]\llbracket \bT \rrbracket\]
v\'erifiant la condition suivante~:
\[\exists \bs > \br, \forall \sigma \in \Sigma_{0},\ \lim_{\bi \to \infty} |f_{\bi}|_{\sigma}^{\eps_{\sigma}}\, \bs^\bi =0.\]
\end{lemm}
\begin{proof}
Pour $\beta\in\R_{>1}$, posons
\[ V_{\beta} := \bigcup_{\sigma \in \Sigma_{0}} \{a_{\sigma}^\eps : \eps \in [0,\beta \, \eps_{\sigma}]\} \cup  \bigcup_{\sigma \in \Sigma \setminus \Sigma_{0}} \{a_{\sigma}^\eps : \eps \in [0,+\infty]\}.\] 
La famille $(V_{\beta})_{\beta >1}$ forme une base de voisinages compacts de~$V$ dans~$\cM(\cA)$. 

Le r\'esultat d\'ecoule alors de la proposition~\ref{prop:disqueglobal} et du lemme~\ref{lem:AOV} en remarquant que la condition $\lim_{\bi \to \infty} |f_{\bi}|_{\sigma}^{\beta \, \eps_{\sigma}}\, \bs^\bi =0$ peut se retraduire par $\lim_{\bi \to \infty} |f_{\bi}|_{\sigma}^{ \eps_{\sigma}}\, (\bs^{1/\beta})^\bi =0$.
\end{proof}

Ajoutons un r\'esultat technique utile sur les fonctions sur des disques relatifs contenus dans une branche.

\begin{lemm}\label{lem:restrictionbrancheiso}\index{Disque!fonctions sur un}
Soit $\sigma \in \Sigma$. Soient $u, v \in \intoo{0,+\infty}$ avec $u\le v$. Si $\sigma\in \Sigma_{\infty}$, supposons que $v\le 1$. Posons $W := \{a_{\sigma}^\eps : \eps \in [u,v]\}$. Soient~$n\in \N$ et $\br \in \intfo{0,1}^n$. Alors, le morphisme de restriction $\cO(\oD_{W}(\br)) \to \cO(\oD_{\cH(a_{\sigma}^v)}(\br))$ est un isomorphisme.

En particulier, l'anneau $\cO(\oD_{W}(\br))$ est noeth\'erien.
\end{lemm}
\begin{proof}
Notons $P := \{(u',v') \in \R_{>0}^2 : 0< u'<u\le v<v'\}$. Pour $(u',v') \in P$, posons $W_{u',v'} := \{a_{\sigma}^\eps : \eps \in [u',v']\}$. On a $\cO(W_{u',v'}) = \hat{K}_{\sigma}$ et $\nm_{W_{u',v'}} = \max(\va_{\sigma}^u,\va_{\sigma}^v)$. 

La famille $(W_{u',v'})_{(u',v') \in P}$ forme une base de voisinages compacts de~$W$ dans~$\cM(\cA)$. On d\'eduit alors de la proposition~\ref{prop:disqueglobal} que l'anneau $\cO(\oD_{W}(\br))$ est constitu\'e des s\'eries $f = \sum_{\bi \ge 0} f_{\bi}\, \bT^\bi  \in \hat{K}_{\sigma}\llbracket \bT \rrbracket$ 
v\'erifiant la condition suivante~:
\[ \exists \bs > \br, \ \lim_{\bi \to \infty} \max(|f_{\bi}|_{\sigma}^{u}\, \bs^\bi , |f_{\bi}|_{\sigma}^{v}\, \bs^\bi )=0,\]
ou encore
\[ \exists \bs > \br, \ \lim_{\bi \to \infty} \max(|f_{\bi}|_{\sigma}\, \bs^{\bi /u} , |f_{\bi}|_{\sigma}\, \bs^{\bi/v})=0,\]
Or $\br \in \intfo{0,1}^n$ et, pour tout $\bs \in \intfo{0,1}^n$, on a $\bs^{1/u} \le  \bs^{1/v}$. La condition est donc encore \'equivalente \`a 
\[ \exists \bs > \br, \ \lim_{\bi \to \infty} |f_{\bi}|_{\sigma}\, \bs^{\bi/v}=0.\]
On en d\'eduit que $\cO(\oD_{W}(\br)) = \cO(\oD_{\cH(a_{\sigma}^v)}(\br))$.

La derni\`ere partie du r\'esultat d\'ecoule des th\'eor\`emes~\ref{th:noetherienarchimedien} et~\ref{th:noetherienultrametrique}. 
\end{proof}

Nous utiliserons \'egalement un r\'esultat de noeth\'erianit\'e pour les faisceaux qui s'\'enonce sous la forme suivante.

\begin{prop}\label{prop:faisceaunoetherien}\index{Faisceau!noetherien@noeth\'erien}
Soit $X$ un espace $\cA$-analytique. Soient~$\cF$ un faisceau coh\'erent sur~$X$ et~$(\cF_{m})_{m\ge 0}$ une suite croissante de sous-faisceaux coh\'erents de~$\cF$. Alors, tout point de~$X$ poss\`ede un voisinage sur lequel la suite $(\cF_{m})_{m\ge 0}$ stationne.
\qed
\end{prop}

Dans le cadre surconvergent qui est le n\^otre, la difficult\'e pour appliquer ce th\'eor\`eme r\'eside dans le fait qu'il requi\`ere que tous les faisceaux soient d\'efinis sur un m\^eme ouvert.

\begin{theo}\label{th:noetherien}
Soient $n\in \N$ et $\br \in \intfo{0,1}^n$. Alors l'anneau $\cO(\oD_{V}(\br))$ est noeth\'erien. 
\end{theo}
\begin{proof}
Soit $(f_{m})_{m\ge 0}$ une suite d'\'el\'ements de $\cO(\oD_{V}(\br))$. Pour tout $m\ge 0$, notons~$\cI_{m}$ le faisceau d'id\'eaux engendr\'e par~$f_{1},\dotsc,f_{m})$. Insistons sur le fait qu'il est \textit{a priori} d\'efini sur un voisinage de $\oD_{V}(\br)$ qui d\'epend de~$m$.

Soit $\sigma \in \Sigma_{0}$. Posons $W_{\sigma} := \{a_{\sigma}^\eps : \eps \in [\eps_{\sigma}/3,\eps_{\sigma}]\}$. D'apr\`es le lemme~\ref{lem:restrictionbrancheiso}, l'anneau $\cO(\oD_{W_{\sigma}}(\br))$ est noeth\'erien.
Il existe donc~$m_{\sigma}$ tel que, pour tout $m\ge m_{\sigma}$, on ait
\[(f_{0},\dotsc,f_{m}) \, \cO(\oD_{W_{\sigma}}(\br)) = (f_{0},\dotsc,f_{m_{\sigma}}) \, \cO(\oD_{W_{\sigma}}(\br)).\] 
D'apr\`es le corollaire~\ref{cor:thA}, pour tout $m\ge m_{\sigma}$ et tout $x\in \oD_{W_{\sigma}}(\br)$, on a donc $(\cI_{m})_{x} = (\cI_{m_{\sigma}})_{x}$.

Pour tout $\alpha\in \intoo{0,1}$, posons 
\begin{align*}
V_{\alpha} &= \{b^\alpha :  b \in V\}\\
& = \bigcup_{\sigma\in\Sigma_{0}} \{a_{\sigma}^\eps : \sigma \in \intff{0, \alpha \, \eps_{\sigma}}\} \cup \bigcup_{\sigma\in\Sigma_{f} \setminus \Sigma_{f,0}} \{a_{\sigma}^\eps : \sigma \in \intff{0, +\infty}\}.
\end{align*}
Il suit du lemme~\ref{lem:descriptionseriesarithmetiques} que l'on a $\cO(\oD_{V_{\alpha}}(\br^{\alpha})) = \cO(\oD_{V}(\br))$. En particulier, toute fonction analytique d\'efinie au voisinage de~$\oD_{V}(\br)$ est encore d\'efinie sur au voisinage de $\oD_{V_{1/2}}(\br^{1/2})$. Notons~$U_{1/2}$ l'int\'erieur de ce dernier polydisque dans~$\E{n}{\cA}$. Puisque~$\br \in \intfo{0,1}^n$, il contient $\oD_{V_{1/3}}(\br)$. On d\'eduit alors de la proposition~\ref{prop:faisceaunoetherien} qu'il existe $m' \ge0$ tel que, pour tout $m\ge m'$ et tout $x\in \oD_{V_{1/3}}(\br)$, on ait $(\cI_{m})_{x} = (\cI_{m'})_{x}$.

Posons $M := \max\big(m', \max_{\sigma\in \Sigma_{0}}(m_{\sigma})\big)$. Soit $m\ge M$. Pour tout $x\in \oD_{V}(\br)$, on a alors $(\cI_{m})_{x} = (\cI_{M})_{x}$. En d'autres termes, le morphisme de faisceaux
\[\fonction{\varphi}{\cO^M}{\cI_{m}}{(g_{1},\dotsc,g_{M})}{\disp\sum_{i=1}^M g_{i} f_{i}}\]
est surjectif sur~$\oD_{V}(\br)$. Son noyau~$\cK$ est un faisceau coh\'erent. De la suite exacte courte $ 0 \to \cK \to \cO^{M} \to \cI_{m} \to 0$, on tire une suite exacte longue
\[ \dotsc \too \cO(\oD_{V}(\br))^M \xrightarrow[]{\varphi(\oD_{V}(\br))} \cI_{m}(\oD_{V}(\br)) \too H^1(\oD_{V}(\br),\cK) \too \dotsc\]
Or, d'apr\`es le corollaire~\ref{cor:thB}, on a $H^1(\oD_{V}(\br),\cK) =0$, d'o\`u il d\'ecoule que $f_{m} \in (f_{1},\dotsc,f_{M})\, \cO(\overline{D}_{V}(\br))$. On en d\'eduit que l'anneau $\cO(\oD_{V}(\br))$ est noeth\'erien.
\end{proof}
\index{Disque!algebre@alg\`ebre d'un!noeth\'erianit\'e de l'|)}

Traduisons ce r\'esultat en termes plus concrets. 

\begin{nota}%
\nomenclature[Cjb]{$R_{\sigma}(f)$}{rayon de convergence d'une s\'erie~$f$ en une place~$\sigma$ d'un corps de nombres}
Soit $n\in \N$. Soit $f \in \cA[1/\Sigma_{0}]\llbracket \bT\rrbracket$. Pour tout $\sigma\in \Sigma_{0}$ et tout $\br \in \R_{\ge 0}^n$, on \'ecrit $R_{\sigma}(f) > \br$ si l'image de~$f$ dans $\hat{K}_{\sigma}\llbracket \bT\rrbracket$ converge au voisinage du disque $D_{\hat{K}_{\sigma}}(\br)$, le corps~$\hat{K}_{\sigma}$ \'etant muni de la valeur absolue~$\va_{\sigma}$.
\end{nota}

\begin{coro}\label{cor:noetherienconcret}
Soient $n\in \N$ et $\br\in \intfo{0,1}^n$. Pour tout $\sigma\in \Sigma_{0}$, soit $\alpha_{\sigma} \in \R_{>0}$. Alors, le sous-anneau de $\cA[1/\Sigma_{0}]\llbracket \bT\rrbracket$ constitu\'e des s\'eries~$f$ telles que
\[ \forall \sigma\in \Sigma_{0},\ R_{\sigma}(f) > \br^{\alpha_{\sigma}}\] 
est noeth\'erien.
\end{coro}
\begin{proof}
Quitte \`a multiplier les~$\alpha_{\sigma}$ par une constante~$\gamma >0$ et remplacer~$\br$ par~$\br^{1/\gamma}$, on peut supposer que, pour tout $\sigma \in \Sigma_{0}$, on a $\alpha_{\sigma} \in \intoo{0,1}$. Posons 
\[ V := \bigcup_{\sigma \in \Sigma_{0}} \{a_{\sigma}^\eps : \eps \in [0,\alpha_{\sigma}]\} \cup  \bigcup_{\sigma \in \Sigma \setminus \Sigma_{0}} \{a_{\sigma}^\eps : \eps \in [0,+\infty]\}.\] 
D'apr\`es le lemme~\ref{lem:descriptionseriesarithmetiques}, l'anneau d\'ecrit dans l'\'enonc\'e n'est autre que~$\cO(\oD_{V}(\br))$. Le th\'eor\`eme~\ref{th:noetherien} assure qu'il est noeth\'erien.
\end{proof}

\index{Anneau!noetherien@noeth\'erien|)}
\index{Series arithmetiques convergentes@S\'eries arithm\'etiques convergentes|)}



\nocite{}
\bibliographystyle{alpha}
\bibliography{biblio}

\printindex

\nomenclature[Aa]{$\cA$}{anneau de Banach\nomnorefpage}

\renewcommand{\nomgroup}[1]{%
\vspace{.3cm}%
\ifthenelse{\equal{#1}{G}}{\item[$\bullet$\ \textbf{Cas de $\cA = \Z$}]}{%
\ifthenelse{\equal{#1}{H}}{\item[$\bullet$\ \textbf{Cas de $\cA = A$, anneau des entiers d'un corps de nombres~$K$}]}{%
\ifthenelse{\equal{#1}{F}}{\item[\textbf{Spectre analytique d'un anneau de Banach}]}{%
\ifthenelse{\equal{#1}{I}}{\item[\textbf{Espace affine analytique sur un anneau de Banach}]}{%
\ifthenelse{\equal{#1}{K}}{\item[\textbf{Espace analytique sur un anneau de Banach}]}{%
\ifthenelse{\equal{#1}{B}}{\item[\textbf{Normes et anneaux de Banach}]}{%
\ifthenelse{\equal{#1}{L}}{\item[\textbf{Syst\`emes de Banach}]}{%
\ifthenelse{\equal{#1}{D}}{\item[\textbf{Espaces topologiques}]}{%
\ifthenelse{\equal{#1}{J}}{\item[\textbf{Droite affine analytique sur un anneau de Banach}]}{
\ifthenelse{\equal{#1}{E}}{\item[\textbf{Faisceaux coh\'erents}]}{%
\ifthenelse{\equal{#1}{R}}{\item[\textbf{Arbres binaires de compacts}]}{%
\ifthenelse{\equal{#1}{C}}{\item[\textbf{S\'eries}]}{%
}}}}}}}}}}}}
\vspace{.2cm}%
}
\printnomenclature[2.2cm]

\end{document}